\documentclass[a4paper,english]{smfart}

\usepackage[T1]{fontenc}
\usepackage{lmodern}
\usepackage[utf8]{inputenc}
\usepackage{babel}
\usepackage{eulervm}

\usepackage{amssymb,textcomp,mathrsfs,braket,commath,relsize,stmaryrd,musicography,booktabs}
    \expandafter\def\csname opt@stmaryrd.sty\endcsname
    {only,shortleftarrow,shortrightarrow}

\usepackage{bm}
    \SetSymbolFont{stmry}{bold}{U}{stmry}{m}{n}

\usepackage{etoolbox}
\newcommand{\addQEDstyle}[2]{\AtBeginEnvironment{#1}{\pushQED{\qed}\renewcommand{\qedsymbol}{#2}}
    \AtEndEnvironment{#1}{\popQED}}
    \addQEDstyle{rema}{$\diamondsuit$}
    \addQEDstyle{rema*}{$\diamondsuit$}
    \addQEDstyle{exem}{$\diamondsuit$}
    \addQEDstyle{exem*}{$\diamondsuit$}

\usepackage{tikz}
    \usetikzlibrary{cd,decorations.pathreplacing,decorations.pathmorphing}

\usepackage{float,extpfeil,accents,enumitem,longtable}
\AtBeginDocument{\def\MR#1{}}
\apptocmd{\sloppy}{\hbadness 10000\relax}{}{}
\usepackage{mathtools}
    \mathtoolsset{showonlyrefs}
\usepackage{smfthm} 
    \NumberTheoremsAs{subsubsection}\SwapTheoremNumbers

\usepackage{mymacros}

\usepackage[marginpar]{todo}

\usepackage{comment}
\excludecomment{compilethis}

\usepackage[table]{xcolor}

\makeatletter
\let\gr@tocwrite\@tocwrite
\def\gr@toctitle{}
\def\toctitle#1{\gdef\gr@toctitle{#1}}
\def\@tocwrite#1#2{%
  \ifx\gr@toctitle\@empty
    \gr@tocwrite{#1}{#2}%
  \else
    \gr@tocwrite{#1}{\gr@toctitle}%
    \gdef\gr@toctitle{}%
  \fi}
\makeatother

\usepackage[colorlinks,psdextra,pdfencoding=auto]{hyperref}
    \hypersetup{colorlinks,linkcolor={red!50!black},filecolor={green!50!black},urlcolor={blue!90!black},citecolor={blue!50!black}}

\title[Twisted local $G$-WMCG\MakeLowercase{s}]{Twisted local $G$-wild mapping class groups}

\author[J.~Douçot]{Jean Douçot\thanks{J.~D. has been funded by the European PNRR Grant CF 44/14.11.2022, 
	`Cohomological Hall Algebras of Smooth Surfaces and Applications', and more recently by the French DAG-Arts ANR project ANR-24-CE40-4098.}}
	\address[J.~Douçot]{`Simion Stoilow' Institute of Mathematics of the Romanian Academy,
	Calea Griviței 21,
	010702-Bucharest, 
	Sector 1, 
	Romania}
\curraddr{\emph{Current address}: Institut de Recherche Mathématique Avancée, 7 rue René Descartes, 67084 Strasbourg Cedex, France
}	
	\email{jean.doucot@math.unistra.fr}

\author[G. Rembado]{Gabriele Rembado\thanks{G.R. has received funding from the European Union's Horizon Europe research and innovation programme under the Marie Skłodowska-Curie grant agreement No 101108575 (project QuantMod, `Quantisation of moduli spaces: Hitchin connections and isomonodromic deformations');
more recently,
also by the \emph{Ministerio de Ciencia,
Innovación y Universidades},
under the grant PID2024-155686NB-I00.}
}
    \address[G.~Rembado]{Institut Montpelliérain Alexander Grothendieck,
    University of Montpellier, 
    Place Eugène Bataillon,
    34090 Montpellier,
    France}
    \curraddr{\emph{Current address}: Departament de Matemàtiques,
            Universitat Politècnica de Catalunya (UPC),
Campus Diagonal Sud, 
Av. Doctor Marañón 44-50,
08028 Barcelona,
Spain}
    \email{gabriele.rembado@upc.edu}

\author[D.~Yamakawa]{Daisuke Yamakawa\thanks{D.~Y. is supported by JSPS KAKENHI Grant Number 24K06695.}
}
	\address[D.~Yamakawa]{Department of Mathematics,
		Faculty of Science Division I,
		Tokyo University of Science, 
		1-3 Kagurazaka, 
		Shinjuku-ku, 
		Tokyo 162-8601, 
		Japan}
	\email{yamakawa@rs.tus.ac.jp}

\begin{document}

\frontmatter

\begin{abstract}
	We consider isomonodromic-deformation parameters of irregular-singular connections on principal $G$-bundles over complex projective curves:
	for any connected reductive $G$,
	pole order,
	ramified formal normal form,
	and twist in the interior of the curve;
    covering the general setup of meromorphic reductive $2d$ gauge theory.

	We focus on the formal moduli,
    studying the deformation spaces of irregular types/classes at each pole:
    their fundamental groups are the eponymous local $G$-wild mapping class groups ($G$-WMCGs),
	akin to the complex braid groups of hyperplane arrangements à la Broué--Malle--Rouquier.
	We prove that they sit in exact sequences involving `wild' versions of Weyl groups,
	and we parameterize their isomorphism classes via `Galois-closed' nested sequences of Levi subcosets à la Broué--Malle--Michel.
	This relies on Lehrer--Springer theory,
	and in the quasigeneric case yields new modular interpretations of complex reflection groups and of Bessis' braid-group lift.
	The subregular regime still yields complex reflections in `relative' Springer theory,
	provided an obstruction vanishes in Howlett theory;
	but in general one exits the reflection realm.

	We also introduce new fission trees to treat structure groups of any classical type,
    	after relating with Bell/Stirling numbers,
	classifying the resulting hyperplane arrangements and describing the local $G$-WMCGs explicitly.
	In particular we single out an infinite family of nonreflection examples in type $D$:
    they remain $K(\pi,1)$ by work of Amend--Deligne--Röhrle,
    fitting the modular viewpoint of `wild' Teichmüller spaces in Boalch's theory of wild curves.

	Finally,
	we extend all of this to the interior twists.
	Algebraically,
	this yields a nonsplit setup where the Weyl group is twisted by an infinitesimal outer automorphism of $G$,
	so that fission takes place in a proper reflection coset.
\end{abstract}

{\let\newpage\relax\maketitle}

\setcounter{tocdepth}{1}
\tableofcontents

\mainmatter

\section{Introduction:
  main aim,
  general outline,
  main results,
  and layout}
\label{sec:intro}

\subsection{Main aim}
\label{sec:main_aim}

This is a text about the formal moduli of connections with \emph{irregular singularities},
defined on principal bundles over smooth complex-algebraic curves---%
with reductive structure group.
We extend the bridge with Lie- and diagram-theoretic combinatorics,
by involving (nested) reflection cosets and new (twisted) fission trees,
pointing out that certain discrete groups have natural modular interpretations in meromorphic reductive $2d$ gauge theory:
they include (and extend) many complex reflection groups,
as well as their associated pure/full complex braid groups.

\subsubsection{}

First,
recall that the formal moduli of irregular-singular connections complement the moduli of pointed curves,
and in many respects behave in the same way.
This leads to the notions of \emph{wild/irregular curves} and  \emph{wild character varieties}~\cite{boalch_2014_geometry_and_braiding_of_stokes_data_fission_and_wild_character_varieties,boalch_yamakawa_twisted_wild_character_varieties}---%
which underlie this work.
One overarching goal is the study of moduli spaces/stacks of wild curves (extending~\cite{doucot_rembado_tamiozzo_moduli_spaces_of_untwisted_wild_riemann_surfaces}),
whose fundamental groups lead to `wild' generalizations of the mapping class groups of pointed surfaces,
and of braid groups à la Artin~\cite{artin_1925_theorie_der_zoepfe,artin_1947_theory_of_braids}/Tits~\cite{tits_1969_le_probleme_des_mots_dans_les_groupes_de_coxeter}/Brieskorn~\cite{brieskorn_1971_die_fundamentalgruppe_des_raumes_der_regulaeren_orbits_einer_endlichen_komplexen_spiegelungsgruppe}/Deligne~\cite{deligne_1972_les_immeubles_des_groupes_de_tresses_generalises} (cf.~the exact sequences of~\eqref{eq:artin_braid_group_sequence}--\eqref{eq:g_braid_group_sequence},
significantly generalized by~\eqref{eq:lwmcg_sequence_intro}).

In particular,
this enters into the study of Poisson/symplectic dynamics on (twisted) wild character varieties,
generalizing the much-studied braid/mapping-class-group actions on the usual (tame) character varieties;
and in many examples this leads to projective/linear actions after quantization,
featuring in conformal/topological quantum field theory (cf.~\S~\ref{sec:flat_connections} for some references).

Here we consider the most general setup in the definition of wild curves,
thereby concluding the `local' part of the research programme started in~\cite{doucot_rembado_tamiozzo_2026_local_wild_mapping_class_groups_and_cabled_braids}\fn{
	The first-named author has also been interested in the classification of irregular isomonodromy systems~\cite{doucot_diagrams_and_irregular_connections_on_the_riemann_sphere,
		doucot_hohl_2025_a_topological_algorithm_for_the_fourier_transform_of_stokes_data_at_infinity,
		doucot_2026_basic_representations_of_genus_zero_nonabelian_hodge_spaces,
		doucot_2026_simplification_of_exponential_factors_of_irregular_connections_on_p1};
	and the second-/third-named authors in their quantization~\cite{rembado_2019_simply_laced_quantum_connections_generalising_kz,
		rembado_2020_symmetries_of_the_simply_laced_quantum_connections_and_quantisation_of_quiver_varieties,
		yamakawa_2022_quantization_of_simply_laced_isomonodromy_systems_by_the_quantum_spectral_curve_method,
		felder_rembado_2023_singular_modules_for_affine_lie_algebras_and_applications_to_irregular_wznw_conformal_blocks,
		calaque_felder_rembado_wentworth_wild_orbits_and_generalised_singularity_modules_stratifications_and_quantisation,chaffe_rembado_yamakawa_wild_genus_zero_quantum_de_rham_spaces} (cf.~\cite{andersen_malusa_rembado_2022_genus_one_complex_quantum_chern_simons_theory,
		andersen_malusa_rembado_2024_sp_1_symmetric_hyperkaehler_quantisation} in the nonsingular case).}~(cf.~the continuation~\cite{doucot_rembado_2025_topology_of_irregular_isomonodromy_times_on_a_fixed_pointed_curve,rembado_2024_a_colourful_classification_of_quasi_root_systems_and_hyperplane_arrangements,boalch_doucot_rembado_2025_twisted_local_wild_mapping_class_groups_configuration_spaces_fission_trees_and_complex_braids,doucot_rembado_tamiozzo_moduli_spaces_of_untwisted_wild_riemann_surfaces}).

\subsubsection{}
\label{sec:ehresmann}

Let $G$ be a connected reductive algebraic group over $\mb C$.
The main motivation for studying moduli of wild curves comes from the theory of \emph{isomonodromic deformations} of irregular-singular connections on principal $G$-bundles over complex curves,
viz.,
those deformations which locally preserve monodromy/Stokes data.
This vastly generalizes the genus-zero `generic' setup of Jimbo--Miwa--Ueno (= JMU~\cite{jimbo_miwa_ueno_1981_monodromy_preserving_deformation_of_linear_ordinary_differential_equations_with_rational_coefficients_i_general_theory_and_tau_function}),
on trivial vector bundles,
where the leading term at each pole is a diagonalizable matrix with simple spectrum.

Concretely,
and geometrically,
one may regard isomonodromic deformations as flat fibre bundles $\ul{\mc M}_{\on B} \to \bm B$ of complex-affine Poisson/symplectic varieties,
whose fibres are `Betti' moduli spaces in meromorphic nonabelian Hodge theory:
the wild archetype is~\cite{boalch_2001_symplectic_manifolds_and_isomonodromic_deformations} (in the setting of JMU;
cf.~the surveys~\cite{boalch_hyperkaehler_manifolds_and_nonabelian_hodge_theory_of_irregular_curves,boalch_2018_wild_character_varieties_meromorphic_hitchin_systems_and_dynkin_diagrams} about the subsequent extensions).
So the fundamental group $\pi_1(\bm B,b)$ acts on the fibre over $b \in \bm B$ by Poisson/symplectic automorphisms,
`braiding' Stokes data~\cite{boalch_2014_geometry_and_braiding_of_stokes_data_fission_and_wild_character_varieties} (cf.~Dubrovin's~\cite[App.~F]{dubrovin_1996_geometry_of_2d_topological_field_theories}).
Moreover,
the base space $\bm B$ parameterizes an \emph{admissible family} of wild curves,
i.e.,
an admissible deformation of the \emph{irregular types/classes} of a connection at each pole,
on top of a usual $\bm B$-family of tame/pointed curves where the poles move without collisions.

This yields intrinsic time-variables for the corresponding systems of nonlinear PDEs,
locally expressed by integrable---%
nonautonomous---%
Hamiltonian systems:
please refer to~\cite[\S~1]{boalch_doucot_rembado_2025_twisted_local_wild_mapping_class_groups_configuration_spaces_fission_trees_and_complex_braids} for more details,
and for references to the past work of many more people besides the authors and their collaborators,
dating back to Fuchs~\cite{fuchs_1907_ueber_lineare_homogene_differentialgleichungen_zweiter_ordnung_mit_drei_im_endlichen_gelegenen_wesentlich_singulaeren_stellen},
Schlesinger~\cite{schlesinger_1905_ueber_die_loesungen_gewisser_linearer_differentialgleichungen_als_funktionen_der_singulaeren_punkte, schlesinger_1912_ueber_eine_klasse_von_differentialsystemen_beliebiger_ordnung_mit_festen_kritischen_punkten},
and Garnier~\cite{	garnier_1912_sur_des_equations_differentielles_du_troisieme_ordre_dont_l_integrale_generale_est_uniforme_et_sur_une_classe_d_equations_nouvelles_d_ordre_superieur_dont_l_integrale_generale_a_ses_points_critiques_fixes};
there are \emph{several} other useful/important approaches to the Stokes phenomenon~\cite{stokes_1857_on_the_discontinuity_of_arbitrary_constants_which_appear_in_divergent_developments} which one would like to relate to.

\subsubsection{}

As far as this text is concerned,
the main point of \S~\ref{sec:ehresmann} is that the geometry of isomonodromic deformations thus involves actions of discrete---%
fundamental---%
groups,
i.e.,
the \emph{wild mapping class groups} (= WMCGs),
generalizing the ordinary MCGs.
E.g.,
the finite orbits for the standard `tame' actions are intimately related to the algebraic solutions of the corresponding isomonodromy equations,
notably including the genus-zero Schlesinger system and---%
as a particular case---%
the much-studied equation of Painlevé VI~\cite{malgrange_1982_sur_les_deformations_isomonodromiques_i_singularites_regulieres,
	dubrovin_mazzocco_2000_monodromy_of_certain_painleve_vi_transcendents_and_reflection_groups,
	iwasaki_2002_a_modular_group_action_on_cubic_surfaces_and_the_monodromy_of_the_painleve_vi_equation,
	boalch_2005_from_klein_to_painleve_via_fourier_laplace_and_jimbo,
	boalch_2006_list_of_the_known_algebraic_solutions_of_painleve_vi,
	lisovyy_tykhyy_2014_algebraic_solutions_of_the_sixth_painleve_equation,
	landesman_litt_2024_canonical_representations_of_surface_groups}.
(Cf.~\cite{paul_ramis_2015_dynamics_on_wild_character_varieties,
	klimes_2024_wild_monodromy_of_the_fifth_painleve_equation_and_its_action_on_wild_character_variety_an_approach_of_confluence} in the irregular-singular case,
relating the dynamics of WMCGs with other Painlevé equations.)

Finally,
recall that the `shape' of generalized monodromy data depends on the irregular types/classes,
e.g.,
determining the block-triangular structure of unipotent Stokes matrices when $G = \GL_m$.
In this text we deal with \emph{arbitrary} ones,
describing their admissible deformations as explicitly as possible:
their fundamental groups are the eponymous \emph{local} $G$-WMCGs.

\subsubsection{}
\label{sec:flat_connections}

Again,
understanding such local pieces of WMCGs in full generality opens the way to broad generalizations of the aforementioned monodromy actions in $2d$ gauge theory,
and their applications.

Notably,
the actions of Hurwitz/De Concini--Kac--Procesi's~\cite{hurwitz_1891_ueber_riemannsche_flaechen_mit_gegebenen_verzweigungspunkten,deconcini_kac_procesi_1992_quantum_coadjoint_action},
on tame character varieties and dual Poisson--Lie groups $G^*$ (respectively),
are Betti versions of semiclassical limits of well-studied monodromies,
coming from `quantum' flat connections in genus zero:
(i) that of Knizhnik--Zamolodchikov (= KZ~\cite{knizhnik_zamolodchikov_1984_current_algebra_and_wess_zumino_model_in_two_dimensions});
and (ii) that of De Concini/Millson--Toledano Laredo/Felder--Markov--Tarasov--Varchenko (= DMT~\cite{millson_toledanolaredo_2005_casimir_operators_and_monodromy_representations_of_generalised_braid_groups}/FMTV~\cite{felder_markov_tarasov_varchenko_2000_differential_equations_compatible_with_kz_equations});
respectively,
as explained in~\cite{reshetikhin_1992_the_knizhnik_zamolodchikov_system_as_a_deformation_of_the_isomonodromy_problem,boalch_2002_g_bundles_isomonodromy_and_quantum_weyl_groups,rembado_2019_simply_laced_quantum_connections_generalising_kz}.
Recall that KZ (resp.,
DMT) features in the Drinfel'd--Kohno theorem~\cite{kohno_1987_monodromy_representations_of_braid_groups_and_yang_baxter_equations,drinfeld_1989_quasi_hopf_algebras,drinfeld_1989_quasi_hopf_algebras_and_knizhnik_zamolodchikov_equations} for the Jimbo--Drinfel'd quantum groups~\cite{jimbo_1985_a_q_difference_analogue_of_u_g_and_the_yang_baxter_equation,drinfeld_1987_quantum_groups} (resp.,
the Drinfel'd--Kohno theorem for $q$-Weyl groups~\cite{toledanolaredo_2002_a_kohno_drinfeld_theorem_for_quantum_weyl_groups}).

Remaining on the quantum side,
one also encounters the connections of Hitchin/Axelrod--Della Pietra--Witten~\cite{hitchin_1990_flat_connections_and_geometric_quantization,axelrod_dellapietra_witten_1991_geometric_quantization_of_chern_simons_gauge_theory} and Bernard/Tsuchiya--Ueno--Yamada (= KZB~\cite{bernard_1988_on_the_wess_zumino_witten_models_on_the_torus,bernard_1988_on_the_wess_zumino_witten_models_on_riemann_surfaces}/TUY~\cite{tsuchiya_ueno_yamada_1989_conformal_field_theory_on_universal_family_of_stable_curves_with_gauge_symmetries}),
which have been related in~\cite{beauville_laszlo_1994_conformal_blocks_and_generalized_theta_functions,faltings_1994_a_proof_for_the_verlinde_formula,laszlo_1998_hitchin_s_and_wzw_connections_are_the_same,egsgaard_2015_hitchin_connections_for_genus_0_quantum_representations,biswas_mukhopadhyay_wentworth_2023_a_hitchin_connection_on_nonabelian_theta_functions_for_parabolic_g_bundles} by bridging nonabelian theta-functions with conformal blocks in the Wess--Zumino--Novikov--Witten model (= WZNW~\cite{wess_zumino_1971_consequences_of_anomalous_ward_identities,novikov_1982_the_hamiltonian_formalism_and_a_multivalued_analogue_of_morse_theory,witten_1983_global_aspects_of_current_algebra,witten_1984_nonabelian_bosonization_in_two_dimensions}).
The corresponding `quantum' projective actions of MCGs are also well-studied (cf.~\cite{andersen_2006_asymptotic_faithfulness_of_the_quantum_su_n_representations_of_the_mapping_class_groups,charles_2012_asymptotic_properties_of_the_quantum_representations_of_the_modular_group}),
and they feature in Witten/Reshetikhin--Turaev's topological invariants for links-in-$3$-manifolds (= WRT~\cite{witten_1989_quantum_field_theory_and_the_jones_polynomial,reshetikhin_turaev_1990_ribbon_graphs_and_their_invariants_derived_from_quantum_groups,reshetikhin_turaev_1991_invariants_of_3_manifolds_via_link_polynomials_and_quantum_groups}),
coming from compact \emph{quantum} Chern--Simons gauge theory---%
cf.~\cite{chern_simons_1974_characteristic_forms_and_geometric_invariants,freed_1995_classical_chern_simons_theory_i,freed_2002_classical_chern_simons_theory_ii} for the semiclassical theory.

\subsubsection{}

The rest of this introduction will:
(i) present the general outline of this work and provide more background,
also stating a few main results (cf.~\S\S~\ref{sec:background_intro}--\ref{sec:background_intro_2});
(ii) move on to a detailed statement of the other main results (cf.~\S\S~\ref{sec:results}--\ref{sec:results_2});
and (iii) provide a layout of the body/appendices of the article (cf.~\S\S~\ref{sec:layout_1}--\ref{sec:layout_2}).

\subsection{General outline (I)}
\label{sec:background_intro}

\subsubsection{}

Let $X$ be a smooth connected projective curve over $\mb C$.
Mark a finite set $\bm x \sse X$ of closed points,
and consider a connection $\mc A$ defined on a principal $G$-bundle $\mc E$ over the open complement $X^o \ceqq X \sm \bm x$.
This canonically determines an \emph{irregular class} $\Theta_x = \Theta_x(\mc A)$ at each point $x \in \bm x$~\cite[Def.~7]{boalch_yamakawa_twisted_wild_character_varieties}.
In turn,
the latter intrinsically encodes the irregular part of the formal normal form of the connection,
cf.~\cite{boalch_yamakawa_2020_diagrams_for_nonabelian_hodge_spaces_on_the_affine_line,
	boalch_2021_topology_of_the_stokes_phenomenon} and~\S~\ref{sec:about_twisted_irr_classes}.
(This also builds on work of Babbitt,
Balser,
Birkhoff,
Cope,
Deligne,
Écalle,
Fabry,
Hukuhara,
Jurkat,
Levelt,
Loday-Richaud,
Lutz,
Malgrange,
Martinet,
Ramis,
Sibuya,
Turrittin,
Varadarajan;
and many more.)
The triple
\begin{equation}
	\label{eq:wild_riemann_surface}
	\bm X
	= (X, \bm x, \bm \Theta),
	\qquad \bm \Theta \ceqq (\Theta_x)_{x \in \bm x},
\end{equation}
is the corresponding \emph{wild/irregular curve},
as per (the complex-algebraic version of)~\cite[\S~4]{boalch_yamakawa_twisted_wild_character_varieties},
which generalizes~\cite[Def.~8.1 +~Rmk.~10.6]{boalch_2014_geometry_and_braiding_of_stokes_data_fission_and_wild_character_varieties}.
Now the isomonodromic deformations of $(\mc E,\mc A)$ can be phrased over the admissible deformations of~\eqref{eq:wild_riemann_surface},
in the \emph{twisted} version of~\cite[Thm.~10.2]{boalch_2014_geometry_and_braiding_of_stokes_data_fission_and_wild_character_varieties} (cf.~\S~\ref{sec:ehresmann}).
Indeed,
the \emph{local} proof of loc.~cit.~generalizes:
the nontrivial bit is extending the argument of the second paragraph,
for which one should observe that the Stokes groups are the same as those occurring in the untwisted case~\cite[App.~A]{boalch_yamakawa_twisted_wild_character_varieties} (cf.~also~\cite[Cor.~1.2]{boalch_doucot_rembado_2025_twisted_local_wild_mapping_class_groups_configuration_spaces_fission_trees_and_complex_braids} in type $A$).

As in~\cite{doucot_rembado_tamiozzo_2026_local_wild_mapping_class_groups_and_cabled_braids,
	doucot_rembado_2025_topology_of_irregular_isomonodromy_times_on_a_fixed_pointed_curve,
	boalch_doucot_rembado_2025_twisted_local_wild_mapping_class_groups_configuration_spaces_fission_trees_and_complex_braids},
here we do \emph{not} vary the underlying pointed curve $(X,\bm x)$,
thereby focussing on the deformations of irregular classes,
i.e.,
the `wild' isomonodromy times.

\subsubsection{}
\label{sec:about_twisted_irr_classes}

Up to pullbacks/restrictions,
one can then consider an algebraic connection on a principal $G$-bundle over a formal punctured neighbourhood of a point $x \in X$,
i.e.,
over a formal punctured disc:
	the multi-point case essentially follows by taking direct products over the finite set $\bm x \sse X$.
   
Introduce also the analytified closed Riemann surface $\Sigma \ceqq X(\mb C)$,
and let $\wh \Sigma$ be the real-oriented blowup of $\Sigma$ at $x$,
with boundary circle $\partial \sse \wh \Sigma$.
Then the above category of connections is equivalent to that of $G$-local systems on $\partial$,
graded by the `exponential' local system $\mc I$~\cite[Thm.~6]{boalch_yamakawa_twisted_wild_character_varieties}.
(Cf.~\cite{babbitt_varadarajan_1983_formal_reduction_theory_of_meromorphic_differential_equations_a_group_theoretic_view,
	katz_1987_on_the_calculation_of_some_differential_galois_groups,
	martinet_ramis_1991_elementary_acceleration_and_multisummability_i};
in type $A$ this is due to Deligne~\cite[Thm.~2.3]{malgrange_1991_equations_differentielles_a_coefficients_polynomiaux},
cf.~\cite{balser_jurkat_lutz_1979_a_general_theory_of_invariants_for_meromorphic_differential_equations_i_formal_invariants}.)
Any such local system has an equivalence class under local isomorphism,
which is precisely its irregular class $\Theta$.

Finally,
the latter has a well-defined integer ramification $r = r(\Theta) \geq 1$:
one says that $\Theta$ is \emph{untwisted/unramified} if $r = 1$,
else it is \emph{twisted/ramified}.

\begin{rema}
	\label{rmk:interior_twists}

	A further extension involves $\mc G$-\emph{local systems},
	where $\mc G$ is a local system of groups on $\partial$.
	These involve so-called `interior' twists~\cite{boalch_yamakawa_twisted_wild_character_varieties},
	involving torsors for \emph{nonsplit} reductive groups,
	and will be treated in \S~\ref{sec:interior_twist} (cf.~\S~\ref{sec:results_2}).
	Until then,
	just consider the \emph{constant} sheaf corresponding to $G$ (as in~\cite{boalch_doucot_rembado_2025_twisted_local_wild_mapping_class_groups_configuration_spaces_fission_trees_and_complex_braids};
	cf.~\cite[Prop.~8]{boalch_yamakawa_twisted_wild_character_varieties}).
\end{rema}

\subsubsection{}
\label{sec:irr_classes_intro}

To study the admissible deformations of $\Theta$ we initially work on top of a degree-$r$---%
cyclic---%
covering of the formal punctured disc above,
killing the monodromy of the relevant local system.
(Cf.~\cite[\S~4.2 +~App.~A]{boalch_yamakawa_twisted_wild_character_varieties} in general,
and~\cite[\S~2.1.4]{boalch_doucot_rembado_2025_twisted_local_wild_mapping_class_groups_configuration_spaces_fission_trees_and_complex_braids} in type $A$;
and cf.~also Rmk.~\ref{rmk:relation_with_twist}.)

Then we consider:
(i) the untwisted pullback $\wh{\Theta}$ of $\Theta$;
and (ii) an irregular type $\wh{Q}$ with underlying irregular class $\wh{\Theta}$.
Concretely,
choosing a maximal torus $T \sse G$,
we view $\wh\Theta$ as the orbit of $\wh Q$ under the Weyl group $W \ceqq N_G(T) \bs T$.
In turn,
choosing also for convenience an $r$-th root $w$ of a uniformizer $z$ at the point $x$,
any $\wh Q \neq 0$ can be explicitly/uniquely written as in~\eqref{eq:explicit_irregular_type}:
it is the principal part of a Laurent series with values in the Cartan subalgebra $\mf t \ceqq \Lie(T)$ of $\mf g \ceqq \Lie(G)$.
(Cf.~\S~\ref{sec:uniformizers} about the intrinsic notions.)

Incidentally,
note that the prominent role here is played by the split pair $(\mf g,\mf t)$,
rather than $(G,T)$ (cf.~\S~\ref{sec:no_isogeny}),
although this requires some more care when twisting the interior of the curve (cf.~\S~\ref{sec:outer_isogeny}).
More radically,
what counts---%
topologically---%
is the underlying rational reflection group $\mb G \ceqq (\mf t,W)$ (cf.~Cor.~\ref{cor:irregular_moduli},
notably).

\subsubsection{}
\label{sec:general_notation_def_spaces}

The admissible deformations are now as in~\cite{boalch_2014_geometry_and_braiding_of_stokes_data_fission_and_wild_character_varieties},
with the important caveat that $\wh Q$ and $\wh\Theta$ are $r$-\emph{Galois-closed} (cf.~Def.~\ref{def:galois_closed_irr_types}),
as they lift $r$-ramified irregular types/classes.

Then in this article we define and study,
in particular:
\begin{enumerate}
	\item
	      the (topological) space $\bm B_r(\wh Q)$,
	      of $r$-admissible deformations of $\wh Q$;

	\item
	      the space $\bm B_r(\wh\Theta)$,
	      of $r$-admissible deformations of $\wh\Theta$;

	\item
	      the \emph{pure} ($r$-ramified) local $G$-WMCG,
	      viz.,
	      the fundamental group
	      \begin{equation}
		      \Gamma_r\bigl(\wh Q\bigr)
		      \ceqq \pi_1 \bigl( \bm B_r (\wh Q), \wh Q \bigr);
	      \end{equation}

	\item
	      and the \emph{full/nonpure} local $G$-WMCG,
	      viz.,
	      $\ol{\Gamma}_r \bigl( \wh\Theta \bigr) \ceqq \pi_1 \bigl( \bm B_r(\wh \Theta), \wh\Theta \bigr)$.

	      (Cf.~\S\S~\ref{sec:setup} +~\ref{sec:tglwmcgs} for more context and exhaustive definitions.)
\end{enumerate}

Again,
the terminology is due to the fact that these deformation spaces can be regarded as---%
fine---%
moduli spaces of irregular types/classes on the fixed formal germ of a pointed curve (cf.~\cite[Rmk.~2.3]{doucot_rembado_tamiozzo_2026_local_wild_mapping_class_groups_and_cabled_braids},
and~\cite[Cor.~3.34]{boalch_doucot_rembado_2025_twisted_local_wild_mapping_class_groups_configuration_spaces_fission_trees_and_complex_braids},
and \S~\ref{sec:general_tglwmcgs}).

When $r \geq 2$,
the crux of the matter is that the \emph{exponential factors} of the fundamental horizontal sections of the connection $\mc A$ around $x \in X$ are multivalued (cf.~\eqref{eq:galois_conjugates}).
In particular,
there are not as many `true' moduli as in the untwisted case,
which significantly complicates the local study.

\subsubsection{}

Nonetheless,
by fixing (and then forgetting) suitable markings for the $r$-Galois-closed irregular types,
in the form of elements $g \in W$ generating their `$r$-Galois-orbit',
we first describe the above topological spaces:
by (i) providing a topological factorization of $\bm B_r \bigl( \wh Q \bigr)$ into complex hyperplane complements;
and (ii) describing $\bm B_r \bigl( \wh\Theta \bigr)$ as the base of a Galois covering with total space $\bm B_r \bigl( \wh Q \bigr)$,
à la Steinberg (cf.~Thm.~\ref{thm:thm_1_intro}).

Provided that certain `nonreflection' group-theoretic complements vanish,\fn{
	\label{fn:howlett}
	As per Howlett's decomposition of normalizers of parabolic subgroups of finite Coxeter groups~\cite{howlett_1980_normalizers_of_parabolic_subgroups_of_reflection_groups} (cf.~Prop.-Def.~\ref{prop:howlett} +~Cor.-Def.~\ref{cor:howlett_twisted_reflection_cosets}).}~and working in the \emph{monolevel} case (cf.~\eqref{eq:levels} +~\S~\ref{sec:pure_nongeneric_case}),
the deck transformations now make \emph{complex} reflection groups appear in meromorphic $2d$ gauge theory~(cf.~Thm.~\ref{thm:complex_refl_groups_from_gauge}).

\subsubsection{}

Afterwards,
we determine a finer decomposition of $\bm B_r \bigl( \wh Q \bigr)$ into irreducible hyperplane complements,
when the Lie algebra $\mf g$ is of \emph{classical} type $\bullet \in \set{A,B,C,D}$ (cf.~Lemm.~\ref{lem:reduction_to_simple_case} +~\ref{lem:reduction_to_simple_full} about the reduction to simple Lie algebras).

Incidentally,
note that the distinction between $\mf{sl}_m(\mb C)$ and the nonsemisimple Lie algebra $\mf{gl}_m(\mb C)$ plays a negligible role here,
for any integer $m \geq 1$:
what matters---%
topologically---%
is the \emph{essential} part of the $W$-action on $\mf t$.
Somewhat analogously,
the distinction between types $B/C$ is essentially moot,
as their root systems are mutually dual and one can identify the corresponding Weyl groups (as being `of type $B\!C$';
cf.~\S\S~\ref{sec:complex_real_rational_reflection_groups} +~\ref{sec:background_type_BCD}).

In any event,
the finer decomposition classifies \emph{all} the hyperplane complements which arise from `classical' $r$-admissible deformations.
As typically happens,
type $D$ is the most convoluted:
there we first require that the underlying `untwisted' admissible deformations---%
forgetting about $r$-Galois-closedness---%
lead to \emph{crystallographic} hyperplane arrangements.\fn{
	This partially overlaps with Howlett's theory:
	cf.~Fn.~\ref{fn:howlett} and~\S~\ref{sec:noncrystallographic_D}.}~Then,
to remove this assumption and prove Thm.~\ref{thm:thm_2_intro},
we introduce $G$-\emph{fission trees} which generalize those of~\cite{doucot_rembado_tamiozzo_2026_local_wild_mapping_class_groups_and_cabled_braids,doucot_rembado_2025_topology_of_irregular_isomonodromy_times_on_a_fixed_pointed_curve,boalch_doucot_rembado_2025_twisted_local_wild_mapping_class_groups_configuration_spaces_fission_trees_and_complex_braids} (cf.~\cite[App.~C]{boalch_irregular_connections_and_kac_moody_root_systems} and~\cite{boalch_2025_counting_the_fission_trees_and_nonabelian_hodge_graphs_untwisted_case}).

\begin{rema}
	The crystallographic setting is dealt with by attaching---%
    rational---%
    `relative' reflections groups to Levi subsystems (cf.~\S~\ref{sec:lehrer_springer_theory}),
	which always contain the corresponding relative Weyl group (cf.~Lem.~\ref{lem:relative_weyl_group}).
    The main point is to set up Lehrer--Springer theory~\cite{lehrer_springer_1999_reflection_subquotients_of_unitary_reflection_groups},
    proving that the relevant hyperplane arrangements come from complex reflections (although the group they generate need \emph{not} match up with the corresponding `wild' Weyl group;
    cf.~\S\S~\ref{sec:background_intro_2} +~\ref{sec:some_lehrer_springer}).

	In particular,
	any Levi subsystem determines an equivalence relation on a finite set,
	and counting the corresponding (signed) set-partitions yields modular interpretations of \emph{Bell/Stirling numbers} of types $\bullet \in \Set{A,B,C,D}$ (and of their `higher-order' versions,
	cf.~Rmkk.~\ref{rmk:bell_and_stirling} +~\ref{rmk:integer_partitions},
	Lem.-Def.~\ref{lem:signed_partitions},
	and Rmkk.~\ref{rmk:bell_and_stirling_BC} +~\ref{rmk:bell_and_stirling_D};
	and references therein).
	The same might work in the twisted setting (cf.~\cite{sagan_swanson_2026_stirling_numbers_for_complex_reflection_groups}).
\end{rema}

\subsubsection{}

But the usage of fission trees goes much further.
Notably,
the trees determine the monodromy group of the aforementioned Galois covering
\begin{equation}
	\label{eq:galois_covering_intro}
	\bm B_r \bigl( \wh Q \bigr) \lthra \bm B_r \bigl( \wh\Theta \bigr).
\end{equation}
E.g.,
in the untwisted type-$A$ case this Galois group depends on the exact nesting/breaking of eigenspaces of the coefficients of $\wh Q$:
it involves iterated wreath products of symmetric groups,
as per the \emph{cabling} of the symmetric group operad~\cite[Prop.~3.2.9]{yau_2021_infinity_operads_and_monoidal_categories_with_group_equivariance}.
So it can be recursively determined along the subtrees of a decorated rooted tree,
just as the corresponding \emph{cabled} Artin braid groups (cf.~\cite{doucot_rembado_tamiozzo_2026_local_wild_mapping_class_groups_and_cabled_braids, doucot_rembado_2025_topology_of_irregular_isomonodromy_times_on_a_fixed_pointed_curve} and~\cite[Prop.~5.2.5]{yau_2021_infinity_operads_and_monoidal_categories_with_group_equivariance}).

The twisted version was considered in~\cite{boalch_doucot_rembado_2025_twisted_local_wild_mapping_class_groups_configuration_spaces_fission_trees_and_complex_braids},
and in this text we generalize such trees to treat \emph{all} the classical twisted examples (cf.~Thm.~\ref{thm:thm_3_intro}).

\begin{rema}
	Upon untwisting the irregular type $Q$ (on top of the $r$-fold covering of the formal punctured disc),
	the trees also govern the `fission'~\cite{boalch_2009_through_the_analytic_halo_fission_via_irregular_singularities,boalch_2014_geometry_and_braiding_of_stokes_data_fission_and_wild_character_varieties} of the structure group at an irregular singularity.
	This breaks $G$ down to the connected reductive Levi factor $L \sse P$ of a parabolic subgroup $P \sse G$,
	by taking Adjoint stabilizers of nested truncations of $\wh Q$ (cf.~\cite[Eqq.~(2.2)--(2.3)]{biquard_boalch_2004_wild_non_abelian_hodge_theory_on_curves} when on vector bundles).
	In turn,
	this operation parameterizes `untwisted' root-valuation strata in the truncated-current version of the Cartan subalgebra $\mf t \sse \mf g$ (cf.~Rmkk.~\ref{rmk:about_stratifications}--\ref{rmk:counting_untwisted_strata}).

	Conversely,
	the stratification is tantamount to considering \emph{all} the admissible deformation families at once (also in the full/nonpure case,
	cf.~\cite{calaque_felder_rembado_wentworth_wild_orbits_and_generalised_singularity_modules_stratifications_and_quantisation}).
	The twisted version extends a truncation of~\cite{goresky_kottwitz_macpherson_2009_codimensions_of_root_valuation_strata,sabitova_2012_a_simplification_of_root_valuation_data_for_classical_groups},
	in equicharacteristic zero---%
	and we plan to spell this out elsewhere.
\end{rema}

\subsection{General outline (II)}
\label{sec:background_intro_2}

\subsubsection{}

In a precise sense,
our starting point is a paragraph of Springer~\cite[\S~5.6]{springer_1974_regular_elements_of_finite_reflection_groups}:

\begin{center}
	\emph{The results of this section lead to certain complex reflection groups [...] \\
		One can [...]
		derive the main results about polyhedral groups \\
		from the theory of Weyl groups [...]
		We shall not go into this matter [...]}
\end{center}

\subsubsection{}

In this text we shall use Springer theory to describe the $r$-admissible-deformation spaces in terms of complex hyperplane arrangements,
lying within eigenspaces of elements $g \in W$.
In particular,
we relate \emph{regular} eigenspaces with \emph{quasigeneric} isomonodromic deformations,
corresponding---%
by definition---%
to irregular types with regular semisimple leading term up to a cyclic covering.\fn{
	The terminology is reminiscent of `quasiunipotent' elements,
	i.e.,
	roots of unipotent elements:
	cf.~the `equivalued' strata of~\cite[\S~4.9]{goresky_kottwitz_macpherson_2009_codimensions_of_root_valuation_strata},
	and the `isoclinic' Deligne--Simpson problem of~\cite{jacob_yun_a_deligne_simpson_problem_for_irregular_g_connections_over_p1,kamgarpour_whitebread_counting_points_on_braid_varieties_and_the_deligne_simpson_problem}.}~This yields new modular interpretations of the generalized symmetric groups in Shephard--Todd's list~\cite{shephard_todd_1954_finite_unitary_reflection_groups,cohen_1976_finite_complex_reflection_groups},
and of the corresponding pure/full braid groups~\cite{broue_malle_rouquier_1998_complex_reflection_groups_braid_groups_hecke_algebras}.

But there are genuinely new examples,
as envisioned in~\cite{boalch_doucot_rembado_2025_twisted_local_wild_mapping_class_groups_configuration_spaces_fission_trees_and_complex_braids}.
E.g.,
a free action of the primitive complex reflection group $G_{31}$ is now seen to govern the formal moduli of quasigeneric isomonodromic deformations on principal $E_8$-bundles,
of suitable ramification $r \geq 4$ (cf.~Prop.-Def.~\ref{prop:full_generic_twisted_deformation_space_1_coeff} +~\S~\ref{sec:bessis_literature}).

Precisely,
more generally:

\begin{coro}[cf.~Propp.~\ref{prop:classification_type_A} +~\ref{prop:classification_type_BC} +~\ref{prop:generic_classification_type_D} +~\S~\ref{sec:springer_regular_reflection_groups}]
	\label{cor:modular_reflection_groups}

	Let $\Theta$ be a \emph{quasigeneric} $r$-ramified irregular class for a \emph{simple} Lie algebra $\mf g$,
	with untwisted pullback $\wh\Theta$;
	choose an irregular type $\wh Q$ with irregular class $\wh\Theta$.
	Then:
	\begin{enumerate}
		\item
		      up to deformation-retraction of $r$-admissible-deformation spaces,
		      there exists an \emph{irreducible} complex reflection group $(V,W')$ such that the Galois covering~\eqref{eq:galois_covering_intro} is---%
		      $\mb C$-linearly---%
		      \emph{isomorphic to} $V_{\reg} \thra V_{\reg} \bs W'$;

		\item
		      and letting $r$ vary,
		      these complex reflection groups are classified as follows:
		      \begin{enumerate}
			      \item
			            if $\mf g$ is of \emph{classical} type $\bullet \in \set{A,B,C,D}$,
			            then $W'$ is:
			            \begin{enumerate}
				            \item
				                  either a generalized symmetric group $G(\rho,1,k) \simeq (\mb Z \bs \rho\mb Z) \wr \mf S_k$ of any ramification $\rho \geq 1$ and rank $k \geq 1$,
				                  exhausting the examples coming from $\bullet \in \set{A,B,C}$;

				            \item
				                  or an index-$2$ subgroup $G(\rho,2,k) \sse G(\rho,1,k)$ ($\rho$ even),
				                  only arising if $\bullet = D$,
				                  with the parity constraint that $\rho \equiv 2 \pmod 4$ if $k$ is \emph{odd};
			            \end{enumerate}

			      \item
			            else,
			            if $\mf g$ is of \emph{exceptional} type $\bullet \in \set{E,F,G}$,
			            then $W'$ is:
			            \begin{enumerate}
				            \item
				                  either cyclic,
				                  of order $\abs{W'} \in \set{ 6,8,9,12,14,18,20,24,30 }$;

				            \item
				                  or a Weyl group of exceptional type;

				            \item
				                  or a \emph{primitive} (exceptional) complex reflection group listed in Tab.~\ref{tab:springer_from_weyl},
                                  \emph{excluding} the first row---%
                                  highlighted in grey.
			            \end{enumerate}
		      \end{enumerate}
	\end{enumerate}
\end{coro}

\begin{table}
	\centering
	\begin{tabular}{ccccc}
		\toprule
		$W'$     & Rank $(V,W')$ & Degrees $(V,W')$ & Source type $\mf g$ & Minimal ramification $r$ \\
		\midrule
        \rowcolor[gray]{0.92} $G_4$ & 2             & $4, 6$           & $D_4$               & $3$    \\
		$G_5$    & 2             & $6, 12$          & $F_4$ / $E_6$       & $3$ / $6$                \\
		$G_8$    & 2             & $8, 12$          & $F_4$ / $E_6$       & $4$ / $4$                \\
		$G_9$    & 2             & $8, 24$          & $E_8$               & $8$                      \\
		$G_{10}$ & 2             & $12, 24$         & $E_8$               & $12$                     \\
		$G_{16}$ & 2             & $20, 30$         & $E_8$               & $5$                      \\
		\midrule
		$G_{25}$ & 3             & $6, 9, 12$       & $E_6$               & $3$                      \\
		$G_{26}$ & 3             & $6, 12, 18$      & $E_7$               & $3$                      \\
		\midrule
		$G_{31}$ & 4             & $8, 12, 20, 24$  & $E_8$               & $4$                      \\
		$G_{32}$ & 4             & $12, 18, 24, 30$ & $E_8$               & $3$                      \\
		\bottomrule
	\end{tabular}
	\caption{
		\label{tab:springer_from_weyl}
		Irreducible nonreal/noncyclic complex reflection groups arising from quasigeneric isomonodromic deformations.
    (The first row only in the \emph{interior-twisted} case; cf.~Cor.~\ref{cor:modular_reflection_groups_twisted}~(2.a.iii.B).)}
\end{table}

\begin{rema}
	If $\mf g$ is semisimple then the group $W'$ splits into irreducible factors along the simple ideals of $\mf g$,
	and adding a central term does not contribute to the Weyl/reflection groups (nor to the braid groups,
	cf.~again Lem.~\ref{lem:reduction_to_simple_case}).
	Nonetheless,
	since one argues recursively on the levels of an irregular type/class (cf.~again~\eqref{eq:levels}),
	fission forces the more general reductive setting.
\end{rema}

\subsubsection{}

The classical examples of Cor.~\ref{cor:modular_reflection_groups}~(2.a) extend the `unramified' Weyl groups
\begin{equation}
	W_A(k)
	= G(1,1,k) \simeq \mf S_k,
	\qquad W_{B\!C}(k)
	= G(2,1,k) \simeq \mb Z^{\ts} \wr \mf S_k
\end{equation}
and $W_D(k) = G(2,2,k) \sse W_{B\!C}(k)$.
As for the exceptional ones,
the first column of Tab.~\ref{tab:springer_from_weyl} gives the label of $W'$ within Shephard--Todd's list.
The next two columns recall its rank/degrees:
they are determined---%
by Springer---%
in terms of the degrees of the rational reflection group $\mb G \ceqq (\mf t,W)$ which are divisible by a \emph{regular number} $d \geq 2$ (cf.~Rmk.~\ref{rmk:about_regular_numbers};
the cyclic examples correspond to regular numbers dividing a unique degree,
e.g.,
Coxeter numbers).
Finally,
according to the type of the simple Lie algebra $\mf g$,
we provide the minimal such regular number;
it divides the ramification $r$,
and so this also yields the minimal viable ramification (cf.~Thm.~\ref{thm:thm_1_intro}~(4.)).

(In Tab.~\ref{tab:springer_from_weyl} we also separate groups of different ranks via horizontal lines.)

\begin{rema}
	Incidentally,
	recall that the groups
	\begin{equation}
		G_i,
		\qquad i \in \set{5,8,9,10,16,25,26,32},
	\end{equation}
	are symmetry groups of \emph{complex polytopes} (a.k.a.~`Shephard groups'~\cite{coxeter_1974_regular_complex_polytopes,orlik_reiner_shepler_2002_the_sign_representation_for_shephard_groups}),
	and this is presumably the intended meaning of the above quote from~\cite[\S~5.6]{springer_1974_regular_elements_of_finite_reflection_groups}.
	In particular,
	they are all well-generated.

	Therefore,
	the \emph{non-well-generated} examples consist of an infinite subfamily of the monomial groups $G(\rho,2,k)$ (with $\rho \geq 4$ even and $k \geq 2$),\fn{
    The parity constraint of Cor.~\ref{cor:modular_reflection_groups}~(2.a.ii) \emph{drops} by also twisting the interior of the curve (cf.~Cor.~\ref{cor:modular_reflection_groups_twisted}~(2.a.ii)).
So in the end one does find the \emph{full} family of the $G(\rho,2,k)$.}~plus the `isolated' $G_{31}$;
	interestingly,
	the latter played a distinguished role in the $K(\pi,1)$ conjecture/theorem for complex-reflection complements~\cite{bessis_2015_finite_complex_reflection_arrangements_are_k_pi_1} (cf.~again \S~\ref{sec:bessis_literature},
	as well as \S~\ref{sec:bessis_intro}).
\end{rema}

\subsubsection{}

In brief,
the modular reflection groups of Cor.~\ref{cor:modular_reflection_groups} arise from \emph{regular} Springer theory applied to irreducible Weyl groups.
But our general nonquasigeneric situation involves `relative/subregular' Springer theory,
intersecting eigenspaces with flats of the root-hyperplane arrangement:
and one \emph{cannot} always reduce to quasigeneric lower-rank cases---%
due to the aforementioned nonreflection obstructions.

Thus,
more generally,
we exit the reflection realm and give modular interpretations to `wild' versions of Weyl groups,
i.e.,
the subquotients $W_r \bigl( \wh Q \bigr)$ of $W$\fn{
	As in `quotients of subgroups of $W$'.}~arising as groups of deck transformations of the Galois covering~\eqref{eq:galois_covering_intro}.
These \emph{wild Weyl groups} (= WWGs) are finite quotients of the corresponding full local WMCGs:
they sit in a short exact group sequence
\begin{equation}
	\label{eq:lwmcg_sequence_intro}
	1 \lra \Gamma_r \bigl( \wh Q \bigr) \lra \ol\Gamma_r \bigl( \wh\Theta \bigr) \lra W_r \bigl( \wh Q \bigr) \lra 1,
\end{equation}
vastly generalizing the usual sequence
\begin{equation}
	\label{eq:artin_braid_group_sequence}
	1 \lra \PBr_m \lra \Br_m \lra \mf S_m \lra 1,
\end{equation}
involving Artin's braid groups~\cite{artin_1925_theorie_der_zoepfe,artin_1947_theory_of_braids} on $m \geq 1$ strands and the symmetric groups;
as well as the more general $\mf g$-braid version~\cite{tits_1969_le_probleme_des_mots_dans_les_groupes_de_coxeter,brieskorn_1971_die_fundamentalgruppe_des_raumes_der_regulaeren_orbits_einer_endlichen_komplexen_spiegelungsgruppe,deligne_1972_les_immeubles_des_groupes_de_tresses_generalises},
i.e.,
\begin{equation}
	\label{eq:g_braid_group_sequence}
	1 \lra \pi_1 \bigl( \mf t_{\reg} \bigr) \lra \pi_1 \bigl( \ol{\mf t}_{\reg} \bigr) \lra W \lra 1,
	\qquad \ol{\mf t}_{\reg}
	\ceqq \mf t_{\reg} \bs W,
\end{equation}
omitting basepoints and involving the $W$-regular part $\mf t_{\reg} \sse \mf t$---%
viz.,
the complement of the root-hyperplanes,
cf.~\eqref{eq:regular_cartan} +~\eqref{eq:G_braid_sequence}.

\subsubsection{}

Moreover,
recall that one main aim involves the classification of admissible-deformation families of wild curves,
leading to classifications of isomonodromy systems,
etc.

This can be phrased in terms of \emph{topological skeleta} of wild curves~\eqref{eq:wild_riemann_surface} (cf.~\cite[\S~3.7]{boalch_doucot_rembado_2025_twisted_local_wild_mapping_class_groups_configuration_spaces_fission_trees_and_complex_braids} in type $A$),
which generalize the genus $g \ceqq \on{genus}(X)$ and number of marked points $n \ceqq \abs{\bm x}$ of an ordinary smooth tame/pointed curve $(X,\bm x)$.
Indeed,
the irreducibility~\cite{deligne_mumford_1969_the_irreducibility_of_the_space_of_curves_of_given_genus} of the stack $\mc M_{g,n}$ implies that another pointed curve $(X',\bm x')$ is an admissible deformation of $(X,\bm x)$ if and only if $\on{genus}(X') = g$ and $\abs{\bm x'} = n$:
one should now replace the discrete/countable set of pairs $(g,n) \in \mb Z_{\geq 0}^2$ with a more general set of parameters,
also governing the admissible-deformation families of irregular types/classes at the marked points.\fn{
	As above,
    hereafter the symbol $g$ typically denotes an element of $W$;
	very rarely a genus.}~(Here we use that the `global' curve $X$ is projective.)

E.g.,
in the twisted type-$A$ case this was done in~\cite{boalch_doucot_rembado_2025_twisted_local_wild_mapping_class_groups_configuration_spaces_fission_trees_and_complex_braids} by replacing the integer $n$ with a `fission forest',
gathering the isomorphism classes of fission trees---%
at the marked points---%
in a multiset.
In this text we now include structure groups $G$ of any classical type $\bullet \in \set{A,B,C,D}$,
using the aforementioned $G$-fission trees of \S\S~\ref{sec:BC_trees}--\ref{sec:D_trees} (cf.~Rmkk.~\ref{rmk:forest_bc} +~\ref{rmk:forest_d}).

\subsubsection{}

But moreover,
if $G$ is an \emph{arbitrary} connected reductive group,
then we also prove a general local result on---%
the $r$-fold covering of---%
a formal punctured disc,
giving definitions in the language of \emph{reflection cosets} à la Broué--Malle--Michel (cf.~\cite{broue_malle_1992_theoremes_de_sylow_generiques_pour_les_groupes_reductifs_sur_les_corps_finis,
	broue_malle_1993_zyklotomische_heckealgebren,
	broue_malle_michel_1993_generic_blocks_of_finite_reductive_groups,
	malle_1998_spetses,
	broue_malle_michel_1999_towards_spetses_i,
	malle_2000_on_the_generic_degrees_of_cyclotomic_algebras,
	broue_malle_michel_2014_split_spetses_for_primitive_reflection_groups}).

To state it,
let $r,s \geq 1$ be integers,
and consider again the rational reflection group $\mb G = (\mf t,W)$ determined by the Weyl group $W$ of $(\mf g,\mf t)$.
The latter plays the prominent role here,
and we introduce a notion of
\begin{center}
	$r$-\emph{Galois-closed Levi-subcoset filtrations of} $\mb G$,
	\emph{of depth} $s$,
\end{center}
which are nested sequences of Levi reflection subcosets satisfying a recursive flat-density condition (cf.~Def.~\ref{def:levi_subcoset_filtrations} +~Lem.-Def.~\ref{lem:twisted_cabled_weyl}~(4.)).
Denote their $W$-set by $\mc L_r(\mb G) = \mc L_r^{\mathsmaller\leq s}(\mb G)$,
and the corresponding quotient set by $\ol{\mc L}_r(\mb G) \ceqq \mc L_r(\mb G) \bs W$.

In addition,
let $\wh Q$ be an $r$-Galois-closed irregular type \emph{of pole order bounded by} $s$,
with underlying irregular class $\wh\Theta \ceqq \wh\Theta \bigl( \wh Q \bigr)$.
Consider then the whole $W$-set of \emph{pure} $r$-admissible-deformation spaces of such irregular types,
in the notation of \S~\ref{sec:general_notation_def_spaces};
i.e.,
the set
\begin{equation}
	\bm B_r(\mb G)
	= \bm B_r^{\mathsmaller\leq s}(\mb G)
	\ceqq \Set{ \bm B_r \bigl( \wh Q \bigr) }_{\!\wh Q},
\end{equation}
as well as the set of \emph{full/nonpure} $r$-admissible-deformation spaces of their irregular classes,
viz.,
\begin{equation}
	\ol{\bm B}_r(\mb G)
	= \ol{\bm B}_r^{\mathsmaller\leq s}(\mb G)
	\ceqq \Set{ \bm B_r \bigl( \wh\Theta \bigr) }_{\!\wh\Theta}.
\end{equation}
These are the quotient sets with respect to the equivalence relations $\sim_r$ of Deff.~\ref{def:twisted_pure_admissible_deformations}~(2.) +~\ref{def:twisted_admissible_deformation_full},
respectively (and~\eqref{eq:strata_permutation} yields the $W$-action on the former).

Then our results yield in particular a proof of the following:

\begin{coro}[cf.~\S~\ref{sec:modular_groups_are_twisted_cabled_weyl} +~Rmk.~\ref{rmk:about_wmcgs}]
	\label{cor:irregular_moduli}
	\leavevmode

	\begin{enumerate}
		\item
		      The irregular type $\wh Q$ (resp.,
		      the irregular class $\wh\Theta$) \emph{canonically} determines an element $\ul{\mb L} = \ul{\mb L} \bigl( \wh Q \bigr) \in \mc L_r(\mb G)$ (resp.,
		      an element $\ol{\ul{\mb L}} = \ol{\ul{\mb L}} \bigl( \wh\Theta \bigr) \in \ol{\mc L}_r (\mb G)$).

		\item
		      Mapping $\wh Q \mt \ul{\mb L}$ (resp.,
		      mapping $\wh\Theta \mt \ol{\ul{\mb L}}$) induces a $W$-\emph{equivariant bijection} $\bm B_r(\mb G) \lxra{\simeq} \mc L_r(\mb G)$ (resp.,
		      a \emph{bijection} $\ol{\bm B}_r(\mb G) \lxra{\simeq} \ol{\mc L}_r(\mb G)$).

		\item
		      The WWG $W_r \bigl( \wh Q \bigr)$ is \emph{uniquely determined by} $\ul{\mb L}$.\fn{
			      Its $W$-conjugacy class,
			      as a subquotient of $W$,
			      is then \emph{uniquely determined} by $\ol{\ul{\mb L}}$.}

		\item
		      And the pure local WMCG $\Gamma_r \bigl( \wh Q \bigr)$ (resp.,
		      the full/nonpure version $\ol\Gamma_r \bigl( \wh\Theta \bigr)$) is also \emph{determined by} $\ul{\mb L}$ (resp.,
		      by $\ol{\ul{\mb L}}$),
		      \emph{up to basepoint-change inner automorphisms}.
	\end{enumerate}
\end{coro}

\subsubsection{}
\label{sec:topological_skeleta}

By Cor.~\ref{cor:irregular_moduli}~(3.)--(4.),
there are \emph{finitely many} isomorphism classes of local $G$-WMCGs and WWGs of given ramification/pole-order-bound;
and \emph{countably many} thereof,
in total.

Moreover,
one may now define the \emph{topological skeleton} of a wild curve~\eqref{eq:wild_riemann_surface} in terms of the multiset of Levi-subcoset-filtration orbits $\ol{\ul{\mb L}}_x$ underlying---%
the untwisted pullback of---%
the tuple of irregular classes $\Theta_x$.
This still yields \emph{countably many} skeleta for any fixed connected reductive structure group $G$.

\begin{rema}
	\label{rmk:modular_group_description_intro}

	Our `modular' $r$-Galois-closedness condition (for subquotients of $W$) can be rephrased,
	whence it provides a significant generalization of Springer's $r$-regularity condition (for elements of $W$;
	cf.~\S~\ref{sec:modular_reparameterization}).
	This involves subsets of an `unramified' WWG,
	stable under conjugation:
	one wonders whether they still consist of \emph{a single} conjugacy class (cf.~Rmk.~\ref{rmk:transitivity_question},
	which relates to the question of~\cite{goresky_kottwitz_macpherson_2009_codimensions_of_root_valuation_strata},
	and to its partial answer~\cite{sabitova_2012_a_simplification_of_root_valuation_data_for_classical_groups}).

	Incidentally,
	note that our setup leads to a much larger class of groups than centralizers of $r$-regular elements.
	Notably,
	suppose that the irregular type $\wh Q$ is only stabilized by the maximal torus $T$ (as per `toral/complete' fission~\cite{boalch_2009_through_the_analytic_halo_fission_via_irregular_singularities},
	cf.~Rmk.~\ref{rmk:toral_fission}).
	Then the WWG $W_r \bigl( \wh Q \bigr)$ is actually a subgroup of $W$,
	which can be described as follows.
	There exist:
	(i) a flag of $\mf t$,
	consisting of \emph{flats of} $\mb G$ (given by the nested parabolic subgroups of $W$ underlying $\ul{\mb L} \bigl( \wh Q \bigr)$);
	and (ii) a \emph{canonically} determined element $g \in W$;
	such that one has $g' \in W_r \bigl( \wh Q \bigr)$ if and only if:
	(a) the element $g'$ preserves the flag;
	and (b) it also commutes with $g$.

	The point now is that \emph{any} $r$-torsion element $g \in W$ can arise in this way---%
	be it $r$-regular,
	or not.
	Then the general case involves elements whose $r$-th power lies in a prescribed parabolic subgroup of $W$ (cf.~Prop.~\ref{prop:modular_examples} +~Cor.~\ref{cor:complete_fission_modular_examples}).
\end{rema}

\subsubsection{}
\label{sec:bessis_intro}

Finally,
just before adding on the interior twists of~\cite{boalch_yamakawa_twisted_wild_character_varieties},
we also rephrase Bessis' braid-group lift~\cite[Thm.~12.4]{bessis_2015_finite_complex_reflection_arrangements_are_k_pi_1} of Springer theory from the viewpoint of quasigeneric local WMCGs (cf.~\S~\ref{sec:braid_springer_theory}).
Since the corresponding hyperplane complements are now interpreted as moduli spaces of formal germs of wild curves,
their universal coverings can be viewed as local pieces of \emph{wild Teichmüller spaces} (cf.~Rmk.~\ref{rmk:stacks} and~\cite[\S~4.1]{doucot_rembado_tamiozzo_moduli_spaces_of_untwisted_wild_riemann_surfaces}).

More generally,
therefore,
one might expect that these hyperplane complements should be $K(\pi,1)$ even for the \emph{nonreflection} examples,
beyond the quasigeneric case;
and interestingly,
the fairly recent~\cite[Thm.~1]{amend_deligne_roehrle_2020_on_the_k_pi_1_problem_for_restrictions_of_complex_reflection_arrangements} implies that this \emph{does} hold for classical structure groups,
covering the `exotic' type-$D$ examples~\eqref{eq:exotic_arrangement} (cf.~\cite{doucot_rembado_tamiozzo_2026_local_wild_mapping_class_groups_and_cabled_braids}).
We leave the general case as an open question,
which hinges on structure groups of exceptional type:

\begin{enonce}{Question}
	\label{question:k_pi_1}

	Let $\wt{\bm B}_r \bigl( \wh Q,i \bigr) \thra \bm B_r \bigl( \wh Q,i \bigr)$ be the universal covering of a complex hyperplane complement arising as in~\eqref{eq:factorization_intro}:
	is it true that $\wt{\bm B}_r \bigl( \wh Q,i \bigr)$ is \emph{contractible}?
\end{enonce}

(A more general analogous question can be asked in the extended setup of \S~\ref{sec:interior_twist}.)

\subsection{More main results}
\label{sec:results}

We now summarize the remaining main results,
up until \S~\ref{sec:tglwmcgs}.

\subsubsection{}

Here is a general description of the spaces of admissible deformations,
for any connected reductive $G$,
choosing---%
again---%
for convenience an $r$-th root $w$ of a uniformizer $z$ for the formal punctured disc at a marked point $x \in \bm x$:\fn
{
	Cf.~again \S~\ref{sec:uniformizers}.
	In addition,
	cf.~Rmkk.~\ref{rmk:intrinsic_flag} +~\ref{rmk:fibre_bundle_tower} about the intrinsic notions---%
	involving towers of \emph{trivializable} topological fibre bundles.}

\begin{theo}[cf.~Prop.~\ref{prop:generic_pure_twisted_deformation_space} +~Thmm.-Deff.~\ref{thm:general_pure_twisted_deformation_space} +~\ref{thm:full_general_twisted_deformation_space} +~\ref{thm:generic_full_twmcgs}]
	\label{thm:thm_1_intro}

	Let $\Theta$ be a \emph{nonvanishing} $r$-ramified irregular class for $\mf g$.
	Denote:
	(i) by $\wh{\Theta}$,
	the untwisted pullback of $\Theta$;
	(ii) by $\wh Q$,
	an irregular type with irregular class $\wh{\Theta}$;
	and (iii) by $A_s \in \mf t$,
	the leading coefficient of $\wh Q$---%
	of pole order $s = s \bigl( \wh Q \bigr) \geq 1$.
	Then:

	\begin{enumerate}
		\item
		      the space of $r$-admissible deformations of $\wh Q$ decomposes as a \emph{direct product}
		      \begin{equation}
			      \label{eq:factorization_intro}
			      \bm B_r \bigl( \wh Q \bigr)
			      = \prod_{i = 1}^s \bm B_{r} \bigl( \wh Q,i \bigr) \sse \mf t^s,
		      \end{equation}
		      where each factor is a \emph{complex hyperplane complement} (cf.~\eqref{eq:general_g_twisted_deformation_factor});

		\item
		      a \emph{subquotient} $W_r \bigl( \wh Q \bigr)$ of the Weyl group $W$ acts \emph{freely} on $\bm B_r \bigl( \wh Q \bigr)$;
		      (This is the aforementioned WWG,
		      cf.~Cor.~\ref{cor:irregular_moduli}.)

		\item
		      the full/nonpure $r$-admissible-deformation space $\bm B_r \bigl( \wh\Theta \bigr) \sse \mb A_s \ceqq \mf t^s \bs W$ is the corresponding \emph{topological quotient of} $\bm B_r \bigl( \wh Q \bigr)$;

		\item
		      and if $A_s$ is a \emph{regular} vector for $W$,
		      then:
		      \begin{enumerate}
			      \item
			            the group $W_r \bigl( \wh Q \bigr)$ is isomorphic to the centralizer $Z_W(g) \sse W$ of a \emph{regular} element $g \in W$ of order $d \ceqq r \bs (r \wdg s) \geq 1$;

			      \item
			            the leading factor $\bm B_r \bigl( \wh Q,s \bigr) \sse \mf t_{\reg}$ of~\eqref{eq:factorization_intro} is the $Z_W(g)$-regular part of the eigenspace of $g$ with eigenvalue $\zeta \ceqq e^{2 \pi \sqrt{-1} \cdot s \slash r} \in \mb C^{\ts}$ (à la Springer,
			            cf.~\eqref{eq:generic_twisted_deformation_space_1_coeff});

			      \item
			            and $\ol \Gamma_r \bigl( \wh\Theta \bigr)$ is isomorphic to Bessis' lift of $Z_W(g)$ in the full $\mf g$-braid group.
		      \end{enumerate}
	\end{enumerate}
	(In general,
	the pure local WMCG also splits as a direct product,
	and again the full version is an extension of $W_r \bigl( \wh Q \bigr)$ by the pure one:
	cf.~Thm.~\ref{thm:wild_mapping_class_groups}.)
\end{theo}

\subsubsection{}

For the classical Lie algebras,
one can refine the factorization~\eqref{eq:factorization_intro} of $\bm B_r \bigl( \wh Q \bigr)$ into a different product of \emph{irreducible} complex hyperplane complements,
and list \emph{all} the factors which can arise.

Precisely,
choose integers $\rho,k,k' \geq 1$.
Consider tuples $\bm \lambda = (\lambda_1,\dc,\lambda_k) \in \mb C^k$ and $\bm \lambda' = (\lambda'_1,\dc,\lambda'_{k'}) \in \mb C^{k'}$,
and define:
\begin{equation}
	\label{eq:standard_complement_type_D}
	\mc M(\rho,k) \ceqq
	\Set{ \bm \lambda \in \mb C^k | \lambda_i^\rho \neq \lambda_j^\rho, \quad i \neq j \in \set{1,\dc,k} },
\end{equation}
which is the reflection complement of $G(\rho,e,k)$ with $e = \rho$;
and
\begin{equation}
	\label{eq:standard_complement}
	\mc M^\sharp (\rho,k) \ceqq
	\Set{ \bm \lambda \in \mc M(\rho,k) | \lambda_i \neq 0, \quad i \in \set{1,\dc,k} },
\end{equation}
which is the reflection complement of $G(\rho,e,k)$ with $e < \rho$ (where $e \mid \rho$);
and finally
\begin{equation}
	\label{eq:exotic_arrangement}
	\mc M^{\musDoubleSharp}(k,k')
	\ceqq \Set{ (\bm \lambda,\bm \lambda') \in \mc M^\sharp(2,k) \ts \mc M(2,k') | \lambda_i^2 \neq (\lambda'_j)^2, \quad 1 \leq i \leq k, \quad 1 \leq j \leq k'  }.
\end{equation}
Cf.~\cite[Lem.~3.3]{broue_malle_rouquier_1998_complex_reflection_groups_braid_groups_hecke_algebras} for~\eqref{eq:standard_complement_type_D}--\eqref{eq:standard_complement},
and note that~\eqref{eq:exotic_arrangement} is typically \emph{not} the complement of the hyperplane arrangement of a complex reflection group,
whence it is in particular \emph{noncrystallographic}.
(Nonetheless,
as already mentioned,
its universal covering is contractible;
cf.~Prop.~\ref{prop:nonreflection_arrangement}.)

Then the classification goes as follows:

\begin{theo}[cf.~\S\S~\ref{sec:type_A}--\ref{sec:type_D} +~\ref{sec:D_trees}]
	\label{thm:thm_2_intro}

	In the notation of Thm.~\ref{thm:thm_1_intro},
	suppose in particular that $\mf g$ is a \emph{classical} simple Lie algebra.
	Then:
	\begin{enumerate}
		\item
		      there exist three non-negative integer-valued functions,
		      denoted by $\mu = \mu_{\wh Q}(\rho,k)$,
		      $\mu^\sharp = \mu^\sharp_{\wh Q}(\rho,k)$,
		      and $\mu^{\musDoubleSharp} = \mu^{\musDoubleSharp}_{\wh Q}(k,k')$,
		      together with a \emph{homeomorphism}
		      \begin{equation}
			      \label{eq:finer_decomposition}
			      \bm B_r \bigl( \wh Q \bigr)
			      \simeq \prod_{\rho,k \geq 1} \Bigl( \mc M(\rho,k)^\mu \ts \mc M^\sharp(\rho,k)^{\mu^\sharp} \Bigr) \ts \prod_{k,k' \geq 1}  \mc M^{\musDoubleSharp}(k,k')^{\mu^{\musDoubleSharp}},
		      \end{equation}
		      endowing the right-hand side with the product topology;

		\item
		      moreover,
		      these functions satisfy the following conditions:
		      \begin{enumerate}
			      \item
			            one has $\mu(\rho,\bullet) = 0$ for $\rho > 2$,
			            and $\mu(2,\bullet) \neq 0$ only in type $D$;

			      \item
			            one has $\mu^{\musDoubleSharp} \leq 1$,
			            and $\mu^{\musDoubleSharp} \neq 0$ only in type $D$;
			            in that case,
			            one has $\mu^{\musDoubleSharp}(k,k') = 1$ for a \emph{unique} pair $(k,k')$.
		      \end{enumerate}
	\end{enumerate}
\end{theo}

\subsubsection{}

Explicitly,
if $\mf g$ is of type $\bullet \in \set{A,B,C}$ then the factorization reads
\begin{equation}
	\bm B_r \bigl( \wh Q \bigr) \simeq \prod_{k \geq 1} \Bigl( \mc M(1,k)^\mu \Bigr) \ts \prod_{\rho,k \geq 1} \mc M^\sharp(\rho,k)^{\mu^\sharp},
\end{equation}
for suitable functions $\mu = \mu(1,k)$ and $\mu^\sharp = \mu^\sharp(\rho,k)$.
And if $\mf g$ is of type $\bullet = D$ then
\begin{equation}
	\bm B_r \bigl( \wh Q \bigr) \simeq \prod_{\rho = 1}^2 \Bigl( \prod_{k \geq 1} \mc M(\rho,k)^\mu \Bigr) \ts \prod_{\rho,k \geq 1} \Bigl( \mc M^\sharp(\rho,k)^{\mu^\sharp} \Bigr) \ts \mc M^{\musDoubleSharp}(k,k')^{\mu^{\musDoubleSharp}},
\end{equation}
for a \emph{unique} suitable pair $(k,k')$,
where $\mu^{\musDoubleSharp} = \mu^{\musDoubleSharp}(k,k') \in \set{0,1}$.
(The usual convention is that $\bm B^0$ is the point-space,
for any topological space $\bm B$.)

\subsubsection{}
\label{sec:fission_trees_intro}

Suppose again that $W$ is of type $\bullet \in \set{A,B\!C,D}$.
In this case,
to get more explicit descriptions of local WMCGs it is convenient---%
and w.l.o.g.---%
to work with irregular classes $\Theta$ of compatible \emph{pointed} irregular types $\dot Q$,
which are particular Galois-closed lists of exponential factors (cf.~\S\S~\ref{sec:notation_for_trees}--\ref{sec:D_trees},
and~\cite{boalch_doucot_rembado_2025_twisted_local_wild_mapping_class_groups_configuration_spaces_fission_trees_and_complex_braids} in type $A$).
One can then spell out the type-dependent notion of admissible deformations,
getting \emph{configuration spaces} $\bm B_{\bullet,r} \bigl( \Theta \bigr)$ and $\bm B_{\bullet,r} (\dot Q)$ which are identified with particular cases of the main definitions.\fn{
	Type $D$ is more convoluted,
	and it requires introducing suitable weaker notions of irregular classes and pointed irregular types (cf.~\S\S~\ref{sec:pseudo_irr_classes}--\ref{sec:pointed_types_D}).}

Then the main new combinatorial tools are:
(i) the \emph{fission tree} $\mc T = \mc T(\Theta)$ of the irregular class;
and (ii) the \emph{labelled fission tree} $\dot{\mc T}  = \mc T(\dot Q)$ of the irregular type (cf.~Deff.~\ref{def:fission_tree_type_BC} +~\ref{def:fission_tree_type_D}).
These decorated rooted trees control the topology of the configuration spaces.
Notably:
(a) one can factor the pure deformation spaces along the vertices of the labelled tree;
and (b) one can associate a \emph{Weyl group} $W_\bullet(\dot{\mc T})$ to it (cf.~Deff.~\ref{def:weyl_group_tree_type_BC} +~\ref{def:weyl_group_tree_type_D}).
The latter provides an explicit description of the corresponding $r$-ramified WWG,
in the sense of \S~\ref{sec:modular_groups_are_twisted_cabled_weyl} (cf.~Cor.~\ref{cor:weyl_groups_are_the_same} +~\S~\ref{sec:weyl_groups_are_the_same_D}),
and so it governs the passage to the full/nonpure case.

Precisely:

\begin{theo}[cf.~\S\S~\ref{sec:notation_for_trees}--\ref{sec:D_trees}]
	\label{thm:thm_3_intro}

	Let $\Theta$ be the $r$-ramified irregular class of a nonvanishing compatible pointed irregular type $\dot Q$,
	of \emph{classical type} $\bullet \in \set{A,B\!C,D}$.
	Denote by $\mc T$ (resp.,
	by $\dot{\mc T}$) the fission tree of $\Theta$ (resp.,
	the \emph{labelled} fission tree of $\dot Q$),
	with set of vertices $\mb V$.
	Then:
	\begin{enumerate}
		\item
		      the isomorphism class of $\mc T$ (resp.,
		      of $\dot{\mc T}$) is a \emph{complete invariant} for the admissible deformations of $\Theta$ (resp.,
		      of $\dot Q$);

		\item
		      there is a homeomorphism $\bm B_{\bullet,r}(\dot Q) \simeq \bm B_\bullet(\dot{\mc T})$,
		      invoking the \emph{configuration space} of the tree (cf.~\eqref{eq:config_space_from_BC_tree} +~\eqref{eq:configuration_space_of_tree_d});

		\item
		      there is a topological factorization
		      \begin{equation}
			      \bm B_\bullet (\dot{\mc T})
			      = \prod_{v \in \mb V} \bm B_\bullet (\dot{\mc T},v),
		      \end{equation}
		      which leads to~\eqref{eq:finer_decomposition},
		      invoking the \emph{local configuration spaces} of the tree (cf.~\eqref{eq:local_conf_space_bc} +~\eqref{eq:local_conf_space_D});

		\item
		      the Weyl group $W_\bullet(\dot{\mc T})$ of the labelled tree acts freely on $\bm B_{\bullet,r}(\dot Q)$;

		\item
		      the space $\bm B_{\bullet,r} \bigl( \Theta \bigr)$ is the corresponding topological quotient of $\bm B_{\bullet,r}(\dot Q)$;

		\item
		      and there is a group isomorphism $W_\bullet(\dot{\mc T}) \simeq W_r \bigl( \wh Q \bigr)$,
		      where $\wh Q$ is the pullback-untwisted irregular type corresponding to $\dot Q$ (cf.~Thm.~\ref{thm:thm_1_intro}~(2.)).
	\end{enumerate}
	(The consequences for these `classical' local WMCGs are spelt out in Thm.~\ref{thm:classical_wmcgs}.)
\end{theo}

\subsection{Interior twists}
\label{sec:results_2}

\subsubsection{}

We also extend the whole setup by adding on the choice of an \emph{outer} automorphism of $G$ preserving the maximal torus $T \sse G$,
in order to allow for the aforementioned interior twists on the `open' curve $X^o \sse X$ (cf.~Rmk.~\ref{rmk:interior_twists}).
After taking tangent maps,
this involves the whole automorphism group of the root system $\Phi \sse \mf t^{\dual}$ of $(\mf g,\mf t)$ (cf.~\S~\ref{sec:root_automorphisms}),
which now plays a role in meromorphic reductive gauge theory on curves.

Concretely,
we consider \emph{nonsplit} reflection cosets $\wt{\mb G} = (\mf t,\dot\varphi W)$ of the Weyl group $W$,
where $\dot\varphi \in \GL_{\mb C}(\mf t)$ is---%
the restriction of---%
a finite-order automorphism of $\mf g$ preserving $\mf t$ (normalizing $W$,
cf.~\S\S~\ref{sec:howlett}--\ref{sec:lehrer_springer_theory} +~\ref{sec:reflection_cosets}).
They lead to a more general `twisted' notion of $r$-Galois-closedness for irregular types/classes,
which we treat by further use of \emph{nonsplit} Springer/Lehrer--Springer theory~\cite[\S~6]{springer_1974_regular_elements_of_finite_reflection_groups} (cf.~\cite{lehrer_springer_1999_intersection_multiplicities_and_reflection_subquotients_of_unitary_reflection_groups_i,
	lehrer_springer_1999_reflection_subquotients_of_unitary_reflection_groups,bonnafe_lehrer_michel_2006_twisted_invariant_theory_for_reflection_groups},
    amongst many others referenced in the body of the paper).

In particular,
we prove that the statements of Thm.~\ref{thm:thm_1_intro} generalize,
apart from the quasigeneric braid-group lift:
note that establishing the analogue of Thm.~\ref{thm:thm_1_intro}~(4.c)~would entail `braid' \emph{nonsplit} Lehrer--Springer theory,
which seems to be an important open problem (cf.~\cite[Thm.-Conj.~1.58]{broue_malle_michel_2014_split_spetses_for_primitive_reflection_groups}).

We also establish that the topological problem---%
still---%
reduces to having \emph{simple} Lie algebras (passing through $\dot\varphi$-\emph{simple} ones,
cf.~\S~\ref{sec:reduction_simple_case_pure_doubly_twisted}).
And then we give the ultimate definition of twisted local $G$-WMCGs,
stating their general properties (cf.~\S~\ref{sec:doubly_twisted_lwmgc}):
the most general pure factorization is in~\eqref{eq:most_general_pure_group_factorization},
and the most general exact group sequence in~\eqref{eq:most_general_exact_group_sequence}.

\subsubsection{}

Moreover,
we still describe the pure quasigeneric examples,
extending the previous classification,
as follows:

\begin{coro}[cf.~Prop.~\ref{prop:twisted_classification_type_A} +~\S~\ref{sec:no_doubly_twist_BC} +~Propp.~\ref{prop:twisted_generic_classification_type_D} +~\ref{prop:triality_classification}]
	\label{cor:modular_reflection_groups_twisted}

	Keep all hypotheses/notations from Cor.~\ref{cor:modular_reflection_groups},
	but allow now $\Theta$ to also be \emph{twisted by} a suitable element $\dot\varphi \in N_{\GL_{\mb C}(\mf t)}(W)$---%
	as above.
	Then:
	\begin{enumerate}
		\item
		      the statement of Cor.~\ref{cor:modular_reflection_groups}~(1.) extends;

\item
	      and moreover,
	      letting the pair $(\dot\varphi,r)$ vary,
	      these complex reflection groups are classified as follows:
	      \begin{enumerate}
		      \item
		            if $\mf g$ is of \emph{classical} type $\bullet \in \set{A,B,C,D}$,
		            then $W'$ is:
		            \begin{enumerate}
			            \item
			                  either as in Cor.~\ref{cor:modular_reflection_groups}~(2.a.i),
			                  still exhausting the examples coming from $\bullet \in \set{A,B,C}$;

			            \item
			                  or an index-$2$ subgroup $G(\rho,2,k) \sse G(\rho,1,k)$ ($\rho$ even),
			                  only arising if $\bullet = D$,
			                  now with \emph{no} parity constraint between $\rho$ and $k$;

			            \item
			                  or,
                              			                  only arising from the \emph{triality} coset when $\bullet = D$ in rank four:
                                                          \begin{enumerate}
                                                              \item
                                                                  the \emph{imprimitive} rational reflection group $G(6,6,2)$;

                                                              \item
                                                                                                  and the \emph{primitive} (exceptional) complex reflection group $G_4$.
                                                          \end{enumerate}
		            \end{enumerate}
                          \item 
          else,
          if $\mf g$ is of \emph{exceptional} type $\bullet \in \set{E,F,G}$,
          but \emph{excluding} $E_6$,
          then the classification remains as in Cor.~\ref{cor:modular_reflection_groups}~(2.b).
	      \end{enumerate}
  \end{enumerate}
\end{coro}

\subsubsection{}

Therefore,
the unique new primitive example is the well-generated group $G_4$ (cf.~the first row of Tab.~\ref{tab:springer_from_weyl}),
now coming from a classical ambient Lie algebra---%
as $G(6,6,2)$ is the Weyl group of type $G_2$.\fn{
    The order-$4$ cyclic group of Prop.~\ref{prop:triality_classification} is accounted for by Cor.~\ref{cor:modular_reflection_groups_twisted}~(2.a.i).
Incidentally,
triality is also briefly treated in \S\S~\ref{sec:classical_howlett_twists} +~\ref{sec:type_D_relative_refl_group}.}

As for the exceptional types,
the classification cannot extend in absence of nontrivial Dynkin-diagram automorphisms (cf.~Rmk.~\ref{rmk:about_regular_numbers_twisted}):
in this text we do \emph{not} consider the $E_6$ flip.

\subsubsection{}

Finally,
the exact nonsplit analogue of Cor.~\ref{cor:irregular_moduli} holds true,
upon extending the notion of Levi-subcoset filtrations to include nested sequences starting at $\wt{\mb G}$ (cf.~\S~\ref{sec:twisted_modular_groups_and_cabled_weyl_groups}):
in brief,
the whole of fission takes place in a \emph{fixed} (rational) reflection coset of $\mb G = (\mf t,W)$.

One may now define topological skeleta of arbitrary \emph{interior-twisted} wild curves purely in terms of---%
discrete---%
reflection-theoretic data,
exactly as in the previous case,
noting that the skeleta of \S~\ref{sec:topological_skeleta} are given by Levi-subcoset filtrations and their orbits:
this does \emph{not} require the ambient reflection coset to split.

In turn,
the transitivity question of Rmk.~\ref{rmk:modular_group_description_intro} extends by involving conjugacy classes inside `wild Weyl cosets',
and this provides a fairly vast generalization of Lehrer--Springer's \emph{nonsplit} $r$-regularity condition (cf.~\S~\ref{sec:twisted_modular_reparameterization}).
Again,
when fission is toral/complete there is a description of the corresponding WWG as the subgroup of elements $g' \in W$ which centralize a given $r$-torsion element $\bm g \in \dot\varphi W \sse \braket{ \dot\varphi } \lts W$,
and which moreover preserve a suitable flag of flats for $\mb G$.
(Again,
\emph{any} torsion element $\bm g$ can arise;
cf.~\eqref{eq:most_general_wwg} for the general case.)

\subsection{Layout (I)}
\label{sec:layout_1}

The body of the article is split into \emph{nine} interconnected parts,
separated by \emph{eight} `interlude' minisections,
as follows.

\subsubsection{}
\label{sec:layout_1_1}

In \S~\ref{sec:setup} we present the general setup without interior twists,
and define admissible deformations of $r$-Galois-closed irregular types/classes.

In \S\S~\ref{sec:generic_pure_case}--\ref{sec:pure_nongeneric_case} we treat the pure---%
marked---%
case,
which we relate to Springer theory.
Then in \S~\ref{sec:no_marking} we get rid of the choice of marking.

\subsubsection{}

In \S\S~\ref{sec:nonpure_case_generic}--\ref{sec:nonpure_case_general} we treat the full/nonpure case.
(Experts might want to start from~Thmm.-Deff.~\ref{thm:general_pure_twisted_deformation_space} +~\ref{thm:full_general_twisted_deformation_space}.)

\subsubsection{}

In \S~\ref{sec:howlett} we link the previous constructions to:
(i) relative Weyl groups;
(ii) normalizers of parabolic subgroups of Weyl groups;
and (iii) Levi/Howlett reflection cosets.

\subsubsection{}

In \S~\ref{sec:modular_groups_are_twisted_cabled_weyl} we expand upon the relation with (sequences of) Levi subcosets,
and define the modular/$r$-ramified WWGs;
then in \S~\ref{sec:modular_reparameterization} we rephrase our modularity condition as a generalization of Springer's regularity.

\subsubsection{}

In \S~\ref{sec:lehrer_springer_theory} we introduce (non-Howlett) relative reflection groups,
aiming to describe the complex-reflection arrangements which arise from the admissible-deformation spaces when $W$ is of classical type.

\subsubsection{}

In \S\S~\ref{sec:type_A}--\ref{sec:type_D} we provide such descriptions for the classical Lie algebras,
singling out (and excluding) the type-$D$ nonreflection obstruction.

\subsubsection{}

In \S~\ref{sec:notation_for_trees} we recall additional notions/notations in order to define and study $G$-fission trees in the subsequent \S\S~\ref{sec:BC_trees}--\ref{sec:D_trees},
for classical simple groups $G$ of type $\bullet \in \set{B,C,D}$.
(Treating type $D$ in full generality.)

\subsubsection{}

In \S~\ref{sec:tglwmcgs} we (re)introduce local $G$-WMCGs,
we point out that lifted Springer theory describes their quasigeneric examples,
and then we specialize to classical types.

\subsubsection{}
\label{sec:layout_1_4}

Finally,
in \S~\ref{sec:interior_twist} we allow for twists in the interior of the curve,
in the language of nonsplit reflection cosets.
(Experts might want to start from Thm.-Def.~\ref{thm:full_general_doubly_twisted_deformation_space}.)

\subsection{Layout (II)}
\label{sec:layout_2}

There are four appendices:
up to fairly standard notions/arguments,
they should make this text essentially self-contained.

\subsubsection{}

App.~\ref{sec:background} contains rather detailed textbook background on:
(i) complex reflection groups;
(ii) reflection cosets;
and (iii) classical Weyl groups/root systems,
putting forward the viewpoint of \emph{signed} indices for the roots of type $\bullet \in \set{B,C,D}$.

App.~\ref{sec:exceptional_types} gathers some material about the exceptional simple cases,
notably about the classification of Cor.~\ref{cor:modular_reflection_groups}~(2.b).

App.~\ref{sec:bessis_literature} gives minimal references about the $K(\pi,1)$ property of finite complex-reflection arrangements.

And in App.~\ref{sec:missing_proofs} we postpone several proofs---%
aiming to ease the reading flow.

\subsection{Notations/conventions}
\label{sec:notations}

\begin{itemize}
	\item
        If $J = \set{d_1,\dc,d_p} \sse \mb Z_{> 0}$ is a finite set of $p \geq 1$ positive integers,
	      we denote the \emph{greatest common divisor} of the elements of $J$ by
	      \begin{equation}
		      \bigwedge J = \bigwedge_{i = 1}^p d_i,
	      \end{equation}
	      and their \emph{least common multiple} by
	      \begin{equation}
		      \bigvee J = \bigvee_{i = 1}^p d_i.
	      \end{equation}

	\item
	      If a rational number $k \in \mb Q_{> 0}$ is expressed as a fraction $k = \frac n d$,
	      where $n,d > 0$ are coprime integers,
	      then we just say that $d \ceqq \on{den}(k)$ (resp.,
	      $n \ceqq \on{num}(k)$) is the \emph{denominator of} $k$ (resp.,
	      its \emph{numerator}).

	\item
	      We fix once and for all a square root $\sqrt{-1}$ of $-1$ in $\mb C$,
          keeping the letter $i$ as an index;
	      given an integer $r \geq 1$,
	      we let $\zeta_r \ceqq e^{2\pi \sqrt{-1} \slash r} \in \mb C^{\ts}$,
	      which is a primitive $r$-th root of $1$;
	      we just denote by $\mu_r = \mu_r(\mb C) \sse \mb C^{\ts}$ the group of $r$-th roots of $1$ in $\mb C$ (without disambiguating from the algebraic subgroup of $\mb G_m$);
	      using $\zeta_r$,
	      we identify $\mu_r$ with the cyclic group $\mb Z \bs r\mb Z$.

    \item 
        If $V$ is a finite-dimensional complex vector space,
        by a \emph{hyperplane of} $V$ we mean a vector subspace $H \sse V$ of complex codimension $1$ (i.e.,
        never the affine version);
        we identify the set of hyperplanes of $V$ with the projectivized dual $\mb P(V^{\dual}) \ceqq \bigl( V^{\dual} \sm \set{0} \bigr) \bs \mb C^{\ts}$,
        mapping $\mb C \cdot \alpha \mt \ker(\alpha)$;
        if $V$ and $V'$ are finite-dimensional complex vector spaces,
        and if $\mc A \sse \mb P(V^{\dual})$ and $\mc A' \sse \mb P\bigl( (V')^{\dual} \bigr)$ are finite sets of linear hyperplanes in $V$ and $V'$ (respectively),
        we often consider the \emph{complex hyperplane complements} $\bm B \ceqq V \sm \bigcup_{\mc A} H \sse V$ and $\bm B' \ceqq V' \sm \bigcup_{\mc A'} H \sse V'$;
        a $\mb C$-\emph{linear isomorphism of} $\bm B$ \emph{and} $\bm B'$ is a $\mb C$-linear isomorphism $V \lxra{\simeq} V'$ inducing a bijection $\mc A \lxra{\simeq} \mc A'$;
        the latter induces a \emph{homeomorphism} $\bm B \lxra{\simeq} \bm B'$,
        viewing the complements as topological subspaces for the canonical structure of topological vector spaces of $V$ and $V'$.

	\item
	      The \emph{space} of nonzero complex numbers is also denoted by $\mb C^{\ast}$---%
	      when thinking topologically,
	      rather than algebraically.

	\item
	      The end of a remark/example is signalled by a $\diamondsuit$,
	      inspired by~\cite{kirillov_2004_lectures_on_the_orbit_method}.
\end{itemize}

\section{Setup and main definitions}
\label{sec:setup}

\subsection{Galois-closed irregular types/classes}

Here we review the basic terminology of twisted irregular types/classes,
phrased on top of the cyclic covering of a formal punctured disc.

\subsubsection{}

Let $G$ be a connected reductive algebraic group over $\mb C$,
with Lie algebra $\mf g \ceqq \Lie(G)$.
Choose a maximal torus $T \sse G$,
and denote by $\mf t \ceqq \Lie(T) \sse \mf g$ the associated Cartan subalgebra.
Consider a variable $w$,
and the field of $\mb C$-valued formal Laurent series $\mc L \ceqq \mb C(\!( w )\!)$.
The standard valuation $v \cl \mc L \to \mb Z \cup \set{ +\infty }$ determines the complete DVR of formal power series,
i.e.,
the valuation subring
\begin{equation}
	\mc L_{\geq 0}
	= v^{-1} ( \ol{\mb Z}_{\geq 0} )
	= \mb C\llb w \rrb,
	\qquad \ol{\mb Z}_{\geq 0}
	\ceqq \mb Z_{\geq 0} \cup \set{ + \infty }.
\end{equation}
The corresponding vector space of untwisted/unramified \emph{irregular types} is
\begin{equation}
	\label{eq:irregular_types}
	\wh{\mc{IT}}
	= \wh{\mc{IT}}(\mf t,w)
	\ceqq \mf t \ots_{\mb C} ( \mc L \bs \mc L_{\geq 0})
	\simeq \mf t (\!( w )\!) \bs \mf t \llb w \rrb.
\end{equation}
One can view irregular types as $\mf t$-valued polynomials in $w^{-1}$ with vanishing constant term,
in the $\mb C$-linear isomorphism $\mc L \bs \mc L_{\geq 0} \simeq w^{-1} \mb C[w^{-1}]$.

\begin{rema}
	\label{rmk:standard_irr_types}

	Strictly speaking,
	the space~\eqref{eq:irregular_types} contains `standard' irregular types,
	as we work over \emph{standard} discs---%
	rather than spectra of arbitrary complete DVRs/DVFs over $\mb C$.
	This is tacit in all that follows,
	where we \emph{describe} the admissible-deformation spaces using the convenient choice of uniformizer (or rather,
	a root thereof);
	but their \emph{definition} does not require it (cf.~\S~\ref{sec:uniformizers}).
\end{rema}

\subsubsection{}

Now let the group of $r$-th roots of $1$ act $\mb C$-linearly on irregular types,
via
\begin{equation}
	\sigma \bigl( \wh Q(w) \bigr)
	= \sigma_r \bigl( \wh Q(w) \bigr)
	\ceqq \wh Q(\zeta_r^{-1} w) \in \wh{\mc{IT}},
\end{equation}
which extends to
\begin{equation}
	\label{eq:monodromy}
	\sigma^j \cl \wh Q(w)
	\lmt \wh Q \bigl( \zeta^{-j}_r w \bigr),
	\qquad j \in \mb Z \bs r \mb Z.
\end{equation}
(Cf.~\eqref{eq:explicit_irregular_type}--\eqref{eq:explicit_global_r_galois_orbit} for explicit expressions.)

\begin{defi}
	\label{def:r_galois_orbit}

	The $r$-\emph{Galois-orbit of} $\wh Q$ is the finite subset
	\begin{equation}
		\label{eq:global_r_galois_orbit}
		\braket{ \wh Q }_r
		\ceqq \Set{ \sigma^j \bigl( \wh Q \bigr) | j \in \mb Z \bs r \mb Z } \sse \wh{\mc{IT}}.
	\end{equation}
\end{defi}

\subsubsection{}

Denote by $W = W(\mf g,\mf t) \simeq N_G(T) \bs T$ the Weyl group of $(\mf g,\mf t)$.
It acts on $\mf t^{\dual} $ and $\mf t$ in the standard way,
whence on irregular types (cf.~\S~\ref{sec:complex_real_rational_reflection_groups}).
The \emph{irregular class of} $\wh Q \in \wh{\mc{IT}}$ is the $W$-orbit
\begin{equation}
	\wh\Theta
	= \wh\Theta \bigl( \wh Q \bigr)
	\ceqq W.\wh Q \in \wh{\mc{IT}} \bs W.
\end{equation}

Now the $W$-action commutes with~\eqref{eq:monodromy},
since they are on different tensor-factors of~\eqref{eq:irregular_types}.
So there is a well-induced cyclic action on irregular classes:
\begin{equation}
	\label{eq:reduced_monodromy}
	\sigma^j \cl \wh\Theta \bigl( \wh Q(w) \bigr)
	\lmt \wh\Theta \bigl( \wh Q(\zeta_r^{-j} w) \bigr) \in \wh{\mc{IT}} \bs W,
	\qquad j \in \mb Z \bs r \mb Z.
\end{equation}

\begin{defi}[cf.~\cite{boalch_doucot_rembado_2025_twisted_local_wild_mapping_class_groups_configuration_spaces_fission_trees_and_complex_braids},
		\S~2.1.4]
	\label{def:galois_closed_irr_types}

	Choose an integer $r \geq 1$.
	Then:
	\begin{enumerate}
		\item
		      an irregular class $\wh\Theta$ is $r$-\emph{Galois-closed} if it is invariant under~\eqref{eq:reduced_monodromy};

		\item
		      an irregular type $\wh Q$ is $r$-\emph{Galois-closed} if this holds for its irregular class,
		      i.e.,
		      if one has $\braket{\wh Q}_{\! r} \sse \wh\Theta \bigl( \wh Q \bigr)$;

		\item
		      and the---%
		      $W$-invariant---%
		      subset of $r$-Galois-closed irregular types is denoted by $\wh{\mc{IT}}_r \sse \wh{\mc{IT}}$.
	\end{enumerate}
\end{defi}

\begin{rema}
	\label{rmk:relation_with_twist}

	Introduce the `descended' variable $z \ceqq w^r$.
	Then $w$ is a root of the monic polynomial
	\begin{equation}
		P
		\ceqq (Z^r - z) \in \mc K [Z],
		\qquad \mc K \ceqq \mb C (\!( z )\!),
	\end{equation}
	and the splitting field of $P$---%
	over $\mc K$---%
	can be identified with $\mc L$.
	The Galois group is cyclic,
	of order $r$,
	generated by the $\mc K$-linear automorphism of $\mc L$ mapping $w \mt \zeta_r^{-1} w$ (cf.~\cite[Chp.~XIII, \S~2]{serre_1979_local_fields};
	as in~\eqref{eq:monodromy},
	here it is convenient to take inverses).
	Thus,
	an $r$-Galois-closed irregular type $\wh Q = \wh Q(w)$ yields an $r$-ramified irregular type $Q = Q(z)$,
	as in~\cite[p.~10]{boalch_yamakawa_twisted_wild_character_varieties} (cf.~also App.~A of op.~cit.).
	In turn,
	this determines an $r$-ramified irregular class $\Theta = \Theta(Q)$;
	conversely,
	all the $r$-ramified irregular types/classes arise in this way.

	Precisely,
	in the present setting/convention,
	note that Eq.~(13) of~\cite{boalch_yamakawa_twisted_wild_character_varieties} reads
	\begin{equation}
		\label{eq:monodromy_realization_lifted}
		\Ad_{\wt g'} \bigl( \wh Q (w) \bigr)
		= \wh Q(\zeta^{-1}_r w),
		\qquad \wh Q \in \wh{\mc{IT}},
	\end{equation}
	for a suitable element $\wt g' \in G$ (cf.~\cite[Eq.~(2.9)]{balser_jurkat_lutz_1979_a_general_theory_of_invariants_for_meromorphic_differential_equations_i_formal_invariants} in type $A$,
	and~\cite[Eq.~(4.4.1)]{goresky_kottwitz_macpherson_2009_codimensions_of_root_valuation_strata} for general $G$).
	Then~\eqref{eq:monodromy_realization_lifted} implies that $\wt g'$ normalizes the connected reductive algebraic subgroup $L \ceqq G^{\wh Q} \sse G$ which stabilizes---%
	the coefficients of---%
	the irregular type $\wh Q$ under the Adjoint action $\Ad \cl G \to \Aut(\mf g)$ (cf.~Lem.~\ref{lem:full_twist_preserves_centralizer} for the general case of~\cite[Eq.~(13)]{boalch_yamakawa_twisted_wild_character_varieties}).
	It follows that an $L$-translate $\wt g \in N_G(L)$ of $\wt g'$ normalizes the maximal torus $T \sse L$,
	since the maximal tori of $L$ are conjugate (cf.~\cite{humphreys_1975_linear_algebraic_groups,borel_1991_linear_algebraic_groups,springer_1998_linear_algebraic_groups}).
	Finally,
	the class $g \in W$ of $\wt g$ acts on $\wh Q$ as in~\eqref{eq:monodromy_realization},
	and it yields an $r$-Galois-closed irregular type as in Def.~\ref{def:galois_closed_irr_types};
	this motivates Def.~\ref{def:twisted_pure_admissible_deformations}.
	(Cf.~also Lem.-Def.~\ref{lem:relative_weyl_group} +~Rmk.~\ref{rmk:no_ramis} for a relation with subtori of $T$.)

	Recall that we view the Galois field extension $\mc K \hra \mc L$ as the contravariant counterpart of a degree-$r$ covering $\Spec \mc L \thra \Spec \mc K$ of the corresponding punctured formal discs.
	When working `globally',
	on a smooth connected projective curve $X$,
	the fixed subfield $\mc K = \mc L^{\mb Z \slash r \mb Z}$ should be identified with the completed local field of $X$ at a closed marked point $x \in X$.
\end{rema}

\subsection{Admissible deformations: pure case}

We now give two notions of admissible deformations for $r$-Galois-closed irregular types,
and define the corresponding (universal) deformation spaces.

\subsubsection{}
\label{sec:exponential_factors}

Let $\Phi = \Phi(\mf g,\mf t) \sse \mf t^{\dual}$ be the root system of $(\mf g,\mf t)$.
Given an irregular type $\wh Q \in \wh{\mc{IT}}$,
and a root $\alpha \in \Phi$,
consider the \emph{exponential factor} $\wh q_\alpha \ceqq \alpha \bigl( \wh Q \bigr) \in w^{-1} \mb C [w^{-1}]$,
which is the principal part of a formal Laurent series (cf.~Def.~\ref{def:exponential_factors} for the `descended' version,
involving \emph{Puiseux series}).
Let then
\begin{equation}
	\label{eq:root_valuations}
	d_\alpha \bigl( \wh Q \bigr)
	\ceqq
	\begin{cases}
		- v \bigl( \wh q_\alpha \bigr) \in \mb Z_{> 0}, \quad & \wh q_\alpha \neq 0, \\
		0, \quad                                              & \wh q_\alpha = 0.
	\end{cases}
\end{equation}
The nonzero pole orders which occur,
viz.,
the elements of the finite set
\begin{equation}
	\label{eq:levels}
	\on{Levels}\bigl( \wh Q \bigr)
	\ceqq \Set{ d_\alpha\bigl( \wh Q \bigr) | \alpha \in \Phi } \sm \set{0} \sse \mb Z_{> 0},
\end{equation}
are the \emph{levels of} $\wh Q$.
In turn,
the \emph{levels} of the irregular class $\wh\Theta \ceqq \wh\Theta \bigl( \wh Q \bigr)$ are well-defined by $\on{Levels} \bigl( \wh\Theta \bigr) \ceqq \on{Levels} \bigl( \wh Q \bigr)$,
noting that for all $\alpha \in \Phi$ one has
\begin{equation}
	d_\alpha \bigl( g \bigl(\wh Q\bigr) \bigr)
	= d_\beta \bigl( \wh Q \bigr),
	\qquad g \in W,
	\quad \beta
	\ceqq g^{-1}(\alpha).
\end{equation}
(Cf.~\S~\ref{sec:complex_real_rational_reflection_groups}.)

\begin{rema}
	\label{rmk:admissibility}

	The `shape' of the corresponding unipotent Stokes automorphisms is controlled by these root valuations (cf.~\S~\ref{sec:main_aim}).
	Therefore,
	it makes much sense to discuss isomonodromic deformations when the function/$\Phi$-tuple
	\begin{equation}
		\label{eq:root_valuations_tuple}
		\bigl( \bm d \bigl( \wh Q \bigr) \cl \alpha
		\lmt d_\alpha \bigl( \wh Q \bigr) \bigr) \in \mb Z_{\geq 0}^{\Phi},
	\end{equation}
	defined by~\eqref{eq:root_valuations},
	is \emph{constant} as $\wh Q$ varies:
	cf.~\cite[Def.~10.1]{boalch_2014_geometry_and_braiding_of_stokes_data_fission_and_wild_character_varieties} and~\cite[\S~4]{yamakawa_2019_fundamental_two_forms_for_isomonodromic_deformations} in the untwisted/unramified setting;
	as well as~\cite[\S~4.6]{goresky_kottwitz_macpherson_2009_codimensions_of_root_valuation_strata} in the present one.
	More precisely,
	as mentioned in the introduction,
	there is a local system of Poisson/symplectic Betti spaces over any space parameterizing such `admissible' families of irregular types.
\end{rema}

\subsubsection{}

In view of Rmk.~\ref{rmk:admissibility},
the equality of the corresponding tuples~\eqref{eq:root_valuations_tuple} defines what it means for two irregular types $\wh Q$ and $\wh Q'$ to be \emph{admissible deformations} of each other.
This is symbolized by $\wh Q \sim \wh Q'$.

In the $r$-ramified setting,
however,
one must also preserve $r$-Galois-closedness throughout the deformation,
which leads to the following:

\begin{enonce}{Lemma-Definition}
	\label{lem:generating_galois_orbit}

	Choose $\wh Q \in \wh{\mc{IT}}$.
	Then:
	\begin{enumerate}
		\item
		      the irregular type $\wh Q$ is $r$-\emph{Galois-closed} if and only if there exists $g \in W$ such that
		      \begin{equation}
			      \label{eq:monodromy_realization}
			      g \bigl( \wh Q \bigr)
			      = \sigma \bigl( \wh Q \bigr)
			      = \wh Q (\zeta_r^{-1} w) \in \wh{\mc{IT}};
		      \end{equation}

		\item
		      in this case,
		      one has the equalities
		      \begin{equation}
			      \label{eq:orbit_comparison}
			      \braket{ \wh Q }_r
			      = \Set{ g^j \bigl( \wh Q \bigr) | j \in \mb Z_{> 0} }
			      = \Set{ g^j \bigl( \wh Q \bigr) | j \in \set{1,\dc,r} } \sse \wh\Theta \bigl( \wh Q \bigr);
		      \end{equation}

		\item
		      therefore,
		      if~\eqref{eq:monodromy_realization} holds we say that $g$ \emph{generates the} $r$-\emph{Galois-orbit of} $\wh Q$;

		\item
		      we also denote by $\mb L_r \bigl( \wh Q \bigr) \sse W$ the set of such generators;
		      (So $\mb L_r \bigl( \wh Q \bigr) \neq \vn$ if and only if $\wh Q$ is $r$-Galois-closed.)

		\item
		      and conversely,
		      the vector subspace of irregular types whose $r$-Galois-orbit is \emph{generated by} $g$ is denoted by $\wh{\mc{IT}}_{g,r} \sse \wh{\mc{IT}}$.
	\end{enumerate}
\end{enonce}

\begin{proof}[Proof postponed to~\ref{proof:lem_generating_galois_orbit}]
\end{proof}

\begin{defi}
	\label{def:twisted_pure_admissible_deformations}

	Let $\wh Q$ and $\wh Q'$ be two $r$-Galois-closed irregular types.
	Then:
	\begin{enumerate}
		\item
		      for a fixed $g \in W$,
		      we say that $\wh Q$ and $\wh Q'$ are \emph{mutual} $(g,r)$-\emph{admissible deformations},
		      which is symbolized by $\wh Q \sim_{g,r} \wh Q'$,
		      if:
		      \begin{enumerate}
			      \item
			            one has $g \in \mb L_r \bigl( \wh Q \bigr) \cap \mb L_r \bigl( \wh Q' \bigr)$,
			            i.e.,
			            their $r$-Galois-orbits are both generated by $g$;

			      \item
			            and one has $\bm d \bigl( \wh Q \bigr) = \bm d \bigl( \wh Q' \bigr)$,
			            i.e.,
			            their $\Phi$-tuples~\eqref{eq:root_valuations_tuple} coincide;
		      \end{enumerate}

		\item
		      more generally,
		      we say that $\wh Q$ and $\wh Q'$ are \emph{mutual} $r$-\emph{admissible deformations},
		      which is symbolized by $\wh Q \sim_r \wh Q'$,
		      if:
		      \begin{enumerate}
			      \item
			            one has $\mb L_r \bigl( \wh Q \bigr) \cap \mb L_r \bigl( \wh Q' \bigr) \neq \vn$,
			            i.e.,
			            there exists $g \in W$ such that their $r$-Galois-orbits are both generated by $g$;

			      \item
			            and again $\bm d \bigl( \wh Q \bigr) = \bm d \bigl( \wh Q' \bigr)$.
		      \end{enumerate}
	\end{enumerate}
\end{defi}

\begin{rema}
	\label{rmk:numerical_vs_admissible}

	Replacing $G$ by the analytified connected complex reductive Lie group $G(\mb C)$,
	Def.~\ref{def:twisted_pure_admissible_deformations} is equivalent to considering \emph{holomorphic} admissible families of irregular types,
	generalizing~\cite{boalch_doucot_rembado_2025_twisted_local_wild_mapping_class_groups_configuration_spaces_fission_trees_and_complex_braids} from type $A$.
    (This relies on the fact that the $r$-admissible-deformation spaces are \emph{connected}.)
\end{rema}

\subsubsection{}
\label{sec:pure_def_spaces}

It is convenient to bound the `global' pole order of irregular types before root-valuations.
(This is w.l.o.g.~for our main purpose,
cf.~\S~\ref{sec:contractible_centre}.)
Precisely,
choose another integer $s \geq 1$,
and denote by $\smash{\wh{\mc{IT}}}^{\mathsmaller\leq s} \sse \wh{\mc{IT}}$ the---%
finite-dimensional---%
vector subspace of irregular types whose degree in $w^{-1}$ is not larger than $s$.\fn{
	Whence the `total' slope in the coordinate $z = w^r$ is bounded by $s \slash r \in \mb Q_{> 0}$,
	cf.~\S~\ref{sec:notation_for_trees} and~\eqref{eq:katz_rank}.}

Then,
given any $\wh Q \in \smash{\wh{\mc{IT}}}^{\mathsmaller\leq s}$,
in~\cite{doucot_rembado_tamiozzo_2026_local_wild_mapping_class_groups_and_cabled_braids} we considered the admissible-deformation space
\begin{equation}
	\label{eq:abstract_untwisted_deformations_space_pure}
	\bm B^{\mathsmaller\leq s} \bigl( \wh Q \bigr)
	\ceqq \Set{ \wh Q' \in \smash{\wh{\mc{IT}}}^{\mathsmaller\leq s} | \wh Q \sim \wh Q' },
\end{equation}
cf.~\cite[Ex.~10.1]{boalch_2014_geometry_and_braiding_of_stokes_data_fission_and_wild_character_varieties}.
Now,
for any pair $(g,r) \in W \ts \mb Z_{> 0}$ as above,
write also
\begin{equation}
	\smash{\wh{\mc{IT}}}^{\mathsmaller\leq s}_{g,r}
	\ceqq \smash{\wh{\mc{IT}}}^{\mathsmaller\leq s} \cap \wh{\mc{IT}}_{g,r} \sse \smash{\wh{\mc{IT}}}^{\mathsmaller\leq s} \cap \wh{\mc{IT}}_r
	\eqqc \smash{\wh{\mc{IT}}}^{\mathsmaller\leq s}_r.
\end{equation}

Then we make the following:

\begin{defi}[cf.~Def.~\ref{def:twisted_pure_admissible_deformations}]
	\label{def:g_r_admissible_deformation_spaces_pure}

	\leavevmode
	\begin{enumerate}
		\item
		      Choose a pole-order-bounded irregular type $\wh Q \in \smash{\wh{\mc{IT}}}^{\mathsmaller\leq s}_{g,r}$.
		      The (universal) $(g,r)$-\emph{admissible-deformation space of} $\wh Q$ is
		      \begin{equation}
			      \label{eq:abstract_g_twisted_deformations_space_pure}
			      \bm B^{\mathsmaller\leq s}_{g,r} \bigl( \wh Q \bigr)
			      \ceqq \Set{ \wh Q' \in \smash{\wh{\mc{IT}}}^{\mathsmaller\leq s}_{g,r} | \wh Q \sim_{g,r} \wh Q' }.
		      \end{equation}

		\item
		      Analogously,
		      choose $\wh Q \in \smash{\wh{\mc{IT}}}^{\mathsmaller\leq s}_r$.
		      The (universal) $r$-\emph{admissible-deformation space of} $\wh Q$ is
		      \begin{equation}
			      \label{eq:abstract_twisted_deformations_space_pure}
			      \bm B^{\mathsmaller\leq s}_r \bigl( \wh Q \bigr)
			      \ceqq \Set{ \wh Q' \in \smash{\wh{\mc{IT}}}^{\mathsmaller\leq s}_r | \wh Q \sim_r \wh Q'}.
		      \end{equation}
	\end{enumerate}
\end{defi}

(We regard these as `spaces' in the topological sense,
rather than just sets;
cf.~\S~\ref{sec:topology}.)

\subsubsection{}

A priori,
Def.~\ref{def:twisted_pure_admissible_deformations} yields
\begin{equation}
	\label{eq:deformation_spaces_union}
	\bm B^{\mathsmaller\leq s}_r \bigl( \wh Q \bigr)
	= \bigcup_{\mb L_r(\wh Q)}  \bm B^{\mathsmaller\leq s}_{g,r} \bigl( \wh Q \bigr)  \sse \smash{\wh{\mc{IT}}}^{\mathsmaller\leq s}_r.
\end{equation}
We will first consider each term of~\eqref{eq:deformation_spaces_union} (cf.~\S\S~\ref{sec:generic_pure_case}--\ref{sec:pure_nongeneric_case}),
and later prove that the union is actually trivial (cf.~\S~\ref{sec:no_marking}):
the independence of the choice of $g$ can be intrinsically expressed in the language of (Levi) reflection subcosets (cf.~\S\S~\ref{sec:howlett}--\ref{sec:modular_groups_are_twisted_cabled_weyl}).

Moreover,
when $\mf g$ is a simple classical Lie algebra,
we will provide much more explicit descriptions of the admissible-deformation spaces (cf.~\S\S~\ref{sec:type_A}--\ref{sec:D_trees}).
But at this stage,
abstractly:

\begin{lemm}
	\label{lem:more_on_irregular_types}

	Choose $\wh Q \in \smash{\wh{\mc{IT}}}^{\mathsmaller\leq s}_{g,r}$.
	Then,
	for any $h \in W$:
	\begin{enumerate}
		\item
		      one has
		      \begin{equation}
			      \label{eq:conjugated_orbit_generators}
			      \mb L_r \bigl( \wh Q' \bigr)
			      = h \cdot \mb L_r \bigl( \wh Q \bigr) \cdot h^{-1} \sse W,
			      \qquad \wh Q' \ceqq h \bigl( \wh Q \bigr) \in \smash{\wh{\mc{IT}}}^{\mathsmaller \leq s}_r;
		      \end{equation}

		\item
		      moreover,
		      one has
		      \begin{equation}
			      \label{eq:deformations_are_intersections}
			      \bm B^{\mathsmaller\leq s}_{g,r} \bigl( \wh Q \bigr)
			      = \bm B^{\mathsmaller\leq s} \bigl( \wh Q \bigr) \cap \wh{\mc{IT}}_{g,r} \sse \smash{\wh{\mc{IT}}}^{\mathsmaller\leq s}_{g,r};
		      \end{equation}

		\item
		      and one also has
		      \begin{equation}
			      \label{eq:g_strata_permutation}
			      h \bigl( \bm B^{\mathsmaller\leq s}_{g,r} \bigl( \wh Q \bigr) \bigr)
			      = \bm B^{\mathsmaller\leq s}_{g',r} \bigl( \wh Q' \bigr) \sse \smash{\wh{\mc{IT}}}^{\mathsmaller\leq s}_{g',r},
			      \qquad g'
			      \ceqq hgh^{-1} \in W.
		      \end{equation}
	\end{enumerate}
\end{lemm}

\begin{proof}[Proof postponed to~\ref{proof:lem_more_on_irregular_types}]
\end{proof}

\subsection{Admissible deformations: full/nonpure case}

Now~\eqref{eq:deformation_spaces_union}--\eqref{eq:g_strata_permutation} imply that
\begin{equation}
	\label{eq:strata_permutation}
	h \bigl( \bm B^{\mathsmaller\leq s}_r \bigl( \wh Q \bigr) \bigr)
	= \bm B^{\mathsmaller\leq s}_r \bigl( h \bigl( \wh Q \bigr) \bigr),
	\qquad h \in W.
\end{equation}
So the $W$-action yields finite orbits of $r$-admissible-deformation spaces of $r$-Galois-closed irregular types.

This goes towards the underlying irregular classes.
Indeed,
in the untwisted setting,
recall that two irregular classes $\wh\Theta$ and $\wh\Theta'$ are admissible deformations of each other,
which is symbolized by $\wh\Theta \sim \wh\Theta'$,
provided that this holds for a pair of irregular types representing them.
In the $r$-ramified setting,
this notion does \emph{not} change:

\begin{defi}[cf.~Def.~\ref{def:twisted_pure_admissible_deformations}]
	\label{def:twisted_admissible_deformation_full}

	Two $r$-Galois-closed irregular classes $\wh\Theta$ and $\wh\Theta'$ are mutual $r$-\emph{admissible deformations},
	which is symbolized by $\wh\Theta \sim_r \wh\Theta'$,
	if there exist two irregular types $\wh Q$ and $\wh Q'$ such that:
	\begin{enumerate}
		\item
		      one has $\wh\Theta = \wh\Theta \bigl( \wh Q \bigr)$ and $\wh\Theta' = \wh\Theta \bigl( \wh Q' \bigr)$;

		\item
		      and one has $\wh Q \sim_r \wh Q'$.
	\end{enumerate}
\end{defi}

\subsubsection{}
\label{sec:full_def_spaces}

In~\cite{doucot_rembado_2025_topology_of_irregular_isomonodromy_times_on_a_fixed_pointed_curve} we considered admissible-deformation spaces of pole-order-bounded untwisted irregular classes $\wh\Theta = \wh\Theta \bigl( \wh Q \bigr) \in \smash{\wh{\mc{IT}}}^{\mathsmaller\leq s} \bs W$,
viz.,
\begin{equation}
	\label{eq:abstract_untwisted_deformation_space_full}
	\bm B^{\mathsmaller\leq s} \bigl( \wh\Theta \bigr)
	\ceqq \Set{ \wh\Theta' \in \smash{\wh{\mc{IT}}}^{\mathsmaller\leq s} \bs W | \wh\Theta \sim \wh\Theta' },
\end{equation}
cf.~\cite[Rmk.~10.6]{boalch_2014_geometry_and_braiding_of_stokes_data_fission_and_wild_character_varieties}.
We now study the twisted case,
beyond the vector-bundle setup of~\cite{boalch_doucot_rembado_2025_twisted_local_wild_mapping_class_groups_configuration_spaces_fission_trees_and_complex_braids}:

\begin{defi}[cf.~Deff.~\ref{def:g_r_admissible_deformation_spaces_pure} +~\ref{def:twisted_admissible_deformation_full}]
	Choose a pole-order-bounded $r$-Galois-closed irregular class $\wh\Theta \in \smash{\wh{\mc{IT}}}_r^{\mathsmaller\leq s} \bs W$.
	The (universal) $r$-\emph{admissible-deformation space of} $\wh\Theta$ is
	\begin{equation}
		\label{eq:abstract_twisted_deformation_space_full}
		\bm B^{\mathsmaller\leq s}_r \bigl( \wh\Theta \bigr)
		\ceqq \Set{ \wh\Theta' \in \smash{\wh{\mc{IT}}}_r^{\mathsmaller\leq s} \bs W | \wh\Theta \sim_r \wh\Theta' }.
	\end{equation}
\end{defi}

(Again,
this set has a natural topology,
cf.~\S~\ref{sec:topology}.)

\subsubsection{}

The space~\eqref{eq:abstract_twisted_deformation_space_full} is first studied in \S\S~\ref{sec:nonpure_case_generic}--\ref{sec:nonpure_case_general},
building on the pure case;
and then again more explicitly in \S\S~\ref{sec:BC_trees}--\ref{sec:D_trees} when $\mf g$ is simple,
of classical type $\bullet \in \set{B,C,D}$.

\subsection{Concluding general remarks/conventions}
\label{sec:concluding_remarks_conventions}

\subsubsection{}
\label{sec:independence_r_root}

One need not fix a primitive $r$-th root of $1$.
Namely,
in the context of Rmk.~\ref{rmk:relation_with_twist},
there is actually a \emph{canonical} group isomorphism
\begin{equation}
	\label{eq:intrinsic_monodromy}
	\mu_r \lxra{\simeq} \on{Gal}(\mc L \bs \mc K),
	\quad \zeta \lmt \bigl( \sigma_\zeta \cl f(w) \mt f(\zeta^{-1}w) \bigr),
	\qquad f = f(w) \in \mc L,
\end{equation}
without identifying either with $\mb Z \bs r \mb Z$.
Then Def.~\ref{def:galois_closed_irr_types}~(2.)~reads as follows:
for any $r$-th root $\zeta \in \mu_r$,
there exists an element $g_\zeta \in W$ such that one has the equality $\wh Q(\zeta^{-1} w) = g_\zeta \bigl( \wh Q(w) \bigr) \in \wh{\mc{IT}}$.
In particular,
the admissible-deformation spaces do \emph{not} depend on the choice of $\zeta_r$;
rather only on the ramification $r \geq 1$.
(Cf.~Rmk.~\ref{rmk:about_rational_forms};
working with real-oriented blowups,
the primitive $r$-th roots $\zeta_r^{\pm 1} \in \mu_r$ are privileged,
and the two are distinguished by the choice of orientation.)

Incidentally,
introduce the stabilizer subgroup $W^{\wh Q} \sse W$ of $\wh Q$,
i.e.,
the Weyl group of the split pair $(L,T)$ from Rmk.~\ref{rmk:relation_with_twist}.
If $N_W(\wh Q) \ceqq N_W \bigl( W^{\wh Q} \bigr)$ is its normalizer,
then $r$-Galois-closedness amounts to the existence of a \emph{group morphism}
\begin{equation}
	\mu_r \lra N_W(\wh Q) \bs W^{\wh Q},
	\qquad \zeta \lmt g_\zeta W^{\wh Q},
\end{equation}
which is determined---%
as we did above---%
by its value on the generator $\zeta_r \in \mu_r$ (the choice of lifts is immaterial,
cf.~\S~\ref{sec:no_marking}).
Again,
this goes towards relative/wild Weyl groups and Levi reflection subcosets (cf.~\S\S~\ref{sec:howlett}--\ref{sec:modular_groups_are_twisted_cabled_weyl}),
and when fission is complete/toral one just finds a group morphism $\mu_r \to W$ (cf.~Rmk.~\ref{rmk:toral_fission}).

\subsubsection{}
\label{sec:no_isogeny}

The definition of $r$-Galois-closed irregular types/classes (and their $r$-admissible deformations) only involves a split reductive Lie algebra $(\mf g,\mf t)$,
its root system $\Phi$,
and---%
the Weyl-classes of---%
the inner Lie-algebra automorphisms of $\mf g$ preserving $\mf t$:
all the split connected reductive groups $(G,T)$ integrating $(\mf g,\mf t)$ determine the same such data.\fn{
	Cf.~Rmk.~\ref{rmk:root_datum},
	on a similar note.
	Much later we will consider \emph{outer} automorphisms as well,
	where some more care is required (cf.~\S~\ref{sec:outer_isogeny}).}

Nonetheless,
taking tangent maps at $1 \in \mb G_m$ embeds the cocharacter lattice of the subtorus $Z(G)^o \sse T$ (the identity component of the centre) into the Lie-centre $\mf Z = \mf Z(\mf g) \sse \mf t$.
Upon extending scalars,
this determines a $\mb Q$-form of the latter,
which is invisible to $\Phi$ and $W$.
Then we will also consider the canonical $W$-invariant $\mb Q$-form $\mf t_{\mb Q} \sse \mf t$ which extends it,
by adding the $\mb Q$-span of the coroots $\alpha^{\dual} \in \Phi^{\dual}$.

\subsubsection{}

Write as customary a \emph{nonvanishing} irregular type of pole order $s \geq 1$ as
\begin{equation}
	\label{eq:explicit_irregular_type}
	\wh Q
	= \wh Q(w)
	= \sum_{i = 1}^s A_i w^{-i},
	\qquad A_i \in \mf t,
	\quad A_s \neq 0.
\end{equation}

Suppose that~\eqref{eq:explicit_irregular_type} is $r$-Galois-closed,
for an integer $r \geq 1$.
Explicitly,
the $r$-Galois-orbit of $\wh Q$ is
\begin{equation}
	\label{eq:explicit_global_r_galois_orbit}
	\braket{ \wh Q}_r
	= \Set{ \sum_{i = 1}^s \zeta_r^{ij} A_i \cdot w^{-i} | j \in \mb Z \bs r \mb Z} \sse \wh{\mc{IT}}^{\mathsmaller\leq s},
\end{equation}
and so the condition is that there exists $g \in W$ such that
\begin{equation}
	\label{eq:explicit_r_galois_closedness}
	\sum_{i = 1}^s g(A_i) \cdot w^{-i}
	= g \bigl( \wh Q(w) \bigr)
	= \wh Q(\zeta_r^{-1}w)
	= \sum_{i = 1}^s \zeta_r^i A_i \cdot w^{-i};
\end{equation}
i.e.,
such that $g(A_i) = \zeta_r^i A_i \in \mf t$ for $i \in \set{1,\dc,s}$.
(This is the basic observation which relates twisted meromorphic gauge theory on curves with Springer theory.)

\subsubsection{}

In addition,
for any integer $\wt r \geq 1$ set $r' \ceqq \wt r r$,
and consider the $r'$-Galois-closed irregular type
\begin{equation}
	\label{eq:dilated_exponents}
	\wh Q'
	\ceqq \wh Q \bigl (w^{\wt r} \bigr) \in \smash{\wh{\mc{IT}}}^{\mathsmaller\leq s'},
	\qquad s'
	\ceqq \wt r s \in \mb Z_{> 0}.
\end{equation}
There is a bijection $\bm B^{\mathsmaller\leq s}_r \bigl( \wh Q \bigr) \lxra{\simeq} \bm B^{\mathsmaller\leq s'}_{r'} \bigl( \wh Q' \bigr)$,
obtained by applying the change of variable~\eqref{eq:dilated_exponents} to the admissible deformations of $\wh Q$;
however,
the $r'$-Galois-orbit of $\wh Q'$ has at most $r \leq r'$ elements.
To only consider equalities,
one can introduce the following:

\begin{enonce}{Lemma-Definition}
	\label{lem:primitive_irregular_types}

	For any $\wh Q \in \wh{\mc{IT}}_r$ there exists a divisor $\wt r \geq 1$ of $r$ such that:
	\begin{enumerate}
		\item
		      the irregular type $\wh Q' \ceqq \wh Q \bigl( w^{1 \slash \tilde r} \, \bigr)$ is---%
		      untwisted and---%
		      $r'$-\emph{Galois-closed},
		      where $r'$ is the quotient of the division;

		\item
		      the $r'$-Galois-orbit of $\wh Q'$ contains precisely $r'$ elements;

		\item
		      and we then say that $\wh Q'$ is \emph{primitive}.
	\end{enumerate}
\end{enonce}

\begin{proof}[Proof postponed to~\ref{proof:lem_primitive_irregular_types}]
\end{proof}

\subsubsection{}

By Lem.~\ref{lem:primitive_irregular_types},
w.l.o.g.,
one can just consider \emph{primitive} $r$-Galois-closed irregular types.
In view of the proof~\ref{proof:lem_generating_galois_orbit},
this means that $k = r$ in~\eqref{eq:galois_orbit_from_generator};
geometrically,
the Galois covering $\Spec \mc L \thra \Spec \mc K$---%
of formal punctured discs---%
has degree equal to the ramification of the twisted/descended irregular class $\Theta$.
(But we shall always consider the integer $r \geq 1$ as the `ramification' parameter.)

\subsubsection{}
\label{sec:topology}

We identify the vector subspace $\smash{\wh{\mc{IT}}}^{\mathsmaller\leq s} \sse \wh{\mc{IT}}$ with $\mf t^s$,
via the $\mb C$-linear isomorphism $\wh Q \mt (A_1,\dc,A_s)$ (cf.~\eqref{eq:explicit_irregular_type}).
In particular,
there are inclusions
\begin{equation}
	\label{eq:irr_type_coefficients}
	\bm B^{\mathsmaller\leq s}_{g,r} \bigl( \wh Q \bigr) \sse \bm B^{\mathsmaller\leq s}_r \bigl( \wh Q \bigr) \sse \bm B^{\mathsmaller\leq s} \bigl( \wh Q \bigr) \sse \mf t^s,
\end{equation}
cf.~\eqref{eq:abstract_g_twisted_deformations_space_pure}--\eqref{eq:abstract_twisted_deformations_space_pure}.
Analogously,
the $W$-action on irregular types is identified with the diagonal $W$-action on the Cartesian power of the Cartan subalgebra,
and so
\begin{equation}
	\label{eq:irr_class_coefficients}
	\bm B^{\mathsmaller\leq s}_r \bigl( \wh\Theta \bigr) \sse \bm B^{\mathsmaller\leq s} \bigl( \wh\Theta \bigr) \sse \mb A_s \ceqq \mf t^s \bs W,
\end{equation}
in the notation of~\eqref{eq:abstract_untwisted_deformation_space_full}--\eqref{eq:abstract_twisted_deformation_space_full}.

In view of this:
(i) we regard the admissible-deformation spaces of irregular types as \emph{topological subspaces} of the complex manifold $\mf t^s \simeq \mb C^{s \cdot m}$,
where $m \ceqq \rk(\mf g)$;
(ii) we equip $\mb A_s$ with the \emph{complex-analytic} quotient topology;\fn{
	Much would work the same in Zariski topology,
	with the notable exception that we consider genuine `topological' fundamental groups (in \S~\ref{sec:tglwmcgs}),
	rather than profinite/étale ones.}~and (iii) the admissible-deformation spaces of irregular classes are viewed as topological subspaces of $\mb A_s$.
The previous actions (resp.,
bijections) are then continuous (resp.,
homeomorphisms).

\begin{rema}
	\label{rmk:singular_quotient}

	The Chevalley identification $\mf g \bs\!\!\bs G \simeq \mf t \bs W$ need \emph{not} extend when $s \geq 2$,
	and the quotient $\mb A_s$ need \emph{not} be smooth.

	Nonetheless,
	we will show that the admissible-deformation spaces of irregular types are---%
	connected---%
	locally-closed complex submanifolds of $\mf t^s$.
	Moreover,
	the $W$-stabilizers are \emph{constant} thereon,
	and the fibres of the restriction of the canonical projection $\mf t^s \thra \mb A_s$ to such spaces are orbits for their setwise stabilizers:
	this embeds the admissible-deformation spaces of irregular classes into $\mb A_s$,
	as bases of suitable Galois coverings,
	which remain---%
	connected---%
	complex manifolds (cf.~\S\S~\ref{sec:generic_pure_case}--\ref{sec:nonpure_case_general}).

	So the singularities of $\mb A_s$ play no role,
	which is the main reason why we can consider the nontoral/noncomplete case of fission (cf.~Rmk.~\ref{rmk:toral_fission}),
	extending (the truncation of) the stratifications of~\cite{goresky_kottwitz_macpherson_2009_codimensions_of_root_valuation_strata}.
\end{rema}

\subsubsection{}
\label{sec:uniformizers}

The embeddings~\eqref{eq:irr_type_coefficients} make sense in the given variable $w$,
by extracting the coefficients of irregular types.

Intrinsically,
one rather uses a complete DVR $\ms O$ over $\mb C$,
with a choice of coefficient field $\mb C \hra \ms O$,
and takes an element
\begin{equation}
	\label{eq:intrinsic_irr_type}
	\wh Q \in \mf t \ots_{\mb C} (\ms L \bs \ms O),
	\qquad \ms L
	\ceqq \on{Frac} ( \ms O  ).
\end{equation}
(So $\ms O = \ms L_{\geq 0} \sse \ms L$ is the valuation subring,
cf.~\cite[Def.~7.1]{boalch_2014_geometry_and_braiding_of_stokes_data_fission_and_wild_character_varieties}.)
The notion of $r$-Galois-closedness can now be given by assuming that $\ms L$ is a degree-$r$ (tame,
totally-ramified) Galois extension of another complete DVF $\ms K$,
identifying $\Gal(\ms L \bs \ms K) \simeq \mu_r$,
and letting the latter act on the right tensor-factor of~\eqref{eq:intrinsic_irr_type};
then one can ask that for all $\zeta \in \mu_r$ there exists $g_\zeta \in W$ such that $\zeta \bigl( \wh Q \bigr) = g_\zeta \bigl( \wh Q \bigr)$ (cf.~\S~\ref{sec:independence_r_root}).
Again one stratifies $r$-Galois-closed irregular types by the discrete valuation of $\ms L$,
upon also evaluating at the roots of $(\mf g,\mf t)$.
(This uniformizer-free viewpoint is taken,
e.g.,
in~\cite{doucot_rembado_tamiozzo_moduli_spaces_of_untwisted_wild_riemann_surfaces},
to define stacks of wild curves in the pure untwisted case.)
As for the topology,
truncating at depth/pole order $s \geq 1$ leads to the vector subspace
\begin{equation}
	\mf m^{-s} \bs \ms O \lhra \ms L \bs \ms O,
\end{equation}
involving the negative powers of the maximal ideal $\mf m \sse \ms O$:
in turn,
now the quotient $\mf t \otimes_{\mb C} (\mf m^{-s} \bs \ms O) \simeq \mf t(\mf m^{-s}) \bs \mf t(\ms O)$ is finite-dimensional,
and so it carries a canonical structure of---%
Hausdorff---%
\emph{topological} vector space.

This provides `intrinsic' admissible-deformation spaces;
and moreover this all descends to irregular classes,
upon taking out the $W$-action on $\mf t$,
which commutes with the $\mu_r$-action on `scalar' principal parts.

Nonetheless,
just as in~\cite{doucot_rembado_tamiozzo_2026_local_wild_mapping_class_groups_and_cabled_braids,doucot_rembado_2025_topology_of_irregular_isomonodromy_times_on_a_fixed_pointed_curve,boalch_doucot_rembado_2025_twisted_local_wild_mapping_class_groups_configuration_spaces_fission_trees_and_complex_braids},
in this text it is convenient to fix---%
the root of---%
a uniformizer in order to give explicit models.
Namely,
a uniformizer $w$ of $\ms O$ provides identifications $\ms L \simeq \mc L = \mb C(\!( w )\!)$ and $\ms K \simeq \mc K = \mb C(\!( z )\!)$,
where again $z = w^r$,
and this maps the intrinsic admissible-deformation spaces \emph{homeomorphically} onto the `standard' subspaces of $\mf t (\!( w )\!) \bs \mf t\llb w \rrb$ considered above:
importantly,
the latter \emph{do} not depend on the choice of $w$.

\subsubsection{}

Indeed,
precisely,
consider the group $H = H(w)$ of---%
continuous---%
$\mb C$-algebra automorphisms of $\mc L_{\geq 0} = \mb C \llb w \rrb$.
It corresponds to `formal' changes of the variable $w$,
acting simply-transitively on the set of uniformizers.
The $H$-action extends uniquely to $\mb C$-algebra automorphisms of $\mc L$,
but it is \emph{not} compatible with the Galois extension $\mc K \hra \mc L$ of Rmk.~\ref{rmk:relation_with_twist}:
one should rather only consider the subgroup of Galois-equivariant automorphisms,
viz.,
\begin{equation}
	\label{eq:equivariant_field_automorphisms}
	H_r
	\ceqq \Set{ f \in H | f(\zeta_r^{-1} w) = \zeta_r^{-1} f(w) } \sse H.
\end{equation}
Finally,
extend the whole $H$-action to $\mf t(\!( w )\!)$ in $\mb C$-bilinear fashion,
so that $\mf t\llb w \rrb \sse \mf t(\!( w )\!)$ is preserved.
Then note that:

\begin{lemm}
	\label{lem:intrinsicality}
	\leavevmode

	\begin{enumerate}
		\item
		      The induced $H_r$-action on $\wh{\mc{IT}}$ \emph{preserves} the subspaces $\wh{\mc{IT}}_{g,r}$ and $\smash{\wh{\mc{IT}}}^{\mathsmaller\leq s}$,
		      for all $g \in W$ and $s \in \mb Z_{\geq 1}$.

		\item
		      The restricted $H_r$-action on $\smash{\wh{\mc{IT}}}^{\mathsmaller\leq s}_{g,r} = \wh{\mc{IT}}_{g,r} \cap \smash{\wh{\mc{IT}}}^{\mathsmaller\leq s}$ \emph{preserves} its partition into $(g,r)$-admissible-deformation spaces.

		\item
		      And there is a short exact group sequence
		      \begin{equation}
			      1 \lra \mb Z \bs r \mb Z \lra H_r \lra H(z) \lra 1,\fn{
				      It does \emph{not} split for $r \geq 2$.}
		      \end{equation}
		      where $H(z)$ is the group of continuous $\mb C$-algebra automorphisms of $\mc K_{\geq 0} \ceqq \mb C \llb z \rrb$.
	\end{enumerate}
\end{lemm}

\begin{proof}[Proof postponed to~\ref{proof:lem_intrinsicality}]
\end{proof}

\subsubsection{}
\label{sec:about_homeo_classes}

By Lem.~\ref{lem:intrinsicality},
the homeomorphism class of the $r$-admissible-deformation spaces does \emph{not} depend on the choice of:
(i) a uniformizer $z$,
at a marked point $x \in X$ of a complex projective curve;
and (ii) an $r$-th root $w$ thereof (geometrically corresponding to an $r$-fold local covering,
etc.)
This also holds in the full/nonpure case,
as the actions of $H$ and $W$ on $\mf t\llb w \rrb \sse \mf t(\!( w )\!)$ commute.

On a related note,
by~\eqref{eq:strata_permutation},
in the pure case the topology does not vary along the $W$-orbits involved.
More generally,
it does \emph{not} vary along the $\Aut(\Phi)$-orbits of irregular types,
as this action permutes the root-valuation tuples---%
and commutes with~\eqref{eq:monodromy},
cf.~\S~\ref{sec:aut_phi_conjugacy}.
(These facts are used in the classification problem solved by Thm.~\ref{thm:thm_2_intro}.)

Incidentally,
up to conjugating by $G$,
the choice of maximal torus $T \sse G$ is also immaterial.

\subsubsection{}
\label{sec:contractible_centre}

Finally,
as in~\cite{doucot_rembado_tamiozzo_2026_local_wild_mapping_class_groups_and_cabled_braids,
	doucot_rembado_2025_topology_of_irregular_isomonodromy_times_on_a_fixed_pointed_curve,
	boalch_doucot_rembado_2025_twisted_local_wild_mapping_class_groups_configuration_spaces_fission_trees_and_complex_braids,
	doucot_rembado_tamiozzo_moduli_spaces_of_untwisted_wild_riemann_surfaces},
note that one could a priori define admissible deformations of~\eqref{eq:explicit_irregular_type} of \emph{unbounded} pole order/slope.
However,
the coefficients of degree $-i < -s$ of any such deformation would be annihilated by all roots,
and thus would lie in the---%
contractible---%
Lie-centre $\mf Z \sse \mf t$.
Therefore,
the homotopy type of the spaces~\eqref{eq:abstract_g_twisted_deformations_space_pure}--\eqref{eq:abstract_twisted_deformations_space_pure} is well-determined by $\wh Q$ and $r$ alone,
and hereafter we abusively remove the corresponding superscripts:
set $\bm B_{g,r} \bigl( \wh Q \bigr) = \bm B_{g,r}^{\mathsmaller\leq s} \bigl( \wh Q \bigr)$,
etc.

\section{Pure setting:
  quasigeneric case}
\label{sec:generic_pure_case}

\subsection{}
\label{sec:monolevel_is_quasigeneric}

Assume first that the function~\eqref{eq:root_valuations_tuple} is constant,
i.e.,
that there is an integer $d \in \set{0,\dc,s}$ such that $\wh q_\alpha$ has pole order $d$ for all $\alpha \in \Phi$ (cf.~\S~\ref{sec:exponential_factors}).
The case $d = 0$ implies that $\wh Q \in w^{-1} \mf Z [w^{-1}]$.
The only possible ramification is $r = 1$ (for primitive $\wh Q$),
and then $\bm B_{g,r}\bigl(\wh Q\bigr) = \bm B \bigl( \wh Q \bigr) = \mf Z^s$ is contractible:
hereafter,
one can assume that this is never the case.
Therefore $d \in \set{1,\dc,s}$,
which is the unique level of $\wh Q$ (cf.~\eqref{eq:levels}),
and up to deformation-retraction of admissible-deformation spaces one can then suppose that $d = s$ (cf.~\S~\ref{sec:contractible_centre}).

So the leading term $A_s$ is \emph{regular},
and this is what we refer to as the \emph{quasigeneric case}.
In particular,
the `descended' twisted irregular type $Q = Q(z)$ has the unique slope $s \slash r \in \mb Q_{> 0}$.

\begin{rema}
	\label{rmk:abelian_case}

	Somewhat analogously,
	if $\mf g$ is \emph{abelian} then $\mf g = \mf t = \mf Z$:
	so $\Phi = \vn$ and the Weyl group is trivial.
	It follows that~\eqref{eq:explicit_irregular_type} is $r$-Galois-closed if and only if $A_i = 0$ whenever $r \nmid i$,
	whence $\wh Q \in w^{-r} \mf t [w^{-r}] = z^{-1} \mf t[z^{-1}]$.
	Now one is back to the untwisted case,
	and again the deformation spaces are contractible.
\end{rema}

\subsection{Minimal pole order}

Consider in particular a homogeneous irregular type of the form $\wh Q = A w^{-1} \in \smash{\wh{\mc{IT}}}_r^{\mathsmaller\leq 1}$.
Then $A \in \mf t$ is a $W$-\emph{regular vector},
i.e.,
one has
\begin{equation}
	\label{eq:regular_cartan}
	A \in \mf t_{\reg}
	= \mf t \, \bigsm \, \bigcup_\Phi H_\alpha,
	\qquad H_\alpha \ceqq \ker(\alpha) \sse \mf t.
\end{equation}

We can now make/prove the basic statement relating twisted meromorphic gauge theory (on curves) with \emph{Springer theory}:

\begin{enonce}{Proposition-Definition}
	\label{prop:generic_twisted_deformation_space_1_coeff}

	Choose $g \in \mb L_r \bigl( \wh Q \bigr)$,
	and consider the eigenspace
	\begin{equation}
		\label{eq:regular_eigenspace}
		\mf t(g,\zeta_r)
		\ceqq \ker( g - \zeta_r \cdot \Id_{\mf t} ) \sse \mf t,
	\end{equation}
	as in~\cite[\S~3.1]{springer_1974_regular_elements_of_finite_reflection_groups}.
	Then:
	\begin{enumerate}
		\item
		      one has
		      \begin{equation}
			      \label{eq:generic_twisted_deformation_space_1_coeff}
			      \bm B_{g,r} \bigl( \wh Q \bigr)
			      = \mf t(g,\zeta_r) \, \bigsm \, \bigcup_\Phi H_\alpha(g,\zeta_r),
			      \qquad H_\alpha(g,\zeta_r) \ceqq  H_\alpha \cap \mf t(g,\zeta_r)  \sse \mf t(g,\zeta_r);
		      \end{equation}

		\item
		      and~\eqref{eq:generic_twisted_deformation_space_1_coeff} is the (nonempty) complement of a \emph{complex-reflection arrangement}.
	\end{enumerate}
\end{enonce}

\begin{proof}
	The condition that $g$ generates the $r$-Galois-orbit of $\wh Q$ means precisely that $A$ lies in~\eqref{eq:regular_eigenspace} (cf.~\eqref{eq:explicit_r_galois_closedness}).
	Now,
	by ~\cite[Thm.~2.5]{denef_loeser_1995_regular_elements_and_monodromy_of_discriminants_of_finite_reflection_groups},
	the intersection of~\eqref{eq:regular_cartan}--\eqref{eq:regular_eigenspace} is a (linear) hyperplane complement contained within the latter,
	which corresponds to removing the reflection arrangement of the centralizer of $g$ in $W$ (cf.~\S~\ref{sec:nonpure_case_generic}).
	The conclusion follows from~\eqref{eq:deformations_are_intersections},
	as $\bm B \bigl( \wh Q \bigr) = \mf t_{\reg}$ here.
\end{proof}

\begin{rema}
	\label{rmk:eigenspace_intersection}

	Again,
	the apparent dependence on the choice of a primitive $r$-th root of $1$ vanishes,
	since
	\begin{equation}
		\label{eq:eigenspace_intersection}
		\mf t(g,\zeta_r)
		= \bigcap_{(\mb Z \slash r\mb Z)^{\ts}} \mf t(g^i,\zeta_r^i)
		= \bigcap_{\mb Z \slash r \mb Z} \mf t(g^i,\zeta_r^i) \sse \mf t,
	\end{equation}
	recalling that $g^r(A) = A$ implies $g^r = 1 \in W$ (cf.~Lem.~\ref{lem:about_pointwise_stabilizers}).
	If $i \wdg r = 1$,
	one can thus relabel by $\zeta' \ceqq \zeta_r^i \in \mu_r$ and $g' \ceqq g^i \in W$.
	(Cf.~\S~\ref{sec:independence_r_root};
	one has $W^{\wh Q} = 1$ here.)
\end{rema}

\subsection{Higher pole order}
\label{sec:generic_case_several_coeff}

The situation is essentially the same if~\eqref{eq:explicit_irregular_type} has regular leading coefficient,
but arbitrary pole order $s \geq 1$.
Namely:

\begin{prop}
	\label{prop:generic_pure_twisted_deformation_space}

	Suppose that $A_s \in \mf t_{\reg}$.
	Then there is a topological factorization $\bm B_{g,r} \bigl( \wh Q \bigr) = V' \ts U$,
	where:
	\begin{enumerate}
		\item
		      the first factor $V' \sse \mf t^{s-1}$ is a vector subspace;

		\item
		      and $U \sse \mf t(g,\zeta_r^s)$ is a hyperplane complement analogous to~\eqref{eq:generic_twisted_deformation_space_1_coeff}.
	\end{enumerate}
\end{prop}

\begin{proof}
	By hypothesis,
	the coefficients $A_1,\dc,A_s \in \mf t$ are eigenvectors for $g \in W$, with corresponding eigenvalues $\zeta_r,\dc,\zeta_r^s \in \mb C^{\ts}$.
	Now we use a particular case of the direct-product decomposition of~\eqref{eq:abstract_untwisted_deformations_space_pure},
	cf.~\cite[Prop.~2.1]{doucot_rembado_tamiozzo_2026_local_wild_mapping_class_groups_and_cabled_braids}.
	Namely,
	one has
	\begin{equation}
		\label{eq:generic_factors_pure_deformation_space}
		\bm B \bigl( \wh Q \bigr)
		= \prod_{i = 1}^s \bm B \bigl( \wh Q,i \bigr),
		\qquad \bm B \bigl( \wh Q,i \bigr) \ceqq
		\begin{cases}
			\mf t_{\reg}, \quad & i = s,                 \\
			\mf t, \quad        & i \in \set{1,\dc,s-1}.
		\end{cases}
	\end{equation}
	Hence,
	using again~\eqref{eq:deformations_are_intersections} yields
	\begin{equation}
		\label{eq:generic_g_admissible_deformation_space}
		\bm B_{g,r} \bigl( \wh Q \bigr)
		= \prod_{i = 1}^s  \bm B \bigl( \wh Q,i \bigr)  \cap \prod_{i = 1}^s  \mf t(g,\zeta_r^i)
		= V'  \ts \Bigl( \mf t_{\reg} \cap \mf t(g,\zeta_r^s) \Bigr) \sse \mf t^s,
	\end{equation}
	where in turn
	\begin{equation}
		V' \ceqq \prod_{i = 1}^{s-1}  \mf t(g,\zeta_r^i) \sse \mf t^{s-1}.
	\end{equation}

	Setting $U \ceqq \mf t_{\reg} \cap \mf t(g,\zeta_r^s)$,
	the conclusion follows again from~\cite{springer_1974_regular_elements_of_finite_reflection_groups,
		denef_loeser_1995_regular_elements_and_monodromy_of_discriminants_of_finite_reflection_groups}.
\end{proof}

\begin{rema}
	\label{rmk:about_regular_numbers}

	In the notation of Prop.~\ref{prop:generic_pure_twisted_deformation_space},
	let $d \ceqq r \bs (r \wdg s) \geq 1$.
	Then $\zeta_r^s \in \mb C^{\ts}$ is a primitive $d$-th root of $1$,
	and so:
	(i) either $d = 1$,
	in which case $g(A_s) = A_s$ forces $g = 1$;
	or (ii) $g \in W$ is a regular element of order $d \geq 2$ (cf.~\cite[Thm.~4.2.(i)]{springer_1974_regular_elements_of_finite_reflection_groups}).
	In particular,
	$d \geq 2$ is a \emph{regular number of} $W$ (cf.~\S~5 of op.~cit.).

	This yields arithmetic constraints for the quasigeneric ramifications/irregularities,
	determined by the structure group.
	E.g.,
	if $W = W_{\bullet}(m)$ is a Weyl group of classical type $\bullet \in \set{A,B\!C,D}$ and rank $m \geq 1$,
	let $h = h_\bullet (m) \geq 1$ be its Coxeter number (i.e.,
	the highest degree);
	then loc.~cit.~implies that:
	\begin{enumerate}
		\item[(A)]
		      if $\bullet  = A$,
		      whence $h = m$,
		      then $d \mid h$ or $d \mid h - 1$;

		\item[(B\!C)]
		      if $\bullet = B\!C$,
		      whence $h = 2m$,
		      then $d \mid h$;

		\item[(D)]
		      and if $\bullet = D$,
		      whence $h = 2m - 2$,
		      then $d \mid m$ or $d \mid h$.
	\end{enumerate}
	(Cf.~\S~\ref{sec:exceptional_types} about exceptional types;
	in general,
	a number is regular if it divides as many degrees as it does codegrees,
	cf.~\cite[Cor.~5.5]{lehrer_springer_1999_reflection_subquotients_of_unitary_reflection_groups}.)
\end{rema}

\begin{rema}
	\label{rmk:about_rational_forms}

	Appealing to the canonical $W$-invariant $\mb Q$-form $\mf t_{\mb Q} \sse \mf t$ (cf.~\S~\ref{sec:no_isogeny}),
	it follows that the element $g$ also possesses an eigenvector $A''_s \in \mf t_{\reg}$ of---%
	regular---%
	eigenvalue $\zeta_d = e^{2\pi \sqrt{-1} \slash d} \in \mb C^{\ts}$,
	where again $d = r \bs (r \wdg s)$.
	(One has $\zeta_r^s = \zeta_d$ if and only if $s \equiv (r \wdg s) \pmod r$.)
	More generally,
	it possesses a regular eigenvector for \emph{any} primitive $d$-th root of $1$.

	Even more generally,
	let $\on{Gal}_{\mb Q}(\zeta_r) \ceqq \Gal \bigl( \mb Q(\zeta_r) \bs \mb Q \bigr) \simeq (\mb Z \bs r \mb Z)^{\ts}$ act on the right tensor-factor of $\mf t_{\mb Q}(\zeta_r) \ceqq \mf t_{\mb Q} \ots_{\mb Q} \mb Q(\zeta_r)$,\fn{
		This should \emph{not} be confused with the cyclic Galois-group action of~Rmk.~\ref{rmk:relation_with_twist}.
	}~commuting with the $W$-action on the left one.
	The given field extension $\mb Q(\zeta_r) \hra \mb C$ yields a $\mb Q(\zeta_r)$-linear embedding
	\begin{equation}
		\mf t_{\mb Q}(\zeta_r) \lhra \mf t \simeq \mf t_{\mb Q} \ots_{\mb Q} \mb C \simeq \mf t_{\mb Q}(\zeta_r) \ots_{\mb Q(\zeta_r)} \mb C,
	\end{equation}
	which we identify with its image.
	Now:
	(i) one can choose new vectors $A'_i$ lying in $\mf t(g,\zeta_r^i) \cap \mf t_{\mb Q}(\zeta_r)$ (for $i \in \set{1,\dc,s}$),
	i.e.,
	lying in the $\mb Q(\zeta_r)$-forms of the eigenspaces of $g$;
	and (ii) one can ensure that $A'_s$ is \emph{regular} for the restricted $\mb Q(\zeta_r)$-linear reflection action of $W$ (as $\mb Q(\zeta_r)$ is \emph{infinite}).
	Therefore,
	the tuple $(A'_1,\dc,A'_s) \in \mf t_{\mb Q}(\zeta_r)^s \sse \mf t^s$ yields a $(g,r)$-admissible deformation $\wh Q' \ceqq \sum_{i = 1}^s A'_i w^{-i}$ of $\wh Q$.
	Finally,
	any element $\tau \in \on{Gal}_{\mb Q}(\zeta_r)$ satisfies
	\begin{equation}
		\label{eq:rational_galois_group_action}
		g \bigl( \tau(A'_i) \bigr)
		= \tau \bigl( g(A'_i) \bigr)
		= \zeta^i \tau (A'_i) \in \mf t_{\mb Q}(\zeta_r),
		\qquad \zeta \ceqq \tau(\zeta_r) \in \mb C^{\ts}.
	\end{equation}
	In particular,
	given \emph{any} primitive $r$-th root $\zeta$ of $1$,
	there exist coefficients $(A''_1,\dc,A''_s) \in \mf t^s$ such that $A''_i \in \mf t(g,\zeta^i)$---%
	with $A''_s$ regular.
	(Generalizing what was claimed above.)

	This all works the same even when the coefficients of the starting irregular type $\wh Q$ lie on arbitrary root hyperplanes,
	as we shall now consider.
\end{rema}

\section{Pure setting:
  general case}
\label{sec:pure_nongeneric_case}

\subsection{}

Suppose more generally to be in the \emph{monolevel} case.
This means that one has $\on{Levels} \bigl( \wh Q \bigr) = \set{d}$ for some $d \in \set{1,\dc,s}$;
but one might also have $\alpha \bigl( \wh Q \bigr) = 0$ for some $\alpha \in \Phi$.
In the notation of~\eqref{eq:explicit_irregular_type},
there exists then a \emph{proper} subset of roots $\phi \sse \Phi$ such that:
\begin{enumerate}
	\item
	      one has
	      \begin{equation}
		      \braket{ \alpha,A_i } = 0,
		      \qquad i \in \set{d+1,\dc,s},
		      \quad \alpha \in \Phi,
	      \end{equation}
	      so that the pole order of $\wh q_\alpha$ is always at most $d$;

	\item
	      one has $\braket{ \alpha,A_d } = 0$ if and only if $\alpha \in \phi$,
	      so that the pole order $d$ does occur in $\wh q_\beta$ for $\beta \in \Phi \sm \phi$;

	\item
	      and one has
	      \begin{equation}
		      \braket{ \alpha,A_i } = 0,
		      \qquad i \in \set{1,\dc,d-1},
		      \quad \alpha \in \phi,
	      \end{equation}
	      so that no lower pole order can occur.
	      (More roots can vanish on $A_1,\dc,A_{d-1}$.)
\end{enumerate}

The previous case had $\phi = \vn$,
and it follows that $\phi$ is the intersection of $\Phi$ with the annihilator $\set{A_1,\dc,A_d}^\perp \sse \mf t^{\dual}$;
so it is a \emph{Levi} (root) \emph{subsystem}.
(Equivalently,
it satisfies $\phi = \spann_{\mb Q}(\phi) \cap \Phi$;
equivalently,
the complement~\eqref{eq:pure_untwisted_nongeneric_deformation_space_1_coeff} is nonempty.)

Again,
as far as local WMCGs are concerned,
one can then assume that $d = s$.

\subsection{Minimal pole order}
\label{sec:pure_nongeneric_case_one_coeff}

Consider in particular a homogeneous irregular type of the form $\wh Q = A w^{-1} \in \smash{\wh{\mc{IT}}}^{\mathsmaller\leq 1}_r$,
where now $A \in \mf t$ can lie on any root hyperplane.
As in~\cite{doucot_rembado_tamiozzo_2026_local_wild_mapping_class_groups_and_cabled_braids},
the untwisted admissible-deformation space $\bm B \bigl( \wh Q \bigr)$ is a complex hyperplane complement inside the flat
\begin{equation}
	\label{eq:kernel_intersection}
	\mf t_\phi \ceqq \ker(\phi)
	= \bigcap_\phi H_\alpha \sse \mf t,
	\qquad \phi = \phi_A \ceqq \Phi \cap \set{ A }^\perp \sse \Phi,
\end{equation}
involving the Levi subsystem of roots annihilating $A$---%
i.e.,
the `Levi annihilator'.
Precisely,
somewhat analogously to~\eqref{eq:generic_twisted_deformation_space_1_coeff},
one has
\begin{equation}
	\label{eq:pure_untwisted_nongeneric_deformation_space_1_coeff}
	\bm B \bigl( \wh Q \bigr)
	= \mf t_\phi \, \bigsm \, \bigcup_{\Phi \sm \phi} \,  H_\alpha(\phi) ,
	\qquad H_\alpha(\phi) \ceqq H_\alpha \cap \mf t_\phi \sse \mf t_\phi.
\end{equation}
The latter complement is \emph{nonempty},
since one has $\mf t_\phi \nsse H_\alpha$ for all $\alpha \in \Phi \sm \phi$,
and the traces---%
on $\mf t_\phi$---%
of the root-hyperplanes yield new hyperplanes
\begin{equation}
	H_\alpha(\phi)
	= \ker (\alpha_\phi),
	\qquad \alpha_\phi
	\ceqq \eval[1]\alpha_{\mf t_\phi} \in \mf t_\phi^{\dual} \sm (0).
\end{equation}
(Cf.~Lem.~\ref{lem:levis_and_flats}).

More generally,
in the twisted case:

\begin{enonce}{Proposition-Definition}
	\label{prop:pure_general_twisted_deformation_space_1_coeff}

	Choose $g \in \mb L_r \bigl( \wh Q \bigr)$.
	Then:
	\begin{enumerate}
		\item
		      the subspace~\eqref{eq:kernel_intersection} is $g$-\emph{invariant},
		      i.e.,
		      \begin{equation}
			      \label{eq:stratum_stabilizer}
			      g \in N_W(\mf t_\phi)
			      \ceqq \Set{ g' \in W | g' (\mf t_\phi) \sse \mf t_\phi };
		      \end{equation}

		\item
		      one has
		      \begin{equation}
			      \label{eq:pure_general_twisted_deformation_space_1_coeff}
			      \bm B_{g,r} \bigl( \wh Q \bigr)
			      = \mf t_\phi(g_\phi,\zeta_r) \, \bigsm \, \bigcup_{\Phi \sm \phi} \,  H_\alpha(g_\phi,\zeta_r) ,
			      \qquad g_\phi \ceqq \eval[1]g_{\mf t_\phi},
		      \end{equation}
		      generalizing the notation of~\eqref{eq:regular_eigenspace},
		      and setting
		      \begin{equation}
			      \label{eq:hyperplane_intersection}
			      H_\alpha(g_\phi,\zeta_r)
			      \ceqq H_\alpha \cap \mf t_\phi(g_\phi,\zeta_r) = H_\alpha(\phi) \cap \mf t(g,\zeta_r) \sse \mf t_\phi(g_\phi,\zeta_r);
		      \end{equation}

		\item
		      and~\eqref{eq:pure_general_twisted_deformation_space_1_coeff} is a \emph{nonempty} complex hyperplane complement.
	\end{enumerate}
\end{enonce}

(Again recall that when $\phi \neq \vn$ then~\eqref{eq:pure_general_twisted_deformation_space_1_coeff} need \emph{not} be the regular part of a complex reflection group;
cf.~\S\S~\ref{sec:howlett_twist} +~\ref{sec:type_D}.)

\begin{proof}
	By hypothesis $g(A) = \zeta_r A \in \bm B \bigl( \wh Q \bigr)$,
	as~\eqref{eq:pure_untwisted_nongeneric_deformation_space_1_coeff} is invariant under dilations by $\mb C^{\ts}$.
	The first statement follows from Lem.~\ref{lem:about_setwise_stabilizers}.

	Moreover,
	since $g$ acts in semisimple fashion on $\mf t$,
	and the vector subspace $\mf t_\phi \sse \mf t$ is $g$-stable,
	one has $\mf t(g,\zeta_r) \cap \mf t_\phi = \mf t_\phi (g_\phi,\zeta_r)$---%
	generalizing~\eqref{eq:regular_eigenspace} from the case $\phi = \vn$.
	We conclude by~\eqref{eq:deformations_are_intersections},
	noting (again) that the intersections~\eqref{eq:hyperplane_intersection} are either hyperplanes of $\mf t_\phi(g_\phi,\zeta_r)$,
	or they coincide with it;
	but since~\eqref{eq:pure_general_twisted_deformation_space_1_coeff} is nonempty%
	---it contains $A$---%
	the former holds.
\end{proof}

\begin{rema}
	\label{rmk:eigenspace_intersection_on_flat}

	Analogously to~\eqref{eq:eigenspace_intersection},
	the equality $g^r(A) = A$ yields $(g^r)_\phi = (g_\phi)^r = 1$ (cf.~Lem.~\ref{lem:about_pointwise_stabilizers}),
	and one has
	\begin{equation}
		\mf t_\phi(g_\phi,\zeta_r)
		= \bigcap_{(\mb Z \slash r \mb Z)^{\ts}} \mf t_\phi(g_\phi^i,\zeta_r^i)
		= \bigcap_{\mb Z \slash r \mb Z} \mf t_\phi(g_\phi^i,\zeta_r^i) \sse \mf t_\phi,
	\end{equation}
	generalizing the quasigeneric case.
\end{rema}

\begin{rema}
	\label{rmk:about_eigenspace_poset}

	One can rephrase Prop.-Def.~\ref{prop:pure_general_twisted_deformation_space_1_coeff},
    involving the natural `$r$-ramified' generalization of the reflection arrangement of $(\mf t,W)$.
	Precisely,
	as in~\cite{miller_2015_eigenspace_arrangements_of_reflection_groups} (amongst many others),
    and in a particular case of \S~\ref{sec:r_regularity},
	consider the whole set
	\begin{equation}
		\label{eq:eigenspace_poset}
		\on E(W,\zeta_r)
		\ceqq \Set{ E = \mf t(g,\zeta_r) | g \in W },
	\end{equation}
	of $\zeta_r$-eigenspaces of elements of $W$.
	It is a $W$-poset,
	coinciding with the intersection lattice $\mc F(W)$ when $r = 1$ (cf.~\cite{orlik_solomon_1980_unitary_reflection_groups_and_cohomology,broue_2001_reflection_groups_braid_groups_hecke_algebras_finite_reductive_groups}).
	Contrary to~\cite{miller_2015_eigenspace_arrangements_of_reflection_groups},
	and consistently with \S~\ref{sec:background_flats},
	we order~\eqref{eq:eigenspace_poset} by \emph{inclusion}:
	so the \emph{maximal} $\zeta_r$-eigenspaces are precisely the \emph{inclusion-maximal} ones.
	(They need not be regular.)
	Then Thm.~1.2 of op.~cit.~states that~\eqref{eq:eigenspace_poset} coincides with the set of intersections $E = E' \cap \mf t_\phi$,
	where $\phi \sse \Phi$ varies amongst Levi subsystems and the eigenspace $E' \in \on E(W,\zeta_r)$ varies amongst maximal ones.

	In our context,
	given an irregular type $\wh Q = Aw^{-1}$ and an element $g \in \mb L_r \bigl( \wh Q \bigr)$,
	up to changing $g$ one can assume~\cite[Thm.~3.4.~(ii)]{springer_1974_regular_elements_of_finite_reflection_groups} that $\mf t(g,\zeta_r) \in \on E(W,\zeta_r)$ is maximal.
	(By \S~\ref{sec:no_marking},
	this is w.l.o.g.~when describing $r$-admissible-deformation spaces.)
	Therefore,
	there exists $h \in W$ such that
	\begin{equation}
		\mf t_\phi(g_\phi,\zeta_r) = \mf t(h,\zeta_r) \in \on E(W,\zeta_r).
	\end{equation}

	In conclusion,
	the collection~\eqref{eq:pure_general_twisted_deformation_space_1_coeff} of $(g,r)$-admissible-deformation spaces corresponds to removing suitable traces of reflection hyperplanes of $W$ from the elements of~\eqref{eq:eigenspace_poset}.
	This is relevant for the parameterization of monolevel \emph{modular groups} in gauge theory;
	cf.~instead Rmk.~\ref{rmk:about_eigenspace_chain} for the general multilevel case,
	using the more general statement of~\cite[Thm.~3.1]{koonin_topology_of_eigenspace_posets_for_unitary_reflection_groups} (cf.~also~\cite{koonin_2014_topology_of_eigenspace_posets_for_imprimitive_reflection_groups}).
\end{rema}

\subsection{Higher pole order}
\label{sec:general_pure_twisted_adm_def_space}

Finally,
we treat the general \emph{multilevel} case,
where~\eqref{eq:explicit_irregular_type} is arbitrary.

The direct-product decomposition of $\bm B \bigl( \wh Q \bigr)$,
already invoked in \S~\ref{sec:generic_case_several_coeff},
involves a filtration of $\Phi$ by Levi subsystems (cf.~\cite[\S~3.4]{goresky_kottwitz_macpherson_2009_codimensions_of_root_valuation_strata},
and~\cite[\S~4]{doucot_rembado_tamiozzo_2026_local_wild_mapping_class_groups_and_cabled_braids},
and~\cite[Def.~3.1.2]{calaque_felder_rembado_wentworth_wild_orbits_and_generalised_singularity_modules_stratifications_and_quantisation}).
Namely,
let $\phi_i \sse \Phi$ be the intersection of $\Phi$ with the annihilator $\set{A_i,\dc,A_s}^\perp \sse \mf t^{\dual}$,
for $i \in \set{1,\dc,s}$,
whence:
(i) there are nested inclusions
\begin{equation}
	\label{eq:levi_filtration}
	\phi_1 \sse \dm \sse \phi_s \sse \phi_{s+1}
	\ceqq \Phi;
\end{equation}
(ii) $\phi_i$ is a Levi subsystem of $\phi_{i+1},\dc,\phi_s$;
and (iii) $\phi_1,\dc,\phi_s$ are all Levi subsystems of $\Phi$.
We typically write $\bm\phi = (\phi_1,\dc,\phi_s)$,
dropping the redundant information of the ambient root system $\phi_{s+1} = \Phi$.

\begin{defi}
	The sequence~\eqref{eq:levi_filtration} is the \emph{Levi-subsystem filtration of} $\Phi$ \emph{underlying} $\wh Q$,
	which is symbolized by $\bm\phi = \bm\phi \bigl( \wh Q \bigr)$.
\end{defi}

\begin{rema}
	\label{rmk:intrinsic_flag}

	As already mentioned in the proof~\ref{proof:lem_intrinsicality},
	the sequence $\bm\phi$ is invariant under changes of the uniformizer-root $w$:
	indeed,
	it can be intrinsically defined as follows.

	Let $\wh Q$ be as in~\eqref{eq:intrinsic_irr_type},
	with pole order $s \geq 1$,
	and consider the tower of $\ms O$-linear surjections
	\begin{equation}
		\mf m^{-s} \bs \ms O \lthra \mf m^{-s} \bs \mf m^{-1} \lthra \dm \lthra \mf m^{-s} \bs \mf m^{1-s} \lthra \mf m^{-s} \bs \mf m^{-s}
		= (0),
	\end{equation}
	with increasing denominators.
	Tensoring with $\mf t$ provides $\mb C$-linear truncation maps
	\begin{equation}
		\label{eq:truncation_tower}
		\wh Q \lmt \tau_i\bigl(\wh Q\bigr) \in \mf t ( \mf m^{-s} ) \bs \mf t ( \mf m^{1-i} )
		\simeq \mf t \ots_{\mb C} \bigl( \mf m^{-s} \bs \mf m^{1-i} \bigr),
		\qquad i \in \set{1,\dc,s+1}.
	\end{equation}
	(This should be compared with~\eqref{eq:truncation},
	when $G$ is a classical simple group,
	and working on the base of the covering $\Spec \mb C(\!( w )\!) \thra \Spec \mb C(\!( z )\!)$ with $z = w^r$;
	etc.)

	Then any linear functional $\lambda \in \mf t^{\dual}$ extends to a sequence of $\mb C$-linear maps
	\begin{equation}
		\wt \lambda_i
		\ceqq \lambda \ots 1 \cl \mf t ( \mf m^{-s} ) \bs \mf t ( \mf m^{1-i} ) \lra \mf m^{-s} \bs \mf m^{1-i},
	\end{equation}
	and one can equivalently pose
	\begin{equation}
		\phi_i
		\ceqq \Set{ \alpha \in \Phi | \tau_i \bigl( \wh Q \bigr) \in \ker \bigl( \wt \alpha_i \bigr) }. \qedhere
	\end{equation}
\end{rema}

\subsubsection{}

Any Levi-subsystem filtration has an associated (decreasing) kernel-flag
\begin{equation}
	\label{eq:kernel_flag}
	\mf t_{\bm \phi} \ceqq \bigl( \mf Z
	= \mf t_{\phi_{s+1}} \sse \mf t_{\phi_s} \sse \dm \sse \mf t_{\phi_1} \sse \mf t \bigr),
\end{equation}
which is a flag of flats for the rational reflection group $\mb G = (\mf t,W)$.
Set then
\begin{equation}
	\label{eq:root_stratum_factorization}
	\bm B \bigl( \wh Q,i \bigr)
	\ceqq \mf t_{\phi_i} \, \bigsm \, \bigcup_{\phi_{i+1} \sm \phi_i} \,  H_\alpha(\phi_i),
	\qquad i \in \set{1,\dc,s},
\end{equation}
in the notation of~\eqref{eq:pure_untwisted_nongeneric_deformation_space_1_coeff}---%
generalizing~\eqref{eq:generic_factors_pure_deformation_space}.
This finally yields the (topological) factorization
\begin{equation}
	\bm B \bigl( \wh Q \bigr)
	= \prod_{i = 1}^s \bm B \bigl( \wh Q,i \bigr) \sse \mf t^s,
\end{equation}
as one can check that this is precisely the stratum of---%
coefficients of---%
irregular types with the same root-valuation tuple as $\wh Q$.

In the ramified case,
one establishes the following generalization:

\begin{enonce}{Theorem-Definition}
	\label{thm:general_pure_twisted_deformation_space}

	Choose $g \in \mb L_r \bigl( \wh Q \bigr)$.
	Then:
	\begin{enumerate}
		\item
		      the subspace $\mf t_{\phi_i} \sse \mf t$ is $g$-\emph{invariant},
		      for $i \in \set{1,\dc,s}$;

		\item
		      and moreover,
		      setting
		      \begin{equation}
			      \label{eq:general_g_twisted_deformation_factor}
			      \bm B_{g,r} \bigl( \wh Q,i \bigr)
			      \ceqq \mf t_{\phi_i} \bigl( g_{\phi_i},\zeta_r^i \bigr) \, \bigsm \, \bigcup_{\phi_{i+1} \sm \phi_i} \,  H_\alpha \bigl( g_{\phi_i},\zeta_r^i \bigr),
		      \end{equation}
		      in the notation of~\eqref{eq:pure_general_twisted_deformation_space_1_coeff}--\eqref{eq:hyperplane_intersection},
		      it follows that:
		      \begin{enumerate}
			      \item
			            there is a (topological) factorization
			            \begin{equation}
				            \label{eq:factorization_pure_multilevel}
				            \bm B_{g,r} \bigl( \wh Q \bigr)
				            = \prod_{i = 1}^s  \bm B_{g,r}\bigl( \wh Q,i\bigr) \sse \mf t^s;
			            \end{equation}

			      \item
			            and the spaces~\eqref{eq:general_g_twisted_deformation_factor} are \emph{nonempty} complex hyperplane complements.
		      \end{enumerate}
	\end{enumerate}
\end{enonce}

\begin{proof}
	The first statement is Lem.-Def.~\ref{lem:flag_stabilizer}.
	Then combine the proofs of Prop.~\ref{prop:generic_pure_twisted_deformation_space} +~Prop.-Def.~\ref{prop:pure_general_twisted_deformation_space_1_coeff}.
\end{proof}

\begin{enonce}{Lemma-Definition}
	\label{lem:flag_stabilizer}

	Consider the (parabolic) subgroup of $\GL_{\mb C}(\mf t)$ preserving the flag~\eqref{eq:kernel_flag},
	i.e.,
	\begin{equation}
		\label{eq:flag_parabolic}
		P_{\bm \phi}
		\ceqq \Set{ \bm g' \in \GL_{\mb C}(\mf t) | \bm g' (\mf t_{\phi_i}) \sse \mf t_{\phi_i} \text{ for } i \in \set{1,\dc,s} }
		= \bigcap_{i = 1}^s N_{\GL_{\mb C}(\mf t)}(\mf t_{\phi_i}),
	\end{equation}
	and let
	\begin{equation}
		\label{eq:flag_parabolic_in_weyl}
		N_W(\mf t_{\bm \phi})
		\ceqq W \cap P_{\bm\phi}
		= \bigcap_{i = 1}^s N_W(\mf t_{\phi_i}) \sse W.
	\end{equation}
	Then one has $\mb L_r \bigl( \wh Q \bigr) \sse N_W(\mf t_{\bm\phi})$.
\end{enonce}

\begin{proof}[Postponed to~\ref{proof:lem_flag_stabilizer}]
\end{proof}

\begin{rema}
	\label{rmk:levels_and_topology}

	Topologically speaking,
	the interesting factors of~\eqref{eq:factorization_pure_multilevel} correspond to \emph{strict} inclusions $\phi_i \ssne \phi_{i+1}$ of Levi subsystems (else $\bm B_{g,r}(\wh Q,i) = \mf t_{\phi_i} ( g_{\phi_i},\zeta_r^i )$ is \emph{contractible}).
	The number of such factors is bounded above by the rank of $\Phi$,
	i.e.,
	by the semisimple rank of $\mf g$---%
	regardless of the pole order $s \geq 1$.

	This is consistent:
	only the levels of $\wh Q$ should matter,
	and indeed one has $\phi_i \ssne \phi_{i+1}$ if and only if there exists $\alpha \in \Phi$ such that
	\begin{equation}
		\braket{ \alpha, A_i } \neq 0 = \braket{ \alpha, A_{i+1} } = \dm = \braket{ \alpha,A_s } \in \mb C,
	\end{equation}
	i.e.,
	if and only if $i \in \on{Levels}(\wh Q)$ (cf.~\eqref{eq:levels}).
\end{rema}

\begin{rema}
	\label{rmk:fibre_bundle_tower}

	Keep all notation from Rmk.~\ref{rmk:intrinsic_flag}:
	taking kernels in~\eqref{eq:truncation_tower} yields a flag
	\begin{equation}
		(0)
		= \ms O \bs \ms O
		\sse \mf m^{-1} \bs \ms O \sse \dm \sse \mf m^{1-s} \bs \ms O \sse \mf m^{-s} \bs \ms O,
	\end{equation}
	and we saw that the relevant notions do \emph{not} depend on a choice of splitting.

	Somewhat analogously,
	one can in principle define coordinate-free versions of `truncated' $(g,r)$-admissible-deformation spaces
	\begin{equation}
		\bm B_{g,r} \bigl( \tau_i \bigl( \wh Q \bigr) \bigr) \sse \mf t \ots_{\mb C} \bigl( \mf m^{-s} \bs \mf m^{1-i} \bigr),
		\qquad i \in \set{1,\dc,s+1},
	\end{equation}
	as in \S~\ref{sec:uniformizers} (using~\eqref{eq:truncation_tower}):
	these topological spaces are determined by the `truncated' root-valuation tuple (cf.~\S~\ref{sec:exponential_factors}),
	now taking values in $\set{i-1,\dc,s}$,
	where the minimal value corresponds to annihilating $\tau_i \bigl( \wh Q \bigr)$;
	equivalently,
	they are determined by the \emph{Levi-subsystem subfiltration} $(\phi_i,\dc,\phi_s,\Phi)$---%
	of $\bm\phi = \bm\phi \bigl( \wh Q \bigr)$.

	Note that this makes sense in the twisted/ramified setting,
	since:
	(i) the action of the Galois group of $\ms K \hra \ms L$ is by valuation-preserving automorphisms,
	whence $\mf m^j \sse \ms L$ is stable for all $j \in \mb Z$;
	and (ii) the `diagonal' $W$-action commutes with the truncations~\eqref{eq:truncation_tower}.
	This yields set-theoretic inclusions
	\begin{equation}
		\label{eq:truncated_orbit_generators}
		\mb L_r \bigl( \wh Q \bigr)
		= \mb L_r \bigl( \tau_1\bigl( \wh Q \bigr) \bigr) \sse \dm \sse \mb L_r \bigl( \tau_{s+1} \bigl( \wh Q \bigr) \bigr)
		= \mb L_r(0)
		= W,
	\end{equation}
	setting
	\begin{equation}
		\mb L_r \bigl( \tau_i\bigl( \wh Q \bigr) \bigr)
		\ceqq \Set{ g \in W | g \bigl( \tau_i \bigl( \wh Q \bigr) \bigr) = \zeta_r \bigl( \tau_i \bigl( \wh Q \bigr) \bigr) },
	\end{equation}
	so that a single $g \in \mb L_r \bigl( \wh Q \bigr)$ generates the $r$-Galois-orbits of all the truncations---%
	which are then $r$-Galois-closed.

	Finally,
	truncating $(g,r)$-admissible deformations of $\wh Q$ yields a tower of continuous maps
	\begin{equation}
		\label{eq:fibre_bundle_tower}
		\bm B_{g,r} \bigl( \wh Q \bigr) \lra \bm B_{g,r} \bigl( \tau_2 \bigl( \wh Q \bigr) \bigr) \lra \dm \lra \bm B_{g,r} \bigl( \tau_s \bigl( \wh Q \bigr) \bigr) \lra (0).
	\end{equation}
	Each arrow is a fibre bundle,
	in fact \emph{globally} trivializable,
	albeit noncanonically:
	a uniformizer-root $w$ yields a compatible sequence of global trivializations with fibres~\eqref{eq:general_g_twisted_deformation_factor},
	whence the factorization~\eqref{eq:factorization_pure_multilevel}---%
	upon unrolling all terms.
\end{rema}

\subsection{Reduction to the simple/irreducible case}

Just as in~\cite[\S~5.2]{doucot_rembado_tamiozzo_2026_local_wild_mapping_class_groups_and_cabled_braids},
the topology of the deformation spaces is determined by the examples where $\mf g$ is a \emph{simple} Lie algebra.
(This further motivates the detailed study of~\S\S~\ref{sec:type_A}--\ref{sec:D_trees},
in the classical setting.)

Indeed:

\begin{lemm}[cf.~\cite{doucot_rembado_tamiozzo_2026_local_wild_mapping_class_groups_and_cabled_braids}, Prop.~5.1]
	\label{lem:reduction_to_simple_case}

	Let $\mf g = \prod_i \mf I_i$ be a (finite) factorization into nonzero \emph{Lie-ideals} $\mf I_i \sse \mf g$,
	and consider the Cartan subalgebras $\mf t_i \ceqq \mf t \cap \mf I_i \sse \mf I_i$.
	Split the root system/Weyl group accordingly,
	viz.:
	(i) $\Phi = \bigcup_i \Phi_i$,
	where $\Phi_i \ceqq \Phi \cap \mf t_i^{\dual}$ (the union is disjoint);
	and (ii) $W = \prod_i W_i$,
	where $W_i \ceqq W(\mf I_i,\mf t_i)$.
	Let also $\wh Q$ be an $r$-Galois-closed irregular type,
	choose an element $g \in \mb L_r \bigl( \wh Q \bigr)$,
	and finally decompose (uniquely)
	\begin{equation}
		\label{eq:irr_type_splitting}
		\wh Q
		= \sum_i \wh Q_i,
		\quad g
		= \prod_i g_i,
		\qquad \wh Q_i \in \mf t_i \ots_{\mb C} ( \mc L \bs \mc L_{\geq 0} ),
		\quad g_i \in W_i.
	\end{equation}
	Then there is a (topological) factorization
	\begin{equation}
		\label{eq:reduction_to_simple_pure}
		\bm B_{g,r} \bigl( \wh Q \bigr)
		= \prod_i \bm B_{g_i,r} \bigl( \wh Q_i \bigr),
	\end{equation}
	in the canonical identification $\mf t^s \simeq \prod_i \mf t_i^s$.
\end{lemm}

\begin{proof}
	Each factor $W_i \sse W$ acts trivially on $\mf t \ominus \mf t_i \ceqq \bops_{j \neq i} \mf t_j \sse \mf t$,
	so that
	\begin{equation}
		\mf t(g,\zeta)
		= \prod_i \mf t_i(g_i,\zeta) \sse \prod_i \mf t_i
		= \mf t,
		\qquad \zeta \in \mb C^{\ts}.
	\end{equation}
	(We identify $\mf t_i^{\dual} \simeq (\mf t \ominus \mf t_i)^\perp \sse \mf t^{\dual}$.)~Therefore,
	the condition that $g \in \mb L_r \bigl( \wh Q \bigr)$ is equivalent to having $g_i \in \mb L_r \bigl( \wh Q_i \bigr)$---%
	for all $i$.
	Moreover,
	a subset $\phi \sse \Phi$ is a Levi subsystem if and only if this holds for the intersections $\phi \cap \Phi_i \sse \Phi_i$,
	and one then has $\mf t_\phi = \prod_i (\mf t_i)_{\phi \cap \Phi_i}$;
	and the traces of root-hyperplanes split as well.
	The result follows from the factorization~\eqref{eq:factorization_pure_multilevel},
	upon commuting products/intersections.
\end{proof}

\subsubsection{}

To reduce to the simple case,
apply first Lem.~\ref{lem:reduction_to_simple_case} to the Lie-algebra splitting $\mf g = \mf Z \ts [\mf g,\mf g]$,
so that the factor of $\bm B_{g,r} \bigl( \wh Q \bigr)$ corresponding to the Lie-centre is \emph{contractible} (cf.~Rmk.~\ref{rmk:abelian_case});
and then further apply the same lemma upon decomposing the semisimple part into simple ideals.

The caveat is that the Levi subsystems~\eqref{eq:levi_filtration} in general correspond to \emph{nonsemisimple} Levi factors of parabolic Lie subalgebras of $\mf g$---%
even when the latter is simple.
Therefore,
as usual,
the reductive setting is the natural one to do recursion by level (cf.~\cite[\S~7.2]{boalch_2014_geometry_and_braiding_of_stokes_data_fission_and_wild_character_varieties}).

\section{Forgetting the marking}
\label{sec:no_marking}

Here we show that the choice of generators for the $r$-Galois-orbits is actually immaterial.

\subsection{Two useful facts about vectors in general position}

Choose any vector $A \in \mf t$,
and consider the irregular type $\wh Q \ceqq A w^{-1}$.
Introduce also (as in~\eqref{eq:kernel_intersection}--\eqref{eq:pure_untwisted_nongeneric_deformation_space_1_coeff}):
(i) the Levi annihilator $\phi = \phi_A \sse \Phi$;
(ii) the kernel intersection $\mf t_\phi \sse \mf t$;
and (iii) the hyperplane complement $\bm B \bigl( \wh Q \bigr) \sse \mf t_\phi$.

The standard statements of Lemm.~\ref{lem:about_setwise_stabilizers}--\ref{lem:about_pointwise_stabilizers} are used many times in this text:

\begin{lemm}[cf.~\cite{doucot_rembado_2025_topology_of_irregular_isomonodromy_times_on_a_fixed_pointed_curve}, Lem.~3.1]
	\label{lem:about_setwise_stabilizers}

	For any $g \in W$,
	the following are equivalent:
	\begin{enumerate}
		\item
		      one has $g \in N_W(\mf t_\phi)$,
		      in the notation of~\eqref{eq:stratum_stabilizer};

		\item
		      one has $g(\phi) = \phi$,
		      in the Weyl-group action on $\mf t^{\dual}$;
		      (Cf.~\S~\ref{sec:complex_real_rational_reflection_groups}.)

		\item
		      one has $g(\phi) = \phi$ and $g(\Phi \sm \phi) = \Phi \sm \phi$;

		\item
		      one has $g \bigl( \bm B\bigl( \wh Q \bigr) \bigr) \sse \bm B \bigl( \wh Q \bigr)$;

		\item
		      one has $g \bigl( \bm B (\wh Q) \bigr) \sse \mf t_\phi$;

		\item
		      there exists $A' \in \bm B \bigl( \wh Q \bigr)$ such that $g(A') \in \mf t_\phi$;

		\item
		      and there exists $A' \in \bm B \bigl( \wh Q \bigr)$ such that $g(A') \in \bm B \bigl( \wh Q \bigr)$.
	\end{enumerate}
\end{lemm}

\begin{proof}[Proof postponed to~\ref{proof:lem_about_setwise_stabilizers}]
\end{proof}

\begin{lemm}[cf.~\cite{doucot_rembado_2025_topology_of_irregular_isomonodromy_times_on_a_fixed_pointed_curve}, Lem.~3.2]
	\label{lem:about_pointwise_stabilizers}

	For any $g \in W$,
	the following are equivalent:
	\begin{enumerate}
		\item
		      the element $g$ acts as the identity on $\mf t_\phi$;

		\item
		      the element $g$ acts as the identity on $\bm B \bigl( \wh Q \bigr)$;

		\item
		      and there exists $A' \in \bm B \bigl( \wh Q \bigr)$ such that $g(A') = A'$.
	\end{enumerate}
\end{lemm}

\begin{proof}[Proof postponed to~\ref{proof:lem_about_pointwise_stabilizers}]
\end{proof}

\begin{rema}
	\label{rmk:general_position}

	Let $V' \sse \mf t$ be any vector subspace.
	Recall that a vector $A \in V'$ is said to be \emph{in general position in} $V'$,
	\emph{with respect to} $W$,
	if
	\begin{equation}
		\label{eq:weyl_parabolic}
		W_{V'}
		= W_A,
		\qquad W_A
		\ceqq \Set{ g \in W | g(A) = A },
	\end{equation}
	introducing the pointwise-stabilizers (which are parabolic subgroups,
	cf.~\eqref{eq:abstract_parabolic_subgroup}).
	Equivalently,
	for any root-hyperplane $H_\alpha \sse \mf t$ ($\alpha \in \Phi$) one has $A \in H_\alpha$ if and only if $V' \sse H_\alpha$,
	which can be symbolized by $\ol{\mb C \cdot A}^{\mathsmaller W} = \ol{V'}^{\mathsmaller W}$---%
	invoking the flat-closure.
	With this terminology,
	the untwisted admissible-deformation space $\bm B \bigl( \wh Q \bigr)$ consists of points/vectors in general position in the flat $\mf t_\phi$.

	As already noted,
	the parabolic subgroup~\eqref{eq:weyl_parabolic} coincides with the Weyl group of $(L,T)$,
	involving the connected reductive Adjoint stabilizer:
	\begin{equation}
		\label{eq:levi_subgroup}
		L
		= G^A \ceqq \Set{ g \in G | \Ad_g(A) = A }.
	\end{equation}
	(Cf.~the proof~\ref{proof:lem_about_pointwise_stabilizers}.)
\end{rema}

\subsection{Quasigeneric case}

\begin{prop}
	\label{prop:unique_orbit_generater}

	Consider a homogeneous irregular type $\wh Q = A w^{-1} \in \smash{\wh{\mc{IT}}}^{\mathsmaller\leq 1}_r$,
	with $A \in \mf t_{\reg}$.
	Then:
	\begin{enumerate}
		\item
		      there exists a \emph{unique} group element $g_A \in W$ generating the $r$-Galois-orbit of $\wh Q$;

		\item
		      one has $\bm B_r \bigl( \wh Q \bigr) = \bm B_{g_A,r} \bigl( \wh Q \bigr)$;

		\item
		      and the same holds if $\wh Q \in \smash{\wh{\mc{IT}}}_r^{\mathsmaller\leq s}$,
		      for any $s \geq 1$,
		      provided that the leading coefficient is a regular vector.
	\end{enumerate}
\end{prop}

\begin{proof}
	The parabolic subgroup~\eqref{eq:weyl_parabolic} is trivial,
	and if $g,g' \in W$ dilate $A$ by the same scalar then $g^{-1}g' \in W_A$;
	the first statement follows.
	For the second statement,
	one now has $\wh Q \sim_r \wh Q'$ if and only if $\wh Q \sim_{g_A,r} \wh Q'$.

	The extension to arbitrary pole order $s \geq 1$ works verbatim,
	considering the trivial parabolic subgroup $W_{A_s} \sse W$ defined by the leading coefficient $A_s \in \mf t_{\reg}$.
\end{proof}

\subsection{General case}
\label{sec:no_marking_general_case}

Choose again $\wh Q = A w^{-1} \in \smash{\wh{\mc{IT}}}^{\mathsmaller\leq 1}_r$ homogeneous,
now with no constraint on $A \in \mf t$.
The set of elements generating the $r$-Galois-orbit of $\wh Q$ is in general \emph{not} a singleton;
nonetheless:

\begin{prop}
	\label{prop:no_marking_general}

	Suppose that $g,g' \in \mb L_r \bigl( \wh Q \bigr)$.
	Then:
	\begin{enumerate}
		\item
		      one has $\bm B_{g,r} \bigl( \wh Q \bigr) = \bm B_{g',r} \bigl( \wh Q \bigr)$;

		\item
		      and the same holds in the general case---%
		      when $s \geq 1$ is arbitrary.
	\end{enumerate}
\end{prop}

\begin{proof}
	The hypothesis yields $\mf t_\phi(g_\phi,\zeta_r) = \mf t_\phi(g'_\phi,\zeta_r) \sse \mf t_\phi$,
	in the notation of Prop.-Def.~\ref{prop:pure_general_twisted_deformation_space_1_coeff} (cf.~Lem.~\ref{lem:no_marking}).
	Then one also has
	\begin{equation}
		H_\alpha(g_\phi,\zeta_r)
		= H_\alpha(g'_\phi,\zeta_r),
		\qquad \alpha \in \Phi \sm \phi,
	\end{equation}
	and the first statement follows from the equality~\eqref{eq:pure_general_twisted_deformation_space_1_coeff}.

	For a general irregular type $\wh Q \in \smash{\wh{\mc{IT}}}_r^{\mathsmaller\leq s}$,
	Lem.~\ref{lem:no_marking} can be applied recursively,
	starting from the leading coefficient $A_s$ of $\wh Q$---%
	in the notation of~\eqref{eq:explicit_irregular_type}.
	Precisely,
	let $\wt \phi_i \ceqq \phi_{A_i} \sse \Phi$ be the Levi annihilator of each coefficient,
	so that the relevant Levi-subsystem filtration of $\Phi$ consists of the intersections~\eqref{eq:levi_filtration_intersection}.
	(Cf.~\eqref{eq:levi_filtration}.)
	Now choose $g,g' \in \mb L_r \bigl( \wh Q \bigr)$,
	whence in particular
	\begin{equation}
		g(A_s)
		= g'(A_s)
		= \zeta_r^s A_s \in \mf t.
	\end{equation}
	It follows that $g$ and $g'$ act in the same way on the---%
	stable---%
	flat $\mf t_{\phi_s} = \mf t_{\wt \phi_s} \sse \mf t$.
	Analogously,
	$g$ and $g'$ act in the same way on $\mf t_{\wt \phi_{s-1}}$,
	and so also on the larger vector subspace
	\begin{equation}
		\label{eq:flat_sum}
		V'
		\ceqq \mf t_{\wt \phi_s} + \mf t_{\wt \phi_{s-1}} \sse \mf t.
	\end{equation}
	The caveat is that the latter is contained in $\mf t_{\phi_{s-1}}$,
	possibly \emph{properly};
	but the element $h \ceqq g^{-1}g' \in W$ lies in the parabolic subgroup of~\eqref{eq:flat_sum},
	and this coincides with the pointwise-stabilizer of the flat-closure
	\begin{equation}
		\ol{V'}^{\mathsmaller W}
		= \mf t_{\wt \phi_s \cap \wt \phi_{s-1}}
		= \mf t_{\phi_{s-1}}.
	\end{equation}
	(Cf.~\S~\ref{sec:levis_and_flats}.)
	Iterating,
	the consecutive restrictions of $g$ and $g'$ on the vector spaces of the flag $\mf t_{\phi_s} \sse \dm \sse \mf t_{\phi_1}$ are all the same.
	In particular,
	one has
	\begin{equation}
		\bm B_{g,r} \bigl( \wh Q,i \bigr)
		= \bm B_{g',r} \bigl( \wh Q,i \bigr),
		\qquad i \in \set{1,\dc,s},
	\end{equation}
	in the notation of~\eqref{eq:general_g_twisted_deformation_factor}.
	The conclusion follows from the factorization~\eqref{eq:factorization_pure_multilevel}.
\end{proof}

\begin{lemm}
	\label{lem:no_marking}

	Choose $g,g' \in W$ such that $A \in \mf t(g,\zeta) \cap \mf t(g',\zeta)$,
	for some $\zeta \in \mb C^{\ts}$.
	Then $g_\phi = g'_\phi \in \GL_{\mb C}(\mf t_\phi)$,
	in the notation of Prop.-Def.~\ref{prop:pure_general_twisted_deformation_space_1_coeff}.
\end{lemm}

\begin{proof}[Proof postponed to~\ref{proof:lem_no_marking}]
\end{proof}

\begin{rema}
	\label{rmk:toral_fission}

	The proof of Prop.~\ref{prop:no_marking_general} relies on the fact that the subset $\mb L_r \bigl( \wh Q \bigr) \sse W$ is a torsor for the right-multiplication/-translation action of the parabolic subgroup $W^{\wh Q} \ceqq W_{A_1} \cap \dm \cap W_{A_s} \sse W$ stabilizing the irregular type.
	(It is a $W^{\wh Q}$-\emph{bitorsor} if one also acts on the left,
	but in this text we only use \emph{left} cosets.)
	The group $W^{\wh Q}$ is also the Weyl group of the connected reductive stabilizer $L = \smash G^{\wh Q} \sse G$ of
	---%
	the coefficients of---%
	the irregular type (with respect to the maximal torus $T$;
	cf.~\S~\ref{sec:independence_r_root}).
	It is obtained at the end of the `fission' chain,
	iterating~\eqref{eq:weyl_parabolic} (cf.~the proof of Thm.-Def.~\ref{thm:full_general_twisted_deformation_space}).

	One always has $T \sse L$,
	and $\mb L_r \bigl( \wh Q \bigr)$ is a singleton if and only if $L = T$:
	in that case we say that fission is `complete/toral' (cf.~\cite{boalch_2009_through_the_analytic_halo_fission_via_irregular_singularities}).
	Beware that one \emph{cannot} assume that $L = T$ in our setting,
	and so in particular one can have $\mb L_r \bigl( \wh Q \bigr) \neq \set{1}$ even if $r = 1$ (cf.~Rmk.~\ref{rmk:1_galois_closed}).
\end{rema}

\subsubsection{}

It follows from Prop.~\ref{prop:no_marking_general} that $\bm B_r \bigl( \wh Q \bigr) = \bm B_{g,r} \bigl( \wh Q \bigr)$,
for \emph{any} $g \in \mb L_r \bigl( \wh Q \bigr)$.
So the union~\eqref{eq:deformation_spaces_union} is trivial and the $r$-admissible-deformation spaces are---%
path-connected---%
products of complex hyperplane complements (cf.~\eqref{eq:factorization_pure_multilevel}).

All of this will be later rephrased in terms of nested sequences of Levi subcosets of $\mb G = (\mf t,W)$,
which automatically neglect the choice of a generator $g$ (cf.~\S\S~\ref{sec:howlett}--\ref{sec:modular_groups_are_twisted_cabled_weyl}).

\section*{First interlude}

So far we worked in the \emph{pure}---%
marked---%
setting,
culminating in Thm.-Def.~\ref{thm:general_pure_twisted_deformation_space};
and \S~\ref{sec:no_marking} shows that the choice of generators for the $r$-Galois-orbits is immaterial.
The next aim is to treat the \emph{full/nonpure} case,
replacing $r$-Galois-closed irregular types with their underlying irregular classes.
Precisely:
(i) in \S~\ref{sec:nonpure_case_generic} we treat the quasigeneric case;
and (ii) in \S~\ref{sec:nonpure_case_general} we treat the general/multilevel case.

This involves groups of deck transformations,
which will be later (re)described in terms of nested sequences of Levi reflection cosets (cf.~\S\S~\ref{sec:howlett}--\ref{sec:modular_reparameterization}).

\section{Full/nonpure setting:
  quasigeneric case}
\label{sec:nonpure_case_generic}

\subsection{}

Let us consider the $r$-admissible deformations of a quasigeneric $r$-Galois-closed irregular class $\wh\Theta = \wh\Theta \bigl( \wh Q \bigr)$.
In view of \S\S~\ref{sec:generic_pure_case}--\ref{sec:no_marking},
one must describe how the $W$-orbits of irregular types intersect $\bm B_{g,r} \bigl( \wh Q \bigr)$,
for any $g \in \mb L_r\bigl( \wh Q \bigr)$,
and so:
(i) we show that these intersections are orbits for the setwise-stabilizers of the deformation space;
(ii) we determine the stabilizers of all points of $\bm B_{g,r} \bigl( \wh Q \bigr)$,
showing that they are constant and coincide with the pointwise-stabilizer of the deformation space;
and (iii) we conclude that the corresponding topological quotient leads to a Galois covering $\bm B_{g,r} \bigl( \wh Q \bigr) \thra \bm B_r \bigl( \wh\Theta \bigr)$---%
à la Steinberg.
(The same exact logic will apply to the multilevel case.)

\subsection{Minimal pole order}

Assume again that $\wh Q = A w^{-1}$ is homogeneous,
with $A \in \mf t_{\reg}$.

\subsubsection{}

The untwisted admissible-deformation space $\bm B \bigl( \wh Q \bigr) = \mf t_{\reg} \sse \mf t$ is $W$-invariant,
and the Weyl group acts \emph{freely} thereon.
This yields a Galois covering $\bm B \bigl( \wh Q \bigr) \thra \bm B \bigl( \wh Q \bigr) \bs W$,
whose base can be identified with $\bm B \bigl( \wh\Theta \bigr)$.

In the twisted case,
instead:

\begin{enonce}{Proposition-Definition}
	\label{prop:full_generic_twisted_deformation_space_1_coeff}

	Denote by $g = g_A$ the \emph{unique} element of $\mb L_r\bigl( \wh Q \bigr)$ (cf.~Prop.~\ref{prop:unique_orbit_generater}).
	Then:

	\begin{enumerate}
		\item
		      the \emph{centralizer of} $g$ \emph{in} $W$,
		      viz.,
		      the subgroup
		      \begin{equation}
			      Z_W(g) \ceqq \Set{ g' \in W | gg' = g'g },
		      \end{equation}
		      acts (faithfully) on $\mf t(g,\zeta_r)$ as the \emph{complex reflection group} of the arrangement~\eqref{eq:generic_twisted_deformation_space_1_coeff};

		\item
		      there is a \emph{Galois covering}
		      \begin{equation}
			      \label{eq:full_generic_twisted_deformation_space_1_coeff}
			      \bm B_r \bigl( \wh Q \bigr)
			      \lthra \bm B_r \bigl( \wh Q \bigr) \bs Z_W(g);
		      \end{equation}

		\item
		      and there is a \emph{homeomorphism} $\bm B_r \bigl( \wh\Theta \bigr) \simeq \bm B_r \bigl( \wh Q \bigr) \bs Z_W(g)$.
	\end{enumerate}
\end{enonce}

\begin{proof}
	Choose $A' \in \mf t_{\reg} \cap \mf t(g,\zeta_r)$,
	so that the irregular type $\wh Q' \ceqq A' w^{-1}$ lies in $\bm B_{g,r} \bigl( \wh Q \bigr)$ (cf.~\eqref{eq:generic_twisted_deformation_space_1_coeff}).
	If $g' \in W$ satisfies $g'(A') \in \bm B_{g,r} \bigl( \wh Q \bigr)$,
	in particular one has $g'(A') \in \mf t(g,\zeta_r)$:
	Lem.~\ref{lem:restricted_generic_twisted_orbit} now implies that $g' \in Z_W(g) = N_W(g,\zeta_r)$.
	Conversely,
	if $g' \in Z_W(g)$ then $g'$ preserves $\bm B_{g,r} \bigl( \wh Q \bigr) \simeq \mf t_{\reg} \cap \mf t(g,\zeta_r)$.
	Hence,
	the intersection of the $W$-orbit of $\wh Q'$ with $\bm B_{g,r} \bigl( \wh Q \bigr)$ coincides with the $Z_W(g)$-orbit of $\wh Q'$.

	Moreover,
	no nontrivial element of $W$ fixes any element of $\mf t_{\reg}$,
	and so the resulting $Z_W(g)$-action remains \emph{free}.
	This yields the Galois covering~\eqref{eq:full_generic_twisted_deformation_space_1_coeff}.
	(All the topological spaces involved are Hausdorff,
	path-connected,
	locally path-connected,
	and semilocally simply-connected,
	cf.~\cite[\S~1.3]{hatcher_2002_algebraic_topology}.)

	The statement about complex reflections follows from~\cite[Prop.~3.5]{springer_1974_regular_elements_of_finite_reflection_groups}.
\end{proof}

\begin{lemm}[cf.~Lem.~\ref{lem:restricted_commutators}]
	\label{lem:restricted_generic_twisted_orbit}

	Choose an element $g \in W$,
	and suppose that it admits a regular eigenvector $A' \in \mf t(g,\zeta) \cap \mf t_{\reg}$,
	for some $\zeta \in \mb C^{\ts}$.
	Then the following conditions are \emph{equivalent},
	for any other element $g' \in W$:
	\begin{enumerate}
		\item
		      one has $g' \in Z_W(g)$;

		\item
		      the element $g'$ lies in the setwise-stabilizer of the eigenspace~\eqref{eq:regular_eigenspace},
		      i.e.,
		      in the subgroup
		      \begin{equation}
			      \label{eq:eigenspace_normalizer}
			      N_W(g,\zeta)
			      = N_W \bigl( \mf t(g,\zeta) \bigr) \ceqq \Set{ g'' \in W | g'' \bigl( \mf t(g,\zeta) \bigr) \sse \mf t(g,\zeta) };
		      \end{equation}

		\item
		      and one has $g'(A') \in \mf t(g,\zeta)$.
	\end{enumerate}
\end{lemm}

\begin{proof}
	This is the `absolute' case of Lem.~\ref{lem:restricted_commutators},
	taking $\phi = \vn$.
\end{proof}

\subsubsection{}

The deformation space $\bm B_r \bigl( \wh\Theta \bigr)$ also admits a different description,
based on Def.~\ref{def:galois_closed_irr_types}:

\begin{lemm}[cf.~\cite{bessis_2015_finite_complex_reflection_arrangements_are_k_pi_1}, Thm.~1.9]
	\label{lem:deformations_as_invariants}

	If $A \in \mf t(g,\zeta_r) \cap \mf t_{\reg}$,
	and if $\wh\Theta = \wh\Theta \bigl( \wh Q \bigr)$ is the irregular class of $\wh Q \ceqq Aw^{-1}$,
	then one has
	\begin{equation}
		\bm B_r \bigl( \wh\Theta \bigr)
		= \smash{\bm B \bigl( \wh\Theta \bigr)}^{\! \mb Z \slash r\mb Z} \sse \mb A,
	\end{equation}
	setting $\mb A \ceqq \mb A_1 = \mf t \bs W$ (in a particular case of~\eqref{eq:irr_class_coefficients}).
\end{lemm}

\begin{proof}
	If $\wh\Theta \sim_r \wh\Theta' = \wh\Theta \bigl( \wh Q' \bigr)$,
	for some irregular type $\wh Q'$,
	then:
	(i) the latter is $r$-Galois-closed;
	and (ii) one has $\wh Q' \in \mf t_{\reg}$.
	This yields
	\begin{equation}
		\bm B_r \bigl( \wh\Theta \bigr) \sse \bigl( \mf t_{\reg} \bs W \bigr)^{\mb Z \slash r\mb Z} = \bm B \bigl( \wh\Theta \bigr)^{\mb Z \slash r\mb Z}.
	\end{equation}

	Conversely,
	suppose that $\wh\Theta' = \wh\Theta \bigl( \wh Q' \bigr) \in \smash{\bm B \bigl( \wh\Theta \bigr)}^{\mb Z \slash r\mb Z}$,
	where $\wh Q' = A'w^{-1}$,
	for some $A' \in \mf t$:
	this means that $A' \in \mf t_{\reg} \cap \mf t(g',\zeta_r)$,
	for suitable $g' \in W$.
	It follows that $g,g' \in W$ are regular elements,
	so that their $\zeta_r$-eigenspaces are (inclusion-)maximal in~\eqref{eq:eigenspace_poset} by~\cite[Thm.~4.2~(ii)]{springer_1974_regular_elements_of_finite_reflection_groups}.
	In turn,
	by Thm.~3.4~(iii) of op.~cit.,
	there exists $h \in W$ such that $h(A') \in \mf t(g,\zeta_r)$,
	so that $\wh Q \sim_r h\bigl( \wh Q' \bigr)$,
	and finally $\wh\Theta \sim_r \wh\Theta'$.
\end{proof}

\subsection{Higher pole order}

As in the pure case,
the situation is essentially the same when $\wh Q = \sum_{i = 1}^s A_iw^{-i}$ has higher pole order $s \geq 1$:

\begin{prop}
	\label{Prop:full_generic_twisted_deformation_space}

	Suppose that $A_s \in \mf t_{\reg}$,
	and denote by $g = g_{A_s} \in W$ the---%
	regular---%
	element determined by $g(A_s) = \zeta_r^s A_s$.
	Then:
	\begin{enumerate}
		\item
		      there is a homeomorphism $\bm B_r \bigl( \wh\Theta \bigr) \simeq \bm B_{g,r} \bigl( \wh Q \bigr) \bs Z_W(g)$,
		      and $\bm B_r \bigl( \wh\Theta \bigr)$ becomes the base of a Galois covering analogous to~\eqref{eq:full_generic_twisted_deformation_space_1_coeff};

		\item
		      and $\bm B_r \bigl( \wh\Theta \bigr)$ has the homotopy type of $U \bs Z_W(g)$ (cf.~Prop.~\ref{prop:generic_pure_twisted_deformation_space}).
	\end{enumerate}
\end{prop}

\begin{proof}
	For the first statement,
	observe that the admissible-deformation space $\bm B \bigl( \wh Q \bigr) = \mf t^{s-1} \ts \mf t_{\reg}$ is (still) $W$-invariant,
	and its subspace $\bm B_r \bigl( \wh Q \bigr) = \bm B_{g,r} \bigl( \wh Q \bigr)$ is stabilized by the centralizer $Z_W(g) \sse W$ (cf.~\eqref{eq:generic_g_admissible_deformation_space}).
	Conversely,
	using again Lem.~\ref{lem:restricted_generic_twisted_orbit},
	the $W$-orbit of any $r$-admissible deformation of $\wh Q$ intersects $\bm B_r \bigl( \wh Q \bigr)$ in a $Z_W(g)$-orbit.
	The free action of the centralizer then yields a Galois covering $\bm B_r \bigl( \wh Q \bigr) \thra \bm B_r \bigl( \wh\Theta \bigr)$,
	with Galois group $Z_W(g)$.

	For the second statement,
	the projection $\bm B_r \bigl( \wh Q \bigr) \thra U = \bm B_{g,r} \bigl( \wh Q,s \bigr)$ induces a continuous map $\bm B_r \bigl( \wh\Theta \bigr) \lthra U \bs Z_W(g)$.
	Then the `zero' section,
	viz.,
	\begin{equation}
		U \lhra \bm B_r \bigl( \wh Q \bigr),
		\qquad A'_s
		\lmt ( \! \underbrace{0,\dc,0}_{s-1 \text{ times }},A'_s),
	\end{equation}
	induces a homotopy-inverse.
\end{proof}

\section{Full/nonpure setting:
  general case}
\label{sec:nonpure_case_general}

In conclusion,
the results of \S\S~\ref{sec:generic_pure_case} +~\ref{sec:nonpure_case_generic} can be summarized by stating the---%
verbatim---%
generalization of Cor.~\ref{cor:modular_reflection_groups}~(1.) for any \emph{reductive} Lie algebra $\mf g$.

We now move on again to the nonquasigeneric case,
in the full/nonpure setting.

\subsection{Minimal pole order}

To treat monolevel examples,
choose once more a homogeneous irregular type $\wh Q = A w^{-1}$,
with arbitrary $A \in \mf t$.
Introduce the Levi subsystem $\phi = \phi_A \sse \Phi$ determined by $A$,
and the kernel intersection $\mf t_\phi \sse \mf t$ of~\eqref{eq:kernel_intersection}.

Now Lem.~\ref{lem:restricted_generic_twisted_orbit} has a relative/subregular extension,
so that one can prove the following:

\begin{enonce}{Proposition-Definition}
	\label{prop:full_general_twisted_deformation_space_1_coeff}

	Consider the largest quotient of~\eqref{eq:stratum_stabilizer} acting faithfully on $\mf t_\phi$,
	i.e.,
	the \emph{setwise-modulo-pointwise-stabilizer}
	\begin{equation}
		\label{eq:effective_weyl_group}
		W(\mf t_\phi)
		= N_W(\mf t_\phi) \bs W_{\mf t_\phi},
		\qquad W_{\mf t_\phi}
		\ceqq \Set{ g \in N_W(\mf t_\phi) | g_\phi = \Id_{\mf t_\phi} },
	\end{equation}
	in the notation of Prop.-Def.~\ref{prop:pure_general_twisted_deformation_space_1_coeff}.
	Moreover,
	choose $g \in \mb L_r \bigl( \wh Q \bigr)$,\fn{
		This is w.l.o.g.~(cf.~\S~\ref{sec:no_marking}).
		We tacitly identify the restriction $g_\phi = \eval[1]g_{\mf t_\phi} \in \GL_{\mb C}(\mf t_\phi)$ with the class of $g$ in $W(\mf t_\phi)$.}~and introduce the centralizer
	\begin{equation}
		\label{eq:centralizer_in_quotient}
		Z_{W,\phi}(g)
		\ceqq Z_{W(\mf t_\phi)} (g_\phi) \sse W(\mf t_\phi).
	\end{equation}
	Then:
	\begin{enumerate}
		\item
		      there is a \emph{Galois covering}
		      \begin{equation}
			      \label{eq:monolevel_galois_covering}
			      \bm B_{g,r} \bigl( \wh Q \bigr) \lthra \bm B_{g,r} \bigl( \wh Q \bigr) \bs Z_{W,\phi}(g);
		      \end{equation}

		\item
		      and one has a \emph{homeomorphism} $\bm B_r \bigl( \wh\Theta \bigr) \simeq \bm B_{g,r} \bigl( \wh Q \bigr) \bs Z_{W,\phi}(g)$.
	\end{enumerate}
\end{enonce}

(Hereafter,
we simplify the notation and set $W(\phi) \ceqq W(\mf t_\phi)$.)

\begin{proof}
	Choose a $(g,r)$-admissible deformation $\wh Q' \in \bm B_{g,r} \bigl( \wh Q \bigr)$,
	determined by a vector $A' \in \mf t_\phi(g_\phi,\zeta_r)$ which is in general position in $\mf t_\phi$.
	Now let $g' \in W$ be such that $g' \bigl( \wh Q' \bigr) \in \bm B_{g,r} \bigl( \wh Q \bigr)$:
	in particular $g'(A') \in \bm B\bigl( \wh Q \bigr)$,
	whence $g' \in N_W(\mf t_\phi)$ (cf.~Lem.~\ref{lem:about_setwise_stabilizers}).
	Moreover,
	the given condition also implies that
	\begin{equation}
		g'(A')
		= g'_\phi(A') \in \mf t_\phi(g_\phi,\zeta_r),
	\end{equation}
	and so the class/restriction $g'_\phi$ commutes with $g_\phi$ (cf.~Lem.~\ref{lem:restricted_commutators}),
	noting that $\phi_{A'} = \phi$ (for if not,
	one would have $\wh Q' \not\sim \wh Q$,
	and a fortiori $\wh Q' \not\sim_{g,r} \wh Q$).
	Conversely,
	denote by $p_\phi \cl N_W(\mf t_\phi) \thra W(\phi)$ the canonical projection.
	Then the subgroup $p_\phi^{-1} \bigl( Z_{W,\phi}(g) \bigr) \sse N_W(\mf t_\phi)$\fn{
		It can be equivalently defined as the normalizer---%
		in $N_W(\mf t_\phi)$---%
		of the Levi reflection subcoset $g W_{\mf t_\phi} \sse N_W(\mf t_\phi)$ of $(\mf t,W)$ (cf.~Corr.~\ref{cor:levi_subcoset} +~\ref{cor:twisted_weyl_from_coset_normalizers}).}~preserves the deformation space,
	in view of~\eqref{eq:pure_general_twisted_deformation_space_1_coeff}:
	it preserves the flat $\mf t_\phi \sse \mf t$,
	as well as the restricted eigenspace $\mf t_\phi(g_\phi,\zeta_r) \sse \mf t_\phi$,
	and it permutes the traces $H_\alpha(g_\phi,\zeta_r) \sse \mf t_\phi(g_\phi,\zeta_r)$ of the `complementary' root-hyperplanes---%
	for $\alpha \in \Phi \sm \phi$.

	Now the quotient $Z_{W,\phi}(g) \simeq p_\phi^{-1} \bigl( Z_{W,\phi}(g) \bigr) \bs W_{\mf t_\phi}$ acts naturally on $\bm B_{g,r} \bigl( \wh Q \bigr)$,
	and this action is free.
	Indeed,
	if $g' \in W$ fixes any point of $\bm B \bigl( \wh Q \bigr) \sse \mf t_\phi$ then $g' \in W_{\mf t_\phi}$ (cf.~Lem.~\ref{lem:about_pointwise_stabilizers});
	therefore,
	an element $g' \in p^{-1}_\phi \bigl( Z_{W,\phi}(g) \bigr) \sse W$ fixing a point of $\bm B_{g,r} \bigl( \wh Q \bigr) \sse \bm B \bigl( \wh Q \bigr)$ acts as the identity on $\mf t_\phi$.

	This yields a Galois covering as in the quasigeneric case.
\end{proof}

\begin{lemm}[cf.~Lem.~\ref{lem:restricted_generic_twisted_orbit}]
	\label{lem:restricted_commutators}

	Choose an element $g \in W$,
	and an eigenvector $A' \in \mf t(g,\zeta)$ with Levi annihilator $\phi = \phi_{A'} \sse \Phi$,
	for some $\zeta \in \mb C^{\ts}$.
	Then,
	in the notation of~\eqref{eq:effective_weyl_group}--\eqref{eq:centralizer_in_quotient} (and extending~\eqref{eq:eigenspace_normalizer}),
	the following conditions are \emph{equivalent} for any other element $g' \in N_W(\mf t_\phi)$:
	\begin{enumerate}
		\item
		      one has $g'_\phi \in Z_{W,\phi}(g)$;

		\item
		      one has $g'_\phi \in N_{W(\phi)}(g_\phi,\zeta) = N_{W(\phi)} \bigl( \mf t_\phi (g_\phi,\zeta) \bigr)$;

		\item
		      and one has $g'_\phi(A') \in \mf t_\phi(g_\phi,\zeta)$.
	\end{enumerate}
\end{lemm}

\begin{proof}[Proof postponed to~\ref{proof:lem_restricted_commutators}]
\end{proof}

\begin{rema}
	\label{rmk:restricted_centralizer_independence}

	By Lem.~\ref{lem:no_marking},
	the group~\eqref{eq:centralizer_in_quotient} does \emph{not} depend on the choice of $g \in \mb L_r \bigl( \wh Q \bigr)$,
	and again this can be rephrased in terms of Levi subcosets of $\mb G = (\mf t,W)$ (cf.~\S\S~\ref{sec:howlett}--\ref{sec:modular_groups_are_twisted_cabled_weyl}).
	We regard it as the $r$-ramified version of a \emph{relative Weyl group}:
	it need not act as a complex reflection group,
	but Thm.~\ref{thm:complex_refl_groups_from_gauge} provides a sufficient condition to that extent based on Springer/Howlett theory---%
	generalizing the quasigeneric case.
\end{rema}

\begin{rema}
	The argument in the proof of Lem.~\ref{lem:deformations_as_invariants} does \emph{not} extend.
	Namely,
	it is true that there are inclusions
	\begin{equation}
		\label{eq:deformations_as_invariants_2}
		\bm B_r \bigl( \wh\Theta \bigr) \sse \bm B \bigl( \wh\Theta \bigr)^{\mb Z \slash r\mb Z} = \bigl( \bm B( \wh Q) \bs W(\phi) \bigr)^{\mb Z \slash r\mb Z} \sse \bm B \bigl( \wh\Theta \bigr).
	\end{equation}
	However,
	the quotient $W(\phi)$ need \emph{not} act as a reflection group on $\mf t_\phi$,
	and so one cannot immediately appeal to~\cite[Thm.~3.4~(iii)]{springer_1974_regular_elements_of_finite_reflection_groups} to prove that the first inclusion of~\eqref{eq:deformations_as_invariants_2} is an equality.
\end{rema}

\begin{rema}
	\label{rmk:1_galois_closed}

	Even in the unramified case,
	there can be nontrivial elements in $\mb L_r \bigl( \wh Q \bigr) = \mb L_1 \bigl( \wh Q \bigr)$:
	precisely when $A \in \mf t$ is \emph{nonregular}.
	This is consistent with the description of deformation spaces in~\cite{doucot_rembado_2025_topology_of_irregular_isomonodromy_times_on_a_fixed_pointed_curve},
	since such elements lie in $W_{\mf t_\phi} \sse N_W(\mf t_\phi)$ (cf.~Lem.~\ref{lem:about_pointwise_stabilizers}),
	and their trivial class/restriction is centralized by the whole of $W(\phi)$;
	so there is no symmetry-breaking.
	(Alternatively,
	imposing w.l.o.g.~that $\mf t(g,1)$ is \emph{maximal} in the poset $\on E(W,1) = \mc F(W)$ yields $g = 1$,
	cf.~Rmk.~\ref{rmk:about_eigenspace_poset}.)
\end{rema}

\subsection{Higher pole order}

For the multilevel case,
consider finally an irregular type $\wh Q \in \smash{\wh{\mc{IT}}}^{\mathsmaller\leq s}_r$ with arbitrary coefficients $A_1,\dc,A_s \in \mf t$.
Denote the corresponding Levi-subsystem filtration of $\Phi$ by
\begin{equation}
	\bm\phi \ceqq (\phi_1 \sse \dm \sse \phi_s \sse \phi_{s+1} = \Phi).
\end{equation}
(Cf.~\S~\ref{sec:general_pure_twisted_adm_def_space}.)
The corresponding flag of flats is as in~\eqref{eq:kernel_flag},
and we introduce the pointwise-/setwise-stabilizers $W_{\mf t_{\phi_i}} \sse N_W(\mf t_{\phi_i})$,
for $i \in \set{1,\dc,s}$ (cf.~\eqref{eq:effective_weyl_group}).
Note that there are nested inclusions $W_{\mf t_{\phi_1}} \sse \dm \sse W_{\mf t_{\phi_s}} \sse W$,
and the smallest parabolic subgroup $W_{\mf t_{\phi_1}} \sse W$ consists of the elements fixing $\mf t_{\phi_1} = \bigcup_{i = 1}^s \mf t_{\phi_i}$ pointwise.

Then:

\begin{enonce}{Theorem-Definition}
	\label{thm:full_general_twisted_deformation_space}

	Choose---%
	w.l.o.g.---%
	an element $g \in \mb L_r \bigl( \wh Q \bigr)$.
	Moreover,
	consider the canonical projections
	\begin{equation}
		p_{\phi_i} \cl N_W(\mf t_{\phi_i}) \lthra W(\phi_i)
		= N_W(\mf t_{\phi_i}) \bs W_{\mf t_{\phi_i}},
		\qquad i \in \set{1,\dc,s},
	\end{equation}
	onto the setwise-modulo-pointwise-stabilizers;
	and set also
	\begin{equation}
		\label{eq:restricted_centralizer_intersection}
		p_{\bm \phi}^{-1}(g)
		\ceqq \bigcap_{i = 1}^s p_{\phi_i}^{-1} \bigl( Z_{W,\phi_i}(g) \bigr) \sse N_W(\mf t_{\bm \phi}),
	\end{equation}
	in the notation of~\eqref{eq:flag_parabolic_in_weyl} +~\eqref{eq:centralizer_in_quotient}.
	Then:
	\begin{enumerate}
		\item
		      there is a \emph{Galois covering}
		      \begin{equation}
			      \label{eq:general_galois_covering}
			      \bm B_r \bigl( \wh Q \bigr) \lthra \bm B_r \bigl( \wh Q \bigr) \bs Z_{W,\bm \phi} (g),
			      \qquad Z_{W,\bm \phi} (g)
			      \ceqq p_{\bm \phi}^{-1} (g) \bs W_{\mf t_{\phi_1}};
		      \end{equation}

		\item
		      and one has a \emph{homeomorphism} $\bm B_r \bigl( \wh Q \bigr) \bs Z_{W,\bm \phi} (g) \simeq \bm B_r \bigl( \wh\Theta \bigr)$.
	\end{enumerate}
\end{enonce}

\begin{proof}
	Let simply $\mf t_i \ceqq \mf t_{\phi_i} \sse \mf t$,
	and in turn $W_i \ceqq W_{\mf t_i}$,
	for $i \in \set{1,\dc,s}$.
	The subgroup $N_W(\mf t_{\bm \phi}) \sse W$ controls the intersection of the $W$-orbit of any admissible deformation of $\wh Q$ with the `untwisted' root-valuation stratum $\bm B \bigl( \wh Q \bigr) \sse \mf t^s$ (cf.~\cite[Lem.~3.3]{doucot_rembado_2025_topology_of_irregular_isomonodromy_times_on_a_fixed_pointed_curve}).
	Moreover,
	the parabolic subgroup $W_1$ is the stabilizer of \emph{all} the points of $\bm B \bigl( \wh Q \bigr)$ (cf.~\cite[Lemm.~3.4--3.5]{doucot_rembado_2025_topology_of_irregular_isomonodromy_times_on_a_fixed_pointed_curve}).
	Therefore,
	in the untwisted case there is a Galois covering
	\begin{equation}
		\label{eq:final_galois_covering}
		\bm B \bigl( \wh Q \bigr) \lthra \bm B \bigl( \wh Q \bigr) \bs W(\bm \phi),
		\qquad W(\bm \phi) \ceqq N_W(\mf t_{\bm\phi}) \bs W_{\mf t_1}.
	\end{equation}
	(This involves an `unramified' wild Weyl group,
	cf.~Lem.-Def.~\ref{lem:cabled_relative_weyl} +~Cor.~\ref{cor:unramified_cabled_weyl_are_modular}.)

	In the twisted setting,
	analogously,
	choose \emph{any} irregular type $\wh Q' = \sum_{i = 1}^s A'_i w^{-i}$ such that $\wh Q \sim_{g,r} \wh Q'$,
	with coefficients $A'_1,\dc,A'_s \in \mf t$,
	and suppose that an element $g' \in W$ satisfies $g' \bigl( \wh Q' \bigr) \in \bm B_{g,r} \bigl( \wh Q \bigr)$:
	then $g' \in N_W(\mf t_{\bm \phi})$.
	Indeed,
	note that the factorization~\eqref{eq:factorization_pure_multilevel} means that $A'_i \in \mf t_i ( g_{\phi_i},\zeta_r^i )$ is in general position in $\mf t_i$,
	\emph{with respect to} $W_{i+1} \sse W$,
	for $i \in \set{1,\dc,s}$---%
	noting that $W_{s+1} = W$.
	In particular $g'(A'_s) \in \mf t_s$ (and it remains in general position therein),
	whence $g'(\phi_s) = \phi_s$ and $g' \in N_W(\mf t_s)$ (cf.~Lem.~\ref{lem:about_setwise_stabilizers}).
	But one also has $g'(A'_{s-1}) \in \mf t_{s-1}$,
	and so $g'(\phi_{s-1}) = \phi_{s-1}$:
	for if not,
	there would exist $\alpha \in \phi_{s-1} \sse \phi_s$ such that $\beta \ceqq (g')^{-1}(\alpha) \in \phi_s \sm \phi_{s-1}$,
	and then
	\begin{equation}
		0 \neq \Braket{ \beta,A'_{s-1} }
		= \Braket{ \alpha,A''_{s-1} },
		\qquad A''_{s-1}
		\ceqq g'(A'_{s-1}) \in \mf t,
	\end{equation}
	which yields $A''_{s-1} \not\in \mf t_{s-1}$;
	etc.
	In particular,
	there are well-defined classes/restrictions $g'_{\phi_i} \in W(\phi_i)$.

	Now consider in addition the commutator $h \ceqq g^{-1}(g')^{-1}g g' \in W$,
	and let $\wt \phi'_i \ceqq \phi_{A'_i} \sse \Phi$ be the Levi annihilator of the coefficient $A'_i \in \mf t$,
	for $i \in \set{1,\dc,s}$.
	Analogously to~\eqref{eq:commutator_in_parabolic},
	one finds $h(A'_i) = A'_i$,
	since by hypothesis $g'(A'_i) = g'_{\phi_i}(A'_i)$ remains in the eigenspace $\mf t_i(g_{\phi_i},\zeta_r^i) \sse \mf t_i$.
	Therefore,
	one has
	\begin{equation}
		h \in \wt W'_i,
		\qquad \wt W'_i
		\ceqq W_{\mf t_{\wt \phi'_i}} \sse W,
		\qquad i \in \set{1,\dc,s}.
	\end{equation}
	(Cf.~Lem.~\ref{lem:about_pointwise_stabilizers}.)
	The caveat is that $\wt \phi'_i$ need \emph{not} coincide with the Levi annihilator of $A_i$,
	when $i \leq s - 1$;
	what matters is that one has the identities
	\begin{equation}
		\wt \phi'_i \cap \dm \cap \wt \phi'_s
		= \phi_i \sse \Phi,
		\qquad i \in \set{1,\dc,s},
	\end{equation}
	analogous to~\eqref{eq:levi_filtration_intersection}:
	otherwise $\wh Q'$ would not have the same root-valuations as $\wh Q$,
	i.e.,
	one would have $\wh Q' \not\in \bm B \bigl( \wh Q \bigr)$---%
	whence a fortiori $\wh Q' \not\in \bm B_{g,r} \bigl( \wh Q \bigr)$.
	It follows that
	\begin{equation}
		\label{eq:parabolic_intersection}
		h \in \wt W'_i \cap \dm \cap \wt W'_s
		= W_i,
	\end{equation}
	in view of \S~\ref{sec:background_complex_reflections},
	and so $h_{\phi_i} = 1 \in W(\phi_i)$.
	Overall,
	this means that one has $g' \in p^{-1}_{\bm\phi}(g)$.
	Conversely,
	the subgroup~\eqref{eq:restricted_centralizer_intersection} acts diagonally on $\mf t^s$ preserving the subspace $\bm B_{g,r} \bigl( \wh Q \bigr) \sse \mf t^s$,
	in view of the factorization~\eqref{eq:factorization_pure_multilevel} (reasoning as in the proof of Prop.-Def.~\ref{prop:full_general_twisted_deformation_space_1_coeff}).
	Therefore,
	it controls the intersections of the admissible-deformation space with the $W$-orbits passing through it.

	Finally,
	an analogous argument proves that the subgroup $W_1 \sse p^{-1}_{\bm \phi}(g)$ is (also) the stabilizer of \emph{all} the points of the subspace $\bm B_{g,r} \bigl( \wh Q \bigr) \sse \bm B \bigl( \wh Q \bigr)$:
	it will be helpful to rephrase this in terms of recursive flat-density (cf.~\S~\ref{sec:twisted_cabled_weyl}).
	Precisely,
	for the nontrivial implications,
	suppose that $g' \in W$ stabilizes a tuple $(A'_1,\dc,A'_s) \in \mf t^s$ with $A'_i \in \bm B_{g,r} \bigl( \wh Q,i \bigr)$.
	Again $g'(A'_s) = A'_s$ implies that $g' \in W_s$,
	and now one can replace the triple $(G,T,W)$ with $(L_s,T,W_s)$,
	i.e.,
	viewing $W_s \simeq N_{L_s}(T) \bs T$ as the Weyl group of the connected reductive Adjoint stabilizer $L_s \ceqq G^{A'_s} \sse G$ (cf.~\eqref{eq:levi_subgroup}).
	The point is that $W_{s-1} = (W_s)_{\mf t_{s-1}}$ is (also) a parabolic subgroup of $W_s$,
	in view of the nesting $\mf t_s \sse \mf t_{s-1} \sse \mf t$,
	and now $g'(A'_{s-1}) = A'_{s-1}$ yields $g' \in W_{s-1}$;
	etc.
\end{proof}

\begin{rema}
	The condition that $g \in N_W(\mf t_{\bm\phi})$ can be equivalently stated as the fact that the $\Phi$-tuple~\eqref{eq:root_valuations_tuple},
	viewed as a function $\Phi \to \mb Z_{\geq 0}$,
	is $g$-\emph{stable} (cf.~\cite[Prop.~4.8.2.~(3.)]{goresky_kottwitz_macpherson_2009_codimensions_of_root_valuation_strata}).
	The fact is that the root-valuations are determined by the Levi-subsystem filtration,
	but conversely \emph{not} all root-valuations arise from such filtrations (cf.~Rmk.~\ref{rmk:counting_untwisted_strata});
	and even when the latter holds,
	one \emph{cannot} take arbitrary elements $g \in N_W(\mf t_{\bm\phi})$ (cf.~\S~\ref{sec:modular_reparameterization}).
\end{rema}

\subsubsection{}

Again,
the quotient group in~\eqref{eq:general_galois_covering} depends on the choice of $\wh Q$,
but it does \emph{not} depend on the choice of an element $g \in \mb L_r \bigl( \wh Q \bigr)$;
so we write $W_r \bigl( \wh Q \bigr) \ceqq Z_{W,\bm\phi}(g)$.

Before turning on the `interior' twists (cf.~\S~\ref{sec:interior_twist}),
these groups of deck transformations are the most general ones admitting modular interpretations in---%
reductive---%
meromorphic $2d$ gauge theory.
We view them as `wild' versions of Weyl groups,
which can be more explicitly described as follows:

\begin{enonce}{Corollary-Definition}
	\label{cor:explicit_deck_group}

	In the notation of Thm.-Def.~\ref{thm:full_general_twisted_deformation_space},
	consider moreover $\bm\phi' \ceqq (\phi_2 \sse \dm \sse \phi_s \sse \phi_{s+1} = \Phi)$,
	viz.,
	the \emph{Levi-subsystem subfiltration of} $\bm\phi$ with $\phi_1$ dropped.
	The latter yields a \emph{subflag}
	\begin{equation}
		\mf t_{\bm\phi'}
		\ceqq \bigl(\mf Z
		= \mf t_{\phi_{s+1}} \sse \mf t_{\phi_s} \sse \dm \sse \mf t_{\phi_2} \sse \mf t_{\phi_1}\bigr),
	\end{equation}
	in the vector subspace $\mf t_{\phi_1} \sse \mf t$,
	whose (parabolic) stabilizer subgroup---%
	in $\GL_{\mb C}(\mf t_{\phi_1})$---%
	is
	\begin{equation}
		\label{eq:subflag_stabilizer}
		P_{\bm\phi'}
		\ceqq \bigcap_{i = 2}^s N_{\GL_{\mb C}(\mf t_{\phi_1})} (\mf t_{\phi_i}) \sse \GL_{\mb C}(\mf t_{\phi_1}).
	\end{equation}
	Then there is a \emph{canonical} group isomorphism
	\begin{equation}
		\label{eq:second_description_modular_group}
		W_r \bigl( \wh Q \bigr) \simeq P_{\bm\phi'} \cap Z_{W,\phi_1}(g) \sse W(\phi_1).
	\end{equation}
\end{enonce}

\begin{proof}[Proof postponed to~\ref{proof:cor_explicit_deck_group}]
\end{proof}

\subsubsection{}

Equivalently,
the right-hand side of~\eqref{eq:second_description_modular_group} is the centralizer of $g_{\bm\phi}$ in
\begin{equation}
	N_{W(\phi_1)}(\mf t_{\bm\phi'}) \ceqq P_{\bm\phi'} \cap W(\phi_1) \sse W(\phi_1).
\end{equation}

In particular,
suppose that fission is \emph{toral/complete},
i.e.,
that $\mb L_r \bigl( \wh Q \bigr)$ is a singleton (cf.~Rmk.~\ref{rmk:toral_fission}).
This is equivalent to asking that $W_{\mf t_{\phi_1}}$ be trivial,
in which case the deck transformations correspond to the action of the subgroup~\eqref{eq:restricted_centralizer_intersection}.
In view of Cor.~\ref{cor:explicit_deck_group},
this consists of elements $g' \in W \simeq W(\phi_1)$ such that:
(i) $g'$ commutes with $g$;
and (ii) $g'$ preserves the whole kernel flag $\mf t_{\bm\phi}$.
(As $\mf t_{\phi_1} = \mf t$ is trivially stable.)

Again,
such finite `modular' groups are further studied in \S~\ref{sec:modular_groups_are_twisted_cabled_weyl},
singling out their natural discrete set of parameters.

\subsection{Reduction to the simple/irreducible case}

Analogously to the pure case,
the deformation spaces of $r$-Galois-closed irregular classes factor into direct products along decompositions of $\mf g$ into Lie-ideals.
Precisely:

\begin{lemm}[cf.~\cite{doucot_rembado_2025_topology_of_irregular_isomonodromy_times_on_a_fixed_pointed_curve}, \S~4.1]
	\label{lem:reduction_to_simple_full}

	In the notation of Lem.~\ref{lem:reduction_to_simple_case},
	let also $\wh\Theta \ceqq \wh\Theta \bigl( \wh Q \bigr)$ (resp.,
	$\wh\Theta_i \ceqq \wh\Theta \bigl( \wh Q_i \bigr)$) be the $r$-Galois-closed irregular classes of $\wh Q$ (resp.,
	of the summand irregular types in~\eqref{eq:irr_type_splitting}).
	Then there is a (topological) factorization
	\begin{equation}
		\bm B_r \bigl( \wh\Theta \bigr)
		= \prod_i \bm B_r \bigl( \wh\Theta_i \bigr),
	\end{equation}
	in the canonical identification $\mb A_s \simeq \prod_i (\mf t_i^s \bs W_i)$.
\end{lemm}

\begin{proof}
	Using the factorization~\eqref{eq:reduction_to_simple_pure},
	an element $g' = \prod_i g'_i \in W$ stabilizes $\bm B_{g,r} \bigl( \wh Q \bigr) \sse \mf t^s$ if and only if each factor $g'_i \in W_i$ stabilizes $\bm B_{g_i,r} \bigl( \wh Q_i \bigr) \sse \mf t_i^s$,
	and it acts trivially thereon if and only if this holds for each factor.
	Moreover,
	the pointwise/setwise-stabilizers $W_{\mf t_\phi} \sse N_W(\mf t_\phi)$ of the Levi kernel $\mf t_{\phi} = \prod_i (\mf t_i)_{\phi \cap \Phi_i}$ of a Levi subsystem $\phi = \bigcup_i \phi \cap \Phi_i \sse \Phi$ also split into direct products over $i$,
	whence a group isomorphism $W(\phi) \simeq \prod_i W(\phi \cap \Phi_i)$.
	Finally,
	the centralizer $Z_{W,\phi}(g) \sse W(\phi)$ factors as well:
	iterating along a Levi-subsystem filtration,
	the conclusion follows from the Galois covering~\eqref{eq:general_galois_covering} (products/intersections will commute in~\eqref{eq:flag_parabolic} +~\eqref{eq:restricted_centralizer_intersection}).
\end{proof}

\subsubsection{}

However,
contrary to Thm.-Def.~\ref{thm:general_pure_twisted_deformation_space},
the full admissible-deformation spaces do \emph{not} just factor along the levels of $\wh\Theta$.
The main point is that the group~\eqref{eq:second_description_modular_group} depends on the coefficients of the irregular type in a more convoluted way,
due to the very definition of $N_W(\mf t_{\bm \phi})$;
indeed,
this is already visible in the untwisted case (cf.~Exmp.~\ref{exmp:no_splitting_no_reflection}).

Nonetheless,
when $\mf g$ is a classical simple Lie algebra the Galois covering~\eqref{eq:general_galois_covering} can be explicitly described in terms of fission trees (cf.~\cite{doucot_rembado_2025_topology_of_irregular_isomonodromy_times_on_a_fixed_pointed_curve,boalch_doucot_rembado_2025_twisted_local_wild_mapping_class_groups_configuration_spaces_fission_trees_and_complex_braids} and~\S\S~\ref{sec:weyl_group_tree_bc} +~\ref{sec:weyl_group_tree_d});\fn{
	Incidentally,
	it makes sense---%
	in general---%
	to consider full/nonpure analogues of the linearly-ordered tower~\eqref{eq:fibre_bundle_tower} (cf.~again~\cite{doucot_rembado_2025_topology_of_irregular_isomonodromy_times_on_a_fixed_pointed_curve},
	which has a `branching' version based on the fission tree).}~again,
in this text we do not particularly delve into the exceptional types (but cf.~\S~\ref{sec:exceptional_types}).

\begin{exem}
	\label{exmp:no_splitting_no_reflection}

	Consider $\mf g = \mf{gl}_4(\mb C)$,
	so that $W = W_A(4) \simeq \mf S_4$ acts on the standard Cartan subalgebra $\mf t = V^+_4 \simeq \mb C^4$ in the usual---%
	nonessential---%
	fashion.
	Then take:
	(i) $r = 1$;
	(ii) $g = 1$;
	(iii) $s = 2$;
	and (iv) $\wh Q = A_2 w^{-2} + A_1 w^{-1}$,
	with
	\begin{equation}
		A_2
		= \begin{pmatrix}
			1 &   &   &   \\
			  & 1 &   &   \\
			  &   & 2 &   \\
			  &   &   & 2
		\end{pmatrix},
		\qquad
		A_1
		= \begin{pmatrix}
			1 &   &   &   \\
			  & 2 &   &   \\
			  &   & 3 &   \\
			  &   &   & 4
		\end{pmatrix} \in \mf t.
	\end{equation}
	Here $\on{Levels} \bigl( \wh Q \bigr) = \set{1,2}$:
	in this unramified \emph{bilevel} case one computes $W_{\mf t_{\phi_1}} = 1$,
	and
	\begin{equation}
		\label{eq:example_unramified_cabled_weyl}
		Z_{W,\bm \phi}(g)
		= W(\bm \phi)
		= N_W(\mf t_{\bm\phi})
		\simeq \mf S_2 \wr \mf S_2 \sse \mf S_4.
	\end{equation}
	This order-$8$ subgroup acts freely on the admissible-deformation space
	\begin{equation}
		\bm B_{g,r} \bigl( \wh Q \bigr)
		= \bm B \bigl( \wh Q \bigr)
		= \Set{ \bigl( (a,b,c,d),(e,e,f,f) \bigr) \in \bigl( \mb C^4 \bigr)^{\! 2} | e \neq f ; a \neq b ; c \neq d },
	\end{equation}
	and the corresponding quotient is \emph{indecomposable}.
\end{exem}

\section*{Second interlude}

This concludes the abstract description of pure and full/nonpure $r$-admissible deformation spaces,
proving most of Thm.~\ref{thm:thm_1_intro} (only missing the quasigeneric relation with Bessis' lift of Springer theory~\cite[Thm.~12.4]{bessis_2015_finite_complex_reflection_arrangements_are_k_pi_1}).

The next main aim is to prove Thm.~\ref{thm:thm_2_intro}.
Precisely:
(i) in \S~\ref{sec:howlett} we introduce more material about reflection cosets in Howlett theory,
understanding which \emph{monolevel} groups $W_r \bigl( \wh Q \bigr)$ of deck transformations of the above Galois coverings $\bm B_r \bigl( \wh Q \bigr) \thra \bm B_r \bigl( \wh\Theta \bigr)$ act as complex reflection groups (cf.~Thm.~\ref{thm:complex_refl_groups_from_gauge});
(ii) in \S~\ref{sec:lehrer_springer_theory} we introduce a different family of `relative' reflection groups in the classical setting,
preparing the classification of the \emph{crystallographic} examples;
(iii) in \S\S~\ref{sec:type_A}--\ref{sec:type_D} we compute that explicit classification;
and (iv) finally,
the nonreflection type-$D$ examples are settled by defining/using fission trees (cf.~\S\S~\ref{sec:notation_for_trees}--\ref{sec:D_trees};
the latter also proves Thm.~\ref{thm:thm_3_intro}).
As already mentioned,
the usage of fission trees goes much further than the classification of hyperplane arrangements,
and so here we also introduce them---%
with separate interest---%
in type $B\!C$.

Afterwards:
(v) the proof of Thm.~\ref{thm:thm_1_intro} will be concluded in \S~\ref{sec:tglwmcgs};
and (vi) in \S~\ref{sec:interior_twist} we consider a further twist on the starting setup of \S~\ref{sec:setup}.

Finally,
the intermediate sections \S\S~\ref{sec:modular_groups_are_twisted_cabled_weyl}--\ref{sec:modular_reparameterization} relate the `modular' groups $W_r \bigl( \wh Q \bigr)$ with nested sequences of Levi subcosets,
and with a generalization of Springer's regularity conditions,
proving in particular Cor.~\ref{cor:irregular_moduli}.

\section{Some Lie/Levi/Weyl/Howlett theory}
\label{sec:howlett}

\subsection{Relative Weyl groups}
\label{sec:relative_weyl_groups}

The short exact group sequence
\begin{equation}
	1 \lra W_{\mf t_\phi} \lra N_W(\mf t_\phi) \lra W(\phi) \lra 1,
\end{equation}
defined by~\eqref{eq:effective_weyl_group},
proved quite helpful to describe the deformation spaces of irregular classes.
We now relate it with standard Lie-theoretic objects.

\subsubsection{}

Choose a Levi subsystem $\phi \sse \Phi = \Phi(\mf g,\mf t)$.
For any $\alpha \in \Phi$,
denote by $\mf g_\alpha \sse \mf g$ the corresponding root line.
Consider then the reductive Lie subalgebra associated with $\phi$,
i.e.,
\begin{equation}
	\label{eq:levi_subalgebra}
	\mf l
	= \mf l_\phi \ceqq \mf t \ops \bops_\phi \mf g_\alpha \sse \mf g,
\end{equation}
which is the Levi factor of a (noncanonical) parabolic Lie subalgebra of $\mf g$.
If $\phi = \phi_A$ is as in \S~\ref{sec:pure_nongeneric_case_one_coeff},
then
\begin{equation}
	\mf l
	= \mf g^A \ceqq \ker(\ad_A)
	= \Set{ Y \in \mf g | \ad_Y(A) = 0 },
\end{equation}
which is the Lie algebra of the connected reductive subgroup $L = L_\phi$ of~\eqref{eq:levi_subgroup}.
The---%
rational---%
subspace $\mf t_\phi \sse \mf t$ of~\eqref{eq:kernel_intersection} is the `Levi centre',
i.e.,
the Lie-centre of $\mf l$,
and it integrates to an algebraic subtorus $T_\phi \sse T$ such that $L = Z_G(T_\phi)$ and $T_\phi = Z(L)^o$.

Then the standard terminology is the following:

\begin{enonce}{Lemma-Definition}
	\label{lem:relative_weyl_group}

	There are \emph{group isomorphisms}
	\begin{equation}
		W_L
		\ceqq N_G(L) \bs L \simeq W(\phi) \simeq N_G(T_\phi) \bs Z_G(T_\phi).
	\end{equation}
	We refer to these isomorphic quotients as the \emph{relative Weyl group of} $(\mf g,\mf t,\phi)$.
\end{enonce}

\begin{proof}[Proof postponed to~\ref{proof:lem_relative_weyl_group} (cf.~\cite{demazure_grothendieck_1970_sga_3_iii}, Exp.~XXVI, \S~7.1)]
\end{proof}

\begin{rema}
	\label{rmk:no_ramis}

	Beware that $T_\phi \sse T$ need \emph{not} coincide with Ramis' torus $\mb T = \mb T_{\wh Q} \sse T$,
	which is rather cut out by \emph{all} the characters of $T$ annihilating---%
	the coefficients of---%
	an irregular type $\wh Q$ (and not just the roots,
	cf.~\cite[\S\S~3.3 +~4.2]{boalch_yamakawa_twisted_wild_character_varieties}).

	Precisely,
	in general,
	if $\wh Q = \sum_{i = 1}^s A_i w^{-i}$ then one considers
	\begin{equation}
		\Lambda_{\wh Q} \ceqq X^*(T) \cap \set{A_1,\dc,A_s}^\perp \sse \mf t^{\dual},
	\end{equation}
	so that
	\begin{equation}
		\phi_1 = \Phi \cap \set{ A_1,\dc,A_s }^\perp = \Phi \cap \Lambda_{\wh Q} \sse \Lambda_{\wh Q};
	\end{equation}
	in turn,
	one has
	\begin{equation}
		\spann_{\mb C} \set{A_1,\dc,A_s} \sse \ker(\Lambda_{\wh Q})
		= \Lie(\mb T) \sse \Lie (T_{\phi_1})
		= \mf t_{\phi_1} \sse \mf t,
	\end{equation}
	where inclusions can be \emph{proper}.
	(As usual,
	we embed $X^*(T) \hra \mf t^{\dual} \cl \chi \mt \dif \chi$,
	and then one has $\mf t^{\dual} \simeq X^*(T) \ots_{\mb Z} \mb C$.)
	The main points are that:
	(i) the connected reductive stabilizer $L \ceqq G^{\wh Q}$ coincides with $Z_G(T_{\phi_1}) = Z_G(\mb T) \sse G$;
	and (ii) the torus $\mb T$ is \emph{not} locally-constant under admissible deformations.

	E.g.,
	consider the $(g,r)$-admissible (quasigeneric) family $\wh Q = \wh Q_{a,b} = A w^{-1}$,
	of irregular types for $\mf g = \mf{gl}_4(\mb C)$,
	where one takes:
	(a) $r = 2$;
	(b) $g = ( 1 \mid 2 ) \cdot (3 \mid 4) \in \mf S_4 \simeq W = W_A(4)$;
	and (c) $A = A_{a,b} \ceqq \diag(a,-a,b,-b) \in \mb C^4 = V^+_4 \simeq \mf t$,
	such that $a,b \neq 0$ and $a \neq \pm b$.
	One has $\phi = \phi_1 = \vn$ and $T_\phi = T \sse G = \GL_4$ (the standard maximal torus),
	while
	\begin{equation}
		\Lambda_{\wh Q}
		= \Set{ (m_1,\dc,m_4) \in \mb Z^4 | a(m_1 - m_2) + b(m_3 - m_4) = 0},
	\end{equation}
	in the standard identifications $X^*(T) \simeq \mb Z^4 \hra \mb C^4 \simeq \mf t^{\dual}$.
	If $a \slash b \not\in \mb Q$ then $\mb T = \mb T_{a,b}$ has dimension equal to
	\begin{equation}
		\rk \bigl( X^*(T) \bigr) - \rk (\Lambda_{\wh Q}) = 4 - 2 = 2,
	\end{equation}
	while---%
	e.g.---%
	if $b = 2a$ then it has dimension $1$.
	Nonetheless,
	one has $L = T = Z_G(\mb T) = Z_G(T) = G^A$ throughout the deformation.
	(Such phenomena survive if working within $\mf{sl}_4(\mb C) \sse \mf g$.)

	So in brief we shall not need/use Ramis' torus in this text.
	In particular,
	hereafter,
	the symbol $\mb T$ rather denotes a toral reflection coset in the sense of \S~\ref{sec:reflection_cosets}.
\end{rema}

\subsubsection{}

With an irregular type $\wh Q = \sum_{i = 1}^s A_i w^{-i}$ we thus associate:
(i) a sequence
\begin{equation}
	W_{A_i} \cap \dm \cap W_{A_s}
	= N_{L_i}(T) \bs T,
	\qquad L_i \ceqq L_{\phi_i} \sse G,
\end{equation}
of nested parabolic subgroups of $W$;
and (ii) a sequence $W(\phi_i)$ of relative Weyl groups.

\subsection{Normalizers of parabolic subgroups and Levi subcosets}
\label{sec:levi_subcosets}

Furthermore,
in the twisted case,
we have also considered an element $g \in W$ such that (in particular) $g \in N_W(\mf t_{\phi_i})$ for $i \in \set{1,\dc,s}$.
It is useful to view such elements as `outer' automorphisms of parabolic subgroups of $W$,
in view of the following standard fact:
given a Levi subsystem $\phi \sse \Phi$,
the setwise-stabilizer~\eqref{eq:stratum_stabilizer} coincides with the normalizer of the parabolic subgroup $W_{\mf t_\phi} \sse W$.
(Since $\mf t_\phi \sse \mf t$ is precisely the fixed-point locus of the latter.)

Let then again $\wh Q = \sum_{i = 1}^s A_i w^{-i}$ be an $r$-Galois-closed irregular type.
The Levi-subsystem filtration $\bm \phi = \bm\phi \bigl( \wh Q \bigr) =  (\phi_1,\dc,\phi_s)$---%
underlying $\wh Q$---%
corresponds to a filtration of split \emph{Levi subcosets} $\mb L(\phi_i) \ceqq (\mf t,W_{\mf t_{\phi_i}})$ of the rational reflection group $\mb G \ceqq (\mf t,W)$ (cf.~\S~\ref{sec:reflection_cosets}).
In the twisted setting,
instead:

\begin{enonce}{Corollary-Definition}
	\label{cor:levi_subcoset}

	Choose $g \in \mb L_r \bigl( \wh Q \bigr)$.
	Then:
	\begin{enumerate}
		\item
		      the (rational) reflection coset $\mb L_r(\phi_i) \ceqq \bigl( \mf t,g W_{\mf t_{\phi_i}} \bigr)$ is a \emph{Levi subcoset} of $\mb G$,
		      for $i \in \set{1,\dc,s}$;

		\item
		      if $r \geq 2$,
		      and if $\wh Q$ is \emph{primitive} (cf.~Lem.-Def.~\ref{lem:primitive_irregular_types}),
		      then the latter is \emph{nonsplit} for some $i$;

		\item
		      one has $\bigl( \mf t,\mb L_r \bigl( \wh Q \bigr) \bigr) = \mb L_r(\phi_1)$;

		\item
		      and $\mb L_r(\phi_i)$ is---%
		      also---%
		      a \emph{Levi subcoset} of $\mb L_r(\phi_j)$,
		      for all $i \leq j \in \set{1,\dc,s}$.
	\end{enumerate}
\end{enonce}

\begin{proof}
	For the first statement,
	one must prove that $g$ normalizes the nested parabolic subgroups:
	this follows from Lem.-Def.~\ref{lem:flag_stabilizer},
	in view of the aforementioned identification with setwise stabilizers.
	Alternatively,
	choose $g_i \in W_{\mf t_{\phi_i}}$ for some $i \in \set{1,\dc,s}$ and compute
	\begin{equation}
		\label{eq:normalized_parabolic}
		h_i(A_j)
		= g g_i (\zeta_r^{-j} A_j)
		= \zeta_r^{-j} g(A_j)
		= A_j,
		\qquad h_i
		\ceqq g g_i g^{-1} \in W,
	\end{equation}
	for all $j \in \set{i,\dc,s}$.
	This uses the equality $W_{\mf t_{\phi_i}} = W_{A_i} \cap \dm \cap W_{A_s} \sse W$ (cf.~\eqref{eq:weyl_parabolic} +~\eqref{eq:levi_filtration_intersection}),
	and the fact that
	\begin{equation}
		\mb L_r \bigl( \wh Q \bigr)
		= \Set{ g' \in W | A_k \in \mf t(g',\zeta_r^k) \text{ for } k \in \set{1,\dc,s} }.
	\end{equation}
	Now conclude by Lem.~\ref{lem:about_pointwise_stabilizers}.

	For the second statement,
	note that $g \in W_{\mf t_{\phi_i}}$ yields $A_i = \zeta_r^i A_i$,
	whence $A_i = 0$ if $r \nmid i$---%
	for any $i \in \set{1,\dc,s}$.
	If all the subcosets were split one would then have $\wh Q (w) = \wh Q \bigl( \zeta_r^{-1} w \bigr) \in \wh{\mc{IT}}$ (cf.~\eqref{eq:monodromy}),
	with trivial $r$-Galois-orbit.

	For the third statement,
	we have already established that $\mb L_r \bigl( \wh Q \bigr)$ is a right torsor under the smallest parabolic subgroup $W_{\mf t_{\phi_1}} \sse W$ (cf.~Rmk.~\ref{rmk:toral_fission}),
	and by construction $g \in \mb L_r \bigl( \wh Q \bigr) \cap \mb L_r(\phi_1)$---%
	and so the two coincide.

	The last statement follows from the fact that $W_{\mf t_i} = (W_{\mf t_j})_{\mf t_i} \sse W_{\mf t_j}$ is---%
	also---%
	a parabolic subgroup of the reflection group $(\mf t,W_{\mf t_j})$,
	for all $i \leq j \in \set{1,\dc,s}$,
	setting for simplicity $\mf t_{i\slash j} \ceqq \mf t_{\phi_{i \slash j}}$.
	Now take $g' \ceqq 1 \in W_{\mf t_j}$ in the relevant definition at \S~\ref{sec:reflection_cosets},
	so that $gg' = g \in g W_{\mf t_j}$ normalizes $W_{\mf t_i}$.
\end{proof}

\subsubsection{}

Therefore,
admissible deformations of irregular types/classes can be phrased via sequences of nested Levi subcosets of $\mb G$.
E.g.,
fission is complete/toral (in the sense of Rmk.~\ref{rmk:toral_fission}) if and only if the smallest Levi subcoset $\mb L_r(\phi_1)$ coincides with the maximal torus $\mb T \ceqq (\mf t,g) \sse \mb G$ (in the sense of \S~\ref{sec:reflection_cosets}).

We will significantly expand upon this in \S~\ref{sec:modular_groups_are_twisted_cabled_weyl}.

\begin{rema}
	The sequence of Levi subcosets associated with $\wh Q$ is also \emph{uniformizer-free} (cf.~\S~\ref{sec:uniformizers} and Rmkk.~\ref{rmk:intrinsic_flag} +~\ref{rmk:fibre_bundle_tower}).
	Indeed,
	for any $i \in \set{1,\dc,s}$ one can consider the set of elements $g \in W$ which generate the $r$-Galois-orbit of the intrinsic truncation $\tau_i \bigl( \wh Q \bigr) \in \mf t \ots_{\mb C} \bigl( \mf m^{-s} \bs \mf m^{1-i} \bigr)$ (cf.~\eqref{eq:truncated_orbit_generators}).
	In particular,
	one might also write $\mb L_r \bigl( \wh Q \bigr) = \mb L_r \bigl( \wh Q,1 \bigr)$---%
	and define more generally $\mb L_r \bigl( \wh Q,i \bigr) \sse W$ in terms of the underlying Levi-subsystem filtration,
	for $i \in \set{1,\dc,s}$.

	Note---%
	again---%
	that the monolevel case is intimately related to the cyclotomic theory of reflection cosets à la Broué--Malle--Michel~\cite{broue_malle_1992_theoremes_de_sylow_generiques_pour_les_groupes_reductifs_sur_les_corps_finis,
		broue_malle_1993_zyklotomische_heckealgebren,
		broue_malle_michel_1993_generic_blocks_of_finite_reductive_groups,
		malle_1998_spetses,
		broue_malle_michel_1999_towards_spetses_i,
		malle_2000_on_the_generic_degrees_of_cyclotomic_algebras,
		broue_malle_michel_2014_split_spetses_for_primitive_reflection_groups}:
	the exact relation will be spelt out elsewhere.
\end{rema}

\subsection{Howlett's reflection cosets}
\label{sec:howlett_twist}

We now recall a different family of rational reflection subcosets:
one main aim is to give a sufficient condition so that the monodromy group $Z_{W,\phi}(g)$ of the Galois covering of Prop.-Def.~\ref{prop:full_general_twisted_deformation_space_1_coeff} acts as a complex reflection group (in the monolevel case;
cf.~Thm.~\ref{thm:complex_refl_groups_from_gauge}).

To this end,
as far as this text is concerned,
we summarize the relevant part of the theory of normalizers of parabolic subgroups of Weyl groups~\cite{howlett_1980_normalizers_of_parabolic_subgroups_of_reflection_groups} (cf.~\cite{lusztig_1977_coxeter_orbits_and_eigenspaces_of_frobenius,
	deriziotis_1977_the_brauer_complex_and_its_applications_to_the_chevalley_groups})\fn{
	It was extended to infinite ones~\cite{borcherds_1998_coxeter_groups_lorentzian_lattices_and_k3_surfaces,
		brink_howlett_1999_normalizers_of_parabolic_subgroups_in_coxeter_groups},
	cf.~\cite{deodhar_1982_on_the_root_system_of_a_coxeter_group,
		nuida_2011_on_centralizers_of_parabolic_subgroups_in_coxeter_groups,
		allcock_2012_normalizers_of_parabolic_subgroups_of_coxeter_groups};
	and to finite \emph{complex} reflection groups~\cite{muraleedaran_taylor_2018_normalisers_of_parabolic_subgroups_in_finite_unitary_reflection_groups} (which seems relevant here,
	cf.~Rmk.~\ref{rmk:relative_howlett}).}~in the statement of Prop.-Def.~\ref{prop:howlett}.

\subsubsection{}

Fix a Levi subsystem $\phi \sse \Phi$.
If $\alpha \in \Phi \sm \phi$ is such that the corresponding reflection $g_\alpha \in W$ stabilizes $\mf t_\phi$,
then the latter yields a (nontrivial) element $(g_\alpha)_\phi \in W(\phi)$ as above.
To generalize this prescription:

\begin{enonce}{Lemma-Definition}
	\label{lem:inner_product}

	Choose a $W$-\emph{invariant} positive-definite (sesquilinear) Hermitian form $( \cdot \mid \cdot ) \cl \mf t \ots_{\mb C} \ol{\mf t} \to \mb C$,
	defined over $\mb Q$,
	and denote by the same symbol its restriction to $\mf t_\phi \sse \mf t$.
	Let also $\alpha_\phi^{\dual} = (\alpha^{\dual})_\phi$ be the (nonvanishing) orthogonal projection of the coroot $\alpha^{\dual} \in \mf t$,
	for $\alpha \in \Phi \sm \phi$,
	onto $\mf t_\phi$.
	Then:
	\begin{enumerate}
		\item
		      the hyperplane arrangement~\eqref{eq:pure_untwisted_nongeneric_deformation_space_1_coeff} consists of orthocomplements (in $\mf t_\phi$) of the vectors lying in the finite set $\set{ \alpha^{\dual}_\phi | \alpha \in \Phi \sm \phi}$;

		\item
		      and one has
		      \begin{equation}
			      \label{eq:rationality_full}
			      \bigl( \alpha_\phi^{\dual} \mid \beta_\phi^{\dual} \bigr) \in \mb Q,
			      \qquad \alpha,\beta \in \Phi \sm \phi.
		      \end{equation}
	\end{enumerate}
\end{enonce}

\begin{proof}[Proof postponed to~\ref{proof:lem_inner_product}]
\end{proof}

\begin{rema}
	One could restrict the Hermitian form onto $\mf t_\phi$,
	and then consider the `dual' vector corresponding to the restriction $\alpha_\phi \in \mf t_\phi^{\dual}$ of the root $\alpha \in \Phi \sm \phi$.
	This (nonzero) vector lies in the line through $\alpha_\phi^{\dual}$,
	and so it determines the same hyperplane---%
	and the same reflection,
	cf.~just below.
\end{rema}

\subsubsection{}

By Lem.-Def.~\ref{lem:inner_product},
there is a well-defined order-$2$ reflection
\begin{equation}
	\label{eq:reduced_reflection}
	g_\alpha(\phi) \in \GL_{\mb C}(\mf t_\phi),
	\qquad g_\alpha(\phi) \cl \alpha^{\dual}_\phi \lmt - \alpha^{\dual}_\phi,
\end{equation}
about the hyperplane $H_\alpha(\phi) \sse \mf t_\phi$,
for all $\alpha \in \Phi \sm \phi$.

The main facts we need about these restricted reflections are as follows:

\begin{enonce}{Proposition-Definition}
	\label{prop:howlett}

	Let $\Delta \sse \Phi$ be a \emph{base} (of simple roots),
	and suppose that $\Delta_\phi \ceqq \Delta \cap \phi$ is a base of $\phi$ (cf.~Rmk.~\ref{rmk:reductive_base}).
	Then:
	\begin{enumerate}
		\item
		      there are involutions $v_\alpha = v[\alpha,\Delta_\phi] \in W(\phi)$ (a.k.a.~`R-elements'),
		      indexed by a suitable subset of roots $\alpha \in \phi'' \sse \Phi \sm \phi$,
		      which generate a \emph{normal} subgroup
		      \begin{equation}
			      W''(\phi) \ceqq \Braket{ v_\alpha | \alpha \in \phi'' } \sse W(\phi);
		      \end{equation}

		\item
		      the group $W''(\phi)$ acts on $\mf t_\phi$ as the \emph{rational reflection group} generated by the reflections $g_\alpha(\phi)$,
		      for $\alpha \in \phi''$ (cf.~\eqref{eq:reduced_reflection});

		\item
		      and the corresponding short exact group sequence splits:
		      \begin{equation}
			      \label{eq:relative_reflection_group_and_twist}
			      1 \lra W''(\phi) \lra W(\phi) \lra \wt W(\phi) \lra 1,
			      \qquad \wt W(\phi) \ceqq W(\phi) \bs W''(\phi).
		      \end{equation}
	\end{enumerate}
	We say that $W''(\phi)$ is \emph{Howlett's relative reflection group},
	and that $\wt W(\phi)$ is \emph{Howlett's obstruction} (a.k.a.~the `nonreflection' part of $W(\phi)$).
\end{enonce}

\begin{proof}
	First,
	the extension $1 \to W_{\mf t_\phi} \to N_W(\mf t_\phi) \to W(\phi) \to 1$ splits.
	Precisely,
	by~\cite[Cor.~3]{howlett_1980_normalizers_of_parabolic_subgroups_of_reflection_groups},
	one has
	\begin{equation}
		\label{eq:splitting_parabolic_normalizers}
		N_W(\mf t_\phi)
		\simeq N_W(\Delta_{\phi}) \lts W_{\mf t_\phi},
		\qquad N_W(\Delta_\phi)
		\ceqq \Set{ g \in W | g(\Delta_\phi) \sse \Delta_\phi },
	\end{equation}
	whence a group isomorphism $W(\phi) \simeq N_W(\Delta_\phi)$.\fn{
		The subgroup $N_W(\Delta_\phi) \sse N_W(\mf t_\phi)$ is a.k.a.~\emph{Howlett's complement} (cf.~\cite[Prop.~2.6]{douglass_pfeiffer_roehrle_parabolic_normalizers_in_finite_coxeter_groups_as_subdirect_products}),
		and it equivalently consists of elements preserving the set of $\Delta_\phi$-positive roots.}~(Note that we will not need/use the exact description of the subset $\phi'' \sse \Phi \sm \phi$.)

	Using Thm.~6 +~Cor.~7 of op.~cit.,
	which is in the more general context of finite Coxeter/real reflection groups,
	it only remains to see that $W''(\phi)$ preserves a $\mb Q$-form of $\mf t_\phi$.
	By construction,
	the coroot $\alpha^{\dual}$ projects to zero onto $\mf t_\phi$ for $\alpha \in \phi$,
	and so one has
	\begin{equation}
		\mf t_\phi
		= \mf Z \ops \spann_{\mb C} \Set{ \alpha_\phi^{\dual} | \alpha \in \Phi \sm \phi }.
	\end{equation}
	Consider the canonical $\mb Q$-form of the Lie-centre $\mf Z \sse \mf t_\phi$ (cf.~\S~\ref{sec:no_isogeny}),
	and extend it to a $\mb Q$-form $\mf t_{\phi,\mb Q}$---%
	of $\mf t_\phi$---%
	by taking the $\mb Q$-span of the complementary projected coroots.
	Then~\eqref{eq:rationality_full} implies that the reflections~\eqref{eq:reduced_reflection} preserve $\mf t_{\phi,\mb Q} \sse \mf t_\phi$,
	and by the above $W''(\phi)$ is generated by a subset thereof.
\end{proof}

\subsubsection{}

Overall,
Prop.-Def.~\ref{prop:howlett} yields a nested semidirect factorization
\begin{equation}
	N_W(\mf t_\phi) \simeq \bigl( \wt W(\phi) \lts W''(\phi) \bigr) \lts W_{\mf t_\phi}.
\end{equation}
(This can be further refined,
cf.~\cite[Cor.~9]{howlett_1980_normalizers_of_parabolic_subgroups_of_reflection_groups};
we will \emph{not} need/use that.)

\begin{rema}
	\label{rmk:dominant_levis}

	The condition that $\Delta_\phi$ is a base of $\phi$ is equivalent to asking that there exists an element $A \in \mf t_\phi$ which is $\Delta$-\emph{dominant} (in the complexified sense,
	cf.~\cite[\S~5.2]{crooks_2019_complex_adjoint_orbits_in_lie_theory_and_geometry} and~\cite[\S~2.2]{collingwood_mcgovern_1993_nilpotent_orbits_in_semisimple_lie_algebras}).
	Then $\phi$ is encoded by the full subdiagram of the Dynkin diagram of $(\mf g,\mf t,\Delta)$ on the subset of nodes $\Delta_\phi \sse \Delta$ (cf.~again Rmk.~\ref{rmk:reductive_base}).
	The corresponding subalgebra~\eqref{eq:levi_subalgebra} is then the Levi factor of a \emph{standard} parabolic Lie subalgebra of $\mf g$:
	we say that $\phi$ is $\Delta$-\emph{dominant}.

	Moreover,
	given any Levi $\phi$ and a base $\Delta$,
	there exists an element $g \in W$ such that $g(\phi)$ is $\Delta$-dominant.
	In particular,
	the fact that Howlett's obstruction is trivial does \emph{not} depend on the choice of a compatible base,
	because---%
	more generally---%
	the isomorphism class of $W(\phi)$ only depends on the $\Aut(\Phi)$-orbit of $\phi$:
	acting by root automorphisms conjugates the parabolic subgroup $W_{\mf t_\phi} \sse W$,
	as well as its normalizer $N_W(\mf t_\phi)$.
	Conversely,
	however,
	\emph{not} all Levis are dominant with respect to a single base.
	(This is relevant in the definition of `pointed' irregular types of \S\S~\ref{sec:BC_trees}--\ref{sec:D_trees}.)

	Finally,
	note that the $\Aut(\Phi)$-orbit of a Levi subsystem might meet several $\Delta$-dominant subsystems,
	corresponding to various subdiagrams of the Dynkin diagram;
	notably,
	those which are related by---%
	outer---%
	diagram automorphisms (cf.~Rmk.~\ref{rmk:about_D_components}).
\end{rema}

\subsubsection{}
\label{sec:classical_howlett_twists}

As far as this text is concerned,
the following is an adequate partial summary of the status of Howlett's obstructions of classical type:

\begin{figure}
	\begin{equation}
		\begin{tikzcd}
			& & & 1 \ar{d} &  \\
			& & & W''(\phi) \ar{d} &  \\
			1 \ar{r} & W_{\mf t_\phi} \ar{r} & N_W(\mf t_\phi) \ar{r} & W(\phi) \ar{d} \ar{r} \ar[bend right=20,dashed]{l} & 1. \\
			& & & \wt W(\phi) \ar{d} \ar[bend right=20,dashed]{u} &  \\
			& & & 1 &
		\end{tikzcd}
	\end{equation}
	\caption{
		\label{fig:levi_diagram}
		Split group extensions determined by a Levi subsystem $\phi \sse \Phi$ and a compatible base $\Delta \sse \Phi$ of simple roots}
\end{figure}

\begin{enonce}{Corollary-Definition}
	\label{cor:howlett_twisted_reflection_cosets}

	Choose an integer $r \geq 1$ and a vector $A \in \mf t$,
	and set $\wh Q \ceqq Aw^{-1}$.
	Suppose that $\wh Q$ is $r$-Galois-closed,
	and let $\Delta \sse \Phi$ be a base of simple roots such that the Levi annihilator $\phi \sse \Phi$ of $A$ is $\Delta$-dominant.
	(Cf.~Prop.-Def.~\ref{prop:howlett}.)
	Then the groups/sections of Fig.~\ref{fig:levi_diagram} are defined,
	and:
	\begin{enumerate}
		\item
		      there is a \emph{rational reflection subcoset} of $\mb G = (\mf t,W)$,
		      defined by
		      \begin{equation}
			      \mb G''_r(\phi)
			      \ceqq \bigl( \mf t_\phi, g_\phi W''(\phi) \bigr),
			      \qquad g \in \mb L_r \bigl( \wh Q \bigr);
		      \end{equation}

		\item
		      and if $W$ is \emph{irreducible},
		      then the coset $\mb G''_r(\phi)$ is \emph{split} provided that one of the following holds:
		      \begin{enumerate}
			      \item[(A)]
			            the group $W$ is of type $A_m$ ($m \geq 1$);

			      \item[(B\!C)]
			            the group $W$ is of type $B\!C_m$ ($m \geq 2$);

			      \item[(D)]
			            or the group $W$ is of type $D_m$ ($m \geq 4$),
			            and then:
			            \begin{enumerate}
				            \item
				                  if $m \geq 5$,
				                  then $\Delta_\phi \sse \Delta$ contains both fork-nodes of the Dynkin diagram of $(\mf g,\mf t,\Delta)$ (cf.~Rmk.~\ref{rmk:about_D_components} +~Fig.~\ref{fig:dynkin_D});

				            \item
				                  and if $m = 4$,
				                  then $\Delta_\phi$ contains \emph{at least two} outer nodes of the Dynkin diagram.
				                  (Here the `fork' is \emph{not} well-defined up to Dynkin-diagram automorphisms;
				                  cf.~Cor.~\ref{cor:outer_root_autos} +~\S~\ref{sec:triality}.)
			            \end{enumerate}
		      \end{enumerate}
	\end{enumerate}
\end{enonce}

\begin{proof}
	For the first statement,
	the previous discussion implies that:
	(i) the twist $g_\phi W''(\phi) \sse \GL_{\mb C}(\mf t_\phi)$ is indeed a coset of Howlett's---%
	rational---%
	relative reflection group~\eqref{eq:relative_reflection_group_and_twist} (as $W(\phi)$ normalizes $W''(\phi)$);
	and (ii) it is a subcoset of $(\mf t,W)$;
	and (iii) it splits if and only if the class of $g_\phi$ in $\wt W(\phi)$ vanishes.
	The second statement then summarizes part of the case-by-case discussion of~\cite{howlett_1980_normalizers_of_parabolic_subgroups_of_reflection_groups}---%
	which establishes that $W''(\phi) = W(\phi)$ in these cases.
\end{proof}

\begin{rema}
	\label{rmk:about_D_components}

	Up to triality in rank $m = 4$,
	the subcase of~Cor.-Def.~\ref{cor:howlett_twisted_reflection_cosets}~(2.D) corresponds to `Case 1' of~\cite[pp.~71--72]{howlett_1980_normalizers_of_parabolic_subgroups_of_reflection_groups} (cf.~\cite[\S~8]{doucot_rembado_tamiozzo_2026_local_wild_mapping_class_groups_and_cabled_braids}).
	Compared to op.~cit.,
	we relabel the simple roots and define the Dynkin diagram with respect to the base $\Delta_D = \set{\theta_1,\dc,\theta_{m-1},\theta_D}$ (cf.~\S~\ref{sec:root_systems_background};
	so we put the fork on the right when $m \geq 5$).

	Cf.~Lem.-Def.~\ref{lem:type_D_part} for an \emph{intrinsic} characterization of the condition of Cor.-Def.~\ref{cor:howlett_twisted_reflection_cosets}~(2.D.i),
	i.e.,
	the fact that $\phi$ has a nonempty part `of type $D$'---%
	without putting it in dominant form.
	(Again this only holds up to triality when $m = 4$.)
\end{rema}

\begin{rema}
	As per the exceptional types,
	recall that~\cite[pp.~73-74]{howlett_1980_normalizers_of_parabolic_subgroups_of_reflection_groups} states that the nonreflection obstruction is trivial in type $F_4$,
	and that it has order at most $2$ for the other exceptionals (cf.~also the subgroup $C$ in the alternative nested semidirect factorization of $N_W(\mf t_\phi)$ of~\cite[Thmm.~1.2--1.3]{douglass_pfeiffer_roehrle_parabolic_normalizers_in_finite_coxeter_groups_as_subdirect_products}).
\end{rema}

\begin{figure}
	\begin{center}
		\begin{tikzpicture}[every node/.style={circle,draw,minimum size=8pt,inner sep=0pt}]
			\node[fill=white, label=below:{$\theta_1$}] (1) at (0,0) {};
			\node[fill=white, label=below:{$\theta_2$}] (2) at (1,0) {};
			\node (dots) at (2,0) [draw=none] {\, $\dc$ \,};
			\node[fill=white, label={[label distance=-7pt]below:{$\theta_{m-2}$}}] (m-2) at (3,0) {};
			\node[fill=black, label={[label distance=3pt]right:{$\theta_{m-1}$}}] (m-1) at (4,0.4) {};
			\node[fill=black, label={[label distance=3pt]right:{$\theta_D$}}] (m) at (4,-0.4) {};
			\draw (1)--(2)--(dots)--(m-2)--(m-1) (m-2)--(m);
		\end{tikzpicture}
	\end{center}
	\caption{
		\label{fig:dynkin_D}
		The nodes which ensure `crystallographic' dominant Levi subsystems in type $D_m$ ($m \geq 5$).}
\end{figure}

\subsubsection{}
\label{sec:noncrystallographic_D}

As far as local WMCGs are concerned,
in type $D$ the crux of the matter is that~\eqref{eq:pure_untwisted_nongeneric_deformation_space_1_coeff} need \emph{not} be the complement of a root-hyperplane arrangement.
Rather,
it involves `generalized' root systems~\cite{dimitrov_fioresi_generalized_root_systems},
which indeed arise upon restrictions/projections of root systems relative to a Levi subsystem~\cite{cuntz_muehlherr_2024_a_classification_of_generalized_root_systems} (cf.~\cite{rembado_2024_a_colourful_classification_of_quasi_root_systems_and_hyperplane_arrangements}).

Concretely,
if the condition Cor.-Def.~\ref{cor:howlett_twisted_reflection_cosets}~(2.D) fails then there are \emph{noncrystallographic} root-valuation strata,
i.e.,
complements of hyperplane arrangements which do \emph{not} come from root systems (nor from complex reflection groups,
more generally;
cf.~Prop.~\ref{prop:nonreflection_arrangement}).
The simplest example involves a rank-$1$ Levi subsystem $\phi \sse \Phi_D(4)$,
leading to $7$ hyperplanes in $\mb C^3$:
the complement of their union is a copy of $\mc M^{\musDoubleSharp}(1,2)$ (cf.~\eqref{eq:exotic_arrangement}).

Beware however that the stated condition is \emph{not} necessary to avoid such arrangements.
Analogously,
Howlett's obstruction can indeed be nontrivial if there are no such type-$D$ parts in $\phi$,
but this criterion is not sharp.
(Cf.~again~\cite[\S~8]{doucot_rembado_tamiozzo_2026_local_wild_mapping_class_groups_and_cabled_braids},
as well as~`Case 2' at~\cite[pp.~72--73]{howlett_1980_normalizers_of_parabolic_subgroups_of_reflection_groups}.)

\subsection{Complex reflections from isomonodromic deformations}

The material of \S~\ref{sec:howlett_twist} clarifies the main obstruction to finding complex reflection groups in \emph{monolevel} meromorphic $2d$ gauge theory,
using `relative' Springer theory:

\begin{theo}
	\label{thm:complex_refl_groups_from_gauge}

	Choose $g \in W$,
	and an eigenvector $A \in \mf t(g,\zeta_r)$;
	let $\phi \ceqq \phi_A \sse \Phi$ be the Levi annihilator of $A$.
	If Howlett's obstruction $W(\phi) \thra \wt W(\phi)$ \emph{vanishes} (cf.~Prop.-Def.~\ref{prop:howlett}),
	then the centralizer $Z_{W,\phi}(g) \sse W(\phi)$ of~\eqref{eq:centralizer_in_quotient} acts---%
	faithfully---%
	on $\mf t_\phi(g_\phi,\zeta_r)$ as a \emph{complex reflection group}.
\end{theo}

\begin{proof}
	By hypothesis,
	the relative Weyl group $W(\phi) = W''(\phi)$ acts---%
	faithfully---%
	on $\mf t_\phi \sse \mf t$ as a rational reflection group.
	Moreover,
	if $g'(A) = A$ for some $g' \in W$,
	then $g'_\phi = 1 \in W(\phi)$ (by Lem.~\ref{lem:about_pointwise_stabilizers};
	the class $g'_\phi$ is well-defined by Lem.~\ref{lem:about_setwise_stabilizers}).
	It follows that the parabolic subgroup $W(\phi)_A \sse W(\phi)$ is trivial (cf.~\eqref{eq:weyl_parabolic}),
	and so the vector $A \in \mf t_\phi$ is $W(\phi)$-regular (cf.~\cite[Prop.~4.1.(i)]{springer_1974_regular_elements_of_finite_reflection_groups}).

	Now consider the regular eigenspace
	\begin{equation}
		E_\phi
		= E_\phi(g,\zeta_r)
		\ceqq \mf t_\phi(g_\phi,\zeta_r) \in E \bigl( W(\phi),\zeta_r \bigr).
	\end{equation}
	(Cf.~\eqref{eq:eigenspace_poset}.)
	Let $N_\phi = N_\phi(g,\zeta_r) \sse W(\phi)$ be the setwise-stabilizer of $E_\phi$ inside $W(\phi)$:
	Springer theory~\cite[Prop.~3.5]{springer_1974_regular_elements_of_finite_reflection_groups} ensures that $N_\phi$ acts on $E_\phi$ as a complex reflection group,
	and here one has $N_\phi = Z_{W,\phi}(g)$ (cf.~Lem.~\ref{lem:restricted_commutators};
	the pointwise-stabilizer is trivial).
\end{proof}

\begin{rema}
	\label{rmk:about_ramified_relative_weyl}

	It follows that the ramification $r$ is a regular number for $W(\phi)$,
	while this need \emph{not} hold for $W$.
	Somewhat analogously,
	the eigenspace $\mf t(g,\zeta_r) \sse \mf t$ need \emph{not} be maximal in $\on E(W,\zeta_r)$---%
	let alone regular.
	(Cf.~Rmk.~\ref{rmk:relative_howlett} for a different viewpoint,
    which instead uses maximality.)
\end{rema}

\subsubsection{}
\label{sec:prompt_for_trees}

In the end,
Lemm.~\ref{lem:restricted_generic_twisted_orbit} +~\ref{lem:restricted_commutators} and Thm.~\ref{thm:complex_refl_groups_from_gauge} relate isomonodromic deformations of irregular-singular connections on smooth projective curves with `relative' (split) Springer theory.
Namely,
in the notation of loc.~cit.,
let $\wh\Theta = \wh\Theta \bigl( \wh Q \bigr)$ be the irregular class of the $r$-Galois-closed irregular type $\wh Q \ceqq Aw^{-1}$:
then the admissible deformations of $\wh\Theta$ correspond to the orbits under the free action of the complex reflection group $\bigl( E_\phi,W_r \bigl( \wh Q \bigr) \bigr)$,
acting on a complex hyperplane complement which is contained within the $W_r \bigl( \wh Q \bigr)$-regular locus---%
possibly \emph{properly},
cf.~Rmk.~\ref{rmk:no_braid_lift_first}.
This generalizes the statement of Cor.~\ref{cor:modular_reflection_groups}~(1.) in the monolevel case,
beyond the quasigeneric one---%
and for any reductive Lie algebra.

The general case,
for any pole order $s \geq 1$,
involves normalizers of nested parabolic subgroups of $W$;
whence in principle nested nonreflection obstructions.
This is explored in \S\S~\ref{sec:modular_groups_are_twisted_cabled_weyl}--\ref{sec:modular_reparameterization},
where we formally introduce `wild' versions of---%
relative---%
Weyl groups.
(Cf.~Rmk.~\ref{rmk:cabled_weyl} when $\mf g$ is of classical type $\bullet \in \set{B,C,D}$.)

\begin{rema}
    \label{rmk:no_braid_lift_first}

    Even assuming that $W(\phi)$ acts as a reflection group,
    beware that its reflection arrangement can be \emph{properly} contained in the set of `complementary' restricted hyperplanes $H_\alpha(\phi) \sse \mf t_\phi$ (cf.~\eqref{eq:pure_untwisted_nongeneric_deformation_space_1_coeff} +~Rmk.~\ref{rmk:no_braid_lift}).
    Therefore,
    upon twisting,
    the reflection arrangement of $W_r \bigl( \wh Q \bigr)$ can be properly contained in the set of hyperplanes $H_\alpha(g_\phi,\zeta_r) \sse E_\phi$ (cf.~\eqref{eq:hyperplane_intersection}):
    so the corresponding pure/full local WMCGs need \emph{not} coincide with the pure/full complex braid groups associated with $\bigl( E_\phi,W_r \bigl( \wh Q \bigr) \bigr)$---%
    albeit this holds in the quasigeneric case.

    As already mentioned,
    in the \emph{classical} simple examples we shall nonetheless find a way around this issue,
    describing the hyperplane arrangements explicitly:
    in the \emph{crystallographic} setting,
    after introducing non-Howlett `relative' reflection groups (cf.~\S\S~\ref{sec:lehrer_springer_theory}--\ref{sec:type_D});
    and then in general,
    after introducing fission trees (cf.~\S\S~\ref{sec:BC_trees}--\ref{sec:D_trees}).
\end{rema}

\begin{rema}
	\label{rmk:relative_howlett}

	As pointed out by J.~Michel,
	the monolevel nonreflection obstruction relates to normalizers of parabolic subgroups in \emph{complex} reflection groups:
	it suffices to work within a fixed eigenspace \emph{before} selecting a flat.
	(But there is a caveat about the order of operations,
	cf.~Exmp.~\ref{exmp:no_full_eigenspace_preservation}.)

	Precisely,
	choose $g \in W$ and an eigenvector $A \in E \ceqq \mf t(g,\zeta_r)$.
	Moreover,
	assume that $E \in \on E(W,\zeta_r)$ is \emph{maximal} (cf.~Rmk.~\ref{rmk:about_eigenspace_poset}).
	The crux of the matter is when $A$ is \emph{not} in general position in $E$ (cf.~Rmk.~\ref{rmk:general_position}),
	which is possible in our situation.

	Denote by $\phi \ceqq \phi_A \sse \Phi$ the Levi annihilator of $A$,
	and by $\pi_{g,r} \cl N_W(g,\zeta_r) \thra W_{g,r}$ the canonical projection onto the setwise-modulo-pointwise-stabilizer of $E$ (cf.~\eqref{eq:eigenspace_normalizer}),
	whose target is a subquotient of $W$ acting on $E$ as a complex reflection group à la Springer/Lehrer--Springer.
	It follows that $E_\phi \ceqq \mf t_\phi \cap E$ is the smallest flat of $W_{g,r}$ containing $A$,
	and one can consider the corresponding parabolic subgroup
	\begin{equation}
		\label{eq:complex_normalizer}
		W_{A,g,r} \ceqq \pi_{g,r} \bigl( W_A \cap N_W(g,\zeta_r) \bigr) \sse W_{g,r},
	\end{equation}
	i.e.,
	the pointwise-stabilizer of $E_\phi$ in $W_{g,r}$.
	(The reflecting hyperplanes of $W_{g,r}$ are the traces on $E$ of the hyperplanes of $W$;
	cf.~Lem.~\ref{lem:restricted_reflection_group}.)

	Once more,
	the faithful quotient,
	viz.,
	\begin{equation}
		\label{eq:alternative_modular_group}
		Z_{W,A,r}(g)
		\ceqq N_{W_{g,r}}(W_{A,g,r}) \bs W_{A,g,r},
	\end{equation}
	need \emph{not} act as a complex reflection group thereon.
\end{rema}

\subsubsection{}

In the \emph{quasigeneric} case the subgroup~\eqref{eq:complex_normalizer} is trivial,
and so~\eqref{eq:alternative_modular_group} recovers the centralizer $W_{g,r} \simeq Z_W(g)$ (cf.~Prop.-Def.~\ref{prop:full_generic_twisted_deformation_space_1_coeff}).

In general,
however,
the setwise-stabilizer of the $(g,r)$-admissible-deformation space of $\wh Q \ceqq Aw^{-1}$ involves---%
classes/restrictions of---%
elements $g' \in N_W(\mf t_\phi)$ which need \emph{not} fix $E$ on the whole,
but rather only its trace $E_\phi$ on $\mf t_\phi$;
so there are discrepancies when $E$ is nonmaximal (cf.~again Exmp.~\ref{exmp:no_full_eigenspace_preservation}).

Conversely,
if $E$ is maximal then (nonsplit) Lehrer--Springer theory for the reflection coset $(\mf t,g W_{\mf t_\phi})$ shows that the two descriptions of modular groups do match up.
This is the starting point to bridge with~\cite[Thm.~1.50]{broue_malle_michel_2014_split_spetses_for_primitive_reflection_groups},
but again note that:
(i) in our setup we do \emph{not} need/use the maximality assumption;
(ii) anyway,
after intersecting with a flat of $\mb G = (\mf t,W)$ we encounter `nonsylow/nonmaximal' reflection cosets,
whence the nonreflection obstructions;
and (iii) the maximality condition seems to be less natural/useful in the multilevel case.

\begin{exem}
	\label{exmp:no_full_eigenspace_preservation}

	Take $\mf g = \mf{gl}_8(\mb C)$ and $W \ceqq W_A(8) \simeq \mf S_8$ (cf.~Exmp.~\ref{exmp:no_splitting_no_reflection}).
	Consider the Levi subsystem
	\begin{equation}
		\phi
		\ceqq \Set{ \pm \alpha_{i,i+1} | i \in \set{1,3,5,7} } \sse \Phi
		= \Phi_A(7),
	\end{equation}
	which consists of four copies of $\Phi_A(1)$,
	and the permutation
	\begin{equation}
		g \ceqq (1 \mid 3) \cdot (2 \mid 4) \cdot ( 5 \mid 7 \mid 6 \mid 8) \in W.
	\end{equation}
	Identifying $\mf t = V^+_8 \simeq \mb C^8$,
	and taking $r = 2$,
	one has
	\begin{equation}
		\label{eq:explicit_eigenspace}
		E
		\ceqq \mf t(g,-1)
		= \Set{ (a,b,-a,-b,c,c,-c,-c) \in \mb C^8 | a,b,c \in \mb C },
	\end{equation}
	which is \emph{not} maximal,
	and
	\begin{equation}
		\mf t_\phi
		= \Set{ (a',a',b',b',c',c',d',d') \in \mb C^8 | a',b',c',d' \in \mb C } \simeq \mb C^4.
	\end{equation}
	Here $W(\phi) \simeq W_A(4)$ acts by permuting the four size-$2$ blocks
	\begin{equation}
		\ol i \ceqq \set{ i,i+1 } \sse \set{1,\dc,8},
		\qquad i \in \set{1,3,5,7}.
	\end{equation}
	Now,
	e.g.,
	the eigenvector $A \ceqq \diag (1,1,-1,-1,2,2,-2,-2) \in E_\phi = E \cap \mf t_\phi$ is in general position in $\mf t_\phi$,
	and in the above identification one has
	\begin{equation}
		g_\phi
		= ( \ol 1 \mid \ol 3 ) \cdot ( \ol 5 \mid \ol 7 ) \in W(\phi).
	\end{equation}
	Finally,
	consider the element $h_\phi \ceqq ( \ol 1 \mid \ol 5 ) \cdot (\ol 3 \mid \ol 7 )$,
	commuting with $g_\phi$ (the two generate a copy of Klein's $4$-group $V_4 = (\mb Z \bs 2\mb Z)^2$):
	then no lift $\wt h \in N_W(\mf t_\phi)$ preserves $E$ on the whole.
	Indeed,
	the sixteen lifts are
	\begin{equation}
		\wt h
		= h h' \in h W_{\mf t_\phi},
		\qquad
		h \ceqq (1 \mid 5) \cdot (2 \mid 6) \cdot (3 \mid 7) \cdot (4 \mid 8),
		\quad h' \in W_{\mf t_\phi} \simeq \mf S_2^4,
	\end{equation}
	and (e.g.)~they move the eigenvector $A' \ceqq \diag (1,2,-1,-2,3,3,-3,-3)$ out of $E$.
\end{exem}

\section*{Third interlude}

The next local aim is to identify the deck groups of \S\S~\ref{sec:nonpure_case_generic}--\ref{sec:nonpure_case_general} as \emph{wild} versions of relative Weyl groups,
in the language of the Levi subcosets of \S~\ref{sec:howlett}.
Precisely:
(i) in \S~\ref{sec:modular_groups_are_twisted_cabled_weyl} we introduce $r$-Galois-closed Levi-subcoset filtrations of $\mb G = (\mf t,W)$,
and define the modular/$r$-ramified WWGs (cf.~Def.~\ref{def:levi_subcoset_filtrations} +~Lem.-Def.~\ref{lem:twisted_cabled_weyl});
and (ii) in \S~\ref{sec:modular_reparameterization} we rephrase the `modularity' condition as a generalization of Springer's regularity,
separating the `unramified' moduli from the twist (cf.~Rmk.~\ref{rmk:modular_group_description_intro}).

In view of Rmk.~\ref{rmk:about_wmcgs},
Levi-subcoset filtrations provide the parameters of admissible-deformation families of irregular types/classes,
whence the topological skeleta of wild curves (cf.~Cor.~\ref{cor:irregular_moduli} +~\S~\ref{sec:topological_skeleta}).

Again,
afterwards:
(iii) the hyperplane arrangements themselves are classified in \S\S~\ref{sec:lehrer_springer_theory}--\ref{sec:D_trees};
(iv) in \S~\ref{sec:tglwmcgs} we formally introduce the local $G$-WMCGs,
and we prove Thm.~\ref{thm:thm_1_intro}~(4.c);
and (v) in \S~\ref{sec:interior_twist} we also twist the interior of the curve,
adding the starting choice of outer automorphisms of the root system.

\section{Modular groups and wild Weyl groups}
\label{sec:modular_groups_are_twisted_cabled_weyl}

\subsection{}

Let again $\wh Q$ be an $r$-Galois-closed irregular type,
with irregular class $\wh\Theta$.
Recall that \S\S~\ref{sec:generic_pure_case}--\ref{sec:nonpure_case_general} characterize the $r$-admissible-deformation space of $\wh\Theta$ in terms of a Galois covering $\bm B_r \bigl( \wh Q \bigr) \thra \bm B_r \bigl( \wh\Theta \bigr)$,
whose Galois group is a subquotient of $W$ (cf.~Thm.-Def.~\ref{thm:full_general_twisted_deformation_space} +~Cor.-Def.~\ref{cor:explicit_deck_group}).

We also view these groups of deck transformations as being `modular',
because their actions govern the local moduli of wild curves in meromorphic $2d$ gauge theory.
Recall that they are finite quotients of full/nonpure local WMCGs---%
and pure ones are the kernels.

\subsubsection{}

Using Howlett's complements,
here we first \emph{noncanonically} embed these groups into $W$ (cf.~Lem.-Def.~\ref{lem:cabled_relative_weyl}~(2.) +~Cor.~\ref{cor:twisted_weyl_from_coset_normalizers}).

Then,
more importantly,
we reinterpret modular groups as \emph{wild} versions of---%
relative---%
Weyl groups,
in the language of nested sequences of Levi reflection subcosets,
building on Cor.~\ref{cor:levi_subcoset}.
This identifies their natural discrete set of parameters.

\subsection{Untwisted/unramified case}

Hereafter,
and until the end of \S~\ref{sec:modular_groups_are_twisted_cabled_weyl},
let $s \geq 1$ be an integer,
and $\bm \phi = (\phi_1,\dc,\phi_s)$ a (depth-$s$) Levi-subsystem filtration of $\Phi$---%
so that $\phi_{s+1} \ceqq \Phi$,
as usual.

The starting point is the following generalization of a relative Weyl group:

\begin{enonce}{Lemma-Definition}
	\label{lem:cabled_relative_weyl}
	\leavevmode

	\begin{enumerate}
		\item
		      The \emph{wild Weyl group} (= WWG) \emph{of} $(\mf g,\mf t,\bm \phi)$ is the subquotient $W(\bm\phi)$---%
		      of $W$---%
		      defined by~\eqref{eq:final_galois_covering},
		      i.e.:
		      \begin{equation}
			      \label{eq:unramified_cabled_weyl_2}
			      W(\bm\phi)
			      = W^{\bm \phi} \bs W_{\mf t_{\phi_1}} \lhra W(\phi_1),
			      \qquad W^{\bm \phi}
			      \ceqq N_W(\mf t_{\bm\phi})
			      = \bigcap_{i = 1}^s N_W(\mf t_{\phi_i}) \sse W.
		      \end{equation}
		      (The split Lie algebra $(\mf g,\mf t)$ is typically tacit.)

		\item
		      And choosing a base $\Delta \sse \Phi$ of simple roots such that $\phi_1$ is $\Delta$-\emph{dominant} yields a group isomorphism
		      \begin{equation}
			      \label{eq:generalized_relative_weyl}
			      W(\bm\phi)
			      \simeq N_W(\Delta_{\phi_1}) \cap \bigcap_{i = 2}^s N_W(\mf t_{\phi_i}) \sse W.
		      \end{equation}
		      (Cf.~\eqref{eq:splitting_parabolic_normalizers}.)
	\end{enumerate}
\end{enonce}

\begin{proof}
	By choice,
	one has
	\begin{equation}
		W_{\mf t_{\phi_1}} \sse W_{\mf t_{\phi_i}} \sse N_W(\mf t_{\phi_i}),
		\qquad i \in \set{2,\dc,s},
	\end{equation}
	and so $W_{\mf t_{\phi_1}} \sse H'' \ceqq \bigcap_{i = 2}^s N_W(\mf t_{\phi_i})$.
	Now apply Lem.~\ref{lem:intersected_complement} with $H' \ceqq W_{\mf t_{\phi_1}}$ and $H \ceqq W$,
	finding
	\begin{equation}
		\label{eq:flag_stabilizer_decomposition}
		N_W(\mf t_{\bm\phi})
		= \bigcap_{i = 1}^s N_W(\mf t_{\phi_i})
		= \Bigl( N_W(\Delta_{\phi_1}) \cap \bigcap_{i = 2}^s N_W(\mf t_{\phi_i}) \Bigr) \lts W_{\mf t_{\phi_1}}. \qedhere
	\end{equation}
\end{proof}

\begin{lemm}
	\label{lem:intersected_complement}

	Let $H$ be a group,
	and $H',H'' \sse H$ two subgroups such that:
	(i) one has $H' \sse H''$;
	and (ii) there exists a semidirect complement $\wt H'$ of $H'$ inside $N_H(H')$,
	i.e.,
	one has $N_H(H') = \wt H' \lts H' \sse H$.
	Then one also has
	\begin{equation}
		N_{H''}(H')
		= N_H(H') \cap H''
		= \bigl(\wt H' \cap H''\bigr) \lts H' \sse H.
	\end{equation}
\end{lemm}

\begin{proof}[Proof omitted]
\end{proof}

\begin{rema}
	\label{rmk:about_cabled_weyl_groups}

	One might say that $W(\bm\phi)$ is the WWG of $\bm\phi$ `\emph{relative to} $\phi_1$',
	generalizing the case $s = 1$.\fn{
		We would also like to use the portmanteau of `Weyld group',
		overlapping `Weyl' and `wild';
		unfortunately,
		that might read like a typo.}

	More importantly,
	the group of Lem.-Def.~\ref{lem:cabled_relative_weyl} might be viewed as a `cabled' version of a Weyl group,
	with terminology originating from the untwisted/unramified setting on vector bundles:
	indeed,
	there $W(\bm\phi)$ can be iteratively computed along the subtrees of a finite decorated rooted tree,
	using the cabling of the symmetric group operad (cf.~the toy-example at Exmp.~\ref{exmp:no_splitting_no_reflection},
	and~\cite{doucot_rembado_2025_topology_of_irregular_isomonodromy_times_on_a_fixed_pointed_curve} in general).

	Again the twisted setting was considered in~\cite{boalch_doucot_rembado_2025_twisted_local_wild_mapping_class_groups_configuration_spaces_fission_trees_and_complex_braids},
	and it will be further extended in \S\S~\ref{sec:BC_trees}--\ref{sec:D_trees}---%
	beyond type $A$.
\end{rema}

\begin{exem}
	\label{exmp:nested_obstructions_1}

	Algebraically,
	one cannot naively iterate~\eqref{eq:flag_stabilizer_decomposition},
	simply because $W_{\mf t_{\phi_i}}$ need \emph{not} lie in $N_W(\mf t_{\bm \phi})$ for $i \in \set{2,\dc,s}$.

	Geometrically,
	the obvious tentative inductive argument to extend Lem.~\ref{lem:cabled_relative_weyl} faces a basic obstacle:
	even assuming that $W(\phi_1)$ acts on $\mf t_{\phi_1}$ as a---%
	rational---%
	reflection group (cf.~Prop.-Def.~\ref{prop:howlett}),
	the subflag $\mf t_{\bm\phi'}$ of Cor.-Def.~\ref{cor:explicit_deck_group} need \emph{not} be made of flats for $W(\phi_1)$.

	E.g.,
	in the context of Exmp.~\ref{exmp:no_splitting_no_reflection},
	take instead
	\begin{equation}
		A_2
		= \begin{pmatrix}
			1 &   &   &   \\
			  & 1 &   &   \\
			  &   & 1 &   \\
			  &   &   & 2
		\end{pmatrix},
		\qquad
		A_1
		= \begin{pmatrix}
			1 &   &   &   \\
			  & 1 &   &   \\
			  &   & 2 &   \\
			  &   &   & 3
		\end{pmatrix} \in \mf t.
	\end{equation}

	Then $\phi_2 \sse \Phi_A(3)$ is a copy of $\Phi_A(2)$,
	and $\phi_1 \sse \phi_2$ is a copy of $\Phi_A(1)$,
	so there are \emph{proper} inclusions $\mf t_{\phi_2} \ssne \mf t_{\phi_1} \ssne \mf t$;
	explicitly,
	one has
	\begin{equation}
		\mf t_{\phi_1}
		= \Set{ \diag(a,a,b,c) | a,b,c \in \mb C } \simeq \mb C^3,
	\end{equation}
	and
	\begin{equation}
		\mf t_{\phi_2}
		= \Set{ \diag(a,a,a,c) | a,c \in \mb C } \simeq \mb C^2.
	\end{equation}
	One finds
	\begin{equation}
		W_{\mf t_{\phi_1}}
		= \set{1, g_{12}} \sse \set{ 1, g_{12}, g_{34}, g_{12}g_{34} }
		= N_W(\mf t_{\phi_1}),
	\end{equation}
	invoking the transpositions $g_{ij} \ceqq g_{\alpha_{ij}} = (i \mid j) \in \mf S_4 \simeq W_A(4)$,
	for $i \neq j \in \set{1,\dc,4}$.
	So $W(\phi_1) = W''(\phi_1) \simeq \mb Z \bs 2 \mb Z$ is generated by a reflection about the hyperplane-trace $H_{\alpha_{34}}(\phi_1) = \Set{ b = c }$.
	Therefore,
	the quotient $W(\phi_1)$ does act on $\mf t_{\phi_1}$ as a rational reflection group,
	with flats $\mf t_{\phi_1}$ and $\Set{ b = c }$ alone:
	but $\mf t_{\phi_2} = \Set{ a = b }$ is \emph{neither}.
	(Incidentally,
	one has $N_W(\mf t_{\bm\phi}) = N_W(\mf t_{\phi_1}) \cap N_W(\mf t_{\phi_2}) = \set{1,g_{12}} = W_{\mf t_{\phi_1}}$,
	so the WWG $W(\bm\phi)$ itself is \emph{trivial} here.)

	The fact is that the two eigenspaces of
	\begin{equation}
		\label{eq:subleading_restriction}
		A'_1 =
		\eval[1]{A_1}_{V'} = \diag(1,1,2),
		\qquad V' \ceqq \ker(A_2 - \Id_{\mb C^4}) \sse \mb C^4,
	\end{equation}
	have different dimensions,
	and do \emph{not} play symmetric roles.
\end{exem}

\begin{rema}
	\label{rmk:no_braid_lift}

	Embedding $\mf{gl}_{\mb C}(V') \simeq \mf{gl}_3(\mb C) \hra \mf{gl}_4(\mb C)$ in the upper-left block,
	one can extract a monolevel case from Exmp.~\ref{exmp:nested_obstructions_1}:
	take instead $\wh Q' \ceqq A'_1 w^{-1}$ (cf.~\eqref{eq:subleading_restriction}).
	Then $\phi \ceqq \phi_1 \sse \Phi_A(2)$ stays as above,
	and this showcases the basic obstruction to lifting Thm.~\ref{thm:complex_refl_groups_from_gauge} to braid groups,
	already in the untwisted setting (where $W(\phi) = Z_{W,\phi}(g)$):
	even if Howlett's obstruction vanishes,
	the hyperplane complement $\bm B_{g,r} \bigl( \wh Q' \bigr) = \bm B \bigl( \wh Q' \bigr) \sse \mf t_\phi$ is \emph{properly} contained inside the $W(\phi)$-regular part of $\mf t_\phi$.

	Indeed,
    on the one hand $W(\phi)$ is trivial,
    while on the other hand the set of traces of root-hyperplanes---%
	of $\Phi_A(2) \sm \phi$---%
	on $\mf t_\phi$ is a singleton which rather generates a non-Howlett `relative' reflection group isomorphic to $\bigl(\mb C^2,\mf S_2\bigr)$ (cf.~\S~\ref{sec:type_A_relative_refl_group}).
\end{rema}

\subsubsection{}

Even without identifying $N_W(\Delta_{\phi_1}) \simeq W(\phi_1)$,
the WWG~\eqref{eq:generalized_relative_weyl} is a subgroup of the relative Weyl group of the smallest Levi subsystem.
Recall that it consists of elements of $W$ stabilizing the flag $\mf t_{\bm\phi}$ of~\eqref{eq:kernel_flag},
up to the pointwise-stabilizer of the largest term;
equivalently,
it consists of elements of $W(\phi_1)$ stabilizing the subflag $\mf t_{\bm\phi'}$ of Cor.-Def.~\ref{cor:explicit_deck_group}.
E.g.,
if fission is toral/complete one has $\mf t_{\phi_1} = \mf t$ and $W_{\mf t_{\phi_1}} = 1$,
whence $W(\bm\phi) \simeq N_W(\mf t_{\bm\phi})$ is the flag-preserving subgroup of $W$.
(Here all is relative to $\phi_1 = \vn$,
i.e.,
the WWG is `absolute'.)

By construction,
any element $g'_{\bm\phi} \ceqq g'_{\phi_1} \in W(\bm\phi)$ is the restriction of an element $g' \in W$ to $\mf t_{\phi_1}$,
and so there is a group embedding $W(\bm\phi) \hra \GL_{\mb C}(\mf t_{\phi_1})$.
More generally,
there are group morphisms
\begin{equation}
	\label{eq:restriction_morphisms}
	g'_{\bm\phi} \lmt \eval[1]{g'_{\bm\phi}}_{\mf t_{\phi_i}}
	= g'_{\phi_i} \cl W(\bm\phi) \lra \GL_{\mb C}(\mf t_{\phi_i}),
	\qquad i \in \set{1,\dc,s},
\end{equation}
whose images are identified with the---%
decreasing---%
quotients
\begin{equation}
	W(\bm\phi) \lthra W(\bm\phi,i)
	\ceqq W(\bm\phi) \bs \bigl( W(\bm\phi) \cap W_{\mf t_{\phi_i}} \bigr)
	\simeq N_W(\mf t_{\bm\phi}) \bs \bigl( N_W(\mf t_{\bm \phi}) \cap W_{\mf t_{\phi_i}} \bigr).
\end{equation}
(Using~\eqref{eq:generalized_relative_weyl},
so that $W(\bm \phi) \lxra{\simeq} W(\bm\phi,1)$;
cf.~also the proof~\ref{proof:cor_explicit_deck_group}.)

\subsubsection{}

The main point at this stage is that:

\begin{coro}
	\label{cor:unramified_cabled_weyl_are_modular}
	\leavevmode

	\begin{enumerate}
		\item
		      All the \emph{untwisted} modular groups are WWGs.

		\item
		      And conversely,
		      all WWGs are \emph{modular}.
	\end{enumerate}
\end{coro}

\begin{proof}
	For the first statement,
	given an irregular type $\wh Q$,
	recall that:
	(i) the Galois group of deck transformations of $\bm B \bigl( \wh Q \bigr) \thra \bm B\bigl( \wh\Theta \bigr)$ only depends on the Levi-subsystem filtration $\bm\phi = \bm\phi \bigl( \wh Q \bigr)$ underlying $\wh Q$ (cf.~\eqref{eq:levi_filtration});
	and (ii) it is then indeed given by~\eqref{eq:unramified_cabled_weyl_2}.

	Conversely,
	starting from any Levi-subsystem filtration $\bm\phi$,
	the second statement is that there exists an irregular type $\wh Q$ such that $\bm\phi = \bm\phi \bigl( \wh Q \bigr)$.
	Indeed,
	it suffices to extract---%
	the coefficients of---%
	$\wh Q$ from the \emph{stratum of} $\bm\phi$,
	i.e.,
	the (nonempty) subspace of $\mf t^s$ defined as the following product of complex hyperplane complements:
	\begin{equation}
		\label{eq:stratum_from_levi_subsystem_filtration}
		\bm B(\bm \phi)
		= \prod_{i = 1}^s \bm B(\bm\phi,i),
		\qquad \bm B(\bm\phi,i)
		\ceqq \mf t_{\phi_i} \sm \, \bigcup_{\phi_{i+1} \sm \phi_i} H_\alpha(\phi_i) \sse \mf t_{\phi_i}. \qedhere
	\end{equation}
\end{proof}

\subsubsection{}

Hereafter,
we generalize the statement/proof of Cor.~\ref{cor:unramified_cabled_weyl_are_modular} in the twisted/ramified case.
The main points are that:
(i) the Levi-subsystem filtration $\bm\phi$ is equivalent to giving a nested sequence of \emph{split} Levi subcosets $\mb L(\phi_i) = \bigl( \mf t,W_{\mf t_{\phi_i}} \bigr)$ of $\mb G = (\mf t,W)$ (cf.~\S~\ref{sec:levi_subcosets});
(ii) the latter uniquely determines the untwisted admissible-deformation class of $\wh Q$;
(iii) one should allow for sequences of \emph{nonsplit} Levi subcosets in general;
and (iv) taking $W$-orbits of such sequences parameterizes $r$-admissible-deformation families of irregular classes.

About the last point,
at this stage observe that the Galois covering~\eqref{eq:final_galois_covering} now reads
\begin{equation}
	\bm B(\bm\phi) \lthra \bm B (\bm\phi) \bs W(\bm\phi),
\end{equation}
and that the base space can be equivalently defined as the (topological) quotient
\begin{equation}
	\label{eq:untwisted_quotient_stratum}
	\bm B \bigl( \, \ol{\bm\phi} \, \bigr)
	\ceqq \bigl( W.\bm B(\bm\phi) \bigr) \bs W,
	\qquad W.\bm B(\bm\phi)
	\ceqq \bigcup_W h \bigl( \bm B(\bm\phi) \bigr) \sse \mf t^s,
\end{equation}
where $\ol{\bm\phi} \ceqq W.\bm\phi$ is the orbit of the Levi-subsystem filtration under the diagonal action of $W$ on subsets of $\Phi$.
Moreover,
one has $W$-\emph{equivariance},
i.e.,
$h \bigl( \bm B(\bm\phi) \bigr) =  \bm B(h.\bm\phi)$ for $h \in W$.
(The WWGs $W(\bm\phi)$ and $W(\bm\phi')$ are \emph{isomorphic} if $\ol{\bm\phi} = \ol{\bm\phi'}$.)

\begin{rema}
	\label{rmk:about_stratifications}

	The spaces~\eqref{eq:stratum_from_levi_subsystem_filtration} are just as in~\eqref{eq:pure_untwisted_nongeneric_deformation_space_1_coeff} +~\eqref{eq:root_stratum_factorization}:
	the notation now records the fact that they only depend on the discrete datum of $\bm \phi$,
	in bijective fashion,
	rather than the continuous datum of $\wh Q$.
	The upshot is the finite \emph{pure root-valuation stratification of} $\mf t^s$,
	in the truncated untwisted version of~\cite{goresky_kottwitz_macpherson_2009_codimensions_of_root_valuation_strata}---%
	over $\mb C$.
	Then the \emph{full/nonpure root-valuation stratification of} $\mb A_s = \mf t^s \bs W$ consists of the locally-closed pieces~\eqref{eq:untwisted_quotient_stratum} (cf.~\cite{calaque_felder_rembado_wentworth_wild_orbits_and_generalised_singularity_modules_stratifications_and_quantisation}).

	In the twisted case,
	one main difference is that there is a proper subset of `modular' sequences,
	matching up with the \emph{nonemptiness} of twisted root-valuation strata (cf.~\cite[Prop.~3.4.1]{goresky_kottwitz_macpherson_2009_codimensions_of_root_valuation_strata}).
	We plan to spell this out elsewhere (cf.~\S~\ref{sec:outlook}).
\end{rema}

\begin{rema}
	\label{rmk:counting_untwisted_strata}

	Nonetheless,
	even in the untwisted case,
	the nonemptiness condition only becomes vacuous after singling out the natural finite index for the root-valuation strata.
	Namely,
	a priori,
    one can consider arbitrary functions/tuples $\bm d \cl \Phi \to \set{0,\dc,s}$:
	but most of them are \emph{not} of the form~\eqref{eq:root_valuations_tuple},
	i.e.,
	they do not come from an irregular type $\wh Q$.

	E.g.,
	if $\mf t$ is the standard Cartan subalgebra of $\mf g = \mf{gl}_3(\mb C)$,
	taking $s = 1$ yields $5$ (nonempty) strata,
	while
	\begin{equation}
		5 < 64
		= 2^{\, \abs \Phi}
		= \bigl| \set{0,1}^{\Phi} \bigr|,
		\qquad \Phi \ceqq \Phi_A(2).
	\end{equation}
	Indeed,
	there are many subsets of roots which are \emph{not} Levi subsystems.

	This becomes more complicated in the multilevel case,
	significantly so when involving the Weyl-group action---%
	counting the full/nonpure strata,
	i.e.,
	the admissible-deformation families of irregular \emph{classes}.
	Nonetheless,
	upon translating to fission trees,
	the full/nonpure counting problem was taken on in~\cite{boalch_2025_counting_the_fission_trees_and_nonabelian_hodge_graphs_untwisted_case}.
	In the viewpoint of Levi-subsystem filtrations,
	this involves orbits of nested set-partitions under the diagonal action of the symmetric group,
	while \emph{signed partitions} naturally appear in types $\bullet \in \set{B,C,D}$---%
	with the additional issue that \emph{not} all root subsystems are Levi,
	cf.~\S\S~\ref{sec:type_BC_relative_refl_group}--\ref{sec:type_D_relative_refl_group}.
\end{rema}

\subsection{Pure twisted/ramified case}

\subsubsection{}
\label{sec:twisted_cabled_weyl}

Now choose another integer $r \geq 1$.
Suppose in addition that there exists an element $g \in W$,
and a tuple $(A_1,\dc,A_s) \in \mf t^s$,
such that:
\begin{equation}
	\label{eq:modular_condition}
	\phi_{A_i} \cap \dm \cap \phi_{A_s}
	= \phi_i,
	\quad A_i \in \mf t (g,\zeta_r^i),
	\qquad i \in \set{1,\dc,s},
\end{equation}
invoking the usual Levi annihilators/eigenspaces.
It follows that $g \in \mb L_r \bigl( \wh Q \bigr)$,
introducing the irregular type $\wh Q \ceqq \sum_{i = 1}^s A_i w^{-i}$.
Therefore,
one has $g \in N_W(\mf t_{\bm \phi})$ (cf.~Lem.-Def.~\ref{lem:flag_stabilizer}),
with well-defined class/restriction $g_{\bm\phi} \ceqq g_{\phi_1} \in W(\bm\phi)$.

Now recall that we consider the following `ramified' version of the WWG,
as per~\eqref{eq:general_galois_covering}:

\begin{coro}
	\label{cor:twisted_weyl_from_coset_normalizers}

	In the notation of Thm.-Def.~\ref{thm:full_general_twisted_deformation_space} +~Lem.-Def.~\ref{lem:cabled_relative_weyl},
	there is a group isomorphism
	\begin{equation}
		\label{eq:cabled_twisted_relative_weyl_with_complement}
		Z_{W,\bm\phi}(g)
		\simeq \bigl( N_W(\Delta_{\phi_1}) \cap N_W(\mf t_{\phi_1},g) \bigr) \cap \bigcap_{i = 2}^s N_W(\mf t_{\phi_i},g) \sse W,
	\end{equation}
	where $N_W(\mf t_{\phi_i},g)$ is the \emph{normalizer}---%
	in $N_W(\mf t_{\phi_i})$---%
	\emph{of the Levi reflection subcoset} $\mb L_r(\phi_i) = (\mf t,g W_{\mf t_{\phi_i}})$,
	for $i \in \set{1,\dc,s}$ (cf.~Cor.-Def.~\ref{cor:levi_subcoset}~(1.)).
\end{coro}

\begin{proof}
	Use~\eqref{eq:generalized_relative_weyl},
	and note that the preimages of the centralizers of $g_{\phi_i} \in W(\phi_i)$,
	along the canonical projections $p_{\phi_i} \cl N_W(\mf t_{\phi_i}) \thra W(\phi_i)$,
	coincide with the given coset-normalizers.

	Indeed,
	choose $h \in N_W(\mf t_{\phi_i})$:
	the equality $h(g W_{\mf t_{\phi_i}})h^{-1} = g W_{\mf t_{\phi_i}}$ yields $hgh^{-1} = g h'$ for suitable $h' \in W_{\mf t_{\phi_i}}$,
	and so $h_{\phi_i} g_{\phi_i} = g_{\phi_i} h_{\phi_i} \in W(\phi_i)$.
	Conversely,
	one has
	\begin{equation}
		h ( gh' ) h^{-1}
		= h g h^{-1} \cdot \wt h'
		= g \cdot h'' \wt h' \in g W_{\mf t_{\phi_i}},
		\qquad h' \in W_{\mf t_{\phi_i}},
		\quad \wt h'
		\ceqq h h' h^{-1},
	\end{equation}
	for a suitable $h'' \in W_{\mf t_{\phi_i}}$,
	provided that $g$ commutes with $h$ modulo the parabolic.
\end{proof}

\subsubsection{}

In view of Cor.-Def.~\ref{cor:explicit_deck_group},
the group~\eqref{eq:cabled_twisted_relative_weyl_with_complement} coincides with the centralizer of $g_{\bm\phi}$ in $W(\bm\phi)$,
generalizing the monolevel case.

Moreover,
analogously to its untwisted version,
there is a group morphism
\begin{equation}
	Z_{W,\bm\phi}(g)
	= Z_{W,\bm\phi}(g,1) \lra \GL_{\mb C} \bigl( \mf t_{\phi_1}(g_{\bm\phi},\zeta_r) \bigr).
\end{equation}
(Beware however that this need \emph{not} be an embedding,
contrary to the untwisted case.)
More generally,
there are `restricted' actions
\begin{equation}
	\label{eq:restriction_morphisms_twisted}
	Z_{W,\bm\phi}(g) \lthra Z_{W,\bm\phi}(g,i)
	\ceqq
	Z_{W,\bm\phi}(g) \bs \bigl( Z_{W,\bm\phi}(g) \cap W_{\mf t_{\phi_i}} \bigr) \lra \GL_{\mb C} \bigl( \mf t_{\phi_i}(g_{\phi_i},\zeta_r^i) \bigr),
\end{equation}
analogous to~\eqref{eq:restriction_morphisms},
for $i \in \set{1,\dc,s}$.

\subsubsection{}

But more importantly,
the group~\eqref{eq:cabled_twisted_relative_weyl_with_complement} only depends on the Levi subcoset $\mb L_r(\phi_1) = (\mf t,gW_{\mf t_{\phi_1}})$,
rather than $g$ itself.

Therefore,
we now completely turn to the viewpoint of \S~\ref{sec:reflection_cosets}.
Precisely:
(i) by definition,
the relative Weyl group $W(\phi)$ of a Levi subsystem $\phi$ consists of the set of Levi subcosets $\mb L = (\mf t,g W_{\mf t_\phi})$ through $W_{\mf t_\phi}$ (for $g \in N_W(\mf t_\phi)$);
and (ii) whenever needed,
we tacitly view these subcosets as classes/restrictions $g_\phi \in W(\phi)$.
In particular,
the $\zeta$-\emph{eigenspace of} $\mb L$ \emph{on} $\mf t_\phi$ is well-defined by
\begin{equation}
	\label{eq:intrinsic_eigenspace}
	\mf t_{\phi}(\mb L,\zeta)
	\ceqq \mf t_\phi(g_\phi,\zeta) \sse \mf t_\phi,
	\qquad \zeta \in \mb C^{\ts}.
\end{equation}

This provides clean terminology,
which we expand in the following series of definitions/statements:

\begin{defi}
	\label{def:levi_subcoset_filtrations}

	Fix an integer $s \geq 1$.
	A \emph{depth}-$s$ \emph{Levi-subcoset filtration of} $\mb G = (\mf t,W)$ is a sequence $\ul{\mb L} = ( \mb L_1,\dc,\mb L_s,\mb L_{s+1} )$ of reflection cosets,
	where:
	\begin{enumerate}
		\item
		      one has $\mb L_{s+1} = \mb G$;

		\item
		      and $\mb L_i$ is a Levi subcoset of $\mb L_{i+1}$,
		      for $i \in \set{1,\dc,s}$.
	\end{enumerate}
\end{defi}

(We typically drop the ambient reflection group $\mb L_{s+1}$ from the notation.)

\begin{rema}
	A Levi-subcoset filtration canonically determines an underlying Levi-subsystem filtration $\bm\phi = \bm\phi(\ul{\mb L})$,
	such that $\mb L_i$ consists of a coset through the parabolic subgroup $W_i \ceqq W_{\mf t_{\phi_i}} \sse W$.\fn{
		If two left cosets $g'W',g''W'' \sse W$ are equal as \emph{subsets} of $W$,
		then in particular $W' = W''$ as \emph{subgroups} of $W$.}

	Moreover,
	being a Levi subcoset is a \emph{transitive} notion,
	and so:
	(i) $\mb L_i$ is a Levi subcoset of $\mb L_{i+1},\dc,\mb L_s$;
	and (ii) $\mb L_1,\dc,\mb L_s$ are all Levi subcosets of $\mb G$.
	Indeed,
	e.g.,
	by definition,
	the two top levels provide elements $g_s,g_{s-1} \in W$ such that
	\begin{equation}
		\mb L_s
		= (\mf t,g_s W_s),
		\quad \mb L_{s-1} = (\mf t,\wt g_{s-1} W_{s-1}),
		\qquad \wt g_{s-1} \ceqq g_s g_{s-1} \in W,
	\end{equation}
	where:
	(i) $g_s \in N_W(W_s) = N_W(\mf t_{\phi_s})$;
	(ii) $g_{s-1} \in W_s$;
	and (iii) $\wt g_{s-1} \in N_W(W_{s-1})$.
\end{rema}

\subsubsection{}

Hereafter,
we will at times---%
unambiguously---%
drop the Cartan subalgebra from the notation.
Then $\mb L_s = g_s W_s = g_s W_{\mf t_{\phi_s}}$ can be also identified with a subset of $W$,
etc.
(Strictly speaking,
this only retains the underlying coset and forgets the ambient reflection representation.)
The fact that $\ul{\mb L}$ is a Levi-subcoset filtration may now be consistently denoted as a nested inclusion $\mb L_1 \sse \dm \sse \mb L_s \sse W$,
because such \emph{set-theoretic} inclusions do hold.

Conversely:

\begin{lemm}
	\label{lem:levi_subcoset_filtration_set_inclusions}

	A sequence $(\mb L_1,\dc,\mb L_s)$ of Levi subcosets of $\mb G$,
	such that there are set-theoretic inclusions $\mb L_1 \sse \dm \sse \mb L_s \sse W$,
	is a \emph{Levi-subcoset filtration}.
\end{lemm}

\begin{proof}[Postponed to~\ref{proof:lem_levi_subcoset_filtration_set_inclusions}]
\end{proof}

\subsubsection{}

Now let $W$ act on Levi subcosets of $\mb G$,
by \emph{conjugation}.
Precisely,
if $g \in W$ preserves the Levi subsystem $\phi \sse \Phi$,
and if $\mb L \ceqq g W_{\mf t_\phi} \sse W$ is the corresponding reflection coset,
then set
\begin{equation}
	\label{eq:levi_subcoset_conjugation}
	h.\mb L
	\ceqq
	h (gW_{\mf t_\phi}) h^{-1}
	= h g h^{-1} \cdot W_{\mf t_{\phi'}} \sse W,
	\qquad
	\quad \phi' \ceqq h(\phi) \sse \Phi.
\end{equation}
(Recall that $W_{\mf t_{\phi'}} = h W_{\mf t_\phi}h^{-1}$.)
In turn,
consider the corresponding \emph{coset stabilizer},
viz.,
the subgroup
\begin{equation}
	\label{eq:levi_subcoset_stabilizer}
	W^{\mb L}
	\ceqq \Set{ g' \in W | g'.\mb L = \mb L }.
\end{equation}
This generalizes at once:
(i) normalizers of parabolic subgroups of $W$,
i.e.,
split Levi subcosets of $\mb G$;
and (ii) centralizers of elements of $W$,
corresponding to Levi subcosets through the trivial subgroup,
i.e.,
maximal tori $\mb T$ of $\mb G$.

Then,
more generally,
we make the following:

\begin{enonce}{Lemma-Definition}
	\label{lem:twisted_cabled_weyl}

	Let $\ul{\mb L} = (\mb L_1,\dc,\mb L_s)$ be a Levi-subcoset filtration of $\mb G$,
	with underlying Levi-subsystem filtration $\bm \phi \ceqq \bm\phi(\ul{\mb L})$,
	and set $W_i \ceqq W_{\mf t_{\phi_i}} \sse W$.
	Then:
	\begin{enumerate}
		\item
		      the smallest parabolic subgroup $W_1$ lies in $W^{\mb L_i}$ for all $i \in \set{1,\dc,s}$,
		      and it is a \emph{normal} subgroup of $W^{\mb L_1}$;

		\item
		      the \emph{wild Weyl group} (= WWG) \emph{of} $(\mf g,\mf t,\ul{\mb L})$ is the quotient
		      \begin{equation}
			      \label{eq:twisted_cabled_weyl}
			      W(\ul{\mb L})
			      = W^{\ul{\mb L}} \bs W_1 \lhra W^{\mb L_1} \bs W_1,
			      \qquad W^{\ul{\mb L}}
			      \ceqq \bigcap_{i = 1}^s W^{\mb L_i} \sse W;
		      \end{equation}
		      (Again this is `\emph{relative to} $\phi_1$',
		      and the pair $(\mf g,\mf t)$ is typically implicit.)

		\item
		      there is an equality
		      \begin{equation}
			      \label{eq:untangling_normalizer_intersection}
			      W^{\ul{\mb L}}
			      = W^{\mb L_1} \cap N_W(\mf t_{\bm\phi}),
		      \end{equation}
		      of subgroups of $W$;

		\item
		      given another integer $r \geq 1$,
		      we say that $\ul{\mb L}$ is $r$-\emph{Galois-closed} if one has
		      \begin{equation}
			      \label{eq:nested_flat_density}
			      \ol{ \mf t_{\phi_i}(\mb L_i,\zeta_r^i)}^{\mathsmaller W_{i+1}}
			      = \mf t_{\phi_i},
			      \qquad i \in \set{1,\dc,s},
		      \end{equation}
		      with the usual convention that $W_{s+1} \ceqq W$;

		\item
		      the above condition is \emph{equivalent} to having
		      \begin{equation}
			      \label{eq:untangled_nested_flat_density}
			      \ol{\bm E_{\geq i}}^{\mathsmaller W}
			      = \mf t_{\phi_i},
			      \qquad \bm E_{\geq i}
			      = \bm E_{\geq i}(\ul{\mb L})
			      \ceqq \sum_{j = i}^s \mf t_{\phi_j}(\mb L_j,\zeta_r^j) \sse \mf t_{\phi_i},
			      \quad i \in \set{1,\dc,s};
		      \end{equation}

		\item
		      and if~\eqref{eq:nested_flat_density}--\eqref{eq:untangled_nested_flat_density} hold,
		      then we say that the WWG~\eqref{eq:twisted_cabled_weyl} is $r$-\emph{ramified};
		      we stress this by adding a subscript,
		      as usual,
		      writing $W_r(\ul{\mb L}) = W(\ul{\mb L})$.
	\end{enumerate}
\end{enonce}

\begin{proof}[Postponed to~\ref{proof:lem_twisted_cabled_weyl}]
\end{proof}

\begin{rema}
	The conditions~\eqref{eq:nested_flat_density}--\eqref{eq:untangled_nested_flat_density} are equivalent when taken in aggregate,
	i.e.,
	for \emph{all} indices $i \in \set{1,\dc,s}$.
	(Strictly speaking,
	the latter is levelwise stronger,
	as it implies the former for any index $i$,
	cf.~the proof~\ref{proof:lem_twisted_cabled_weyl}.)
\end{rema}

\begin{rema}
	The larger the parabolic subgroup $W_{i+1} \sse W$,
	the finer the associated flat-closure operator.
	E.g.,
	if $W_{i+1}$ is trivial then $\ol V^{\mathsmaller{W_{i+1}}} = \mf t$ for \emph{any} subspace $V \sse \mf t$.
	This is consistent,
	since the inclusion $W_i \sse W_{i+1}$ implies that $\mf t_{\phi_i} = \mf t$,
	which is the unique flat of $(\mf t,W_{i+1})$---%
	and the condition~\eqref{eq:nested_flat_density} is then vacuous.
\end{rema}

\begin{rema}
	\label{rmk:about_eigenspace_chain}

	The sequence of vector spaces
	\begin{equation}
		\label{eq:trace_eigenspace_sequence}
		\ul E
		= (E_1,\dc,E_s),
		\qquad E_i
		= E_i (\ul{\mb L})
		\ceqq \mf t_{\phi_i}(\mb L_i,\zeta_r^i) \sse \mf t_{\phi_i},
	\end{equation}
	is far from arbitrary.
	First,
	the statement of~\cite[Thm.~3.1]{koonin_topology_of_eigenspace_posets_for_unitary_reflection_groups} implies that $E_i \sse \mf t$ is also a $\zeta_r^i$-eigenspace of an element of $W$ (generalizing~\cite[Thm.~1.2]{miller_2015_eigenspace_arrangements_of_reflection_groups}):
	so there exist elements $h_1,\dc,h_s \in W$ such that one has
	\begin{equation}
		E_i
		= \mf t(h_i,\zeta_r^i) \in \on E(W,\zeta_r^i),
		\qquad i \in \set{1,\dc,s}.
	\end{equation}
	Cf.~Rmk.~\ref{rmk:about_eigenspace_poset},
	but note that this does \emph{not} require that $E'_i \ceqq \mf t(g,\zeta_r^i)$ be maximal in $\on E(W,\zeta_r^i)$,
	where $g \in \mb L_1$ is any `generator' of $\ul{\mb L}$ (so that $\mb L_i = g_{\phi_i} \in W(\phi_i)$,
	cf.~Lem.-Def.~\ref{lem:principal_levi_subcoset_filtrations}).
	Precisely,
	rather,
	the intersections $E_i = E'_i \cap \mf t_{\phi_i}$ exhaust the subposet of $\zeta_r^i$-eigenspaces contained in $E'_i$,
	letting $\phi_i \sse \Phi$ vary amongst Levi subsystems (cf.~\S~\ref{sec:levis_and_flats};
	this is the identity of the posets (i) and (iii) in~\cite[Thm.~3.1]{koonin_topology_of_eigenspace_posets_for_unitary_reflection_groups}.)

	Moreover,
	the sequence~\eqref{eq:trace_eigenspace_sequence} is \emph{not} just any element $\ul E \in \prod_{i = 1}^s \on E(W,\zeta_r^i)$:
	again,
	its terms are the traces of the eigenspaces of a \emph{single} element of $W$,
	with rotating eigenvalues $\zeta_r,\dc,\zeta_r^s \in \mb C^{\ts}$,
	onto a flag of flats $\mf t_{\phi_i} = \ol{E_i + \dm + E_s}^{\mathsmaller W}$ (cf.~\eqref{eq:untangled_nested_flat_density}).

	E.g.,
	let
	\begin{equation}
		\ul E^{(k)} \ceqq \sum_{\substack{1 \leq i \leq s, \\ i \equiv k \!\!\! \pmod r}} E_i \sse \mf t,
		\qquad k \in \set{1,\dc,r}.
	\end{equation}
	Then:
	(i) the sum $\ul E^{(1)} + \dm + \ul E^{(r)} \sse \mf t$ is \emph{direct};
	and (ii) by Springer,
	one has
	\begin{equation}
		\dim_{\mb C} \bigl( \ul E^{(k)} \bigr) \leq a(d_k),
		\qquad d_k
		\ceqq r \slash (r \wdg k) \geq 1,
	\end{equation}
	invoking the number $a(d_k) \geq 0$ of degrees of $W$ which are divisible by the order of $\zeta_r^k \in \mb C^{\ts}$---%
	as a root of $1$.
\end{rema}

\begin{enonce}{Lemma-Definition}
	\label{lem:principal_levi_subcoset_filtrations}

	Choose a Levi-subcoset filtration $\ul{\mb L} = (\mb L_1,\dc,\mb L_s)$ of $\mb G$ (cf.~Def.~\ref{def:levi_subcoset_filtrations}).
	We say that $g \in W$ \emph{generates} $\ul{\mb L}$ if one has
	\begin{equation}
		\mb L_i = gW_i \sse W,
		\quad i \in \set{1,\dc,s},
	\end{equation}
	where as usual we write $W_i \ceqq W_{\mf t_{\phi_i}} \sse W$,
	invoking the underlying Levi-subsystem filtration $\bm \phi = \bm \phi(\ul{\mb L}) = (\phi_1,\dc,\phi_s)$.
	It follows that:
	\begin{enumerate}
		\item
		      if $g \in W$ generates $\ul{\mb L}$,
		      then:
		      \begin{enumerate}
			      \item
			            one has $g \in \mb L_1$;

			      \item
			            the element $g$ lies in the flag-stabilizer subgroup of Lem.-Def.~\ref{lem:flag_stabilizer};

			      \item
			            and if $\ul{\mb L}$ is $r$-Galois-closed,
			            then $g^r \in W_1$,
			            i.e.,
			            one has $\delta_{\mb L_1} \mid r$---%
			            involving the \emph{twist-order of} $\mb L_1$;
		      \end{enumerate}
		\item
		      and conversely,
		      any element $g \in \mb L_1$ generates $\ul{\mb L}$.
	\end{enumerate}
\end{enonce}

\begin{proof}[Postponed to~\ref{proof:lem_principal_levi_subcoset_filtrations}]
\end{proof}

\begin{enonce}{Theorem-Definition}
	\label{thm:levi_subcosets_and_modular_groups}

	\leavevmode
	\begin{enumerate}
		\item
		      All the groups of deck transformations appearing in \S\S~\ref{sec:nonpure_case_generic}--\ref{sec:nonpure_case_general} are $r$-\emph{ramified WWGs},
		      for some integer $r \geq 1$.

		\item
		      Conversely,
		      let $\ul{\mb L} = (\mb L_1,\dc,\mb L_s)$ be an $r$-\emph{Galois-closed} Levi-subcoset filtration of $\mb G$,
		      with underlying Levi-subsystem filtration $\bm \phi \ceqq \bm\phi(\ul{\mb L})$.
		      Then:
		      \begin{enumerate}
			      \item
			            the (topological) spaces
			            \begin{equation}
				            \label{eq:intrinsic_pure_deformation_factors}
				            \bm B_r (\ul{\mb L},i)
				            \ceqq \mf t_{\phi_i} (\mb L_i,\zeta_r^i) \, \bigsm \, \bigcup_{\phi_{i+1} \sm \phi_i} H_\alpha(\mb L_i,\zeta_r^i),
			            \end{equation}
			            for $i \in \set{1,\dc,s}$,
			            are \emph{nonempty} complex hyperplane complements,
			            setting
			            \begin{equation}
				            \label{eq:hyperplane_intersection_2}
				            H_\alpha(\mb L_i,\zeta)
				            \ceqq H_\alpha \cap \mf t_{\phi_i}(\mb L_i,\zeta) \sse \mf t_{\phi_i}(\mb L_i,\zeta),
				            \qquad \zeta \in \mb C^{\ts},
				            \quad \alpha \in \phi_{i+1} \sm \phi_i.
			            \end{equation}
			            (Cf.~\eqref{eq:hyperplane_intersection} +~\eqref{eq:general_g_twisted_deformation_factor} +~\eqref{eq:intrinsic_eigenspace}.)

			      \item
			            and there exists an $r$-Galois-closed irregular type $\wh Q \in \smash{\wh{\mc{IT}}}_r^{\mathsmaller\leq s}$ such that:
			            \begin{enumerate}
				            \item
				                  the Levi-subsystem filtration underlying $\wh Q$ coincides with $\bm \phi$;

				            \item
				                  one has $\mb L_r \bigl( \wh Q \bigr) = \mb L_1$---%
				                  viewing the latter as a subset of $W$;

				            \item
				                  there is an \emph{equality}
				                  \begin{equation}
					                  \label{eq:factors_equality}
					                  \bm B_r (\ul{\mb L},i)
					                  = \bm B_{g,r} \bigl( \wh Q,i \bigr) \sse \mf t_{\phi_i}(g_{\phi_i},\zeta_r^i),
					                  \qquad g \in \mb L_1,
					                  \quad i \in \set{1,\dc,s},
				                  \end{equation}
				                  of complex hyperplane complements;
				                  (The former now manifestly does \emph{not} depend on the choice of a generator $g$.)

				            \item
				                  and there is a \emph{canonical} group isomorphism
				                  \begin{equation}
					                  \label{eq:modular_subquotients_as_ramified_cabled_weyl_groups}
					                  Z_{W,\bm\phi}(g) \simeq W_r (\ul{\mb L}),
					                  \qquad g \in \mb L_1.
				                  \end{equation}
			            \end{enumerate}
		      \end{enumerate}
	\end{enumerate}
\end{enonce}

\begin{proof}
	For the first statement,
	choose an $r$-Galois-closed irregular type $\wh Q$ as in~\eqref{eq:explicit_irregular_type}:
	it determines a Levi-subsystem filtration $\bm\phi \ceqq \bm\phi \bigl( \wh Q \bigr)$ from the nested Levi annihilators of the coefficients $A_1,\dc,A_s \in \mf t$ (cf.~\eqref{eq:levi_filtration}),
	and then further a Levi-subcoset filtration $\mb L_i \ceqq \mb L_r(\phi_i)$ (cf.~Cor.-Def.~\ref{cor:levi_subcoset}~(1.)).
	Let then $g \in \mb L_1 = \mb L_r \bigl( \wh Q \bigr)$ generate $\ul{\mb L} \ceqq (\mb L_1,\dc,\mb L_s)$,
	i.e.,
	equivalently,
	let $g$ generate the $r$-Galois orbit of $\wh Q$ (cf.~Lem.-Def.~\ref{lem:principal_levi_subcoset_filtrations}).
	Then $g_{\phi_i} = \mb L_i \in W(\phi_i)$ has the $\zeta_r^i$-eigenvector $A_i$ in general position in $\mf t_{\phi_i}$,
	\emph{with respect to} the reflection group $(\mf t,W_{\mf t_{\phi_{i+1}}})$;
	the flat-density condition~\eqref{eq:nested_flat_density} holds.
	Indeed,
	by construction,
	one has:
	(i) $A_i \in \mf t_{\phi_{A_i}} \sse \mf t_{\phi_i}$;
	and (ii) $A_i \not\in H_\alpha$ for $\alpha \in \phi_{i+1} \sm \phi_i$.
	The conclusion follows from Thm.-Def.~\ref{thm:full_general_twisted_deformation_space},
	in view of Cor.~\ref{cor:twisted_weyl_from_coset_normalizers} +~Lem.~\ref{lem:coset_normalizer_is_stabilizer}.
	(Here one need \emph{not} choose a complement for $W_{\mf t_{\phi_1}} \sse N_W(\mf t_{\phi_1})$.)

	For the second statement,
	conversely,
	choose an $r$-Galois-closed Levi-subcoset filtration $\ul{\mb L}$,
	with underlying Levi-subsystem filtration $\bm\phi \ceqq \bm\phi(\ul{\mb L})$.
	The condition~\eqref{eq:nested_flat_density} ensures that the spaces~\eqref{eq:intrinsic_pure_deformation_factors} are indeed \emph{nonempty},
	whence they are complex hyperplane complements:
	let then $\wh Q \ceqq \sum_{i = 1}^s A_i w^{-i}$ be \emph{any} irregular type with coefficients $A_i \in \bm B_r(\ul{\mb L},i)$,
	for $i \in \set{1,\dc,s}$.
	By choice,
	the vector $A_s \in \mf t_{\phi_s}(\mb L_s,\zeta_r^s) \sse \mf t_{\phi_s}$ is not annihilated by any root in $\Phi \sm \phi_s$,
	and so $\phi_{A_s} = \phi_s \sse \Phi$.
	Recursively,
	one then proves that $\phi_i = \phi_{A_i} \cap \dm \cap \phi_{A_s}$---%
	for $i \in \set{1,\dc,s-1}$.
	Indeed,
	since $A_i \in \mf t_{\phi_i}$ and $\phi_i \sse \phi_{i+1}$,
	one has
	\begin{equation}
		\phi_i \sse \phi_{A_i} \cap \phi_{i+1} = \phi_{A_i} \cap \phi_{A_{i+1}} \cap \dm \cap \phi_{A_s} = \phi_{A_i} \cap \dm \cap \phi_{A_s}.
	\end{equation}
	Conversely,
	if $\alpha \in \phi_{i+1}$ satisfies $\braket{ \alpha,A_i } = 0$,
	then $\alpha \in \phi_i$,
	as $A_i$ is out of the relevant root hyperplanes.
	In particular,
	the smallest parabolic subgroup $W_1 \ceqq W_{\mf t_{\phi_1}} \sse W$ coincides with the Weyl group of the connected reductive stabilizer $L \ceqq G^{\wh Q} \sse G$ of $\wh Q$.
	Up to here,
	this would work the same in the untwisted case;
	but moreover any element $g \in \mb L_1$,
	i.e.,
	any generator of $\ul{\mb L}$,
	now generates the $r$-Galois-orbit of $\wh Q$,
	because---%
	by choice---%
	one has
	\begin{equation}
		g(A_i)
		= g_{\phi_i}(A_i)
		= \mb L_i(A_i)
		= \zeta_r^i A_i,
		\qquad i \in \set{1,\dc,s}.
	\end{equation}
	So indeed $\mb L_r \bigl( \wh Q \bigr) = g W_1 = \mb L_1$ is the set of such generators (cf.~Rmk.~\ref{rmk:toral_fission}).
	In addition,
	the equality $g_{\phi_i} = \mb L_i \in W(\phi_i) \hra \GL_{\mb C}(\mf t_{\phi_i})$ yields
	\begin{equation}
		\mf t_{\phi_i}(g_{\phi_i},\zeta_r^i)
		= \mf t_{\phi_i}(\mb L_i,\zeta_r^i),
		\qquad H_\alpha(g_{\phi_i},\zeta_r^i)
		= H_\alpha(\mb L_i,\zeta_r^i),
	\end{equation}
	using all the previous notation.
	Therefore,
	the relevant complements (removing the hyperplane union over $\alpha \in \phi_{i+1} \sm \phi_i$) coincide.
	The conclusion follows again from Thm.-Def.~\ref{thm:full_general_twisted_deformation_space} +~Cor.~\ref{cor:twisted_weyl_from_coset_normalizers} +~Lem.~\ref{lem:coset_normalizer_is_stabilizer}.
\end{proof}

\begin{lemm}
	\label{lem:coset_normalizer_is_stabilizer}

	Let $\mb L$ be a Levi subcoset of $\mb G$,
	with underlying Levi subsystem $\phi \ceqq \phi(\mb L)$.
	Then the stabilizer~\eqref{eq:levi_subcoset_stabilizer} coincides with the normalizer of $\mb L$,
	taken \emph{inside} $N_W(\mf t_\phi) \sse W$.
\end{lemm}

\begin{proof}[Proof omitted]
\end{proof}

\begin{rema}
	A posteriori,
	Lem.-Def.~\ref{lem:twisted_cabled_weyl}~(2.)~matches up $W^{\ul{\mb L}}$ with the intersection~\eqref{eq:restricted_centralizer_intersection},
	while~\eqref{eq:untangling_normalizer_intersection} corresponds to the description of Cor.-Def.~\ref{cor:explicit_deck_group}.
\end{rema}

\begin{enonce}{Corollary-Definition}
	\label{cor:levi_subcoset_filtrations_are_complete_invariants}

	Let $\ul{\mb L} = (\mb L_1,\dc,\mb L_s)$ be a Levi-subcoset filtration,
	and $r \geq 1$ an integer.
	The $r$-\emph{ramified stratum of} $\ul{\mb L}$ is the (topological) subspace
	\begin{equation}
		\label{eq:stratum_of_subcoset_filtration}
		\bm B_r \bigl( \ul{\mb L} \bigr)
		\ceqq \prod_{i = 1}^s \bm B_r \bigl( \ul{\mb L},i \bigr) \sse \mf t^s,
	\end{equation}
	in the notation of~\eqref{eq:intrinsic_pure_deformation_factors}.
	Then:
	\begin{enumerate}
		\item
		      the space~\eqref{eq:stratum_of_subcoset_filtration} is \emph{nonempty} if and only if $\ul{\mb L}$ is $r$-Galois-closed;

		\item
		      conversely,
		      if $\wh Q$ is an $r$-Galois-closed irregular type,
		      then one has $\bm B_r \bigl( \wh Q \bigr) = \bm B_r (\ul{\mb L})$,
		      where $\ul{\mb L} \ceqq \ul{\mb L} \bigl( \wh Q \bigr)$ is the---%
		      $r$-Galois-closed---%
		      \emph{Levi-subcoset filtration underlying} $\wh Q$ (cf.~Cor.-Def.~\ref{cor:levi_subcoset});

		\item
		      and if $\wh Q'$ is any other $r$-Galois-closed irregular type,
		      then one has $\wh Q' \sim_r \wh Q$ if and only if $\ul{\mb L} \bigl( \wh Q' \bigr) = \ul{\mb L} \bigl( \wh Q \bigr)$.
	\end{enumerate}
\end{enonce}

\begin{proof}
	One implication of the first statement is in Thm.-Def.~\ref{thm:levi_subcosets_and_modular_groups}~(2.a).
	For the opposite one,
	note that one has $\ol{\mf t_{\phi_i}(\mb L_i,\zeta_r^i)}^{\mathsmaller{W_{i+1}}} \sse \mf t_{\phi_i}$,
	for $i \in \set{1,\dc,s}$,
	as $\mf t_{\phi_i} \sse \mf t$ is a flat of $W_{i+1} \ceqq W_{\mf t_{\phi_{i+1}}}$.
	Now the fact that $\ul{\mb L}$ is \emph{not} $r$-Galois-closed means that there exists $j \in \set{1,\dc,s}$ such that the inclusion is proper,
	i.e.,
	such that
	\begin{equation}
		\mf t_{\phi_j}(\mb L_j,\zeta_r^j) \sse \bigcup_{\phi_{j+1} \sm \phi_j} H_\alpha,
	\end{equation}
	whence $\bm B_r (\ul{\mb L},j) = \bm B_r (\ul{\mb L}) = \vn$.

	The second statement follows from the factorization ~\eqref{eq:factorization_pure_multilevel},
	together with~\eqref{eq:factors_equality}.

	For the last statement,
	suppose first that $\wh Q' \sim_r \wh Q$.
	This means that one has $\wh Q' = \sum_{i = 1}^s A'_iw^{-i}$,
	with coefficients $(A'_1,\dc,A'_s)$ lying in $\bm B_r \bigl( \wh Q \bigr)$;
	i.e.,
	by the above,
	lying in~\eqref{eq:stratum_of_subcoset_filtration}.
	Then the proof of Thm.-Def.~\ref{thm:levi_subcosets_and_modular_groups}~(2.) shows that $\ul{\mb L}' \ceqq \ul{\mb L} \bigl( \wh Q' \bigr)$ equals $\ul{\mb L}$.
	Conversely,
	the equality $\ul{\mb L} = \ul{\mb L}'$ now yields
	\begin{equation}
		\wh Q \in \bm B_r \bigl( \wh Q \bigr)
		= \bm B_r (\ul{\mb L})
		= \bm B_r (\ul{\mb L}')
		= \bm B_r \bigl( \wh Q' \bigr) \ni \wh Q'. \qedhere
	\end{equation}
\end{proof}

\subsubsection{}

Cor.-Def.~\ref{cor:levi_subcoset_filtrations_are_complete_invariants}~(3.)~states that the Levi-subcoset filtration $\ul{\mb L} = \ul{\mb L} \bigl( \wh Q \bigr)$ is a \emph{complete invariant} for the $r$-admissible deformations of $\wh Q$.
In turn,
a WWG $W(\ul{\mb L})$ is $r$-ramified/modular if and only if $\ul{\mb L}$ is $r$-Galois-closed,
i.e.,
if and only if it comes from an $r$-Galois-closed irregular type.
(Again,
the latter is a genuine restriction when $r \geq 2$.)

\subsection{Full/nonpure twisted ramified case}

Finally,
one can draw analogous conclusions about $r$-Galois-closed irregular \emph{classes}.

\subsubsection{}

Let again $\ul{\mb L} = (\mb L_1,\dc,\mb L_s)$ be a Levi-subcoset filtration of $\mb G$ with underlying Levi-subsystem filtration $\bm\phi \ceqq \bm\phi(\ul{\mb L}) = (\phi_1,\dc,\phi_s)$.
Consider the orbit
\begin{equation}
	\label{eq:levi_subcoset_filtration_orbit}
	\ol{\ul{\mb L}}
	= W.\ul{\mb L}
	= \Set{ h.\ul{\mb L} | h \in W },
	\qquad h.\ul{\mb L}
	\ceqq (h.\mb L_1,\dc,h.\mb L_s),
\end{equation}
of $\ul{\mb L}$,
under the \emph{diagonal} $W$-action~\eqref{eq:levi_subcoset_conjugation}.
This consists of Levi-subcoset filtrations of $\mb G$ (cf.~Lem.~\ref{lem:levi_subcoset_filtration_set_inclusions}).
Now $\ul{\mb L}$ is $r$-Galois-closed if $\mf t_{\phi_i}(\mb L_i,\zeta_r^i) \sse \mf t_{\phi_i}$ is flat-dense in $\mf t_{\phi_i}$,
with respect to $W_{i+1} \ceqq W_{\mf t_{\phi_{i+1}}} \sse W$,
for $i \in \set{1,\dc,s}$ (cf.~Lem.-Def.~\ref{lem:twisted_cabled_weyl}~(4.)).
But this happens if and only if the translated eigenspaces,
i.e.,
\begin{equation}
	\label{eq:translated_levi_subcoset_eigenspace}
	h \bigl( \mf t_{\phi_i}(\mb L_i,\zeta_r^i) \bigr)
	= \mf t_{\phi'_i}(\mb L'_i,\zeta_r^i) \sse \mf t_{\phi'_i},
	\qquad \phi'_i
	\ceqq h(\phi_i) \sse \Phi,
	\quad \mb L'_i
	\ceqq h.\mb L_i \sse W,
\end{equation}
are flat-dense in $\mf t_{\phi'_i}$ with respect to $W'_{i+1} \ceqq h.W_{i+1}$---%
for $h \in W$ and for all $i$,
using that $h(\phi_{i+1} \sm \phi_i) = \phi'_{i+1} \sm \phi'_i \sse \Phi$,
etc.
So it makes sense to say/ask that the orbit~\eqref{eq:levi_subcoset_filtration_orbit} is $r$-\emph{Galois-closed}.

In addition,
one proves the following:

\begin{enonce}{Corollary-Definition}
	\label{cor:levi_subcoset_orbit_from_irr_class}

	Let $\wh Q$ be an $r$-Galois-closed irregular type with irregular class $\wh\Theta \ceqq \wh\Theta \bigl( \wh Q \bigr)$.
	Then:
	\begin{enumerate}
		\item
		      the irregular class $\wh\Theta$ \emph{canonically} determines an $r$-Galois-closed \emph{Levi-subcoset-filtration orbit},
		      via $\ol{\ul{\mb L}} \bigl( \wh\Theta \bigr)
			      \ceqq W. \bigl( \ul{\mb L} \bigl( \wh Q \bigr) \bigr)$;

		\item
		      conversely,
		      if a Levi-subcoset-filtration orbit $\ol{\ul{\mb L}}$ is $r$-\emph{Galois-closed},
		      then there exists an $r$-Galois-closed irregular class $\wh\Theta$ such that $\ol{\ul{\mb L}} = \ol{\ul{\mb L}} \bigl( \wh\Theta \bigr)$;

		\item
		      and if $\wh\Theta'$ is any other $r$-Galois-closed irregular class,
		      then one has $\wh\Theta' \sim_r \wh\Theta$ if and only if $\ol{\ul{\mb L}} \bigl( \wh\Theta' \bigr) = \ol{\ul{\mb L}} \bigl( \wh\Theta \bigr)$.
	\end{enumerate}
\end{enonce}

\begin{proof}
	The first statement follows from~\eqref{eq:levi_subcoset_equivariance}---%
	as $\ul{\mb L} \bigl( \wh Q \bigr)$ is $r$-Galois-closed.

	The second statement is due to Thm.-Def.~\ref{thm:levi_subcosets_and_modular_groups}~(2.b):
	the orbit $\ol{\ul{\mb L}}$ is $r$-Galois-closed if and only if this holds for $\ul{\mb L}$,
	whence $\ul{\mb L} = \ul{\mb L} \bigl( \wh Q \bigr)$ for a suitable irregular type $\wh Q$,
	and in turn $\ol{\ul{\mb L}} = \ol{\ul{\mb L}} \bigl( \wh\Theta \bigr)$ for the irregular class $\wh\Theta \ceqq \wh\Theta \bigl( \wh Q \bigr)$.

	For the last statement,
	suppose first that $\wh\Theta' \sim_r \wh\Theta$.
	Up to acting by $W$,
	there exists now an irregular type $\wh Q'$ such that $\wh\Theta' = \wh\Theta \bigl( \wh Q' \bigr)$ and $\wh Q' \sim_r \wh Q$ (cf.~Def.~\ref{def:twisted_admissible_deformation_full}).
	Therefore,
	one has $\ul{\mb L} \bigl( \wh Q' \bigr) = \ul{\mb L}\bigl( \wh Q \bigr)$ (cf.~Cor.-Def.~\ref{cor:levi_subcoset_filtrations_are_complete_invariants}~(3.)),
	and their $W$-orbits coincide.
	Conversely,
	if $\wh\Theta' = \wh\Theta \bigl( \wh Q' \bigr)$ and $\ol{\ul{\mb L}} \bigl( \wh\Theta \bigr) = \ol{\ul{\mb L}} \bigl( \wh\Theta' \bigr)$,
	then the Levi-subcoset filtrations $\ul{\mb L} \ceqq \ul{\mb L} \bigl( \wh Q \bigr)$ and $\ul{\mb L}' \ceqq \ul{\mb L} \bigl( \wh Q' \bigr)$ lie in the same $W$-orbit.
	Thus,
	there exists $h \in W$ such that
	\begin{equation}
		\ul{\mb L}'
		= h.\ul{\mb L}
		= \ul{\mb L} \bigl( \wh Q'' \bigr),
		\qquad \wh Q''
		\ceqq h \bigl( \wh Q \bigr),
	\end{equation}
	using~\eqref{eq:levi_subcoset_equivariance}.
	This implies that $\wh Q' \sim_r \wh Q''$ (cf.~again Cor.-Def.~\ref{cor:levi_subcoset_filtrations_are_complete_invariants}~(3.)),
	whence in turn $\wh \Theta \bigl( \wh Q'' \bigr) = \wh \Theta \sim_r \wh\Theta'$.
\end{proof}

\begin{lemm}
	\label{lem:irr_type_levi_subcosets_equivariance}

	The association $\wh Q \mt \ul{\mb L} \bigl( \wh Q \bigr)$,
	from $r$-Galois-closed irregular types to Levi-subcoset filtrations,
	is $W$-\emph{equivariant}:
	\begin{equation}
		\label{eq:levi_subcoset_equivariance}
		\ul{\mb L} \bigl( h\bigl( \wh Q \bigr) \bigr)
		= h.\ul{\mb L},
		\qquad
		\ul{\mb L} \ceqq \ul{\mb L} \bigl( \wh Q \bigr).
	\end{equation}
\end{lemm}

\begin{proof}[Proof postponed to~\ref{proof:lem_irr_type_levi_subcoset_equivariance}]
\end{proof}

\subsubsection{}

Cor.-Def.~\ref{cor:levi_subcoset_orbit_from_irr_class}~(3.)~states that the Levi-subcoset-filtration orbit $\ol{\ul{\mb L}} = \ol{\ul{\mb L}} \bigl( \wh\Theta \bigr)$ is a \emph{complete invariant} for the $r$-admissible deformations of $\wh\Theta$.

Somewhat conversely,
the material of \S\S~\ref{sec:nonpure_case_generic}--\ref{sec:nonpure_case_general} makes it possible to directly attach $r$-admissible deformations spaces to such Levi-subcoset-filtration orbits.
Precisely,
let again $\ul{\mb L}$ be a Levi-subcoset filtration with underlying Levi-subsystem filtration $\bm\phi \ceqq \bm\phi(\ul{\mb L})$.
Its associated (possibly empty) $r$-ramified stratum $\bm B_r(\ul{\mb L}) \sse \mf t^s$ is as in~\eqref{eq:stratum_of_subcoset_filtration},
and the above implies that one has
\begin{equation}
	h \bigl( \bm B_r (\ul{\mb L}) \bigr)
	= \bm B_r (\ul{\mb L}') \sse \mf t^s,
	\qquad h \in W,
	\quad \ul{\mb L}'
	\ceqq h.\ul{\mb L},
\end{equation}
analogously to~\eqref{eq:strata_permutation}.
Therefore,
one may permute the strata and make the:

\begin{enonce}{Corollary-Definition}[cf.~Cor.-Def.~\ref{cor:levi_subcoset_filtrations_are_complete_invariants}]
	\label{cor:levi_subcoset_filtration_orbits_are_complete_invariants}

	The $r$-\emph{ramified (quotient) stratum of} $\ol{\ul{\mb L}}$ is the (topological) subspace
	\begin{equation}
		\label{eq:stratum_of_subcoset_orbit}
		\bm B_r \bigl( \ol{\ul{\mb L}} \bigr)
		\ceqq \bigl( W.\bm B_r ( \ul{\mb L} ) \bigr) \bs W \sse \mb A_s,
		\qquad W.\bm B_r ( \ul{\mb L} )
		\ceqq \bigcup_W h \bigl( \bm B_r(\ul{\mb L}) \bigr).
	\end{equation}
	(Recall that $\mb A_s = \mf t^s \bs W$.)
	Then:
	\begin{enumerate}
		\item
		      the space~\eqref{eq:stratum_of_subcoset_orbit} is \emph{nonempty} if and only if $\ol{\ul{\mb L}}$ is $r$-Galois-closed;

		\item
		      in that case,
		      there is a Galois covering
		      \begin{equation}
			      \label{eq:galois_covering_stratum}
			      \bm B_r(\ul{\mb L}) \lthra \bm B_r (\ol{\ul{\mb L}}),
		      \end{equation}
		      with Galois group $W_r(\ul{\mb L})$;

		\item
		      conversely,
		      if $\wh \Theta$ is an $r$-Galois-closed irregular class,
		      then one has $\bm B_r \bigl( \wh \Theta \bigr) = \bm B_r (\ol{\ul{\mb L}})$,
		      where $\ol{\ul{\mb L}} \ceqq \ol{\ul{\mb L}} \bigl( \wh \Theta \bigr)$ is the---%
		      $r$-Galois-closed---%
		      Levi-subcoset-filtration orbit underlying $\wh \Theta$ (cf.~Cor.-Def.~\ref{cor:levi_subcoset_orbit_from_irr_class}~(1.));

		\item
		      and in that case,
		      the Galois covering~\eqref{eq:galois_covering_stratum} is \emph{canonically} isomorphic to~\eqref{eq:general_galois_covering}.
	\end{enumerate}
\end{enonce}

\begin{proof}
	The first statement follows from Cor.-Def.~\ref{cor:levi_subcoset_filtrations_are_complete_invariants}~(1.),
	as $\bm B_r \bigl( \ol{\ul{\mb L}} \bigr) \neq \vn$ holds---%
	in $\mb A_s$---%
	if and only if $\bm B_r(\ul{\mb L}) \neq \vn$ does---%
	in $\mf t^s$.

	For the third statement,
	choose an irregular type $\wh Q$ such that $\wh \Theta = \wh \Theta \bigl( \wh Q \bigr)$,
	so that $\ol{\ul{\mb L}} = W.\ul{\mb L}$ with $\ul{\mb L} \ceqq \ul{\mb L} \bigl( \wh Q \bigr)$.
	Then Cor.-Def.~\ref{cor:levi_subcoset_filtrations_are_complete_invariants}~(2.) +~\eqref{eq:levi_subcoset_equivariance} yield
	\begin{equation}
		W.\bm B_r ( \ul{\mb L} )
		= \bigcup_W \bm B_r \bigl( h \bigl( \wh Q \bigr) \bigr) \sse \mf t^s.
	\end{equation}
	Now the point is that the (topological) quotient of this $W$-invariant subspace coincides with the space of $r$-admissible deformations of $\wh \Theta$.
	Indeed,
	one has $\wh\Theta' \sim_r \wh \Theta$ if and only if there exists an irregular type $\wh Q'$ such that:
	(i) $\wh\Theta' = \wh\Theta(\wh Q')$;
	and (ii) $\wh Q' \sim_r h \bigl( \wh Q \bigr)$ for some $h \in W$,
	i.e.,
	$\wh Q' \in \bm B_r \bigl( h \bigl( \wh Q \bigr) \bigr)$.

	For the second statement,
	one can prove that $W^{\ul{\mb L}} \sse W$ coincides with the setwise stabilizer of $\bm B_r (\ul{\mb L}) \sse \mf t^s$,
	and that the normal subgroup $W_{\mf t_{\phi_1}} \sse W^{\ul{\mb L}}$ is the pointwise-stabilizer of all points therein,
	where as usual $\bm\phi \ceqq \bm\phi(\ul{\mb L}) =(\phi_1,\dc,\phi_s)$.
	Indeed,
	choosing an $r$-Galois-closed irregular type $\wh Q$ such that $\ul{\mb L} = \ul{\mb L} \bigl( \wh Q \bigr)$ and $\ol{\ul{\mb L}} = \ol{\ul{\mb L}} \bigl( \wh \Theta \bigl( \wh Q \bigr)\bigr)$,
	this rephrases the proof of Thm.-Def.~\ref{thm:full_general_twisted_deformation_space}.
	Now the last statement follows from the group isomorphism~\eqref{eq:modular_subquotients_as_ramified_cabled_weyl_groups}.
\end{proof}

\section{Modular reparameterizations}
\label{sec:modular_reparameterization}

Here we provide a different parameterization of $r$-admissible-deformation spaces of $r$-Galois-closed irregular types/classes,
separating their `unramified' moduli from the twist.

Rephrasing the modularity condition shows that the latter generalizes the notion of regular elements of $W$---%
à la Springer.
In particular,
it is natural to ask whether $r$-regular elements $g \in W$ are a particular case of more general conjugacy classes of elements of WWGs.

\subsection{Modular pairs/classes}

Given a Levi-subcoset filtration $\ul{\mb L} = (\mb L_1,\dc,\mb L_s)$ as in Def.~\ref{def:levi_subcoset_filtrations},
the WWG $W(\ul{\mb L})$ is the subquotient of $W$ defined by the pair of nested subgroups $(W_1 \sse W^{\ul{\mb L}})$ of~\eqref{eq:twisted_cabled_weyl}.
In turn,
the $W$-conjugacy class of that pair only depends on the underlying orbit $\ol{\ul{\mb L}} \ceqq W.\ul{\mb L}$.
This follows from:
(i) the equality
\begin{equation}
	\label{eq:coset_stabilizer_conjugacy}
	h W^{\mb L} h^{-1}
	= W^{h.\mb L} \sse W,
	\qquad h \in W,
\end{equation}
which holds for any Levi subcoset $\mb L$;
and (ii) the usual fact that
\begin{equation}
	\label{eq:parabolic_conjugacy}
	h W_{\mf t_\phi} h^{-1} = W_{\mf t_{\phi'}} \sse W,
	\qquad \phi'
	\ceqq h(\phi),
\end{equation}
for any Levi subsystem $\phi \sse \Phi$.

Therefore,
varying the Levi-subcoset-filtration orbit parameterizes certain $W$-conjugacy classes of pairs $(W'' \sse W')$ of nested subgroups of $W$,
such that $W' \sse N_W(W'')$.
They are quite rare,
e.g.,
since $W''$ is a parabolic subgroup;
but moreover,
one cares about the subfamily of $r$-\emph{ramified/modular} ones---%
corresponding to the $r$-Galois-closedness of $\ol{\ul{\mb L}}$.

\subsubsection{}
\label{sec:modular_pairs}

In turn,
the data of $\ol{\ul{\mb L}}$ is also equivalent to the $W$-orbit of a suitable pair.

Namely,
consider a pair $(\bm\phi,g_{\bm\phi})$,
where:
(i) $\bm \phi$ is a Levi-subsystem filtration;
and (ii) $g_{\bm\phi} \in W(\bm\phi)$ is an element of the corresponding---%
unramified---%
WWG $W(\bm \phi)$ (cf.~\eqref{eq:unramified_cabled_weyl_2}).
Let $\ul{\mb L} = \ul{\mb L}(\bm\phi,g_{\bm \phi})$ be the Levi-subcoset filtration defined by the following sequence of reflection subcosets of $\mb G$:
\begin{equation}
	\label{eq:levi_subcoset_filtration_from_pair}
	\mb L_i(\bm\phi,g_{\bm\phi})
	\ceqq (\mf t,gW_{\mf t_{\phi_i}}),
	\qquad i \in \set{1,\dc,s},
\end{equation}
where $g \in N_W(\mf t_{\bm\phi})$ is any flag-preserving lift of $g_{\bm\phi}$.
Moreover,
given in addition an element $h \in W$,
consider the new pair $h.(\bm\phi,g_{\bm\phi}) = (\bm\phi',g'_{\bm\phi'})$ consisting of:
(a) the translated Levi-subsystem filtration $\bm\phi' \ceqq h(\bm\phi) = \bigl( h(\phi_1),\dc,h(\phi_s) \bigr)$,
under the diagonal $W$-action;
and (b) the class/restriction $g'_{\bm\phi'} \in W(\bm\phi')$ of the conjugate element $g' = hgh^{-1}$,
for \emph{any} lift $g \in N_W(\mf t_{\bm\phi})$ of $g_{\bm\phi}$.

Then:
\begin{prop}
	\label{prop:pair_parameters_levi_subcoset_filtrations}

	The map $(\bm\phi,g_{\bm\phi}) \mt \ul{\mb L}(\bm\phi,g_{\bm\phi})$,
	targeting the set of Levi-subcoset filtrations of $\mb G$,
	is a $W$-\emph{equivariant bijection}.
\end{prop}

\begin{proof}
	Surjectivity follows from Lem.-Def.~\ref{lem:principal_levi_subcoset_filtrations}.
	Indeed,
	given a Levi-subcoset filtration $\ul{\mb L}$,
	set $\bm\phi \ceqq \bm\phi(\ul{\mb L})$,
	and let $g$ be a generator of $\ul{\mb L}$.
	Then:
	(i) one has $g \in N_W(\mf t_{\bm \phi})$;
	(ii) there is a well-defined class/restriction $g_{\bm\phi} \ceqq g_{\phi_1} \in W(\bm\phi)$;
	and (iii) one has $\mb L_i = \mb L_i(\bm\phi,g_{\bm\phi})$,
	for $i \in \set{1,\dc,s}$.
	Now injectivity is essentially due to the fact that any two generators have the same class modulo $W_{\mf t_{\phi_1}}$.
	Precisely,
	first,
	the equality $\ul{\mb L}(\bm \phi,g_{\bm\phi}) = \ul{\mb L} (\bm\phi',g'_{\bm\phi'})$ implies that $\bm\phi = \bm\phi'$---%
	by construction.
	Second,
	the equality $\ul{\mb L}(\bm\phi,g_{\bm\phi}) = \ul{\mb L}(\bm\phi,g'_{\bm\phi})$ yields $g W_{\mf t_{\phi_1}} = g' W_{\mf t_{\phi_1}}$,
	for any choice of lifts $g,g' \in N_W(\mf t_{\bm\phi})$ of $g_{\bm\phi},g'_{\bm\phi} \in W(\bm\phi)$---%
	respectively.
	Hence $g_{\bm \phi} = g'_{\bm \phi}$.

	Finally,
	$W$-equivariance follows from~\eqref{eq:levi_subcoset_conjugation}.
\end{proof}

\begin{rema}
	Bounding as usual the depth $s \geq 1$ yields a \emph{finite} set of Levi-subcoset filtrations.
	Its cardinality is then the same as that of the set of pairs $(\bm\phi,g_{\bm\phi})$,
	where $\bm\phi$ also has depth/length $s$.
\end{rema}

\subsubsection{}

Prop.~\ref{prop:pair_parameters_levi_subcoset_filtrations} also yields a well-induced \emph{bijection}
\begin{equation}
	\label{eq:pair_orbit}
	\ol{(\bm\phi,g_{\bm\phi})} \lmt W.\bigl( \ul{\mb L}(\bm\phi,g_{\bm\phi}) \bigr),
	\qquad
	\ol{(\bm\phi,g_{\bm\phi})}
	\ceqq \Set{ h.(\bm\phi,g_{\bm\phi}) | h \in W },
\end{equation}
onto the set of Levi-subcoset-filtration orbits.
Moreover,
in the reparameterization of Prop.~\ref{prop:pair_parameters_levi_subcoset_filtrations},
the modularity condition reads as follows:
one has $\ul{\mb L}(\bm\phi,g_{\bm\phi}) = \ul{\mb L} \bigl( \wh Q \bigr)$,
for an $r$-Galois-closed irregular type $\wh Q$,
if and only if $\bm \phi = \bm \phi \bigl( \wh Q \bigr)$ and $g_{\bm\phi} = \mb L_r \bigl( \wh Q \bigr)$---%
viewing the latter as a coset of $W_{\mf t_{\phi_1}} = W^{\wh Q} \sse W$.

In this case one might say that $(\bm\phi,g_{\bm\phi})$ is an $r$-\emph{modular pair}.

\subsubsection{}

In turn,
for any fixed Levi-subsystem filtration $\bm\phi$,
one can impose the modularity condition on $g_{\bm\phi}$.
Precisely,
define the subset of $(\bm\phi,r)$-\emph{modular classes} as
\begin{equation}
	\label{eq:modular_classes}
	M_r ( \bm\phi )
	\ceqq \Set{ g_{\bm\phi} \in W(\bm\phi) | \ul{\mb L}(\bm\phi,g_{\bm\phi}) \text{ is } r\text{-Galois-closed} }.
\end{equation}
This is precisely the subset of cosets $\mb L_r \bigl( \wh Q \bigr)$,
where $\wh Q$ is an $r$-Galois closed irregular type such that $\bm\phi \bigl( \wh Q \bigr) = \bm\phi$:
it is stable under conjugation by $W(\bm\phi)$,
and as usual the inclusion $M_r (\bm\phi) \sse W(\bm\phi)$ is typically \emph{proper} when $r \geq 2$.

We regard this as a fairly vast generalization of Springer's $r$-regularity condition,
in view of the following:

\begin{exem}
	\label{exmp:springer_again}

	Consider $s = 1$ and $\phi = \phi_1 \ceqq \vn$,
	whence as usual $\mf t_\phi = \mf t$ and $W(\phi) = W$.
	An element $g_\phi = g \in W$ is $(\phi,r)$-modular if and only if $\ol{\mf t(g,\zeta_r)}^{\mathsmaller W} = \mf t$,
	i.e.,
	if and only if $\mf t(g,\zeta_r) \cap \mf t_{\reg} \neq \vn$.
	So $g$ is $r$-\emph{regular},
	in the sense of~\cite[\S~4]{springer_1974_regular_elements_of_finite_reflection_groups} and \S~\ref{sec:r_regularity} (and it is trivial if $r = 1$).

	Thus~\eqref{eq:modular_classes} is Springer's $W$-conjugacy class of $r$-regular elements,
	and the WWG is the centralizer of the maximal torus
	\begin{equation}
		\mb T \ceqq \mb L(\vn,g) = (\mf t,g) \sse \mb G,
	\end{equation}
	i.e.,
	the centralizer of $g$.

	This generalizes to arbitrary $s \geq 1$,
	taking the `empty' constant Levi-subsystem filtration
	\begin{equation}
		\bm\phi \ceqq (\underbrace{\vn,\dc,\vn}_{s \text{ times}}),
	\end{equation}
	which yields the quasigeneric case.
	Again $W(\bm\phi) = W$,
	and $g \in W$ is $(\bm\phi,r)$-modular if and only if it admits a regular eigenvector of eigenvalue $\zeta_d \in \mb C^{\ts}$,
	where $d \ceqq r \bs (r \wdg s) \geq 1$ (cf.~Rmk.~\ref{rmk:about_rational_forms}).
	Indeed,
	note that the intermediate flat-density conditions~\eqref{eq:nested_flat_density} are vacuous,
	since an equality $W_i = W_{i+1}$ of parabolic subgroups yields
	\begin{equation}
		\ol{(0)}^{W_{i+1}}
		= \ol{(0)}^{W_i}
		= \bigcap_{\phi_i} H_\alpha
		= \mf t_{\phi_i} \sse \mf t,
		\qquad i \in \set{1,\dc,s-1}.
	\end{equation}
	(This rephrases Rmk.~\ref{rmk:levels_and_topology}.)
\end{exem}

\begin{exem}
	\label{exmp:modular_toral_fission}

	Suppose more generally to have complete/toral fission (for any $s \geq 1$),
	i.e.,
	assume that $\phi_1 = \vn$---%
	in $\bm\phi$.
	Then $W(\bm\phi) = N_W(\mf t_{\bm\phi}) \sse W$,
	and the nested $(\bm\phi,r)$-modularity condition for flag-preserving elements $g$ is genuinely more convoluted.
	As already mentioned,
	the corresponding WWG $W_r(\ul{\mb L}) = W^{\ul{\mb L}} \sse W$ coincides with the centralizer of $g$ in $N_W(\mf t_{\bm\phi})$,
	where $\ul{\mb L} \ceqq \ul{\mb L}(\bm\phi,g)$.
\end{exem}

\begin{rema}
	\label{rmk:transitivity_question}

	It is natural to ask whether the subset $M_r (\bm\phi) \sse W(\bm\phi)$ consists of a \emph{single} $W(\bm\phi)$-conjugacy class,
	for suitable parameters $(\bm\phi,r)$,
	generalizing Springer's:
	this extends (the truncated version of) the question of~\cite[\S~4.9]{goresky_kottwitz_macpherson_2009_codimensions_of_root_valuation_strata} beyond the regime of complete/toral fission.

	Recall that the answer to Goresky--Kottwitz--MacPherson's question is \emph{positive} when $\mf g$ is a simple Lie algebra of classical type~\cite{sabitova_2012_a_simplification_of_root_valuation_data_for_classical_groups}:
	in our monolevel setting,
	a sufficient condition for \emph{any} reductive $\mf g$ follows again from `relative' Springer theory (as in~Thm.~\ref{thm:complex_refl_groups_from_gauge};
	we plan to spell this out elsewhere).
\end{rema}

\subsubsection{}
\label{sec:nonreflection_obstructions_again}

In the context of Exmp.~\ref{exmp:modular_toral_fission},
beware that $N_W(\mf t_{\bm\phi})$ need \emph{not} act as a reflection group on $\mf t_{\phi_1} = \mf t$,
even assuming that Howlett's obstruction $W(\phi_i) \thra \wt W(\phi_i)$ vanishes for all $i \in \set{1,\dc,s}$.
E.g.,
in the type-$A$ situation of Exmp.~\ref{exmp:no_splitting_no_reflection},
the group~\eqref{eq:example_unramified_cabled_weyl} is not generated by reflections:
the transpositions it contains only generate the order-$4$ normal subgroup $1 \lts \mf S_2^2 \sse \mf S_2 \lts \mf S_2^2$.

So there is \emph{no} immediate extension of the argument of Thm.~\ref{thm:complex_refl_groups_from_gauge},
when $s \geq 2$,
to find complex reflection groups;
and a fortiori in the general case---%
where $\phi_1 \neq \vn$ is possible.\fn{
	Of course one might still study nonreflection obstructions for the slotwise actions~\eqref{eq:restriction_morphisms} +~\eqref{eq:restriction_morphisms_twisted}.}~This is not surprising,
as the family of modular subquotients of $W$ is significantly larger than Springer's centralizers.

Indeed,
e.g.:

\begin{prop}
	\label{prop:modular_examples}

	Let $\phi \sse \Phi$ be a Levi subsystem,
	and choose \emph{any} $r$-twisted Levi subcoset $\mb L = g W_{\mf t_\phi} \sse W$ through the corresponding parabolic subgroup (i.e.,
	one has $g \in N_W(\mf t_\phi)$ and $g^r \in W_{\mf t_\phi}$).
	Then there exists a family of Levi-subsystem filtrations $\bm\phi$,
	parameterized---%
	in nonbijective fashion---%
	by the points of the `monolevel' stratum $\bm B(\phi) \sse \mf t_\phi$ (cf.~\eqref{eq:stratum_from_levi_subsystem_filtration}),
	such that:
	\begin{enumerate}
		\item
		      one has $\phi_1 = \phi$;

		\item
		      the length/depth of $\bm\phi$ is at most $r$;

		\item
		      one has $g \in N_W(\mf t_{\bm\phi})$;

		\item
		      and its class $g_{\bm\phi}$ is $(\bm\phi,r)$-\emph{modular},
		      i.e.,
		      one has $g_{\bm\phi} \in M_r \bigl( \bm \phi \bigr)$.
	\end{enumerate}
\end{prop}

\begin{proof}[Proof postponed to~\ref{proof:prop_modular_examples}]
\end{proof}

\subsubsection{}

As a particular case of Prop.~\ref{prop:modular_examples},
for any integer $r \geq 1$ consider the $r$-\emph{torsion subset}
\begin{equation}
	\on{Tor}(W,r)
	\ceqq \Set{ g \in W | g^r = 1 }
	= \bigcup_{d \mid r} \Set{ g \in W | \on{ord}(g) = d },
\end{equation}
which is invariant under conjugation.
Then:

\begin{coro}
	\label{cor:complete_fission_modular_examples}

	Choose \emph{any} $g \in \on{Tor}(W,r)$.
	Then there exists a family of Levi-subsystem filtrations $\bm\phi$,
	parameterized---%
	in nonbijective fashion---%
	by the points of $\mf t_{\reg} \sse \mf t$,
	such that:
	\begin{enumerate}
		\item
		      the conclusions of Prop.~\ref{prop:modular_examples} hold for the pair $(\bm\phi,g_{\bm\phi})$;

		\item
		      and moreover fission is toral/complete,
		      i.e.,
		      one has $\phi_1 = \vn$.
		      (Whence actually $g_{\bm\phi} = g \in W$.)
	\end{enumerate}
\end{coro}

\begin{proof}
	By hypothesis,
	the maximal torus $\mb T \ceqq (\mf t,g) \sse \mb G$ is an $r$-twisted Levi subcoset:
	now apply Prop.~\ref{prop:modular_examples}---%
    with $\phi = \vn$.
\end{proof}

\section*{Fourth interlude}

After establishing the relation with Levi-subcoset filtrations,
and their $W$-orbits,
we now resume the main aim of classifying the (irreducible) complex hyperplane arrangements arising in twisted meromorphic $2d$ gauge theory,
when the associated Weyl group is of \emph{classical} type (proving Thm.~\ref{thm:thm_2_intro}).

Precisely,
in the next \S~\ref{sec:lehrer_springer_theory} we introduce auxiliary reflection groups when $\mf g$ is a \emph{classical} simple Lie algebra,
in order to make a few more statements in \emph{split} Lehrer--Springer theory (cf.~\S~\ref{sec:reflection_cosets} and~\cite{lehrer_1995_poincare_polynomials_for_unitary_reflection_groups,broue_michel_1997_sur_certains_elements_reguliers_des_groupes_de_weyl_et_les_varietes_de_deligne_lusztig_associees,lehrer_springer_1999_intersection_multiplicities_and_reflection_subquotients_of_unitary_reflection_groups_i,
	lehrer_springer_1999_reflection_subquotients_of_unitary_reflection_groups,
	bonnafe_lehrer_michel_2006_twisted_invariant_theory_for_reflection_groups};
the \emph{nonsplit} setup will be essential to treat the `interior' twists of~\cite{boalch_yamakawa_twisted_wild_character_varieties},
cf.~\S~\ref{sec:interior_twist}).
This bypasses the basic obstruction recorded in Rmkk.~\ref{rmk:no_braid_lift_first} +~\ref{rmk:no_braid_lift}:
if $\phi \sse \Phi$ is a Levi subsystem,
then the reflection arrangement of Howlett's relative reflection group $W''(\phi) \sse W(\phi)$ need \emph{not} coincide with the hyperplane arrangement of the complement~\eqref{eq:pure_untwisted_nongeneric_deformation_space_1_coeff} (cf.~Cor.-Def.~\ref{cor:howlett_twisted_reflection_cosets};
this can fail even when $W''(\phi) = W(\phi)$).
But the local WMCGs \emph{do} involve such restricted root-hyperplane arrangements.

\toctitle{Relative reflection groups and signed partitions}
\section{Relative reflection groups (of classical type) and signed partitions}
\label{sec:lehrer_springer_theory}

\subsection{}

We now define a different---%
rational---%
reflection group $G_\bullet(\phi) \sse \GL_{\mb C}(\mf t_\phi)$,
starting from Weyl groups of classical type $\bullet \in \set{A,B\!C,D}$,
such that:
(i) its reflection arrangement coincides with the set of the hyperplanes $H_\alpha(\phi) \sse \mf t_\phi$;
and (ii) the relative Weyl group $W(\phi)$ is contained in $G_\bullet(\phi)$.
Then Lehrer--Springer theory will circumvent the above basic obstruction,
with the usual caveat in type $D$:
this approach treats the \emph{crystallographic} examples.

(Again,
we postpone the `nonreflection' type-$D$ issues to \S~\ref{sec:D_trees},
where they are treated by fission trees;
and we leave the \emph{exceptional} simple types $\bullet \in \set{E,F,G}$ to another occasion,
cf.~\S~\ref{sec:outlook}.)

\subsection{Type \texorpdfstring{$A$}{A}}
\label{sec:type_A_relative_refl_group}

Let us start from $\mf g = \mf{gl}_m(\mb C)$,
for an integer $m \geq 1$.
(Recall that we abusively regard this as `type $A$',
since what matters is the essential part of the Weyl-group action,
and so this also covers $\mf g = \mf{sl}_m(\mb C)$.)

\subsubsection{}

The root system $\Phi = \Phi_A(m-1)$ consists of the covectors $\alpha_{ij} \in (V_m^+)^{\dual}$,
for $i \neq j \in \ul m^+$,
in the notation of~\S\S~\ref{sec:background_type_A} +~\ref{sec:root_systems_background} (so $\mf t \simeq V_m^+ = \mb C^m$,
and $\ul m^+ = \set{1,\dc,m}$,
etc.)
In this case any root subsystem $\phi \sse \Phi$ is Levi,
and it determines a \emph{set-partition of} $\ul m^+$ (with nonempty parts) in terms of the \emph{equivalence relation} $\sim_{A,\phi}$ defined by:
\begin{equation}
	\label{eq:levi_partition_type_A}
	i \sim_{A,\phi} j \qquad \text{if} \qquad \alpha_{ij} \in \phi \cup \set{0} \sse \mf t^{\dual},
\end{equation}
with the convention that $\alpha_{ii} \ceqq 0$---%
which yields reflexivity.
(Hereafter,
we just abusively write $i \sim_A j$ and omit the Levi $\phi$.)

The fact that the binary relation~\eqref{eq:levi_partition_type_A} is symmetric follows from the identity $\phi = -\phi \sse \mf t^{\dual}$,
while transitivity follows from the closure of $\phi$ under its own reflections,
i.e.,
under the elements $g_\alpha \in W$ with $\alpha \in \phi$.
Indeed,
here $W = W_A(m) \simeq \mf S_m$,
and this closure amounts to the additivity relations
\begin{equation}
	\label{eq:transitivity_type_A}
	\alpha_{ik}
	= \alpha_{ij} + \alpha_{jk}
	= g_{\alpha_{ij}}(\alpha_{jk}) \in \Phi,
	\qquad i \neq j \neq k \neq i \in \ul m^+.
\end{equation}
(In general,
transitivity relies on having \emph{closed} subsets of roots,
and the distinction matters for the other classical types;
cf.~\S\S~\ref{sec:type_BC_relative_refl_group}--\ref{sec:type_D_relative_refl_group}.)

Let $\ol i \ceqq [i]_{\sim_A} \sse \ul m^+$ denote equivalence classes,
i.e.,
the blocks/parts of the partition.
We give a \emph{total order} to the set
\begin{equation}
	J_\phi \ceqq \Set{ \ol i | i \in \ul m^+ },
\end{equation}
of blocks,
ordering them by their minimal element.
If $m_\phi \ceqq \abs{J_\phi} \geq 1$ is the number of parts,
then there is a $\mb C$-linear isomorphism $\mf t_\phi \simeq \mb C^{m_\phi}$,
taking the common value of the coordinates on each block (cf.,
e.g.,
~\cite{doucot_rembado_tamiozzo_2026_local_wild_mapping_class_groups_and_cabled_braids,rembado_2024_a_colourful_classification_of_quasi_root_systems_and_hyperplane_arrangements}).

Now we can make the:

\begin{defi}
	The \emph{type}-$A$ \emph{relative reflection group of} $\phi$ is
	\begin{equation}
		\label{eq:type_A_relative_refl_group}
		G_A(\phi)
		\ceqq W_A(m_\phi) \simeq \mf S_{m_\phi}.
	\end{equation}
\end{defi}

\subsubsection{}

So~\eqref{eq:type_A_relative_refl_group} is the symmetric group of the set $\ul m_\phi^+ = \set{1,\dc,m_\phi}$,
acting (faithfully,
in nonessential fashion) on the vector space $V_{m_\phi}^+ = \mb C^{m_\phi}$ by coordinate permutation.
Compared to $W$,
this yields a \emph{lower-rank} type-$A$ Weyl group,
and indeed:

\begin{prop}
	\label{prop:relative_refl_group_A}

	Let $\phi \sse \Phi_A(m-1)$ be a (Levi) root subsystem.
	Then:
	\begin{enumerate}
		\item
		      one has $W(\phi) \sse G_A(\phi)$;

		\item
		      and the reflection arrangement of $G_A(\phi)$ is the set
		      \begin{equation}
			      \mc A_A(\phi)
			      \ceqq \mc A \bigl( G_A(\phi) \bigr)
			      = \Set{ H_\alpha(\phi) | \alpha \in \Phi \sm \phi } \sse \mb P \bigl( \mf t_\phi^{\dual} \bigr),
		      \end{equation}
		      of `complementary' restricted root-hyperplanes.
	\end{enumerate}
\end{prop}

\begin{proof}[Postponed to~\ref{proof:lem_relative_refl_group_A}]
\end{proof}

\begin{rema}
	The proof~\ref{proof:lem_relative_refl_group_A} also implies that any $g \in N_W(\mf t_\phi)$ can only permute blocks of \emph{equal} size,
	as $g$ restricts to a \emph{bijection} $\ol i \lxra{\simeq} g(\ol i)$.
	And one can prove that $W(\phi) \sse G_A(\phi)$ is precisely the subgroup of size-preserving block-permutations (cf.~\cite{doucot_rembado_2025_topology_of_irregular_isomonodromy_times_on_a_fixed_pointed_curve}).
\end{rema}

\begin{rema}
	\label{rmk:bell_and_stirling}

	The association $\phi \mt \sim_A$ establishes a \emph{bijection} between the set of Levi subsystems of $\Phi_A(m-1)$ and the set of set-partitions of $\ul m^+$.
	Indeed,
	given any equivalence relation $\sim$ on $\ul m^+$,
	one recovers a subset of roots via
	\begin{equation}
		\phi_\sim
		\ceqq \Set{ \alpha_{ij} | i \neq j \text{ and } i \sim j },
	\end{equation}
	and this is a (Levi) root subsystem---%
	in view of~\eqref{eq:transitivity_type_A},
	which governs the nontrivial reflection-relations amongst the type-$A$ roots;
	or one can argue analogously to Rmk.~\ref{rmk:bell_and_stirling_BC}.
	(One can show that this determines a two-sided inverse.)

	Therefore,
	the total number of Levi subsystems equals the (\emph{type}-$A$) $m$-\emph{th Bell number} $\on{Bell}_m = \on{Bell}_A(m)$~\cite{bell_1934_exponential_numbers,bell_1934_exponential_polynomials} (cf.~\cite{stanley_2012_enumerative_combinatorics_volume_1}).
	Moreover,
	by the above,
	the dimension of $\mf t_\phi$ equals the number of blocks of the set-partition corresponding to $\phi$.
	In turn,
	the number of set-partitions of $\ul m^+$ with $k$ blocks ($k \in \set{1,\dc,m}$) is given by the (\emph{type}-$A$) \emph{Stirling numbers of the second kind} $S(m,k) = S_A(m,k)$ (cf.~again~\cite{stanley_2012_enumerative_combinatorics_volume_1};
	so one has $\on{Bell}(m) = \sum_{k = 1}^m S(m,k)$.)
	Therefore,
	the latter counts the number of \emph{rank}-$k$ \emph{strata} for the intersection lattice of $(\mf t,W)$;
	this rank is `complementary' to $m - k = \rk(\phi)$,
	cf.~\S~\ref{sec:levis_and_flats}.
	(A shift by one occurs when working within $\mf{sl}_m(\mb C)$.)

	Now the multilevel case is about \emph{nested set-partitions},
	i.e.,
	each block of a set-partition can split into nonempty/disjoint parts as the fission determined by the irregular type goes on.
	(This involves \emph{higher-order} Bell/Stirling numbers,
	related to counting type-$A$ untwisted \emph{labelled} fission trees.)
\end{rema}

\begin{rema}
	\label{rmk:integer_partitions}

	The admissible-deformation families of irregular \emph{classes} correspond---%
	bijectively---%
	to $W$-orbits of admissible-deformation families of irregular \emph{types}.
	Therefore,
	when $\wh Q = Aw^{-1}$ has minimal pole order and coefficient $A \in \mf t$,
	setting as usual $\phi \ceqq \phi \bigl( \wh Q \bigr) = \phi_A$,
	what matters here is the orbit of the underlying Levi subsystem $\phi$ under the symmetric group $W_A(m) \simeq \mf S_m$.

	In view of Rmk.~\ref{rmk:bell_and_stirling} +~Lem.~\ref{lem:equivariance_levi_partitions},
	the total number of such families is the same as the number of set-partitions of $\ul m ^+$,
	identified up to permuting their blocks:
	what remains is the data of the multiset of block sizes,
	i.e.,
	an \emph{integer partition of} $m$,
	whose number $p(m)$ is much-studied (cf.~\cite{hardy_ramanujan_1918_asymptotic_formulae_in_combinatory_analysis,andrews_1976_the_theory_of_partitions}).
	Equivalently,
	this is the number of strata for the quotient stratification of $\mb A = \mf t \bs W$ induced from the---%
	$W$-equivariant---%
	stratification of $\mf t$;
	it is an example of a `discriminant' stratification (cf.~\cite{bessis_2015_finite_complex_reflection_arrangements_are_k_pi_1,calaque_felder_rembado_wentworth_wild_orbits_and_generalised_singularity_modules_stratifications_and_quantisation}).
	Again the number of \emph{rank}-$k$ \emph{strata} is then given by the number $p(m,k)$,
	of integer partitions of $m$ into $k$ parts ($k \in \set{1,\dc,m}$),
	so that $p(m) = \sum_{k = 1}^m p(m,k)$.

	E.g.,
	the minimal rank $k = 1$ corresponds to $\phi = \Phi$,
	with trivial $W$-orbit,
	whose corresponding---%
	closed---%
	stratum is the $1$-dimensional centre $\mf Z \sse \mf t$ in the pure case,
	and $\mf Z \simeq \mf Z \bs W \sse \mb A$ in the full/nonpure case;
	and indeed $p(m,1) = 1$ for all $m$.
	(Again this differs in the centreless special-linear case.)

	Finally,
	the multilevel case is about $\mf S_m$-orbits of nested set-partitions:
	cf.~again~\cite{boalch_2025_counting_the_fission_trees_and_nonabelian_hodge_graphs_untwisted_case} for the solution,
	counting type-$A$ untwisted \emph{unlabelled} fission trees.
\end{rema}

\subsection{Type \texorpdfstring{$B\!C$}{BC}}
\label{sec:type_BC_relative_refl_group}

Consider now $\mf g = \mf{sp}_{2m}(\mb C)$,
for a rank $m \geq 2$ (to avoid isomorphisms across types),
and refer to \S\S~\ref{sec:background_type_BCD} +~\ref{sec:root_systems_background} for background/notation:
in particular,
there is a standard Cartan subalgebra $\mf t \sse \mf g$ whose underlying vector space is identified with~\eqref{eq:classical_cartan},
and the Weyl group with~\eqref{eq:weyl_type_BC}.

As usual,
what matters is actually the intersection lattice of $\bigl(\mf t,W_C(m)\bigr)$,
and so---%
up to isomorphisms---%
this also covers the `dual/inverse' case of the Lie algebra $\mf g = \mf{so}_{2m+1}(\mb C)$ (and its Weyl group $W_B(m)$;
recall that we identify $W_B(m) = W_C(m) = W_{B\!C}(m) \sse \mf S^\pm_m$ with the group of signed-permutations of the set $\ul m^\pm = \set{\pm 1,\dc,\pm m}$.)

\subsubsection{}

Choose a Levi subsystem $\phi \sse \Phi = \Phi_C(m)$.
The ambient root system consists of suitable linear functionals $\wt \alpha_{ij}$ on $\wt V_m \sse V_m^+ \oplus V^-_m \simeq \mb C^{2m}$,
where $i \neq j \in \ul m^\pm$.\fn{
	Type-$C$ roots are more convenient to write out,
	compared to type $B$,
	since they all involve \emph{pairs} of distinct indices (without need to single out the `coordinates' $\wt \alpha_i$;
	cf.~\eqref{eq:type_b_c_roots}).}~Analogously to \S~\ref{sec:type_A_relative_refl_group},
one can define a binary relation $\sim_{C,\phi}$ on $\ul m^\pm$ via
\begin{equation}
	\label{eq:equivalence_relation_levi_BC}
	i \sim_{C,\phi} j \qquad \text{if} \qquad \wt \alpha_{ij} \in \phi \cup \set{0}.
\end{equation}
(Hereafter,
we just abusively write $i \sim_C j$.)
Again this relation is reflexive---%
by construction---%
and symmetric---%
as $\phi = -\phi$.
Moreover:

\begin{lemm}
	\label{lem:transitivity_BC_relation}
	The relation~\eqref{eq:equivalence_relation_levi_BC} is \emph{transitive}.
\end{lemm}

\begin{proof}[Proof postponed to~\ref{proof:lem_transitivity_BC_relation}]
\end{proof}

\subsubsection{}

Lem.~\ref{lem:transitivity_BC_relation} means that any Levi subsystem $\phi \sse \Phi_C(m)$ yields a set-partition of $\ul m^\pm$.
(Strictly speaking,
the proof~\ref{proof:lem_transitivity_BC_relation} shows that this holds for any \emph{closed} subset of roots,
symmetric about the origin.)

But more is true.
Precisely,
let $J$ be a finite set,
and $i \mt \ol i \cl \ul m^\pm \thra J$ a surjection;
its fibres determine a set-partition of $\ul m^\pm$,
with nonempty parts indexed by $J$,
which we refer to as a $J$-\emph{partition of} $\ul m^\pm$.
(We abusively identify an element $\ol i \in J$ with the corresponding block of elements of $\ul m^\pm$.)

Now:

\begin{enonce}{Lemma-Definition}[cf.~Def.~2.4 of~\cite{bagno_garber_2022_signed_partitions_a_balls_into_urns_approach} and Def.~3.2 of~\cite{sagan_swanson_2024_q_stirling_numbers_in_type_b}]
	\label{lem:signed_partitions}

	Let $i \mt \ol i$ be a $J$-partition of $\ul m^\pm$.
	Then:
	\begin{enumerate}
		\item
		      the underlying set-partition is \emph{signed} if:
		      \begin{enumerate}
			      \item
			            for any $\ol i \in J$,
			            the \emph{opposite block} $- \ol i \ceqq \set{ -j | j \in \ol i } \sse \ul m^\pm$ also appears in the partition;

			      \item
			            and there is \emph{at most one block} satisfying $\ol i \cap (-\ol i) \neq \vn$,
			            i.e.,
			            such that $\ol i = - \ol i$;
			            if it exists,
			            we denote it by $\ol 0 \in J$---%
			            and call it the \emph{zero block};
		      \end{enumerate}

		\item
		      with this terminology,
		      the relation~\eqref{eq:equivalence_relation_levi_BC} determines a \emph{signed set-partition of} $\ul m^\pm$.
	\end{enumerate}
\end{enonce}

\begin{proof}
	For the second statement,
	let $J = J_\phi$ be the index set for the set-partition determined by $\phi$.

	First,
	the existence of opposite blocks means that the relation $i \sim_C j$ should imply that $(-i) \sim_C (-j)$,
	for all $i,j \in \ul m^\pm$.
	If $i = j$ the statement is clear,
	so suppose that $i \neq j$.
	Now one has $\wt \alpha_{i,j} = - \wt \alpha_{-i,-j} \in \mf t^{\dual}$,
	and the conclusion follows---%
	again---%
	from the fact that $\phi = -\phi \sse \mf t^{\dual}$.

	As for the zero block,
	suppose that one has both $\ol i = -\ol i$ and $\ol j = -\ol j$ in $J$.
	Then in particular $i \sim_C (-i)$ and $j \sim_C (-j)$ (since $-\ol i = \ol{-i}$ and $-\ol j = \ol{-j}$,
	by the above).
	If $i = \pm j$ there is nothing to prove,
	so suppose that $i \neq \pm j$.
	Now one has $\wt \alpha_{i,-i}, \wt \alpha_{j,-j} \in \phi$,
	so that
	\begin{equation}
		\wt \alpha_{i,-j}
		= \frac 1 2 ( \wt \alpha_{i,-i} + \wt \alpha_{j,-j} ) \in \spann_{\mb Q}(\phi) \cap \Phi
		= \phi,
	\end{equation}
	since $\phi$ is closed under the $\mb Q$-linear combinations of its elements which remain roots.
	This means that $i \sim_C (-j)$,
	whence $i \sim_C j$ by symmetry/transitivity---%
	and so $\ol i = \ol j$.
\end{proof}

\begin{rema}
	\label{rmk:type_C_component}

	The zero block,
	if present,
	corresponds to the \emph{unique} type-$C$ irreducible component of $\phi$,
	i.e.:
	\begin{equation}
		\label{eq:type_C_component}
		\phi_C
		\ceqq \Set{ \wt\alpha_{ij} | i \neq j \text{ and } i,j \in \ol 0 } \sse \phi.
	\end{equation}
	One can prove that this is a Levi subsystem of $\Phi$,
	and so also of $\phi$,
	with argument analogous to the proof~\ref{proof:lem_inverse_bijection_BC} (but easier):
	consider a vector $\bm \lambda \in \wt V_m$ whose components $\lambda_i \in \mb C$ vanish for $i \in \ol 0$,
	and satisfy
	\begin{equation}
		0 \neq \lambda_i \neq \pm \lambda_j \neq 0,
		\qquad i \neq j \in \ul m^+ \sm \ol 0^+,
		\quad \ol 0^+
		\ceqq \ol 0 \cap \ul m^+ \sse \ul m^\pm.
	\end{equation}
	It follows that $\phi_C = \phi_{\bm\lambda}$ is the annihilator.
	(The associated signed set-partition of $\phi_C$ is \emph{finer} than $\phi$'s,
	having the same zero block together with the remaining singletons $\set{i}$,
	for $i \in \ul m^\pm \sm \ol 0$.)

	E.g.,
	if $\phi$ is $\Delta$-dominant with respect to the standard base of simple roots $\Delta = \Delta_C$,
	then $\phi_C$ is the component containing the long simple root $\theta_C$.
	This matches up with a connected component on one end of the subdiagram of $\phi$,
	which has the `shape' of a Dynkin diagram of type $C_k$,
	where $k \geq 1$ is the rank of $\phi_C$.\fn{
		When $k = 1$,
		such a component is distinguished by root length:
		the exceptional isomorphism $\Phi_A(1) \simeq \Phi_C(1)$ plays no role,
		as we are looking at \emph{subsystems of} $\Phi_C(m)$---%
		rather than abstract root systems.}
\end{rema}

\subsubsection{}
\label{sec:signed_partitions_to_relative_BC_groups}

So the index set $J$ of a signed set-partition of $\ul m^\pm$ comes with a distinguished involution $\ol i \mt -\ol i = \ol{-i}$,
which has possibly one distinguished fixed element $\ol 0 \in J$.

In the case where the signed set-partition comes from a Levi subsystem $\phi$ in type $C$,
denoting again by $J_\phi$ the set of parts/blocks,
we define an integer $m_\phi \geq 0$ via:
\begin{equation}
	2 m_\phi
	\ceqq
	\begin{cases}
		\abs{J_\phi \sm \set{\ol 0}}, & \quad \ol 0 \in J_\phi; \\
		\abs{J_\phi},                 & \quad \text{else}.
	\end{cases}
\end{equation}
Then there is a $\mb C$-linear isomorphism $\mf t_\phi \simeq \mb C^{m_\phi}$,
taking the common value of the coordinates on a chosen `positive' block $\ol i$ in each pair $\set{\ol i,-\ol i} \sse J_\phi$ of \emph{nonzero} blocks,
those on $-\ol i$ being opposite,
while all coordinates vanish on the zero block (cf.~again,
e.g.,~\cite{doucot_rembado_tamiozzo_2026_local_wild_mapping_class_groups_and_cabled_braids,rembado_2024_a_colourful_classification_of_quasi_root_systems_and_hyperplane_arrangements},
where however this is expressed in terms of the positive index set $\ul m^+ \sse \ul m^\pm$,
writing the roots as in~Rmk.~\ref{rmk:roots_on_positive_indices}).

If $\ol 0 \not\in J_\phi$,
we can now identify $J_\phi \lxra{\simeq} \ul m_\phi^\pm$ under a suitable choice of ordering of the `positive' blocks,
in equivariant fashion with respect to the involutions.
(I.e.,
if $\ol i \in J$ maps to $j \in \ul m_\phi^\pm$,
then $-\ol i \mt -j$.)
Conversely,
if $\ol 0 \in J_\phi$ we extend the above identification to a bijection
\begin{equation}
	J_\phi \lxra{\simeq} \set{-m_\phi,\dc,-1,0,1,\dc,m_\phi} = \ul m^\pm_\phi \cup \set{0} \sse \mb Z.
\end{equation}

Now we make the:

\begin{defi}
	The \emph{type}-$B\!C$ \emph{relative reflection group of} $\phi$ is
	\begin{equation}
		\label{eq:type_BC_relative_refl_group}
		G_{B\!C}(\phi)
		\ceqq W_{B\!C}(m_\phi) \simeq \mb Z^{\ts} \wr \mf S_{m_\phi}.
	\end{equation}
\end{defi}

\subsubsection{}

So~\eqref{eq:type_BC_relative_refl_group} is the signed-symmetric group of the set $\ul m_\phi^\pm$,
acting (faithfully,
in essential fashion) on $\wt V_{m_\phi}$ by signed-permutation of the coordinates.
Again,
this yields a \emph{lower-rank} type-$B\!C$ Weyl group,
and one can prove the:

\begin{prop}
	\label{prop:relative_refl_group_BC}

	Let $\phi \sse \Phi_C(m)$ be a Levi subsystem.
	Then:
	\begin{enumerate}
		\item
		      one has $W(\phi) \sse G_{B\!C}(\phi)$;

		\item
		      and the reflection arrangement of $G_{B\!C}(\phi)$ is the set
		      \begin{equation}
			      \mc A_{B\!C}(\phi)
			      \ceqq \mc A \bigl( G_{B\!C}(\phi) \bigr)
			      = \Set{ H_\alpha(\phi) | \alpha \in \Phi \sm \phi } \sse \mb P \bigl( \mf t_\phi^{\dual} \bigr),
		      \end{equation}
		      of `complementary' restricted root-hyperplanes.
	\end{enumerate}
\end{prop}

\begin{proof}[Proof postponed to~\ref{proof:prop_relative_refl_group_BC}]
\end{proof}

\begin{rema}
	\label{rmk:bell_and_stirling_BC}

	The analogues of Rmkk.~\ref{rmk:bell_and_stirling}--\ref{rmk:integer_partitions} hold true.

	Precisely,
	first,
	the association $\phi \mt \sim_C$ matches up Levi subsystems \emph{bijectively},
	and in $W$-equivariant fashion,
	with signed set-partitions of $\ul m^\pm$.
	This follows from Lem.~\ref{lem:equivariance_levi_partitions},
	and from the explicit two-sided inverse
	\begin{equation}
		\label{eq:BC_levi_from_signed_partition}
		\sim \lmt \phi_\sim
		\ceqq \Set{ \wt \alpha_{ij} | i \neq j \text{ and } i \sim j } \sse \Phi_C(m),
	\end{equation}
	where $\sim$ is a \emph{sign-equivariant} equivalence relation on $\ul m^\pm$ with \emph{at most one sign-invariant} equivalence class---%
	the `zero' class $\ol 0 = \ol 0_\sim$.
	The caveat is that one must check that the subset $\phi_\sim \sse \Phi$ is a Levi subsystem (cf.~Lem.~\ref{lem:inverse_bijection_BC}).

	Therefore,
	the number of Levi subsystems $\phi \sse \Phi_C(m)$ which determine a rank-$k$ flat $\mf t_\phi \sse \mf t$ for the intersection lattice of $(\mf t,W)$,
	i.e.,
	of dimension $k \in \set{0,\dc,m}$,
	is equal to the number of signed set-partitions of $\ul m^\pm$ having $k$ pairs of nonzero blocks.
	This is given---%
	by definition---%
	by the \emph{type}-$B$ \emph{Stirling numbers of the second kind} $S_B(m,k)$,
	and the total number of Levis/strata is the \emph{type}-$B$ \emph{Bell number} $\on{Bell}_B(m)$ (which satisfies $\on{Bell}_B(m) = \sum_{k = 0}^m S_B(m,k)$;
	cf.~\cite{dowling_1973_a_class_of_geometric_lattices_based_on_finite_groups,zaslavsky_1981_the_geometry_of_root_systems_and_signed_graphs,dolgachev_lunts_1994_a_character_formula_for_the_representation_of_a_weyl_group_in_the_cohomology_of_the_associated_toric_variety,benoumhani_1996_on_whitney_numbers_of_dowling_lattices,reiner_1997_non_crossing_partitions_for_classical_reflection_groups}).
	As usual,
	the counts are the same when working with the root system of type $B_m$,
	or within the \emph{nonreduced} root system $\Phi_{B\!C}(m) = \Phi_B(m) \cup \Phi_C(m)$---%
	of type $B\!C_m$;
	so these might be equivalently referred to as Stirling/Bell numbers `of type $C$',
	or `of type $B\!C$'.
	Again the multilevel case seems to involve \emph{higher-order} numbers,
	related to counting type-$B\!C$ untwisted labelled fission trees:
	here one considers \emph{nested signed set-partitions},
	i.e.,
	the zero block may break into a `smaller' signed set-partition,
	while the nonzero blocks only break into `type-$A$' set-partitions---%
	which are ordinary `unsigned' ones.
	(This seems to lead to $2$-coloured operads,
	which help studying the corresponding \emph{cabled} braid groups,
	cf.~\cite{doucot_rembado_tamiozzo_2026_local_wild_mapping_class_groups_and_cabled_braids};
	we still plan to spell this out elsewhere.)

	As per the irregular classes,
	the Weyl group $W = W_C(m)$ acts on signed set-partitions of $\ul m^\pm$ by matching up zero blocks and opposite pairs of nonzero blocks---%
	preserving sizes;
	conversely,
	two partitions with the same zero-block size and the same multiset of pair sizes are $W$-conjugate.
	So the invariants of a $W$-orbit correspond to a number $2l = \abs{\ol 0}$,
	for a suitable $l = l_{\ol 0} \in \set{0,\dc,m}$,
	together with a multiset of numbers $2m_{\ol i} = \lvert \ol i \cup (-\ol i) \rvert$,
	where $m_{\ol i} \ceqq \lvert \ol i \rvert = \lvert -\ol i \rvert \in \set{1,\dc,m-l}$ are such that $\sum_{\ol i} m_{\ol i} = m - l$.
	Therefore,
	the total number of full/nonpure admissible-deformation families equals the $m$-\emph{th partial sum} of the partition function $p$,
	i.e.,
	the number
	\begin{equation}
		P(m)
		\ceqq \sum_{i = 0}^m p(i)
		= \sum_{l = 0}^m p(m-l).
	\end{equation}
	(Letting $m$ vary,
	this yields the integer sequence \emph{A000070} of the on-line encyclopedia~\cite{oeis_a000070}.)
	Analogously,
	the rank-$k$ count involves partitioning the integer $m - l$ into $k$ parts ($k \in \set{0,\dc,m}$),
	and so one finds as many Levi-subsystem orbits as
	\begin{equation}
		P(m,k)
		\ceqq \sum_{i = 0}^m p(i,k)
		= \sum_{l = 0}^m p(m-l,k).
	\end{equation}
	E.g.,
	the convention $p(0,0) = 1$ corresponds to counting the unique rank-$0$ stratum,
	when $\phi = \Phi$ and $l = m$;
	it is the unique orbit with $k = 0$,
	and indeed $P(m,0) = 1$ for all $m$.

	Then again the multilevel case is about $W$-orbits of nested signed set-partitions,
	thereby counting type-$B\!C$ untwisted \emph{unlabelled} fission trees;
	this count is \emph{not} considered in this text,
	but again we generalize the $B\!C$ fission trees into the twisted case in~\S~\ref{sec:BC_trees}.
\end{rema}

\begin{lemm}
	\label{lem:inverse_bijection_BC}

	The subset of roots $\phi_\sim \sse \Phi_C(m)$ of~\eqref{eq:BC_levi_from_signed_partition} is a \emph{Levi subsystem}.
\end{lemm}

\begin{proof}[Proof postponed to \S~\ref{proof:lem_inverse_bijection_BC}]
\end{proof}

\subsection{Type \texorpdfstring{$D$}{D}}
\label{sec:type_D_relative_refl_group}

Consider finally $\mf g \ceqq \mf{so}_{2m}(\mb C)$,
in rank $m \geq 4$,
and cf.~again \S\S~\ref{sec:background_type_BCD} +~\ref{sec:root_systems_background} for background/notation.
Once more,
there is a standard Cartan subalgebra $\mf t \sse \mf g$ identified with~\eqref{eq:classical_cartan},
and the Weyl group with~\eqref{eq:weyl_type_D}.

(Recall that we identify $W_D(m)$ with an index-$2$ subgroup of $W_{B\!C}(m)$,
consisting of signed permutations of $\ul m^\pm$ with positive `global' sign.)

\subsubsection{}

Choose a Levi subsystem $\phi \sse \Phi_D(m)$.
The ambient root system coincides with the subset $\Phi_B(m) \cap \Phi_C(m) \sse \Phi_{B\!C}(m)$,
consisting of the linear functionals $\wt\alpha_{ij} \cl \wt V_m \to \mb C$ such that $i \neq \pm j \in \ul m^\pm$:
the main point is that there are \emph{no} short/long roots $\wt \alpha_i$ and $\wt \alpha_{i,-i} = 2\wt \alpha_i$.
So one can still define a binary relation on $\ul m^\pm$ as in~\eqref{eq:levi_partition_type_A} +~\eqref{eq:equivalence_relation_levi_BC},
but the exact relation with signed set-partitions is more complicated than Lem.~\ref{lem:transitivity_BC_relation} +~Lem.-Def.~\ref{lem:signed_partitions}~(2.) (cf.~Rmk.~\ref{rmk:bell_and_stirling_D}).

Nevertheless,
there is a \emph{different} relation which works better.
Precisely,
let us keep writing vectors of $\wt V_m \simeq \mf t$ as $\bm \lambda = (\lambda_1,\dc,\lambda_m,\lambda_{-1},\dc,\lambda_{-m})$,
so that the numbers $\lambda_i \in \mb C$ satisfy $\lambda_i + \lambda_{-i} = 0$ for $i \in \ul m^\pm$.
Then we make the following:

\begin{enonce}{Lemma-Definition}
	\label{lem:good_relation_type_D}

	Consider the binary relation on $\ul m^\pm$ given by
	\begin{equation}
		\label{eq:equivalence_relation_levi_D}
		i \sim_{D,\phi} j \qquad \text{if} \qquad \lambda_i
		= \lambda_j \quad \text{for all} \quad \bm \lambda \in \mf t_\phi \sse \wt V_m.
	\end{equation}
	(Hereafter,
	we just write $i \sim_D j$.)
	Then:
	\begin{enumerate}
		\item
		      this yields a \emph{sign-equivariant} equivalence relation with \emph{at most one} sign-invariant `zero' equivalence class,
		      i.e.,
		      a signed set-partition of $\ul m^\pm$ (cf.~Lem.-Def.~\ref{lem:signed_partitions});

		\item
		      and if there is \emph{no} zero block,
		      then~\eqref{eq:equivalence_relation_levi_D} coincides with the relation defined as in~\eqref{eq:levi_partition_type_A} +~\eqref{eq:equivalence_relation_levi_BC};
		      in general,
		      it coincides with the \emph{transitive closure} thereof.
	\end{enumerate}
\end{enonce}

\begin{proof}
	For the first statement,
	note that reflexivity,
	symmetry,
	and transitivity,
	follow from the fact that~\eqref{eq:equivalence_relation_levi_D} is defined in terms of equalities---%
	of complex numbers.
	Then sign-equivariance is due to the identities $\lambda_i = -\lambda_{-i}$,
	for all $\bm\lambda \in \wt V_m$ and $i \in \ul m^\pm$,
	and finally the zero block---%
	if present---%
	is uniquely determined as
	\begin{equation}
		\label{eq:type_D_zero_block}
		\ol 0
		= \Set{ i \in \ul m^\pm | \lambda_i = 0 \text{ for all } \bm \lambda \in \mf t_\phi }.
	\end{equation}
	Indeed,
	this satisfies $- \ol 0 = \ol 0$,
	and if $i \sim_D (-i)$ then the identities $\lambda_i = \lambda_{-i} = -\lambda_i$ yield $\lambda_i = 0$---%
	for any $i \in \ul m^\pm$.

	For the second statement,
	consider the binary relation on $\ul m^\pm$ defined as in~\eqref{eq:levi_partition_type_A} +~\eqref{eq:equivalence_relation_levi_BC},
	i.e.:
	\begin{equation}
		\label{eq:weak_equivalence_relation_levi_D}
		i \approx_D j \qquad \text{if} \qquad \wt \alpha_{i,j} \in \phi \cup \set{0}.
	\end{equation}
	This is still reflexive and symmetric,
	but it need \emph{not} be transitive.
	Namely,
	suppose that there exist indices $i \neq \pm j \in \ul m^\pm$ such that $\set{\wt \alpha_{i,j},\wt \alpha_{i,-j}} \sse \phi$.
	Then one has $\wt \alpha_{j,-i} = \wt \alpha_{i,-j} \in \phi$,
	but
	\begin{equation}
		2 \wt \alpha_i
		= \wt \alpha_{i,-i}
		= \wt \alpha_{i,j} + \wt \alpha_{j,-i} \not\in \phi,
	\end{equation}
	simply because the latter is a long $C$-root which does \emph{not} lie in the ambient type-$D$ root system.
	This means that one has $i \approx_D j$,
	and $j \approx_D (-i)$,
	and $i \not\approx_D (-i)$.

	So consider the \emph{transitive closure} $\approx_D^*$ \emph{of} $\approx_D$,
	i.e.,
	for any pair $i,j \in \ul m^\pm$ one has $i \approx^*_D j$ if---%
	and only if---%
	there exists a chain of relations
	\begin{equation}
		\label{eq:equivalence_chain}
		i
		\eqqc
		i_0 \approx_D i_1 \approx_D \dm \approx_D i_{k-1} \approx_D i_k
		\ceqq j,
	\end{equation}
	of (finite) length $k \geq 1$,
	for suitable intermediate indices $i_1,\dc,i_{k-1} \in \ul m^\pm$.

	Let us show that $\approx^*_D$ coincides with~\eqref{eq:equivalence_relation_levi_D}:
	\begin{enumerate}
		\item
		      Choose $i,j \in \ul m^\pm$ such that $i \approx_D^* j$.
		      Up to deleting identical consecutive indices in~\eqref{eq:equivalence_chain},
		      assume that $i_l \neq i_{l+1}$ for $l \in \set{0,\dc,k-1}$.
		      Now note that $i_l = - i_{l+1}$ is not possible,
		      because---%
		      as above---%
		      one has $i' \not\approx_D (-i')$ for all $i' \in \ul m^\pm$.
		      So $i_l \neq \pm i_{l+1}$,
		      and then one has a subset of roots
		      \begin{equation}
			      \phi(i_0,\dc,i_k) \ceqq \Set{ \wt \alpha_{i_0 i_1},\dc,\wt \alpha_{i_{k-1} i_k} } \sse \phi.
		      \end{equation}

		      It follows that
		      \begin{align}
			      \mf t_\phi \sse \ker \bigl( \phi( & i_0,\dc,i_k) \bigr) \sse \ker \bigl( \wt \alpha_{i_0 i_1} + \dm + \wt \alpha_{i_{k-1}i_k} \bigr) \\
			                                        & = \ker \bigl( \wt\alpha_{ij} \bigr)
			      = \wt H_{ij}
			      = \Set{ \bm \lambda \in \wt V_m | \lambda_i = \lambda_j },
		      \end{align}
		      in the notation of~\eqref{eq:BC_root_hyperplanes},
		      whence $i \sim_D j$.

		\item
		      Conversely,
		      choose $i,j \in \ul m^\pm$ such that $i \sim_D j$.
		      If $i = j$ the conclusion is clear,
		      so assume that $i \neq j$.
		      The hypothesis means precisely that
		      \begin{equation}
			      \mf t_\phi \sse \wt H_{ij} = \ker\bigl( \wt\alpha_{ij} \bigr),
		      \end{equation}
		      and again the caveat is that the latter need \emph{not} be a type-$D$ root.
		      Now $\wt \alpha_{ij} \in \Phi_D(m)$ if and only if $i \neq \pm j$,
		      whence $\set{ \pm \wt \alpha_{ij} } \sse \phi$---%
		      since $\phi$ and $\set{\pm \wt \alpha_{ij} }$ are both Levi subsystems (as per the antitone Galois correspondence of flats/Levis,
		      cf.~\S~\ref{sec:levis_and_flats}).
		      This yields $i \approx_D j$,
		      and a fortiori $i \approx^*_D j$.

		      Otherwise,
		      finally,
		      choose an index $i \in \ul m^\pm$ such that $i \sim_D (-i)$.
		      (This is the tricky case.)
		      Then one has $i \in \ol 0$,
		      and $\mf t_\phi$ lies within the coordinate hyperplane $\wt H_i = \ker \bigl( \wt \alpha_{i,-i} \bigr) \sse \wt V_m$,
		      so that $\wt \alpha_{i,-i} \in \spann_{\mb C}(\phi) \cap \Phi_C(m)$.
		      By contradiction,
		      suppose that $i \not\approx^*_D (-i)$,
		      and observe that the binary relation $\approx_D$ is sign-equivariant (by the same argument as in Lem.-Def.~\ref{lem:signed_partitions}~(2.)),
		      so that this also holds for its transitive closure.

		      Let now $\ol j^* \ceqq [j]_{\approx_D^*} \sse \ul m^\pm$ denote the equivalence classes of $\approx^*_D$ ($j \in \ul m^\pm$),
		      so that as usual one has $- \ol i^* = \ol{-i}^*$:
		      the current hypothesis is that $\ol i^* \neq - \ol i^*$,
		      whence $\ol i^* \cap (- \ol i^*) = \vn$.
		      So the following vector $\bm\lambda = (\lambda_1,\dc,\lambda_m,\lambda_{-1},\dc,\lambda_{-m}) \in V^+_m \ops V^-_m$ is well-defined:
		      \begin{equation}
			      \lambda_j
			      \ceqq \begin{cases}
				      1,  & \quad j \in \ol i^*,                                                \\
				      -1, & \quad j \in - \ol i^*,                                              \\
				      0,  & \quad j \in \ul m^\pm \sm \bigl( \, \ol i^* \cup (-\ol i^*) \bigr).
			      \end{cases}
		      \end{equation}
		      Now observe that $\bm\lambda \in \wt V_m \sse V^+_m \ops V^-_m$,
		      since $- \ol i^* = \set{ - j | j \in \ol i^* }$.
		      Moreover,
		      by construction,
		      the vector $\bm \lambda$ has constant coordinates on each $\approx^*_D$-equivalence class,
		      i.e.,
		      one has $\lambda_j = \lambda_k$ whenever $j \approx^*_D k$.
		      So choose in particular any pair $j \neq \pm k \in \ul m^\pm$ such that $j \approx_D k$,
		      which means that $\wt \alpha_{jk} \in \phi$;
		      then one has $\bm \lambda \in \ker( \wt\alpha_{jk})$,
		      and overall $\bm \lambda \in \ker(\phi) = \mf t_\phi$ by letting $j$ and $k$ vary.
		      Finally,
		      the fact that $\lambda_i \neq \lambda_{-i}$ implies that $i \not\sim_D (-i)$,
		      which is \emph{absurd}.
	\end{enumerate}
	
    To conclude the proof,
	it is now enough to show that the relations $\approx_D$ and $\approx^*_D$ coincide whenever $\ol 0 = \vn$ (cf.~\eqref{eq:type_D_zero_block};
	by what was just established,
	this is also the zero block of $\approx_D^*$).
	So we prove that~\eqref{eq:weak_equivalence_relation_levi_D} is already \emph{transitive} in this case.

	Choose indices $i,j,k \in \ul m^\pm$ such that $i \approx_D j$ and $j \approx_D k$.
	If two of them coincide the conclusion is clear,
	so assume that $i \neq j \neq k \neq i$.
	Then one has necessarily $i \neq \pm j$,
	and $j \neq \pm k$,
	and $\set{ \wt \alpha_{ij}, \wt\alpha_{jk}} \sse \phi$.
	Moreover,
	one \emph{cannot} have $k = -i$,
	since that would imply as usual that
	\begin{equation}
		\mf t_\phi \sse \ker \bigl( \wt \alpha_{ij} + \wt \alpha_{j,-i} \bigr)
		= \ker\bigl( \wt \alpha_{i,-i} \bigr)
		= \wt H_i,
	\end{equation}
	whence $i \in \ol 0 \neq \vn$.
	So in conclusion one has $\pm i \neq \pm j \neq \pm k \neq \pm i$,\fn{
		As usual $\pm i \neq \pm j$ means $\abs i \neq \abs j \in \ul m^+$,
		involving absolute values;
		etc.}~and here one can appeal to the closure of $\phi$ \emph{inside of} $\Phi = \Phi_D(m)$:
	\begin{equation}
		\wt \alpha_{ik}
		= \wt \alpha_{ij} + \wt \alpha_{jk} \in (\phi + \phi) \cap \Phi \sse \phi,
	\end{equation}
	which yields $i \approx_D k$---%
	as in the proof~\ref{proof:lem_transitivity_BC_relation}.
\end{proof}

\subsubsection{}

The relation~\eqref{eq:equivalence_relation_levi_D} also makes sense in types $A$ and $C$:
there it \emph{always} matches up with the relations that were introduced above via root membership---%
rather than coordinate identities.
(In type $B$,
the zero block is detected by $\wt \alpha_i$,
rather then $\wt \alpha_{i,-i} = 2 \wt \alpha_i$.)

Then in addition one also establishes that:

\begin{lemm}
	\label{lem:equivariance_levi_partitions}

	The correspondence $\phi \mt \sim_\bullet$,
	where $\phi$ is a Levi subsystem of $\Phi_\bullet$ of \emph{any} classical type $\bullet \in \set{A,B,C,D}$,
	is $W_\bullet$-\emph{equivariant}---%
	in any rank.
\end{lemm}

\begin{proof}
	This follows from~\eqref{eq:permutation_action},
	and from the fact that $g(\mf t_\phi) = \mf t_{\phi'}$ for all Levi subsystems $\phi$ and Weyl-group elements $g \in W$,
	where $\phi' \ceqq g(\phi)$.
\end{proof}
\subsubsection{}

In particular,
Lem.~\ref{lem:equivariance_levi_partitions} implies that a Levi subsystem $\phi \sse \Phi_D(m)$ still determines a signed set-partition of $\ul m^\pm$,
in $W_D(m)$-equivariant fashion (cf.~again Rmk.~\ref{rmk:bell_and_stirling_D}).

Now one can define a relative reflection group,
distinguishing cases:
\begin{enumerate}
	\item
	      one has $\phi = \vn$,
	      i.e.,
	      the relation $\sim_D$ is the identity---%
	      and the corresponding signed set-partition is \emph{discrete};

	\item
	      or one has $\phi \neq \vn$,
	      and then there are two subcases:
	      \begin{enumerate}
		      \item
		            the corresponding signed set-partition \emph{possesses} the zero block;

		      \item
		            or there is \emph{no} zero block.
	      \end{enumerate}
\end{enumerate}

The first case is our quasigeneric setting,
and it is quickly settled:

\begin{defi}
	The \emph{type}-$D$ \emph{relative reflection group of} $\vn$ is the `absolute' Weyl group
	\begin{equation}
		G_D(\vn)
		\ceqq W_D(m) \lhra \mb Z^{\ts} \wr \mf S_m.
	\end{equation}
\end{defi}

\subsubsection{}
\label{sec:relative_refl_group_D_quasigeneric}

Here the exact analogues of Propp.~\ref{prop:relative_refl_group_A} +~\ref{prop:relative_refl_group_BC} hold,
as $W(\phi) = W_D(m)$ and $H_\alpha(\phi) = H_\alpha$ for all $\alpha \in \Phi \sm \phi = \Phi$.

But of course,
the genuinely relative case is that where $\phi \neq \vn$.
Then the main distinction---%
for our purposes---%
is clarified by the following statement,
which relates the zero block of~\eqref{eq:equivalence_relation_levi_D} with the decomposition of $\phi$ into Levi subsystems of `inner'  type $A$ or $D$:

\begin{enonce}{Lemma-Definition}
	\label{lem:type_D_part}

	Given \emph{any} Levi subsystem $\phi \sse \Phi_D(m)$,
	its \emph{type}-$D$ \emph{part}\fn{
		We call this the type-$D$ `part',
		rather than the type-$D$ `component',
		as it need \emph{not} be irreducible:
		this fails precisely when it has rank $2$.}~is the subset of roots
	\begin{equation}
		\label{eq:type_D_part}
		\phi_D
		\ceqq \Set{ \wt\alpha_{ij} | i \neq \pm j \text{ and } i,j \in \ol 0 } \sse \phi,
	\end{equation}
	where $\ol 0$ is the zero block for the signed set-partition associated with the relation~\eqref{eq:equivalence_relation_levi_D}.
	Then:
	\begin{enumerate}
		\item
		      the subset~\eqref{eq:type_D_part} is a \emph{Levi subsystem} of $\Phi$---%
		      whence also of $\phi$;

		\item
		      one has $\phi_D = \vn$---%
		      in $\Phi$---%
		      if and only if $\ol 0 = \vn$---%
		      in $\ul m^\pm$;

		\item
		      if $\phi_D \neq \vn$,
		      then it has rank \emph{at least} $2$;

		\item
		      and letting $\ol 0^+ \ceqq \ol 0 \cap \ul m^+ \sse \ul m^\pm$ be the `positive' part of the zero block,
		      one has
		      \begin{equation}
			      \label{eq:type_D_subspace}
			      \spann_{\mb C}(\phi_D^{\dual})
			      = \bops_{i \in \ol 0^+} \mb C \cdot \wt e_{i,-i} \sse \mf t,
			      \qquad \wt e_{i,-i}
			      \ceqq e_i - e_{-i},
		      \end{equation}
		      using the canonical bases $(e_i)_i$ and $(e_{-i})_i$ of $V^+_m$ and $V^-_m$---%
		      respectively,
		      cf.~again \S~\ref{sec:background_type_BCD}.
	\end{enumerate}
\end{enonce}

\begin{proof}[Proof postponed to \ref{proof:lem_type_D_part}]
\end{proof}

\subsubsection{}

Therefore,
the Levi subsystem $\phi_D$ is the \emph{full} type-$D$ root system on the subspace~\eqref{eq:type_D_subspace}:
this is a property of the embedding of $\phi$ into $\Phi_D(m)$,
and so the exceptional isomorphisms $\Phi_D(2) \simeq \Phi_A(1) \ops \Phi_A(1)$ and $\Phi_D(3) \simeq \Phi_A(3)$ play no role.
In particular,
Lem.~\ref{lem:equivariance_levi_partitions} implies that
\begin{equation}
	\label{eq:type_D_preservation}
	g(\phi_D) = \phi'_D,
	\qquad \phi'
	\ceqq g(\phi),
\end{equation}
for any Levi subsystem $\phi \sse \Phi$ and any signed permutation $g \in W_D(m)$---%
looking at the zero blocks of the corresponding set partitions.

\begin{rema}
	After putting $\phi$ in dominant form under the $W_D(m)$-action,
	one can show that the condition that $\phi_D \neq \vn$ coincides with the subcase of~Cor.-Def.~\ref{cor:howlett_twisted_reflection_cosets}~(2.D.i)---%
	in rank $m \geq 5$.
	Indeed,
	if $\Delta_\phi$ contains both fork nodes $\theta_{m-1}$ and $\theta_D$ for the standard base $\Delta = \Delta_D = \set{\theta_1,\dc,\theta_{m-1},\theta_D}$,
	then one has $\lambda_{m-1} = \lambda_m = 0$ for all $\bm \lambda \in \mf t_\phi$,
	whence $\set{ \pm (m-1),\pm m } \sse \ol 0$.
	Conversely,
	if $\theta_D \not\in \Delta_\phi$ then the vector $\bm \lambda \in \wt V_m$ with coordinates $\lambda_i \ceqq 1$ for $i \in \ul m^+$ lies in $\mf t_\phi$---%
	and so $\ol 0 = \vn$;
	if instead $\theta_{m-1} \not\in \Delta_\phi$,
	one can swap the sign of $\lambda_m$.

	When $m = 4$,
	however,
	the condition that $\phi_D \neq \vn$ is \emph{stronger} than the subcase of Cor.-Def.~\ref{cor:howlett_twisted_reflection_cosets}~(2.D.ii) (taking,
	e.g.,
	$\Delta_\phi = \set{\theta_1,\theta_3}$):
	the difference is that the former is invariant under the action of all signed permutations,
	i.e.,
	under the action of $W_{B\!C}(m) \sse \Aut \bigl( \Phi_D(m) \bigr)$,\fn{
		Extending the argument of the proof of Lem.~\ref{lem:equivariance_levi_partitions},
		which is relevant when adding on the interior twists (cf.~the proof of Lem.~\ref{lem:twisted_normalization_relative_refl_group}).}~and the latter inclusion is \emph{strict} precisely when $m = 4$.

	(Using fission trees,
	in \S~\ref{sec:D_trees},
	we will treat \emph{any} rank.)
\end{rema}

\subsubsection{}

What matters here is that we have a natural \emph{sufficient} condition which leads to crystallographic restricted hyperplane arrangements,
as per the following:

\begin{enonce}{Proposition-Definition}
	\label{prop:relative_refl_group_D}

	Let $\phi \sse \Phi_D(m)$ be a Levi subsystem such that $\phi_D \neq \vn$ (cf.~Lem.-Def.~\ref{lem:type_D_part}).
	Then:
	\begin{enumerate}
		\item
		      the \emph{type}-$D$ \emph{relative reflection group of} $\phi$ is
		      \begin{equation}
			      G_D(\phi)
			      \ceqq W_{B\!C}(m_\phi),
		      \end{equation}
		      where $m_\phi \geq 0$ is the number of nonzero opposite pairs of the set-partition of $\ul m^\pm$ associated with $\phi$ (cf.~\S~\ref{sec:signed_partitions_to_relative_BC_groups});

		\item
		      one has $W(\phi) \sse G_D(\phi)$;

		\item
		      and the reflection arrangement of $G_D(\phi)$ is the set
		      \begin{equation}
			      \mc A_D(\phi)
			      \ceqq \mc A \bigl( G_D(\phi) \bigr)
			      = \Set{ H_\alpha(\phi) | \alpha \in \Phi \sm \phi } \sse \mb P \bigl( \mf t_\phi^{\dual} \bigr),
		      \end{equation}
		      of `complementary' restricted root-hyperplanes.
	\end{enumerate}
\end{enonce}

\begin{proof}[Proof postponed to~\ref{proof:prop_relative_refl_group_D}]
\end{proof}

\subsubsection{}

It remains to consider the case where $\phi \neq \vn$ and $\phi_D = \vn$:
here in principle Howlett's obstruction need \emph{not} vanish.
In any event,
one now gets a signed set-partition of $\ul m^\pm$ with \emph{no} zero block,
and one can still identify the set of parts $J = J_\phi$ with $\ul m_\phi^\pm$---%
where $m_\phi \ceqq \dim_{\mb C}(\mf t_\phi)$,
as usual.
(In view of Lem.-Def.~\ref{lem:good_relation_type_D},
here $\sim_D$ coincides with the root-membership relation~\eqref{eq:weak_equivalence_relation_levi_D}.)

Then one proves the following:

\begin{prop}
	\label{prop:exotic_arrangement}

	In the usual identifications $\mf t_\phi \simeq \wt V_{m_\phi} \simeq \mb C^{m_\phi}$:
	\begin{enumerate}
		\item
		      the restricted hyperplane arrangement of~\eqref{eq:pure_untwisted_nongeneric_deformation_space_1_coeff} coincides with the set
		      \begin{equation}
			      \label{eq:exotic_arrangement_hyperplanes}
			      \Set{ \wt H_{\ol i \, \ol j} | \ol i \neq \pm \ol j \in J } \cup \Set{ \wt H_{\ol i} | \ol i \in J \text{ and } \lvert \ol i \rvert \geq 2 } \sse \mb P \bigl( \mf t_\phi^{\dual} \bigr);
		      \end{equation}

		\item
		      the corresponding hyperplane complement is $\mc M^{\musDoubleSharp}(k,k') \sse \mb C^{k + k'}$ (cf.~\eqref{eq:exotic_arrangement}),
		      where $2k \geq 2$ (resp.,
		      $2k' \geq 0$) is the number of \emph{nontrivial} blocks $\ol i \in J$,
		      i.e.,
		      such that $\lvert \ol i \rvert \geq 2$ (resp.,
		      such that $\ol i$ is \emph{trivial},
		      i.e.,
		      a singleton);

		\item
		      the reflection group generated by~\eqref{eq:exotic_arrangement_hyperplanes},
		      with respect to the \emph{standard} Hermitian form on $\mb C^{m_\phi}$ (i.e.,
		      with respect to the coroots of $\Phi_{B\!C}(m_\phi)$ in $\wt V_{m_\phi}$,
		      cf.~Rmk.~\ref{rmk:different_coroots}),
		      equals the whole of $W_{B\!C}(m_\phi)$;

		\item
		      and one has $W(\phi) \sse W_{B\!C}(m_\phi)$.
	\end{enumerate}
\end{prop}

\begin{proof}
	For the first statement,
	the proof~\ref{proof:prop_relative_refl_group_D} translates verbatim to prove that all diagonals/antidiagonals arise from---%
	the kernels of---%
	the restrictions of the `complementary' roots of $\phi$.
	As per the coordinates,
	there are two cases.
	\begin{enumerate}
		\item
		      If $\ol i \in J \simeq \ul m^\pm_\phi$ is nontrivial,
		      choose indices $k \neq l \in \ol i$.
		      Since $\ol i$ is \emph{not} the zero block,
		      one must have $k \neq -l \in -\ol i$ (to avoid $k \sim_D (-k)$),
		      and so the roots $\wt \alpha \ceqq \wt \alpha_{k,\pm l}$ lie in $\Phi$:
		      but now $\wt \alpha_{kl}$ lies in $\phi$ (by choice of indices),
		      and so $\wt \alpha \ceqq \wt \alpha_{k,-l}$ \emph{cannot} lie therein;
		      for otherwise one would have $k,l \in \ol 0$ (contradicting $\ol 0 = \vn$;
		      cf.~Lemm.-Deff.~\ref{lem:good_relation_type_D} +~\ref{lem:type_D_part}).
		      Now $\wt \alpha_\phi = \wt \alpha_{\ol i,-\ol i} = 2 \wt \alpha_{\ol i}$,
		      and so the corresponding coordinate hyperplane does arise.

		\item
		      Otherwise $\ol i = \set{i}$,
		      and the above \emph{cannot} be done.
		      Moreover,
		      one cannot appeal to the zero block as in the proof~\ref{proof:prop_relative_refl_group_D};
		      and actually no suitable root exists.

		      Indeed,
		      no type-$D$ root $\wt \alpha_{kl}$ ($k \neq \pm l \in \ul m^\pm$) can ever restrict to a multiple of $\wt \alpha_{\ol i}$ if $k \neq \pm i \neq l$.
		      Up to swapping $k$ and $l$ (using $\wt \alpha_{kl} = - \wt \alpha_{lk}$),
		      suppose first that $k = i$,
		      whence $l \neq \pm i$.
		      Since the block of $i$ is trivial,
		      this implies that $i \not\sim_D l$,
		      and so $\wt \alpha \ceqq \wt \alpha_{il} \in \Phi \sm \phi$ has nonzero restriction onto $\mf t_\phi$;
		      however,
		      if $\ol j \in J$ is the block of $l$,
		      then $\wt \alpha_\phi = \wt \alpha_{\ol i \, \ol j} = \wt \alpha_{\ol i} - \wt \alpha_{\ol j}$ is linearly independent from $\wt \alpha_{\ol i}$.
		      Analogously,
		      if $k = -i$ then $\wt \alpha_{-i,l} = -\wt \alpha_{i,-l} \in \Phi \sm \phi$ restricts to $-\wt \alpha_{\ol i} - \wt \alpha_{\ol j} \not\in \mb C \cdot \wt \alpha_{\ol i}$.
	\end{enumerate}

	The second statement now follows,
	noting that $m_\phi = k + k'$---%
	by construction.
	(We are assuming $\phi \neq \vn$,
	which yields $k > 0$;
	else the relation $\sim_D$ is the identity.)

	For the third statement,
	observe that the generated reflection group contains the whole of $W_D(m_\phi)$,
	since all the type-$D$ reflection hyperplanes arise.
	In addition,
	there is \emph{at least one} reflection about a coordinate hyperplane,
	i.e.,
	a negative transposition:
	then one can generate any signed permutation of $\ul m^\pm_\phi$.

	The last statement follows again from the fact that any element $g \in N_W(\mf t_\phi)$ restricts to a signed permutation of $\ul m^\pm_\phi$ (cf.~again the proof~\ref{proof:prop_relative_refl_group_D}).
\end{proof}

\begin{rema}
	\label{rmk:different_coroots}

	Beware that the $B\!C$-coroots in $\wt V_{m_\phi} \simeq \mf t_\phi$ are \emph{not} the restricted coroots $\alpha^{\dual}_\phi$ of~Lem.-Def.~\ref{lem:inner_product},
	which involve the orthogonal projection with respect to a $W$-invariant Hermitian form on $\mf t$.
	Moreover,
	if $\alpha \in \Phi \sm \phi$ is such that $\alpha_\phi = \wt \alpha_{\ol i\,\ol j} \in \Phi_{B\!C}(m_\phi)$,
	with $\ol i \neq \pm \ol j \in J$,
	then the corresponding diagonal/antidiagonal reflection $g_{\wt \alpha_{\ol i \, \ol j}} \in W_{B\!C}(m_\phi)$ is \emph{different} from the restricted reflection $g_\alpha(\phi)$ of~\eqref{eq:reduced_reflection} whenever $\lvert \ol i \rvert \neq \lvert \ol j \rvert$.
	(But there is no difference about the coordinate reflections.)

	In other words,
	here we change the restricted $W$-invariant Hermitian form on $\mf t_\phi$ to be `unweighted',
	and this yields a \emph{finite} reflection group;
	whereas the group generated by the $g_\alpha(\phi)$---%
	about those same hyperplanes---%
	need \emph{not} be finite.
	So we could not use the latter in the context of classical Lehrer--Springer theory (cf.~\S~\ref{sec:some_lehrer_springer}).

	(An analogous remark applies retroactively in type $\bullet \in \set{A,B,C}$.)
\end{rema}

\subsubsection{}

Therefore,
in the notation of Prop.~\ref{prop:exotic_arrangement},
if $k' = 0$ one finds the \emph{full} hyperplane arrangement of type $B\!C$ in the flat $\mf t_\phi$.
In this case,
it still makes sense to define the \emph{type}-$D$ \emph{relative reflection group of} $\phi$ as $G_D(\phi) \ceqq W_{B\!C}(m_\phi)$,
because the restricted arrangement of hyperplanes $H_\alpha(\phi) \sse \mf t_\phi$ ($\alpha \in \Phi \sm \phi$) remains a reflection arrangement.
One still finds $W(\phi) \sse G_D(\phi)$.

Otherwise,
the arrangement~\eqref{eq:exotic_arrangement_hyperplanes} sits \emph{strictly} between type $D_{m_\phi}$ and type $B\!C_{m_\phi}$,
and one proves the following:

\begin{prop}
	\label{prop:nonreflection_arrangement}

	Choose integers $k,k' > 0$.
	Then:
	\begin{enumerate}
		\item
		      if $k + k' = 2$,
		      i.e.,
		      if $k = k ' = 1$,
		      then the space~\eqref{eq:exotic_arrangement} is $\mb C$-linearly isomorphic to the (crystallographic) hyperplane complement of type $A_2$;

		\item
		      otherwise,
		      if $k + k' \geq 3$,
		      then~\eqref{eq:exotic_arrangement} is \emph{not} the complement of a complex-reflection arrangement---%
		      and so in particular it is non-Coxeter/noncrystallographic;

		\item
		      nonetheless,
		      the complement~\eqref{eq:exotic_arrangement} is always a $K(\pi,1)$ space.
	\end{enumerate}
\end{prop}

\begin{proof}[Postponed to~\ref{proof:prop_nonreflection_arrangement}]
\end{proof}

\begin{rema}
	\label{rmk:bell_and_stirling_D}

	One can partially extend the correspondences of Rmk.~\ref{rmk:bell_and_stirling_BC}.

	Namely,
	the association $\phi \mt \sim_D$,
	from Levi subsystems $\phi \sse \Phi_D(m)$ to signed set-partitions of $\ul m^\pm$ remains \emph{injective} and $W$-equivariant (cf.~Lem.-Def.~\ref{lem:good_relation_type_D} +~Lem.~\ref{lem:equivariance_levi_partitions}).
	The caveat is that \emph{not} all signed set-partitions arise from Levi subsystems.
	Indeed,
	the proof~\ref{proof:lem_type_D_part} shows that the zero block of any such partition must contain at least four elements---%
	if present.
	Conversely,
	one can prove that this is the unique constraint.
	Precisely,
	a Levi subsystem is reconstructed from a suitable sign-equivariant equivalence relation $\sim$ via
	\begin{equation}
		\phi_{\sim}
		\ceqq \Set{ \wt \alpha_{kl} | k \neq \pm l \text{ and } k \sim l } \sse \Phi_D(m).
	\end{equation}
	(Cf.~\eqref{eq:BC_levi_from_signed_partition},
	and the proof~\ref{proof:lem_inverse_bijection_BC},
	which suitably translates.)

	This matches up with the notion of \emph{type}-$D$ \emph{signed set-partitions},
	as per~\cite{reiner_1997_non_crossing_partitions_for_classical_reflection_groups,bagno_biagioli_garber_2019_some_identities_involving_second_kind_stirling_numbers_of_types_b_and_d} (cf.~also~\cite{bjoerner_sagan_1996_subspace_arrangements_of_type_b_n_and_d_n,blass_sagan_1998_characteristic_and_ehrhart_polynomials,suter_2000_two_analogues_of_a_classical_sequence,bjoerner_wachs_2004_geometrically_constructed_bases_for_homology_of_partition_lattices_of_types_a_b_and_d,chen_wang_2010_minimally_intersecting_set_partitions_of_type_b,chen_wang_2011_singletons_and_adjacencies_of_set_partitions_of_type_b} for more past work about signed set partitions of type $B$ and $D$).
	So one introduces the \emph{type}-$D$ $m$-\emph{th Bell number} $\on{Bell}_D(m)$,
	which in our setting counts the Levis/strata of the intersection lattice of $\bigl(\mf t,W_D(m)\bigr)$;
	as well as the \emph{type}-$D$ \emph{Stirling numbers of the second kind} $S_D(m,k)$,
	counting the rank-$k$ objects ($k \in \set{0,\dc,m}$).
	The passage to the full/nonpure case involves orbits of type-$D$ signed set-partitions under $W_D(m)$,
	which---%
	compared to $W_{B\!C}(m)$---%
	may split into two when the usual block-sizes are prescribed.
	(This relates to Carter's description of conjugacy classes of Weyl groups;
	cf.~\S~\ref{sec:background_type_BCD}.)

	As usual,
	the pure multilevel case involves type-$D$ signed set-partitions,
	where zero blocks can split into smaller type-$D$ partitions;
	while the full/nonpure case is about $W$-orbits thereof.
	We will \emph{not} study the corresponding count in this text.
\end{rema}

\subsection{Some more (split) Lehrer--Springer theory}
\label{sec:some_lehrer_springer}

Fix again a classical type $\bullet \in \set{A,B,C,D}$,
as well as a Levi subsystem $\phi$ of the corresponding root system $\Phi_\bullet$---%
of any rank.

\subsubsection{}

Consider also the relative reflection group $G_\bullet(\phi)$ which was defined in \S\S~\ref{sec:type_A_relative_refl_group}--\ref{sec:type_D_relative_refl_group};
in type $\bullet = D$,
this means that we only consider the cases where:
(i) $\phi = \vn$;
or (ii) $\phi_D \neq \vn$;
or (iii) $\phi \neq \vn$,
and $\phi_D = \vn$,
and the signed set-partition associated to $\phi$ has no trivial---%
nonzero---%
classes.

Then we state the following standard:

\begin{lemm}
	\label{lem:restricted_reflection_group}

	Choose $g \in N_W(\mf t_\phi)$,
	where $W = W_\bullet$ is the classical type-$\bullet$ Weyl group.
	Suppose that
	\begin{equation}
		\label{eq:maximality_in_coset}
		\dim_{\mb C} \bigl( \mf t_\phi( g_\phi,\zeta_r ) \bigr) \geq \dim_{\mb C} \bigl( \mf t_\phi( g',\zeta_r ) \bigr),
		\qquad g' \in G_\bullet(\phi).
	\end{equation}
	Write  also
	\begin{equation}
		\label{eq:outer_normalizer}
		N
		\ceqq \Set{ g' \in G_\bullet(\phi) | g' \bigl( \mf t_\phi( g_\phi,\zeta_r) \bigr) \sse \mf t_\phi (g_\phi,\zeta_r) },
	\end{equation}
	and
	\begin{equation}
		\label{eq:outer_centralizer_1}
		Z
		\ceqq \Set{ g' \in N | g'(A) = A \text{ for all } A \in \mf t_\phi (g_\phi,\zeta_r) }.
	\end{equation}

	Then:
	\begin{enumerate}
		\item
		      the quotient $\ol N \ceqq N \bs Z$ acts as a \emph{complex reflection group} on $\mf t_\phi (g_\phi,\zeta_r)$;

		\item
		      the reflecting hyperplanes of $\ol N$ are the traces on $\mf t_\phi(g_\phi,\zeta_r)$ of those of $G_\bullet(\phi)$,
		      i.e.,
		      the intersections $H \cap \mf t_\phi (g_\phi,\zeta_r) \sse \mf t_\phi(g_\phi,\zeta_r)$,
		      with $H \in \mc A_\bullet(\phi)$ not containing the eigenspace;

		\item
		      and if $\bigl( \mf t_\phi, G_\bullet(\phi) \bigr)$ is \emph{irreducible},
		      then so is $\bigl( \mf t_\phi(g_\phi,\zeta_r), \ol N \bigr)$.
	\end{enumerate}
\end{lemm}

\begin{proof}
	The inclusion $W(\phi) \sse G_\bullet (\phi)$ was established in \S\S~\ref{sec:type_A_relative_refl_group}--\ref{sec:type_D_relative_refl_group}:
	so one has $g_\phi \in G_\bullet(\phi)$,
	and the---%
	rational---%
	reflection coset $\bigl(\mf t_\phi,g_\phi G_\bullet(\phi)\bigr)$ is \emph{split}.

	Now the statement is the split version of~\cite[Thm.~1.1 + Thm.~A]{lehrer_springer_1999_reflection_subquotients_of_unitary_reflection_groups}.
\end{proof}

\begin{coro}
	\label{cor:restricted_reflection_group}

	Choose $g \in N_W(\mf t_\phi)$,
	and suppose that its class/restriction $g_\phi \in W(\phi)$ admits a $G_\bullet(\phi)$-\emph{regular} $\zeta_r$-eigenvector.\fn{
		One can find examples where this \emph{fails},
		in type $D$,
		even if $g$ itself is regular.}~Then:
	\begin{enumerate}
		\item
		      the conclusions of Lem.~\ref{lem:restricted_reflection_group} hold;

		\item
		      the pointwise-stabilizer~\eqref{eq:outer_centralizer_1} is trivial;

		\item
		      and the setwise-stabilizer~\eqref{eq:outer_normalizer} coincides with the centralizer of $g_\phi$ in $G_\bullet(\phi)$,
		      i.e.,
		      with the group
		      \begin{equation}
			      \label{eq:outer_centralizer_2}
			      Z_{G_\bullet(\phi)}(g_\phi)
			      = \Set{ g' \in G_\bullet(\phi) | g_\phi g' = g' g_\phi }.
		      \end{equation}
	\end{enumerate}
\end{coro}

\begin{proof}
	This follows from Lem.~\ref{lem:restricted_reflection_group},
	using~\cite[Prop.~6.3 + Thm.~6.4]{springer_1974_regular_elements_of_finite_reflection_groups} (cf.~Lem.~\ref{lem:restricted_generic_twisted_orbit} for the identification of the setwise stabilizer).
\end{proof}

\subsubsection{}

As already mentioned,
if $A \in \mf t_\phi(g_\phi,\zeta_r)$ is the regular eigenvector in the statement of Cor.~\ref{cor:restricted_reflection_group} (with trivial stabilizer in $G_\bullet(\phi)$),
then the irregular type $\wh Q \ceqq A w^{-1}$ gives rise to a WWG $W_r \bigl( \wh Q \bigr)$ which was identified with the centralizer $Z_{W,\phi}(g) \sse W(\phi)$ of~\eqref{eq:centralizer_in_quotient}.
Then there is an inclusion
\begin{equation}
    W_r \bigl( \wh Q \bigr)
    = W(\phi) \cap Z_{G_\bullet(\phi)}(g) \sse Z_{G_\bullet(\phi)}(g),
\end{equation}
taking centralizers in a larger group,
which can be \emph{proper}---%
already when $r = 1$ and $g = 1$,
cf.~again~Rmk.~\ref{rmk:no_braid_lift} in type $A$.
Therefore,
the subtlety is that the group~\eqref{eq:outer_centralizer_2} in general is \emph{not} modular,
i.e.,
it does not control the deck transformations of the Galois covering $\bm B_r \bigl( \wh Q \bigr) \thra \bm B_r \bigl( \wh \Theta \bigr)$ (where $\wh \Theta \ceqq \wh\Theta \bigl( \wh Q \bigr)$ as usual);
but on the contrary,
the pure $r$-admissible deformation space $\bm B_r \bigl( \wh Q \bigr)$ \emph{does} coincide with its regular part,
in view of the following:

\begin{coro}
	\label{cor:pure_factors_as_reflection_complements}

	Choose $g \in W$,
	and an eigenvector $A \in \mf t(g,\zeta_r)$;
	let also $\phi = \phi_A \sse \Phi_\bullet$ be the Levi annihilator of $A$.
	(Such that the above hypotheses hold for $\phi$,
	when in type $D$.)
	Then:
	\begin{enumerate}
		\item
		      the conclusions of Cor.~\ref{cor:restricted_reflection_group} hold;

		\item
		      and the hyperplane complement~\eqref{eq:pure_general_twisted_deformation_space_1_coeff} coincides with the regular part of $\mf t_{\phi}(g_\phi,\zeta_r)$ for the action of the complex reflection group~\eqref{eq:outer_centralizer_2}.
	\end{enumerate}
\end{coro}

\begin{proof}
	As usual,
	the given hypotheses yield $g \in N_W(\mf t_\phi)$,
	and so $g_\phi \in W(\phi)$ is well-defined and such that $g_\phi(A) = g(A) = \zeta_r A$.

	Furthermore,
	we have shown in \S\S~\ref{sec:type_A_relative_refl_group}--\ref{sec:type_D_relative_refl_group} that~\eqref{eq:pure_untwisted_nongeneric_deformation_space_1_coeff} is the $G_\bullet(\phi)$-regular part of $\mf t_\phi$,
	so that $A$ is $G_\bullet(\phi)$-regular---%
	as $A \not\in H_\alpha(\phi)$ for all $\alpha \in \Phi_\bullet \sm \phi$,
	by construction.

	The statement now follows from Cor.~\ref{cor:restricted_reflection_group},
	as the traces of the reflecting hyperplanes of $G_\bullet(\phi)$ on $\mf t_{\phi}(g_\phi,\zeta_r)$ are precisely the hyperplanes~\eqref{eq:hyperplane_intersection}.
\end{proof}

\subsubsection{}

The statement of Cor.~\ref{cor:pure_factors_as_reflection_complements}~(2.)~is the main conclusion that we were after in this section.
So the complete picture is that $\bm B_r \bigl( \wh Q \bigr)$ is still the regular part of a complex reflection group,
and that $\bm B_r \bigl( \wh \Theta \bigr)$ is the topological quotient with respect to the free action of a subgroup thereof.
(Again,
in the quasigeneric case this distinction vanishes.)

\section*{Fifth interlude}

We now make use of the results of \S~\ref{sec:lehrer_springer_theory} to classify the complex hyperplane arrangements arising from meromorphic $2d$ gauge theory,
when $\mf g$ is a \emph{classical} simple Lie algebra,
temporarily excluding the nonreflection obstructions in type $D$.

Precisely,
if $\mf g$ is of type $\bullet  \in \set{A,B,C}$,
then the material of \S\S~\ref{sec:type_A}--\ref{sec:type_BC} implies that the deformation spaces of \emph{arbitrary} $r$-Galois-closed irregular types can be factored into copies of the reflection complements $\mc M(1,k), \mc M^\sharp(\rho,k) \sse \mb C^k$ of generalized symmetric groups $G(\rho,1,k)$ (for suitable integers $\rho,k \geq 1$,
cf.~\eqref{eq:standard_complement_type_D}--\eqref{eq:standard_complement}).
In type $D$ instead,
the material of \S\S~\ref{sec:type_D} +~\ref{sec:D_trees} implies that $\mc M^\sharp(\rho,k)$ also arises as the reflection complement of index-$2$ subgroups $G(\rho,2,k) \sse G(\rho,1,k)$ (for suitably parity-constrained ramifications/ranks),
and that the unique new examples are then:
(i) the `unramified' complements $\mc M(2,k)$,
i.e.,
the regular vectors for the Weyl group $W_D(k)$;
and (ii) the infinite `exotic' family $\mc M^{\musDoubleSharp}(k,k') \sse \mb C^{k + k'}$,
interpolating between the reflection complements of $W_D(k+k')$ and $W_{B\!C}(k+k')$ (for another integer $k' \geq 1$,
cf.~\eqref{eq:exotic_arrangement}).
The latter already arises in the untwisted setting~\cite{doucot_rembado_tamiozzo_2026_local_wild_mapping_class_groups_and_cabled_braids},
and here we prove that there are \emph{no} new examples after twisting.
(This establishes Thm.~\ref{thm:thm_2_intro}.)

\section{Pure type \texorpdfstring{$A$}{A} (reviewed)}
\label{sec:type_A}

\subsection{Reduction to the quasigeneric case}

Keeping all notation therein,
the upshot of \S\S~\ref{sec:type_A_relative_refl_group} +~\ref{sec:some_lehrer_springer} is that every monolevel type-$A$ example can be viewed as a \emph{quasigeneric} type-$A$ example of lower rank,
and that all the (rational) reflection cosets involved are split.

\subsection{Quasigeneric classification}

Therefore,
noting that any (Levi) root subsystem $\phi \sse \Phi_A(m-1)$ splits into a disjoint union of irreducible components of type $A$,
the \emph{complete} classification of hyperplane arrangements in the general multilevel case follows (recursively) from the:

\begin{prop}
	\label{prop:classification_type_A}

	Choose integers $m \geq 1$ and $r \geq 2$ (to focus on the ramified case).
	Let also $g$ be a regular element of $\bigl( V^+_m,W_A(m) \bigr)$ of order $r$,
	and set
	\begin{equation}
		\bm B_{A,r}(m) = \bm B_{A,g,r}(m) \ceqq V^+_m(g,\zeta_r) \cap V^+_{m,\reg}.
	\end{equation}
	Then:
	\begin{enumerate}
		\item
		      either $r \mid m$ or $r \mid (m-1)$,
		      and if $k \geq 1$ is the quotient of the division then there is a $\mb C$-\emph{linear isomorphism} $\bm B_{A,r}(m) \simeq \mc M^\sharp(r,k)$ (cf.~\eqref{eq:standard_complement}),
		      of complex hyperplane complements;

		\item
		      and the corresponding complex reflection group is the generalized symmetric group $G(r,1,k) \simeq (\mb Z \bs r\mb Z) \wr \mf S_k$ in Shephard--Todd's \emph{monomial} list~\cite{shephard_todd_1954_finite_unitary_reflection_groups,cohen_1976_finite_complex_reflection_groups}.
	\end{enumerate}
\end{prop}

\begin{proof}
	For the first statement,
	by~\cite[\S~5.1]{springer_1974_regular_elements_of_finite_reflection_groups},
	regular elements split into disjoint cycles of one and the same length,
	and fix at most one element of $\ul m^+$.
	In particular,
	in any such factorization of $g$,
	the length of all cycles equals $r$.

	Suppose first that $g$ fixes no element:
	then $r \mid m$ and $m = kr$,
	so that $g$ is the product of $k$ disjoint (positive) $r$-cycles.
	Up to conjugation,
	one can assume that
	\begin{equation}
		\label{eq:cycle_decomposition_type_A}
		g
		= c^+_1 \dm c^+_k,
		\qquad c^+_j \ceqq \bigl( (j-1)r + 1 \mid (j-1)r + 2 \mid \dm \mid jr - 1 \mid jr \bigr),
	\end{equation}
	in the notation of~\eqref{eq:type_A_cycle}.
	Then
	\begin{equation}
		\label{eq:eigenspace_base_type_A}
		V^+_m(g,\zeta_r)
		= \spann_{\mb C} \set{A_1,\dc,A_k},
		\quad A_i \ceqq \prescript{t}{} ( \underbrace{0, \dc, 0,}_{(i-1)r \text{ times }} 1,\zeta_r^{r-1},\dc,\zeta_r, \underbrace{0, \dc, 0}_{(k-i)r \text{ times }} ).
	\end{equation}
	Hence,
	the regular eigenvectors are of the form
	\begin{equation}
		\label{eq:regular_eigenvector_type_A}
		A
		= \sum_{i = 1}^k \lambda_i A_i,
		\qquad 0 \neq \zeta_r^l\lambda_i \neq \zeta_r^{\wt l} \lambda_j,
		\qquad i \neq j \in \ul k^+,
		\quad l,\wt l \in \mb Z \bs r \mb Z.
	\end{equation}

	If instead $g$ fixes one element,
	then $r \mid (m-1)$.
	By hypothesis $\zeta_r \neq 1$,
	and so a fixed coordinate in any eigenvector must vanish:
	now the above applies verbatim inside $V^+_{m-1} \simeq \mb C^{m-1} \hra \mb C^m \simeq V^+_m$.

	For the second statement,
	it follows,
	e.g.,
	from~\cite[Lem.~1]{kezlan_rhee_1999_a_characterization_of_the_centralizer_of_a_permutation},
	that one has
	\begin{equation}
		\label{eq:generalized_symmetric_group}
		Z_{W_A(m)}(g)
		= \Set{ (c^+_1)^{l_1} \dm (c^+_k)^{l_k} \cdot g' | l_1,\dc,l_k \in \mb Z \bs r \mb Z, \, g' \in \mf S_k }
		\simeq \mf S_k \lts \bigl( \mb Z \bs r \mb Z \bigr)^k,
	\end{equation}
	with tacit use of the block-permutation embedding $\mf S_k \hra \mf S_m$ (cf.~\cite[\S~4]{doucot_rembado_2025_topology_of_irregular_isomonodromy_times_on_a_fixed_pointed_curve};
	one could also conclude by looking at the multiset of degrees of $Z_{W_A(m)}(g)$,
	together with its reflection arrangement,
	cf.~the proof of Prop.~\ref{prop:classification_type_BC}).
\end{proof}

\section{Pure type \texorpdfstring{$B\!C$}{BC}}
\label{sec:type_BC}

We now extend the classification beyond the case of algebraic connections on (bare) vector bundles.

\subsection{Reduction to the quasigeneric case}

Keeping all the notation therein,
the upshot of \S\S~\ref{sec:type_BC_relative_refl_group} +~\ref{sec:some_lehrer_springer} is that every monolevel type-$B\!C$ example can be viewed as a \emph{quasigeneric} type-$B\!C$ example of lower rank,
and that all the (rational) reflection cosets involved are split.

\subsection{Quasigeneric classification}

Therefore,
noting that any Levi subsystem $\phi \sse \Phi_{B/C}(m)$ splits into a disjoint union of at most one irreducible component of type $B/C$,
and several components of type $A$ (cf.~Rmk.~\ref{rmk:type_C_component}),
the \emph{complete} classification of hyperplane arrangements in the general multilevel case follows (recursively) from Prop.~\ref{prop:classification_type_A},
together with the:

\begin{prop}
	\label{prop:classification_type_BC}

	Choose integers $m,r \geq 2$,
	and consider a regular element $g$ of $\bigl( \wt V_m,W_{B\!C}(m) \bigr)$ of order $r$.
	Set also
	\begin{equation}
		\label{eq:generic_classification_BC}
		\bm B_{B\!C,r}(m)
		= \bm B_{B\!C,g,r}(m)
		\ceqq \wt V_m(g,\zeta_r) \cap \wt V_{m,\reg}.
	\end{equation}
	Then one of the two following (mutually-exclusive) situations occurs:
	\begin{enumerate}
		\item
		      the integer $r$ is odd and $r \mid m$;
		      if $k \geq 1$ is the quotient of the division then:
		      \begin{enumerate}
			      \item
			            there is a $\mb C$-\emph{linear isomorphism} $\bm B_{B\!C,r}(m) \simeq \mc M^\sharp(2r,k)$,
			            of complex hyperplane complements (cf.~\eqref{eq:standard_complement});

			      \item
			            and the corresponding complex reflection group is the generalized symmetric group $G(2r,1,k)$;
		      \end{enumerate}

		\item
		      or $r$ is even and $r \mid (2m)$;
		      if again $k \geq 1$ is the quotient then:
		      \begin{enumerate}
			      \item
			            one has $\bm B_{B\!C,r}(m) \simeq \mc M^\sharp(r,k)$;

			      \item
			            and the corresponding complex reflection group is $G(r,1,k)$.
		      \end{enumerate}
	\end{enumerate}
\end{prop}

\begin{proof}
	For the hyperplane complements,
	by~\cite[\S~5.2]{springer_1974_regular_elements_of_finite_reflection_groups},
	a regular element of order $r$ falls into two (mutually-exclusive) classes:
	\begin{enumerate}
		\item
		      (i) the integer $r$ is odd and $r \mid m$%
		      ---say $m = rk$ as in the statement;
		      and (ii) $g$ is a product of $k$ positive disjoint $r$-cycles;

		\item
		      or (i) $r$ is even and $r \mid (2m)$%
		      ---say $r = 2r'$ and $m = r'k$;
		      and (ii) $g$ is a product of $k$ negative disjoint $r'$-cycles.
	\end{enumerate}

	Up to conjugation,
	in the first case it is enough to consider the element
	\begin{equation}
		g
		= c_1^+ c_1^- \dm c_k^+ c_k^-,
		\qquad c_j^- \ceqq \bigl( (1-j)r - 1 \mid (1-j)r - 2 \mid \dm \mid 1 - jr \mid - jr \bigr),
	\end{equation}
	with $c^+_j$ as in~\eqref{eq:cycle_decomposition_type_A} (cf.~\eqref{eq:positive_cycle}).
	Then one has $\wt V_m(g,\zeta_r) = \spann_{\mb C} \set{ A_1,\dc,A_k }$,
	where
	\begin{equation}
		A_i
		\ceqq \prescript{t}{} ( \underbrace{0,\dc,0,}_{(i-1)r \text{ times }} 1,\zeta_r^{r-1},\dc,\zeta_r, \underbrace{0,\dc,0,}_{r(k-1) \text{ times }} -1,-\zeta_r^{r-1},\dc,-\zeta_r, \underbrace{0,\dc,0}_{(k-i)r \text{ times }} ),
	\end{equation}
	cf.~\eqref{eq:eigenspace_base_type_A}.
	Hence,
	the regular eigenvectors are of the form
	\begin{equation}
		A
		= \sum_{i = 1}^k \lambda_i A_i,
		\qquad 0 \neq \zeta_r^l \lambda_i \neq \pm \zeta_r^{\wt l} \lambda_j,
		\quad i \neq j \in \ul k^+,
		\quad l,\wt l \in \mb Z \bs r \mb Z,
	\end{equation}
	cf.~\eqref{eq:regular_eigenvector_type_A}.
	We conclude by Lem.~\ref{lem:even_odd_roots_of_1}~(1.).

	Analogously,
	in the second case it is enough to consider the element $g = \wt c_1 \dm \wt c_k$,
	where
	\begin{equation}
		\wt c_j
		\ceqq \bigl( (j-1)r' + 1 \mid \dm \mid jr' - 1 \mid jr' \mid  (1-j)r' - 1 \mid \dm \mid 1 - jr' \mid -jr' \bigr),
	\end{equation}
	cf.~\eqref{eq:negative_cycle}.
	Again we find a $k$-dimensional eigenspace,
	with basis
	\begin{equation}
		\wt A_i
		= \prescript{t}{} ( \underbrace{0,\dc,0,}_{ (i-1)r' \text{ times }} 1,\zeta_r^{r-1},\dc,\zeta_r^{r'+1}, \underbrace{0,\dc,0,}_{ r'(k-1) \text{ times }} \zeta_r^{r'},\zeta_r^{r'-1},\dc,\zeta_r, \underbrace{0,\dc,0}_{ (k-i)r' \text{ times}} ),
	\end{equation}
	for $i \in \set{1,\dc,k}$.
	(This expression makes sense by Lem.~\ref{lem:even_odd_roots_of_1}~(2.).)
	The regular eigenvectors are of the form
	\begin{equation}
		\wt A
		= \sum_{i = 1}^k \lambda_i \wt A_i,
		\qquad \lambda_i^r \neq \lambda_j^r \in \mb C^{\ts}.
	\end{equation}

	As for the complex reflection groups,
	recall that $W \ceqq W_{B\!C}(m) = G(2,1,m)$ in Shephard--Todd's list.
	Then~\cite[Prop.~5.2]{lehrer_springer_1999_reflection_subquotients_of_unitary_reflection_groups} states that $W(r) \ceqq Z_W(g)$ is isomorphic to a monomial group $G(\rho,e,k)$,
    for suitable integer parameters $\rho,e \geq 1$,
	with $e \mid \rho$---%
	the rank $k$ is determined here.
	Now:
	\begin{enumerate}
		\item
		      the degrees of $G(\rho,e,k)$ are given by the \emph{multiset} $\set{\rho,2\rho,\dc,(k-1)\rho; d k}$,
		      writing $\rho = de$ for a suitable integer $d \geq 1$;

		\item
		      and the degrees of $W(r)$ are the integers $d_i \in \set{2,4,\dc,2m}$ of $W$ which are divisible by $r$ (cf.~\cite[Thm.~4.2~(iii)]{springer_1974_regular_elements_of_finite_reflection_groups}),
		      whence:
		      \begin{enumerate}
			      \item
			            if $r$ is odd,
			            these are the integers $d'_i \in \set{2r,4r,\dc,2kr = 2m}$;

			      \item
			            and if $r = 2r'$ is even,
			            they are $d'_i \in \set{2r',4r',\dc,2kr' = 2m}$.
		      \end{enumerate}
	\end{enumerate}
	If $k = 1$,
	in the odd case the unique degree forces the cyclic group
	\begin{equation}
		\mb Z \bs (2r) \mb Z \simeq G(2er,e,1) \simeq G(2r,1,1),
		\qquad e \geq 1,
	\end{equation}
	and analogously the even case yields $\mb Z \bs r\mb Z \simeq G(r,1,1)$.
	If $k \geq 2$,
	the reflection arrangement of $G(\rho,e,k)$ has $\rho$ hyperplanes for each pair of distinct coordinates,
	and so by the above one must have $\rho = 2r$ in the odd case (resp.,
	$\rho = 2r'$ in the even one).
	This equality forces $dk = 2kr$ to be the maximal degree (resp.,
	$dk = 2kr'$),
	whence $e = \rho \slash d = 1$.
\end{proof}

\begin{lemm}
	\label{lem:even_odd_roots_of_1}

	Let $r \geq 1$ be an integer.
	Then:
	\begin{enumerate}
		\item
		      if $r$ is \emph{odd},
		      then the set $\set{ \pm \zeta_r, \pm \zeta_r^2,\dc, \pm \zeta_r^{r-1},\pm 1} \sse \mb C^{\ts}$ consists of all the $(2r)$-th roots of $1$;

		\item
		      if $r = 2r'$ is \emph{even},
		      then $\zeta_r^l + \zeta_r^{l-r'} = 0$ for $l \in \set{r',\dc,r}$;
		      (And so  the set $\set{ \pm \zeta_r, \pm \zeta_r^2,\dc, \pm \zeta_r^{r-1},\pm 1 }$ still consists of all the $r$-th roots of $1$.)

		\item
		      and the order of $-\zeta_r \in \mb C^{\ts}$ equals:
		      \begin{enumerate}
			      \item
			            $2r$,
			            if $r$ is odd;

			      \item
			            $r$,
			            if $r \equiv 0 \pmod{4}$;

			      \item
			            or $r \slash 2$,
			            if $r \equiv 2 \pmod{4}$.
		      \end{enumerate}
	\end{enumerate}
\end{lemm}

\begin{proof}[Proof omitted]
\end{proof}

\section{Pure type \texorpdfstring{$D$}{D}:
  crystallographic examples}
\label{sec:type_D}

\subsection{Partial reduction to the quasigeneric case}
\label{sec:setup_type_D}

Keeping the notation therein,
the upshot of \S\S~\ref{sec:type_D_relative_refl_group} +~\ref{sec:some_lehrer_springer} is that any monolevel type-$D$ example is:
(i) either quasigeneric (and treated by~Prop.~\ref{prop:generic_classification_type_D});
or (ii) equivalent to a lower-rank quasigeneric type-$B\!C$ example (and treated by~Prop.~\ref{prop:classification_type_BC});
or (iii) it involves the `exotic' complement $\mc M^{\musDoubleSharp}(k,k')$ of~\eqref{eq:exotic_arrangement},
where $k,k' > 0$.

(Howlett's obstruction need \emph{not} be trivial outside of case (i).)

\subsubsection{}

Since every Levi subsystem $\phi \sse \Phi_m(D)$ splits into the disjoint union of at most one part of type $D$,
plus several type-$A$ components,
the crux of the matter to complete the multilevel classification is whether case (iii) arises at some step of fission:
this will be treated by means of $D$-fission trees (cf.~\S~\ref{sec:D_trees},
incidentally showing that there is at most one nonreflection factor).
Note that this also includes the `boundary' case of $\mc M^{\musDoubleSharp}(1,1) \sse \mb C^2$,
which is nonunitarily equivalent to the configuration space of ordered triples of distinct points in the complex affine line (cf.~Prop.~\ref{prop:nonreflection_arrangement}~(1.)).

\subsection{Quasigeneric classification}

\begin{prop}
	\label{prop:generic_classification_type_D}

	Choose integers $r \geq 2$ and $m \geq 4$,
	and consider a regular element $g$ of $\bigl( \wt V_m,W_D(m) \bigr)$ of order $r$.
	Set also
	\begin{equation}
		\label{eq:quasigeneric_classification_type_D}
		\bm B_{D,r}(m)
		= \bm B_{D,g,r}(m)
		\ceqq \wt V_m(g,\zeta_r) \cap \wt V_{m,\reg},
	\end{equation}
	now taking that $W_D(m)$-regular part (compared to~\eqref{eq:generic_classification_BC}).
	Then one of the six following (mutually-exclusive) situations occurs:
	\begin{enumerate}
		\item
		      the integer $r$ is odd and $r \mid m$,
		      and if $k \geq 1$ is the quotient of the division then:
		      \begin{enumerate}
			      \item
			            there is a $\mb C$-\emph{linear isomorphism} $\bm B_{D,r}(m) \simeq \mc M^\sharp(2r,k)$,
			            of complex hyperplane complements;

			      \item
			            and the corresponding complex reflection group is the monomial entry $G(2r,2,k)$ in Shephard--Todd's list;
		      \end{enumerate}

		\item
		      or $r$ is odd and $r \mid (m-1)$,
		      and if again $k \geq 1$ is the quotient then:
		      \begin{enumerate}
			      \item
			            there is still a $\mb C$-\emph{linear isomorphism} $\bm B_{D,r}(m) \simeq \mc M^\sharp(2r,k)$;

			      \item
			            but the corresponding complex reflection group is the generalized symmetric group $G(2r,1,k)$;
		      \end{enumerate}

		\item
		      or $r \geq 4$ is even and $r \mid m$,
		      and if $k \geq 1$ is the quotient then:
		      \begin{enumerate}
			      \item
			            one has $\bm B_{D,r}(m) \simeq \mc M^\sharp(r,2k)$;

			      \item
			            and the corresponding complex reflection group is $G(r,2,2k)$;
		      \end{enumerate}

		\item
		      or $r \geq 4$ is even and $r \mid (2m -2)$,
		      and if $k \geq 1$ is the quotient then:
		      \begin{enumerate}
			      \item
			            one has $\bm B_{D,r}(m) \simeq \mc M^\sharp(r,k)$;

			      \item
			            and the corresponding complex reflection group is $G(r,1,k)$;
		      \end{enumerate}

		\item
		      or $r = 2$ and $m$ is \emph{even},
		      and then:
		      \begin{enumerate}
			      \item
			            one has $\bm B_{D,r}(m) = \bm B_{D,2}(m) \simeq \mc M(2,m)$ (cf.~\eqref{eq:standard_complement_type_D});

			      \item
			            and the corresponding reflection group is $G(2,2,m) \simeq W_D(m)$;
		      \end{enumerate}

		\item
		      or finally $r = 2$ and $m$ is \emph{odd},
		      and then:
		      \begin{enumerate}
			      \item
			            one has $\bm B_{D,r}(m) = \bm B_{D,2}(m) \simeq \mc M^\sharp(2,m-1)$;

			      \item
			            and the corresponding reflection group is $G(2,1,m-1) \simeq W_{B\!C}(m-1)$.
		      \end{enumerate}
	\end{enumerate}
\end{prop}

\begin{proof}
	For the hyperplane complements,
	by~\cite[\S~5.3]{springer_1974_regular_elements_of_finite_reflection_groups},
	one of the following (non-mutually-exclusive) situations occurs:
	\begin{enumerate}
		\item
		      (i) the integer $r$ is odd and $r \mid m$%
		      ---say $m = rk$ for an integer $k \geq 1$;
		      and (ii) $g$ is a product of $k$ positive disjoint $r$-cycles;

		\item
		      or (i) $r$ is odd and $r \mid (m-1)$%
		      ---say $m-1 = rk$ for an integer $k \geq 1$;
		      (ii) $g$ is a product of $k$ positive disjoint $r$-cycles;
		      and (iii) $g$ fixes each element of an opposite pair;

		\item
		      or (i) $r$ is even and $r \mid m$%
		      ---say $r = 2r'$ and $m = 2kr'$,
		      for integers $r',k \geq 1$;
		      and (ii) $g$ is a product of $2k$ negative disjoint $r'$-cycles;

		\item
		      or (i) $r$ is even and $r \mid (2m-2)$%
		      ---say $r = 2r'$ and $m-1 = r'k$ for integers $r',k \geq 1$;
		      (ii) $g$ is a product of $k$ negative disjoint $r'$-cycles,
		      plus possibly a negative transposition;
		      and (iii) the latter appears if and only if $k$ is \emph{odd}.
		      (I.e.,
		      if $k$ is even then $g$ fixes each element of an opposite pair.)
	\end{enumerate}

	The first case is equivalent to Prop.~\ref{prop:classification_type_BC}~(1.),
	and since $\zeta_r \neq 1$ the second case reduces to it when looking for eigenvectors.
	Analogously,
	the third case is equivalent to a particular case of Prop.~\ref{prop:classification_type_BC}~(2.) (with an even number of cycles),
	and the fourth case reduces to it if $k$ is even;
	if instead $k$ is odd,
	but $r \geq 4$,
	then one can use that $-\zeta_r \neq 1$ to reduce to a previous case.
	Finally,
	if $r = 2$ in the fourth case,
	whence $r' = 1$ and $m = k + 1$,
	then there are two mutually-exclusive situations.
	\begin{enumerate}
		\item
		      If $k$ is odd,
		      one finds the sign-swapping $g \cl A \mt -A$.
		      Therefore,
		      one has $\wt V_m(g,\zeta_r) = \wt V_m(g,-1) = \wt V_m$,
		      which contains the whole $W_D(m)$-regular locus.

		\item
		      Or $k$ is even.
		      Up to conjugation---%
		      in $W_D(m)$---%
		      one has
		      \begin{equation}
			      g
			      = (1 \mid -1) \dm (m-1 \mid 1-m) \in W_D(m),
		      \end{equation}
		      and the eigenspace $\wt V_m(g,-1) \sse \wt V_m$ consists of the vectors
		      \begin{equation}
			      (\lambda_1,\dc,\lambda_{m-1},0,-\lambda_1,\dc,-\lambda_{m-1},0),
			      \qquad \lambda_i \in \mb C.
		      \end{equation}
		      Intersecting with the $W_D(m)$-regular parts yields a copy of the complement $\mc M^\sharp(2,k)$---%
		      i.e.,
		      the regular part of $W_{B\!C}(k)$,
		      since now coordinates cannot vanish.
	\end{enumerate}

	As for the complex reflection groups,
	let $W \ceqq W_D(m)$,
	whose multiset of degrees is $\set{ 2,4,\dc,2(m-1); m }$.
	As in the proof of Prop.~\ref{prop:classification_type_BC},
	the point is to identify $W(r) \ceqq Z_W(g)$ with a monomial group $G(\rho,e,k')$,
	for integers $\rho,e \geq 1$ with $e \mid \rho$ and $k' \in \set{k,2k}$.
	Now,
	by Springer,
	the degrees of $W(r)$ are:
	\begin{enumerate}
		\item
		      if $m = kr$ with $r$ odd,
		      the integers $d'_i \in \set{ 2r,4r,\dc,2(k-1)r;kr }$;

		\item
		      if $m = kr + 1$ with $r$ odd,
		      the integers $d'_i \in \set{ 2r,4r,\dc,2kr }$ ($r \nmid m$ here);

		\item
		      if $m = kr$ with $r = 2r'$ even,
		      the integers $d'_i \in \set{ 2r',4r',\dc,2(2k-1)r';2kr' }$;

		\item
		      and if $m = kr' + 1$ with $r = 2r'$ even,
		      then:
		      \begin{enumerate}
			      \item
			            if $r \geq 4$,
			            the integers $\set{ 2r',4r',\dc,2kr' }$ ($r \nmid m$ here);

			      \item
			            and if $r = 2$,
			            whence again $m = k+1$,
			            then:
			            \begin{enumerate}
				            \item
				                  the integers $d'_i \in \set{ 2,4,\dc,2k }$,
				                  if $k$ is even ($r \nmid m$ here);

				            \item
				                  and the integers $d'_i \in \set{ 2,4,\dc,2k;k+1 }$,
				                  if $k$ is odd.
			            \end{enumerate}
		      \end{enumerate}
	\end{enumerate}
	Reasoning as in the proof of Prop.~\ref{prop:classification_type_BC},
	the first case yields $W(r) \simeq G(2r,2,k)$ for all $k \geq 1$,
	with the usual identifications $\mb Z \bs r\mb Z \simeq G(2r,2,1) \simeq G(r,1,1)$.
	Analogously,
	the second case (resp.,
	the third one) yields $W(r) \simeq G(2r,1,k)$ (resp.,
	$W(r) \simeq G(r,2,2k)$).
	Finally,
	the case $(4.a)$ yields $W(r) \simeq G(r,1,k)$,
	while if $r = 2$ then
	\begin{equation}
		W(r) = W(2) \simeq
		\begin{cases}
			G(2,1,k) = W_{B\!C}(k), & \quad k \geq 2 \text{ even}; \\
			G(2,2,k+1) = W_D(k+1),  & \quad k \geq 1 \text{ odd}.
		\end{cases} \qedhere
	\end{equation}
\end{proof}

\begin{rema}
	Recall that $G(2r,1,k) \sse \GL_{\mb C}(k)$ consists of \emph{monomial} matrices whose nonzero entries lie in the group $\mu_{2r} \sse \mb C^{\ts}$ of $(2r)$-th roots of $1$,
	while $G(2r,2,k) \sse G(2r,1,k)$ is the index-$2$ subgroup of matrices such that the product of the $k$ nonvanishing entries lies in the subgroup $\mu_r \sse \mu_{2r}$.
	(This is the $r$-ramified version of the index-$2$ embedding $W_D(m) \sse W_{B\!C}(m)$.)
	The subtlety in the classification of Prop.~\ref{prop:generic_classification_type_D}~(1.)--(2.) is that $2 < 2r$:
	the reflection arrangements of $G(2r,1,k)$ and $G(2r,2,k)$ coincide,
	but we disambiguate the complex reflections groups by their degrees;
	and similarly for some other cases.

	Moreover,
	the genuinely complex reflection groups which do \emph{not} already arise in types $A/B/C$ are the $G(\rho,2,k)$ with $\rho \geq 4$ (even),
	where moreover $\rho$ is twice an odd number if $k$ is odd (cf.~Cor.~\ref{cor:modular_reflection_groups}~(2.a.ii);
	the latter constraint drops by looking at the 2-twisted $B\!C$-reflection coset of $W_D(m)$,
	cf.~Prop.~\ref{prop:twisted_generic_classification_type_D}~(2.b)).
\end{rema}

\section*{Sixth interlude}

The next aim is to:
(i) conclude the proof of the type-$D$ examples in Thm.~\ref{thm:thm_2_intro},
involving the hyperplane complements~\eqref{eq:exotic_arrangement};
and (ii) prove Thm.~\ref{thm:thm_3_intro}.
This occupies \S\S~\ref{sec:BC_trees}--\ref{sec:D_trees},
while the next (preliminary) \S~\ref{sec:notation_for_trees} introduces the notions/notations needed to define and study twisted generalizations of the fission trees of~\cite{doucot_rembado_tamiozzo_2026_local_wild_mapping_class_groups_and_cabled_braids,doucot_rembado_2025_topology_of_irregular_isomonodromy_times_on_a_fixed_pointed_curve},
i.e.,
an extension of the trees of~\cite{boalch_doucot_rembado_2025_twisted_local_wild_mapping_class_groups_configuration_spaces_fission_trees_and_complex_braids} beyond type $A$ (cf.~also~\cite{boalch_2025_counting_the_fission_trees_and_nonabelian_hodge_graphs_untwisted_case}).

\begin{rema}
	\label{rmk:abusive_notation}

	Beware that in \S\S~\ref{sec:BC_trees}--\ref{sec:D_trees} we shall use the symbol $\mb A$ (resp.,
	$\mb L$) to denote sets of `admissible' exponents/vertices (resp.,
	`mandatory' ones).
	This will locally \emph{overwrite} the intended meaning of the previous sections---%
	which was rather related to Adjoint quotients and Levi subcosets.
\end{rema}

\section{Preparation for fission trees}
\label{sec:notation_for_trees}

\subsection{Exponential factors,
	Stokes circles,
	etc.}

Let again $x$ be a closed marked point on a smooth connected projective curve $X$ over $\mb C$,
analytified by the closed Riemann surface $\Sigma \ceqq X(\mb C)$ (cf.~\S~\ref{sec:about_twisted_irr_classes}).
Choose a local coordinate $z$ on $\Sigma$,
vanishing at $x$,
and fix:
(i) a point $d \in \partial \sse \wh\Sigma \to \Sigma$ in the oriented circle of real directions at $x$;
and (ii) a branch of $\log(z)$ in---%
a germ of---%
an open sector about $d$ (cf.~\cite[\S~3.3]{boalch_yamakawa_twisted_wild_character_varieties}).
Then:

\begin{defi}
	\label{def:exponential_factors}

	An \emph{exponential factor along} $d$ is the principal part $q$ of a $\mb C$-valued (formal) Puiseux series in $z$.
\end{defi}

\subsubsection{}

Recall that Puiseux series have rational exponents---%
and bounded denominators.
A nonvanishing exponential factor thus looks like $q = \sum_k \wt a_k z^{-k}$,
where $z^{-k} \ceqq e^{-k \cdot \log(z)}$;
it has coefficients $\wt a_k \in \mb C$,
such that $\wt a_k \neq 0$ for \emph{finitely many} exponents $k \in \mb Q_{> 0}$.
Therefore,
there exists a \emph{minimal} integer $r \geq 1$ such that $r \cdot k \in \mb Z_{> 0}$ whenever $\wt a_k \neq 0$,
and one can uniquely write
\begin{equation}
	\label{eq:exponential_factor}
	q
	= \sum_{i = 1}^s a_i z^{-i \slash r},
	\qquad a_i
	\ceqq \wt a_{i \slash r},
\end{equation}
for a suitable integer $s \geq 1$ such that $a_s \neq 0$ (which is the `leading' coefficient).

We set $\on{ram}(q) \ceqq r$ and $\irr(q) \ceqq s$,
which are---%
respectively---%
the \emph{ramification} and \emph{irregularity of} $q$.
Then we say that $q$ is $r$-\emph{ramified},
and the \emph{slope of} $q$ is the number
\begin{equation}
	\label{eq:slope}
	\on{slope}(q) \ceqq s \slash r = \on{irr}(q) \slash \on{ram}(q) \in \mb Q_{> 0}.
\end{equation}
Finally,
extend this terminology via $\on{ram}(0) \ceqq 1$ and $\on{irr}(0) \ceqq 0$.

\begin{rema}
	\label{rmk:exp_factors_as_diag_terms}

	Hereafter,
	and until \S~\ref{sec:tglwmcgs},
	we suppose that $\mf g$ is a simple Lie algebra of classical type $\bullet \in \set{B,C,D}$.
	As in \S\S~\ref{sec:type_BC}--\ref{sec:type_D},
	we identify the standard `diagonal' Cartan subalgebra $\mf t \sse \mf g$ with the vector space $\wt V_m \simeq \mb C^m$,
	where $m \geq 1$ is the rank of $\mf g$.
	Then one can write the coefficients of an irregular type~\eqref{eq:explicit_irregular_type} as
	\begin{equation}
		A_j
		= \bigl( a_{1,j},\dc,a_{m,j},-a_{1,j},\dc,-a_{m,j} \bigr) \in \mf t,
		\qquad  j \in \set{1,\dc,s},
	\end{equation}
	and in turn
	\begin{equation}
		\wh Q
		= \bigl( \wh q_1,\dc,\wh q_m,-\wh q_1,\dc,-\wh q_m \bigr),
		\qquad \wh q_i = \sum_j a_{i,j} w^{-j}.
	\end{equation}
	Now each term $\wh q_i$ is the pullback of an exponential factor as in~\eqref{eq:exponential_factor},
	along the usual $r$-fold covering $w \mt w^r = z$.
\end{rema}

(Hereafter,
and until \S~\ref{sec:tglwmcgs},
however,
we mostly work with \emph{rational} powers of $z$.)

\subsubsection{}

Let $\sigma$ denote the monodromy of the exponential local system $\mc I$ on $\partial$ (cf.~the action~\eqref{eq:monodromy}).
If $r \ceqq \on{ram}(q)$,
then the $j$-\emph{th Galois conjugate of} $q = q^{(0)}$ is
\begin{equation}
	\label{eq:galois_conjugates}
	q^{(j)}
	\ceqq
	\sigma^j(q)
	= \sum_{i = 1}^s  \zeta_r^{ij} a_i \cdot z^{-i/r},
	\qquad j \in \mb Z \bs r \mb Z,
\end{equation}
cf.~\eqref{eq:explicit_global_r_galois_orbit}.
The ($r$-)Galois orbit
\begin{equation}
	\braket q = \braket q_r \ceqq \set{q^{(0)},\dc,q^{(r-1)}},
\end{equation}
of conjugates of $q$,
is the \emph{Stokes circle of} $q$.
Denote by $\mc S$ the set of all Stokes circles $I = \braket q$,
whose ramifications/irregularities are well-defined by $\on{ram} \bigl( \braket q \bigr) \ceqq \on{ram}(q)$ and $\on{irr} \bigl( \braket q \bigr) \ceqq \on{irr}(q)$ (respectively);
so we also write $\on{slope}(I) \ceqq \on{irr}(I) \bs \on{ram}(I) \in \mb Q_{> 0}$.

The \emph{height}-$k$ \emph{truncation} of exponential factors/Stokes circles,
for any $k \in \mb Q_{\geq 0}$,
is
\begin{equation}
	\label{eq:truncation}
	\tau_k(I)
	= \Braket{ \tau_k(q) },
	\qquad \tau_k(q)
	\ceqq \sum_{i = \lceil kr \rceil}^s a_i z^{-i/r}.
\end{equation}

If $q$ is a nonvanishing exponential factor as in~\eqref{eq:exponential_factor},
denote by $E(q) \sse \mb Q_{>0}$ the set of numbers $i \slash r$ such that $a_i \neq 0$.
Analogously,
if $I = \braket q$ is a Stokes circle,
let $E(I) \ceqq E(q)$.
We shall order the elements of $E(q) = E \bigl( \braket q \bigr)$ in strictly-decreasing fashion,
and set $E(0) = E \bigl( \braket 0 \bigr) \ceqq \vn$.
(Recall that $I = \braket{0}$ is a.k.a.~the \emph{tame circle};
it plays an important role in type $D$).

Finally,
we recall the notions of `level data' and `admissible/inconsequential exponents',
in type $A$ (cf.~\cite[Def.~3.4]{boalch_doucot_rembado_2025_twisted_local_wild_mapping_class_groups_configuration_spaces_fission_trees_and_complex_braids}):

\begin{defi}
	Let $q$ be an $r$-ramified exponential factor with Stokes circle $I \ceqq \braket q$.
	Then:
	\begin{enumerate}
		\item
		      the $A$-\emph{level datum of} $q$,
		      or (equivalently) \emph{of} $I$,
		      is the set
		      \begin{equation}
			      \label{eq:type_a_level_data}
			      L_A(q)
			      = L_A(I)
			      \ceqq \Set{ \on{slope} \bigl( q^{(i)} - q^{(j)} \bigr) | i,j \in \mb Z \slash r \mb Z } \sm \set{0},
		      \end{equation}
		      in the notation of~\eqref{eq:slope} +~\eqref{eq:galois_conjugates};

		\item
		      and conversely,
		      a finite subset $\bm L \sse \mb Q_{> 0}$ is an $A$-\emph{level datum} if there exists a Stokes circle $I$ such that $L_A(I) = \bm L$.
	\end{enumerate}
\end{defi}

\begin{defi}
	Let $\bm L$ be an $A$-level datum.
	Then:
	\begin{enumerate}
		\item
		      The set of $A$-\emph{admissible exponents of} $\bm L$ is
		      \begin{align}
			      \on{Adm}_A(\bm L)
			       & \ceqq \Set{ k\in \mb Q_{>0} | \text{ there exists } I \in \mc S \text{ such that } L_A(I) = \bm L \text{ and } k \in E(I) } \\
			       & = \bigcup_{I \colon L_A(I) = \bm L} E(I);
		      \end{align}

		\item
		      and the set of $A$-\emph{inconsequential exponents of} $\bm L$ is the complement
		      \begin{equation}
			      \label{eq:a_inconsequential_exponents}
			      \on{Inc}_A(\bm L)
			      \ceqq \on{Adm}_A(\bm L) \sm \bm L.
		      \end{equation}

		      (One always has $\bm L \sse \on{Adm}_A(\bm L)$.)
	\end{enumerate}
\end{defi}

\begin{rema}
	\label{rmk:levels_are_admissible}

	Note that $L_A(q) = L_A(-q)$:
	an analogous symmetry under sign-swap will be quite relevant to treat the Weyl groups of type $B\!C$ and $D$.
\end{rema}

\section{Twisted fission trees of type \texorpdfstring{$B\!C$}{BC}}
\label{sec:BC_trees}

\subsection{Special exponential factors and Stokes-circles-up-to-sign}

Indeed,
the Weyl group of type $B\!C$ involves \emph{signed} permutations,
and so it is necessary to study Stokes circles up to signs of exponential factors.

\subsubsection{}

Let $q$ be an $r$-ramified exponential factor,
and consider the set
\begin{equation}
	\label{eq:stokes_circle_up_to_sign}
	\braket{ \pm q }
	\ceqq
	\braket q \cup \braket{-q}
	= \set{ q^{(0)},\dc,q^{(r-1)},-q^{(0)},\dc,-q^{(r-1)} }.
\end{equation}
It consists of the Galois conjugates of $q$,
and their opposites.
If $I = \braket q$ is a Stokes circle,
write
\begin{equation}
	-I
	\ceqq \braket{-q},
	\qquad \pm I
	\ceqq
	I \cup (-I)
	= \braket{\pm q},
\end{equation}
and set also
\begin{equation}
	\label{eq:ramification_stokes_circles_up_to_sign}
	\on{ram}(\pm I)
	\ceqq \on{ram}(q)
	= \on{ram}(-q),
	\qquad
	E(\pm I)
	\ceqq E(q)
	= E(-q),
\end{equation}
and
\begin{equation}
	\on{irr}(\pm I)
	\ceqq \on{irr}(q)
	= \on{irr}(-q),
	\qquad \on{slope}(\pm I)
	\ceqq \on{slope}(q)
	= \on{slope}(-q).
\end{equation}

Unlike the usual Galois orbits,
the cardinality of $\pm I$ depends on all the active exponents $k \in E(\pm I)$,
as the ramification only determines it up to a possible factor of $2$.
Precisely:

\begin{lemm}
	\label{lem:special_exponential_factors}

	Let $q$ be a nonzero $r$-ramified exponential factor with Stokes circle $I \ceqq \braket q$.
	Then $\abs{\pm I} \in \set{r,2r}$.
\end{lemm}

\begin{proof}
	The \emph{list} of Galois conjugates of $q$,
	and their opposites,
	is
	\begin{equation}
		\label{eq:galois_list}
		(q^{(0)}, \dc, q^{(r-1)}, -q^{(0)}, \dc, -q^{(r-1)}),
	\end{equation}
	and there are two (mutually-exclusive) situations:
	\begin{enumerate}
		\item
		      either there exists $i \in \mb Z \bs r \mb Z$,
		      $i \neq 0$,
		      such that $q^{(i)} + q^{(0)} = 0$,
		      whence
		      \begin{equation}
			      q^{(i + j)}
			      = -q^{(j)},
			      \qquad j \in \mb Z \slash r \mb Z,
		      \end{equation}
		      and~\eqref{eq:galois_list} has $r$ distinct elements---%
		      only;

		\item
		      otherwise the elements of~\eqref{eq:galois_list} are pairwise distinct. \qedhere
	\end{enumerate}
\end{proof}

\begin{defi}
	\label{def:special_expo_factor}

	Let $q$ be a \emph{nonzero} $r$-ramified exponential factor.
	Then:
	\begin{enumerate}
		\item
		      we say that $q$,
		      or (equivalently) $I = \braket q$,
		      is \emph{special},
		      if $\abs{\pm I} = r$;

		\item
		      else,
		      $q$ and $I$ are \emph{nonspecial}.
	\end{enumerate}
	The subset of special (resp.,
	nonspecial) Stokes circles is denoted by $\mc S^{\on s} \sse \mc S \sm \set{\braket 0}$ (resp.,
	by $\mc S^{\on{ns}}$).
\end{defi}

\subsubsection{}

Denote by $\mc S \bs \mb Z^{\ts}$ the quotient set of \emph{Stokes-circles-up-to-sign},
where one identifies $I \in \mc S$ with the `opposite' circle $-I$.
(We view~\eqref{eq:stokes_circle_up_to_sign} as an element of this quotient.)
The restricted canonical projection $\mc S^{\on s} \to \mc S^{\on s} \bs \mb Z^{\ts}$ is a bijection,
while $\mc S^{\on{ns}} \to \mc S^{\on{ns}} \bs \mb Z^{\ts}$ is $2$-to-$1$.
The elements of $\mc S^{\on s} \bs \mb Z^{\ts} \sse \mc S \bs \mb Z^{\ts}$ are \emph{special} Stokes-circles-up-to-sign,
while those of $\mc S^{\on{ns}} \bs \mb Z^{\ts}$ are \emph{nonspecial} ones.

With the aim of characterizing `specialness',
we introduce the following:

\begin{defi}
	\label{def:special_number_sequence}

	Let $\bm d = (d_i)_i$ (resp.,
	$\bm E = (k_i)_i$) be a finite nonempty sequence of positive integers (resp.,
	of positive rational numbers).
	Then:
	\begin{enumerate}
		\item
		      the sequence $\bm d$ is \emph{special} if:
		      \begin{enumerate}
			      \item
			            the integer $d_i$ is even for all $i$;

			      \item
			            and $2m \slash d_i$ is odd for all $i$,
			            where $m \ceqq \bigvee_i (d_i \slash 2)$;
		      \end{enumerate}

		\item
		      the sequence $\bm E$ is \emph{special} if this holds for its sequence of denominators;

		\item
		      else,
		      $\bm d$ and $\bm E$ are \emph{nonspecial}.
	\end{enumerate}
\end{defi}

\begin{prop}
	Let $q$ be a \emph{nonzero} $r$-ramified exponential factor,
	with Stokes-circle-up-to-sign $\pm I = \braket{\pm q}$.
	Then $q$,
	or (equivalently) $\pm I$, is special (as per Def.~\ref{def:special_expo_factor}) if and only if the set $E(\pm I) = E(q)$ is special (as per Def.~\ref{def:special_number_sequence}).
\end{prop}

\begin{proof}
	Set $\bm E \ceqq E(q) = ( k_1,\dc,k_p)$,
	so that $k_1 > \dm > k_p \in \mb Q_{> 0}$ for a suitable integer $p \geq 1$,
	and $q = \sum_{i = 1}^p a_i z^{-k_i}$.
	Then define
	\begin{equation}
		d_i
		\ceqq \on{den}(k_i),
		\quad b_i
		\ceqq \on{num}(k_i),
		\qquad i \in \set{1,\dc,p}.
	\end{equation}

	If $q = q^{(0)}$ is special then there exists $j \in \mb Z \bs r \mb Z$,
	$j \neq 0$,
	such that $q^{(0)} + q^{(j)} = 0$;
	i.e.,
	one has:
	\begin{equation}
		\sum_{i = 1}^p a_i \bigl( 1 + \wt\zeta_i^{\, j} \bigr) z^{-b_i/d_i}
		= 0,
		\qquad \wt\zeta_i
		\ceqq \zeta_{d_i}^{b_i}
		= \exp \Bigl(2\pi \sqrt{-1} \cdot \frac{b_i}{d_i} \Bigr) \in \mb C^{\ts}.
	\end{equation}
	(Cf.~\S~\ref{sec:notations}.)
	Since $\wt\zeta_i$ is a primitive $(d_i)$-th root of $1$,
	the identity $1+ \wt \zeta_i^{\, j} = 0$ implies that $d_i$ is even,
	and that $j \equiv \frac{d_i}2 \pmod{d_i}$,
	so that $j$ is an odd multiple of $\frac{d_i}2$.
	Thus,
	$j$ is a multiple of $r' \ceqq \bigvee_{i = 1}^p (d_i \slash 2)$,
	and $r'$ is necessarily an odd multiple of $\frac{d_i}2$ for $i \in \set{1,\dc,p}$.
	It follows that the sequence $(d_1, \dc, d_p)$ is special,
	and that $j$ is an odd multiple of $r'$;
	in other words,
	one has $j \equiv \frac r 2 \pmod r$,
	because $r = \bigvee_{i = 1}^p (d_i) = 2 r'$.

	Conversely,
	assume that $\bm E$ is special,
	i.e.,
	that the sequence $\bm d = (d_1, \dc, d_p)$ is special.
	Then $r = \bigvee \bm d$ is even,
	and the integer $j \ceqq \frac r 2$ is an odd multiple of $\frac{d_i}2$ for $i \in \set{1, \dc, p}$:
	the previous computations now yield $q^{(0)} + q^{(j)} = 0$.
\end{proof}

\subsection{Full \texorpdfstring{$B\!C$}{BC}-irregular types}

Recall that~\cite{boalch_doucot_rembado_2025_twisted_local_wild_mapping_class_groups_configuration_spaces_fission_trees_and_complex_braids} gave a definition of `full' irregular types in type $A$,
involving Galois-closed lists of exponential factors.
We now extend it in type $B\!C$,
rewriting a particular case of Def.~\ref{def:galois_closed_irr_types},
as follows:

\begin{defi}
	\label{def:full_irr_type_bc}

	Let $m \geq 1$ be an integer.
	A \emph{rank}-$m$ \emph{full irregular type of type} $B\!C$ is a $W_{B\!C}(m)$-Galois-closed list of exponential factors of the form
	\begin{equation}
		\label{eq:full_irr_type_BC}
		Q
		= (q_1, \dc, q_m, -q_1, \dc, -q_m),
	\end{equation}
	where $q_1,\dc,q_m$ are as in~\eqref{eq:exponential_factor}.
	I.e.,
	if $\sigma$ (still) denotes the monodromy of the exponential local system,
	then there exists a group element $g \in W_{B\!C}(m)$ such that
	\begin{equation}
		\label{eq:galois_closed_list_BC}
		g(Q)
		= \sigma(Q)
		\ceqq \bigl( \sigma(q_1),\dc,\sigma(q_m),\sigma(-q_1),\dc,\sigma(-q_m) \bigr).
	\end{equation}
	(For short,
	we say that~\eqref{eq:full_irr_type_BC} is a \emph{full} $B\!C_m$-\emph{irregular type}.)
\end{defi}

\subsubsection{}

Explicitly,
in the identification~\eqref{eq:weyl_type_BC},
the identity~\eqref{eq:galois_closed_list_BC} means that there exists a permutation $\rho \in \mf S_m$,
and a tuple of signs $\bm\varepsilon = (\varepsilon_1, \dc, \varepsilon_m) \in \smash{\bigl( \mb Z^{\ts} \bigr)}^{\!m}$,
such that
\begin{equation}
	\label{eq:monodromy_realization_BC}
	\sigma(q_i) =
	\varepsilon_i q_{\rho(i)},
	\qquad i \in \set{1,\dc,m}.
\end{equation}
In the terminology of \S~\ref{sec:setup},
the element $g \ceqq \rho \cdot \bm\varepsilon \in W_{B\!C}(m)$ generates the $r$-Galois-orbit of the untwisted pullback $\wh Q$ of $Q$,
for a suitable integer $r \geq 1$ (cf.~\eqref{eq:total_ramification});
thus,
if~\eqref{eq:galois_closed_list_BC} holds,
we will also say that $g$ \emph{generates the Galois-orbit of} $Q$.

Now an irregular class can be (re)defined as an equivalence class of full irregular types,
in a particular case of \S~\ref{sec:setup}:

\begin{defi}
	\label{def:bc_equivalence}

	Let $m,m' \geq 1$ be integers,
	and let $Q$ (resp.,
	$Q'$) be a full $B\!C$-irregular type of rank $m$ (resp.,
	of rank $m'$).
	Then $Q$ and $Q'$ are $B\!C$-\emph{equivalent} if:
	\begin{enumerate}
		\item
		      one has $m = m '$;

		\item
		      and there exists $g \in W_{B\!C}(m)$ such that $Q' = g(Q)$.
	\end{enumerate}
\end{defi}

\subsubsection{}

We (still) write $\Theta = \Theta(Q)$ for the $B\!C$-\emph{irregular class of} $Q$,
i.e.,
its equivalence class for the above relation.

Now recall also from~\cite[Prop.~8]{boalch_yamakawa_twisted_wild_character_varieties} (cf.~\cite{boalch_doucot_rembado_2025_twisted_local_wild_mapping_class_groups_configuration_spaces_fission_trees_and_complex_braids}) that a type-$A$ irregular class is equivalent to the data of a finite multiset of Stokes circles $I_i$.
This can be encoded in a linear combination
\begin{equation}
	\Theta
	= \sum_{i = 1}^p n_i \cdot I_i,
	\qquad n_i \in \mb Z_{> 0},
\end{equation}
for some integer $p \geq 1$,
where $I_i \neq I_j$ for $i \neq j \in \set{1,\dc,p}$---%
so that the rank is the integer $m \ceqq \sum_{i = 1}^p n_i \on{ram}(I_i)$.
Then we obtain an analogous explicit description of irregular classes in type $B\!C$:

\begin{prop}
	\label{prop:form_type_BC_irregular_class}

	A $B\!C_m$-irregular class is (equivalent to) the data of a linear combination
	\begin{equation}
		\label{eq:type_BC_irr_class}
		\Theta
		= \sum_{i = 1}^p n_i \cdot (\pm I_i),
		\qquad n_i \in \mb Q_{> 0},
	\end{equation}
	for some integer $p \geq 1$,
	where $\pm I_i$ are Stokes-circles-up-to-sign,
	such that:
	\begin{enumerate}
		\item
		      one has $\pm I_i \neq \pm I_j$,
		      for $i \neq j \in \set{1,\dc,p}$;

		\item
		      one has $n_i \in \mb Z_{> 0}$ if $\pm I_i$ is nonspecial or tame;

		\item
		      one has $n_i \in \frac 1 2 \mb Z_{> 0}$ if $\pm I_i$ is special;

		\item
		      and one has $m = \sum_{i = 1}^p n_i \cdot \on{ram}(\pm I_i)$.
	\end{enumerate}
\end{prop}

\begin{proof}
	First,
	let $Q$ be a full $B\!C_m$-irregular type,
	and $I = \braket q \in \mc S$ the Stokes circle of an exponential factor $q$:
	denote by $m_Q(I) \in \mb Z_{\geq 0}$ be the number of occurrences of $q$ in $Q$,
	so that $I$ is \emph{active} if $m_Q(I) > 0$.
	(By Galois-closedness,
	this does \emph{not} depend on the representative of $I$.)
	Then $m_Q(I) = m_Q(-I)$,
	and one can define a function
	\begin{equation}
		\theta_Q \cl \mc S \bs \mb Z^{\ts} \lra \mb Z_{\geq 0},
		\qquad (\pm I)
		\lmt m_Q(\pm I).
	\end{equation}
	By Galois-closedness,
	the map $\theta_Q$ only depends on the irregular class $\Theta = \Theta(Q)$,
	and conversely $\Theta$ is determined by $\theta_Q$.

	Second,
	to relate the rank $m$ with the multiplicity function $m_Q$,
	let $\pm I = \braket{\pm q}$ be an \emph{active} Stokes-circle-up-to-sign.
	If $I = \braket 0$ is tame,
	then $m_Q(I)$ is even.
	If instead $\pm I \neq \braket 0$,
	then there are two (mutually-exclusive) situations:
	\begin{enumerate}
		\item
		      if $I$ is nonspecial,
		      by Galois-closedness,
		      the number of elements in $Q$ that belong to $\pm I$ equals $2\on{ram}(I) \cdot m_Q(I)$;

		\item
		      conversely,
		      if $I$ is special,
		      the number of elements in $Q$ that belong to $\pm I$ is $\on{ram}(I) \cdot m_Q(I)$.
	\end{enumerate}

	It follows that
	\begin{equation}
		2m
		= m_Q \bigl( \braket 0 \bigr) + \sum_{\mc S^{\on{ns}} \slash \mb Z^{\ts}} 2m_Q(\pm I) \cdot \on{ram}(\pm I) + \sum_{\mc S^{\on s}\slash \mb Z^{\ts}} m_Q(\pm I) \cdot \on{ram}(\pm I).
	\end{equation}
	One can rewrite this identity in symmetric fashion,
	by defining
	\begin{equation}
		n_Q(I)
		\ceqq
		\begin{cases}
			m_Q \bigl( \braket 0 \bigr) \slash 2, & \quad I = \braket 0,         \\
			m_Q(I),                               & \quad I \in \mc S^{\on{ns}}, \\
			m_Q(I) \slash 2,                      & \quad I \in \mc S^{\on s}.
		\end{cases}
	\end{equation}
	Now $n_Q(I) \in \mb Z_{\geq 0}$ if $I \in \mc S^{\on{ns}} \cup \set{ \braket{0} }$,
	and $n_Q(I) \in \frac 1 2\mb Z_{\geq 0}$ if $I \in \mc S^{\on s}$,
	and finally
	\begin{equation}
		m
		= \sum_{\mc S \slash \mb Z^{\ts}} n_Q(\pm I) \cdot \on{ram}(\pm I). \qedhere
	\end{equation}
\end{proof}

\begin{rema}
	\label{rmk:convention_integer_multiplicities_BC}

	It will at times be helpful to also use integer multiplicities for special Stokes circles.
	Concretely,
	if $\Theta$ is as in~\eqref{eq:type_BC_irr_class},
	then one might instead consider the $\mb Z$-linear combination $\sum_{i = 1}^p m_i \cdot (\pm I_i)$,
	where:
	(i) $m_i \ceqq 2n_i$, if $\pm I_i$ is special;
	and (ii) $m_i \ceqq n_i$,
	otherwise.
	Such a combination will also be denoted by $\Theta$,
	and the $m_i$ will \emph{consistently} denote integer multiplicities.
	The same convention applies to other objects both in this section and in \S~\ref{sec:D_trees}.

	(Beware however that the rank need \emph{not} equal $\sum_{i = 1}^p m_i \cdot \on{ram}(\pm I_i)$.)
\end{rema}

\subsection{Pointed irregular types and configuration spaces}
\label{sec:pointed_types_BC}

Up to the Weyl-group action,
it is convenient to consider a smaller space of deformation parameters for the irregular types.
This relies on Galois-closedness,
and on the natural cyclic order of Stokes circles,
and it is introduced as follows:

\begin{defi}[cf.~\cite{boalch_doucot_rembado_2025_twisted_local_wild_mapping_class_groups_configuration_spaces_fission_trees_and_complex_braids}, Def.~2.4]
	\label{def:pointed_irreg_type_BC}

	A \emph{pointed} $B\!C_m$-\emph{irregular type} is an ordered list
	\begin{equation}
		\label{eq:pointed_irr_type_BC}
		\dot Q
		= \bigl( (n_1, q_1), \dc, (n_p, q_p) \bigr),
	\end{equation}
	for some integer $p \geq 1$,
	where:
	(i) $n_1,\dc,n_p > 0$ are rational numbers;
	and (ii) $q_1,\dc,q_p$ are exponential factors;
	such that:
	\begin{enumerate}
		\item
		      one has $\braket{\pm q_i} \neq \braket{\pm q_j}$,
		      for $i \neq j \in \set{1,\dc,p}$;

		\item
		      one has $n_i \in \mb Z$ if $q_i$ is nonspecial,
		      or if $q_i = 0$;

		\item
		      one has $n_i \in \frac 1 2 \mb Z$ if $q_i$ is special;

		\item
		      and one has $m = \sum_{i = 1}^p n_i \cdot \on{ram}(q_i)$.
	\end{enumerate}
\end{defi}

(Hereafter,
when no reference to the rank is needed,
we will just speak of `pointed irregular types' and `irregular classes',
all tacitly of type $B\!C$ until \S~\ref{sec:D_trees}.)

\begin{rema}
	\label{rmk:pointed_to_full}

	A pointed irregular type determines a full one,
	of a specific form.
	Namely,
	if $\dot Q$ is as in~\eqref{eq:pointed_irr_type_BC},
	set $r_i \ceqq \on{ram}(q_i)$,
	and let again $q_i^{(j)}$ be the $j$-th Galois conjugate of $q_i = q_i^{(0)}$---%
	for $j \in \mb Z \bs r_i \mb Z$.
	Then the corresponding full irregular type is the concatenation of lists
	\begin{equation}
		\label{eq:pointed_type_as_list}
		Q
		= \bigl( l^\pm_1, \dc, l^\pm_p, -l^\pm_1, \dc, -l^\pm_p \bigr),
	\end{equation}
	where in turn (for $i \in \set{1,\dc,p}$):
	\begin{enumerate}
		\item
		      we set
		      \begin{equation}
			      \label{eq:nonspecial_list}
			      l^+_i
			      \ceqq \bigl( \underbrace{q_i^{(0)}, \dc, q_i^{(0)}}_{n_i \text{ times }}, \dc,\dc, \underbrace{q_i^{(r_i-1)}, \dc, q_i^{(r_i-1)}}_{n_i \text { times}} \bigr),
		      \end{equation}
		      if $q_i$ is nonspecial,
		      or if $q_i=0$;

		\item
		      and we let
		      \begin{equation}
			      \label{eq:special_list}
			      l^-_i
			      \ceqq \bigl( \underbrace{q_i^{(0)}, \dc, q_i^{(0)}}_{m_i \text{ times }}, \dc,\dc, \underbrace{q_i^{(r_i \slash 2 - 1)}, \dc, q_i^{(r_i \slash 2 - 1)}}_{m_i \text { times}} \bigr),
		      \end{equation}
		      if $q_i$ is special,
		      where again $m_i \ceqq 2n_i$ (cf.~Rmk.~\ref{rmk:convention_integer_multiplicities_BC}). \qedhere
	\end{enumerate}
\end{rema}

\begin{rema}
	\label{rmk:decomposition_pit}

	In the notation of~\eqref{eq:nonspecial_list},
	consider the (sub)list of exponential factors $Q^+_i \ceqq (l^+_i,-l^+_i)$.
	Note that
	\begin{equation}
		\sigma (Q^+_i)
		= g^+_i(Q^+_i),
	\end{equation}
	where $g^+_i \in W_{B\!C}(n_i \cdot r_i)$ is the product of the following (pairwise-commuting) \emph{positive} disjoint $r_i$-cycles (cf.~\S~\ref{sec:background_type_BCD}):
	\begin{equation}
		\label{eq:nonspecial_positive_cycles}
		\wt c_{ij}
		= c^+_{ij} c^-_{ij},
		\qquad c^+_{ij}
		\ceqq \bigl( j \mid n_i + j \mid \dm \mid (r_i - 1)n_i + j \bigr),
		\quad j \in \set{1,\dc,n_i}.
	\end{equation}
	In particular,
	$Q^+_i$ is a full irregular type of rank $n_i \cdot r_i \geq 1$,
	and---%
	just as in type $A$---%
	it is helpful to view it as a `sub-irregular type' of~\eqref{eq:pointed_type_as_list}.

	Analogously,
	in the notation of~\eqref{eq:special_list},
	if we set $r_i' \ceqq r_i \slash 2$ and $Q^-_i \ceqq (l_i^-,-l_i^-)$ then
	\begin{equation}
		\sigma(Q^-_i)
		= g^-_i (Q^-_i),
	\end{equation}
	where $g^-_i \in W_{B\!C}(m_i \cdot r_i')$ is the product of the following \emph{negative} disjoint $r_i'$-cycles:
	\begin{equation}
		\label{eq:special_negative_cycles}
		\wt c_{ij}
		= \bigl( j \mid m_i + j \mid \dm \mid (r_i' - 1)m_i + j \mid -j \mid -m_i - j \mid \dm \mid (1 - r_i')m_i - j \bigr),
		\quad j \in \set{1,\dc,m_i}.
	\end{equation}
	Again,
	we find a full irregular type of rank $m_i \cdot r'_i = n_i \cdot r_i \geq 1$.
\end{rema}

\subsubsection{}

Henceforth,
the symbol $\dot Q$ will (abusively) denote both a pointed irregular type and its underlying full irregular type,
as per Rmk.~\ref{rmk:pointed_to_full}.

Now the important observation is that:

\begin{lemm}
	\label{lem:full_gives_pointed_BC}

	Any irregular class $\Theta$ admits a pointed representative $\dot Q$.
\end{lemm}

\begin{proof}
	This follows at once from Prop.~\ref{prop:form_type_BC_irregular_class}:
	cf.~\ref{proof:lem_full_gives_pointed_BC} for an alternative direct proof involving the conjugacy classes of the Weyl group of type $B\!C$.
\end{proof}

\subsubsection{}

We write $\Theta = \Theta(\dot Q)$ for the irregular class of (the full irregular type underlying) $\dot Q$.
Finally,
w.l.o.g.~(under the Weyl-group action),
it will be useful to further restrict focus to a subspace of pointed irregular types:

\begin{defi}
	\label{def:compatible_pit}

	A pointed irregular type~\eqref{eq:pointed_irr_type_BC} is \emph{compatible} if the equality $\braket{\pm \tau_k(q_i)} = \braket{\pm \tau_k(q_j)}$ implies that one has $\tau_k(q_i) = \tau_k(q_j)$:
	for all truncation-heights $k \in \mb Q_{>0}$,
	and for all $i \neq j \in \set{1,\dc,p}$ (cf.~\eqref{eq:truncation}).
\end{defi}

\begin{lemm}
	\label{lem:pointed_gives_compatible}

	Any pointed irregular type $\dot Q$ is $B\!C$-equivalent to a \emph{compatible} one.
\end{lemm}

\begin{proof}[Proof omitted;
		cf.~Fn.~\ref{fn:compatible_pit}]
\end{proof}

\subsubsection{}

Finally,
in view of the explicit descriptions of the root systems of type $B/C$ (cf.~\S~\ref{sec:root_systems_background}),
Deff.~\ref{def:twisted_pure_admissible_deformations} +~\ref{def:twisted_admissible_deformation_full} are equivalent to the following numerical criterion (cf.~Rmk.~\ref{rmk:numerical_vs_admissible}):

\begin{defi}
	\label{def:numerical_equivalence_bc}

	Consider two full irregular types
	\begin{equation}
		Q
		= (q_1,\dc, q_m, -q_1, \dc, -q_m),
		\qquad  Q'
		= (q'_1,\dc, q'_{m'}, -q'_1, \dc, -q'_{m'}),
	\end{equation}
	of ranks $m,m' \geq 1$---%
	respectively.
	Then:
	\begin{enumerate}
		\item
		      $Q$ and $Q'$ are \emph{mutual} $B\!C$-\emph{admissible deformations},
		      which is symbolized by $Q \sim_{B\!C} Q'$,
		      if:
		      \begin{enumerate}
			      \item
			            one has $m = m'$;

			      \item
			            there exists $g \in W_{B\!C}(m)$ such that their Galois-orbits are---%
			            both---%
			            generated by $g$;

			      \item
			            and one has the equalities
			            \begin{equation}
				            \label{eq:numerical_equivalence_BC}
				            \on{slope}(q_i\pm q_j)
				            = \on{slope}(q'_i\pm q'_j),
				            \quad \on{slope}(q_i)
				            = \on{slope}(q'_i),
				            \qquad i \neq j \in \set{1,\dc,m};
			            \end{equation}
		      \end{enumerate}

		\item
		      and two irregular classes $\Theta$ and $\Theta'$ are \emph{mutual} $B\!C$-\emph{admissible deformations},
		      which is symbolized by $\Theta \sim_{B\!C} \Theta'$,
		      if there exist full irregular types $Q$ and $Q'$ such that
		      \begin{enumerate}
			      \item
			            one has $\Theta = \Theta(Q)$ and $\Theta' = \Theta(Q')$;

			      \item
			            and one has $Q \sim_{B\!C} Q'$.
		      \end{enumerate}
	\end{enumerate}
\end{defi}

\begin{rema}
	\label{rmk:pointed_is_enough}

	One can prove that if $Q \sim_{B\!C} Q'$,
	and $Q$ is pointed,
	then so is $Q'$.
	Thus,
	two pointed irregular types $\dot Q$ and $\dot Q'$ are said to be \emph{mutual} $B\!C$-\emph{admissible deformations},
	which is also symbolized by $\dot Q \sim_{B\!C} \dot Q'$,
	if---%
	and only if---%
	this holds for their underlying full irregular types.

	Analogously,
	if $\dot Q \sim_{B\!C} \dot Q'$,
	and $\dot Q$ is compatible,
	then so is $\dot Q'$.
\end{rema}

\subsubsection{}
\label{sec:type_BC_configuration_space}

On the one hand,
restricting to compatible pointed irregular types selects a \emph{proper} subset of the pure admissible-deformation spaces.
On the other hand,
by Lemm.~\ref{lem:full_gives_pointed_BC} +~\ref{lem:pointed_gives_compatible} and Rmk.~\ref{rmk:pointed_is_enough},
doing so yields a complete set of representatives for their homeomorphism classes.
Therefore,
hereafter we tacitly work with \emph{compatible} pointed irregular types.

Now,
analogously to \S~\ref{sec:setup},
we attach a deformation space---%
a.k.a.~`configuration space'~\cite{boalch_doucot_rembado_2025_twisted_local_wild_mapping_class_groups_configuration_spaces_fission_trees_and_complex_braids}---%
to a pointed irregular type $\dot Q = \bigl( (n_1,q_1),\dc,(n_p,q_p) \bigr)$,
and to its irregular class $\Theta \ceqq \Theta(\dot Q)$.
Namely,
let
\begin{equation}
	\label{eq:total_ramification}
	r
	= \on{ram}(\dot Q)
	\ceqq \bigvee_{i = 1}^p \on{ram}(q_i) \in \mb Z_{>0},
\end{equation}
which is the \emph{total ramification of} $\dot Q$,
and define the \emph{Poincaré--Katz rank of} $\dot Q$  as the largest slope
\begin{equation}
	\label{eq:katz_rank}
	K
	=
	\on{Katz}(\dot Q)
	\ceqq \max \Set{ \on{slope}(q_1),\dc,\on{slope}(q_p) } \in \mb Q_{>0}.
\end{equation}
Thus,
the exponential factors appearing in $\dot Q$ can be expressed as principal parts of \emph{Laurent} series in
\begin{equation}
	w
	\ceqq z^{1/r}
	= e^{\frac 1 r \log(z)},
\end{equation}
of irregularity bounded by $s \ceqq rK \geq 1$.
Moreover,
upon such untwisting,
any pointed irregular type with the same number $p$ of exponential factors,
the same multiplicities/ramifications,
and of Poincaré--Katz rank at most $K$,
will be of the form:
\begin{equation}
	\label{eq:naivecoeffs_bc}
	\dot Q_{\bm a}
	= \bigl( (n_1,\wh q_{1,\bm a}), \dc, (n_p,\wh q_{p,\bm a}) \bigr),
	\qquad \wh q_{i,\bm a}
	= \sum_{j = 1}^s a_{ij} w^{-j},
\end{equation}
for a unique collection of numbers $\bm a = (a_{ij})_{ij} \in \mb C^{ps}$.
In particular,
the admissible deformations of $\dot Q$ are of this form,
and so (by Def.~\ref{def:numerical_equivalence_bc}) we make the following:

\begin{defi}[cf.~\cite{boalch_doucot_rembado_2025_twisted_local_wild_mapping_class_groups_configuration_spaces_fission_trees_and_complex_braids},
		Def. 2.10]
	\label{def:config_space_type_BC}

	In the notation of~\eqref{eq:naivecoeffs_bc}:
	\begin{enumerate}
		\item
		      the $B\!C$-\emph{deformation space of} $\dot Q$ is the following topological subspace of $\mb C^{ps}$:
		      \begin{equation}
			      \label{eq:configuration_space_bc}
			      \bm B_{B\!C,r}(\dot Q)
			      = \bm B_{B\!C,r}^{\leq s,p}(\dot Q)
			      \ceqq  \Set{ \bm a \in \mb C^{ps} | \dot Q \sim_{B\!C} \dot Q_{\bm a} };
		      \end{equation}

		\item
		      and the $B\!C$-\emph{deformation space} of $\Theta$ is:
		      \begin{equation}
			      \label{eq:configuration_space_BC_quotient}
			      \bm B_{B\!C,r}(\Theta)
			      = \bm B_{B\!C,r}^{\leq s,p}(\Theta)
			      \ceqq \Set{ \Theta_{\bm a} = \Theta(\dot Q_{\bm a}) | \Theta \sim_{B\!C} \Theta_{\bm a} }.
		      \end{equation}
	\end{enumerate}
\end{defi}

(We still work in strong/classical topology,
and we regard~\eqref{eq:configuration_space_BC_quotient} as a topological quotient of~\eqref{eq:configuration_space_bc}.)

\begin{rema}
	\label{rmk:config_spaces_are_the_same}

	Again,
	these may be referred to as $B\!C$-\emph{configuration spaces},
	and they give a particular case of \S~\ref{sec:setup}.

	Precisely,
	if $\wh Q$ is the $r$-Galois-closed untwisted pullback of (the full irregular type underlying) $\dot Q$,
	then there is a homeomorphism $\bm B_{B\!C,r}(\dot Q) \simeq \bm B_r \bigl( \wh Q \bigr)$ (cf.~Rmkk.~\ref{rmk:exp_factors_as_diag_terms} +~\ref{rmk:pointed_is_enough});
	analogously,
	if $\wh\Theta = \wh\Theta \bigl( \wh Q \bigr)$,
	then $\bm B_{B\!C,r}(\Theta) \simeq \bm B_r\bigl( \wh\Theta \bigr)$.
\end{rema}

\subsubsection{}
\label{sec:about_slopes}

In the remainder of this section we will explicitly describe the topological spaces~\eqref{eq:configuration_space_bc}--\eqref{eq:configuration_space_BC_quotient}.
To this end,
the slopes that appear in~\eqref{eq:numerical_equivalence_BC} can be separated into:
\begin{enumerate}
	\item \emph{interior} slopes,
	      i.e.:
	      \begin{enumerate}
		      \item
		            the slope of an individual exponential factor;

		      \item
		            the slopes of sums/differences of exponential factors belonging to \emph{the same} Stokes-circle-up-to-sign;
	      \end{enumerate}
	\item
	      and \emph{exterior} slopes,
	      i.e.,
	      the slopes of sums/differences of exponential factors belonging to \emph{different} Stokes-circles-up-to-sign.
\end{enumerate}

To study them separately,
in the next sections we will consider pointed irregular types:
(i) having only one Stokes circle (cf.~\S\S~\ref{sec:bc_one_circle}--\ref{sec:bc_one_circle_end}),
leading to $B\!C$-\emph{level data};
(ii) having exactly two Stokes circles (cf.~\S~\ref{sec:bc_two_circles}),
leading to $B\!C$-\emph{fission exponents};
and finally (iii) the general case (cf.~\S\S~\ref{sec:bc_general}--\ref{sec:weyl_group_tree_bc}),
leading to $B\!C$-\emph{fission trees}.

\subsection{One Stokes circle:
	level data}
\label{sec:bc_one_circle}

\subsubsection{}

Let us say that a pointed irregular type is \emph{elementary} if it has exactly one active Stokes circle.

\begin{defi}
	\label{def:bc_level_data}

	Let $\dot Q = (n,q)$ be an elementary pointed irregular type.
	The $B\!C$-\emph{level datum of} $\dot Q$ is the set of nonzero slopes governing the $B\!C$-admissible deformations of $\dot Q$ (cf.~Def.~\ref{def:numerical_equivalence_bc});
	i.e.:
	\begin{equation}
		L_{B\!C}(\dot Q)
		= L_{B\!C}(n,q)
		\ceqq \Set{ \on{slope}(\wt q_i),
			\, \on{slope}(\wt q_i \pm \wt q_j) | i \neq j \in \set{1,\dc,m} } \sm \set{0} \sse \mb Q_{>0},
	\end{equation}
	where $Q = (\wt q_1,\dc,\wt q_m,-\wt q_1,\dc,-\wt q_m)$ is the full irregular type associated with $\dot Q$.~
\end{defi}

\begin{lemm}
	Let $q$ be an $r$-ramified exponential factor,
	and consider \emph{any} allowed multiplicity $n \in \frac 1 2\mb Z_{> 0}$ in the elementary pointed irregular type $\dot Q \ceqq (n,q)$.
	Then
	\begin{equation}
		\label{eq:bc_level_data_one_circle}
		L_{B\!C} (\dot Q)
		= \Bigl( \set{\on{slope}(q)} \cup \Set{ \on{slope} \bigl( q^{(i)} \pm q^{(j)}  \bigr) | i \neq j \in \mb Z \bs r\mb Z } \Bigr) \sm \set{0}.
	\end{equation}
\end{lemm}

\begin{proof}[Proof omitted]
\end{proof}

\subsubsection{}

Clearly the set~\eqref{eq:bc_level_data_one_circle} does \emph{not} depend on the multiplicity $n$.
(This is an instance of the fact that any type-$B\!C$ example can be viewed as a quasigeneric example of lower rank,
cf.~\S~\ref{sec:type_BC}.)

This observation motivates the following:

\begin{defi}
	Let $q$ be an $r$-ramified exponential factor,
	with Stokes-circle-up-to-sign $\pm I = \braket{\pm q }$.
	Then:
	\begin{enumerate}
		\item
		      the $B\!C$-\emph{level datum of} $q$,
		      or (equivalently) \emph{of} $\pm I$,
		      is the set
		      \begin{equation}
			      L_{B\!C}(\pm I)
			      = L_{B\!C}(q)
			      \ceqq \Bigl( \set{ \on{slope}(q) } \cup \Set{ \on{slope} \bigl( q^{(i)} \pm q^{(j)} \bigr) | i \neq j \in \mb Z \bs r\mb Z } \Bigr) \sm \set{0};
		      \end{equation}

		\item
		      conversely,
		      a finite set $\bm L \sse \mb Q_{>0}$ is a $B\!C$-\emph{level datum} if there exists a Stokes-circle-up-to-sign $\pm I \in \mc S \bs \mb Z^{\ts}$ such that $\bm L = L_{B\!C}(\pm I)$;

		\item
		      and the \emph{ramification of} $\bm L$ is then $\on{ram}(\bm L) \ceqq \on{ram}(\pm I)$.
	\end{enumerate}
\end{defi}

\subsubsection{}

It follows that $L_{B\!C}(q) = \vn$ if and only if $q = 0$:
the next aim is to find an explicit description of $B\!C$-level data for nontame Stokes-circles-up-to-sign.
To this end,
we rely on the description of $A$-level data (cf.~\S~\ref{sec:notation_for_trees}).
Indeed,
if $\pm I \neq \braket{0}$ is $r$-ramified,
then by definition
\begin{equation}
	L_{B\!C}(\pm I)
	= \Bigl( \set{ \on{slope}(\pm I) } \cup L_A(I) \cup \Set{ \on{slope} \bigl( q^{(i)} + q^{(j)} \bigr) | i \neq j \in \mb Z \slash r \mb Z } \Bigr) \sm \set{0},
\end{equation}
in the notation of~\eqref{eq:type_a_level_data}.
Thus,
any exponent of $I$ that is an $A$-level of $I$ is also a $B\!C$-level of $\pm I$.

To study the new levels,
one can first understand when the slope of $I$ is \emph{not} already an $A$-level:

\begin{lemm}
	\label{lem:type_BC_levels}

	Let $I \neq \braket 0$ be a Stokes circle,
	and $k \ceqq \on{slope}(I)$.
	Then $k \in L_A(I)$ if and only if $k \not\in \mb Z$.
\end{lemm}

\begin{proof}[Proof postponed to~\ref{proof:lem_type_BC_levels}]
\end{proof}

\subsubsection{}

To understand the rest of the new $B\!C$-levels,
the point is to describe the subset
\begin{equation}
	L_{B\!C}^+(\pm I)
	\ceqq \Bigl( \Set{ \on{slope} \bigl( q^{(i)} + q^{(j)} \bigr) | i \neq j \in \mb Z \slash r\mb Z } \sm \set{0} \Bigr) \sm L_A(I) \sse L_{B\!C}(\pm I),
\end{equation}
where $I \neq \braket 0$ is an $r$-ramified Stokes circle.

\begin{prop}[cf.~\cite{boalch_doucot_rembado_2025_twisted_local_wild_mapping_class_groups_configuration_spaces_fission_trees_and_complex_braids}, Prop.~3.1]
	\label{prop:new_level_data_type_BC}

	Let $I \neq \braket 0$ be a Stokes circle.
	Write $E(I) = ( k_1, \dc, k_p )$,
	with $k_1 >\dc > k_p \in \mb Q_{> 0}$ for an integer $p \geq 1$,
	and let $d_i \ceqq \on{den}(k_i)$ and $b_i \ceqq \on{num}(k_i)$.
	Then:
	\begin{enumerate}
		\item if the sequence $(d_1, \dc, d_p)$ starts with a \emph{nonspecial} subsequence
		      \begin{equation}
			      \bm d' = (d_1,\dc, d_{n-1}, d_n),
		      \end{equation}
		      such that $(d_1, \dc, d_{n-1})$ is special and $d_n$ divides $\bigvee_{i = 1}^{n-1} (d_i)$,
		      then $\bm d'$ is uniquely determined and $L_{B\!C}^+(\pm I) = \set{ k_n }$;

		\item
		      else $L_{B\!C}^+(\pm I) = \vn$.
	\end{enumerate}
\end{prop}

\begin{proof}
	Let $k\in E(I)$ be an exponent of $I$ such that $k \notin L_A(I)$ and $k \neq \on{slope}(I)$.
	Suppose that $k\in L^+_{B\!C}(\pm I)$.
	This means that there exist $i \neq j \in \mb Z \bs r\mb Z$ such that $\on{slope} \bigl( q^{(i)} + q^{(j)} \bigr) = k$.
	By Galois-closedness,
	we can assume that $i = 0$.
	Let then $n > 0$ be the integer such that $k = k_n$.
	Truncating at height $k_{n-1}$ yields
	\begin{equation}
		\tau_{k_{n-1}} \bigl( q^{(0)} + q^{(j)} \bigr)
		= 0.
	\end{equation}
	This means that the truncation $\tau_{k_{n-1}} \bigl( q^{(0)} \bigr)$ is a special exponential factor,
	and so in particular the list $(d_1, \dc, d_{n-1})$ is special:
	$j$ is an odd multiple of $m' \ceqq \bigvee_{i = 1}^{n-1} ( d_i \slash 2 )$,
	or equivalently
	\begin{equation}
		j \equiv \frac m 2 \pmod m,
		\qquad m \ceqq \bigvee_{i = 1}^{n-1} (d_i)
		= 2 m'.
	\end{equation}

	Now consider the coefficients of the exponent $k$.
	First,
	since by hypothesis $k_n$ is \emph{not} an $A$-level of $I$,
	one has $d_n \mid m$.
	In turn,
	the truncation $\tau_{k_{n}} \bigl( q^{(j)} \bigr)$ is uniquely determined:
	since $\on{ram}\bigl( \tau_{k_n}(q) \bigr) = m$,
	and $j \equiv \frac m 2 \pmod m$,
	one has
	\begin{equation}
		\tau_{k_n} \bigl( q^{(j)} \bigr)
		= \sigma^{m/2} \bigl( \tau_{k_{n}}(q) \bigr).
	\end{equation}
	Furthermore,
	the coefficient of exponent $k$ in $q^{(0)} + q^{(j)}$ is nonzero,
	i.e.,
	\begin{equation}
		1 + \wt \zeta_n^{\, j}
		\neq 0,
		\qquad \wt \zeta_n
		\ceqq \zeta_{d_n}^{b_n}
		= \exp \Bigl( 2\pi\sqrt{-1} \cdot \frac{b_n}{d_n} \Bigr) \in \mb C^{\ts}.
	\end{equation}
	This implies that either $d_n$ is odd,
	or $d_n$ is even and $j \not\equiv \frac{d_n}2 \pmod{d_n}$.
	In particular,
	the sequence $(d_1, \dc, d_{n-1}, d_n)$ is \emph{not} special,
	and it follows that there is no other exponent $k' \in E(I)$,
	with $k' \neq k$,
	such that $k' \in L_{B\!C}^+(\pm I)$.
	In summary,
	the above implies that if $L_{B\!C}^+(\pm I)$ is nonempty then there exists a \emph{unique} subsequence $(d_1, \dc, d_n)$ of $(d_1,\dc, d_p)$ such that:
	i) $(d_1, \dc, d_{n-1})$ is special;
	ii) $(d_1, \dc, d_n)$ is \emph{not} special;
	iii) $d_n$ divides $m = \bigvee_{i = 1}^{n-1} (d_i)$;
	and iv) $L_{B\!C}^+(\pm I) = \set{k_n}$.

	But the previous computations also yield the converse.
	Namely,
	if there exists a subsequence $(d_1, \dc, d_n)$ with these properties,
	then,
	taking $j \ceqq \bigvee_{i = 1}^{n-1} (d_i \slash 2)$,
	it follows that $j$ is an odd multiple of $\frac{d_i}2$ for $i \in \set{1,\dc,n-1}$ (since $(d_1,\dc,d_{n-1})$ is special),
	and $\wt \zeta_n^j \neq -1$,
	whence $\on{slope} \bigl( q^{(0)} + q^{(j)} \bigr) = k_n$:
	thus,
	$k_n \in L_{B\!C}^+(\pm I) \neq \vn$.
\end{proof}

\begin{rema}
	In particular,
	if $I = -I$ is special then
	\begin{equation}
		L_{B\!C}^+(\pm I)
		= \vn,
		\qquad L_{B\!C}(\pm I)
		=L_{A}(I). \qedhere
	\end{equation}
\end{rema}

\subsubsection{}

Now one can prove that the level data control admissible deformations:

\begin{prop}
	\label{prop:numerical_equiv_from_level_data}

	Let $q$ and $q'$ be two exponential factors,
	and $\dot Q = (n,q)$ and $\dot Q' = (n',q')$ two elementary pointed irregular types featuring them.
	Then one has $\dot Q \sim_{B\!C} \dot Q'$ if and only if:
	\begin{enumerate}
		\item
		      one has $n = n' \in \frac 1 2 \mb Z_{> 0}$;

		\item
		      and one has $L_{B\!C}(q) = L_{B\!C}(q')$.
	\end{enumerate}
\end{prop}

\begin{proof}
	For the nontrivial implication,
	let $q$,
	$n$,
	and $\dot Q$ be as in the statement,
	with ramification $r \ceqq \on{ram}(q)$ and rank $m \ceqq nr \geq 1$.
	The point is that the `ramified' root-valuation map
	\begin{equation}
		\Phi_{B/C}(m) \lra \mb Q_{\geq 0},
		\qquad \alpha
		\lmt \on{slope} \bigl( \alpha(\dot Q) \bigr),
	\end{equation}
	where $\Phi_{B/C}(m)$ is the root system of type $B_m/C_m$ (cf.~\eqref{eq:root_valuations} +~\eqref{eq:type_b_c_roots}),
	is determined by the data of the multiplicity $n$ and the set $L_{B\!C}(q)$.
	For the differences corresponding to the type-$A$ roots,
	this follows from~\cite{boalch_doucot_rembado_2025_twisted_local_wild_mapping_class_groups_configuration_spaces_fission_trees_and_complex_braids}.
	For the other roots,
	instead,
	this follows from the previous discussion---%
	of $B\!C$-levels that are \emph{not} $A$-levels.
\end{proof}

\subsubsection{}

This description of the $B\!C$-level datum of a Stokes-circle-up-to-sign leads to an explicit characterization of \emph{all} possible $B\!C$-level data,
involving a decomposition which is easy to determine in practice (cf.~Exmp.~\ref{exmp:type_BC_levels}).

To state this precisely,
we introduce the following:

\begin{defi}
	Let $\bm E = ( k_1,\dc,k_p )$ be a finite strictly-decreasing sequence of positive rational numbers,
	for some integer $p \geq 1$.
	Set $d_i \ceqq \on{den}(k_i)$ and $k_{p+1} \ceqq 0$;
	moreover,
	choose a number $k \in \mb Q_{>0}$,
	with denominator $d \ceqq \on{den}(k)$.
	Then:
	\begin{enumerate}
		\item the sequence $\bm E$ has \emph{special beginning} if there exists an integer $n \leq p$ such that the sequence of denominators $(d_1, \dc, d_n)$ is special;

		\item and $k$ is a \emph{good breaking of specialness of} $\bm E$ if there exists an integer $n \leq p$ such that:
		      \begin{enumerate}
			      \item
			            one has $k_n>k>k_{n+1}$;

			      \item
			            the sequence $(d_1,\dc, d_n)$ is special;

			      \item
			            the sequence $(d_1, \dc, d_n,d)$ is nonspecial;

			      \item
			            and one has $d \mid \bigvee_{i = 1}^n (d_i)$.
		      \end{enumerate}
	\end{enumerate}
	We denote by $\beta(\bm E) \sse \mb Q_{>0}$ the (finite) set of good breakings of specialness of $\bm E$.
\end{defi}

\begin{coro}
	Let $\bm L \neq \vn$ be a finite set of positive rational numbers.
	Then $\bm L$ is a $B\!C$-\emph{level datum} if and only if it can be written as a \emph{disjoint} union
	\begin{equation}
		\label{eq:form_BC_level_data}
		\bm L
		= \bm S \cup \bm L_A \cup \bm L^+,
	\end{equation}
	where:
	\begin{enumerate}
		\item
		      the set $\bm L_A$ is a (possibly empty) type-$A$ level datum;

		\item
		      the set $\bm S$ is either empty,
		      or it is a singleton $\bm S = \set{k}$,
		      where $k$ is an integer such that $k > \max(\bm L_A)$;

		\item
		      and the set $\bm L^+$ satisfies the following:
		      \begin{enumerate}
			      \item
			            if $\bm S \neq \vn$,
			            then $\bm L^+ = \vn$;

			      \item
			            if $\bm S = \vn$,
			            and if $\bm L_A$ has a special beginning,
			            then either $\bm L^+ = \vn$ or $\bm L^+ = \set{k}$,
			            where $k \in \beta(\bm L_A)$;

			      \item
			            else,
			            one has $\bm L^+ = \vn$.
		      \end{enumerate}
	\end{enumerate}
\end{coro}

\begin{proof}
	The previous discussion implies that all $B\!C$-level data have this form (cf.~particularly Prop.~\ref{prop:new_level_data_type_BC}).
	Conversely,
	if $\bm L$ satisfies the conditions of the statement,
	then one can prove that $\bm L = L_{B\!C}(\pm I)$ for any Stokes-circle-up-to-sign $\pm I$ such that $E(\pm I) = \bm L$.
\end{proof}

\begin{exem}
	\label{exmp:type_BC_levels}

	Let us illustrate the previous discussion on six examples.
	\begin{enumerate}
		\item
		      If $\pm I$ is an unramified Stokes-circle-up-to-sign of (integer) slope $k$,
		      then $L_{B\!C}(\pm I) = \set{k}$.
		      In this case $L_{B\!C}(\pm I) = \bm S$ and $\bm L_A = \bm L^+ = \vn$.

		\item
		      If $\pm I$ is a ramified Stokes-circle-up-to-sign with a single exponent,
		      i.e.,
		      if $E(I)=\set{k}$ with $k \in \mb Q_{>0} \sm \mb Z_{>0}$,
		      then $L_{B\!C}(\pm I) = \set{k}$.
		      In this case $L_{B\!C}(\pm I) = \bm L_A$ and $\bm S = \bm L^+ = \vn$.

		\item
		      Consider the Stokes-circle-up-to-sign $\pm I = \braket{\pm q}$,
		      with $q \ceqq z^{-3/2} + z^{-1}$.
		      Then $L_{B\!C}(\pm I) = \set{ 3 \slash 2,1 }$,
		      with $\bm S = \vn$,
		      $\bm L_A = \set{ 3 \slash 2 }$,
		      and $\bm L^+ = \set{1}$.
		      Indeed,
		      the sequence of denominators of $\bm E \ceqq \set{ 3 \slash 2, 1 }$ is $(2, 1)$,
		      the subsequence $(2)$ is special,
		      and $1 \in \beta(\set{3 \slash 2})$.
		      More explicitly,
		      the Galois conjugates of $q$ are
		      \begin{equation}
			      q^{(0)}
			      = q
			      = z^{-3/2} + z^{-1},
			      \qquad q^{(1)}
			      = -z^{-3/2} + z^{-1}.
		      \end{equation}
		      Then the identity $q^{(0)} - q^{(1)} = 2z^{3/2}$ yields $\frac 3 2 \in L_A(I)$,
		      and $q^{(0)} + q^{(1)} = 2z^{-1}$ yields $1 \in L_{B\!C}^+(\pm I)$.

		\item
		      Consider the Stokes-circle-up-to-sign $\pm I = \braket{ \pm q}$,
		      with $q \ceqq z^{-3/2} + z^{-1/2}$.
		      Then $L_{B\!C}(\pm I) = \set{ 3 \slash 2 }$,
		      with $\bm S = \vn$,
		      $\bm L_A = L_{B\!C}(\pm I)$,
		      and $\bm L^+ = \vn$.
		      Indeed,
		      the sequence of denominators of $\bm E \ceqq \set{ 3 \slash 2, 1 \slash 2 }$ is $(2,2)$,
		      so it is special;
		      but there is \emph{no} breaking of specialness.

		\item
		      Consider the Stokes-circle-up-to-sign $I = \braket {\pm q}$,
		      with $q \ceqq z^{-3/2} + z^{-5/6} + z^{-1/3}$.
		      Then
		      \begin{equation}
			      L_{B\!C}(\pm I)
			      = \set{ 3 \slash 2,5 \slash 6,1 \slash 3 },
			      \quad \bm S
			      = \vn,
			      \quad \bm L_A
			      = \set{ 3 \slash 2,5 \slash 6 },
			      \quad \bm L^+
			      = \set{ 1 \slash 3 }.
		      \end{equation}
		      Indeed,
		      here $E(I) = \set{ 3 \slash 2, 5 \slash 6, 1 \slash 3 }$,
		      the sequence of denominators is $(2,6,3)$,
		      the subsequence $(2,6)$ is special,
		      the subsequence $(2,6,3)$ is \emph{not} special,
		      and $3$ divides $6 = 2 \vee 6$;
		      so $\frac 1 3 \in \beta \bigl( \set{ 3 \slash 2,5 \slash 6 } \bigr)$.

		\item
		      Consider finally $\pm I = \braket{\pm q}$,
		      with $q \ceqq z^{-3/2} + z^{-5/6} + z^{-2/3} + z^{-1/3}$.
		      Then
		      \begin{equation}
			      L_{B\!C}(\pm I)
			      = \set{ 3 \slash 2, 5 \slash 6, 2 \slash 3 },
			      \quad \bm S
			      = \vn,
			      \quad \bm L_A
			      = \set{ 3 \slash 2, 5 \slash  6 },
			      \quad \bm L^+
			      = \set{ 2 \slash 3 }.
		      \end{equation}

		      One should compare with the previous example:
		      here $\frac 1 3 \notin L_{B\!C}(\pm I)$,
		      although $\frac 1 3 \in \beta(\bm L_A)$.
		      This is because the specialness of $\bm L_A$ is first broken by $\frac 2 3$,
		      and it is only the \emph{highest} good breaking which belongs to the $B\!C$-level datum.
	\end{enumerate}
	Note that the Stokes circles of examples (3.)~and (4.)~have the same $A$-level datum;
	and analogously for examples (5.)~and (6.).
\end{exem}

\subsection{One Stokes circle:
	admissible/inconsequential exponents}

We now associate admissible/inconsequential exponents to a $B\!C$-level datum,
in order to determine the `elementary' configuration spaces,
and define fission trees later on.

\begin{defi}[cf.~again~\cite{boalch_doucot_rembado_2025_twisted_local_wild_mapping_class_groups_configuration_spaces_fission_trees_and_complex_braids}, Def.~3.4]
	\label{def:admissible_inconsequential_bc}

	Let $\bm L$ be a $B\!C$-level datum.
	Then:
	\begin{enumerate}
		\item
		      the set of $B\!C$-\emph{admissible exponents of} $\bm L$ is
		      \begin{align}
			      \on{Adm}_{B\!C}(\bm L)
			       & \ceqq \Set{ k \in \mb Q_{>0} | \text{ there exists } \pm I \in \mc S \bs \mb Z^{\ts} \text{ such that } L_{B\!C}(\pm I) = \bm L \text{ and } k \in E(I) } \\
			       & = \bigcup_{ \pm I \cl L_{B\!C}(\pm I) = \bm L} E(I);
		      \end{align}
		\item
		      the set of $B\!C$-\emph{inconsequential exponents of} $\bm L$ is the complement
		      \begin{equation}
			      \on{Inc}_{B\!C}(\bm L)
			      \ceqq \on{Adm}_{B\!C}(\bm L) \sm \bm L;
		      \end{equation}

		\item
		      and if $\bm L = L_{B\!C}(\pm I)$,
		      with $\pm I = \braket{\pm q}$,
		      then we set
		      \begin{equation}
			      \on{Adm}_{B\!C}(\pm I)
			      =\on{Adm}_{B\!C}(q)
			      \ceqq     \on{Adm}_{B\!C}(\bm L),
			      \qquad
			      \on{Inc}_{B\!C}(\pm I)
			      = \on{Inc}_{B\!C}(q)
			      \ceqq \on{Inc}_{B\!C}(\bm L).
		      \end{equation}
	\end{enumerate}
\end{defi}

\begin{prop}
	\label{prop:admissible_inconsequential_bc}

	Let $\bm L$ be a \emph{nonempty} $B\!C$-level datum,
	decomposed as in~\eqref{eq:form_BC_level_data}.
	Then:
	\begin{enumerate}
		\item
		      if $\bm S \cup \bm L_A$ has no special beginnings (e.g.,
		      if $\bm S \neq \vn$),
		      then one has
		      \begin{equation}
			      \on{Adm}_{B\!C}(\bm L)
			      = \on{Adm}_A(\bm L_A) \cap \Set{ k\in \mb Q_{>0} | k \leq \max(\bm S \cup \bm L_A) };
		      \end{equation}
		\item
		      if $\bm S \cup \bm L_A$ has special beginnings (whence $\bm S = \vn$),
		      and $\bm L^+ = \vn$,
		      then one has
		      \begin{equation}
			      \on{Adm}_{B\!C}(\bm L)
			      = \Set{ k\in \mb Q_{>0} | k \leq \max(\bm L_A), \, k \in \on{Adm}_A(\bm L_A), \text{ and } k \notin \beta(\bm L_A) };
		      \end{equation}
		\item
		      and if $\bm S \cup \bm L_A$ has special beginnings,
		      and $\bm L^+ = \set{k_b}$,
		      then one has
		      \begin{equation}
			      \on{Adm}_{B\!C}(\bm L)
			      = \Set{ k\in \mb Q_{>0} | k \leq \max(\bm L_A), \, k \in \on{Adm}_A(\bm L_A); \text{ if } k > k_b \text{ then } k \notin \beta(\bm L_A) }.
		      \end{equation}
	\end{enumerate}
\end{prop}

\begin{proof}[Proof postponed to~\ref{proof:prop_admissible_inconsequential_bc}]
\end{proof}

\begin{exem}
	Let us showcase the admissible/inconsequential exponents of Exmp.~\ref{exmp:type_BC_levels} (keeping the notation therein).
	\begin{enumerate}
		\item
		      In this case $\on{Adm}_{B\!C}(\pm I) = \set{1, \dc, k}$ and $\on{Inc}_{B\!C}(\pm I) = \set{1, \dc, k-1}$.

		\item
		      Let $d = \on{den}(k)$.
		      Then if $d$ is odd, we have $\on{Adm}_{B\!C}(\pm I) = \set{ l \in \frac 1 d \mb Z_{>0} | l\leq k }$ and $\on{Inc}_{B\!C}(\pm I) = \set{ l \in \frac 1 d \mb Z_{>0} | l < k }$. 
If $d$ is even, we have $\on{Adm}_{B\!C}(\pm I) = \set{ l \in \frac 1 d (2\mb Z_{\geq 0}+1) | l\leq k }$ and $\on{Inc}_{B\!C}(\pm I) = \set{ l \in \frac 1 d (2\mb Z_{\geq 0}+1) | l < k }$.

		\item
		      In this case $\on{Adm}_{B\!C}(\pm I) = \set{ 3 \slash 2,1,1 \slash 2 }$ and $\on{Inc}_{B\!C}(\pm I) = \set{ 1 \slash 2 }$.
		      Note that $\on{Adm}_{B\!C}(\pm I) = \on{Adm}_A(I)\cap [0, 3/2]$,
		      while $\on{Inc}_A(I)\cap [0, 3/2] = \set{ 1, 1 \slash 2 } \neq \on{Inc}_{B\!C}(\pm I)$.

		\item
		      In this case $\on{Adm}_{B\!C}(\pm I) = \set{ 3 \slash 2, 1 \slash 2 }$ and $\on{Inc}_{B\!C}(\pm I) = \set{ 1 \slash 2 }$.
		      Incidentally,
		      note that $\on{Adm}_{B\!C}(\pm I) \ssne \on{Adm}_A(I)\cap [0, 3/2] = \set{ 3 \slash 2,1,1 \slash 2}$,
		      because $1 \in \beta \bigl( \set{3 \slash 2} \bigr)$.
		      (In type $A$ this exponent would be admissible.)

		\item
		      In this case $\on{Adm}_{B\!C}(\pm I) = \set{ 3 \slash 2, 5 \slash 6, 1 \slash 2, 1 \slash 3, 1 \slash 6 }$ and $\on{Inc}_{B\!C}(\pm I) = \set{ 1 \slash 2, 1 \slash 6 }$.
		      Indeed,
		      one has $\on{Adm}_A(I)\cap [0, 3/2] = \on{Adm}_{B\!C}(\pm I) \cup \set{1,2 \slash 3}$,
		      but $ 1 \notin \on{Adm}_{B\!C}(\pm I)$ because $1 \in \beta \bigl( \set{3 \slash 2} \bigr)$ is greater than the good breaking $\frac 1 3$;
		      analogously,
		      $\frac 2 3 \notin \on{Adm}_{B\!C}(\pm I)$ because $\frac 2 3 \in \beta \bigl( \set{3 \slash 2,5 \slash 6} \bigr)$ is greater than $\frac 1 3$.
		      The $A$-admissible exponent $\frac 1 2$ is also greater than $\frac 1 3$,
		      but $\frac 1 2 \not\in \beta \bigl( \set{ 3 \slash 2,5 \slash 6 } \bigr)$ (the sequence of denominators $(2,6,2)$ is special),
		      so it is also $B\!C$-admissible.
		      Finally,
		      the exponent $\frac 1 6$ is smaller than $\frac 1 3$,
		      so it is $B\!C$-admissible.

		\item
		      Analogously,
		      in this case $\on{Adm}_{B\!C}(\pm I) = \set{ 3 \slash 2,5 \slash 6, 2 \slash 3, 1 \slash 2, 1 \slash 3, 1 \slash 6 }$ and $\on{Inc}_{B\!C}(\pm I) = \set{1 \slash 2, 1 \slash 3, 1 \slash 6 }$. \qedhere
	\end{enumerate}
\end{exem}

\subsection{One Stokes circle:
	configuration spaces}

In particular,
one can now explicitly describe the `elementary' $B\!C$-deformation spaces:

\begin{coro}[cf.~\cite{boalch_doucot_rembado_2025_twisted_local_wild_mapping_class_groups_configuration_spaces_fission_trees_and_complex_braids}, Prop.~3.8]
	\label{cor:elementary_BC_space}

	Let $\dot Q = (n,q)$ be an elementary pointed irregular type,
	and set $\bm L \ceqq L_{B\!C}(q)$.
	Then there is a \emph{homeomorphism}
	\begin{equation}
		\bm B_{B\!C,r}(\dot Q)
		\simeq (\mb C^{\ast})^{\, \abs{\bm L}} \ts \mb C^{\, \abs{\mb I}},
		\qquad \mb I
		\ceqq \on{Inc}_{B\!C}(\bm L),
	\end{equation}
	equipping the right-hand-side with the product topology.
\end{coro}

\begin{proof}
	Let $\mb A = \mb A(\bm L) \ceqq \on{Adm}_{B\!C}(\bm L)$,
	and identify $\mb C^{\mb A} \simeq \mb C^{\, \abs{\mb A }}$ via the total order of $\mb A \sse \mb Q_{>0}$.
	Then assume that $(n,q) \sim_{B\!C} (n',q')$,
	so that necessarily $n = n'$,
	and set $r \ceqq \on{ram}(q)$ and $s \ceqq \on{irr}(q)$.
	Finally,
	write $q'$ as in~\eqref{eq:exponential_factor},
	for suitable coefficients $\bm a' = (a'_i)_i \in \mb C^s$.
	By definition of admissible exponents,
	if the number $i \slash r$ does \emph{not} lie in $\mb A = \on{Adm}_{B\!C}(q')$,
	then $a'_i = 0$.
	This yields a (sub)collection $(a'_k)_k \in \mb C^{\, \abs{\mb A}}$ such that $q' = \sum_k a'_k z ^{-k}$.
	Moreover,
	if $k \in \bm L$,
	then $a'_k\neq 0$ (otherwise $k$ would \emph{not} be a level),
	so that $\bm a'$ lies in the space of the statement.

	Conversely,
	choose a tuple $\bm a' = (a'_k)_k \in (\mb C^\ast)^{\abs{\bm L}} \ts \mb C^{\, \abs{\mb I}} \sse \mb C^{\, \abs{\mb A}}$.
	Then the exponential factor
	\begin{equation}
		q_{\bm a'}
		\ceqq \sum_{\mb A} a'_k z^{-k},
	\end{equation}
	whose irregularity is bounded by $s$,
	satisfies $L_{B\!C}(q_{\bm a'}) = \bm L$.
\end{proof}

\begin{rema}
	As in type $A$,
	inconsequential exponents do \emph{not} contribute to the homotopy class of the deformation space---%
	whence their name.
	Contrary to type $A$,
	however,
	there is no need to truncate (cf.~Rmk.~\ref{rmk:truncations_trees_BC}).
\end{rema}

\subsection{One Stokes circle:
	full branches}
\label{sec:bc_one_circle_end}

As a first step towards defining twisted fission trees in type $B\!C$,
it will be convenient to picture the above data in graphical way.
To this end,
we make the following:

\begin{defi}
	\label{def:full_branch_of_level_datum_type_BC}

	Let $\bm L$ be a $B\!C$-level datum.
	The \emph{full} $B\!C$-\emph{branch of} $\bm L$ is the triple $\mc B_{\bm L} = (\mc B, \mb A, \mb L)$,
	where:
	\begin{enumerate}
		\item
		      $\mc B$ is a totally-ordered set,
		      equipped with an order-preserving bijective \emph{height function} $h \cl \mc B \lxra{\simeq} \ol{\mb R}_{\geq 0} \ceqq \mb R_{\geq 0} \cup \set{+\infty}$;

		\item
		      and $\mb A, \mb L \sse \mc B$ are two finite subsets,
		      defined as follows:
		      \begin{enumerate}
			      \item
			            the set of \emph{admissible vertices} is $\mb A \ceqq h^{-1}\bigl( \on{Adm}_{B\!C}(\bm L) \bigr)$;

			      \item
			            and the set of \emph{mandatory vertices} is $\mb L \ceqq h^{-1}(\bm L) \sse \mb A$.
		      \end{enumerate}
	\end{enumerate}
\end{defi}

\subsubsection{}

In the notation of Def.~\ref{def:full_branch_of_level_datum_type_BC},
the set of \emph{inconsequential vertices} is the complement $\mb I \ceqq \mb A \sm \mb L \sse \mb A$.
The element $v_\infty \ceqq h^{-1}(+\infty) \in \mc B$ is the \emph{root} of the full branch,
and $v_0 \ceqq h^{-1}(0)$ is the \emph{leaf}.

We will regard a full branch as a tree with no branching,
and so we now equip it with a set of \emph{edges}:

\begin{defi}
	\label{def:edges_BC_full_branch}

	Let $\mc B_{\bm L}$ be the full $B\!C$-branch of a $B\!C$-level datum $\bm L$.
	Choose a pair $\set{k,l}$ with $0\leq k < l$, such that: either $k,l\in h(\mb A)$ are consecutive heights of admissible vertices of $\mc B_{\bm L}$; or $k=0$ and $l=\min \bigl( h(\mb A) \bigr)$.
	Then:
	\begin{enumerate}
		\item
		      the corresponding \emph{finite edge of} $\mc B_{\bm L}$ is the subset
		      \begin{equation}
			      e_{k,l}
			      \ceqq  h^{-1} \bigl( (k,l) \bigr) \sse \mc B;
		      \end{equation}

		\item
		      the number $l \eqqc h^+(e_{k,l}) \in \mb Q_{> 0}$ is the \emph{parent-height of} $e_{k,l}$;

		\item
		      and the (unique) \emph{infinite edge} is
		      \begin{equation}
			      e_\infty
			      = e_{\ul k,\infty}
			      = h^{-1} \bigl( (\ul k,+\infty) \bigr) \sse \mc B,
			      \qquad \ul k
			      \ceqq \max \bigl( h(\mb A) \bigr) < +\infty.
		      \end{equation}
	\end{enumerate}
\end{defi}

\subsubsection{}

We denote by $\mb E$ the resulting set of edges of $\mc B_{\bm L}$.
Finally,
we decorate the edges according to the specialness---%
or lack thereof---%
of the top part of the full branch,
just above them.
Precisely,
consider the set $\set{E,S,NS}$,
whose elements stand for `\emph{empty}',
`\emph{special}',
and `\emph{nonspecial}',
respectively.
Then:

\begin{defi}[cf.~Def.~\ref{def:edges_BC_full_branch}]
	\label{def:type_of_edges_type_BC}

	The \emph{type} $t(e) \in \set{E,S,NS}$ of an edge $e \in \mb E$ is:
	\begin{equation}
		t(e)
		\ceqq
		\begin{cases}
			E,  & \quad e = e_\infty,                                                                   \\
			S,  & \quad \bm L \cap \bigl[ h^+(e), +\infty \bigr) \sse \mb Q_{> 0} \text{ is special},   \\
			NS, & \quad \bm L \cap \bigl[h^+(e), +\infty \bigr) \sse \mb Q_{> 0} \text{ is nonspecial},
		\end{cases}
	\end{equation}
	as per Def.~\ref{def:special_number_sequence}.
\end{defi}

\subsubsection{}

We will draw a full branch with the same conventions of~\cite{boalch_doucot_rembado_2025_twisted_local_wild_mapping_class_groups_configuration_spaces_fission_trees_and_complex_braids},
adding decorations for the edge-type.
Thus,
the mandatory vertices (resp.,
inconsequential ones) are painted black (resp.,
white),
and the root is a square node.
In addition,
empty edges (resp.,
special;
resp.,
nonspecial) are depicted by a dotted line (resp.,
a dashed line;
resp.,
a full line).

Given a vertex $v$ in a full $B\!C$-branch $\mc B_{\bm L}$,
all combinations for the type of $v$,
and for the edges incident at $v$,
are shown in Fig.~\ref{fig:bc_edge_type}.

\begin{figure}
	\begin{center}
		\begin{tikzpicture}
			\tikzstyle{authorized}=[circle,fill=white,minimum size=5pt,draw, inner sep=0pt]
			\tikzstyle{mandatory}=[circle,fill=black,minimum size=5pt,draw,inner sep=0pt]
			\begin{scope}
				\node (A) at (0,2){};
				\node[mandatory] (B) at (0,1){};
				\node (C) at (0,0){};
				\draw[dotted, thick] (A)--(B);
				\draw[dashed] (B)--(C);
			\end{scope}
			\begin{scope}[xshift=1.2cm]
				\node (A) at (0,2){};
				\node[mandatory] (B) at (0,1){};
				\node (C) at (0,0){};
				\draw[dotted, thick] (A)--(B);
				\draw (B)--(C);
			\end{scope}
			\begin{scope}[xshift=2.4cm]
				\node (A) at (0,2){};
				\node[authorized] (B) at (0,1){};
				\node (C) at (0,0){};
				\draw [dashed] (A)--(B);
				\draw[dashed] (B)--(C);
			\end{scope}
			\begin{scope}[xshift=3.6cm]
				\node (A) at (0,2){};
				\node[mandatory] (B) at (0,1){};
				\node (C) at (0,0){};
				\draw [dashed] (A)--(B);
				\draw[dashed] (B)--(C);
			\end{scope}
			\begin{scope}[xshift=4.8cm]
				\node (A) at (0,2){};
				\node[mandatory] (B) at (0,1){};
				\node (C) at (0,0){};
				\draw [dashed] (A)--(B);
				\draw (B)--(C);
			\end{scope}
			\begin{scope}[xshift=6.0cm]
				\node (A) at (0,2){};
				\node[authorized] (B) at (0,1){};
				\node (C) at (0,0){};
				\draw (A)--(B);
				\draw (B)--(C);
			\end{scope}
			\begin{scope}[xshift=7.2cm]
				\node (A) at (0,2){};
				\node[mandatory] (B) at (0,1){};
				\node (C) at (0,0){};
				\draw (A)--(B);
				\draw (B)--(C);
			\end{scope}
		\end{tikzpicture}
	\end{center}
	\caption{Combinations of vertex- and edge-type in a full $B\!C$-branch.}
	\label{fig:bc_edge_type}
\end{figure}

\begin{exem}
	Consider the Stokes-circle-up-to-sign $\pm I = \braket{ \pm q}$,
	with $q = z^{-3/2} + z^{-1}$,
	as in Exmp.~\ref{exmp:type_BC_levels}~(3.).
	Recall that $L_{B\!C}(\pm I) = \set{ 3 \slash 2, 1 }$ and $\on{Inc}_{B\!C}(\pm I) = \set{ 1 \slash 2 }$,
	while $L_A(I) = \set{ 3 \slash 2 }$ and $\on{Inc}_A(I) = \set{1 \slash 2, 1, 2, 3,\dc}$.
	The corresponding type-$A$ and -$B\!C$ full branches are depicted in Fig.~\ref{fig:comparison_branches}:
	as expected,
	the sets of admissible/mandatory/inconsequential vertices are different.
	(But mandatory vertices in type $A$ are mandatory in type $B\!C$.)
\end{exem}

\begin{figure}
	\begin{center}
		\begin{tikzpicture}
			\tikzstyle{authorized}=[circle,fill=white,minimum size=5pt,draw, inner sep=0pt]
			\tikzstyle{mandatory}=[circle,fill=black,minimum size=5pt,draw,inner sep=0pt]
			\tikzstyle{empty}=[circle,fill=black,minimum size=0pt,inner sep=0pt]
			\tikzstyle{root}=[fill=black,minimum size=5pt,draw,inner sep=0pt]
			\tikzstyle{indeterminate}=[circle,densely dotted,fill=white,minimum size=5pt,draw, inner sep=0pt]
			\begin{scope}
				\node[root] (R) at (0,6.5){};
				\node (X) at (0,5.5){$\vdots$};
				\node (Y) at (0,5){};
				\node[mandatory] (A) at (0,4){};
				\node[mandatory] (B) at (0,3){};
				\node[authorized] (C) at (0,2){};
				\node (D) at (0,1){};
				\draw[thick, dotted] (R)--(X);
				\draw[thick, dotted] (Y)--(A);
				\draw[dashed] (A)--(B);
				\draw (B)--(C);
				\draw (C)--(D);
				\draw (-2, 6.5) node {$+ \infty$};
				\draw (-2, 5) node {$2$};
				\draw (-2, 4) node {$\frac 3 2$};
				\draw (-2, 3) node {$1$};
				\draw (-2, 2) node {$\frac 1 2$};
				\draw (0, 0.5) node {full $B\!C$-branch of $\pm I$};
			\end{scope}
			\begin{scope}[xshift=6cm]
				\node[root] (R) at (0,6.5){};
				\node (X) at (0,5.5){$\vdots$};
				\node[authorized] (Y) at (0,5){};
				\node[mandatory] (A) at (0,4){};
				\node[authorized] (B) at (0,3){};
				\node[authorized] (C) at (0,2){};
				\node (D) at (0,1){};
				\draw (R)--(X);
				\draw (Y)--(A);
				\draw (A)--(B);
				\draw (B)--(C);
				\draw (C)--(D);
				\draw (0, 0.5) node {full $A$-branch of $I$};
			\end{scope}
		\end{tikzpicture}
	\end{center}
	\caption{Comparison between full branches in type $A$ and $B\!C$}
	\label{fig:comparison_branches}
\end{figure}

\subsection{Two Stokes circles:
	fission exponents}
\label{sec:bc_two_circles}

Throughout this section,
consider a pointed irregular type of the form
\begin{equation}
	\dot Q
	= \bigl( (n,q), (\wt n, \wt q) \bigr),
\end{equation}
where $q$ and $\wt q$ are exponential factors with distinct Stokes-circles-up-to-sign $\pm I \ceqq \braket{ \pm q}$ and $\pm \wt I \ceqq \braket{ \pm \wt q}$,
and with multiplicities $n,\wt n \in \frac 1 2 \mb Z_{> 0}$.
Let also $r \ceqq \on{ram}(q)$ and $\wt r \ceqq \on{ram}(\wt q)$.

By~\eqref{eq:truncation},
for $k \in \mb Q_{\geq 0}$ large enough one has $\braket{\pm\tau_k(q)} = \braket{\pm \tau_k(\wt q)} \in \mc S \bs \mb Z^{\ts}$.
Conversely,
denoting by $\ul k = \ul k(q,\wt q)$ the \emph{minimal} such number among the exponents of $q$ and $\wt q$,
we make the following:

\begin{defi}
	\label{def:common_different_part}

	\leavevmode

	\begin{enumerate}
		\item
		      The \emph{common part of} $q$ \emph{and} $\wt q$ is the exponential factor
		      \begin{equation}
			      q_c
			      = \wt q_c
			      \ceqq \tau_{\ul k}(q)
			      = \tau_{\ul k}(\wt q).
		      \end{equation}
		      (Recall that $\dot Q$ is assumed to be compatible,
		      cf.~Def.~\ref{def:compatible_pit}.)

		\item
		      The \emph{different parts of} $q$ \emph{and} $\wt q$ are the exponential factors
		      \begin{equation}
			      q_d
			      \ceqq q - q_c,
			      \qquad \wt q_d
			      \ceqq \wt q - \wt q_c.
		      \end{equation}

		\item
		      And the $B\!C$-\emph{fission exponent of} $q$ \emph{and} $\wt q$,
		      or (equivalently) \emph{of} $\pm I$ \emph{and} $\pm \wt I$,
		      is the number
		      \begin{equation}
			      f_{\pm I, \pm \wt I}= f_{q,\wt q}
			      \ceqq \max \Set{ \on{slope}(q_d), \on{slope}(\wt q_d) } \in \mb Q_{> 0}.
		      \end{equation}
	\end{enumerate}
\end{defi}

\begin{lemm}[cf.~\cite{boalch_doucot_rembado_2025_twisted_local_wild_mapping_class_groups_configuration_spaces_fission_trees_and_complex_braids}, Lem.~3.11]
	\label{lem:exterior_slopes_two_circles_bc}

	\leavevmode

	\begin{enumerate}
		\item
		      The set of exterior slopes,
		      governing the $B\!C$-admissible deformations of $\dot Q$ (cf.~Def.~\ref{def:numerical_equivalence_bc}),
		      is given by
		      \begin{equation}
			      \Set{ \on{slope} \bigl( q^{(i)} \pm \wt q^{(j)} \bigr) | (i,j) \in \mb Z \bs r\mb Z \ts \mb Z \bs \tilde r \mb Z }
			      = L_{B\!C}(q_c) \cup \{f_{q,\wt q}\} \sse \mb Q_{>0}.
		      \end{equation}

		\item
		      Furthermore,
		      the maps $(i,j) \mt \on{slope} \bigl( q^{(i)} \pm \wt q^{(j)} \bigr)$ are determined by $L_{B\!C}(q_c)$ and $f_{q,\wt q}$.
		      (This uses the fact that $\dot Q$ is compatible.)
	\end{enumerate}
\end{lemm}

\begin{proof}[Proof postponed to~\ref{proof:lem_exterior_slopes_two_circles_bc}]
\end{proof}

\begin{prop}
	\label{prop:characterisation_fission_exp_BC}

	Let $k \ceqq f_{q,\wt q}$ be the fission exponent of $q$ and $\wt q$.
	Then a \emph{compatible} pointed irregular type $\dot Q'$ satisfies $\dot Q \sim_{B\!C} \dot Q'$ if and only if it is of the form $\dot Q' = \bigl ((n ,q'),(\wt n, \wt q') \bigr)$,
	for two exponential factors $q'$  and $\wt q'$ with common part $q'_c = \wt q'_c$,
	such that:
	\begin{enumerate}
		\item
		      one has $L_{B\!C}(q) = L_{B\!C}(q')$ and $L_{B\!C}(\wt q) = L_{B\!C}(\wt q')$;

		\item
		      one has $L_{B\!C}(q_c) = L_{B\!C}(q'_c)$;

		\item and one has $f_{q',\wt q'} = k$.
	\end{enumerate}
\end{prop}

\begin{proof}
	The proof is analogous to that of~\cite[Prop.~3.12]{boalch_doucot_rembado_2025_twisted_local_wild_mapping_class_groups_configuration_spaces_fission_trees_and_complex_braids}.
\end{proof}

\subsubsection{}

Relying on the case of a single circle (in \S~\ref{sec:bc_one_circle}),
and the description of $B\!C$-levels,
one can spell out the first two conditions of Prop.~\ref{prop:characterisation_fission_exp_BC}.
It remains to investigate the third condition in more detail,
which is an important technical step in all that follows:

\begin{prop}[cf.~\cite{boalch_doucot_rembado_2025_twisted_local_wild_mapping_class_groups_configuration_spaces_fission_trees_and_complex_braids}, Prop.~3.13]
	\label{prop:types_of_fission_BC}

	Choose a number $k \in \mb Q_{>0}$,
	and an exponential factor $q_c$ whose exponents are all strictly greater than $k$;
	let $n \ceqq \on{num}(k)$ and $d \ceqq \on{den}(k)$.
	Moreover,
	for any pair of numbers $a,\wt a\in \mb C$  consider the two exponential factors
	\begin{equation}
		\label{eq:types_of_fission_BC}
		q
		\ceqq q_c + a z^{-k} + b,
		\qquad \wt q
		\ceqq q_c + \wt a z^{-k} + \wt b,
	\end{equation}
	where in turn $b,\wt b$ are exponential factors of slope strictly smaller than $k$.
	Finally,
	let $r_c \ceqq \on{ram}(q_c)$, $r \ceqq r_c \mathsmaller\vee d $ and $N \ceqq \frac{r}{r_c}$.
	Then:
	\begin{enumerate}
		\item
		      one has $f_{q,\wt q} = k$ if and only if one of the following (mutually-exclusive) situations occurs:
		      \begin{enumerate}
			      \item[(I)]
			            $q_c\neq 0$ is not special,
			            and:
			            \begin{enumerate}
				            \item[(1)]
				                  $k \in \on{Inc}_{B\!C}(q_c)$ and $a \neq \wt a$;

				            \item[(2a)]
				                  $k \notin \on{Inc}_{B\!C}(q_c)$ and exactly one of the two numbers $a$,
				                  $\wt a$ is nonzero;

				            \item[(2b)]
				                  $k \notin \on{Inc}_{B\!C}(q_c)$ and $a\neq 0$,
				                  $\wt a \neq 0$,
				                  $a^N \neq \wt a^N$;
			            \end{enumerate}
			      \item[(II)]
			            or $q_c \neq 0$ is special,
			            and:
			            \begin{enumerate}
				            \item[(1)]
				                  $k \in \on{Inc}_{B\!C}(q_c)$ and $a\neq \wt a$;

				            \item[(2a)]
				                  $k \notin \on{Inc}_{B\!C}(q_c)$ and exactly one of the two numbers $a$,
				                  $\wt a$ is nonzero;

				            \item[(2b)]
				                  $k\notin \on{Inc}_{B\!C}(q_c)$ and $a\neq 0$,
				                  $\wt a\neq 0$,
				                  and:
				                  \begin{enumerate}
					                  \item
					                        either $N = 1$,
					                        $k$ is a breaking of specialness of $E(q_c)$,
					                        and $a\neq \pm \wt a$;

					                  \item
					                        or $N > 1$,
					                        and:
					                        \begin{equation}
						                        \begin{cases}
							                        a^N \neq \wt a^{\, N},     & \quad N+kr \text{ even},  \\
							                        a^{2N} \neq \wt a^{\, 2N}, & \quad N+kr \text{ odd};
						                        \end{cases}
					                        \end{equation}
				                  \end{enumerate}
			            \end{enumerate}
			      \item[(III)]
			            or $q_c = 0$,\fn{
				            In this case $E(q_c) = \vn$,
				            so any (nonexistent) exponent of $q_c$ is vacuously strictly larger than $k$.}~and:
			            \begin{enumerate}
				            \item[(a)]
				                  exactly one of the two numbers $a$,
				                  $\wt a$ is nonzero;

				            \item[(b)]
				                  $a\neq 0$,
				                  $\wt a\neq 0$,
				                  and:
				                  \begin{equation}
					                  \begin{cases}
						                  a^N \neq \wt a^{\, N},     & \quad N \text{ even}, \\
						                  a^{2N} \neq \wt a^{\, 2N}, & \quad N \text{ odd};
					                  \end{cases}
				                  \end{equation}
			            \end{enumerate}
		      \end{enumerate}

		\item
		      furthermore,
		      the exponent $k$ satisfies the following:
		      \begin{enumerate}
			      \item
			            in cases (I.1) and (II.1),
			            one has $k \in \on{Inc}_{B\!C}(q) = \on{Inc}_{B\!C}(\wt q)$;

			      \item
			            in cases (I.2a),
			            (II.2a),
			            and (III.a),
			            if say $a=0$ and $\wt a \neq 0$,
			            one has $k \in L_{B\!C}(\wt q)$ and $k \notin E(q)$;

			      \item
			            and in cases (I.2b),
			            (II.2b),
			            and (III.b),
			            one has $k\in L_{B\!C}(q) = L_{B\!C}(\wt q)$.
		      \end{enumerate}
	\end{enumerate}
\end{prop}

\begin{proof}[Proof postponed to~\ref{proof:prop_types_of_fission_BC}]
\end{proof}

\begin{defi}
	\label{def:partial_ramification}

	In the notation of~\eqref{eq:types_of_fission_BC},
	the \emph{partial ramification order of} $q$ \emph{and} $\wt q$ is the integer
	\begin{equation}
		N
		\ceqq \frac{ \on{ram}(q_c) \vee \on{den}(k) }{\on{ram}(q_c) } \geq 1.
	\end{equation}
\end{defi}

\subsubsection{}

With a view towards the definition of $B\!C$-fission trees,
Fig.~\ref{fig:different_types_fission_two_circles_type_BC} sketches the various cases of Prop.~\ref{prop:types_of_fission_BC} in terms of the `fission' of a vertex $q_c$,
into two vertices,
corresponding to the coefficients $a$ and $\wt a$.
The types of the vertices correspond to the status of $k$ as an exponent of $q$ and $\wt q$,
respectively,
and the types of edges are as in Def.~\ref{def:type_of_edges_type_BC}.

\begin{figure}
	\begin{center}
		\begin{tikzpicture}
			\tikzstyle{mandatory}=[circle,fill=black,minimum size=6pt,draw, inner sep=0pt]
			\tikzstyle{authorised}=[circle,fill=white,minimum size=6pt,draw,inner sep=0pt]
			\tikzstyle{empty}=[circle,fill=white,minimum size=0pt,inner sep=0pt]
			\begin{scope}
				\node[authorised] (A1) at (1-4,0){};
				\node[authorised] (A2) at (2-4,0){};
				\node[empty] (A0) at (1.5-4,1){};
				\node[empty] (B1) at (1,0){};
				\node[mandatory] (B2) at (2,0){};
				\node[empty] (B0) at (1.5,1){};
				\node[mandatory] (C1) at (1+4,0){};
				\node[mandatory] (C2) at (2+4,0){};
				\node[empty] (C0) at (1.5+4,1){};
				\foreach \from/\to in {A1/A0,A2/A0,B1/B0,B2/B0,C1/C0,C2/C0}
				\draw (\from) -- (\to);
				\draw (-5,1) node {(I)};
				\draw (1.5-4,-0.7) node {(1)};
				\draw (1.5,-0.7) node {(2a)};
				\draw (1.5+4,-0.7) node {(2b)};
			\end{scope}
			\begin{scope}[yshift=-2.5cm]
				\node[authorised] (A1) at (1-4,0){};
				\node[authorised] (A2) at (2-4,0){};
				\node[empty] (A0) at (1.5-4,1){};
				\node[empty] (B1) at (1,0){};
				\node[mandatory] (B2) at (2,0){};
				\node[empty] (B0) at (1.5,1){};
				\node[mandatory] (C1) at (1+4,0){};
				\node[mandatory] (C2) at (2+4,0){};
				\node[empty] (C0) at (1.5+4,1){};
				\foreach \from/\to in {A1/A0,A2/A0,B1/B0,B2/B0,C1/C0,C2/C0}
				\draw[dashed] (\from) -- (\to);
				\draw (-5,1) node {(II)};
				\draw (1.5-4,-0.7) node {(1)};
				\draw (1.5,-0.7) node {(2a)};
				\draw (1.5+4,-0.7) node {(2b)};
			\end{scope}
			\begin{scope}[yshift=-5cm]
				\node[empty] (B1) at (1,0){};
				\node[mandatory] (B2) at (2,0){};
				\node[empty] (B0) at (1.5,1){};
				\node[mandatory] (C1) at (1+4,0){};
				\node[mandatory] (C2) at (2+4,0){};
				\node[empty] (C0) at (1.5+4,1){};
				\foreach \from/\to in {B1/B0,B2/B0,C1/C0,C2/C0}
				\draw[dotted, thick] (\from) -- (\to);
				\draw (-5,1) node {(III)};
				\draw (1.5,-0.7) node {(a)};
				\draw (1.5+4,-0.7) node {(b)};
			\end{scope}
		\end{tikzpicture}
	\end{center}
	\caption{The (eight) different types of fission for two $B\!C$-Stokes circles.}
	\label{fig:different_types_fission_two_circles_type_BC}
\end{figure}

\begin{exem}
	Let us illustrate the different cases of Prop.~\ref{prop:types_of_fission_BC} on four examples:
	choose (again) numbers $a,\wt a \in \mb C$.
	\begin{enumerate}
		\item Take $q_c=0$,
		      $k = 1/3$,
		      and
		      \begin{equation}
			      q
			      = a z^{-1/3},
			      \qquad \wt q
			      = \wt a z^{-1/3}.
		      \end{equation}
		      This is case (III) with $N = 3$,
		      which is odd,
		      so $\braket{\pm q} \neq \braket{\pm \wt q}$ if and only if $a^6 \neq \wt a^6$.

		\item Take $q_c = z^{-3/2}$ and $k = 1$,
		      i.e.,
		      \begin{equation}
			      q
			      = z^{-3/2} + az^{-1},
			      \qquad \wt q
			      = z^{-3/2}+ \wt az^{-1}.
		      \end{equation}
		      Here $q_c$ is special,
		      $N=1$,
		      and $k$ is \emph{not} inconsequential for $L_{B\!C}(q_c) = \set{3 \slash 2}$ (as it is a breaking of specialness),
		      so we are in case (II.2b.A) and $\braket{\pm q} \neq \braket{\pm \wt q}$ if and only if $a \neq \pm \wt a$.
		\item
		      Take $q_c = z^{-1/2}$ and $k = 1/4$,
		      i.e.,
		      \begin{equation}
			      q
			      = z^{-1/2} + az^{-1/4},
			      \qquad \wt q
			      = z^{-1/2} + \wt az^{-1/4}.
		      \end{equation}
		      Here $q_c$ is special,
		      $N = 2 > 1$,
		      and $k$ is \emph{not} inconsequential for $L_{B\!C}(q_c) = \set{ 1 \slash 2 }$,
		      so we are in case (II.2b.B).
		      If $a \neq 0$ then $L_{B\!C}(q) = \set{ 1 \slash 2,1 \slash 4 }$,
		      and similarly for $\wt q$.
		      We have $N = 2$, $r=4$ and $N+kr=3$ is odd.
		      Thus $\braket{\pm q} \neq \braket{\pm \wt q}$ if and only if $a^4 \neq \wt a^4$.

		\item
		      Take $q_c = z^{-3/2}$ and $k = \frac 1 2$,
		      i.e.,
		      \begin{equation}
			      q
			      = z^{-3/2} + az^{-1/2},
			      \qquad \wt q
			      = z^{-3/2} + \wt az^{-1/2}.
		      \end{equation}
		      Here $q_c$ is special,
		      $N = 1$,
		      and $k$ is inconsequential for $L_{B\!C}(q_c) = \set{ 3 \slash 2 }$ (as it is \emph{not} a breaking of specialness),
		      so we are in case (II.1) and $\braket{\pm q} \neq \braket{\pm \wt q}$ if and only if $a \neq \wt a$. \qedhere
	\end{enumerate}
\end{exem}

\subsection{General case:
	fission data}
\label{sec:bc_general}

Everything has been set up to treat the general case,
starting from the:

\begin{defi}[cf.~\cite{boalch_doucot_rembado_2025_twisted_local_wild_mapping_class_groups_configuration_spaces_fission_trees_and_complex_braids}, Def.~3.16]
	\label{def:fission_datum_BC}

	Let $p \geq 1$ be an integer.
	A $B\!C$-\emph{fission datum} is a pair $\mc F = (\mc L,f)$,
	consisting of:
	\begin{enumerate}
		\item
		      a multiset
		      \begin{equation}
			      \mc L
			      = \sum_{i = 1}^p (n_i,\bm L_i),
		      \end{equation}
		      of pairs $(n_i,\bm L_i)$,
		      where in turn---for $i \in \set{1,\dc,p}$:
		      \begin{enumerate}
			      \item
			            $\bm L_i$ is a $B\!C$-level datum;

			      \item
			            and $n_i$ is an allowed multiplicity for $\bm L_i$,
			            i.e.,
			            $n_i\in \frac 1 2 \mb Z_{>0}$ is such that $n_i$ is an integer if $\bm L_i$ is nonspecial or empty;
		      \end{enumerate}
		\item
		      and a \emph{list of fission exponents} $f$,
		      i.e.,
		      the choice of a rational number $f_{ij} = f_{ji} \geq 0$,
		      for $i,j \in \set{1, \dc ,p}$,
		      such that $f_{ij} = 0$ if and only if $i = j$.
	\end{enumerate}
\end{defi}

\subsubsection{}
\label{sec:fission_datum_from_irr_class_bc}

An irregular class $\Theta$ determines a $B\!C$-fission datum $\mc F(\Theta) = \bigl( \mc L(\Theta),f(\Theta) \bigr)$,
as follows.
As in Prop.~\ref{prop:form_type_BC_irregular_class},
write $\Theta = \sum_{i = 1}^p n_i \cdot (\pm I_i)$,
where $\pm I_1,\dc,\pm I_p$ are pairwise-distinct Stokes-circles-up-to-sign,
and taking suitable half-integer multiplicities $n_i > 0$.
Then let
\begin{equation}
	\label{eq:fission_datum_from_irr_class_bc}
	\mc L(\Theta) \ceqq
	\sum_{i = 1}^p \bigl( n_i,L_{B\!C}(\pm I_i) \bigr),
	\qquad f_{ij}(\Theta) \ceqq f_{\pm I_i, \pm I_j} \in \mb Q_{\geq 0},
\end{equation}
with the convention $f_{\pm I,\pm I} \ceqq 0$.

\begin{defi}
	\label{def:labelled_fission_datum_bc}

	A \emph{labelled} $B\!C$-\emph{fission datum} is a pair $\dot{\mc F} = (\dot{\mc L},f)$,
	consisting of:
	\begin{enumerate}
		\item
		      an ordered list
		      \begin{equation}
			      \dot{\mc L}
			      = \bigl( (n_1,\bm L_1),\dc,(n_p,\bm L_p) \bigr),
		      \end{equation}
		      of pairs $(n_i,\bm L_i)$ as in Def.~\ref{def:fission_datum_BC};

		\item
		      and a \emph{list of fission exponents} $f$,
		      as in Def.~\ref{def:fission_datum_BC}.
	\end{enumerate}
\end{defi}

\subsubsection{}

Analogously to~\eqref{eq:fission_datum_from_irr_class_bc},
a pointed irregular type $\dot Q$ determines a labelled $B\!C$-fission datum $\dot{\mc F}(\dot Q) = \bigl( \dot{\mc L}(\dot Q),f(\dot Q) \bigr)$.
Moreover,
a labelled fission datum determines an unlabelled one by forgetting the ordering of the list,
and this is compatible with passing from a pointed irregular type to its underlying irregular class.
The material of \S\S~\ref{sec:bc_one_circle}--\ref{sec:bc_two_circles} then yields the following important:

\begin{theo}[cf.~\cite{boalch_doucot_rembado_2025_twisted_local_wild_mapping_class_groups_configuration_spaces_fission_trees_and_complex_braids}, Thm.~3.17]
	\label{thm:admissible_deformation_iff_same_fission_data_type_BC}

	Two irregular classes (resp.,
	two compatible pointed irregular types)
	are mutual $B\!C$-admissible deformations if and only if they determine the same fission datum (resp.,
	the same labelled fission datum).
\end{theo}

\begin{proof}
	By construction,
	the fission datum of an irregular class is equivalent to the levels of its active circles,
	together with the common part and the fission exponent of any pair of distinct active circles.
	Etc.
\end{proof}

\subsection{General case:
	fission trees}
\label{sec:type_BC_fission_trees}

We now encode fission data in a graphical way,
using fission trees,
which are particular decorated rooted trees.
Consider thus a tuple $(\mc T,\mb V,\mb A,\mb L, h, m)$,
where:
\begin{enumerate}
	\item
	      $\mc T$ is a \emph{topological} tree,
	      with set of vertices $\mb V \sse \mc T$;\fn{
		      I.e.,
		      $\mc T$ is a contractible 1-dimensional CW-complex,
		      with 0-skeleton $\mb V$,
		      cf.~\cite[Chp.~0]{hatcher_2002_algebraic_topology}.}

	\item
	      $\mb A \sse \mb V$ is a subset of \emph{admissible} vertices;

	\item
	      $\mb L \sse \mb A$ is a finite subset of \emph{mandatory} vertices (a.k.a.~\emph{interior levels});

	\item
	      $h \cl \mc T \to \ol{\mb R}_{\geq 0}$ is a \emph{height function},
	      mapping each edge bijectively onto an open interval,
	      such that $\mb V_0 \ceqq h^{-1}(0) \sse \mc T$ is the set of \emph{leaves};

	\item
	      there is a unique vertex $v_\infty \in \mb V$ at height $h(v_\infty) = +\infty$,
	      which is the \emph{root} of $\mc T$;

	\item
	      and $m \cl \mb V_0 \to \mb Z_{>0}$ is a map assigning an integer \emph{multiplicity} to each leaf.
\end{enumerate}

The set of \emph{edges of} $\mc T$ is $\mb E \ceqq \pi_0(\mc T \sm \mb V)$.
We use the standard terminology for the \emph{child}-/\emph{parent-edges} and the \emph{child}-/\emph{parent-vertices} of each vertex.
The set of child-vertices of $v \in \mb V$ is denoted by $\on{Ch}(v) \sse \mb V$.
The \emph{branch-vertices} are vertices with more than one child-edge;
their set is denoted by $\mb Y \sse \mb V$.
The vertices in $\mb I \ceqq \mb A \sm \mb L \sse \mb A$ are called \emph{inconsequential},
and those in $\mb V \sm \mb A \sse \mb V$  are said to be \emph{empty}.
We also suppose that $h(\mb A) \sse \mb R_{> 0}$,
so that the root/leaves are empty.

The \emph{full branch} $\mc B_i$ of a leaf $i \in \mb V_0$ is the minimal subspace of $\mc T$ spanning all the way from $i$ to the root $v_\infty$.
Finally,
let $\mb L_i \ceqq \mb L \cap \mc B_i$ and $\mb A_i \ceqq \mb A \cap \mc B_i$ be the sets of mandatory/admissible vertices on the $i$-th full branch---%
respectively.

\begin{defi}[cf.~\cite{boalch_doucot_rembado_2025_twisted_local_wild_mapping_class_groups_configuration_spaces_fission_trees_and_complex_braids}, Def.~3.18]
	\label{def:bc_quasi_fission_tree}

	A tuple $(\mc T,\mb V,\mb A,\mb L, h, m)$ as above is a \emph{pre}-$B\!C$-\emph{fission tree} if it satisfies the following conditions:
	\begin{enumerate}
		\item
		      the vertices of $\mc T$ are exactly the leaves,
		      the points mapping to $h(\mb A)$,
		      and the root $v_\infty$,
		      i.e.,
		      $\mb V = h^{-1} \bigl( \set{0} \cup h(\mb A) \cup \set{+ \infty} \bigr)$;

		\item
		      the function $h$ maps each full branch bijectively onto $\ol{\mb R}_{\geq 0}$;

		\item
		      the interior levels of any full branch map to a $B\!C$-level datum,
		      i.e.,
		      the set $\bm L_i \ceqq h(\mb L_i) \sse \mb Q_{>0}$ is a $B\!C$-level datum,
		      for any leaf $i \in \mb V_0$;

		\item
		      and $\on{Adm}_i \ceqq h(\mb A_i) \sse \mb Q_{>0}$ coincides with the set $\on{Adm}_{B\!C}(\bm L_i)$ of $B\!C$-admissible exponents of $\bm L_i$,
		      for any leaf $i \in \mb V_0$;
		      i.e.,
		      forgetting the empty vertices on $\mc B_i$,
		      the triple $(\mc B_i, \mb A_i,\mb L_i)$ is a copy of the full $B\!C$-branch $\mc B_{\bm L_i}$ of the level datum $\bm L_i$,
		      as per Def.~\ref{def:full_branch_of_level_datum_type_BC}.
	\end{enumerate}
\end{defi}

(At times,
the tuple $(\mc T,\mb V,\mb A,\mb L,h,m)$ will be abusively shortened to $\mc T$ alone.)

\begin{defi}
	\label{def:bc_edge_type}

	Let $\mc T$ be a pre-$B\!C$-fission tree,
	and $e \in \mb E$ an edge contained in the full branch $\mc B_i \sse \mc T$ of a leaf $i \in \mb V_0$.
	Under the identification $\mc B_i \simeq \mc B_{\bm L_i}$,
	$e$ is contained within a unique edge $e_i$ of $\mc B_{\bm L_i}$:
	the \emph{type of} $e$ is (well-)defined by
	\begin{equation}
		t(e) \ceqq t(e_i) \in \set{E,S,NS},
	\end{equation}
	as per Def.~\ref{def:type_of_edges_type_BC}.
\end{defi}

\subsubsection{}

With a view towards local WMCGs,
we only consider trees which satisfy the following additional branching axiom (cf.~Fig.~\ref{fig:possible_branching_type_BC_fission_tree}):

\begin{defi}[cf.~Def.~\ref{def:bc_quasi_fission_tree}]
	\label{def:fission_tree_type_BC}

	A pre-$B\!C$-fission tree $\mc T$ is a $B\!C$-\emph{fission tree} if moreover:
	\begin{enumerate}
		\setcounter{enumi}{4}
		\item
		      for any branch-vertex $v \in \mb Y$,
		      exactly one of the following occurs:
		      \begin{enumerate}
			      \item[(I)] the child-edges of $v$ are nonspecial (of type $NS$),
			            and:
			            \begin{enumerate}
				            \item[(1)] the child-vertices in $\on{Ch}(v)$ are inconsequential;

				            \item[(2a)] one vertex in $\on{Ch}(v)$ is empty,
				                  and the others are mandatory;

				            \item[(2b)] the vertices in $\on{Ch}(v)$ are mandatory;
			            \end{enumerate}
			      \item[(II)] or the child-edges of $v$ are special (of type $S$),
			            and:
			            \begin{enumerate}
				            \item[(1)] the child-vertices in $\on{Ch}(v)$ are inconsequential;

				            \item[(2a)] one vertex in $\on{Ch}(v)$ is empty,
				                  and the others are mandatory;

				            \item[(2b)] the vertices in $\on{Ch}(v)$ are mandatory;
			            \end{enumerate}
			      \item[(III)] or the child-edges of $v$ are empty (of type $E$),
			            and:
			            \begin{enumerate}
				            \item[(a)] one child-vertex in $\on{Ch}(v)$ is empty,
				                  and the others are mandatory;

				            \item[(b)] the vertices in $\on{Ch}(v)$ are mandatory.
			            \end{enumerate}
		      \end{enumerate}
	\end{enumerate}
\end{defi}

\begin{rema}
	\label{rmk:about_def_trees_BC}

	As in~\cite[Rmk.~3.20]{boalch_doucot_rembado_2025_twisted_local_wild_mapping_class_groups_configuration_spaces_fission_trees_and_complex_braids},
	the last two axioms of Def.~\ref{def:bc_quasi_fission_tree} imply that the branching condition of Def.~\ref{def:fission_tree_type_BC} could be replaced by the requirement that $\on{Ch}(v)$ contains \emph{at most} one empty vertex,
	for any branch-vertex $v \in \mb Y$.
\end{rema}

\begin{rema}
	Recall that~\cite{boalch_doucot_rembado_2025_twisted_local_wild_mapping_class_groups_configuration_spaces_fission_trees_and_complex_braids} rather considered \emph{metrized trees}~\cite{zhang_1993_admissible_pairing_on_a_curve} (cf.~\cite[Def.~2]{baker_faber_metrized_graphs_electrical_networks_and_fourier_analysis}).
	The caveat is that in type $A$ there exists a smallest admissible height which is strictly larger than the height of the largest fission exponent,
	while here there is in general \emph{no} such height.
	Thus,
	wanting to glue the full branches strictly above the largest fission exponent,
	thereby getting a tree rather than a forest,
	one must allow for a root at `infinite' height.
	(Cf.~\S~\ref{sec:fission_tree_from_BC_irr_class} and~Rmk.~\ref{rmk:truncations_trees_BC}.)

	Importantly,
	while a metric/topology is convenient to have,
	it is actually \emph{not} needed to define the configuration spaces that we care about.
\end{rema}

\begin{figure}
	\begin{center}
		\begin{tikzpicture}
			\tikzstyle{mandatory}=[circle,fill=black,minimum size=6pt,draw, inner sep=0pt]
			\tikzstyle{authorised}=[circle,fill=white,minimum size=6pt,draw,inner sep=0pt]
			\tikzstyle{empty}=[circle,fill=white,minimum size=0pt,inner sep=0pt]
			\tikzstyle{indeterminate}=[circle,dotted,thick,minimum size=6pt,draw, inner sep=0pt]
			\begin{scope}
				\node[authorised] (A1) at (1-4,0){};
				\node[authorised] (A2) at (1.8-4,0){};
				\node[authorised] (A3) at (3-4,0){};
				\draw (2.4-4,0) node {$\dc$};
				\node[indeterminate] (A0) at (2-4,1){};
				\node[empty] (B1) at (1,0){};
				\node[mandatory] (B2) at (1.8,0){};
				\node[mandatory] (B3) at (3,0){};
				\draw (2.4,0) node {$\dc$};
				\node[indeterminate] (B0) at (2,1){};
				\node[mandatory] (C1) at (1+4,0){};
				\node[mandatory] (C2) at (1.8+4,0){};
				\node[mandatory] (C3) at (3+4,0){};
				\draw (2.4+4,0) node {$\dc$};
				\node[indeterminate] (C0) at (2+4,1){};
				\foreach \from/\to in {A1/A0,A2/A0,A3/A0,B1/B0,B2/B0,B3/B0,C1/C0,C2/C0, C3/C0}
				\draw (\from) -- (\to);
				\draw (-5,1) node {(I)};
				\draw (2-4,-0.7) node {(1)};
				\draw (2,-0.7) node {(2a)};
				\draw (2+4,-0.7) node {(2b)};
			\end{scope}
			\begin{scope}[yshift=-2.5cm]
				\node[authorised] (A1) at (1-4,0){};
				\node[authorised] (A2) at (1.8-4,0){};
				\node[authorised] (A3) at (3-4,0){};
				\draw (2.4-4,0) node {$\dc$};
				\node[indeterminate] (A0) at (2-4,1){};
				\node[empty] (B1) at (1,0){};
				\node[mandatory] (B2) at (1.8,0){};
				\node[mandatory] (B3) at (3,0){};
				\draw (2.4,0) node {$\dc$};
				\node[indeterminate] (B0) at (2,1){};
				\node[mandatory] (C1) at (1+4,0){};
				\node[mandatory] (C2) at (1.8+4,0){};
				\node[mandatory] (C3) at (3+4,0){};
				\draw (2.4+4,0) node {$\dc$};
				\node[indeterminate] (C0) at (2+4,1){};
				\foreach \from/\to in {A1/A0,A2/A0,A3/A0,B1/B0,B2/B0,B3/B0,C1/C0,C2/C0, C3/C0}
				\draw[dashed] (\from) -- (\to);
				\draw (-5,1) node {(II)};
				\draw (2-4,-0.7) node {(1)};
				\draw (2,-0.7) node {(2a)};
				\draw (2+4,-0.7) node {(2b)};
			\end{scope}
			\begin{scope}[yshift=-5cm]
				\node[empty] (B1) at (1,0){};
				\node[mandatory] (B2) at (1.8,0){};
				\node[mandatory] (B3) at (3,0){};
				\draw (2.4,0) node {$\dc$};
				\node[indeterminate] (B0) at (2,1){};
				\node[mandatory] (C1) at (1+4,0){};
				\node[mandatory] (C2) at (1.8+4,0){};
				\node[mandatory] (C3) at (3+4,0){};
				\draw (2.4+4,0) node {$\dc$};
				\node[indeterminate] (C0) at (2+4,1){};
				\foreach \from/\to in {B1/B0,B2/B0,B3/B0,C1/C0,C2/C0, C3/C0}
				\draw[dotted, thick] (\from) -- (\to);
				\draw (-5,1) node {(III)};
				\draw (2,-0.7) node {(a)};
				\draw (2+4,-0.7) node {(b)};
			\end{scope}
		\end{tikzpicture}
	\end{center}
	\caption{Type-$B\!C$ branchings.
		The branching vertex is dotted,
		indicating that its type is \emph{not} prescribed by the branching type.}
	\label{fig:possible_branching_type_BC_fission_tree}
\end{figure}

\begin{defi}
	\label{def:tree_labels}

	\leavevmode

	\begin{enumerate}
		\item
		      A \emph{labelling of a} $B\!C$-\emph{fission tree} $\mc T$ is a total ordering of its leaves,
		      i.e.,
		      a bijection $\psi \cl \set{1,\dc,p} \lxra{\simeq} \mb V_0$,
		      where $p \ceqq \abs{\mb V_0}$.

		\item
		      The corresponding \emph{labelled} $B\!C$-\emph{fission tree} is the pair $\dot{\mc T} \ceqq (\mc T,\psi)$.
	\end{enumerate}
\end{defi}

\begin{defi}
	\label{def:tree_isos}

	\leavevmode
	\begin{enumerate}
		\item
		      An \emph{isomorphism} $\mc T \lxra{\simeq} \mc T'$ \emph{of fission trees} is an isomorphism of the underlying (topological) trees which matches up the distinguished subsets of vertices and the height/multiplicity functions.

		\item
		      An \emph{isomorphism} $\dot{\mc T} \lxra{\simeq} \dot{\mc T}'$ \emph{of labelled fission trees} is an isomorphism of the underlying (unlabelled) fission trees,
		      which also preserves the labellings.
	\end{enumerate}
\end{defi}

(Until \S~\ref{sec:D_trees},
all (pre-)fission trees are tacitly of type $B\!C$.)

\begin{rema}
	\label{rmk:ordering_vertices}

	The labelled fission tree of a pointed irregular type always admits a \emph{planar} structure---%
	up to relabelling.

	Therefore,
	we shall draw the labelled leaves $v_i \ceqq \psi(i) \in \mb V_0$ of a labelled fission tree $\dot{\mc T}$ in the plane,
	ordering them left-to-right in a row for $i \in \set{1,\dc,p}$;
	and analogously for the set of vertices at each height.
	In turn,
	filtering by height---%
	in lexicographic fashion---%
	yields a \emph{total order} on the set $\mb V$ of vertices,
	such that the root is the greatest element.
	(This is the convention of Fig.~\ref{fig:example_tree_type_BC},
	as well as Fig.~\ref{fig:example_tree_type_D} in type $D$.)

	In particular,
	there is also a \emph{bijection} $\set{1,\dc,\abs{\mb A}} \lxra{\simeq} \mb A$.
\end{rema}

\begin{rema}
	\label{rmk:tree_autos_embedding}

	The group $\Aut(\mc T)$,
	of (unlabelled) fission-tree isomorphisms $\mc T \lxra{\simeq} \mc T$,
	embeds in the symmetric group $\mf S_{\mb V_0}$,
	of permutations of the \emph{unordered} set $\mb V_0$ of leaves---%
	as the action on the leaves determines everything else.
	Choosing a labelling $\psi$ of $\mc T$,
	the group $\Aut(\mc T)$ then embeds in the symmetric group $\mf S_p$,
	of permutations of the \emph{ordered} set $\set{1,\dc,p} \simeq \mb V_0$:
	this uses the full structure of the labelled fission tree $\dot{\mc T} = (\mc T,\psi)$.

	(Note that the group $\Aut(\dot{\mc T})$,
	of labelled-fission-tree automorphisms,
	is \emph{trivial}.)
\end{rema}

\subsubsection{}

To a fission tree $\mc T$ one can associate a $B\!C$-fission datum (using Rmk.~\ref{rmk:convention_integer_multiplicities_BC}),
via
\begin{equation}
	\mc F(\mc T)
	\ceqq \sum_{i \in \mb V_0} \bigl( m_i,h(\mb L_i) \bigr),
	\qquad m_i \ceqq m(i),
\end{equation}
together with the corresponding fission exponents.
Analogously,
to any labelled fission tree $\dot{\mc T}$ one can associate a labelled $B\!C$-fission datum,
via
\begin{equation}
	\dot{\mc F} \bigl( \dot{\mc T} \bigr)
	\ceqq \bigl( \bigl( m_1,h(\mb L_1) \bigr),\dc, \bigl(m_p,h(\mb L_p)\bigr) \bigr),
	\qquad p \ceqq \abs{\mb V_0},
\end{equation}
under a labelling bijection $\set{1,\dc, p } \simeq \mb V_0$.
Then:

\begin{lemm}
	\label{lem:equivalence_tree_fission_datum_BC}

	Two (labelled) fission trees are isomorphic if and only if they determine the same (labelled) $B\!C$-fission datum.
\end{lemm}

\begin{proof}
	The proof is analogous to that of~\cite[Lem.~3.21]{boalch_doucot_rembado_2025_twisted_local_wild_mapping_class_groups_configuration_spaces_fission_trees_and_complex_braids}.
\end{proof}

\subsubsection{}

Moreover,
an irregular class $\Theta$ defines a fission tree.

Namely,
if $\Theta$ is as is~\eqref{eq:type_BC_irr_class},
but with the integer-multiplicities convention,
then each distinct Stokes-circle-up-to-sign $\pm I_i = \braket{\pm q_i}$ defines a full $B\!C$-branch $\mc B_i$ as above:
thus,
$\mc B_i$ is a totally-ordered set equipped with subsets $\mb L_i \sse \mb A_i\sse \mc B_i$,
and with an order-preserving bijection $h \cl \mc B_i \lxra{\simeq} \ol{\mb R}_{\geq 0}$.
Then consider the discrete subset $\on{Adm}(\Theta) \ceqq \bigcup_i \on{Adm}_i \sse \mb Q_{>0}$.
For any $k \in \mb Q_{>0}$,
the \emph{successor}
\begin{equation}
	\on{succ}(k) \in \ol{\on{Adm}}(\Theta) \ceqq \on{Adm}(\Theta) \cup \set{+ \infty},
\end{equation}
is well-defined.

\begin{defi}
	If $f_{ij} = f_{\pm I_i,\pm I_j}$ is the $B\!C$-fission exponent of ${\pm I_i,\pm I_j}$,
	then their $B\!C$-\emph{gluing exponent} is the next admissible exponent,
	i.e.,
	\begin{equation}
		\label{eq:gluing_exponent}
		g_{ij}
		\ceqq \on{succ}(f_{ij}) \in \ol{\on{Adm}}(\Theta).
	\end{equation}
\end{defi}

\subsubsection{}
\label{sec:fission_tree_from_BC_irr_class}

Now,
for $i,j$ with $i \neq j$,
glue the full branches $\mc B_i$ and $\mc B_j$ as follows:
(i) over the interval $[g_{ij}, + \infty)$, if $g_{ij} < + \infty$;
or (ii) at the root,
if $g_{ij} = + \infty$.
This defines the tree $\mc T = \mc T(\Theta)$,
already equipped with a height function $h \cl \mc T \to \ol{\mb R}_{\geq 0}$.
The subsets $\mb L_i \sse \mb A_i \sse \mc B_i$ then fit together to yield $\mb L \sse \mb A \sse \mc T$,
and the overall set of vertices is
\begin{equation}
	\mb V
	\ceqq h^{-1} \bigl( \set{0} \cup \ol{\on{Adm}}(\Theta) \bigr) \sse \mc T.
\end{equation}
The latter is declared to be the 0-skeleton of $\mc T$,
and the standard topology of the open intervals of $\ol{\mb R}_{\geq 0}$ is transferred in bijective fashion to the edges---%
via $h$---%
to define the 1-cells;
the gluing is already determined.
Finally,
the function $m \cl \mb V_0 \to \mb Z_{> 0}$ takes the multiplicity of each Stokes-circle-up-to-sign in $\Theta$.
(The fact that the branching is correct is established in Lem.~\ref{lem:characterisation_right_branchings_type_BC}.)

In addition,
a pointed irregular type $\dot Q$ defines a labelled fission tree $\dot{\mc T} = \dot{\mc T}(\dot Q)$ by labelling the leaves as per the order of the exponential factors in $\dot Q$,
and in conclusion:

\begin{coro}
	\label{cor:adm_def_by_trees}
	Two irregular classes (resp.,
	two compatible pointed irregular types) are mutual $B\!C$-admissible deformations if and only if they determine isomorphic fission trees (resp.,
	isomorphic labelled fission trees).
\end{coro}

\begin{proof}
	Combine Thm.~\ref{thm:admissible_deformation_iff_same_fission_data_type_BC} +~Lem.~\ref{lem:equivalence_tree_fission_datum_BC}.
\end{proof}

\subsubsection{}

Furthermore,
the nonroot vertices of $\mc T(\Theta)$ (and $\dot{\mc T}(\dot Q)$) may be interpreted in terms of truncated Stokes-circles-up-to-sign,
as follows:

\begin{enonce}{Lemma-Definition}
	\label{lem:link_vertices_truncated_circles_BC}

	Let $\mb V_k \ceqq h^{-1}(k) \sse \mc T(\Theta)$ be the subset of \emph{height}-$k$ vertices.
	If $k \in \on{Adm}(\Theta) \cup \set{0}$,
	then there is a canonical bijection
	\begin{equation}
		\mb V_k \lxra{\simeq} \Set{ \Braket{\pm\tau_k(q_i)} | i \in \mb V_0 }.
	\end{equation}
\end{enonce}

\begin{proof}
	Analogous to the proof of~\cite[Lem.~3.22]{boalch_doucot_rembado_2025_twisted_local_wild_mapping_class_groups_configuration_spaces_fission_trees_and_complex_braids}.
\end{proof}

\begin{rema}
	If $k > l$ are two admissible exponents,
	or if $k > 0$ is admissible and $l = 0$,
	then Lem.-Def.~\ref{lem:link_vertices_truncated_circles_BC} yields a surjection
	\begin{equation}
		\phi_{kl} \cl \mb V_l \lthra \mb V_k,
		\qquad \phi_{kl} \bigl( \braket{\pm\tau_l(q_i)} \bigr)
		\ceqq \braket{\pm\tau_k(q_i)}.
	\end{equation}
	This controls the structure of the tree,
	as it determines the unique parent-vertex of each node when $k,l$ are consecutive.
	(Incidentally,
	it also proves that the gluings of the full branches can be done consistently.)

	Thus,
	two vertices $\Braket{\pm\tau_l({q_i})}, \braket{\pm\tau_l({q_j})} \in \mb V_l$ are descendants of one and the same vertex in $\mb V_k$ if and only if they have the same height-$k$ truncation.
	Furthermore,
	if $\braket{\pm q_i}$ and $\braket{\pm q_j}$ are two active Stokes-circles-up-to-sign,
	corresponding to two leaves of $\mc T$,
	then their closest common ancestor-vertex in the tree corresponds to their common part.
\end{rema}

\begin{rema}
	\label{rmk:comparison_bichromatic_trees_BC_DR}

	In~\cite{doucot_rembado_tamiozzo_2026_local_wild_mapping_class_groups_and_cabled_braids} we introduced \emph{bichromatic fission trees} to encode deformations of \emph{untwisted} $B\!C$-irregular types.
	They are equivalent to a particular case of Def.~\ref{def:fission_tree_type_BC}.

	Namely,
	let $\dot Q$ be an untwisted compatible pointed irregular type.
	We associate to it a labelled fission tree $\dot{\mc T}$ as above,
	as well as bichromatic fission tree $\wt{\mc T}$ as in op.~cit.
	The bichromatic tree  can be viewed as a tuple $(\wt{\mc T}, \wt{\mb V}, \wt h, \wt m)$,
	with set of vertices $\wt{\mb V}$,
	height function $\wt h$,
	and the function $\wt m$ assigning multiplicities to the leaves.
	(Strictly speaking,
	this is an ordinary combinatorial tree;
	it can be naturally made into a topological one,
	denoted the same.)

	Then:
	\begin{enumerate}
		\item
		      there is a height-preserving function $\dot{\mc T} \to \wt{\mc T}$,
		      inducing a bijection $\mb V \lxra{\simeq} \wt{\mb V}$;

		\item
		      and the coloured vertices of $\wt{\mc T}$ are related to vertices/edges of $\dot{\mc T}$ as follows:
		      \begin{enumerate}
			      \item
			            the blue vertices of $\wt{\mc T}$ map to the vertices of $\dot{\mc T}$ whose parent-edge is empty;

			      \item
			            and the green vertices of $\wt{\mc T}$ map to the vertices of $\dot{\mc T}$ whose parent-edge is \emph{not} empty.
		      \end{enumerate}
	\end{enumerate}
	The axioms are consistent:
	green vertices in $\wt{\mc T}$ only have green child-vertices,
	which corresponds to the fact that if $v \in \mb V$ has nonempty parent-edge then the its child-edges are nonempty;
	and blue vertices in $\wt{\mc T}$ have at most one blue child-vertex,
	which corresponds to the fact that if $v \in \mb V$ has empty parent-edge then at most one of its child-edges is empty.
\end{rema}

\begin{rema}
	\label{rmk:truncations_trees_BC}

	The truncation of fission trees is different from~\cite{boalch_doucot_rembado_2025_twisted_local_wild_mapping_class_groups_configuration_spaces_fission_trees_and_complex_braids}.

	Namely,
	recall that in op.~cit.~the set of admissible exponents is unbounded above,
	hence infinite.
	Then it was convenient to introduce a truncation (to the smallest admissible exponent common to all branches),
	reducing to a tree with a finite number of admissible vertices,
	getting a finite-dimensional space of admissible deformations.
	In the present situation,
	the set of admissible exponents is automatically \emph{finite}.
\end{rema}

\begin{exem}
	\label{ex:example_fission_tree}

	Consider the pointed irregular type
	\begin{equation}
		\dot Q
		= \bigl( (1,q_1), (1 \slash 2, q_2), (1, q_3), (1, q_4) \bigr),
	\end{equation}
	where $q_1 = z^{-1} + z^{-1/2} + z^{-1/3}$,
	$q_2 = z^{-1/2} + z^{-1/6}$,
	$q_3 = z^{-1/2} + z^{-1/4}$,
	and $q_4 = z^{-1/2} +2 z^{-1/4}$.
	(So $q_2$ is special,
	while $q_1$,
	$q_3$,
	and $q_4$, are not.)
	The labelled fission tree $\dot{\mc T}(\dot Q)$ is drawn in Fig.~\ref{fig:example_tree_type_BC}.
	There are 8 nonempty nonroot vertices,
	of which 7 are mandatory.
\end{exem}

\begin{figure}
	\begin{center}
		\begin{tikzpicture}
			\tikzstyle{root}=[fill=black,minimum size=6pt,draw,inner sep=0pt]
			\tikzstyle{mandatory}=[circle,fill=black,minimum size=6pt,draw, inner sep=0pt]
			\tikzstyle{authorised}=[circle,fill=white,minimum size=6pt,draw,inner sep=0pt]
			\tikzstyle{empty}=[circle,fill=white,minimum size=0pt,inner sep=0pt]
			\node[root] (R) at (1,6){};
			\node[mandatory] (A1) at (0,5){};
			\node[empty] (A2) at (2,5){};
			\node[mandatory] (B1) at (0,4){};
			\node[mandatory] (B2) at (2,4){};
			\node[mandatory] (C1) at (0,3){};
			\node[empty] (C2) at (2,3){};
			\node[empty] (D1) at (0,2){};
			\node[empty] (D2) at (1,2){};
			\node[mandatory] (D3) at (2,2){};
			\node[mandatory] (D4) at (3,2){};
			\node[authorised] (E1) at (0,1){};
			\node[mandatory] (E2) at (1,1){};
			\node[empty] (E3) at (2,1){};
			\node[empty] (E4) at (3,1){};
			\node[empty] (F1) at (0,0){};
			\node[empty] (F2) at (1,0){};
			\node[empty] (F3) at (2,0){};
			\node[empty] (F4) at (3,0){};
			\draw[dotted, thick] (R)--(A1); \draw[dotted, thick] (R)--(A2);
			\draw (A1)--(B1); \draw[dotted, thick] (A2)--(B2);
			\draw (B1)--(C1); \draw[dashed] (B2)--(C2);
			\draw (C1)--(D1); \draw[dashed] (C2)--(D2); \draw[dashed] (C2)--(D3); \draw[dashed] (C2)--(D4);
			\draw (D1)--(E1); \draw[dashed] (D2)--(E2); \draw (D3)--(E3); \draw (D4)--(E4);
			\draw (E1)--(F1); \draw[dashed] (E2)--(F2); \draw (E3)--(F3); \draw (E4)--(F4);
			\draw (-1,6) node {$+ \infty$};
			\draw (-1,5) node {$1$};
			\draw (-1,4) node {$1/2$};
			\draw (-1,3) node {$1/3$};
			\draw (-1,2) node {$1/4$};
			\draw (-1,1) node {$1/6$};
			\draw (-1,0) node {$0$};
			\draw (0,-0.5) node {$q_1$};
			\draw (1,-0.5) node {$q_2$};
			\draw (2,-0.5) node {$q_3$};
			\draw (3,-0.5) node {$q_4$};
		\end{tikzpicture}
	\end{center}
	\caption{Labelled fission tree of Exmp.~\ref{ex:example_fission_tree}.}
	\label{fig:example_tree_type_BC}
\end{figure}

\subsection{General case:
	tree realizations}
\label{sec:realization}

With a view towards configuration spaces,
consider a fission tree $\mc T$ and a function $c \cl \mb A \to \mb C$.
For each leaf $i \in \mb V_0$,
define the exponential factor
\begin{equation}
	\label{eq:exp_factor_from_map_c}
	q_{c,i}
	\ceqq \sum_{v \in \mb A_i} c(v) z^{-h(v)}.
\end{equation}
This yields a linear combination
\begin{equation}
	\label{eq:irr_class_from_map_c}
	\Theta_c
	= \sum_{i \in \mb V_0} m_i \cdot (\pm I_{c,i}),
	\qquad
	m_i \ceqq m(i),
	\quad \pm I_{c,i} \ceqq \braket{ \pm q_{c,i} },
\end{equation}
Moreover,
a labelling bijection $\psi \cl \set{ 1,\dc,p } \lxra{\simeq} \mb V_0$ yields a list
\begin{equation}
	\label{eq:pit_from_map_c}
	\dot Q_c
	\ceqq \bigl( (m_1,q_{c,1}),\dc,(m_p,q_{c,p}) \bigr).
\end{equation}

\begin{defi}
	\label{def:realization_type_BC}

	The `coefficient' map $c$ is a \emph{realization of} $\mc T$ (resp.,
	\emph{of} $\dot{\mc T} = (\mc T,\psi)$) if:
	\begin{enumerate}
		\item
		      the formal sum~\eqref{eq:irr_class_from_map_c} is an irregular class (resp.,
		      the list~\eqref{eq:pit_from_map_c} is a pointed irregular type);

		\item
		      and one has $\mc T(\Theta_c) \simeq \mc T$ as fission trees (resp.,
		      $\dot{\mc T}(\dot Q_c) \simeq \dot{\mc T}$ as labelled fission trees).
	\end{enumerate}
\end{defi}

\subsubsection{}
\label{sec:realization_check}

By definition,
$c$ realizes $\mc T$ if and only if it realizes $\dot{\mc T}$:
the next aim is to give explicit conditions so that a map $c \cl \mb A \to \mb C$ is a realization of either.

In brief,
one must check whether the tree determined by~\eqref{eq:irr_class_from_map_c} has the desired branching/mandatory nodes,
which can be checked independently at each branchpoint.
Consider,
therefore,
two \emph{consecutive} heights $l > k$ of the tree,
and let $v \in \mb V_l$ be a vertex with child-vertices $\on{Ch}(v) = \set{v_1,\dc, v_n} \sse \mb V_k$,
for some integer $n \geq 1$.
Moreover,
denote by $q = q_{c,v}$ the exponential factor determined by $c$ at the node $v$,
i.e.,
$q = \tau_l(q_j)$ for any leaf $j$ which is a descendant-vertex of $v$.
Finally,
let
\begin{equation}
	\label{eq:galois_orbits_children}
	q_i
	= q + c_i z^{-k},
	\qquad c_i \ceqq c(v_i),
	\quad i \in \set{1,\dc,n},
\end{equation}
be the corresponding exponential factors of the child-vertices, and set $r:=\on{ram}(q)\vee\on{den}(k)$.

\begin{rema}
	\label{rmk:realization_check}
	Each pair $(q_i,q_j)$ in~\eqref{eq:galois_orbits_children},
	for $i \neq j \in \set{1,\dc,n}$,
	has $q$ as its common part;
	and by Def.~\ref{def:partial_ramification},
	the partial ramification order $N_{ij}$ of the former is a function of $\on{ram}(q)$ and $\on{den}(k)$ only.
	Thus,
	given a function $c \cl \mb A \to \mb C$,
	there is a \emph{common} partial ramification order to the child-vertices of $v$, given by $N:=\frac{r}{\ram{q}}\geq 1$.
\end{rema}

\subsubsection{}

Now,
for $c$ to be a realization,
the numbers $c_1,\dc,c_n$ have to be such that the Stokes-circles-up-to-sign $\braket{\pm q_1},\dc,\braket{\pm q_n}$ are pairwise-distinct,
and to this extent:

\begin{lemm}[cf.~\cite{boalch_doucot_rembado_2025_twisted_local_wild_mapping_class_groups_configuration_spaces_fission_trees_and_complex_braids}, Prop.~3.26]
	\label{lem:characterisation_right_branchings_type_BC}

	The Stokes-circles-up-to-sign of the exponential factors~\eqref{eq:galois_orbits_children} are pairwise-distinct if and only if one of the following (mutually-exclusive) conditions holds:
	\begin{enumerate}
		\item[(I)] $q \neq 0$ is not special,
		      and:
		      \begin{enumerate}
			      \item[(1)] $k \in \on{Inc}_{B\!C}(q)$,
			            and $c_i\neq c_j$ for $i \neq j \in \set{1,\dc,n}$;

			      \item[(2a)] $k \notin \on{Inc}_{B\!C}(q)$,
			            and there is a unique index $i_0 \in \set{1, \dc, n}$ such that $c_{i_0}=0$;
			            moreover,
			            if $i \neq i_0$ then $c_i \neq 0$,
			            and if $i \neq j \in \set{1, \dc, n} \sm \set{i_0}$ then $c_i^N \neq c_j^N$---%
			            in the notation of Rmk.~\ref{rmk:realization_check};

			      \item[(2b)] $k \notin \on{Inc}_{B\!C}(q)$,
			            and $c_i \neq 0$ for $i \in \set{1,\dc,n}$,
			            and $c_i^N \neq c_j^N$ for $i \neq j \in \set{1,\dc,n}$;
		      \end{enumerate}
		\item[(II)] $q\neq 0$ is special,
		      and:
		      \begin{enumerate}
			      \item[(1)] $k \in \on{Inc}_{B\!C}(q)$,
			            and $c_i\neq c_j$ for $i \neq j \in \set{1,\dc,n}$;

			      \item[(2a)] $k \notin \on{Inc}_{B\!C}(q)$,
			            and there is a unique index $i_0 \in \set{1, \dc, n}$ such that $c_{i_0} = 0$;
			            moreover,
			            if $i \neq i_0$ then $c_i \neq 0$,
			            and if $i\neq j \in \set{1, \dc, n} \sm \set{i_0}$ then:
			            \begin{equation}
				            \begin{cases}
					            c_i^N \neq c_j^N,       & \quad N+kr \text{ even},  \\
					            c_i^{2N} \neq c_j^{2N}, & \quad N+kr \text{ odd};
				            \end{cases}
			            \end{equation}
			      \item[(2b)] $k \notin \on{Inc}_{B\!C}(q)$,
			            and:
			            \begin{enumerate}
				            \item
				                  either $N=1$,
				                  and $c_i \neq 0$ for $i \in \set{1,\dc,n}$,
				                  and $c_i \neq \pm c_j$ for $i \neq j \in \set{1,\dc,n}$;

				            \item  or $N>1$,
				                  and $c_i \neq 0$ for $i \in \set{1,\dc,n}$,
				                  and if $i \neq j \in \set{1,\dc,n}$ then:
				                  \begin{equation}
					                  \begin{cases}
						                  c_i^N \neq c_j^N,      & \quad N+kr \text{ even},  \\
						                  c_i^{2N} \neq c_j^{2N} & \quad N+kr \text{ odd};
					                  \end{cases}
				                  \end{equation}
			            \end{enumerate}
		      \end{enumerate}
		\item[(III)] $q = 0$,
		      and:
		      \begin{enumerate}
			      \item[(a)] there is a unique index $i_0 \in \set{1, \dc, n}$ such that $c_{i_0} = 0$;
			            moreover,
			            for $i \neq i_0$ one has $c_i\neq 0$,
			            and if $i\neq j \in \set{1, \dc, n} \sm \set{i_0}$ then:
			            \begin{equation}
				            \begin{cases}
					            c_i^N \neq  c_j^N,       & \quad N \text{ even}, \\
					            c_i^{2N} \neq  c_j^{2N}, & \quad N \text{ odd};
				            \end{cases}
			            \end{equation}
			      \item[(b)] $c_i\neq 0$ for $i \in \set{1,\dc,n}$,
			            and if $i \neq j \in \set{1, \dc, n}$ then:
			            \begin{equation}
				            \begin{cases}
					            c_i^N \neq  c_j^N,       & N \text{ even},      \\
					            c_i^{2N} \neq  c_j^{2N}, & \quad N \text{ odd}.
				            \end{cases}
			            \end{equation}
		      \end{enumerate}
	\end{enumerate}
	(This justifies the branching axiom of Def.~\ref{def:fission_tree_type_BC}.)
\end{lemm}

\begin{proof}
	Apply Prop.~\ref{prop:types_of_fission_BC} to pairs of distinct indices.
\end{proof}

\begin{coro}[cf.~\cite{boalch_doucot_rembado_2025_twisted_local_wild_mapping_class_groups_configuration_spaces_fission_trees_and_complex_braids}, Thm.~3.27]
	\label{cor:type_BC_characterisation_of_realizations}

	The map $c \cl \mb A\to \mb C$ is a realization of $\mc T$ and/or $\dot{\mc T}$ if and only if:
	\begin{enumerate}
		\item
		      one has $c(\mb L) \sse \mb C^\ast$,
		      i.e.,
		      $c(v) \neq 0$ for any mandatory node $v$;

		\item
		      and for any pair of nonempty sibling-vertices $u,v$,
		      denoting by $k$ their height, $r$ their ramification order, and $N \geq 1$ their partial ramification order (cf.~again \S\ref{sec:realization_check} and Rmk.~\ref{rmk:realization_check}):
		      \begin{enumerate}
			      \item[(I)]
			            if $u,v$ have nonspecial parent-edges,
			            then $c(u)^N \neq c(v)^N$;

			      \item[(II)]
			            if $u,v$ have special parent-edges,
			            then:
			            \begin{enumerate}
				            \item[(i)]
				                  If $N=1$ and $u,v$ are mandatory,
				                  then $c(u) \neq \pm c(v)$;

				            \item[(ii)]
				                  otherwise:
				                  \begin{equation}
					                  \begin{cases}
						                  c(u)^N \neq c(v)^N,       & \quad N+kr \text{ even},  \\
						                  c(u)^{2N} \neq c(v)^{2N}, & \quad N+kr \text{ odd};
					                  \end{cases}
				                  \end{equation}
			            \end{enumerate}

			      \item[(III)]
			            and if $u,v$ have empty parent-edges,
			            then:
			            \begin{equation}
				            \begin{cases}
					            c(u)^N \neq  c(v)^N,       & \quad N \text{ even}, \\
					            c(u)^{2N} \neq  c(v)^{2N}, & \quad N \text{ odd}.
				            \end{cases}
			            \end{equation}
		      \end{enumerate}
	\end{enumerate}
\end{coro}

\begin{proof}
	The first condition,
	and the definition of $\mb A$,
	means that each full branch has the correct interior levels.
	The second condition means that $\mc T(\Theta_c)$ has the correct branchings (cf.~Lem.~\ref{lem:characterisation_right_branchings_type_BC}).
\end{proof}

\begin{coro}
	\label{cor:type_BC_realisability}

	Any (labelled) fission tree admits a realization.
\end{coro}

\begin{proof}
	Analogous to the proof of~\cite[Cor.~3.28]{boalch_doucot_rembado_2025_twisted_local_wild_mapping_class_groups_configuration_spaces_fission_trees_and_complex_braids}.
\end{proof}

\subsubsection{}

Hereafter,
and until \S~\ref{sec:weyl_group_tree_bc},
let $\dot Q$ be a compatible pointed irregular type,
and write $\dot{\mc T} \ceqq \dot{\mc T}(\dot Q)$.

\begin{lemm}[cf.~\cite{boalch_doucot_rembado_2025_twisted_local_wild_mapping_class_groups_configuration_spaces_fission_trees_and_complex_braids}, Lem.~3.29]
	\label{lem:any_Q_comes_from_realization_type_BC}

	Set $r \ceqq \on{ram}(\dot Q)$,
	$w \ceqq z^{1/r}$,
	$K \ceqq \on{Katz}(\dot Q)$,
	and $s \ceqq rK$ (cf.~\S~\ref{sec:type_BC_configuration_space}),
	so that
	\begin{equation}
		\label{eq: naivecoeffs_BC_2}
		\dot Q
		= \bigl( (m_1,\wh q_1), \dc, (m_p,\wh q_p) \bigr),
		\qquad \wh q_i
		= \sum_{j = 1}^s a_{ij} w^{-j},
	\end{equation}
	for a suitable collection of coefficients $\bm a = (a_{ij})_{i,j}\in \mb C^{ps}$.
	Then there is a \emph{unique} realization $c = c_{\dot Q} \cl \mb A \to \mb C$ of $\dot{\mc T}$ such that $c(v) = a_{ik}$ for all $i,k$,
	where in turn $v = \braket{\pm \tau_{k/r}(\wh q_i)}\in \mb A$ is the vertex of $\dot{\mc T}$ determined by the truncation of the exponential factor $\wh q_i$ (cf.~Lem.-Def.~\ref{lem:link_vertices_truncated_circles_BC}).
\end{lemm}

\begin{proof}
	This follows from the facts that:
	i) $a_{ik} = 0$ if $\braket{\pm \tau_{k/r}(\wh q_i)} \in \mb V \sm \mb A$,
	i.e.,
	if the vertex of $\dot{\mc T}$ corresponding to the truncated Stokes-circle-up-to-sign $\braket{\pm \tau_{k/r}(\wh q_i)}$ is \emph{not} admissible;
	and ii) $a_{ik} = a_{jk}$ if $\braket{\pm \tau_{k/r}(\wh q_i)} = \braket{\pm \tau_{k/r}(\wh q_j)}$,
	i.e.,
	if the truncations determine the same vertex.
\end{proof}

\begin{enonce}{Theorem-Definition}[cf.~\cite{boalch_doucot_rembado_2025_twisted_local_wild_mapping_class_groups_configuration_spaces_fission_trees_and_complex_braids}, Thm.~3.30]
	\label{thm:type_BC_configuration_space_from_realizations}

	There is a \emph{homeomorphism}
	\begin{equation}
		\bm B_{B\!C,r}(\dot Q) \simeq \bm B_{B\!C}(\dot{\mc T}),
	\end{equation}
	in the notation of Def.~\ref{def:config_space_type_BC},
	where the $B\!C$-\emph{configuration space of} $\dot{\mc T}$ is defined by
	\begin{equation}
		\label{eq:config_space_from_BC_tree}
		\bm B_{B\!C} (\dot{\mc T})
		\ceqq  \Set{ c \cl \mb A \to \mb C |  c \text{ is a realization of } \dot{\mc T} }.
	\end{equation}
	(Regard the latter as a topological subspace of $\mb C^{\mb A} \simeq \mb C^{\,\abs{\mb A}}$,
	cf.~Rmk.~\ref{rmk:ordering_vertices}.)
\end{enonce}

\begin{proof}
	This follows from Thm.~\ref{thm:admissible_deformation_iff_same_fission_data_type_BC} +~Lemm.~\ref{lem:equivalence_tree_fission_datum_BC}--\ref{lem:any_Q_comes_from_realization_type_BC}:
	any $B\!C$-admissible deformation $\dot Q'$ of $\dot Q$ yields a realization $c' = c_{\dot Q'} \in\mb C^{\,\abs{\mb A}}$ of $\dot{\mc T}$,
	in bijective fashion.
	(We omit the proof that this operation,
	and its inverse,
	are continuous.)
\end{proof}

\subsubsection{}

The next aim is to obtain a (topological) factorization of the space~\eqref{eq:config_space_from_BC_tree},
along the vertices of the fission trees,
refining that of Thm.-Def.~\ref{thm:general_pure_twisted_deformation_space}.

\begin{defi}
	\label{def:local_conf_space_bc}

	Let $v \in \mb V$ be a vertex of $\dot{\mc T}$,
	and $\on{Ch}_{\mb A}(v) \ceqq \mb A \cap \on{Ch}(v)$ the set of admissible child-vertices of $v$.
	The \emph{local} $B\!C$-\emph{configuration space of} $(\dot{\mc T},v)$ is the topological space
	\begin{equation}
		\label{eq:local_conf_space_bc}
		\bm B_{B\!C}(\dot{\mc T},v)
		\ceqq \Set{ c_v \cl \on{Ch}_{\mb A}(v) \to \mb C | c_v \text{ satisfies the conditions of Cor.~\ref{cor:type_BC_characterisation_of_realizations}} }.
	\end{equation}
	(View this as a topological subspace of $\mb C^n$,
	where $n = n(v) \ceqq \abs{\on{Ch}_{\mb A}(v)} \geq 0$.)
\end{defi}

\begin{coro}[cf.~\cite{boalch_doucot_rembado_2025_twisted_local_wild_mapping_class_groups_configuration_spaces_fission_trees_and_complex_braids}, Cor.~3.31]
	\label{cor:config_product_decomposition_type_BC}

	There is a \emph{homeomorphism}
	\begin{equation}
		\bm B_{B\!C,r}(\dot Q)
		\simeq \prod_{v \in \mb V} \bm B_{B\!C}(\dot{\mc T},v),
	\end{equation}
	endowing the target with the product topology.
\end{coro}

\begin{proof}
	This follows Thm.-Def.~\ref{thm:type_BC_configuration_space_from_realizations}:
	the branching conditions of Cor.~\ref{cor:type_BC_characterisation_of_realizations} are \emph{independent} at each vertex.
\end{proof}

\subsubsection{}
\label{sec:local_conf_space_classification}

To classify the possible factors,
let also $n^+ = n^+(v) \geq 0$ be the number of \emph{nonempty} child-vertices of $v$.
	Nonempty child-vertices are admissible,
	and $\on{Ch}(v)$ contains at most one empty vertex (cf.~Rmk.~\ref{rmk:about_def_trees_BC}),
	so that $n^+(v) \in \set{ n(v)-1, n(v) }$;
	the case $n^+(v) = n(v)-1$ is that of the unique empty child-vertex $i_0$,
    whose coordinate $c_{i_0}$ is forced to vanish.

Now observe that the space~\eqref{eq:local_conf_space_bc} is a point if $v$ has no nonempty child-vertex.
Otherwise,
suppose that $v$ has $n^+ \geq 1$ nonempty child-vertices,
with common partial ramification order $N \geq 1$ (cf.~Rmk.~\ref{rmk:realization_check}).
Then $\bm B_{B\!C}(\dot{\mc T},v)$ is homeomorphic to one of the hyperplane complements $\mc M(1,n^+)$,
$\mc M^\sharp(N,n^+)$,
or $\mc M^\sharp(2N,n^+)$,
in the notation of~\eqref{eq:standard_complement_type_D}--\eqref{eq:standard_complement}.

\begin{exem}
	Consider the pointed irregular type $\dot Q$ of Exmp.~\ref{ex:example_fission_tree}.
	From the fission tree of Fig.~\ref{fig:example_tree_type_BC} one directly reads the topological factorization
	\begin{equation}
		\bm B_{B\!C,r}(\dot Q)
		\simeq \mb C \ts (\mb C^\ast)^5 \ts \mc M^\sharp(4,2). \qedhere
	\end{equation}
\end{exem}

\begin{rema}
	\label{rmk:forest_bc}

	As in~\cite[\S~3.7]{boalch_doucot_rembado_2025_twisted_local_wild_mapping_class_groups_configuration_spaces_fission_trees_and_complex_braids},
	with a view towards the global case,
	one can define a $B\!C$-\emph{fission forest} $\bm F$ as a multiset of isomorphism classes of	$B\!C$-fission trees of one and the same rank $m \geq 1$.\fn{
		The \emph{rank} of a fission tree $\mc T$ is well-defined as the rank of the irregular class $\Theta_c$,
		for any choice of a realization $c$ of $\mc T \simeq \mc T(\Theta_c)$.}~In turn,
	given a wild curve $\bm X = (X,\bm x,\bm\Theta)$ with structure group $G \in \set{\SO_{2m+1},\Sp_{2m}}$,
	one can declare that the $B\!C$-\emph{topological skeleton of} $\bm X$ is the pair $\bigl(\on{genus}(X),\bm F\bigr)$,
	where $\bm F$ is the $B\!C$-fission forest underlying the $\bm x$-tuple of $B\!C$-fission trees
	\begin{equation}
		\bigl( \mc T_x \ceqq \mc T ( \Theta_x ) \bigr)_{\! x \in \bm x}.
	\end{equation}
	Analogously to op.~cit.,
	one can prove that two such wild curves are mutual admissible deformations if and only if they have the same topological skeleta.
\end{rema}

\subsection{General case:
	Weyl groups}
\label{sec:weyl_group_tree_bc}

Until \S~\ref{sec:D_trees},
let $\dot Q$ be a compatible pointed $r$-ramified irregular type with irregular class $\Theta = \Theta(\dot Q)$;
set $\dot{\mc T} \ceqq \dot{\mc T}(\dot Q)$ and $\mc T \ceqq \mc T(\Theta)$.

Up to untwisting (cf.~Rmk.~\ref{rmk:config_spaces_are_the_same}),
the admissible deformations of $\Theta$ amount to the base of a Galois covering:
its total space is that of admissible deformations of $\dot Q$,
and the deck transformations correspond to the free action of a subquotient of the Weyl group of type $B\!C$ (cf.~\S\S~\ref{sec:nonpure_case_generic}--\ref{sec:nonpure_case_general}).
Here we provide a different,
more explicit description of the full deformation space,
and of the Galois group of this covering.

\subsubsection{}

Namely,
we will present the deformation space~\eqref{eq:configuration_space_BC_quotient} as a different topological quotient,
in terms of the action of fission-tree automorphisms,
complemented by `interior' groups associated with the full branches.

Choose first a (labelled) leaf $v_i \ceqq \psi(i) \in \mb V_0$,
for $i \in \set{1,\dc,p}$ (cf.~Rmk.~\ref{rmk:ordering_vertices}),
and denote by $r_i \ceqq \on{ram}(\bm L_i) \geq 1$ the ramification order of the corresponding $B\!C$-level datum $\bm L_i = h(\mb L_i) \sse \mb Q_{>0}$.
Then:

\begin{defi}
	\label{def:interior_weyl_group_leaf_type_BC}

	The \emph{type}-$B\!C$ \emph{interior Weyl group of} $v_i$ is the abelian group defined by
	\begin{equation}
		\label{eq:interior_weyl_group}
		W_{B\!C}(\dot{\mc T},i)
		\ceqq
		\begin{cases}
			\bigl( \mb Z \bs r_i \mb Z \bigr) \ts \mb Z^{\ts}, & \quad \bm L_i \text { is nonempty and nonspecial}, \\
			\mb Z \bs r_i \mb Z,                               & \quad \bm L_i \text { is nonempty and special},    \\
			(0),                                               & \quad \bm L_i = \vn.
		\end{cases}
	\end{equation}
\end{defi}

\subsubsection{}
\label{sec:interior_weyl_elements}

We shall write---%
in any case---%
an element of~\eqref{eq:interior_weyl_group} as a pair
\begin{equation}
	g_i
	= (d_i, \varepsilon_i),
	\qquad d_i \in \mb Z \bs r_i \mb Z,
\end{equation}
where:
(i) $\varepsilon_i \in \mb Z^{\ts}$ if $\bm L_i$ is nonempty and nonspecial;
(ii) $\varepsilon_i \ceqq 1$ if $\bm L_i$ is nonempty and special;
and (iii) $(d_i, \varepsilon_i) \ceqq (0,1)$ if $\bm L_i = \vn$.

\subsubsection{}
\label{sec:weyl_group_tree_BC_continued}

By Rmk.~\ref{rmk:tree_autos_embedding},
there is a group embedding $\Aut(\mc T) \hra \mf S_{\mb V_0} \simeq \mf S_p$.
Moreover,
if two full branches are isomorphic,
then the corresponding interior Weyl groups~\eqref{eq:interior_weyl_group} are the same,
and they can be swapped by a fission-tree automorphism.
Thus,
the group $\Aut(\mc T)$ acts on the direct product of interior Weyl groups by permuting the identical factors.

Consider first the corresponding semidirect product
\begin{equation}
	\label{eq:group_acting_on_type_BC_extended_tree}
	\Aut(\mc T) \lts \prod_{i = 1}^p W_{B\!C}(\dot{\mc T},i).
\end{equation}
While this group acts freely on pointed irregular types (cf.~the proof of Thm.~\ref{thm:full_wmcg_from_tree_type_BC}),
it does \emph{not} preserve compatible ones,
and so it does not stabilize the pure deformation spaces.
Therefore,
we also identify the subgroup that achieves that.

Namely,
let $v_i \neq v_j \in \mb V_0$ be two distinct leaves.
If $v_{ij} \in \mb V$ is the nearest common ancestor-vertex of $v_i$ and $v_j$ (where the corresponding full branches meet),
set
\begin{equation}
	\label{eq:common_ramification}
	r_{ij}
	\ceqq \on{ram} \bigl( \bm L_i\cap \bigl[ h(v_{ij}), +\infty \bigr) \bigr)
	= \on{ram} \bigl( \bm L_j \cap \bigl[ h(v_{ij}), + \infty \bigr) \bigr).
\end{equation}

Then:
\begin{defi}[cf.~\cite{boalch_doucot_rembado_2025_twisted_local_wild_mapping_class_groups_configuration_spaces_fission_trees_and_complex_braids}, Def.~4.5]
	\label{def:weyl_group_tree_type_BC}

	The $B\!C$-\emph{Weyl group of} $\dot{\mc T}$ is the subgroup of
	\eqref{eq:group_acting_on_type_BC_extended_tree} defined by
	\begin{equation}
		\label{eq:weyl_group_tree_bc}
		W_{B\!C}(\dot{\mc T})
		= \Aut(\mc T) \lts W'_{B\!C}(\dot{\mc T}),
	\end{equation}
	where in turn
	\begin{equation}
		W'_{B\!C}(\dot{\mc T})
		\ceqq \Set{ (g_i)_i \in \prod_{i = 1}^p W_{B\!C}(\dot{\mc T},i) | \varepsilon_i\varepsilon_j = \zeta_{r_{ij}}^{d_i - d_j} \in \mb C^{\ts},
			\text{ for } i \neq j \in \set{1,\dc,p} }.
	\end{equation}
\end{defi}

\begin{theo}[cf.~\cite{boalch_doucot_rembado_2025_twisted_local_wild_mapping_class_groups_configuration_spaces_fission_trees_and_complex_braids}, Thm.~4.6]
	\label{thm:full_wmcg_from_tree_type_BC}

	\leavevmode
	\begin{enumerate}
		\item
		      The Weyl group~\eqref{eq:weyl_group_tree_bc} acts \emph{freely} on $\bm B_{B\!C,r}(\dot Q)$,
		      preserving irregular classes.

		\item
		      And there is a \emph{homeomorphism}
		      \begin{equation}
			      \bm B_{B\!C,r}(\Theta) \simeq \bm B_{B\!C,r}(\dot Q) \bs W_{B\!C}(\dot{\mc T}),
		      \end{equation}
		      in the notation of~\eqref{eq:configuration_space_BC_quotient}.

	\end{enumerate}
\end{theo}

\begin{proof}

	For the first statement,
	let us regard pointed irregular types as full irregular types (cf.~Rmk.~\ref{rmk:pointed_to_full}):
	this involves lists $l^\pm_1,\dc,l^\pm_p$ of exponential factors.

	In the tame/nonspecial case,
	let an element $g_i = (d_i,\varepsilon_i) \in W_{B\!C}(\dot{\mc T},i)$ act on a list~\eqref{eq:nonspecial_list} by
	\begin{equation}
		\label{eq:nonspecial_list_action}
		g_i \cdot l^+_i
		\ceqq \varepsilon_i \bigl( q_i^{(d_i)}, \dc, q_i^{(d_i)}, \dc,\dc, q_i^{(r_i - 1 + d_i)}, \dc, q_i^{(r_i - 1 + d_i)} \bigr).
	\end{equation}
	In the special case,
	instead,
	an element $g_i = (d_i,1)$ should act on~\eqref{eq:special_list} by
	\begin{equation}
		\label{eq:special_list_action}
		g_i \cdot l^-_i
		\ceqq \bigl( q_i^{(d_i)}, \dc, q_i^{(d_i)}, \dc,\dc, q_i^{(r_i \slash 2 - 1 + d_i)}, \dc, q_i^{(r_i \slash 2 - 1 + d_i)} \bigr).
	\end{equation}
	Both actions are free,
	and one gets a free action of the group~\eqref{eq:group_acting_on_type_BC_extended_tree} on the set of (pointed) irregular types of the form $\dot Q_0 = (l^\pm_1,\dc,l^\pm_p,-l^\pm_1,\dc,-l^\pm_p)$.
	Namely,
	if
	\begin{equation}
		g
		= \rho \cdot  \bm g,
		\qquad
		\rho \in \Aut(\mc T),
		\quad \bm g
		= (g_i)_i \in \prod_{i = 1}^p W_{B\!C}(\dot{\mc T},i),
	\end{equation}
	then one maps $\dot Q_0 \mt g \cdot \dot Q_0$,
	where
	\begin{equation}
		\label{eq:transformed_pointed_irr_type}
		g \cdot \dot Q_0
		\ceqq \bigl( g_{1_\rho} \cdot l^\pm_{1_\rho},\dc, g_{p_\rho} \cdot l^\pm_{p_\rho},-g_{1_\rho} \cdot l^\pm_{1_\rho}, \dc, -g_{p_\rho} \cdot l^\pm_{p_\rho} \bigr),
	\end{equation}
	writing $i_\rho \ceqq \rho^{-1}(i) \in \set{1,\dc,p}$.
	While this action preserves irregular classes,
	in general~\eqref{eq:transformed_pointed_irr_type} is \emph{not} an admissible deformation of $\dot Q_0$;
	namely,
	this happens if and only if the pointed irregular type~\eqref{eq:transformed_pointed_irr_type} is compatible.
	In turn,
	one can prove that the latter holds if and only if $\bm g \in W'_{B\!C}(\dot{\mc T})$ (thereby retroactively motivating Def.~\ref{def:weyl_group_tree_type_BC},
	cf.~\cite[Lem.~4.7]{boalch_doucot_rembado_2025_twisted_local_wild_mapping_class_groups_configuration_spaces_fission_trees_and_complex_braids}).

	For the second statement,
	somewhat conversely,
	suppose that $\dot Q_0$ and $\dot Q_0'$ are both pointed/compatible and such that:
	(i) $\dot Q_0 \sim_{B\!C} \dot Q_0'$;
	and (ii) $\Theta (\dot Q_0) = \Theta (\dot Q_0')$.
	If $\dot Q_0 = \bigl( (n_1,q_1),\dc,(n_p,q_p) \bigr)$,
	for integers $p,n_i \geq 1$ and exponential factors $q_i$,
	then by hypothesis $\dot Q_0' = \bigl( (n_1,q'_1),\dc,(n_p,q'_p) \bigr)$,
	for another list of exponential factors $q'_i$.
	Now,
	by Cor.~\ref{cor:adm_def_by_trees},
	up to identifying the labelled fission trees $\dot{\mc T}_0$ and $\dot{\mc T}'_0$ of $\dot Q_0$ and $\dot Q'_0$ (respectively),
	there exists an automorphism $\rho$ of the underlying unlabelled fission tree $\mc T_0$ (of $\dot{\mc T}_0$) such that
	\begin{equation}
		\braket{ \pm q_i }
		= \rho \bigl( \braket{\pm q'_i} \bigr) \in \mb V_0,
		\qquad i \in \set{1,\dc,p},
	\end{equation}
	viewing the leaves as the Stokes-circles-up-to-sign of $\dot Q_0$.
	Thus,
	there are also indices $j = j_i$ such that $q'_i = \pm \rho \bigl( q_i^{(j)} \bigr)$,
	and this means precisely that $\dot Q_0' = g \cdot \dot Q_0$ for a suitable element of the semidirect product~\eqref{eq:group_acting_on_type_BC_extended_tree}.
	Finally,
	the fact that $\dot Q_0'$ is compatible implies that $g \in W_{B\!C}(\dot{\mc T})$.
\end{proof}

\begin{coro}
	\label{cor:weyl_groups_are_the_same}

	Let $\wh Q$ be the $r$-Galois-closed untwisting of $\dot Q$,
	and $\bm\phi \ceqq \bm\phi \bigl( \wh Q \bigr)$ the Levi-subsystem filtration underlying $\wh Q$.
	Then there is a \emph{group isomorphism}
	\begin{equation}
		Z_{W_{B\!C},\bm\phi} (g) \simeq W_{B\!C}(\dot{\mc T}),
		\qquad g \in \mb L_r \bigl( \wh Q \bigr),
	\end{equation}
	in the notation of~\eqref{eq:general_galois_covering},
	writing $W_{B\!C} = W_{B\!C}(m)$---%
	where $m \geq 1$ is the rank of $\dot Q$.
\end{coro}

\begin{proof}[Proof postponed to~\ref{proof:cor_weyl_groups_are_the_same}]
\end{proof}

\begin{rema}
	\label{rmk:cabled_weyl}

	There is an explicit recursive expression for the Weyl group of the tree,
	in terms of its maximal proper subtrees:
	the precise statement,
	which we omit,
	is an analogue of~\cite[Thm.~4.10]{boalch_doucot_rembado_2025_twisted_local_wild_mapping_class_groups_configuration_spaces_fission_trees_and_complex_braids}.
	(Cf.~\S~\ref{sec:weyl_group_tree_d} in type $D$.)

	In particular,
	in view of Cor.~\ref{cor:weyl_groups_are_the_same},
	fission trees provide a description of the Galois groups of the Galois coverings of \S~\ref{sec:nonpure_case_general} in terms of nested semidirect factorizations,
	which we regard as a cabled/operadic generalization of~\eqref{eq:splitting_parabolic_normalizers}.
\end{rema}

\subsubsection{}

The results of this section yield a proof of Thm.~\ref{thm:thm_3_intro} when $\mf g$ is of type $B$ or $C$.
In the next \S~\ref{sec:D_trees},
instead,
we shall deal with the type-$D$ examples,
thereby concluding in one go the proofs of the two main Thmm.~\ref{thm:thm_2_intro} +~\ref{thm:thm_3_intro}.

\section{Twisted fission trees of type D}
\label{sec:D_trees}

\subsection{Irregular types/classes}
\label{sec:pseudo_irr_classes}

To treat the most complicated classical case,
we first introduce a weaker notion of irregular classes.

\subsubsection{}

Choose an integer $m \geq 1$.
A \emph{full} $D$-\emph{irregular type of rank} $m$ (or just a `full $D_m$-irregular type') is a $W_D(m)$-Galois-closed list of exponential factors;
i.e.,
it is a tuple
\begin{equation}
	Q
	= (q_1, \dc, q_m, -q_1, \dc, -q_m),
\end{equation}
such that there exists $g \in W_D(m)$ satisfying $\sigma(Q) = g(Q)$,
where $\sigma$ is the monodromy of the exponential local system (cf.~Def.~\ref{def:full_irr_type_bc}).
In view of~\eqref{eq:weyl_type_BC}--\eqref{eq:weyl_type_D},
this can be spelt out as in~\eqref{eq:monodromy_realization_BC},
where moreover $\varepsilon_1 \dm \varepsilon_ m = 1$.
(Then $g$ \emph{generates the Galois-orbit of} $Q$,
etc.)

The crux of the matter is the new latter parity condition.
Nonetheless,
we---%
still---%
regard a type-$D$ irregular class $\Theta = \Theta(Q)$ as the $D$-\emph{equivalence class} of a full irregular type $Q$ (analogously to~Def.~\ref{def:bc_equivalence}).
Then their admissible deformations read as follows,
in view of~\eqref{eq:type_d_roots}:

\begin{defi}[cf.~Def.~\ref{def:numerical_equivalence_bc}]
	\label{def:numerical_equivalence_d}

	Consider two full $D$-irregular types
	\begin{equation}
		Q = (q_1,\dc, q_m, -q_1, \dc, -q_m),
		\qquad Q' = (q'_1,\dc, q'_{m'}, -q'_1, \dc, -q'_{m'}),
	\end{equation}
	of ranks $m,m' \geq 1$---%
	respectively.
	Then:
	\begin{enumerate}
		\item
		      $Q$ and $Q'$ are \emph{mutual} $D$-\emph{admissible deformations},
		      which is symbolized by $Q \sim_D Q'$,\fn{
                  The notation is unambiguous,
              as the relation $\sim_D$ defined by a Levi subsystem of $\Phi_D(m)$ is on the finite set $\ul m^\pm$ (cf.~\S~\ref{sec:type_D_relative_refl_group}.}
		      if:
		      \begin{enumerate}
			      \item
			            one has $m = m'$;

			      \item
			            there exists $g \in W_D(m)$ generating the Galois-orbits of $Q$ and $Q'$;

			      \item
			            and one has
			            \begin{equation}
				            \on{slope}(q_i \pm q_j)
				            = \on{slope}(q'_i\pm q'_j),
				            \qquad i \neq j \in \set{1,\dc,m};
			            \end{equation}
		      \end{enumerate}

		\item
		      and two $D$-irregular classes $\Theta$ and $\Theta'$ are \emph{mutual} $D$-\emph{admissible deformations},
		      which is symbolized by $\Theta \sim_D \Theta'$,
		      if there exist full $D$-irregular types $Q$ and $Q'$ such that:
		      \begin{enumerate}
			      \item
			            one has $\Theta = \Theta(Q)$ and $\Theta' = \Theta(Q')$;

			      \item
			            and one has $Q \sim_D Q'$.
		      \end{enumerate}
	\end{enumerate}
\end{defi}

\subsubsection{}
\label{sec:pseudo_irr_class}

As in type $B\!C$,
a $D$-irregular class $\Theta$ determines linear combinations of Stokes-circles-up-to-sign.
The caveat is that,
conversely,
one might need the choice of a `global' sign in order to recover the irregular class from such data.

To state this precisely,
consider the following:

\begin{defi}
	\label{def:type_D_pseudo_irreg_class}

	A \emph{pseudo-}$D_m$-\emph{irregular class} is a linear combination of pairwise-distinct Stokes-circles-up-to-sign
	\begin{equation}
		\wt\Theta
		= \sum_{i = 1}^p n_i \cdot (\pm I_i),
		\qquad n_i \in \mb Q_{> 0},
	\end{equation}
	for some integer $p \geq 1$,
	such that:
	\begin{enumerate}
		\item
		      one has $n_i \in \mb Z_{> 0}$ if $\pm I_i$ is nonspecial or tame;

		\item
		      one has $n_i \in \frac 1 2\mb Z_{> 0}$ if $\pm I_i$ is special;

		\item
		      one has $m = \sum_i n_i \cdot \on{ram}(\pm I_i)$;

		\item
		      and moreover $\wt\Theta$ is $D$-\emph{compatible},
		      i.e.:
		      \begin{enumerate}
			      \item
			            either $\wt\Theta$ contains the tame circle;

			      \item
			            or one has $\sum_{i = 1}^p n_i \in \mb Z$.
		      \end{enumerate}
	\end{enumerate}
\end{defi}

\begin{prop}
	\label{prop:form_type_D_irreg_classes}

	Let $\Theta$ (resp.,
	$\wt\Theta$) be a $D_m$-irregular class (resp.,
	a pseudo-$D_m$-irregular class).
	Then:
	\begin{enumerate}
		\item
		      the irregular class $\Theta$ canonically determines a pseudo-$D_m$-irregular class;

		\item
		      and conversely,
		      the set of irregular classes with pseudo-irregular class $\wt\Theta$ is:
		      \begin{enumerate}
			      \item
			            a singleton,
			            if $\wt\Theta$ contains the tame circle;

			      \item
			            and a $\mb Z^{\ts}$-torsor,
			            otherwise.
		      \end{enumerate}
	\end{enumerate}
\end{prop}

\begin{proof}[Postponed to \ref{proof:prop_form_type_D_irreg_classes}]
\end{proof}

\begin{rema}
	The map~\eqref{eq:sign_function} in~\ref{proof:prop_form_type_D_irreg_classes} depends on the choice of representatives for the Stokes circles,
	but the $\mb Z^{\ts}$-action does \emph{not}.
	Similarly,
	the negation involves choosing an element $g \in W_{B\!C}(m)$ which lies in the nontrivial $W_D(m)$-coset,
	whose class in $\mb Z^{\ts} \simeq W_{B\!C}(m) \bs W_D(m)$ is well-defined.
\end{rema}

\subsection{Pointed (quasi-)irregular types and quasi-irregular classes}
\label{sec:pointed_types_D}

Again we will consider a smaller number of deformation parameters,
in the notion of pointed irregular types;
but we will also need to introduce a weaker variant thereof.

\subsubsection{}

Namely,
as in types $A$ and $B\!C$,
we will proceed by:
(i) breaking irregular types/classes into elementary pieces;
(ii) analyzing their internal structure;
and (iii) understanding how to glue them together.
The caveat however is that the `global' condition of $D$-compatibility of Def.~\ref{def:type_D_pseudo_irreg_class}~(4.) is \emph{not} preserved upon looking at such elementary pieces.
Therefore,
we introduce the following definition (and later suitably `enhance' it,
cf.~\S~\ref{sec:enhancements}):

\begin{defi}
	\label{def:pointed_quasi_irr_types}

	Let $p \geq 1$ be an integer.
	Then:
	\begin{enumerate}
		\item
		      a \emph{pointed quasi}-$D_m$-\emph{irregular type} is an ordered list
		      \begin{equation}
			      \label{eq:pointed_irr_type_D}
			      \dot Q
			      = \bigl( (n_1, q_1), \dc, (n_p, q_p) \bigr),
		      \end{equation}
		      satisfying the properties of~Def.~\ref{def:pointed_irreg_type_BC};

		\item
		      the \emph{quasi}-$D_m$-\emph{irregular class of} $\dot Q$ is the linear combination
		      \begin{equation}
			      \label{eq:quasi_irr_class}
			      \wt\Theta
			      = \wt\Theta(\dot Q)
			      \ceqq \sum_{i = 1}^p n_i \cdot (\pm I_i),
			      \qquad \pm I_i
			      \ceqq \braket{\pm q_i};
		      \end{equation}

		\item
		      a pointed quasi-$D_m$-irregular type~\eqref{eq:pointed_irr_type_D} is a \emph{pointed} $D_m$-\emph{irregular type} if~\eqref{eq:quasi_irr_class} is a pseudo-$D_m$-irregular class,
		      i.e.,
		      if it is $D$-compatible---%
		      in the sense of Def.~\ref{def:type_D_pseudo_irreg_class}~(4.);

		\item
		      in the latter case,
		      $\wt\Theta$ is the \emph{pseudo}-$D_m$-\emph{irregular class of} $\dot Q$.
	\end{enumerate}
\end{defi}

\subsubsection{}

By definition,
a pointed quasi-$D_m$-irregular type is just a pointed $B\!C_m$-irregular type,
and this change of viewpoint/terminology will prove useful when discussing level data.
In any event,
analogously to Rmk.~\ref{rmk:pointed_to_full},
a pointed $D_m$-irregular type can be viewed as a full $D_m$-irregular type of a specific form---%
denoted by the same symbol.
Analogously,
a pointed quasi-$D_m$-irregular type can be viewed as a \emph{full quasi}-$D_m$-\emph{irregular type},
i.e.~(by definition),
a list of exponential factors satisfying the properties of a full $B\!C_m$-irregular type.

Finally,
a pointed $D_m$-irregular type is \emph{compatible} if it satisfies the truncation conditions of Def.~\ref{def:compatible_pit},
and we extend this notion to pointed quasi-$D_m$-irregular types in the natural way.
(Compatibility works as in type $A$ and $B\!C$,
because it involves acting with \emph{positive} cycles in the Weyl group,
which have no parity restriction.)

\begin{rema}
	\label{rmk:no_pointed_lift}

	Analogously to Lemm.~\ref{lem:full_gives_pointed_BC} +~\ref{lem:pointed_gives_compatible},
	any quasi/pseudo-$D_m$-irregular class admits a compatible pointed lift.
	This does \emph{not} hold in general for a $D_m$-irregular class $\Theta$,
	which may simply fail to admit a pointed representative---%
	compatible or not.\fn{
		The failure of proving the analogue of Lem.~\ref{lem:full_gives_pointed_BC} is an instance of~\cite[Prop.~25]{carter_1972_conjugacy_classes_in_the_weyl_group},
		cf.~\S~\ref{sec:background_type_BCD}.}~Namely,
	if $\wt\Theta$ is the pseudo-$D_m$-irregular class of $\Theta$ (cf.~Prop.~\ref{prop:form_type_D_irreg_classes}),
	then the failure occurs when the pointed $D_m$-irregular type $\dot Q$ with pseudo-irregular class $\wt\Theta = \wt\Theta(\dot Q)$ has the wrong sign (cf.~\eqref{eq:sign_function}).
\end{rema}

(Until \S~\ref{sec:tglwmcgs},
when no type/rank is specified,
we tacitly work in type $D$ and rank $m$.)

\begin{rema}
	It will be often convenient to use \emph{integer} multiplicities for Stokes-circles-up-to-sign (cf.~Rmk.~\ref{rmk:convention_integer_multiplicities_BC}).
	In that case,
	write
	\begin{equation}
		\wt\Theta
		= \sum_{i = 1}^p m_i \cdot (\pm I_i), \qquad
		\qquad \dot Q
		= \bigl( (m_1, q_1), \dc, (m_p, q_p) \bigr),
	\end{equation}
	with:
	(i) $m_i \ceqq 2n_i \in \mb Z_{> 0}$,
	if $q_i$ is special;
	and (ii) $m_i \ceqq n_i$,
	otherwise.
\end{rema}

\subsubsection{}
\label{sec:full_to_pointed_D}

Now one can prove that $D$-admissible deformations preserve the subspace of pointed irregular types,
viewed as full ones (cf.~Rmk.~\ref{rmk:pointed_is_enough}).
In particular,
the corresponding relation for pointed irregular types is also symbolized by $\dot Q \sim_D \dot Q'$;
and the subspaces of compatible ones are invariant under $D$-admissible deformations.
We say that two pseudo-irregular classes $\wt \Theta$ and $\wt\Theta'$ are \emph{mutual} $D$-\emph{admissible deformations} if this holds for two pointed irregular types representing them:
this is symbolized by $\wt\Theta \sim_D \wt\Theta'$.

\subsection{Enhancements}
\label{sec:enhancements}

Because of the constraint on the global sign,
the naive extension of Def.~\ref{def:numerical_equivalence_d} to full/pointed quasi-irregular types is \emph{not} useful to study the topology of admissible-deformation spaces.

\begin{exem}
	\label{exmp:need_for_enhancement}

	The basic issue goes as follows.
	Choose numbers $a,\wt a \in \mb C$,
	and consider the following pointed quasi-irregular types,
	with the integer-multiplicities convention:
	\begin{equation}
		\dot Q_1
		= (1, a z^{-1}),
		\qquad \dot Q_2
		= (1,0),
		\qquad \dot Q_3
		= (1, \wt a z^{-1/2}).
	\end{equation}
	The corresponding rank-1 full quasi-irregular types are
	\begin{equation}
		Q_1
		= (a z^{-1}, -a z^{-1}),
		\qquad Q_2
		= (0, 0),
		\qquad Q_3
		= (\wt a z^{-1/2}, -\wt a z^{-1/2}),
	\end{equation}
	and all trivially have the same type-$D$ slopes.
	However,
	while $Q_1$ and $Q_2$ are proper $D_1$-irregular types,
	$Q_3$ is \emph{not}:
	there is precisely one sign-change between $Q_3$ and $\sigma(Q_3)$.
\end{exem}

\subsubsection{}

Wanting to make the global sign constraint into a notion which behaves well under breaking into `local' pieces,
we introduce the following objects:

\begin{defi}
	\label{def:enhanced_irr_types}

	Given an integer $p \geq 1$,
	consider a list
	\begin{equation}
		\label{eq:enhanced_quasi_irr_type}
		\ul{\dot Q}
		= \bigl( (m_1, q_1,\varepsilon_1), \dc, (m_p, q_p, \varepsilon_p) \bigr),
	\end{equation}
	where:
	(i) $q_1,\dc,q_p$ are exponential factors with pairwise-distinct Stokes-circles-up-to-sign;
	(ii) $m_1,\dc,m_p \in \mb Z_{> 0}$ are \emph{multiplicities};
	and (iii) $\varepsilon_1,\dc,\varepsilon_p \in \mb Z^{\ts}$ are signs.
	Then:
	\begin{enumerate}
		\item
		      the tuple~\eqref{eq:enhanced_quasi_irr_type} is an \emph{enhanced pointed quasi-irregular type} if one has
		      \begin{equation}
			      \varepsilon_i
			      =
			      \begin{cases}
				      (-1)^{m_i}, & \quad q_i \text{ is special},    \\
				      1,          & \quad q_i \text{ is nonspecial};
			      \end{cases}
		      \end{equation}
		      (By design,
		      there is \emph{no} constraint on $\varepsilon_i$ if $q_i = 0$.)

		\item
		      the \emph{global sign of} $\ul{\dot Q}$ is $\varepsilon(\ul{\dot Q}) \ceqq \prod_{i = 1}^p \varepsilon_i \in \mb Z^{\ts}$,
		      and an enhanced pointed quasi-irregular type $\ul{\dot Q}$ is an \emph{enhanced pointed irregular type} if $\varepsilon(\ul{\dot Q}) = 1$;

		\item
		      if $\ul{\dot Q}$ is an enhanced pointed (quasi-)irregular type,
		      then it is an \emph{enhancement} of the underlying pointed (quasi-)irregular type $\dot Q$---%
		      forgetting the signs;

		\item
		      and $\ul{\dot Q}$ is a \emph{compatible} enhanced pointed (quasi-)irregular type if $\dot Q$ is a compatible pointed (quasi-)irregular type.
	\end{enumerate}
\end{defi}

\subsubsection{}

In brief,
Def.~\ref{def:enhanced_irr_types} encodes the parity of the number of sign-changes between the lists $l_i$ and $\sigma(l_i)$,
in the notation of the proof~\ref{proof:prop_form_type_D_irreg_classes}.

With this terminology,
one can now prove the following:

\begin{lemm}
	\label{lem:enhancements}

	Let $\dot Q =  \bigl( (m_1, q_1), \dc, (m_p, q_p) \bigr)$ be a pointed quasi-irregular type.
	Then:
	\begin{enumerate}
		\item
		      if $\dot Q$ contains the tame circle,
		      then for all $\varepsilon \in \mb Z^{\ts}$ there exists a \emph{unique} enhanced pointed quasi-irregular type $\ul{\dot Q}$ such that:
		      \begin{enumerate}
			      \item
			            $\ul{\dot Q}$ enhances $\dot Q$;

			      \item
			            and one has $\varepsilon(\ul{\dot Q}) = \varepsilon$;
		      \end{enumerate}

		\item otherwise,
		      there exists a \emph{unique} enhanced pointed quasi-irregular type $\ul{\dot Q}$ enhancing $\dot Q$.
	\end{enumerate}
\end{lemm}

\begin{proof}
	Analogous to the proof of Prop.~\ref{prop:form_type_D_irreg_classes}.
\end{proof}

\begin{coro}
	\label{cor:enhancements}

	Any pointed irregular type $\dot Q$ admits a \emph{unique} enhancement $\ul{\dot Q}$ which is an enhanced pointed irregular type.
\end{coro}

\begin{proof}[Proof postponed to~\ref{proof:cor_enhancements}]
\end{proof}

\subsubsection{}

With a view towards the full case,
we also provide the analogous notions modulo the Weyl-group action:

\begin{defi}[cf.~Def.~\ref{def:enhanced_irr_types}]
	\label{def:enhanced_irr_classes}

	Given an integer $p \geq 1$,
	consider a linear combination
	\begin{equation}
		\label{eq:enhanced_quasi_irr_class}
		\ul{\wt\Theta}
		= \sum_{i = 1}^p m_i \cdot (\pm I_i, \varepsilon_i),
	\end{equation}
	where (as usual) $\pm I_1, \dc, \pm I_p$ are pairwise-distinct Stokes-circles-up-to-sign,
	$m_1,\dc,m_p > 0$ are integers,
	and $\varepsilon_1,\dc,\varepsilon_p \in \mb Z^{\ts}$ are signs.
	Then:
	\begin{enumerate}
		\item
		      the element~\eqref{eq:enhanced_quasi_irr_class} is an \emph{enhanced quasi-irregular class} if
		      \begin{equation}
			      \varepsilon_i
			      =
			      \begin{cases}
				      (-1)^{m_i}, & \quad \pm I_i \text{ is special},    \\
				      1,          & \quad \pm I_i \text{ is nonspecial};
			      \end{cases}
		      \end{equation}
		      (Again,
		      with \emph{no} constraint on $\varepsilon_i$ if $\pm I_i = \braket{0}$.)

		\item
		      the \emph{global sign of} $\ul{\wt\Theta}$ is $\varepsilon \bigl( \ul{\wt\Theta} \bigr) \ceqq \prod_{i = 1}^p \varepsilon_i \in \mb Z^{\ts}$,
		      and an enhanced quasi-irregular class $\ul{\wt\Theta}$ is an \emph{enhanced pseudo-irregular class} if $\varepsilon \bigl( \ul{\wt\Theta} \bigr) = 1$;

		\item
		      and if $\ul{\wt\Theta}$ is an enhanced quasi/pseudo-irregular class,
		      then it is an \emph{enhancement} of the underlying quasi/pseudo-irregular class $\wt\Theta$---%
		      forgetting the signs.
	\end{enumerate}
\end{defi}

\subsubsection{}

Now an enhanced pointed quasi-irregular type $\ul{\dot Q}$ determines an enhanced quasi-irregular class $\ul{\wt\Theta} \ceqq \ul{\wt\Theta}(\ul{\dot Q})$,
in the natural way,
and the global signs match up.
In particular,
an enhanced pointed irregular type has a well-defined enhanced pseudo-irregular class.

\begin{rema}
	\label{rmk:enhancements_for_classes}

	As for pointed irregular types,
	a pseudo-irregular class $\wt\Theta$ admits a \emph{unique} enhancement $\ul{\wt\Theta}$ which is an enhanced pseudo-irregular class (cf.~Cor.~\ref{cor:enhancements}).
	Moreover,
	if $\dot Q$ is a pointed irregular type enhanced by $\ul{\dot Q}$,
	then the enhanced pseudo-irregular class of the latter is the unique enhancement of the pseudo-irregular class of the former.
\end{rema}

\subsubsection{}

Finally,
there is now a natural notion of deformations which behaves well under splitting into pieces:

\begin{defi}
	\label{def:enhanced_numerical_equivalence}

	Consider two enhanced pointed quasi-irregular types
	\begin{equation}
		\ul{\dot Q}
		= \bigl( (m_1, q_1,\varepsilon_1), \dc, (m_p, q_p, \varepsilon_p) \bigr),
		\qquad \smash{\ul{\dot Q}}'
		= \bigl( (m'_1, q'_1,\varepsilon'_1), \dc, (m'_{p'}, q'_{p'}, \varepsilon'_{p'}) \bigr),
	\end{equation}
	containing $p,p' \geq 1$ exponential factors---%
	respectively.
	Then:
	\begin{enumerate}
		\item
		      $\ul{\dot Q}$ and $\smash{\ul{\dot Q}}'$ are \emph{mutual}-$D$-\emph{admissible deformations},
		      which is symbolized by $\ul{\dot Q} \sim_D \smash{\ul{\dot Q}}'$,
		      if:
		      \begin{enumerate}
			      \item
			            one has $p = p'$;

			      \item
			            $\ul{\dot Q}$ and $\smash{\ul{\dot Q}}'$ have the same signs,
			            i.e.,
			            $\varepsilon_i = \varepsilon'_i$ for $i \in \set{1,\dc,p}$;

			      \item
			            and the underlying pointed quasi-irregular types $\dot Q$ and $\dot Q'$ are mutual $D$-admissible deformations,
			            in the sense of (the obvious extension to pointed quasi-irregular types of) \S~\ref{sec:full_to_pointed_D};
		      \end{enumerate}

		\item
		      and two enhanced pseudo-irregular classes $\ul{\wt\Theta}$ and $\smash{\ul{\wt\Theta}}'$ are \emph{mutual} $D$-\emph{admissible deformations},
		      which is symbolized by $\ul{\wt\Theta} \sim_D \smash{\ul{\wt\Theta}}'$,
		      if there exist enhanced pointed quasi-irregular types $\ul{\dot Q}$ and $\smash{\ul{\dot Q}}'$ such that:
		      \begin{enumerate}
			      \item
			            one has $\ul{\wt\Theta} = \ul{\wt\Theta}(\ul{\dot Q})$ and $\smash{\ul{\wt\Theta}}' = \ul{\wt\Theta}(\smash{\ul{\dot Q}}')$;

			      \item
			            and one has $\ul{\dot Q} \sim_D \smash{\ul{\dot Q}}'$.
		      \end{enumerate}
	\end{enumerate}
\end{defi}

\subsubsection{}

Albeit perhaps a priori seemingly too strong,
by using $D$-fission trees---%
in noncircular fashion---%
one can show that Def.~\ref{def:enhanced_numerical_equivalence} implies that:

\begin{lemm}
	\label{lem:equivalence_enhanced_non_enhanced}

	Two pointed irregular types (resp.,
	two pseudo-irregular classes) are mutual $D$-admissible deformations if and only if this holds for the corresponding \emph{enhanced} pointed irregular types (resp.,
	\emph{enhanced} pseudo-irregular classes),
	as per Cor.~\ref{cor:enhancements} (resp.,
	Rmk.~\ref{rmk:enhancements_for_classes}).
\end{lemm}

\begin{proof}[Proof postponed to~\ref{proof:lem_equivalence_enhanced_non_enhanced}]
\end{proof}

\subsection{Configuration spaces}

As in \S~\ref{sec:type_BC_configuration_space},
we now (re)define spaces of admissible deformations/configurations.

\subsubsection{}

Let $\ul{\dot Q} = \bigl( (m_1,q_1,\varepsilon_1),\dc,(m_p,q_p,\varepsilon_p) \bigr)$ be a compatible enhanced pointed irregular type (cf.~Def.~\ref{def:enhanced_irr_types}),
with enhanced pseudo-irregular class $\ul{\wt\Theta} \ceqq \ul{\wt\Theta} \bigl( \ul{\dot Q} \bigr)$ (cf.~Def.~\ref{def:enhanced_irr_classes}).
Moreover,
let $\dot Q$ (resp.,
$\wt\Theta$) be the compatible pointed irregular type underlying $\ul{\dot Q}$ (resp.,
the pseudo-irregular class underlying $\ul{\wt\Theta}$).

Now keep the notation of~\eqref{eq:total_ramification}--\eqref{eq:katz_rank} for the total ramification $r = \on{ram}(\dot Q)$ and the Poincaré--Katz rank $K = \on{Katz}(\dot Q)$ of $\dot Q$---%
but using integer multiplicities.
Upon untwisting,
any enhanced $D$-admissible deformation of $\ul{\dot Q}$ has the form
\begin{equation}
	\label{eq:naivecoeffs_d}
	\ul{\dot Q}_{\bm a}
	= \bigl( (m_1,\wh q_{1,\bm a},\varepsilon_1), \dc,
	(m_p , \wh q_{p,\bm a}, \varepsilon_p) \bigr),
	\qquad \wh q_{i,\bm a}
	= \sum_{j = 1}^s a_{ij} w^{-j},
\end{equation}
for a \emph{unique} collection of numbers $\bm a = (a_{ij})_{ij} \in \mb C^{ps}$,
where as usual $s \ceqq rK$ and $w \ceqq z^{1 \slash r}$.
As per Deff.~\ref{def:numerical_equivalence_d} +~\ref{def:enhanced_numerical_equivalence},
we make the following:

\begin{defi}[cf.~Def.~\ref{def:config_space_type_BC}]
	\label{def:config_space_type_D}

	In the notation of~\eqref{eq:naivecoeffs_d}:
	\begin{enumerate}
		\item
		      the $D$-\emph{deformation space of} $\ul{\dot Q}$ (resp.,
		      \emph{of} $\dot Q$) is the following topological space of $\mb C^{ps}$:
		      \begin{equation}
			      \label{eq:config_space_type_D_enhanced}
			      \bm B_D \bigl( \ul{\dot Q} \bigr)
			      = \bm B_{D,r}^{\leq s,p} \bigl( \ul{\dot Q} \bigr)
			      \ceqq  \Set{ \bm a \in \mb C^{ps} | \ul{\dot Q} \sim_D \ul{\dot Q}_{\bm a} }
		      \end{equation}
		      (resp.,
		      $\bm B_D(\dot Q) = \bm B_{D,r}^{\leq s,p} (\dot Q) \ceqq \set{ \bm a | \dot Q \sim_D \dot Q_{\bm a} }$);

		\item
		      and the $D$-\emph{deformation space of} $\ul{\wt\Theta}$ (resp.,
		      \emph{of} $\wt\Theta$) is the following topological quotient thereof:
		      \begin{equation}
			      \label{eq:config_space_type_D_enhanced_quotient}
			      \bm B_D \bigl( \ul{\wt\Theta} \bigr)
			      = \bm B_{D,r}^{\leq s,p} \bigl( \ul{\wt\Theta} \bigr)
			      \ceqq \Set{ \ul{\wt\Theta}_{\bm a} = \ul{\wt\Theta} \bigl( \ul{\dot Q}_{\bm a} \bigr) | \ul{\wt\Theta} \sim_D \ul{\wt\Theta}_{\bm a } }
		      \end{equation}
		      (resp.,
		      $\bm B_D \bigl( \wt\Theta \bigr) = \bm B_{D,r}^{\leq s,p} \bigl( \wt\Theta \bigr) \ceqq \set{ \wt\Theta_{\bm a} = \wt\Theta (\dot Q_{\bm a}) | \wt\Theta \sim_D \wt\Theta_{\bm a} }$).
	\end{enumerate}
\end{defi}

\begin{rema}
	\label{rmk:config_spaces_are_the_same_d}

	First,
	in view of Lem.~\ref{lem:equivalence_enhanced_non_enhanced},
	there are homeomorphisms
	\begin{equation}
		\label{eq:config_space_type_D}
		\bm B_{D,r}(\dot Q)
		\simeq \bm B_{D,r} \bigl( \ul{\dot Q} \bigr),
		\qquad \bm B_{D,r} \bigl( \wt\Theta \bigr) \simeq \bm B_{D,r} \bigl( \ul{\wt\Theta} \bigr).
	\end{equation}
	Thus,
	as far as the topology of deformation spaces is concerned,
	one can work with either version.

	Second,
	in the pure case,
	the topological space $\bm B_{D,r}(\dot Q)$ also describes the admissible deformation of the $r$-Galois-closed untwisting $\wh Q$ of $\dot Q$ (cf.~Rmk.~\ref{rmk:config_spaces_are_the_same}).
	Compared to type $A$/$B\!C$,
	the caveat is that \emph{not} all full irregular types are $D$-equivalent to pointed ones.
	Nonetheless,
	the latter holds up to the former involution in~\eqref{eq:sign_involution},
	and the $D$-admissible-deformation spaces of $Q$ and $-Q$ are mapped homeomorphically onto each other by that very involution:
	this expresses the fact that $W_{B\!C}(m)$ acts on $\mf t \simeq \wt V_m$ by automorphisms of $\Phi_D(m) = \Phi_B(m) \cap \Phi_C(m)$.
	Therefore,
	restricting to compatible pointed irregular types is w.l.o.g.

	Third,
	in the full case,
	recall that in principle we are interested in the $D$-\emph{deformation space} $\bm B_{D,r}(\Theta)$ of a (non-pseudo) irregular class $\Theta$,
	which can also be formally introduced---%
	as per \S~\ref{sec:full_to_pointed_D}.
	But again,
	up to the latter involution of~\eqref{eq:sign_involution},
	which does not modify the underlying pseudo-irregular class $\wt\Theta$,
	nor the topology of the deformation space,
	there will be a pointed irregular type representing $\Theta$;
	whence a homeomorphism $\bm B_{D,r}(\Theta) \simeq \bm B_{D,r} \bigl( \wt\Theta \bigr)$.
	(Cf.~again Rmk.~\ref{rmk:no_pointed_lift},
	and the end of~\ref{proof:prop_form_type_D_irreg_classes}.)

	Overall,
	we are thus still describing particular instances of the $r$-admissible-deformation spaces of \S~\ref{sec:setup},
	but now covering the topology of \emph{all} the type-$D$ examples (contrary to \S~\ref{sec:type_D}).
\end{rema}

\subsubsection{}

Hereafter,
unless otherwise specified,
all pointed (quasi-)irregular types are \emph{compatible}---%
whether enhanced or not.

Then,
as in \S~\ref{sec:about_slopes},
the slopes that appear in Def.~\ref{def:numerical_equivalence_d} can be classified into interior/exterior ones.
We will study them separately,
considering:
(i) enhanced pointed irregular types having only one Stokes circle (cf.~\S~\ref{sec:one_circle_type_d});
(ii) then two Stokes circles (cf.~\S~\ref{sec:two_circles_type_d});
and (iii) the general case (cf.~\S\S~\ref{sec:general_case_type_d}--\ref{sec:weyl_group_tree_d}).

\subsection{One Stokes circle:
	(enhanced) level data and full branches}
\label{sec:one_circle_type_d}

As in type $B\!C$,
a pointed quasi-irregular type is said to be \emph{elementary} if $p = 1$ in Def.~\ref{def:pointed_quasi_irr_types}.
With the integer-valued convention,
it can be written $\dot Q = (\ul m,q)$,
where $\ul m \in \mb Z_{> 0}$,
and $q$ is an exponential factor.
Its ramification is denoted by $r \ceqq \on{ram}(q)$,
and its rank by $m$.
Analogously,
an \emph{elementary} enhanced pointed quasi-irregular type is of the form $\ul{\dot Q} = (\ul m,q,\varepsilon)$,
where in addition $\varepsilon \in \mb Z^{\ts}$ is a sign.

\begin{rema}[cf.~Cor.~\ref{cor:enhancements}]
	If $q \neq 0$,
	then $\dot Q$ has a unique elementary enhancement $\ul{\dot Q} = (\ul m,q,\varepsilon)$,
	with:
	i) $\varepsilon = 1$,
	if $q$ is nonspecial;
	and ii) $\varepsilon = (-1)^{\ul m}$,
	if $q$ is special.
	Conversely,
	the triples $(\ul m, 0, \pm 1)$ are the elementary enhancements of $\dot Q = (\ul m,0)$.
\end{rema}

\begin{defi}[cf.~Def.~\ref{def:bc_level_data}]
	\label{def:type_d_level_data}

	Let $\dot Q = (\ul m,q)$ be an elementary pointed quasi-irregular type.
	Then:
	\begin{enumerate}
		\item
		      the \emph{naive} $D$-\emph{level datum of} $\dot Q$ is the set
		      \begin{equation}
			      \wt L_D(\dot Q)
			      = \wt L_D(\ul m,q)
			      \ceqq \Set{ \on{slope}(\wt q_i \pm \wt q_j) | i \neq j \in \set{1,\dc,m} } \sm \set{0},
		      \end{equation}
		      where $Q = (\wt q_1, \dc, \wt q_m, -\wt q_1, \dc, -\wt q_m)$ is the full quasi-irregular type corresponding to $\dot Q$;
		\item
		      if $\ul{\dot Q} = (\ul m,q,\varepsilon)$ is an enhancement of $\dot Q$,
		      then its \emph{enhanced naive} $D$-\emph{level datum} is the pair
		      \begin{equation}
			      \wt L_D \bigl( \ul{\dot Q} \bigr)
			      = \wt L_D(\ul m,q,\varepsilon)
			      \ceqq \bigl( \, \wt L_D(\dot Q), \varepsilon \bigr).
		      \end{equation}
	\end{enumerate}
\end{defi}

\begin{lemm}
	\label{lem:type_d_level_data}

	Let $\dot Q = (\ul m,q)$ be an elementary pointed quasi-irregular type with nonzero exponential factor.
	Then:
	\begin{enumerate}
		\item
		      if $r = 1$ and $\ul m = 1$,
		      one has $\wt L_D(\dot Q) = \vn$;

		\item
		      if $r = 2$,
		      $q$ is special,
		      and $\ul m = 1$,
		      one has $\wt L_D(\dot Q) = \vn$;

		\item
		      otherwise:
		      \begin{enumerate}
			      \item
			            one has $\on{slope}(q) \in \wt L_D(\dot Q)$;

			      \item
			            and one has $\wt L_D(\dot Q) = L_{B\!C}(\dot Q) = L_{B\!C}(q) \neq \vn$ (regarding $\dot Q$ as a pointed $B\!C_m$-irregular type).
		      \end{enumerate}
	\end{enumerate}
\end{lemm}

\begin{proof}[Proof postponed to~\ref{proof:lem_type_d_level_data}]
\end{proof}

\begin{coro}
	\label{cor:when_level_datum_empty_type_D}

	Let $\dot Q = (\ul m,q)$ be an elementary pointed quasi-irregular type.
	Then $\wt L_D(Q) = \vn$ if and only if we are in one of the following (mutually-exclusive) situations:
	\begin{enumerate}
		\item[(1)]
		      one has $\ul m = 1$,
		      and:
		      \begin{enumerate}
			      \item $q \neq 0$ is untwisted;

			      \item $q \neq 0$ is special,
			            and $r = 2$;

			      \item or $q = 0$;
		      \end{enumerate}
		\item[(2)]
		      or one has $\ul m \geq 2$ and $q = 0$.
	\end{enumerate}
\end{coro}

\begin{proof}
	This follows from Lem.~\ref{lem:type_d_level_data},
	and from the fact that for any exponential factor $q$ we have $L_{B\!C}(q) = \vn$ if and only if $q = 0$.
\end{proof}

\subsubsection{}

In brief,
apart from two exceptions,
the naive $D$-level datum coincides with that of type $B\!C$.
However,
the enhanced deformations are quite different in the cases $(1)$ and $(2)$ of Cor.~\ref{cor:when_level_datum_empty_type_D}.
Namely:
\begin{enumerate}
	\item two enhanced nonzero pointed irregular types $\ul{\dot Q} = (1, q, \varepsilon)$ and $\smash{\ul{\dot Q}}' = (1, q', \varepsilon')$,
	      whose underlying (non-enhanced) pointed irregular types $\dot Q$ and $\dot Q'$ fall in case $(1)$,
	      are mutual $D$-admissible deformations if and only if $\varepsilon = \varepsilon'$;
	      in particular:
	      \begin{enumerate}
		      \item
		            if $\dot Q$ falls in subcase $(1.a)$,
		            then $\ul{\dot Q}$ has positive sign,
		            so that $\ul{\dot Q} \sim_D (1, 0, 1)$;

		      \item
		            and if $\dot Q$ falls in case $(1.b)$,
		            then $\ul{\dot Q}$ has negative sign,
		            so that $\ul{\dot Q} \sim_D (1, 0, -1)$.
	      \end{enumerate}
	\item conversely,
	      if $\ul m \geq 2$ and $q \neq 0$ then $(\ul m,q, \varepsilon')$ is \emph{not} a $D$-admissible deformations of $(\ul m,0, \varepsilon)$.
	      (I.e.,
	      $\ul{\dot Q} \ceqq (\ul m,0, \varepsilon)$ has no nontrivial deformations.)
\end{enumerate}

Compared to \S~\ref{sec:bc_one_circle},
the multiplicities now play a role in determining the levels,
which is reminiscent of the fact that one \emph{cannot} reduce any type-$D$ example to a quasigeneric one (cf.~\S~\ref{sec:type_D}).
In view of this,
we overwrite particular cases of Def.~\ref{def:type_d_level_data},
so that enhanced level data will characterize enhanced $D$-admissible deformations.

Let us therefore introduce two symbols $\vn_D$ and $\vn_{B\!C}$,
which can be regarded as two different `flavours' of empty sets,
and make the following:

\begin{defi}
	\leavevmode

	\begin{enumerate}
		\item
		      Let $\dot Q = (\ul m,q)$ be an elementary pointed quasi-irregular type.
		      Its $D$-\emph{level datum} is defined as follows:
		      \begin{enumerate}
			      \item
			            if $\dot Q$ falls in case $(1)$ of Cor.~\ref{cor:when_level_datum_empty_type_D},
			            set $L_D(\dot Q) = L_D(\ul m,q) \ceqq \vn_D$;

			      \item
			            if $\dot Q$ falls in case $(2)$ of Cor.~\ref{cor:when_level_datum_empty_type_D},
			            set $L_D(\dot Q) = L_D(\ul m,q) \ceqq \vn_{B\!C}$;

			      \item
			            otherwise,
			            let $L_D(\dot Q) \ceqq \wt L_D(\dot Q)$.
		      \end{enumerate}
		\item
		      Let $\ul{\dot Q} = (\ul m, q, \varepsilon)$ be an enhanced pointed quasi-irregular type.
		      Its \emph{enhanced} $D$-\emph{level datum} $L_D \bigl( \ul{\dot Q} \bigr)$ is defined as follows:
		      \begin{enumerate}
			      \item
			            if $L_D(\ul m,q) = \vn_D$,
			            set $L_D \bigl( \ul{\dot Q} \bigr) = L_D(\ul m,q,\varepsilon) \ceqq \bigl( \vn_D,\varepsilon \bigr)$;

			      \item
			            and if $L_D(\ul m,q) = \vn_{B\!C}$,
			            set $L_D \bigl( \ul{\dot Q} \bigr) = L_D(\ul m,q,\varepsilon) \ceqq \bigl( \vn_{B\!C} ,\varepsilon \bigr)$;

			      \item
			            otherwise,
			            let $L_D \bigl( \ul{\dot Q} \bigr) \ceqq \wt L_D \bigl( \ul{\dot Q} \bigr)$.
		      \end{enumerate}
	\end{enumerate}
\end{defi}

(Note that $L_D \bigl( \ul{\dot Q} \bigr)$ only depends on the triple $(\ul m, \braket{ \pm q}, \varepsilon)$;
hereafter,
we shall \emph{never} use the naive versions.)

\begin{prop}[cf.~Prop.~\ref{prop:numerical_equiv_from_level_data}]
	\label{prop:numerical_equiv_from_level_data_D}

	Let $\ul{\dot Q} = (\ul m,q,\varepsilon)$ and $\smash{\ul{\dot Q}}' = (\ul m',q', \varepsilon')$ be two elementary enhanced pointed quasi-irregular types.
	Then $\ul{\dot Q} \sim_D \smash{\ul{\dot Q}}'$ if and only if:
	\begin{enumerate}
		\item
		      one has $\ul m = \ul m'$;

		\item
		      and one has $L_D \bigl( \ul{\dot Q} \bigr) = L_D \bigl( \smash{\ul{\dot Q}}' \bigr)$.
	\end{enumerate}
\end{prop}

\begin{proof}[Proof omitted]
\end{proof}

\subsubsection{}
\label{sec:enhanced_level_data}

By definition,
a $D$-\emph{level datum} $\bm L$ is either equal to $\vn_D$,
or to $\vn_{B\!C}$,
or it is a nonempty subset of $\mb Q_{>0}$ such that there exists an elementary pointed quasi-irregular type $\dot Q = (\ul m,q)$ satisfying $L_D(\dot Q) = \bm L$.
If $\bm L \neq \vn_D$,
then the \emph{ramification of} $\bm L$ is (well-)defined by $\on{ram}(\bm L) \ceqq \on{ram}(q)$
(e.g.,
$\on{ram}(\vn_{B\!C}) = 1$);
else,
we let $\on{ram}(\bm L) \ceqq 1$.
In turn,
an \emph{enhanced} $D$-\emph{level datum} is a pair $\ul{\bm L} = (\bm L,\varepsilon)$,
consisting of a $D$-level datum $\bm L$ and a sign $\varepsilon \in \mb Z^{\ts}$.

\begin{defi}
	Let $\ul{\bm L} \ceqq (\bm L,\varepsilon)$ be an enhanced $D$-level datum.
	Then:
	\begin{enumerate}
		\item
		      the set $\on{Adm}_D(\ul{\bm L})$ of $D$-\emph{admissible exponents of} $\ul{\bm L}$ is the set of all positive rational numbers $k > 0$ such that there exists an elementary enhanced pointed quasi-irregular type $\ul{\dot Q} = (\ul m, q, \varepsilon)$ satisfying:
		      \begin{enumerate}
			      \item
			            $L_D \bigl( \ul{\dot Q} \bigr) = \ul{\bm L}$;

			      \item
			            and $k\in E(q)$;
		      \end{enumerate}

		\item and the set of $D$-\emph{inconsequential exponents of} $\ul{\bm L}$ is defined by:
		      \begin{enumerate}
			      \item
			            the complement
			            \begin{equation}
				            \on{Inc}_D(\ul{\bm L})
				            \ceqq \on{Adm}_D(\ul{\bm L}) \sm \bm L \sse \mb Q_{> 0},
			            \end{equation}
			            if $\bm L \neq \vn_D$ and $\bm L \neq \vn _{B\!C}$;

			      \item
			            and $\on{Inc}_D(\ul{\bm L}) \ceqq \on{Adm}_D(\ul{\bm L})$,
			            otherwise.
		      \end{enumerate}
	\end{enumerate}
\end{defi}

\begin{lemm}
	\label{lem:admissible_incons_exponents_type_D}

	Let $\ul{\bm L} = (\bm L, \varepsilon)$ be an enhanced $D$-level datum.
	Then:
	\begin{enumerate}
		\item
		      if $\bm L = \vn_D$,
		      one has
		      \begin{equation}
			      \on{Adm}_D(\ul{\bm L})
			      = \on{Inc}_D(\ul{\bm L})
			      =
			      \begin{cases}
				      \mb Z_{> 0},                & \quad \varepsilon = 1;  \\
				      \frac 1 2 + \mb Z_{\geq 0}, & \quad \varepsilon = -1;
			      \end{cases}
		      \end{equation}

		\item
		      if $\bm L = \vn_{B\!C}$,
		      one has $\on{Adm}_D(\ul{\bm L}) = \on{Inc}_D(\ul{\bm L}) = \vn$;

		\item
		      otherwise,
		      one has $\on{Adm}_D(\ul{\bm L}) = \on{Adm}_{B\!C}(\bm L)$ and $\on{Inc}_D(\ul{\bm L}) = \on{Inc}_{B\!C}(\bm L)$.
	\end{enumerate}
\end{lemm}

\begin{proof}
	For $\bm L = \vn_D$,
	the cases $\varepsilon = 1$ and $\varepsilon = -1$ correspond (respectively) to subcases $(1.a)$ and $(1.b)$ of Cor.~\ref{cor:when_level_datum_empty_type_D},
	and the other statements follow from the previous discussion.
\end{proof}

\subsubsection{}

Note that Cor.~\ref{cor:elementary_BC_space} does \emph{not} extend verbatim when involving the `half-empty' $D$-level datum $\bm L = \vn_D$,
as one needs to truncate the admissible exponents in order to bound the Poincaré--Katz rank of admissible deformations.
Rather then stating the exact analogue,
we move on towards the definition of $D$-fission trees,
which will treat the general case:

\begin{defi}[cf.~Def.~\ref{def:full_branch_of_level_datum_type_BC}]
	\label{def:type_d_branch}

	Let $\ul{\bm L} = (\bm L, \varepsilon)$ be an enhanced $D$-level datum.
	The \emph{full} $D$-\emph{branch of} $\ul{\bm L}$ is the quadruple $\mc B_{\ul{\bm L}} = (\mc B,\mb V, \mb A,\mb L)$,
	where $\mc B$ is a copy of $\ol{\mb R}_{\geq 0}$,
	and $\mc B \supseteq \mb V \supseteq \mb A \supseteq \mb L$ are subsets,
	whose elements are called \emph{vertices},
	\emph{admissible vertices},
	and \emph{mandatory vertices},
	respectively,
	defined as follows:
	\begin{enumerate}
		\item if $\bm L \neq \vn_D$,
		      then $(\mc B,\mb A, \mb L)$ is the full $B\!C$-branch of $\bm L$,
		      and $\mb V \ceqq \mb A$;

		\item otherwise,
		      using the given \emph{height function} $h \cl \mc B \lxra{\simeq} \ol{\mb R}_{\geq 0}$,
		      let:
		      \begin{enumerate}
			      \item
			            $\mb L \ceqq \vn$;

			      \item $\mb A \ceqq h^{-1} \bigl( \on{Inc}_D(\ul{\bm L}) \bigr)$;

			      \item and $\mb V \ceqq \mb A$.
		      \end{enumerate}
	\end{enumerate}
	The set $\mb I$ of \emph{inconsequential vertices},
	the \emph{root} $v_\infty$,
	and the \emph{leaf} $v_0$,
	are as in type $B\!C$.
\end{defi}

\begin{rema}
	Later on,
	when gluing,
	we will define slightly modified full branches associated with the level datum $\vn_D$ (cf.~Def.~\ref{def:hybrid_full_branches_type_D}):
	the set $\mb V$ will be \emph{different} from $\mb A$,
	and in order to ensure uniform notation it is thus convenient to include $\mb V$ in the data of Def.~\ref{def:type_d_branch}.
\end{rema}

\subsubsection{}

We still associate a \emph{type} to each edge of a full $D$-branch.
If $\ul{\bm L} \neq \vn_D$,
these are as in type $B\!C$,
while the edges of $\mc B_{\vn_D}$ are the subsets
\begin{equation}
	e_k
	\ceqq h^{-1} \bigl( (k, k+1) \bigr) \sse \mc B,
	\qquad k\in \set{ 0 } \cup\on{Inc}_D(\ul{\bm L}).
\end{equation}
In the latter situation,
each edge is defined to be of type $t(e_k) \ceqq HE$ (as in `\emph{half-empty}').

\begin{exem}
	The full $D$-branches which \emph{do} not appear in type $B\!C$ are drawn in Fig.~\ref{fig:new_type_D_full_branches}.
	Half-empty edges are depicted by loosely dotted lines.
\end{exem}

\begin{figure}
	\begin{center}
		\begin{tikzpicture}
			\tikzstyle{authorized}=[circle,fill=white,minimum size=5pt,draw, inner sep=0pt]
			\tikzstyle{mandatory}=[circle,fill=black,minimum size=5pt,draw,inner sep=0pt]
			\tikzstyle{empty}=[circle,fill=black,minimum size=0pt,inner sep=0pt]
			\tikzstyle{root}=[fill=black,minimum size=5pt,draw,inner sep=0pt]
			\tikzstyle{indeterminate}=[circle,densely dotted,fill=white,minimum size=5pt,draw, inner sep=0pt]
			\begin{scope}
				\node[root] (R) at (0,6.5){};
				\node (X) at (0,5.5){$\vdots$};
				\node (Y) at (0,5){};
				\node[authorized] (A) at (0,4){};
				\node[authorized] (B) at (0,3){};
				\node[authorized] (C) at (0,2){};
				\node (D) at (0,1){};
				\draw[thick, loosely dotted] (R)--(X);
				\draw[thick, loosely dotted] (Y)--(A);
				\draw[thick, loosely dotted] (A)--(B);
				\draw[thick, loosely dotted] (B)--(C);
				\draw[thick, loosely dotted] (C)--(D);
				\draw (-2, 6.5) node {$+ \infty$};
				\draw (-2, 5) node {$4$};
				\draw (-2, 4) node {$3$};
				\draw (-2, 3) node {$2$};
				\draw (-2, 2) node {$1$};
				\draw (0, 0.5) node {$\mc B_{(\vn_D, 1)}$};
			\end{scope}
			\begin{scope}[xshift=6cm]
				\node[root] (R) at (0,6.5){};
				\node (X) at (0,5){$\vdots$};
				\node (Y) at (0,4.5){};
				\node[authorized] (A) at (0,3.5){};
				\node[authorized] (B) at (0,2.5){};
				\node[authorized] (C) at (0,1.5){};
				\node (D) at (0,1){};
				\draw[thick, loosely dotted] (R)--(X);
				\draw[thick, loosely dotted] (Y)--(A);
				\draw[thick, loosely dotted] (A)--(B);
				\draw[thick, loosely dotted] (B)--(C);
				\draw[thick, loosely dotted] (C)--(D);
				\draw (-2, 6.5) node {$+ \infty$};
				\draw (-2, 4.5) node {$9/2$};
				\draw (-2, 3.5) node {$7/2$};
				\draw (-2, 2.5) node {$3/2$};
				\draw (-2, 1.5) node {$1/2$};
				\draw (0, 0.5) node {$\mc B_{(\vn_D, -1)}$};
			\end{scope}
		\end{tikzpicture}
		\caption{Full $D$-branches with level datum $\vn_D$.}
		\label{fig:new_type_D_full_branches}
	\end{center}
\end{figure}

\subsection{Two Stokes circles:
	fission exponents}
\label{sec:two_circles_type_d}

Throughout this section,
consider an enhanced pointed quasi-irregular type with two active exponential factors,
i.e.,
of the form
\begin{equation}
	\ul{\dot Q}
	= \bigl( (\ul m,q, \varepsilon), (\wt{\ul m}, \wt q, \wt\varepsilon \,) \bigr),
\end{equation}
with underlying pointed quasi-irregular type $\dot Q = \bigl( (\ul m,q),(\wt{\ul m},\wt q) \bigr)$.
Set $r \ceqq \on{ram}(q)$ and $\wt r \ceqq \on{ram}(\wt q)$.
Then consider the common/different parts of $q$ and $\wt q$,
denoted by $q_c = \wt q_c$,
$q_d$,
and $\wt q_d$ (cf.~Def.~\ref{def:common_different_part}).
The $D$-\emph{fission exponent of} $q$ \emph{and} $\wt q$ is the number
\begin{equation}
	\label{eq:type_d_fission_exponent}
	f_{q,\wt q}
	\ceqq \max \Set{ \on{slope}(q_d), \on{slope}(\wt q_d) } \in \mb Q_{> 0}.
\end{equation}
It only depends on the underlying Stokes-circles-up-to-sign $\pm I \ceqq \braket{\pm q}$ and $\pm \wt I \ceqq \braket{ \pm\wt q}$.

\begin{lemm}
	\label{lem:slopes_differences_two_circles_type_D}

	\leavevmode

	\begin{enumerate}
		\item
		      The set of exterior slopes,
		      governing the enhanced admissible deformations of $\ul{\dot Q}$ (cf.~Def.~\ref{def:enhanced_numerical_equivalence}),
		      is given by
		      \begin{equation}
			      \Set{ \on{slope} \bigl( q^{(i)} \pm \wt q^{(j)} \bigr) | (i,j) \in \mb Z \bs r\mb Z \ts \mb Z \bs \tilde r \mb Z }
			      = L_{B\!C}(q_c) \cup \set{f_{q,\wt q}} \sse \mb Q_{>0}.
		      \end{equation}

		\item
		      Furthermore,
		      the maps $(i,j) \mt \on{slope} \bigl( q^{(i)} \pm \wt q^{(j)} \bigr)$ are determined by $L_{B\!C}(q_c)$ and $f_{q,\wt q}$.
	\end{enumerate}
\end{lemm}

\begin{proof}
	This follows from Lemm.~\ref{lem:exterior_slopes_two_circles_bc} +~\ref{lem:type_d_level_data} (and it uses the fact that $\dot Q$ is compatible).
	Indeed,
	the truncated pointed quasi-irregular type $\dot Q_c \ceqq \tau_k(\dot Q)$,
	defined in the natural way,
	is elementary and of multiplicity $m_c \geq 2$,
	so that $L_D(\dot Q_c) = L_{B\!C}(q_c)$.
\end{proof}

\begin{prop}[cf.~Prop.~\ref{prop:characterisation_fission_exp_BC}]
	\label{prop:characterisation_admissible_def_two_circles_type_D}

	Let $k \ceqq f_{q,\wt q}$ be the fission exponent of $q$ and $\wt q$.
	Then a \emph{compatible} enhanced pointed quasi-irregular type $\smash{\ul{\dot Q}}'$ satisfies $\smash{\ul{\dot Q}}' \sim_D \ul{\dot Q}$ if and only if it is of the form $\smash{\ul{\dot Q}}' = \bigl( (\ul m,q', \varepsilon'),(\wt{\ul m}, \wt q', \wt\varepsilon') \bigr)$,
	for two exponential factors $q'$ and $\wt q'$ with common part $q'_c = \wt q'_c$,
	such that:
	\begin{enumerate}
		\item
		      one has $L_D(\ul m,q, \varepsilon) = L_D (\ul m,q', \varepsilon')$ and $L_D(\wt{\ul m},\wt q, \wt\varepsilon) = L_D(\wt{\ul m},\wt q', \wt\varepsilon')$;

		\item
		      one has $L_D(\ul m + \wt{\ul m},q_c) = L_D (\ul m + \wt{\ul m},q'_c)$;\fn{
			      This is equivalent to having $L_{B\!C}(q_c) = L_{B\!C}(q'_c)$.}

		\item
		      and one has $f_{q',\wt q'} = k$.
	\end{enumerate}
\end{prop}

\begin{proof}[Proof omitted]
\end{proof}

\begin{prop}
	\label{prop:types_of_fission_D}

	In the notation of~\eqref{eq:types_of_fission_BC}:
	\begin{enumerate}
		\item
		      if $q_c = 0$ and $k \in \frac 1 2 \mb Z_{> 0}$,
		      then $f_{q,\wt q} = k$ if and only if $a \neq \pm\wt a$;

		\item else,
		      the condition for having $f_{q,\wt q} = k$ is the same as in type $B\!C$.
	\end{enumerate}
\end{prop}

\begin{proof}
	The former follows from the definition~\eqref{eq:type_d_fission_exponent},
	and the latter from Prop.~\ref{prop:characterisation_admissible_def_two_circles_type_D}.
\end{proof}

\begin{lemm}[cf.~Prop.~\ref{prop:types_of_fission_BC}]
	In the setting of Prop.~\ref{prop:types_of_fission_D}~(1.),
	setting $\ul{\bm L} = (\bm L, \varepsilon) \ceqq L_D(\ul m, q, \varepsilon)$,
	the fission exponent $k \ceqq f_{q,\wt q}$ satisfies the following:
	\begin{enumerate}
		\item if $\bm L \neq \vn_D$ and $a\neq 0$,
		      then $k$ is mandatory for $\ul{\bm L}$;

		\item
		      if $\bm L\neq \vn_D$ and $a = 0$,
		      then $k$ is \emph{not} admissible for $\ul{\bm L}$;

		\item
		      and if $\bm L = \vn_D$,
		      then $k$ is inconsequential for $\ul{\bm L}$.
	\end{enumerate}
	(And analogously for the enhanced level datum $\bigl( \, \wt{\bm L}, \wt\varepsilon \, \bigr) = L_D(\wt{\ul m}, \wt q, \wt\varepsilon)$.)
\end{lemm}

\begin{proof}
	This follows from Lem.~\ref{lem:admissible_incons_exponents_type_D}.
\end{proof}

\subsubsection{}

Again we can pictorially represent the types of fission for two Stokes circles.
With a view towards $D$-fission trees,
we focus on the examples where $q_c = 0$ and $k \in \frac 1 2 \mb Z_{> 0}$,
cf.~Fig.~\ref{fig:new_types_of_fission_2_circles_type_D}.
Some cases are already present in type $B\!C$,
when the level datum $\vn_D$ is not involved;
and then there are three new possibilities.
In the first of these (in the middle of the figure) one has $a = 0$,
so the condition $f_{q, \wt q} = k$ forces $\wt a \neq 0$,
although $k$ is \emph{not} a mandatory exponent for $\wt q$:
we thus say that $k$ is \emph{indirectly mandatory for} $\wt q$.
(Note also that the fourth case of Fig.~\ref{fig:new_types_of_fission_2_circles_type_D},
from the left,
involves two nonempty sibling-vertices of different types,
a unique feature in type $D$.)

\begin{figure}
	\begin{center}
		\begin{tikzpicture}[scale=0.95]
			\tikzstyle{mandatory}=[circle,fill=black,minimum size=6pt,draw, inner sep=0pt]
			\tikzstyle{authorised}=[circle,fill=white,minimum size=6pt,draw,inner sep=0pt]
			\tikzstyle{empty}=[circle,fill=white,minimum size=0pt,inner sep=0pt]
			\begin{scope}
				\node[empty] (B1) at (1,0){};
				\node[mandatory] (B2) at (2,0){};
				\node[empty] (A) at (1.5,1){};
				\draw[thick, dotted] (A)--(B1);
				\draw[thick, dotted]  (A)--(B2);
				\draw (0.9,-0.5) node {\scriptsize $\bm L\neq\vn_D$};
				\draw (2.1,-0.46) node {\scriptsize $\wt{\bm L}\neq\vn_D$};
			\end{scope}
			\begin{scope}[xshift=2.7cm]
				\node[mandatory] (B1) at (1,0){};
				\node[mandatory] (B2) at (2,0){};
				\node[empty] (A) at (1.5,1){};
				\draw[thick, dotted] (A)--(B1);
				\draw[thick, dotted]  (A)--(B2);
				\draw (0.9,-0.5) node {\scriptsize $\bm L\neq\vn_D$};
				\draw (2.1,-0.46) node {\scriptsize $\wt{\bm L}\neq\vn_D$};
			\end{scope}
			\begin{scope}[xshift=5.4cm]
				\node[empty] (B1) at (1,0){};
				\node[authorised] (B2) at (2,0){};
				\node[empty] (A) at (1.5,1){};
				\draw[thick, dotted] (A)--(B1);
				\draw[thick, loosely dotted]  (A)--(B2);
				\draw (0.9,-0.5) node {\scriptsize $\bm L\neq\vn_D$};
				\draw (2.1,-0.46) node {\scriptsize $\wt{\bm L}=\vn_D$};
			\end{scope}
			\begin{scope}[xshift=8.1cm]
				\node[mandatory] (B1) at (1,0){};
				\node[authorised] (B2) at (2,0){};
				\node[empty] (A) at (1.5,1){};
				\draw[thick, dotted] (A)--(B1);
				\draw[thick, loosely dotted]  (A)--(B2);
				\draw (0.9,-0.5) node {\scriptsize $\bm L\neq\vn_D$};
				\draw (2.1,-0.46) node {\scriptsize $\wt{\bm L}=\vn_D$};
			\end{scope}
			\begin{scope}[xshift=10.8cm]
				\node[authorised] (B1) at (1,0){};
				\node[authorised] (B2) at (2,0){};
				\node[empty] (A) at (1.5,1){};
				\draw[thick, loosely dotted] (A)--(B1);
				\draw[thick, loosely dotted]  (A)--(B2);
				\draw (0.9,-0.5) node {\scriptsize $\bm L=\vn_D$};
				\draw (2.1,-0.46) node {\scriptsize $\wt{\bm L}=\vn_D$};
			\end{scope}
		\end{tikzpicture}
		\caption{$D$-fission of two Stokes circles with vanishing common part and integer/half-integer fission exponent.
			(The three rightmost cases are specific to type $D$.)}
		\label{fig:new_types_of_fission_2_circles_type_D}
	\end{center}
\end{figure}

\subsection{General case:
	enhanced fission data}
\label{sec:general_case_type_d}

It is convenient to introduce versions of fission data involving an explicit choice of signs:

\begin{defi}[cf.~Def.~\ref{def:fission_datum_BC}]
	Let $p \geq 1$ be an integer.
	An \emph{enhanced} $D$-\emph{fission datum} is a pair $\ul{\mc F} =(\ul{\mc L},f)$,
	consisting of:
	\begin{enumerate}
		\item
		      a multiset
		      \begin{equation}
			      \ul{\mc L}
			      = \sum_{i = 1}^p (m_i,\ul{\bm L}_i),
		      \end{equation}
		      of (possibly nondistinct) pairs $(m_i,\ul{\bm L}_i)$,
		      where in turn---%
		      for $i \in \set{1,\dc,p}$:
		      \begin{enumerate}
			      \item
			            $\ul{\bm L}_i$ is an enhanced $D$-level datum;

			      \item
			            and $m_i > 0$ is an integer;
		      \end{enumerate}
		\item
		      and a \emph{list of fission exponents} $f$,
		      i.e.,
		      the choice of a rational number $f_{ij} = f_{ji} \geq 0$,
		      for $i,j \in \set{1,\dc,p}$,
		      such that $f_{ij} = 0$ if and only if $i = j$.
	\end{enumerate}
\end{defi}

\subsubsection{}
\label{sec:fission_datum_from_pseudo_irr_class}

Similarly to \S~\ref{sec:fission_datum_from_irr_class_bc},
to any enhanced pseudo-irregular class one can associate an enhanced $D$-fission datum,
but in a subtler way (as Lem.~\ref{lem:slopes_differences_two_circles_type_D} involves the $B\!C$-level datum of the common part,
rather than the type-$D$ one).

Namely,
with the usual notation,
consider an enhanced pseudo-irregular class
\begin{equation}
	\ul{\wt\Theta}
	= \sum_{i = 1}^p m_i \cdot (\pm I_i, \varepsilon_i),
	\qquad m_i \in \mb Z_{> 0}.
\end{equation}
For $i \in \set{1, \dc, p}$ define an enhanced level datum $\ul{\bm L}_i = \ul{\bm L}_i \bigl( \ul{\wt\Theta} \bigr)$,
as follows.
Write $\pm I_i = \braket{\pm q_i}$ for suitable exponential factors $q_1,\dc,q_p$,
and then:
\begin{enumerate}
	\item if $L_D(m_i,q_i) \neq \vn_D$,
	      set $\ul{\bm L}_i \ceqq L_D(m_i,q_i, \varepsilon_i) = \bigl( L_{B\!C}(q_i), \varepsilon_i \bigr)$;

	\item otherwise:
	      \begin{enumerate}
		      \item
		            if $\pm I_i$ has vanishing common part with $\pm I_j$,
		            for all $j \neq i$,
		            set $\ul{\bm L}_i \ceqq (\vn_D, \varepsilon_i)$;

		      \item
		            else,
		            if there exists $j \neq i \in \set{1,\dc,p}$ such that $\pm I_i$ and $\pm I_j$ have nonzero common part,
		            set
		            \begin{equation}
			            \ul{\bm L}_i
			            = (\bm L_i, \varepsilon_i),
			            \qquad \bm L_i \ceqq L_{B\!C}(q_i)
			            = \set{ \on{slope}(q_i)}.
		            \end{equation}
	      \end{enumerate}
\end{enumerate}
Then let
\begin{equation}
	\ul{\mc L} \bigl( \ul{\wt\Theta} \bigr) \ceqq (m_1,\ul{\bm L}_1) + \dm + (m_p,\ul{\bm L}_p),
\end{equation}
and
\begin{equation}
	f_{ij}
	= f \bigl( \ul{\wt\Theta} \bigr)_{ij} \ceqq f_{\pm I_i,\pm I_j} \in \mb Q_{\geq 0},
	\qquad i,j \in \set{1,\dc,p},
\end{equation}
with the convention $f_{\pm I_i,\pm I_i} \ceqq 0$.

In turn,
if $\wt\Theta$ is a pseudo-irregular class,
we let $\ul{\mc F} \bigl( \wt\Theta \bigr) \ceqq \ul{\mc F} \bigl( \ul{\wt\Theta} \bigr)$ be its associated \emph{enhanced} $D$-fission datum,
considering the corresponding enhanced pseudo-irregular class $\ul{\wt\Theta}$ (cf.~Rmk.~\ref{rmk:enhancements_for_classes}.)
Finally,
if $\Theta$ is a (non-pseudo) irregular class,
its fission datum is---%
by definition---%
that of the underlying pseudo-irregular class $\wt\Theta$ (cf. Prop.~\ref{prop:form_type_D_irreg_classes}).

\subsubsection{}

Analogously to Def.~\ref{def:labelled_fission_datum_bc},
a \emph{labelled enhanced} $D$-\emph{fission datum} is a pair $\ul{\dot{\mc F}} = \bigl( \ul{\dot{\mc L}},f \bigr)$,
consisting of an ordered list (with allowed repetitions)
\begin{equation}
	\ul{\dot{\mc L}}
	= \bigl( (m_1,\ul{\bm L}_1),\dc,(m_p,\ul{\bm L}_p) \bigr),
\end{equation}
for some integer $p \geq 1$,
where:
(i) $\ul{\bm L}_1, \dc \ul{\bm L}_p$ are enhanced $D$-level data;
(ii) $m_i > 0$ are integer multiplicities;
and (iii) $f$ is a list of fission exponents $f_{ij} = f_{ji} \in \mb Q_{\geq 0}$,
for $i,j \in \set{1, \dc, p}$,
such that $f_{ij}=0$ if and only if $i = j$.
Again,
a labelled enhanced $D$-fission datum determines an (unlabelled) enhanced $D$-fission datum by forgetting the ordering of the pairs $(m_i, \ul{\bm L}_i)$.

Furthermore,
an enhanced pointed irregular type $\ul{\dot Q}$ naturally determines a labelled enhanced $D$-fission datum $\ul{\dot{\mc F}} \bigl( \ul{\dot Q} \bigr) = \bigl( \ul{\dot{\mc L}}\bigl( \ul{\dot Q} \bigr),f \bigl( \ul{\dot Q} \bigr) \bigr)$.
In turn,
as in the full case,
if $\dot Q$ is a pointed irregular type then it determines a labelled \emph{enhanced} $D$-fission datum via $\ul{\dot{\mc F}} (\dot Q) \ceqq \ul{\dot{\mc F}} \bigl( \ul{\dot Q} \bigr)$,
considering its unique pointed enhancement $\ul{\dot Q}$.
Finally,
if $\wt\Theta = \wt\Theta(\dot Q)$ is the pseudo-irregular class of a pointed irregular type $\dot Q$,
then $\ul{\mc F} \bigl( \wt\Theta \bigr) = \ul{\mc F} \bigl( \ul{\wt\Theta} \bigr)$ is obtained from $\ul{\dot{\mc F}}(\dot Q) = \ul{\dot{\mc F}} \bigl( \ul{\dot Q} \bigr)$ by forgetting the labels,
and one can prove the following important:

\begin{theo}[cf.~Thm.~\ref{thm:admissible_deformation_iff_same_fission_data_type_BC}]
	\label{thm:admissible_deformation_iff_same_fission_data_type_D}

	Two enhanced pseudo-irregular classes (resp.,
	two enhanced pointed irregular types) are mutual $D$-admissible deformations if and only if they determine the same enhanced fission datum (resp.
	the same labelled enhanced fission datum).
\end{theo}

\begin{proof}
	Use Prop.~\ref{prop:characterisation_admissible_def_two_circles_type_D}.
\end{proof}

\subsection{General case:
	hybrid branches}

To define $D$-fission trees,
one needs to associate a different kind of branch to the enhanced $D$-fission data of the form $(\vn_D, \pm 1)$.
Moreover,
when such branches are involved,
the gluing is \emph{not} as straightforward.

The upshot is the following:

\begin{defi}[cf.~Def.~\ref{def:type_d_branch}]
	\label{def:hybrid_full_branches_type_D}

	Choose $k \in \mb Q_{> 0}$ and $\varepsilon \in \mb Z^{\ts}$.
	Then:
	\begin{enumerate}
		\item
		      the \emph{hybrid full} $D$-\emph{branch of} $(k,\varepsilon)$ is the quadruple $\mc B_{(\vn_D, \varepsilon)}(k) = (\mc B,\mb V,\mb A, \mb L)$,
		      where $\mc B$ is a copy of $\ol{\mb R}_{\geq 0}$,
		      and  $\mc B \supseteq \mb V \supseteq \mb A \supseteq \mb L$ are subsets whose elements are called \emph{vertices},
		      \emph{admissible vertices},
		      and \emph{mandatory vertices},
		      respectively,
		      defined as follows (using the given identification $h \cl \mc B \lxra{\simeq} \ol{\mb R}_{\geq 0}$):
		      \begin{enumerate}
			      \item
			            $\mb L \ceqq \vn$;

			      \item
			            $\mb A \ceqq h^{-1} \bigl( \on{Inc}_D(\vn_D, \varepsilon) \cap (0, k) \bigr)$;

			      \item
			            and $\mb V \ceqq \mb A \cup \set{v_k}$,
			            where $v_k \ceqq h^{-1}(k)$ is the \emph{hybridization vertex};
		      \end{enumerate}
		\item
		      the set of \emph{inconsequential vertices},
		      the \emph{root} $v_\infty$,
		      and the \emph{leaf} $v_0$,
		      are defined as in the non-hybrid case;

		\item
		      and each edge $e$ of $\mc B_{(\vn_D, \varepsilon)}(k)$ has a \emph{type} $t(e) \in \set{E, HE}$ (either `\emph{empty}' or `\emph{half-empty}',
		      respectively),
		      defined as follows:
		      \begin{enumerate}
			      \item
			            if $e = h^{-1}\bigl( (m, m+1) \bigr)$,
			            with $m, m+1 \in h(\mb A)$,
			            set $t(e) \ceqq HE$;

			      \item
			            if $e = h^{-1} \bigl( (\ul k,k) \bigr)$,
			            with $\ul k \ceqq \max \bigl( h(\mb A) \bigr)$,
			            set $t(e) \ceqq HE$;

			      \item
			            if $e$ is the parent-edge of $v_0$,
			            set $t(e) \ceqq HE$;

			      \item
			            and if $e=h^{-1} \bigl( (k, + \infty) \bigr)$,
			            set $t(e) \ceqq E$.
		      \end{enumerate}
	\end{enumerate}
\end{defi}

\begin{rema}
	In brief,
	to get the hybrid full $D$-branch $\mc B_{(\vn_D,\varepsilon)}(k)$ one changes the edge-types of $\mc B_{(\vn_D,\varepsilon)}$ from half-empty to empty (above the height $k$,
	cf.~Fig.~\ref{fig:hybrid_full_branches_type_D}).
\end{rema}

\begin{figure}
	\begin{center}
		\begin{tikzpicture}
			\tikzstyle{authorized}=[circle,fill=white,minimum size=5pt,draw, inner sep=0pt]
			\tikzstyle{mandatory}=[circle,fill=black,minimum size=5pt,draw,inner sep=0pt]
			\tikzstyle{hybridization}=[circle,fill=black,minimum size=3pt,inner sep=0pt]
			\tikzstyle{root}=[fill=black,minimum size=5pt,draw,inner sep=0pt]
			\tikzstyle{indeterminate}=[circle,densely dotted,fill=white,minimum size=5pt,draw, inner sep=0pt]
			\begin{scope}
				\node[root] (R) at (0,7.5){};
				\node (X) at (0,6.5){$\vdots$};
				\node (Y) at (0,6){};
				\node[hybridization] (Z) at (0,5){};
				\node[authorized] (A) at (0,4){};
				\node[authorized] (B) at (0,3){};
				\node[authorized] (C) at (0,2){};
				\node (D) at (0,1){};
				\draw[thick, dotted] (R)--(X);
				\draw[thick, dotted] (Y)--(Z);
				\draw[thick, loosely dotted] (Z)--(A);
				\draw[thick, loosely dotted] (A)--(B);
				\draw[thick, loosely dotted] (B)--(C);
				\draw[thick, loosely dotted] (C)--(D);
				\draw (-2, 7.5) node {$+ \infty$};
				\draw (-2, 5) node {$k$};
				\draw (-2, 4) node {$l$};
				\draw (-2, 2) node {$1$};
				\draw (0, 0) node {$\mc B_{(\vn_D, 1)}(k)$};
			\end{scope}
			\begin{scope}[xshift=5cm]
				\node[root] (R) at (0,7.5){};
				\node (X) at (0,6.5){$\vdots$};
				\node (Y) at (0,6){};
				\node[hybridization] (Z) at (0,5){};
				\node[authorized] (A) at (0,4.5){};
				\node[authorized] (B) at (0,3.5){};
				\node[authorized] (C) at (0,2.5){};
				\node[authorized] (D) at (0,1.5){};
				\node (E) at (0,1){};
				\draw[thick, dotted] (R)--(X);
				\draw[thick, dotted] (Y)--(Z);
				\draw[thick, loosely dotted] (Z)--(A);
				\draw[thick, loosely dotted] (A)--(B);
				\draw[thick, loosely dotted] (B)--(C);
				\draw[thick, loosely dotted] (C)--(D);
				\draw[thick, loosely dotted] (D)--(E);
				\draw (-2, 7.5) node {$+\infty$};
				\draw (-2, 5) node {$k$};
				\draw (-2, 4.5) node {$l+1/2$};
				\draw (-2, 1.5) node {$1/2$};
				\draw (0, 0) node {$\mc B_{(\vn_D, -1)}(k)$};
			\end{scope}
		\end{tikzpicture}
		\caption{Hybrid full $D$-branches:
			on the left (resp.,
			on the right) with integer-height vertices (resp.,
			with half-integer-height vertices).}
		\label{fig:hybrid_full_branches_type_D}
	\end{center}
\end{figure}

\subsection{General case:
	fission trees}

Keep all notation from \S~\ref{sec:type_BC_fission_trees}.
The natural notion of (pre-)fission trees in type $D$ already comes with enhancements:

\begin{defi}[cf.~Def.~\ref{def:bc_quasi_fission_tree}]
	\label{def:d_quasi_fission_tree}

	A tuple $(\mc T,\mb V,\mb A,\mb L, h, m)$ is a \emph{pre}-$D$-\emph{fission tree} if it satisfies the following conditions:
	\begin{enumerate}
		\item
		      $\mb V = h^{-1} \bigl( \set{0}\cup h(\mb A) \cup \set{+ \infty} \bigr)$;

		\item
		      the function $h$ maps each full branch bijectively onto $\ol{\mb R}_{\geq 0}$;

		\item
		      for any leaf $i$,
		      the set $\bm L_i \ceqq h(\mb L_i) \sse \mb Q_{>0}$ is a $D$-level datum;

		\item
		      for any leaf $i$,
		      denoting by $\mb V_i \ceqq \mc B_i \cap \mb V$ the set of vertices belonging to $\mc B_i$,
		      then---%
		      up to additional empty vertices---%
		      the quadruple $(\mc B_i,\mb V_i, \mb A_i, \mb L_i)$ is a copy of the full $D$-branch $\wt{\mc B}_i$ of an enhanced $D$-level datum $\ul{\bm L}_i = (\bm L_i, \varepsilon_i)$,
		      with underlying level datum $\bm L_i$;
		      namely:
		      \begin{enumerate}
			      \item
			            if $\bm L_i\neq \vn_D$,
			            then---%
			            up to adding empty vertices---%
			            $\wt{\mc B}_i$ is equal to the full $D$-branch $\mc B_{\ul{\bm L}_i}$ associated with $\ul{\bm L}_i$;

			      \item
			            if $\bm L_i = \vn_D$,
			            then---%
			            up to adding empty vertices---%
			            $\wt{\mc B}_i$ is equal to the hybrid full $D$-branch $\mc B_{(\vn_D, \varepsilon_i)}(k)$,
			            where $k$ is the smallest height of a branching vertex in $\mc B_i$;

			      \item
			            if $\wt{\mc B}_i = \mc B_{(\vn_D, \varepsilon_i)}(k)$,
			            then $m_i = 1$;

			      \item
			            and if $\wt{\mc B}_i = \mc B_{(\vn_{B\!C}, \varepsilon_i)}$,
			            then $m_i \geq 2$;
		      \end{enumerate}
		\item
		      and one has $\prod_{\mb V_0} \varepsilon_i = 1$.\fn{
			      Note that the sign $\varepsilon_i$ is determined by the tree if and only if $\bm L_i \neq \vn_{B\!C}$.}
	\end{enumerate}
\end{defi}

\subsubsection{}

As in Def.~\ref{def:bc_edge_type},
we associate a \emph{type} $t(e) \in \set{HE,E,S,NS}$ to any edge $e$ of $\mc T$.
Finally,
we impose a last branching axiom to relate to local WMCGs:

\begin{defi}[cf.~Deff.~\ref{def:fission_tree_type_BC} +~\ref{def:d_quasi_fission_tree}]
	\label{def:fission_tree_type_D}

	A pre-$D$-fission tree is a $D$-\emph{fission tree} if moreover:
	\begin{enumerate}
		\setcounter{enumi}{5}
		\item

		      for any branch-vertex $v\in \mb Y$,
		      exactly one of the following occurs:
		      \begin{enumerate}
			      \item[(I--III)]
			            $v$ has no half-empty child-edge,
			            and the branching is identical to a branching in a $B\!C$-fission tree;

			      \item[(IV)]
			            $v$ has at least one half-empty child-edge (whence $v$ is empty),
			            and:
			            \begin{enumerate}
				            \item[(IV.a)]
				                  exactly one child-vertex of $v$ is empty,
				                  and the other child-vertices are either inconsequential with a half-empty parent-edge,
				                  or mandatory with an empty parent-edge;

				            \item[(IV.b)]
				                  the child-vertices of $v$ are nonempty,
				                  and they are either inconsequential with a half-empty parent-edge,
				                  or mandatory with an empty parent-edge.
			            \end{enumerate}
		      \end{enumerate}
	\end{enumerate}
\end{defi}

\begin{rema}
	The branching axiom could be replaced by the statement that $\on{Ch}(v)$ contains at most one empty vertex for any $v \in \mb Y$ (cf.~Rmk.~\ref{rmk:about_def_trees_BC}).
	Moreover,
	the notions of isomorphisms/labellings carry over from type $B\!C$ (cf.~Deff.~\ref{def:tree_labels}--\ref{def:tree_isos}).

	Conversely,
	there are now new types of branchings involving the hybrid full branches,
	viz.,
	subcases (IV.a)--(IV.b) of Def.~\ref{def:fission_tree_type_D} (cf.~Fig.~\ref{fig:possible_branching_type_D_fission_tree}).
\end{rema}

(Until \S~\ref{sec:tglwmcgs},
all (pre-)fission trees are tacitly of type $D$.)

\begin{figure}
	\begin{center}
		\begin{tikzpicture}
			\tikzstyle{mandatory}=[circle,fill=black,minimum size=6pt,draw, inner sep=0pt]
			\tikzstyle{authorised}=[circle,fill=white,minimum size=6pt,draw,inner sep=0pt]
			\tikzstyle{empty}=[circle,fill=white,minimum size=0pt,inner sep=0pt]
			\tikzstyle{indeterminate}=[circle,dotted,thick,minimum size=6pt,draw, inner sep=0pt]
			\begin{scope}
				\node[empty] (B0) at (3,2){};
				\node[empty] (B1) at (1,0){};
				\node[authorised] (B2) at (2,0){};
				\node[authorised] (B3) at (3,0){};
				\draw (2.5,0) node {$\dc$};
				\node[mandatory] (B4) at (4,0){};
				\draw (4.5,0) node {$\dc$};
				\node[mandatory] (B5) at (5,0){};
				\draw[thick, dotted] (B0)--(B1);
				\draw[thick, loosely dotted] (B0)--(B2);
				\draw[thick, loosely dotted] (B0)--(B3);
				\draw[thick, dotted] (B0)--(B4);
				\draw[thick, dotted] (B0)--(B5);
				\draw [thick, decoration={brace, mirror},decorate] (2,-0.4)--(3,-0.4);
				\draw (2.5, -0.8) node {$n_1$ times};
				\draw [thick, decoration={brace, mirror},decorate] (4,-0.4)--(5,-0.4);
				\draw (4.5, -0.8) node {$n_2$ times};
				\draw (3,-1.4) node {(IV.a)};
			\end{scope}
			\begin{scope} [xshift=6cm]
				\node[empty] (B0) at (3.5,2){};
				\node[authorised] (B2) at (2,0){};
				\node[authorised] (B3) at (3,0){};
				\draw (2.5,0) node {$\dc$};
				\node[mandatory] (B4) at (4,0){};
				\draw (4.5,0) node {$\dc$};
				\node[mandatory] (B5) at (5,0){};
				\draw[thick, loosely dotted] (B0)--(B2);
				\draw[thick, loosely dotted] (B0)--(B3);
				\draw[thick, dotted] (B0)--(B4);
				\draw[thick, dotted] (B0)--(B5);
				\draw [thick, decoration={brace, mirror},decorate] (2,-0.4)--(3,-0.4);
				\draw (2.5, -0.8) node {$n_1$ times};
				\draw [thick, decoration={brace, mirror},decorate] (4,-0.4)--(5,-0.4);
				\draw (4.5, -0.8) node {$n_2$ times};
				\draw (3.5,-1.4) node {(IV.b)};
			\end{scope}
		\end{tikzpicture}
	\end{center}
	\caption{Hybrid branchings (involving half-empty edges;
		the branching vertex is empty in this case).
		In subcase (IV.a),
		the inconsequential vertices are indirectly mandatory since they have an empty sibling-vertex.}
	\label{fig:possible_branching_type_D_fission_tree}
\end{figure}

\subsubsection{}

One can now associate an enhanced $D$-fission datum $\ul{\mc F} = \ul{\mc F}(\mc T)$ to any fission tree $\mc T$,
and prove that fission data/trees are equivalent.
Namely:
\begin{enumerate}
	\item
	      every full branch $\mc B_i$ determines a $D$-level datum $\bm L_i \ceqq h(\mb L_i)$;
	      now:
	      \begin{enumerate}
		      \item
		            if $\bm L_i\neq \vn_{B\!C}$,
		            then the full branch  $\mc B_i$ actually determines an enhanced level datum $\ul{\bm L}_i=(\bm L_i,\varepsilon_i)$,
		            such that $\mc B_i = \mc B_{\ul{\bm L}_i}$---%
		            up to empty vertices;

		      \item
		            if $\bm L_{i_0} = \vn_{B\!C}$ for some (unique) $i_0 \in \mb V_0$,
		            then $\mc B_{i_0}$ does \emph{not} directly determines an enhancement $\ul{\bm L}_{i_0}=(\bm L_{i_0},\varepsilon_{i_0})$ of $\bm L_{i_0}$,
		            but the sign $\varepsilon_{i_0} \in \mb Z^{\ts}$ is recovered from the identity $\prod_{\mb V_0} \varepsilon_i = 1$;
	      \end{enumerate}
	\item
	      and the fission exponent $f_{ij}$ is obtained as the first admissible exponent below the branching heights of the branches $\mc B_i$ and $\mc B_j$.
\end{enumerate}

Analogously,
a labelled fission tree $\dot{\mc T}$ determines a labelled enhanced $D$-fission datum $\ul{\dot{\mc F}} = \ul{\dot{\mc F}} (\dot{\mc T})$,
and then:

\begin{lemm}[cf.~Lem.~\ref{lem:equivalence_tree_fission_datum_BC}]
	Two (labelled) fission trees are isomorphic if and only if they determine the same (labelled) enhanced fission datum.
\end{lemm}

\begin{proof}[Proof omitted]
\end{proof}

\subsubsection{}
\label{sec:trees_from_d_irr_types_classes}

Moreover,
an enhanced pseudo-irregular class $\ul{\wt\Theta} = \sum_{i=1}^p m_i \cdot (\pm I_i,\varepsilon_i)$ determines a fission tree,
as follows.
Write $\pm I_i = \braket{\pm q_i}$,
and let $\ul{\mc F} = \ul{\mc F} \bigl( \ul{\wt\Theta} \bigr) = (\ul{\mc L},f)$ be the fission datum of $\ul{\wt\Theta}$ (cf.~\S~\ref{sec:fission_datum_from_pseudo_irr_class}),
where
\begin{equation}
	\ul{\mc L}
	= (m_1,\ul{\bm L}_1) + \dm + (m_p,\ul{\bm L}_p),
	\qquad \ul{\bm L}_i
	= (\bm L_i, \varepsilon_i).
\end{equation}

We start from the gluing exponents.
(They determine the hybridization points of hybrid branches.)
Let $\on{Adm} \bigl( \ul{\wt\Theta} \bigr) \ceqq \bigcup_i \on{Adm}_D (\ul{\bm L}_i) \sse \mb Q_{>0}$,
so that any number $k \in \mb Q_{> 0}$ has a \emph{successor}
\begin{equation}
	\on{succ}(k) \in \ol{\on{Adm}} \bigl( \ul{\wt\Theta} \bigr) \ceqq \on{Adm} \bigl( \ul{\wt\Theta} \bigr) \cup \set{+ \infty}.
\end{equation}
If $f_{ij} \ceqq f_{\pm I_i,\pm I_j}$ is the fission exponent between ${\pm I_i,\pm I_j}$,
define the \emph{gluing exponent} as in~\eqref{eq:gluing_exponent}.
Now for $i \in \set{1,\dc,p}$ define a full $D$-branch $\wt{\mc B}_i = (\wt{\mc B}_i,\mb V_i,\mb A_i, \mb L_i)$,
as follows:
\begin{enumerate}
	\item
	      if $\bm L_i\neq\vn_D$,
	      set $\wt{\mc B}_i \ceqq \mc B_{\ul{\bm L}_i}$;

	\item
	      otherwise,
	      let $\wt{\mc B}_i \ceqq \mc B_{(\vn_D,\varepsilon_i)}(k_i)$ be the hybrid full branch with hybridization vertex at height $k_i$,
	      where $k_i \ceqq \min\limits_{j\neq i}(g_{ij})$.
\end{enumerate}
The rest of the construction is analogous to \S~\ref{sec:fission_tree_from_BC_irr_class},
and yields the fission tree $\mc T = \mc T \bigl( \ul{\wt\Theta} \bigr)$.
(The branching is correct because one can establish the type-$D$ analogue of Lem.~\ref{lem:characterisation_right_branchings_type_BC},
which we will \emph{not} state/prove.)
Moreover,
the nonroot vertices of $\mc T$ may still be interpreted in terms of truncated Stokes-circles-up-to-sign,
and Lem.-Def.~\ref{lem:link_vertices_truncated_circles_BC} translates verbatim.

In turn,
a pseudo-irregular class $\wt\Theta$ also determines a fission tree $\mc T = \mc T(\wt\Theta)$ by considering its associated enhanced pseudo-irregular class $\ul{\wt\Theta}$;
and for an irregular class $\Theta$ we let $\mc T(\Theta) \ceqq \mc T(\wt\Theta)$,
where $\wt\Theta$ is the underlying pseudo-irregular class.

Finally,
an enhanced pointed irregular type $\ul{\dot Q}$ yields a labelled fission tree $\dot{\mc T} = \dot{\mc T} \bigl( \ul{\dot Q} \bigr)$,
by labelling the leaves according to the order of the exponential factors,
and this is by definition also the labelled fission tree of the underlying pointed irregular type $\dot Q$.
Then the exact analogue of Cor.~\ref{cor:adm_def_by_trees} holds true.

\begin{rema}[cf.~Rmk.~\ref{rmk:truncations_trees_BC}]
	The set $\on{Adm} \bigl( \ul{\wt\Theta} \bigr)$ can be infinite,
	since the level datum $\vn_D$ has infinitely many admissible exponents.
	Nonetheless,
	the set of heights of admissible vertices of fission trees is \emph{finite}.
	(So there is a natural truncation height for drawing them,
	etc.)
\end{rema}

\begin{rema}[cf.~Rmk.~\ref{rmk:comparison_bichromatic_trees_BC_DR}]
	\label{rmk:comparison_generalised_trees_D_DR}

	In \cite{doucot_rembado_tamiozzo_2026_local_wild_mapping_class_groups_and_cabled_braids} we also introduced \emph{generalized fission trees},
	to encode admissible deformations of \emph{untwisted} $D$-irregular types.
	These trees are equivalent to a particular case of Def.~\ref{def:fission_tree_type_D},
	as follows.

	Let $\dot Q$ be an untwisted compatible pointed $D$-irregular type.
	We associate to it a labelled fission tree $\dot{\mc T}$ as above,
	as well as a generalized fission tree $\wt{\mc T}$ as in op.~cit.
	The generalized tree can be viewed as a tuple $(\wt{\mc T}, \wt{\mb V}, \wt h, \wt m)$,
	and then:
	\begin{enumerate}
		\item there is a height-preserving function $\dot{\mc T} \to \wt{\mc T}$,
		      inducing a bijection $\mb V \lxra{\simeq} \wt{\mb V}$;

		\item and the colours/sizes of the vertices of $\wt{\mc T}$ are related to the vertices/edges of $\dot{\mc T}$ as follows:
		      \begin{enumerate}
			      \item the blue vertices of $\wt{\mc T}$ map to empty vertices of $\dot{\mc T}$;

			      \item the large green vertices of $\wt{\mc T}$ map to nonempty vertices of $\dot{\mc T}$,
			            such that the sum of the multiplicities of their descendant leaves is greater than $1$,
			            i.e.,
			            the nonempty vertices belonging to several full branches,
			            or to a single full branch with multiplicity greater than $1$;

			      \item the small green vertices of $\wt{\mc T}$ map to nonempty vertices of $\dot{\mc T}$,
			            such that the sum of the multiplicities of their descendant leaves is $1$,
			            i.e.,
			            the nonempty vertices belonging to a single full branch with multiplicity $1$.
		      \end{enumerate}
	\end{enumerate}
	The axioms are consistent:
	green vertices in $\wt{\mc T}$ can only have green child-vertices,
	which maps to the fact that if $v \in \mb V$ is a nonempty vertex then its child-vertices are also nonempty;
	large green vertices in $\wt{\mc T}$ are allowed to have both large/small green child-vertices,
	while small green vertices cannot be branching points,
	which links with the sum of the multiplicities of their descendant leaves;
	and a blue vertex has at most one blue child-vertex,
	which maps to the fact that if $v \in \mc T$ is empty then at most one of its child-vertices is empty.
\end{rema}

\begin{exem}
	\label{exmp:example_fission_tree_D}

	Consider the following pointed irregular type,
	with the integer-multiplicities convention:
	\begin{equation}
		\dot Q
		= \bigl( (2, q_1), (1, q_2), (1, q_3), (1, q_4), (1, q_5) \bigr),
	\end{equation}
	where
	\begin{equation}
		q_1
		\ceqq 0,
		\quad q_2
		\ceqq z^{-2},
		\quad q_3
		\ceqq 2z^{-2},
		\quad q_4
		\ceqq z^{-3},
		\quad q_5
		\ceqq z^{-3} + z^{-3/2}.
	\end{equation}
	Its enhancement is
	\begin{equation}
		\ul{\dot  Q}
		= \bigl( (2, q_1, 1), (1, q_2, 1), (1, q_3, 1), (1, q_4, 1), (1, q_5, 1) \bigr),
	\end{equation}
	and the corresponding labelled fission tree is drawn in Fig.~\ref{fig:example_tree_type_D}.

	Note that $q_1,\dc,q_4$ are untwisted,
	but do \emph{not} have the same full branches.
	Indeed, $q_1 = 0$ appears with multiplicity $2$,
	so that $\ul{\bm L}_1 = (L_{B\!C}(q_1), 1) = (\vn_{B\!C}, 1)$.
	On the other hand,
	$q_2$ and $q_3$ have multiplicity 1,
	and both have vanishing common parts with other exponential factors,
	so that $\ul{\bm L}_2 = \ul{\bm L}_3 = (\vn_D, 1)$.
	Finally,
	$q_4$ also has multiplicity 1,
	but it has nonzero common part with $q_5$ (equal to $z^{-3}$),
	so that $\ul{\bm L}_4 = (L_{B\!C}(q_4), 1)$ and $\bm L_4 = \set{ \on{slope}(q_4) } = \set{3}$.
	Moreover,
	the inconsequential vertices of height $2$ on the full branches $\mc B_2$ and $\mc B_3$ are indirectly mandatory,
	since they have an empty sibling-vertex---%
	on the full branch $\mc B_1$.
\end{exem}

\begin{figure}
	\begin{center}
		\begin{tikzpicture}
			\tikzstyle{root}=[fill=black,minimum size=6pt,draw,inner sep=0pt]
			\tikzstyle{mandatory}=[circle,fill=black,minimum size=6pt,draw, inner sep=0pt]
			\tikzstyle{authorised}=[circle,fill=white,minimum size=6pt,draw,inner sep=0pt]
			\tikzstyle{empty}=[circle,fill=white,minimum size=0pt,inner sep=0pt]

			\node[root] (R) at (2,4){};

			\node[empty] (B1) at (1,3){};
			\node[mandatory] (B2) at (3.5,3){};

			\node[empty] (C1) at (0,2){};
			\node[authorised] (C2) at (1,2){};
			\node[authorised] (C3) at (2,2){};
			\node[authorised] (C4) at (3,2){};
			\node[authorised] (C5) at (4,2){};

			\node[mandatory] (C_D5) at (4,1.5){};

			\node[empty] (D1) at (0,1){};
			\node[authorised] (D2) at (1,1){};
			\node[authorised] (D3) at (2,1){};
			\node[authorised] (D4) at (3,1){};
			\node[authorised] (D5) at (4,1){};

			\node[empty] (D1) at (0,1){};
			\node[authorised] (D2) at (1,1){};
			\node[authorised] (D3) at (2,1){};
			\node[authorised] (D4) at (3,1){};
			\node[authorised] (D5) at (4,1){};

			\node[authorised] (DE5) at (4,0.5){};

			\node[empty] (E1) at (0,0){};
			\node[empty] (E2) at (1,0){};
			\node[empty] (E3) at (2,0){};
			\node[empty] (E4) at (3,0){};
			\node[empty] (E5) at (4,0){};

			\draw[thick, dotted] (R)--(B1);
			\draw[thick, dotted] (R)--(B2);

			\draw[thick, dotted] (B1)--(C1);
			\draw[thick, loosely dotted] (B1)--(C2);
			\draw[thick, loosely dotted] (B1)--(C3);
			\draw (B2)--(C4);
			\draw (B2)--(C5);

			\draw[thick, dotted] (C1)--(D1);
			\draw[thick, loosely dotted] (C2)--(D2);
			\draw[thick, loosely dotted] (C3)--(D3);
			\draw (C4)--(D4);
			\draw (C5)--(C_D5)--(D5);

			\draw[thick, dotted] (D1)--(E1);
			\draw[thick, loosely dotted] (D2)--(E2);
			\draw[thick, loosely dotted] (D3)--(E3);
			\draw (D4)--(E4);
			\draw (D5)--(DE5)--(E5);

			\draw (0,-0.5) node {$q_1$};
			\draw (1,-0.5) node {$q_2$};
			\draw (2,-0.5) node {$q_3$};
			\draw (3,-0.5) node {$q_4$};
			\draw (4,-0.5) node {$q_5$};

			\draw (-1.5, 4) node {$+ \infty$};
			\draw (-1.5, 3) node {$3$};
			\draw (-1.5, 2) node {$2$};
			\draw (-1.5, 1.5) node {$3/2$};
			\draw (-1.5, 1) node {$1$};
			\draw (-1.5, 0.5) node {$1/2$};

		\end{tikzpicture}
	\end{center}
	\caption{Labelled fission tree of Exmp.~\ref{exmp:example_fission_tree_D}.}
	\label{fig:example_tree_type_D}
\end{figure}

\subsection{General case:
	tree realizations}

Consider a fission tree $\mc T$ and a map $c \cl \mb A\to \mb C$.
Analogously to \S~\ref{sec:realization},
we associate to the pair $(\mc T,c)$:
(i) an exponential factor $q_{c,i}$ for each leaf $i \in \mb V_0$;
(ii) a $\mb Z_{> 0}$-linear combination $\wt\Theta_c$ of pairwise-distinct Stokes-circles-up-to-sign;
and (iii) a list $\dot Q_c$ of exponential factors with multiplicities,
under a labelling $\set{1,\dc,p} \lxra{\simeq} \mb V_0$ of $\mc T$---%
with $p \ceqq \abs{\mb V_0} \geq 1$.
As in Def.~\ref{def:realization_type_BC},
we then say that $c$ is a \emph{realization of} $\mc T$ (resp.,
\emph{of} $\dot{\mc T}$) if $\wt\Theta_c$ is a pseudo-irregular class such that $\mc T \bigl( \wt\Theta_c \bigr) \simeq \mc T$ (resp.,
if $\dot Q_c$ is a pointed irregular type such that $\dot{\mc T} \bigl( \dot Q_c \bigr) \simeq \dot{\mc T}$).

Now one can proceed as in \S~\ref{sec:realization_check},
and Lem.~\ref{lem:characterisation_right_branchings_type_BC} translates verbatim,
justifying the axioms for $D$-branchings.
This leads to the following:

\begin{prop}[cf.~Cor.~\ref{cor:type_BC_characterisation_of_realizations}]
	\label{prop:type_D_characterisation_of_realizationsg}

	A function $c \cl \mb A \to \mb C$ is a realization of $\mc T$ and/or $\dot{\mc T}$ if and only if the following conditions hold:
	\begin{enumerate}
		\item
		      one has $c(\mb L)\sse \mb C^\ast$;

		\item
		      for any node $u$ with half-empty parent-edge,
		      if $u$ has an empty sibling-vertex then $c(u) \neq 0$;

		\item
		      for any pair of distinct nonempty sibling-vertices $u,v$,
		      denoting by $k$ their height, $r$ their ramification order, and $N \geq 1$ their partial ramification order (cf.~again \S\ref{sec:realization_check} and Rmk.~\ref{rmk:realization_check}):
		      \begin{enumerate}
			      \item[(I)]
			            if $u,v$ have nonspecial parent-edges,
			            then $c(u)^N\neq c(v)^N$;

			      \item[(II)]
			            if $u,v$ have special parent-edges:
			            \begin{enumerate}
				            \item[(i)]
				                  if $N=1$ and $u,v$ are mandatory,
				                  then $c(u)\neq \pm c(v);$

				            \item[(ii)]
				                  otherwise:
				                  \begin{equation}
					                  \begin{cases}
						                  c(u)^N \neq c(v)^N,       & \quad N+kr \text{ even},  \\
						                  c(u)^{2N} \neq c(v)^{2N}, & \quad N+kr \text{ odd};
					                  \end{cases}
				                  \end{equation}
			            \end{enumerate}
			      \item[(III)]
			            and if $u,v$ have empty or half-empty parent-edges,
			            then:
			            \begin{equation}
				            \begin{cases}
					            c(u)^N \neq  c(v)^N,      & \quad N \text{ even}, \\
					            c(u)^{2N} \neq c(v)^{2N}, & \quad N \text{ odd}.
				            \end{cases}
			            \end{equation}
		      \end{enumerate}
	\end{enumerate}
\end{prop}

\begin{proof}
	The first condition (and the definition of $\mb A$) means that each full branch has the correct interior levels.
	The second condition ensures that,
	for a branching of type (IV.a) in Def.~\ref{def:fission_tree_type_D},
	the child-vertices with half-empty parent-edge are indirectly mandatory.
	Finally,
	the third condition corresponds to the fact that Lem.~\ref{lem:characterisation_right_branchings_type_BC} remains true in type $D$.
\end{proof}

\subsubsection{}

In particular,
the type-$D$ version of~Cor.~\ref{cor:type_BC_realisability} holds true.
Hereafter,
and until \S~\ref{sec:weyl_group_tree_d},
let $\dot Q$ be a compatible pointed $D$-irregular type,
and write $\dot{\mc T} \ceqq \dot{\mc T}(\dot Q)$.

First,
one can state/prove the analogue of Lem.~\ref{lem:any_Q_comes_from_realization_type_BC},
which we omit.
Then to the tree $\dot{\mc T}$ we associate a (global) $D$-\emph{configuration space}
\begin{equation}
	\label{eq:configuration_space_of_tree_d}
	\bm B_D(\dot{\mc T})
	\ceqq  \Set{ c \cl \mb A \to \mb C | c \text{ is a realization of } \dot{\mc T} },
\end{equation}
which we still regard as a topological subspace of $\mb C^{\mb A} \simeq \mb C^{\,\abs{\mb A}}$ (cf.~Rmk.~\ref{rmk:ordering_vertices}).

\begin{theo}
	\label{thm:type_D_configuration_space_from_realizations}

	There is a \emph{homeomorphism}
	\begin{equation}
		\bm B_{D,r}(\dot Q) \simeq \bm B_D(\dot{\mc T}),
	\end{equation}
	in the notation of~\eqref{eq:config_space_type_D}.
\end{theo}

\begin{proof}
	Analogous to the proof of Thm.-Def.~\ref{thm:type_BC_configuration_space_from_realizations}.
\end{proof}

\subsubsection{}

Again,
the fact that the conditions of Prop.~\ref{prop:type_D_characterisation_of_realizationsg} are independent at each vertex $v \in \mb V$ of $\dot{\mc T}$ yields a factorization of the configuration space.
To state this,
introduce the \emph{local} $D$-\emph{configuration space of} $(\dot{\mc T},v)$ as
\begin{equation}
	\label{eq:local_conf_space_D}
	\bm B_D(\dot{\mc T},v)
	\ceqq \Set{ c_v \cl \on{Ch}_{\mb A}(v) \to \mb C | c_v \text{ satisfies the conditions of Prop.~\ref{prop:type_D_characterisation_of_realizationsg}} },
\end{equation}
where in turn $\on{Ch}_{\mb A}(v) \ceqq \mb A \cap \on{Ch}(v) \sse \mb A$ (cf.~Def.~\ref{def:local_conf_space_bc}).

\begin{coro}
	\label{cor:config_product_decomposition_type_D}

	There is a \emph{homeomorphism}
	\begin{equation}
		\label{eq:factorization_D_trees}
		\bm B_{D,r}(\dot Q)
		\simeq \prod_{v \in \mb V} \bm B_D(\dot{\mc T},v),
	\end{equation}
	endowing the target with the product topology.
\end{coro}

\begin{proof}
	Analogous to Cor.~\ref{cor:config_product_decomposition_type_BC},
	using Prop.~\ref{prop:type_D_characterisation_of_realizationsg} +~Thm.~\ref{thm:type_D_configuration_space_from_realizations}.
\end{proof}

\subsubsection{}

The space~\eqref{eq:local_conf_space_D} is a point if $v$ has no nonempty child-vertex.
Otherwise,
if $v$ has $n^+ \geq 1$ of them,
and if $N \geq 1$ is their common partial ramification order (cf.~\S~\ref{sec:local_conf_space_classification}),
then~\eqref{eq:local_conf_space_D} is homeomorphic to one of the hyperplane complements $\mc M(1,n)$,
$\mc M(2,n^+)$,
$\mc M^\sharp(N,n^+)$,
$\mc M^\sharp(2N,n^+)$,
or $\mc M^{\musDoubleSharp}(n',n^+ - n')$,
for some integer $n' \geq 1$,
in the notation of~\eqref{eq:standard_complement_type_D}--\eqref{eq:exotic_arrangement} (recall that $\dim_{\mb C}\bigl( \mc M^{\musDoubleSharp}(k,k') = k + k'$ for all integers $k,k' \geq 0$).
As already mentioned,
the latter is specific to type $D$,
and it appears when the branching at $v$ is of type (IV.b) (cf.~Def.~\ref{def:fission_tree_type_D}).
Coherently with Thm.~\ref{thm:complex_refl_groups_from_gauge},
this factor already arises in the untwisted/unramified setting (cf.~\cite{doucot_rembado_tamiozzo_2026_local_wild_mapping_class_groups_and_cabled_braids}).
Again,
the weight of the present classification is that there are \emph{no} new complements in the twisted one.

Moreover,
each irregular type leads to a \emph{unique} noncrystallographic factor,
if one exists:

\begin{lemm}
	The product~\eqref{eq:factorization_D_trees} contains at most \emph{one} factor of the form~\eqref{eq:exotic_arrangement}.
\end{lemm}

\begin{proof}
	In subcase (IV.b) the fission-vertex $v \in \mb Y$ and all its ancestor fission-vertices are empty.
	Therefore,
	any such vertex is found by starting at the root and descending along empty edges,
	always following the unique empty edge with empty child-vertex;
	the lowest fission-vertex reached this way,
	i.e.,
	with no empty child-vertex,
	is the \emph{unique} fission-vertex of type (IV.b)---%
	if it exists.
\end{proof}

\begin{exem}
	\label{exmp:type_D_tree}

	For the pointed irregular type of Exmp.~\ref{exmp:example_fission_tree_D},
	the fission tree encodes the topological factorization
	\begin{equation}
		\bm B_{D,r}(\dot Q)
		\simeq (\mb C^\ast)^2 \ts \mb C^5 \ts \mc M(1,2) \ts \mc M^\sharp(2,2).
	\end{equation}
	Indeed,
	the branching at the common ancestor-vertex of the leaves $i \in \set{1,2,3}$ is of type (IV.a),
	with two inconsequential vertices that are indirectly mandatory,
	which gives the rightmost factor;
	and the branching at the common ancestor of the leaves $i \in \set{4,5}$ is of type (I.1),
	which gives the second-to-right one.
\end{exem}

\begin{rema}
	\label{rmk:forest_d}

	One can define a $D$-\emph{fission forest} $\bm F$ by gathering isomorphism classes of $D$-fission trees of constant rank,
	and the exact analogue of Rmk.~\ref{rmk:forest_bc} holds---%
	for the structure group $G = \SO_{2m}$.
\end{rema}

\subsection{General case:
	Weyl groups}
\label{sec:weyl_group_tree_d}

Finally,
we describe the configuration spaces of irregular classes via topological quotients,
modulo the action of Weyl groups of (labelled) fission trees (cf.~Def.~\ref{def:config_space_type_D}).
Again,
this yields in particular a recursive description of the Galois group of the general Galois covering of \S~\ref{sec:nonpure_case_general}.

Until \S~\ref{sec:tglwmcgs},
let $\dot Q$ be an $r$-ramified pointed irregular type with irregular class $\Theta \ceqq \Theta(\dot Q)$,
and set $\dot{\mc T} \ceqq \dot{\mc T}(\dot Q)$ and $\mc T \ceqq \mc T(\Theta)$.

\subsubsection{}

Let $\ul{\bm L}_i = (\bm L_i, \varepsilon_i)$ be the enhanced fission datum corresponding to the full branch of a (labelled) leaf $v_i \ceqq \psi(i) \in \mb V_0$,
for $i \in \set{1,\dc,p}$ (with $p \ceqq \abs{\mb V_0} \geq 1$),
of ramification $r_i \geq 1$ (cf.~\S~\ref{sec:enhanced_level_data}).
Let us say that $\bm L_i$ is \emph{nonempty} if it is different both from $\vn_{B\!C}$ and $\vn_D$;
then again we attach to it a finite abelian group:

\begin{defi}[cf.~Def.~\ref{def:interior_weyl_group_leaf_type_BC}]
	The \emph{type}-$D$ \emph{interior Weyl group of} $v_i$ is
	\begin{equation}
		W_D(\dot{\mc T},i)
		\ceqq
		\begin{cases}
			\bigl( \mb Z \bs r_i \mb Z \bigr) \ts \mb Z^{\ts}, & \quad \bm L_i \text { is nonempty and nonspecial},      \\
			\mb Z \bs r_i \mb Z,                               & \quad \bm L_i \text { is nonempty and special},         \\
			\mb Z^{\ts},                                       & \quad \bm L_i = \vn_{B\!C} \text{ or } \bm L_i = \vn_D.
		\end{cases}
	\end{equation}
\end{defi}

\subsubsection{}

We write the elements $g_i = (d_i,\varepsilon_i)$ of $W_D(\dot{\mc T},i)$ as in \S~\ref{sec:interior_weyl_elements},
provided that $\bm L_i$ is nonempty.
Else,
we set $(d_i,\varepsilon_i) \ceqq (0,\varepsilon_i)$,
with $\varepsilon_i \in \mb Z^{\ts}$.

Now the semidirect product~\eqref{eq:group_acting_on_type_BC_extended_tree} is defined in the same way,
but we consider a smaller subgroup thereof to take into account the parity condition of the number of sign-changes.
Namely,
for $i \in \mb V_0$ let
\begin{equation}
	\label{eq:new_signs}
	\wt{\varepsilon}_i
	\ceqq
	\begin{cases}
		\varepsilon_i^{r_i m_i}, & \quad \bm L_i \text{ is nonempty and nonspecial},       \\
		(-1)^{d_i m_i},          & \quad \bm L_i \text{ is nonempty and special},          \\
		\varepsilon_i,           & \quad \bm L_i = \vn_{B\!C} \text{ or } \bm L_i = \vn_D.
	\end{cases}
\end{equation}
Finally,
overwrite~\eqref{eq:common_ramification} as follows,
for any pair of leaves $v_i \neq v_j \in \mb V_0$:
\begin{equation}
	\wt r_{ij}
	\ceqq
	\begin{cases}
		r_{ij}, & \quad \bm L_i \neq \vn_D \neq \bm L_j,
		\\
		1,      & \quad \text{else}.
	\end{cases}
\end{equation}
Then pose the following:

\begin{defi}[cf.~Def.~\ref{def:weyl_group_tree_type_BC}]
	\label{def:weyl_group_tree_type_D}

	The $D$-\emph{Weyl group of} $\dot{\mc T}$ is the subgroup of~\eqref{eq:group_acting_on_type_BC_extended_tree} defined by
	\begin{equation}
		\label{eq:weyl_group_tree_D}
		W_D(\dot{\mc T})
		= \Aut(\mc T) \lts W'_D(\dot{\mc T}),
	\end{equation}
	where:
	\begin{enumerate}
		\item
		      we let
		      \begin{equation}
			      W'_D(\dot{\mc T})
			      \ceqq \Set{ (g_i)_i \in \prod_{i = 1}^p W_D(\dot{\mc T},i) | \wt{\bm\varepsilon} = 1 \text{ and } \varepsilon_i \varepsilon_j = \zeta_{\wt r_{ij}}^{d_i - d_j},
				      \text{ for } i \neq j \in \set{1,\dc,p} };
		      \end{equation}

		\item
		      and---%
		      in turn---%
		      we set $\wt{\bm\varepsilon} \ceqq \prod_{i = 1}^p \wt\varepsilon_i \in \mb Z^{\ts}$.
	\end{enumerate}
\end{defi}

\begin{theo}[cf.~Thm.~\ref{thm:full_wmcg_from_tree_type_BC}]
	\label{thm:full_wmcg_from_tree_type_D}

	\leavevmode
	\begin{enumerate}
		\item
		      The Weyl group~\eqref{eq:weyl_group_tree_D} acts \emph{freely} on $\bm B_{D,r}(\dot Q)$,
		      preserving irregular classes.

		\item
		      And there is a \emph{homeomorphism}
		      \begin{equation}
			      \bm B_{D,r} (\Theta) \simeq \bm B_{D,r}(\dot Q) \bs W_D(\dot{\mc T}).
		      \end{equation}
	\end{enumerate}
\end{theo}

\begin{proof}
	The only difference from the proof of Thm.~\ref{thm:full_wmcg_from_tree_type_BC} is the type-$D$ constraint on the signs,
	which we now discuss---%
	keeping the notation of loc.~cit.

	Choose $g_i \in W_D(\dot{\mc T},i)$,
	for $i \in \set{1,\dc,p} \simeq \mb V_0$.
	The parity of the number of sign-changes between the lists of exponential factors $l^\pm_i$ and $g_i \cdot l^\pm_i$ is given by~\eqref{eq:new_signs},
	provided that $q_i \neq 0$.
	Thus,
	if $\dot Q_0 = (l^\pm_1,\dc,l^\pm_p,-l^\pm_1,\dc,-l^\pm_p)$ is a pointed irregular type,
	then there are the two following usual (mutually-exclusive) situations.

	\begin{enumerate}
		\item
		      The pointed irregular type $\dot Q_0$ does \emph{not} contain the tame circle.
		      Then the product $\wt{\bm \varepsilon} \in \mb Z^{\ts}$ corresponds to the parity of the number of sign-changes between $\dot Q_0$ and $\dot Q_0' \ceqq g \cdot \dot Q_0$,
		      in the notation of~\eqref{eq:transformed_pointed_irr_type}.
		      In view of Prop.~\ref{prop:form_type_D_irreg_classes},
		      this implies that $\dot Q_0$ and $\dot Q_0'$ have the same irregular class if and only if $\wt{\bm\varepsilon} = 1$.
		      (We omit the proof that the resulting action is free.)

		\item
		      Otherwise,
		      $\dot Q_0$ contains a vanishing exponential factor $q_{i_0} = 0$,
		      for a \emph{unique} $i_0 \in \set{1,\dc,p}$.
		      Then all irregular types with given pseudo-irregular class will have the same irregular class (cf.~the proof~\ref{proof:prop_form_type_D_irreg_classes}).
		      Thus,
		      the semidirect product~\eqref{eq:group_acting_on_type_BC_extended_tree} acts on the set of pointed $D$-irregular types with given irregular class.
		      The caveat is that this action is \emph{not} free,
		      and that it restricts to a free action of the subgroup $W_D(\dot{\mc T})$.
		      Indeed,
		      if $\dot Q_0'$ is a pointed irregular type with the same (pseudo-)irregular class as $\dot Q_0$,
		      then the equality
		      \begin{equation}
			      \dot Q_0'
			      = \bigl( \rho \cdot \bm g \bigr) \cdot \dot Q_0,
			      \qquad \rho \in \Aut(\mc T),
			      \quad \bm g
			      = (g_i)_i \in \prod_{i = 1}^p W_D(\dot{\mc T},i),
		      \end{equation}
		      uniquely determines $\rho$ and $g_i$,
		      for $i \neq i_0 \in \set{1,\dc,p}$;
		      finally,
		      the identity $\wt{\bm\varepsilon} = 1$ also forces the choice of $g_{i_0} \in W_D(\dot{\mc T},i_0)$. \qedhere
	\end{enumerate}
\end{proof}

\subsubsection{}
\label{sec:weyl_groups_are_the_same_D}

Finally,
the proof of Cor.~\ref{cor:weyl_groups_are_the_same} translates verbatim to establish that $Z_{W_D,\bm\phi}(g) \simeq W_D(\dot{\mc T})$,
keeping the notation of loc.~cit.
This concludes the proofs of Thmm.~\ref{thm:thm_2_intro} +~\ref{thm:thm_3_intro}.

\section*{Seventh interlude}

In the next section we finally give the complete explicit definition of the local $G$-wild mapping class groups (= WMCGs),
generalizing~\cite{doucot_rembado_tamiozzo_2026_local_wild_mapping_class_groups_and_cabled_braids,
	doucot_rembado_2025_topology_of_irregular_isomonodromy_times_on_a_fixed_pointed_curve,
	boalch_doucot_rembado_2025_twisted_local_wild_mapping_class_groups_configuration_spaces_fission_trees_and_complex_braids}.
Then we state their main properties,
by building on all the previous sections.
Moreover,
we relate the quasigeneric setting to Bessis'~\cite[Thm.~12.4]{bessis_2015_finite_complex_reflection_arrangements_are_k_pi_1},
which in our language embeds the full/nonpure ramified local WMCGs into their unramified versions---%
i.e.,
the full/nonpure $\mf g$-braid groups.
In addition,
we use fission trees to provide more explicit descriptions,
whenever $G$ is a classical group.

Finally,
again,
we twist the interior of the curve in \S~\ref{sec:interior_twist},
covering the general setup of meromorphic reductive $2d$ gauge theory.

\section{Local \texorpdfstring{$G$}{G}-wild mapping class groups in the ramified setting}
\label{sec:tglwmcgs}

\subsection{Main statement}
\label{sec:general_tglwmcgs}

\subsubsection{}

As in \S~\ref{sec:setup},
let $\mc L = \mb C (\!( w )\!)$ be the field of formal Laurent series in a variable $w$,
with valuation subring $\mc L_{\geq 0} = \mb C \llb w \rrb$.
Consider then the standard formal disc $\mc D \ceqq \Spec(\mc L_{\geq 0})$,
with closed marked point $0 \ceqq w \mb C \llb w \rrb$---%
the origin.
We view $(\mc D,0)$ as the standard model for the formal germ of a smooth $1$-pointed projective curve over $\mb C$ (cf.~\S~\ref{sec:background_intro});
or rather,
a ramified covering thereof.
Again the choice of a uniformizer-root is not needed in the definitions,
but it makes them more concrete (cf.~\S~\ref{sec:uniformizers}).

Then:
(i) an irregular type $\wh Q$ as in~\eqref{eq:explicit_irregular_type} yields a triple $\bm{\mc D} \ceqq \bigl(\mc D,0,\wh Q\bigr)$,
which we regard as the formal germ of a \emph{dressed/labelled wild curve};
and (ii) taking the underlying irregular class $\wh\Theta \ceqq \wh\Theta \bigl( \wh Q \bigr)$ yields $\ol{\bm{\mc D}} \ceqq \bigl(\mc D,0,\wh\Theta\bigr)$,
i.e.,
the formal germ of a \emph{bare/unlabelled wild curve} (cf.~\eqref{eq:wild_riemann_surface}).

Now we make the:

\begin{defi}
	\label{def:tglwmcgs}

	Let $r,s \geq 1$ be integers such that:
	(i) $\wh Q$ is $r$-Galois-closed;
	and (ii) $\wh Q$ has pole order bounded by $s$.
	(I.e.,
	one has $\wh Q \in \smash{\wh{\mc{IT}}}^{\mathsmaller\leq s}_r$).
	Then:
	\begin{enumerate}
		\item
		      the \emph{pure} $r$-\emph{ramified} \emph{local WMCG of} $\bm{\mc D}$ is the fundamental group
		      \begin{equation}
			      \label{eq:pure_wmcg}
			      \Gamma_r
			      = \Gamma^{\mathsmaller\leq s}_r \bigl( \wh Q\bigr)
			      \ceqq \pi_1 \bigl( \bm B_r( \wh Q),\wh Q \bigr);
		      \end{equation}

		\item
		      and the (\emph{full/nonpure}) $r$-\emph{ramified} \emph{local WMCG of} $\ol{\bm{\mc D}}$ is the fundamental group
		      \begin{equation}
			      \label{eq:full_wmcg}
			      \ol\Gamma_r
			      = \ol \Gamma^{\mathsmaller\leq s}_r \bigl( \wh\Theta\bigr)
			      \ceqq \pi_1 \bigl( \bm B_r( \wh\Theta),\wh\Theta \bigr).
		      \end{equation}
	\end{enumerate}
\end{defi}

(Cf.~\S~\ref{sec:contractible_centre} about the abusive notation.)

\begin{rema}
	\label{rmk:about_wmcgs}

	By \S~\ref{sec:no_marking},
	in the pure case it would be the same to also fix an element $g \in \mb L_r \bigl( \wh Q \bigr)$,
	and set $\Gamma_{g,r} \bigl( \wh Q \bigr) \ceqq \pi_1 \bigl( \bm B_{g,r}(\wh Q),\wh Q \bigr)$.

	Moreover,
	by \S\S~\ref{sec:generic_pure_case}--\ref{sec:nonpure_case_general},
	the topological spaces $\bm B_{g,r} \bigl( \wh Q \bigr) = \bm B_r \bigl( \wh Q \bigr)$ and $\bm B_r \bigl( \wh\Theta \bigr)$ are path-connected,
	and so changing the base irregular type/class does \emph{not} affect the isomorphism class of the local WMCGs.
	Indeed,
	analogously to the discrete invariants for the admissible deformations of smooth pointed projective curves,
	there is a discrete set of parameters governing the topology of the admissible-deformation spaces:
	the choice of an $r$-Galois-closed depth-$s$ Levi-subcoset filtration $\ul{\mb L}$ of $\mb G = (\mf t,W)$,
	and its underlying $W$-orbit $\ol{\ul{\mb L}}$ (cf.~\S~\ref{sec:modular_groups_are_twisted_cabled_weyl}).
	Therefore,
	neglecting basepoints one might well write
	\begin{equation}
		\Gamma_r (\ul{\mb L})
		\ceqq \pi_1 \bigl( \bm B_r(\ul{\mb L}) \bigr),
		\qquad
		\ol\Gamma_r \bigl( \ol{\ul{\mb L}} \bigr)
		\ceqq \pi_1 \bigl( \bm B_r(\ol{\ul{\mb L}}) \bigr),
	\end{equation}
	in the notation of~\eqref{eq:stratum_of_subcoset_filtration} +~\eqref{eq:stratum_of_subcoset_orbit} (cf.~Cor.~\ref{cor:irregular_moduli}).
\end{rema}

\begin{theo}
	\label{thm:wild_mapping_class_groups}

	Write $\wh Q = \sum_{i = 1}^s A_iw^{-i}$,
	and let $\ul{\mb L} \ceqq \ul{\mb L} \bigl( \wh Q \bigr)$ be the underlying Levi-subcoset filtration.
	Define also the group $Z_{W,\bm \phi}(g) \sse N_W(\mf t_{\bm \phi}) \bs W_{\mf t_{\phi_1}}$ as in Thm.-Def.~\ref{thm:full_general_twisted_deformation_space} (with $g \in \mb L_r \bigl( \wh Q \bigr)$),
	and identify it with the $r$-ramified WWG $W_r \bigl( \wh Q \bigr) = W_r \bigl( \ul{\mb L} \bigr)$ (cf.~\eqref{eq:modular_subquotients_as_ramified_cabled_weyl_groups}).
	Then:
	\begin{enumerate}
		\item
		      there is a \emph{group isomorphism}
		      \begin{equation}
			      \Gamma_r \bigl( \wh Q \bigr)
			      \simeq \prod_{i = 1}^s \Gamma_r \bigl( \wh Q,i \bigr),
			      \qquad \Gamma_r \bigl( \wh Q,i \bigr)
			      \ceqq \pi_1 \bigl( \bm B_r(\ul{\mb L},i),A_i \bigr),
		      \end{equation}
		      in the notation of~\eqref{eq:intrinsic_pure_deformation_factors};

		\item
		      and there is a (typically nonsplit) short \emph{exact group sequence}
		      \begin{equation}
			      \label{eq:wmcg_sequence}
			      1 \lra \Gamma_r \bigl( \wh Q \bigr) \lra \ol \Gamma_r \bigl( \wh \Theta \bigr) \lra W_r \bigl( \wh Q \bigr) \lra 1.
		      \end{equation}
	\end{enumerate}
\end{theo}

\begin{proof}
	The first statement follows from the direct-product decomposition of Thm.-Def.~\ref{thm:general_pure_twisted_deformation_space},
	rephrased more intrinsically by Thm.-Def.~\ref{thm:levi_subcosets_and_modular_groups}.

	The second statement follows from the Galois covering~\eqref{eq:general_galois_covering},
	also rephrased in \S~\ref{sec:modular_groups_are_twisted_cabled_weyl}:
	again,
	the `augmentation' surjection in the sequence~\eqref{eq:wmcg_sequence} corresponds to the action of deck transformations of $\bm B_r \bigl( \wh Q \bigr) \thra \bm B_r \bigl( \wh\Theta \bigr)$.
\end{proof}

\begin{rema}
	\label{rmk:g_braid_group}

	Suppose that $r = s = 1$,
	and write as usual $\wh Q = Aw^{-1}$.
	In this case $\Gamma_1 \bigl( \wh Q \bigr)$ is the fundamental group of the hyperplane complement~\eqref{eq:pure_untwisted_nongeneric_deformation_space_1_coeff},
	and if $A$ is regular then it is the pure $\mf g$-braid group $\PBr_{\mf g} \ceqq \pi_1 (\mf t_{\reg},A)$.
	In turn,
	the sequence~\eqref{eq:wmcg_sequence} becomes
	\begin{equation}
		\label{eq:G_braid_sequence}
		1 \lra \PBr_{\mf g} \lra \Br_{\mf g} \lxra{\pi} W \lra 1,
		\qquad \Br_{\mf g} \ceqq \pi_1 \bigl( \mf t_{\reg} \bs W,W. A \bigr),
	\end{equation}
	and it is centred around the (full) $\mf g$-braid group~\cite{brieskorn_1971_die_fundamentalgruppe_des_raumes_der_regulaeren_orbits_einer_endlichen_komplexen_spiegelungsgruppe,
		deligne_1972_les_immeubles_des_groupes_de_tresses_generalises}.
	One can view the latter as the local WMCG $\ol \Gamma_1 \bigl( \wh \Theta \bigr)$,
	invoking the underlying irregular class $\wh\Theta \ceqq \wh\Theta \bigl( \wh Q \bigr)$.
\end{rema}

\begin{rema}
	Recall that in the local setting we do \emph{not} vary the underlying pointed curve,
	and so here we fix the underlying germ $(\mc D,0)$ of the pointed curve;
	in particular,
	the automorphisms of the scheme $\mc D$ play \emph{no} role in the definition of the local WMCG.
	(Although they act on irregular types/classes by preserving their admissible-deformation spaces;
	cf.~Lem.~\ref{lem:intrinsicality}.)
\end{rema}

\subsection{Braid Springer theory}
\label{sec:braid_springer_theory}

It turns out that the quasigeneric examples of local WMCGs already feature in the seminal work~\cite{bessis_2015_finite_complex_reflection_arrangements_are_k_pi_1} (cf.~\cite{shvartsman_1996_torsion_in_the_quotient_group_of_the_artin_brieskorn_braid_group_with_respect_to_the_center_and_regular_springer_numbers,
	broue_michel_1997_sur_certains_elements_reguliers_des_groupes_de_weyl_et_les_varietes_de_deligne_lusztig_associees,
	bessis_digne_michel_2002_springer_theory_in_braid_groups_and_the_birman_ko_lee_monoid} and \S~\ref{sec:bessis_literature}).

Precisely,
in the context/notation of Rmk.~\ref{rmk:g_braid_group},
but with \emph{arbitrary} ramification $r \geq 1$,
one has $\mb L_r \bigl( \wh Q \bigr) = \set{g}$,
where $g \in W$ is the \emph{unique} element such that $g(A) = \zeta_r A$.
As usual,
it follows that $r$ is a regular number for $W$,
and then:

\begin{enonce}{Theorem-Definition}
	\label{thm:generic_full_twmcgs}

	Let $\beta \in \Br_{\mf g}$ be a $\mf g$-braid such that $\pi(\beta) = g$ (cf.~\eqref{eq:G_braid_sequence}).
	Then:
	\begin{enumerate}
		\item
		      one can choose $\beta$ such that $\tau \ceqq \beta^r \in \PBr_{\mf g} \sse \Br_{\mf g}$ is the `full twist',
		      lying in the centre of the full $\mf g$-braid group;

		\item
		      all the $r$-th roots of $\tau$ are \emph{conjugate} in $\Br_{\mf g}$;

		\item
		      and there is a \emph{group isomorphism}
		      \begin{equation}
			      \label{eq:full_wmcg_as_braid_centralizer}
			      \ol \Gamma_r \bigl( \wh\Theta \bigr)
			      \simeq Z_{\Br_{\mf g}}(\beta) \sse \Br_{\mf g}.
		      \end{equation}
	\end{enumerate}
\end{enonce}

\begin{proof}
	This follows from~\cite[Thm.~12.4]{bessis_2015_finite_complex_reflection_arrangements_are_k_pi_1};
	we give some more details for the first statement only.

	The main fact is that there exists an element $\wt \beta \in \Br_{\mf g}$ such that $\wt \beta^r$ is pure and equals the homotopy class of the loop in $\mf t_{\reg}$ given by
	\begin{equation}
		t \lmt e^{2\pi \sqrt{-1} \cdot t} A,
		\qquad t \in [0,1],
	\end{equation}
	based at $A \in \mf t_{\reg}$.
	(This holds precisely because $r$ is a regular number.)
	The latter is the full twist $\tau$,
	and one proves that:
    (i) it is central in $\Br_{\mf g}$;
	and (ii) the element $\wt g \ceqq \pi(\wt \beta)$ is $r$-regular.
	By Springer,
	there exists now $g' \in W$ such that $g = g' \wt g (g')^{-1}$,
	and choosing a lift $\beta' \in \Br_{\mf g}$ of $g'$ yields
	\begin{equation}
		\beta^r
		= \beta' \wt \beta^r (\beta')^{-1}
		= \wt \beta^r
		= \tau,
		\qquad \beta
		\ceqq \beta' \wt \beta (\beta')^{-1} \in \Br_{\mf g},
	\end{equation}
	as $\tau$ is central---%
	and now indeed $\pi(\beta) = g$.
\end{proof}

\begin{rema}
	When the pole order is higher,
	an analogous statement in the quasigeneric case follows from Prop.~\ref{prop:generic_pure_twisted_deformation_space}---%
	up to deformation-retractions.

	As usual,
	by Thm.-Def.~\ref{thm:generic_full_twmcgs}~(2.),
	the isomorphism class of~\eqref{eq:full_wmcg_as_braid_centralizer} does \emph{not} depend on the choice of $\beta$%
	---but only on the ramification $r$.
\end{rema}

\subsection{Classical examples}

Suppose that $\mf g$ is of type $\bullet \in \set{A,B,C,D}$.
Rephrasing a particular case of Def.~\ref{def:tglwmcgs} in terms of the material of~\cite{boalch_doucot_rembado_2025_twisted_local_wild_mapping_class_groups_configuration_spaces_fission_trees_and_complex_braids} and \S\S~\ref{sec:notation_for_trees}--\ref{sec:D_trees},
we make the following:

\begin{defi}
	Choose an irregular class $\Theta$ of type $\bullet$,
	and suppose\fn{
		W.l.o.g.,
		cf.~Lemm.~\ref{lem:full_gives_pointed_BC} +~\ref{lem:pointed_gives_compatible} and Rmk.~\ref{rmk:config_spaces_are_the_same_d}.}~that $\Theta$ is represented by an $r$-ramified compatible pointed irregular type $\dot Q$,
	of Poincaré--Katz rank $K \ceqq \on{Katz}(\dot Q)$.
	Then:
	\begin{enumerate}
		\item
		      the \emph{pure local WMCG of} $\bm{\mc D} = (\mc D,0,\dot Q)$ is the fundamental group
		      \begin{equation}
			      \Gamma_{\bullet,r} (\dot Q)
			      = \Gamma_{\bullet,r}^{\mathsmaller\leq K} (\dot Q)
			      \ceqq \pi_1 \bigl( \bm B_{\bullet,r}(\dot Q),\dot Q \bigr),
		      \end{equation}
		      where in turn
		      \begin{equation}
			      \bm B_{\bullet,r}(\dot Q)
			      = \bm B_{\bullet,r}^{\leq K}(\dot Q)
			      \ceqq \Set{ \dot Q' \text{ pointed irr.~type } | \on{Katz}(\dot Q') \leq K \text{ and } \dot Q \sim_\bullet \dot Q' };
		      \end{equation}

		\item
		      and the (\emph{full}) \emph{local WMCG of} $\ol{\bm{\mc D}} = (\mc D,0,\Theta)$ is the fundamental group
		      \begin{equation}
			      \ol\Gamma_{\bullet,r}  (\Theta)
			      = \ol\Gamma_{\bullet,r}^{\mathsmaller\leq K} (\Theta)
			      \ceqq \pi_1 \bigl( \bm B_{\bullet,r} (\Theta),\Theta \bigr),
		      \end{equation}
		      where in turn
		      \begin{equation}
			      \bm B_{\bullet,r} (\Theta)
			      = \bm B_{\bullet,r}^{\leq K}  (\Theta)
			      \ceqq \Set{ \Theta' = \Theta(\dot Q') \text{ irr.~class} | \on{Katz}(\dot Q') \leq K \text{ and } \Theta \sim_{\bullet} \Theta' }.
		      \end{equation}
	\end{enumerate}
\end{defi}

\subsubsection{}

Then the main statement is the following:

\begin{theo}
	\label{thm:classical_wmcgs}

	Let $\mc T \ceqq \mc T (\Theta)$ (resp.,
	$\dot{\mc T} \ceqq \dot{\mc T}(\dot Q)$) be the $\bullet$-fission tree of $\Theta$ (resp.,
	the labelled $\bullet$-fission tree of $\dot Q$),
	with vertices $\mb V$ and admissible vertices $\mb A \sse \mb V$.
	(Again,
	the latter should \emph{not} be confused with the Adjoint quotient $\mf g \bs\!\!\bs G \simeq \mf t \bs W$,
	cf.~Rmk.~\ref{rmk:abusive_notation}.)
	Then:
	\begin{enumerate}
		\item
		      there is a group isomorphism
		      \begin{equation}
			      \Gamma_{\bullet,r}(\dot Q) \simeq \Gamma_\bullet(\dot{\mc T}),
			      \qquad \Gamma_\bullet(\dot{\mc T})
			      \ceqq \pi_1 \bigl( \bm B_\bullet(\dot{\mc T}),c \bigr),
		      \end{equation}
		      where $\bm B_\bullet(\dot{\mc T}) \sse \mb C^{\mb A} \simeq \mb C^{\,\abs{\mb A}}$ is the space of realizations of $\dot{\mc T}$,
		      and $c = c_{\dot Q} \cl \mb A \to \mb C$ is the realization canonically associated with $\dot Q$;

		\item
		      there is also a group isomorphism
		      \begin{equation}
			      \Gamma_{\bullet,r}(\dot Q)
			      \simeq \prod_{v \in \mb V} \Gamma_\bullet (\dot{\mc T},v),
			      \qquad \Gamma_\bullet(\dot{\mc T},v)
			      \ceqq \pi_1 \bigl( \bm B_\bullet(\dot{\mc T},v),c_v \bigr),
		      \end{equation}
		      where:
		      (i) $\bm B_\bullet(\dot{\mc T},v) \sse \mb C^{n(v)}$ is the local configuration space of $(\dot{\mc T},v)$;
		      (ii) $c_v$ is the restriction of $c$ to the admissible child-vertices of $v \in \mb V$;
		      and (iii) $n(v) \geq 0$ is the number of admissible child-vertices of $v$;

		\item
		      there is a (typically nonsplit) short exact group sequence
		      \begin{equation}
			      1 \lra \Gamma_{\bullet,r} (\dot Q) \lra \ol\Gamma_{\bullet,r} ( \Theta ) \lra W_\bullet(\dot{\mc T}) \lra 1,
		      \end{equation}
		      where $W_\bullet(\dot{\mc T})$ is the Weyl group of the labelled fission tree;

		\item
		      and the latter sequence is isomorphic to~\eqref{eq:wmcg_sequence},
		      taking the usual $r$-Galois-closed untwisted pullbacks $\wh Q$ and $\wh\Theta$---%
		      of $\dot Q$ and $\Theta$,
		      respectively.
	\end{enumerate}
\end{theo}

\begin{proof}
	The statement summarizes part of~\cite{boalch_doucot_rembado_2025_twisted_local_wild_mapping_class_groups_configuration_spaces_fission_trees_and_complex_braids} (for $\bullet = A$),
	of \S~\ref{sec:BC_trees} (for $\bullet \in \set{B,C}$),
	and of \S~\ref{sec:D_trees} (for $\bullet = D$).
\end{proof}

\section*{Eighth/last interlude}

The last aim is to extend most of the material of \S\S~\ref{sec:setup}--\ref{sec:type_D} to include a different flavour of `twist',
on top of the ramified exponential factors of irregular-singular connections on principal $G$-bundles,
as per~\cite{boalch_yamakawa_twisted_wild_character_varieties}:
cf.~particularly just below Thm.~6 of op.~cit.,
as well as \S~1.1 therein,
and~\cite[Rmk.~8.6]{boalch_2014_geometry_and_braiding_of_stokes_data_fission_and_wild_character_varieties}.
This occupies the next \S~\ref{sec:interior_twist},
and will be phrased in terms of Lehrer--Springer theory for \emph{nonsplit} reflection cosets.

\toctitle{Outer automorphisms for interior twists}
\section{Outer automorphisms (of the group) for interior twists (of the curve)}
\label{sec:interior_twist}

\subsection{Setup for twisted local systems of groups}

Another input from meromorphic $2d$ gauge theory is that one can add on the choice of a background (outer) group automorphism of $G$:
it governs the monodromy of a local system of groups $\mc G$,
defined on the boundary circle $\partial$ of the real-oriented blowup $\wh\Sigma$ of a marked point on a closed Riemann surface $\Sigma$ (cf.~Rmk.~\ref{rmk:interior_twists}).
The infinitesimal/tangent version of this automorphism acts on irregular types/classes,
leading to a more general `twisted' notion of $r$-Galois-closedness (cf.~\S~\ref{sec:setup}).

\subsubsection{}

Denote by $\Aut(\mf g) = \Aut_{\Lie}(\mf g)$ the complex-algebraic group of Lie-algebra automorphisms of $\mf g$.
Consider the short exact group sequence
\begin{equation}
	\label{eq:outer_automorphisms_sequence}
	1 \lra \Inn(\mf g) \lra \Aut(\mf g) \lra \Out(\mf g) \lra 1,
	\qquad \Out(\mf g)
	\ceqq \Aut(\mf g) \bs \Inn(\mf g),
\end{equation}
invoking:
(i) the normal subgroup
\begin{equation}
	\Inn(\mf g)
	= \Inn_{\Lie}(\mf g)
	\ceqq \braket{ \, e^{\ad_X} \, | \, X \in \mf g \, }
	= \Set{ \Ad_g | g \in G }
	\sse \Aut(\mf g),
\end{equation}
of \emph{inner automorphisms};
and (ii) the quotient group of \emph{outermorphisms} (as in `outer automorphisms').

We use the following standard facts,
which break down $\Aut(\mf g)$ into three pieces (recall that we fixed a maximal torus $T \sse G$):

\begin{enonce}{Lemma-Definition}
	\label{lem:outer_automorphisms}

	Let $\mf g' \ceqq [\mf g,\mf g] \sse \mf g$ be the \emph{derived} Lie subalgebra,
	corresponding to the (connected~\cite{wofsey_2022_commutator_subgroup_of_connected_group}) semisimple \emph{derived} algebraic subgroup $G' \ceqq [G,G] \sse G$.
	Denote also by $\mf t' \ceqq \mf t \cap \mf g' \sse \mf g'$ the Cartan subalgebra corresponding to $T' \ceqq T \cap G' \sse G'$---%
	which remains a \emph{maximal torus} in $G'$.
	Then:
	\begin{enumerate}
		\item
		      there is a direct-product decomposition $\Aut(\mf g) \simeq \GL_{\mb C} ( \mf Z ) \ts \Aut(\mf g')$;
		      (Recall that $\mf Z = \mf Z(\mf g) \sse \mf g$ is the Lie-centre.)

		\item
		      the sequence corresponding to~\eqref{eq:outer_automorphisms_sequence},
		      for the \emph{semisimple} Lie algebra $\mf g'$,
		      \emph{splits};

		\item
		      and the section $\Out(\mf g') \hra \Aut(\mf g')$ can---%
		      and will---%
		      be chosen so that its image \emph{preserves} $\mf t'$.
	\end{enumerate}
\end{enonce}

\begin{proof}[Proof postponed to~\ref{proof:lem_outer_automorphisms} (cf.~\cite{fulton_harris_1991_representation_theory}, Prop.~D.40)]
\end{proof}

\begin{coro}
	\label{cor:outer_automorphisms}

	In the notation of Lem.-Def.~\ref{lem:outer_automorphisms},
	there are also:
	\begin{enumerate}
		\item
		      a canonical group isomorphism $\Inn(\mf g) \simeq \Inn(\mf g')$;

		\item
		      and a direct-product decomposition $\Out(\mf g) \simeq \GL_{\mb C} ( \mf Z ) \ts \Out(\mf g')$.
	\end{enumerate}
\end{coro}

\begin{proof}[Proof omitted]
\end{proof}

\subsubsection{}
\label{sec:infinitesimal_automorphisms}

Note that Lem.~\ref{lem:outer_automorphisms}~(2.) requires choosing a noncanonical \emph{pinning},
which is tacit in what follows.
In particular,
one can also identify $\Out(\mf g)$ with a subgroup of Lie-algebra automorphisms of $\mf g$ preserving $\mf t$:
the central factor in Cor.~\ref{cor:outer_automorphisms}~(2.)~acts trivially on $\mf t'$,
and conversely for the `semisimple' part.

Now the general version of~\cite[Eq.~(13)]{boalch_yamakawa_twisted_wild_character_varieties} requires:
(i) looking at an (untwisted) irregular type $\wh Q \in \mf t (\!( w )\!) \bs \mf t \llb w \rrb$;
(ii) choosing an integer $r \geq 1$;
and (iii) imposing that there exists a group automorphism $\bm\varphi \in \Aut(G)$ such that
\begin{equation}
	\label{eq:doubly_twisted_irr_types}
	\sigma \bigl( \wh Q \bigr)
	= \wh Q (\zeta^{-1}_r w)
	= \dot{\bm \varphi} \bigl( \wh Q \bigr),
	\qquad \dot{\bm \varphi} \ceqq T_1(\bm\varphi) \in \Aut(\mf g).
\end{equation}
(We take the tangent map of $\bm \varphi$ at the group unit $1 \in G$.)

Then,
in a generalization of Rmk.~\ref{rmk:relation_with_twist}:

\begin{lemm}
	\label{lem:full_twist_preserves_centralizer}

	One has $\bm \varphi(L) \sse L$,
	where $L \ceqq G^{\wh Q} \sse G$ is the stabilizer of---%
	the coefficients of---%
	$\wh Q$.
\end{lemm}

\begin{proof}[Proof postponed to~\ref{proof:lem_full_twist_preserves_centralizer}]
\end{proof}

\subsubsection{}
\label{sec:explicit_phi_monodromy_realization}

By Lem.~\ref{lem:full_twist_preserves_centralizer},
there exists $\wt g \in L$ such that $\wt{\bm \varphi} \ceqq \bm{\varphi} \circ C_{\wt g} \in \Aut(G)$ also preserves the given maximal torus $T \sse L$---%
as $L$ is connected reductive.
Then apply Lem.-Def.~\ref{lem:outer_automorphisms} to the tangent map $T_1 \wt{\bm \varphi} \in \Aut(\mf g)$,
splitting it into three pieces:
(a) a $\mb C$-linear automorphism $f \in \GL_{\mb C} ( \mf Z )$;
(b) an outermorphism $\dot\varphi' \in \Out(\mf g')$---%
preserving $\mf t'$;
and (c) an inner automorphism $\Ad_{\wt g'} \in \Inn(\mf g')$,
for $\wt g' \in G'$.
The product $\dot\varphi' \cdot \Ad_{\wt g'} \in \Aut(\mf g')$ preserves $\mf t'$,
and so the same holds for $\Ad_{\wt g'}$.
In particular,
the element $\wt g'$ has a well-defined class $g' \in W$.
(Recall that we identify the Weyl groups of $(\mf g,\mf t)$ and $(\mf g',\mf t')$ as \emph{abstract} groups,
cf.~\S~\ref{sec:root_automorphisms}.)

Then,
if we---%
uniquely---%
decompose the coefficients $A_1,\dc,A_s \in \mf t$ of $\wh Q$ as
\begin{equation}
	A_i
	= A^{\mf Z}_i + A'_i,
	\qquad A^{\mf Z}_i \in \mf Z,
	\quad A'_i \in \mf t',
\end{equation}
the condition~\eqref{eq:doubly_twisted_irr_types} is equivalent to the spectral constraints
\begin{equation}
	\label{eq:phi_monodromy_realization}
	\begin{cases}
		f \bigl( A^{\mf Z}_i \bigr)
		= \zeta_r^i A^{\mf Z}_i,                &                            \\
		\dot\varphi' g'(A'_i) = \zeta_r^i A'_i, & \quad i \in \set{1,\dc,s}.
	\end{cases}
\end{equation}

This leads to the following sequence of definitions,
where:
(i) $\dot\varphi \in \Out(\mf g)$ is a \emph{fixed} outermorphism,
regarded either as a Lie-algebra automorphism of $\mf g$ preserving $\mf t$,
or as a $\mb C$-linear automorphism of $\mf t$ (as needed);\fn{
	Viewed on $\mf t$,
	it lies in the group $\Aut(\Phi^{\dual})$,
	which is `dual' to~\eqref{eq:root_autos_big}---%
	and identified with it.}~and (ii) $g' \in W$ corresponds to---%
the class of---%
an inner automorphism of $\mf g$ preserving $\mf t$.

In addition,
however,
we assume that $\dot\varphi$ has \emph{finite order};
i.e.,
equivalently,
that this holds for the central component $f$:
this is consistent with the assumption of~\cite{boalch_yamakawa_2026_polystability_of_stokes_representations_and_differential_galois_groups},
that the monodromy group of the local system of groups $\mc G$ on $\partial$ has finite image in $\Out(G)$.
(However,
this is \emph{not} used to construct twisted wild character varieties.)
As far as this text is concerned,
the main point is that we thus retain usage of \emph{finite} complex reflection groups/cosets.

\begin{defi}[cf.~Def.~\ref{def:galois_closed_irr_types}]
	\label{def:phi_r_galois_closed_irr_types}
	\leavevmode

	\begin{enumerate}
		\item
		      An irregular class $\wh\Theta = \wh\Theta \bigl( \wh Q \bigr) \in \wh{\mc{IT}} \bs W$ is $(\dot\varphi,r)$-\emph{Galois-closed} if one has
		      \begin{equation}
			      \label{eq:phi_r_galois_closed_irr_class}
			      \wh\Theta \bigl( \wh Q(\zeta^{-1}_r w) \bigr)
			      = \dot\varphi \bigl( \wh\Theta \bigr) \ceqq \wh\Theta \bigl( \dot\varphi(\wh Q) \bigr) \in \wh{\mc{IT}} \bs W.
		      \end{equation}

		\item
		      An irregular type $\wh Q \in \wh{\mc{IT}}$ is $(\dot\varphi,r)$-\emph{Galois-closed} if this holds for its irregular class.

		\item
		      And the---%
		      $W$-invariant---%
		      subset of $(\dot\varphi,r)$-Galois-closed irregular types is denoted by $\wh{\mc{IT}}_{\dot\varphi,r} \sse \wh{\mc{IT}}$.
	\end{enumerate}
\end{defi}

\begin{rema}
	The condition~\eqref{eq:phi_r_galois_closed_irr_class} makes sense,
	because the irregular classes of $\dot\varphi \bigl( \wh Q \bigr)$ and $\dot\varphi \bigl( g(\wh Q) \bigr)$ coincide,
	for any $g \in W$.
	Precisely,
	one has
	\begin{equation}
		\dot\varphi \bigl( g(\wh Q) \bigr)
		= \wt g \bigl( \dot\varphi (\wh Q) \bigr),
		\qquad \wt g
		= \wt g_{\dot\varphi'}
		\ceqq \dot\varphi' g (\dot\varphi')^{-1} \in W. \qedhere
	\end{equation}
\end{rema}

\subsubsection{}

The condition of Def.~\ref{def:phi_r_galois_closed_irr_types}~(2.)~can be rephrased by \emph{twisting} the Galois action on irregular types/classes.
Namely,
let again $z = w^r$ be the `descended' formal variable corresponding to the cyclic Galois extension
\begin{equation}
	\mc K
	= \mb C (\!( z )\!) \lhra \mc L
	= \mb C (\!( w )\!),
\end{equation}
of complete DVFs.
Recall that we let the generator $\zeta_r \in \mu_r \simeq \Gal (\mc L \bs \mc K)$ act on `lifted' irregular types by the \emph{inverse} action $\sigma \bigl( \wh Q(w) \bigr) = \wh Q (\zeta_r^{-1} w)$,
which is natural when working with principal parts of formal Laurent series.

Define then the $\dot\varphi$-\emph{twisted} version of this action as
\begin{equation}
	\wt \sigma \bigl( \wh Q(w) \bigr)
	= \wt \sigma_{\dot\varphi,r} \bigl( \wh Q(w) \bigr)
	\ceqq \dot\varphi^{-1} \sigma \bigl( \wh Q \bigr)
	= \dot\varphi^{-1} \wh Q( \zeta_r^{-1} w) \in \wh{\mc{IT}},\fn{
		Strictly speaking,
		we consider the action of $\dot\varphi^{-1} \ots \sigma$ on the tensor product $\mf t \ots_{\mb C} \mc L$;
		the vector subspace $\mf t \ots_{\mb C} \mc L_{\geq 0} \sse \mf t \ots_{\mb C} \mc L$ is preserved,
		and this induces an action on the quotient $\wh{\mc{IT}} = \mf t \ots_{\mb C} (\mc L \bs \mc L_{\geq 0})$.
	}
\end{equation}
where it remains convenient to take the inverse outermorphism:
an irregular class $\wh\Theta = \wh \Theta \bigl( \wh Q \bigr)$ is now $(\dot\varphi,r)$-Galois-closed if and only if one has
\begin{equation}
	\wh \Theta
	= \wt\sigma \bigl( \wh\Theta \bigr)
	\ceqq \wh\Theta \bigl( \wt\sigma\bigl( \wh Q \bigr) \bigr) \in \wh{\mc{IT}} \bs W.
\end{equation}

Then one can make the following series of statements:

\begin{enonce}{Lemma-Definition}
    \leavevmode

    \begin{enumerate}
        \item
    The least common multiple
    \begin{equation}
	\label{eq:twisted_ramification}
	\wt r
	= r \vee \dot r,
	\qquad \dot r
	\ceqq \ord (\dot\varphi) \geq 1,
\end{equation}
is the \emph{twisted ramification of the pair} $(\dot\varphi,r)$.

\item
And then the order of the cyclic group $\braket{ \wt\sigma } \sse \GL_{\mb C} \bigl( \wh{\mc{IT}} \bigr)$ equals $\wt r$.
    \end{enumerate}
\end{enonce}

\begin{proof}
For the second statement,
observe that for any integer $j > 0$ the identity $\wt \sigma^j = \Id_{\wh{\mc{IT}}}$ implies that $\dot\varphi^j = \zeta_r^{ij} \cdot \Id_{\mf t} \in \GL_{\mb C}(\mf t)$ for all $i > 0$.
In particular,
one has $\dot\varphi^j = \zeta_r^{rj} \cdot \Id_{\mf t} = \Id_{\mf t}$ (taking $i = r$),
and so $\zeta_r^j = 1$ (taking $i = 1$ in the new equality $\Id_{\mf t} = \zeta_r^{ij}$).
Therefore,
the integer $\wt r \geq 1$ divides the order.
Conversely,
one has $\wt\sigma^j = \dot\varphi^{-j} \sigma^j$ for all $j \in \mb Z$,
as the actions of $\dot\varphi$ and $\sigma$ on irregular types commute---%
and so $\wt \sigma^{\wt r} = 1$.
\end{proof}

\begin{defi}[cf.~Def.~\ref{def:r_galois_orbit}]
	The $(\dot\varphi,r)$-\emph{Galois-orbit of} $\wh Q$ is the finite subset
	\begin{equation}
		\label{eq:twisted_galois_orbit}
		\braket{ \wh Q }_{\!\dot\varphi,r}
		\ceqq \Set{ \wt \sigma^j \bigl( \wh Q \bigr) | j \in \mb Z \bs \, \wt r \, \mb Z } \sse \wh{\mc{IT}}.
	\end{equation}
\end{defi}

\subsubsection{}

Now an irregular type $\wh Q$ is $(\dot\varphi,r)$-Galois-closed if and only if its $(\dot\varphi,r)$-Galois-orbit lies completely within its irregular class,
and to this extent:

\begin{enonce}{Lemma-Definition (cf.~Lem.-Def.~\ref{lem:generating_galois_orbit})}
	\label{lem:generating_twisted_galois_orbit}

	Choose $\wh Q \in \wh{\mc{IT}}$.
	Then:
	\begin{enumerate}
		\item
		      the irregular type $\wh Q$ is $(\dot\varphi,r)$-\emph{Galois-closed} if and only if there exists $g' \in W$ such that
		      \begin{equation}
			      \label{eq:twisted_monodromy_realization}
			      g' \bigl( \wh Q \bigr)
			      = \wt\sigma \bigl( \wh Q \bigr) \in \wh{\mc{IT}};
		      \end{equation}

		\item
		      in this case,
		      one has the equalities
		      \begin{equation}
			      \label{eq:twisted_orbit_comparison}
			      \braket{ \wh Q }_{\!\dot\varphi,r}
			      = \Set{ g'_j \bigl( \wh Q \bigr) | j \in \mb Z_{> 0} }
			      = \Set{ g'_j \bigl( \wh Q \bigr) | j \in \set{1,\dc,\wt r} } \sse \wh\Theta \bigl( \wh Q \bigr),
		      \end{equation}
		      where
		      \begin{equation}
			      \label{eq:twisted_power}
			      g'_j
			      \ceqq \dot\varphi^{1-j} (g'\dot\varphi)^{j-1} g' \in W,
			      \qquad j \in \mb Z;
		      \end{equation}

		\item
		      therefore,
		      if~\eqref{eq:twisted_monodromy_realization} holds we say that $g'$ \emph{generates the} $(\dot\varphi,r)$-\emph{Galois-orbit of} $\wh Q$;

		\item
		      we also denote by $\mb L_{\dot\varphi,r} \bigl( \wh Q \bigr) \sse W$ the set of such generators;
		      (So $\mb L_{\dot\varphi,r} \bigl( \wh Q \bigr) \neq \vn$ if and only if $\wh Q$ is $(\dot\varphi,r)$-Galois-closed.)

		\item
		      and conversely,
		      the vector subspace of irregular types whose $(\dot\varphi,r)$-Galois-orbit is \emph{generated by} $g'$ is denoted by $\wh{\mc{IT}}_{\dot\varphi g',r} \sse \wh{\mc{IT}}$.
	\end{enumerate}
\end{enonce}

\begin{proof}
	The proof~\ref{proof:lem_generating_galois_orbit} extends,
	after establishing that the identity~\eqref{eq:twisted_monodromy_realization} yields
	\begin{equation}
		g'_j \bigl( \wh Q \bigr)
		= \wt\sigma^j \bigl( \wh Q \bigr),
		\qquad j \in \mb Z_{> 0},
	\end{equation}
	in the notation of~\eqref{eq:twisted_power}.
	Indeed,
	inductively,
	for $j \geq 2$ one has
	\begin{align}
		g'_j \bigl( \wh Q \bigr)
		 & = \dot\varphi^{1-j} (g'\dot\varphi)^{j-1} \cdot \dot\varphi ^{-1} \sigma \bigl( \wh Q \bigr)
		= \dot\varphi ^{-1} \sigma \bigl( \dot\varphi^{2-j} (g'\dot\varphi)^{j-2} \cdot g' \dot\varphi \cdot \dot\varphi ^{-1} \bigl( \wh Q \bigr) \bigr) \\
		 & = \wt \sigma \bigl( \dot\varphi^{2-j}(g'\dot\varphi)^{j-2} g' \bigl( \wh Q \bigr) \bigr)
		= \wt \sigma \bigl( \wt\sigma^{j-1} \bigl( \wh Q \bigr) \bigr),
	\end{align}
	commuting $\sigma$ past $g'$ and $\dot\varphi^{\pm 1}$.
	In particular,
	one now has $g'_{\wt r} \bigl( \wh Q \bigr) = \wh Q$,
	invoking the `twisted' ramification~\eqref{eq:twisted_ramification},
	so that:
	(i) the notation $g'_j \bigl( \wh Q \bigr)$ makes sense for $j \in \mb Z \bs \wt r \, \mb Z$;
	and (ii) the cardinality of the orbit~\eqref{eq:twisted_galois_orbit} equals the minimal divisor $j$ of $\wt r$ such that $g'_j \in W$ stabilizes $\wh Q$.
	(Beware that this does \emph{not} prove that $g'$ normalizes the stabilizer of $\wh Q$,
	cf.~\S~\ref{sec:independence_r_root_twisted} +~Cor.-Def.~\ref{cor:levi_subcoset_twisted}.)
\end{proof}

\begin{defi}[cf.~Deff.~\ref{def:twisted_pure_admissible_deformations} +~\ref{def:twisted_admissible_deformation_full}]
	\label{def:phi_r_admissible_deformations}
	\leavevmode

	\begin{enumerate}
		\item
		      Two $(\dot\varphi,r)$-Galois-closed irregular types $\wh Q$ and $\wh Q'$ are \emph{mutual} $(g',\dot\varphi,r)$-\emph{admissible deformations} (resp.,
		      \emph{mutual}-$(\dot\varphi,r)$~\emph{admissible deformations}),
		      which is symbolized by $\wh Q \sim_{g',\dot\varphi,r} \wh Q'$ (resp.,
		      by $\wh Q \sim_{\dot\varphi,r} \wh Q'$),
		      if:
		      \begin{enumerate}
			      \item
			            one has $g' \in \mb L_{\dot\varphi,r} \bigl( \wh Q \bigr) \cap \mb L_{\dot\varphi,r} \bigl( \wh Q' \bigr)$ (resp.,
			            $\mb L_{\dot\varphi,r} \bigl( \wh Q \bigr) \cap \mb L_{\dot\varphi,r} \bigl( \wh Q' \bigr) \neq \vn$);

			      \item
			            and their $\Phi$-tuples~\eqref{eq:root_valuations_tuple} coincide.
		      \end{enumerate}

		\item
		      And two $(\dot\varphi,r)$-Galois-closed irregular classes $\wh\Theta$ and $\wh\Theta'$ are \emph{mutual}-$(\dot\varphi,r)$~\emph{admissible deformations},
		      which is symbolized by $\wh\Theta \sim_{\dot\varphi,r} \wh\Theta'$,
		      if there exist two irregular types $\wh Q$ and $\wh Q'$ such that:
		      \begin{enumerate}
			      \item
			            one has $\wh\Theta = \wh\Theta \bigl( \wh Q \bigr)$ and $\wh\Theta' = \wh\Theta \bigl( \wh Q' \bigr)$;

			      \item
			            and one has $\wh Q \sim_{\dot\varphi,r} \wh Q'$.
		      \end{enumerate}
	\end{enumerate}
\end{defi}

\subsubsection{}

After bounding the pole order by an integer $s \geq 1$,
the spaces of $(\dot\varphi,r)$-admissible deformations (resp.,
of $(g',\dot\varphi,r)$-admissible deformations) are denoted by $\bm B^{\mathsmaller\leq s}_{\dot\varphi,r} \bigl( \wh Q \bigr) = \bm B_{\dot\varphi,r} \bigl( \wh Q \bigr)$ and~$\bm B^{\mathsmaller\leq s}_{\dot\varphi,r} \bigl( \wh\Theta \bigr) = \bm B_{\dot\varphi,r} \bigl( \wh\Theta \bigr)$ (resp.,
by~$\bm B^{\mathsmaller\leq s}_{\dot\varphi g',r} \bigl( \wh Q \bigr) = \bm B_{\dot\varphi g',r} \bigl( \wh Q \bigr)$;
cf.~\S\S~\ref{sec:pure_def_spaces} +~\ref{sec:full_def_spaces}).
We still regard them as topological subspaces of $\mf t^s$,
using the uniformizer-root $w$ (cf.~\S~\ref{sec:topology}).

As in the previous case where $\dot\varphi = 1$,
one has a priori
\begin{equation}
	\bm B_{\dot\varphi,r} \bigl( \wh Q \bigr)
	= \bigcup_{\mb L_{\dot\varphi,r} (\wh Q)} \bm B_{\dot\varphi g',r} \bigl( \wh Q \bigr);
\end{equation}
but we will show (in \S~\ref{sec:no_doubly_twisted_marking}) that
\begin{equation}
	\bm B_{\dot\varphi g',r} \bigl( \wh Q \bigr) = \bm B_{\dot\varphi g'',r} \bigl( \wh Q \bigr),
	\qquad g',g'' \in \mb L_{\dot\varphi,r} \bigl( \wh Q \bigr).
\end{equation}

\subsection{More general remarks/conventions}

The remainder of \S~\ref{sec:interior_twist} extends the material of \S\S~\ref{sec:generic_pure_case}--\ref{sec:type_D} to the present `doubly-twisted' situation.
In preparation for that,
we conclude the extended setup with a few general comments (cf.~\S~\ref{sec:concluding_remarks_conventions}).

\subsubsection{}
\label{sec:outer_isogeny}

Differentiating $\bm\varphi \mt \dot{\bm\varphi}$ yields a group embedding $\Aut(G) \hra \Aut(\mf g)$,
as $G$ is connected (cf.~\S~\ref{sec:infinitesimal_automorphisms}).
The image of this subgroup in $\Out(\mf g)$ is isomorphic to $\Out(G) \ceqq \Aut(G) \bs \Inn(G)$,
invoking the normal subgroup $\Inn(G) \sse \Aut(G)$ of \emph{inner automorphisms/conjugations}.
(The embedding $\Inn(G) \hra \Aut(\mf g)$ surjects precisely onto $\Ad(G) = \Inn(\mf g) \sse \Aut(\mf g)$.)
Contrary to the inner part (cf.~\S~\ref{sec:no_isogeny}),
the subgroup $\Out(G) \hra \Out(\mf g)$ \emph{does} depend on the choice of connected group integrating $\mf g$---%
even when $\mf g$ is semisimple,
via isogeny.
Hereafter we take the uniform approach of treating \emph{all} elements $\dot\varphi \in \Out(\mf g)$,
under the finite-order hypothesis.

In addition,
however,
note that if $\dot\varphi$ comes from an outermorphism of $G$ then its central part $f \in \GL_{\mb C} ( \mf Z )$ \emph{preserves} the $\mb Q$-form of the Lie-centre determined by the subtorus $Z(G)^o \sse T$:
so we shall also assume that this is the case.

\subsubsection{}
\label{sec:outer_autos_from_diagram}

The standard arguments in~\ref{proof:lem_outer_automorphisms} also yield a proof of the following standard facts:
\begin{coro}
	\label{cor:outer_root_autos}
	\leavevmode

	\begin{enumerate}
		\item
		      There is a group isomorphism
		      \begin{equation}
			      \Out(\mf g') \simeq \Out(\Phi')
			      = \Aut(\Phi') \bs W,
		      \end{equation}
		      invoking the group of \emph{outermorphisms} of the root system $\Phi'$ (cf.~\S~\ref{sec:root_automorphisms}).

		\item
		      Moreover,
		      choosing a base $\Delta' \sse \Phi'$ of simple roots provides a semidirect splitting $\Aut(\Phi') \simeq \Aut(\Delta') \lts W$ (cf.~\eqref{eq:diagram_automorphisms}),
		      thereby identifying $\Out(\Phi')$ with the group of automorphisms of the Dynkin diagram of $(\mf g',\mf t',\Delta')$---%
		      and with a subgroup of $\GL_{\mb C}(\mf t)$,
		      extending by the identity of the (dual) Lie-centre.
	\end{enumerate}
\end{coro}

\begin{proof}[Proof omitted]
\end{proof}

\subsubsection{}

The statement of Cor.~\ref{cor:outer_root_autos}~(1.)~extends to the reductive case,
via the splitting $\Aut(\Phi) \simeq \GL_{\mb C} \bigl( \mf Z^{\dual} \bigr) \ts \Aut(\Phi')$ (compare with Lem.-Def.~\ref{lem:outer_automorphisms}~(1.)).

It follows that $\Out(\Phi) \simeq \Out(\mf g)$,
by Cor.~\ref{cor:outer_automorphisms} (cf.~\eqref{eq:outer_root_autos};
this was implicitly used in \S~\ref{sec:howlett_twist}.)
The statement of Cor.~\ref{cor:outer_root_autos}~(2.)~also extends,
provided that one defines
\begin{equation}
	\Aut(\Delta)
	\ceqq \Set{ \tau \in \Aut(\Phi) | \tau(\Delta) = \Delta }
	\simeq \GL_{\mb C} \bigl( \mf Z^{\dual} \bigr) \ts \Aut(\Delta'),
\end{equation}
where the base $\Delta \ceqq \Delta'$ of $\Phi$ is seen as a subset of $(\mf t')^{\dual} \simeq ( \mf Z )^\perp \sse \mf t^{\dual}$ (cf.~Rmk.~\ref{rmk:reductive_base}).

\subsubsection{}

Hereafter,
we fix the above notation/splitting:
\begin{equation}
	\dot\varphi
	= (f,\dot\varphi') \in \GL_{\mb C} ( \mf Z ) \ts \Out(\Phi').
\end{equation}
Moreover,
the symbol $\bm g$ will \emph{always} denote a product
\begin{equation}
	\bm g \ceqq \dot\varphi g'
	= (f,\dot\varphi' g') \in \GL_{\mb C} ( \mf Z ) \ts \Aut(\Phi'),
	\qquad g' \in W,
\end{equation}
which we view as an element of $\GL_{\mb C}(\mf t)$;
and we write $\bm g' \ceqq \dot\varphi' g' \in \Aut(\Phi')$ for the `semisimple' part of $\bm g$.

In the end,
one works within the rational reflection coset $\wt{\mb G} \ceqq (\mf t, \dot\varphi W)$,
translating $\mb G = (\mf t,W)$,
whose underlying coset lies in $\Aut(\Phi')$ only if $f$ is trivial (cf.~\S~\ref{sec:reflection_cosets}).
This is the reflection-theoretic trace of the `twist' of $G$ considered in~\cite[\S~1.1]{boalch_yamakawa_twisted_wild_character_varieties}:
importantly,
only the restricted twisted action on $\mf t$ is visible when considering deformation spaces $\bm B$ of wild curves/Riemann surfaces $\bm X$;
while more information about $G$ is used when shifting focus to the twisted wild character varieties,
which assemble into Poisson/symplectic local systems over $\bm B$,
etc.
(Cf.~\S~\ref{sec:intro}.)

\subsubsection{}
\label{sec:aut_phi_conjugacy}

The material of \S~\ref{sec:uniformizers} extends to the present situation,
establishing the independence of the admissible-deformation spaces from the choice of uniformizer,
and from an $r$-th root thereof,
up to homeomorphism.
Precisely,
the proof~\ref{proof:lem_intrinsicality} extends because the whole of the diagonal action of $\Aut(\Phi)$ (on coefficients of irregular types) commutes with the Galois action~\eqref{eq:monodromy}.

On a related note,
the permutation of deformation spaces of~\eqref{eq:g_strata_permutation} also extends.
Precisely,
choosing $\bm h \in \Aut(\Phi)$ yields
\begin{equation}
	\label{eq:twisted_g_strata_permutation}
	\bm h \bigl( \bm B_{\bm g,r} \bigl( \wh Q \bigr) \bigr)
	= \bm B_{\bm{\wt g},r} \bigl( \wh Q' \bigr) \sse \smash{\wh{\mc{IT}}}_{\bm{\wt g},r},
	\qquad \bm{\wt g}
	\ceqq \bm h \bm g \bm h^{-1},
	\quad \wh Q'
	\ceqq \bm h \bigl( \wh Q \bigr).
\end{equation}
Therefore,
as usual,
it is only the $\Aut(\Phi)$-orbit of these spaces which matters---%
topologically.\fn{
	Of course conjugating by $\bm h$ permutes the $W$-cosets in $\Aut(\Phi)$,
	moving $\dot\varphi$;
	but this is only relevant if $\Out(\Phi)$ is \emph{nonabelian}.}~Beware however that the immediate generalization of~\eqref{eq:conjugated_orbit_generators} only holds after twisting the subset $\mb L_{\dot\varphi,r} \bigl( \wh Q \bigr) \sse W$ by $\dot\varphi$ (cf.~Cor.-Def.~\ref{cor:levi_subcoset_twisted}~(3.) +~\eqref{eq:twisted_conjugated_orbit_generators}).

\subsubsection{}
\label{sec:independence_r_root_twisted}

Again the choice of a primitive $r$-th root of $1$ is immaterial (cf.~\S~\ref{sec:independence_r_root}).

Namely,
let $\braket{\dot\varphi,W} \simeq \braket{\dot\varphi} \lts W \sse \GL_{\mb C}(\mf t)$ be the subgroup generated by $W$ and $\dot\varphi$,
and consider again the stabilizer $W^{\wh Q} \sse W$ of an irregular type $\wh Q$.
Then the $(\dot\varphi,r)$-Galois-closedness of $\wh Q$ is equivalent to the existence of a group morphism
\begin{equation}
	\label{eq:galois_closedness_morphism}
	\mu_r \lra N_{\braket{\dot\varphi,W}} \bigl( W^{\wh Q} \bigr) \bs W^{\wh Q},
	\qquad \zeta \lmt \bm g_\zeta W^{\wh Q},
\end{equation}
where the choice of lifts is also immaterial,
such that
\begin{equation}
	\sigma_\zeta \bigl( \wh Q (w) \bigr)
	= \wh Q (\zeta^{-1}w)
	= \bm g_\zeta \bigl( \wh Q(w) \bigr),
	\qquad \zeta \in \mu_r,
\end{equation}
in the notation of~\eqref{eq:intrinsic_monodromy}.
Here it remains convenient to fix the generator $\zeta_r \in \mu_r$,
so that~\eqref{eq:galois_closedness_morphism} is determined---%
as we did above---%
by $\bm g_{\zeta_r} W^{\wh Q} = \dot\varphi g' W^{\wh Q}$.

In particular,
if one still defines the $r$-\emph{Galois-orbit} $\braket{ \wh Q }_r$ as in~\eqref{eq:global_r_galois_orbit},
then the latter now lies in the $\braket{\dot\varphi,W}$-orbit of $\wh Q$,
rather than its $W$-orbit/irregular class.
(Compare with Lem.-Def.~\ref{lem:generating_twisted_galois_orbit}.)
Precisely,
one has
\begin{equation}
	\label{eq:twisted_r_galois_orbit}
	\braket{\wh Q}_{\! r}
	= \Set{ \bm g_r^j \bigl( \wh Q \bigr) | j \in \mb Z \bs r \mb Z } \sse \wh{\mc{IT}},
\end{equation}
with unambiguous notation for cyclic powers:
so again one may say that $\bm g$ \emph{generates the} $r$-\emph{Galois-orbit of} $\wh Q$,
and this happens---%
by definition---%
if and only if $g' = \dot\varphi^{-1} \bm g$ lies in $\mb L_{\dot\varphi,r} \bigl( \wh Q \bigr)$.
The cardinality of~\eqref{eq:twisted_r_galois_orbit} equals the twist-order of $\bigl(\mf t,\bm g W^{\wh Q}\bigr)$,
which is a Levi subcoset of $\wt{\mb G}$,
and again one may say that $\wh Q$ is \emph{primitive} if this number equals $r$ (rather than just a divisor thereof,
cf.~Lem.-Def.~\ref{lem:primitive_irregular_types} +~Cor.-Def.~\ref{cor:levi_subcoset_twisted}).

Incidentally,
observe that:
(i) one has $\bm g^j \in \dot\varphi^j W$ for $j \in \mb Z$;
and (ii) the `starting' coset $\dot\varphi W \sse \braket{ \dot\varphi,W }$ can be replaced by a power thereof,
coprime with $\dot r = \ord(\dot\varphi) \geq 1$,
without changing the cyclic group $\braket{\dot\varphi} \sse \Out(\Phi)$.

\subsubsection{}
\label{sec:twisted_loop_algebra}

Finally,
one should relate to \emph{twisted loop algebras}:
cf.~the survey~\cite{senesi_2010_finite_dimensional_representation_theory_of_loop_algebras_a_survey},
and let $\mf g$ be semisimple here for the sake of simplifying the notation.

Let $\dot{\bm\varphi} \in \Aut(\mf g)$ be a finite-order Lie-algebra automorphism.
Consider the formal loop algebra $\mc L\mf g = \mf g (\!( w )\!) \ceqq \mf g \ots_{\mb C} \mb C (\!( w )\!)$.
For any integer $r \geq 1$,
define the \emph{twisted loop algebra of} $(\mf g,\dot{\bm\varphi},r)$ as the Lie subalgebra of invariants
\begin{equation}
	\label{eq:twisted_loop_algebra}
	\mc L_r(\mf g,\dot{\bm\varphi})
	\ceqq \mf g(\!( w )\!)^{\mc L\dot{\bm\varphi}} = \Set{ \bm Y \in \mf g(\!( w )\!) | \mc L \dot{\bm\varphi} (\bm Y) = \bm Y },
\end{equation}
where in turn $\mc L \dot{\bm\varphi} = \mc L_r \dot{\bm\varphi} \in \Aut \bigl( \mf g (\!( w )\!) \bigr)$ is the Lie-algebra automorphism obtained by (continuous) $\mb C$-bilinear extension of the prescription
\begin{equation}
	Y \ots w^i
	\lmt \zeta_r^i \dot{\bm\varphi}(Y) \ots w^i,
	\qquad i \in \mb Z, \quad Y \in \mf g.
\end{equation}
So the elements of~\eqref{eq:twisted_loop_algebra} consist of $\mf g$-valued formal Laurent series $\bm Y = \bm Y(w)$ such that $\bm Y(\zeta^{-1}_r w) = \dot{\bm \varphi} \bigl( \bm Y(w) \bigr)$---%
acting by $\dot{\bm\varphi}$ on each coefficient.

To relate to our setting,
choose a Cartan subalgebra $\mf t \sse \mf g$ such that $\dot{\bm \varphi}(\mf t) \sse \mf t$ (cf.~\cite[\S~8.1--8.3]{kac_1990_infinite_dimensional_lie_algebras}).
Then,
writing as above
\begin{equation}
	\dot\varphi g'
    = \bm g
	\ceqq \eval[1]{\dot{\bm \varphi}}_{\mf t} \in \GL_{\mb C}(\mf t) \simeq \Aut(\mf t),
\end{equation}
there is now a $\mb C$-linear isomorphism:
\begin{equation}
	\wh{\mc{IT}}_{\bm g,r}
	\simeq \mc L_r(\mf t,\bm g) \bs \bigl( \mc L_r(\mf t,\bm g) \cap \mf t \llb w \rrb \bigr) \sse \mf t (\!( w )\!) \bs \mf t \llb w \rrb.
\end{equation}

\begin{rema}
	\label{rmk:twisted_loop_algebra_general}

	Strictly speaking,
    a priori the invariants of~\eqref{eq:twisted_loop_algebra} yield a slightly more general notion than the standard one of~\cite{senesi_2010_finite_dimensional_representation_theory_of_loop_algebras_a_survey}:
    the latter corresponds to the case where $\dot r \ceqq \ord (\dot{\bm\varphi}) \geq 1$ divides $r$,
    and it is possible to get back into the standard setting.
	
    Indeed,
    precisely,
	set in general $\ul r \ceqq r \wdg \dot r \geq 1$,
	and note that the degree-$i$ coefficient of an invariant element lies in the eigenspace $\mf g ( \dot{\bm\varphi},\zeta^{-i}_r ) \sse \mf g$:
    this space vanishes unless $\zeta^{-i}_r \in \mu_{\dot r} \cap \mu_r = \mu_{\ul r}$,
	i.e.,
    unless $i$ is a multiple of the integer $r \slash \ul r = \wt r \slash \dot r$,
	in the notation of~\eqref{eq:twisted_ramification}.
	Therefore,
    one has
	\begin{equation}
		\mc L_r(\mf g,\dot{\bm\varphi})
        = \bops_{\mb Z} \mf g ( \dot{\bm\varphi},\zeta^{-j}_{\ul r} ) \ots_{\mb C} (\mb C \cdot \ul w^j),
		\qquad \ul w
		\ceqq w^{\wt r \slash \dot r},
	\end{equation}
    where as usual $\zeta_{\ul r} = e^{2\pi \sqrt{-1} \slash \ul r} \in \mb C^{\ts}$.
    The latter is the twisted loop algebra of the triple
    \begin{equation} 
        (\mf g_{\ul r},\ul{\dot\varphi},\ul r),
        \qquad \ul{\dot\varphi}
    \ceqq \eval[1]{\dot{\bm\varphi}}_{\mf g_{\ul r}},
    \end{equation}
    in the uniformizer $\ul w$,
	invoking the $\dot{\bm\varphi}$-stable Lie subalgebra
	\begin{equation}
        \mf g_{\ul r}
        \ceqq \bops_{\mu_{\ul r}} \mf g ( \dot{\bm\varphi},\zeta ) \sse \mf g.
	\end{equation}
    (Note that $\bigl[ \mf g(\dot{\bm\varphi},\zeta),\mf g(\dot{\bm\varphi},\zeta') \bigr] \sse \mf g(\dot{\bm\varphi},\zeta\zeta')$ for $\zeta,\zeta' \in \mu_{\ul r}$.)
    
    This is now in the standard sense,
    as $\ul{\dot\varphi}^{\ul r} = \Id_{\mf g_{\ul r}}$.
    Explicitly:
    (i) at one extreme,
    one has $\mf g_{\ul r} = \mf g$ if and only if $\dot r \mid r$,
	i.e. if and only if $\wt r = r$;
	and (ii) at the other extreme,
    if $r \wdg \dot r = 1$ then $\mf g_{\ul r} = \mf g^{\dot{\bm\varphi}}$,
    and the loop algebra $\mc L_r (\mf g,\dot{\bm\varphi}) \simeq \mf g^{\dot{\bm\varphi}} (\!( z )\!)$ is actually \emph{untwisted}---%
    with the `descended' uniformizer $\ul w = w^r = z$.
\end{rema}

\subsection{Pure quasigeneric case:
	minimal pole order}

We still consider as `quasigeneric' a $(\dot\varphi,r)$-Galois-closed irregular type whose leading term is a $W$-regular vector:
hereafter,
we say that $g' \in W$ is $\dot\varphi$-\emph{regular} if $\bm g \in \GL_{\mb C}(\mf t)$ admits a regular eigenvector,
i.e.,
if $\bm g$ is a \emph{regular element} of the reflection coset $\wt{\mb G}$.

Suppose then first that $\wh Q = Aw^{-1}$ is homogeneous,
with $A \in \mf t_{\reg}$.
Then we can make/prove that basic statement relating `doubly-twisted' meromorphic gauge theory on curves with \emph{nonsplit} Lehrer--Springer theory:

\begin{enonce}{Proposition-Definition (cf.~Prop.-Def.~\ref{prop:generic_twisted_deformation_space_1_coeff})}
	\label{prop:generic_doubly_twisted_deformation_space_1_coeff}

	Choose $g' \in \mb L_{\dot\varphi,r} \bigl( \wh Q \bigr)$,
	and extend the notation of~\eqref{eq:regular_eigenspace} via
	\begin{equation}
		\label{eq:twisted_regular_eigenspace}
		\mf t(\bm g,\zeta_r)
		\ceqq \ker (\bm g - \zeta_r \cdot \Id_{\mf t}) \sse \mf t,
	\end{equation}
	as in~\cite[\S~6]{springer_1974_regular_elements_of_finite_reflection_groups}.
	Then:
	\begin{enumerate}
		\item
		      one has
		      \begin{equation}
			      \label{eq:generic_doubly_twisted_deformation_space_1_coeff}
			      \bm B_{\bm g,r} \bigl( \wh Q \bigr)
			      = \mf t(\bm g,\zeta_r) \bigsm \bigcup_{\Phi} H_\alpha(\bm g,\zeta_r),
			      \qquad H_\alpha(\bm g,\zeta_r) \ceqq H_\alpha \cap \mf t(\bm g,\zeta_r) \sse \mf t(\bm g,\zeta_r);
		      \end{equation}

		\item
		      and~\eqref{eq:generic_doubly_twisted_deformation_space_1_coeff} is the (nonempty) complement of a \emph{complex-reflection arrangement}.
	\end{enumerate}
\end{enonce}

\begin{proof}
	Here $A$ is a regular eigenvector of $\bm g \in \GL_{\mb C}(\mf t)$,
    and by construction one has $\bm g \in N_{\GL_{\mb C}(\mf t)}(W)$.
    Then~\cite[Thmm.~6.2~(i) +~6.4~(ii)]{springer_1974_regular_elements_of_finite_reflection_groups} yield
	\begin{equation}
		\dim_{\mb C} \bigl( \mf t(\bm g,\zeta_r) \bigr) \geq \dim_{\mb C} \bigl( \mf t(\bm g g'',\zeta_r) \bigr),
		\qquad g'' \in W.
	\end{equation}
	Now the conclusion follows from the \emph{nonsplit} version of~\cite[Thm.~1.1]{lehrer_springer_1999_reflection_subquotients_of_unitary_reflection_groups} (the complex reflection group in question is the `centralizer' of $\bm g$ in $W$,
	cf.~\eqref{eq:doubly_twisted_centralizer}).
\end{proof}

\begin{rema}
	\label{rmk:twisted_eigenspace_intersection}

	As in Rmk.~\ref{rmk:eigenspace_intersection},
	the apparent dependence on the choice of a primitive $r$-th root of $1$ also vanishes at the level of eigenspaces---%
	provided that one uses the `twisted' ramification~\eqref{eq:twisted_ramification}.

	The point is that one has
	\begin{equation}
		\label{eq:twisted_eigenspace_intersection}
		\mf t(\bm g,\zeta_r)
		= \bigcap_{(\mb Z \slash \wt r \mb Z)^{\ts}} \mf t ( \bm g^i,\zeta^i_r )
		= \bigcap_{\mb Z \slash \wt r \mb Z} \mf t ( \bm g^i,\zeta^i_r ) \sse \mf t,
	\end{equation}
	where the notation  makes sense since $r \mid \wt r$ and $\wt{\bm g}^r = 1$.
	Indeed,
	the equality $\dot\varphi^{\wt r} = 1$ yields $\bm g^{\wt r} \in W$,
	and moreover $\bm g^{\wt r}(A) = \zeta^{\wt r}_r A = A$;
	now use Lem.~\ref{lem:about_pointwise_stabilizers}.

	So if $i$ is coprime with $\wt r$ one may relabel by $\zeta' \ceqq \zeta^i_r \in \mu_r$ and $\bm g' \ceqq \bm g^i$.
	The caveat is that this implies $i \wdg r = 1$ and $i \wdg \dot r = 1$,
	so that:
	(i) the root $\zeta'$ still generates $\mu_r$;
	and (ii) one has $\braket{\dot\varphi^i} = \braket{\dot\varphi}$.
	(Coprimality with $r$ does \emph{not} suffice.)

	A more general remark,
	which we will not state,
	holds in the nonquasigeneric monolevel case (cf.~Rmk.~\ref{rmk:eigenspace_intersection_on_flat}.)
\end{rema}

\subsection{Pure quasigeneric case:
	higher pole order}

Suppose instead that $\wh Q = \sum_{i = 1}^s A_iw^{-i}$,
with $A_s \in \mf t_{\reg}$.
Then,
extending Prop.~\ref{prop:generic_doubly_twisted_deformation_space_1_coeff}:

\begin{prop}
	\label{prop:pure_generic_doubly_twisted_def_space}

	There is a factorization $\bm B_{\bm g,r} \bigl( \wh Q \bigr) = V' \ts U$,
	where:
	\begin{enumerate}
		\item
		      the first factor $V' \sse \mf t^{s-1}$ is a vector subspace;

		\item
		      and $U \sse \mf t(\bm g,\zeta_r^s)$ is a hyperplane complement analogous to~\eqref{eq:generic_doubly_twisted_deformation_space_1_coeff}.
	\end{enumerate}
\end{prop}

\begin{proof}
	The proof of Prop.~\ref{prop:generic_pure_twisted_deformation_space} extends.
\end{proof}

\begin{rema}
    \label{rmk:about_regular_numbers_twisted}

	Analogously to Rmk.~\ref{rmk:about_regular_numbers},
	there are constraints on the order of the group elements involved.
	Namely,
	if we let again $d \ceqq r \slash (r \wdg s) \geq 1$,
	Prop.~\ref{prop:pure_generic_doubly_twisted_def_space} states in particular that $\bm g$ admits a regular eigenvector whose eigenvalue is a primitive $d$-th root of $1$.
	Therefore,
	if $\dot\varphi^d = 1$ then the whole of $\bm g$ has order $d$ in $\GL_{\mb C}(\mf t)$ (cf.~\cite[Thm.~6.4.(i)]{springer_1974_regular_elements_of_finite_reflection_groups}).
	E.g.,
	if $W = W_\bullet (m)$ is a Weyl group of type $\bullet \in \set{A,B\!C,D,E,F,G}$ and rank $m \geq 1$,
	recall that looking at Dynkin-diagram automorphisms shows that $\dot\varphi = 1$ unless:
	\begin{enumerate}
		\item[(A)]
		      one has $\bullet = A$ and $m \geq 3$,
		      whence however $\dot\varphi^2 = 1$;
              (Recall that in our notation $W_A(m) \simeq \mf S_m$ is the Weyl group of the root system $\Phi_A(m-1)$,
              of type $A_{m-1}$.)

		\item[(D)]
		      or one has $\bullet = D$,
		      whence however:
		      \begin{enumerate}
			      \item
			            if $m \geq 5$,
			            one has $\dot\varphi^2 = 1$;

			      \item
			            else $m = 4$,
			            and one has $\dot\varphi^k = 1$ for some $k \in \set{2,3}$;
		      \end{enumerate}

		\item[(E)]
		      or one has $\bullet = E$ and $m = 6$,
		      whence however $\dot\varphi^2 = 1$. \qedhere
	\end{enumerate}
\end{rema}

\subsection{Pure general case:
	minimal pole order}

Choose now $\wh Q = A w^{-1}$ homogeneous,
but with $A \in \mf t$ arbitrary;
let $\phi = \phi_A \sse \Phi$ be the Levi annihilator of $A$ (cf.~\eqref{eq:kernel_intersection}).
This monolevel case is treated by the following:

\begin{enonce}{Proposition-Definition}
	\label{prop:pure_doubly_twisted_def_space_general_1_coeff}

	Choose $g' \in \mb L_{\dot\varphi,r} \bigl( \wh Q \bigr)$.
	Then:
	\begin{enumerate}
		\item
		      the subspace~\eqref{eq:kernel_intersection} is $\bm g$-invariant,
		      i.e.,
		      one has $\bm g \in N_{\GL_{\mb C}(\mf t)}(\mf t_\phi)$;

		\item
		      one has
		      \begin{equation}
			      \label{eq:doubly_twisted_general_pure_def_space_1_coeff}
			      \bm B_{\bm g,r} \bigl( \wh Q \bigr)
			      = \mf t_\phi(\bm g_\phi,\zeta_r) \, \bigsm \, \bigcup_{\Phi \sm \phi} H_\alpha(\bm g_\phi,\zeta_r),
			      \qquad \bm g_\phi \ceqq \eval[1]{\bm g}_{\mf t_\phi},
		      \end{equation}
		      in the notation of~\eqref{eq:twisted_regular_eigenspace},
		      setting
		      \begin{equation}
			      \label{eq:doubly_twisted_hyperplane_restriction}
			      H_\alpha(\bm g_\phi,\zeta_r)
			      \ceqq H_\alpha \cap \mf t_\phi(\bm g_\phi,\zeta_r)
			      = H_\alpha(\phi) \cap \mf t(\bm g,\zeta_r) \sse \mf t_\phi(\bm g_\phi,\zeta_r);
		      \end{equation}

		\item
		      and~\eqref{eq:doubly_twisted_general_pure_def_space_1_coeff} is a \emph{nonempty} complex hyperplane complement.
	\end{enumerate}
\end{enonce}

\begin{proof}
	The proof of Prop.-Def.~\ref{prop:pure_general_twisted_deformation_space_1_coeff} extends,
	after establishing Lem.~\ref{lem:doubly_twisted_setwise_stabilizer},
	and noting that $\mf t(\bm g,\zeta_r) \cap \mf t_\phi = \mf t_\phi(\bm g_\phi,\zeta_r)$.
	In turn,
	the latter follows from: (i) the $\mb C$-linear splitting
	\begin{equation}
		\mf t(\bm g,\zeta_r)
		= \mf Z(f,\zeta_r) \ops \mf t'(\bm g',\zeta_r) \sse \mf Z \ops \mf t';
	\end{equation}
	(ii) the fact that $\bm g' \in \Aut(\Phi')$ acts in semisimple fashion on $\mf t'$;
	and (iii) the inclusion $\mf Z \sse \mf t_\phi$,
	whence $\bm g_\phi = (f,\bm g'_\phi) \in \GL_{\mb C}(\mf t_\phi)$.
\end{proof}

\begin{lemm}[cf.~Lem.~\ref{lem:about_setwise_stabilizers}]
	\label{lem:doubly_twisted_setwise_stabilizer}

	For any $g' \in W$:
	\begin{enumerate}
		\item
		      the following are equivalent:
		      \begin{enumerate}
			      \item
			            one has $\bm g' \in N_{\GL_{\mb C}(\mf t)}(\mf t_\phi)$;

			      \item
			            one has $\bm g'(\phi) = \phi$,
			            in the action of root-system automorphisms on $\mf t^{\dual}$;
			            (Cf.~\S~\ref{sec:root_automorphisms}.)

			      \item
			            one has $\bm g'(\phi) = \phi$ and $\bm g'(\Phi \sm \phi) = \Phi \sm \phi$;

			      \item
			            one has $\bm g' \bigl( \bm B\bigl( \wh Q \bigr) \bigr) \sse \bm B \bigl( \wh Q \bigr)$;

			      \item
			            one has $\bm g' \bigl( \bm B (\wh Q) \bigr) \sse \mf t_\phi$;

			      \item
			            there exists $A' \in \bm B \bigl( \wh Q \bigr)$ such that $\bm g'(A') \in \mf t_\phi$;

			      \item
			            and there exists $A' \in \bm B \bigl( \wh Q \bigr)$ such that $\bm g'(A') \in \bm B \bigl( \wh Q \bigr)$;
		      \end{enumerate}
		\item
		      and each of the above conditions holds if and only if it holds for $\bm g = (f,\bm g')$.
	\end{enumerate}
\end{lemm}

\begin{proof}[Proof postponed to~\ref{proof:lem_doubly_twisted_setwise_stablizer}]
\end{proof}

\begin{rema}
	So Lem.~\ref{lem:doubly_twisted_setwise_stabilizer} generalizes Lem.~\ref{lem:about_setwise_stabilizers}.

	On the contrary,
	Lem.~\ref{lem:about_pointwise_stabilizers} does \emph{not} extend:
	there are in general elements of $\Aut(\Phi')$ which do not act as the identity on $\mf t_\phi$,
	even though they fix a vector in general position therein.
	The previous results still generalize because:
	(i) we work within a single reflection coset of $(\mf t,W)$,
	rather than changing the twist at each step of the fission (cf.~the proof~\ref{proof:lem_no_doubly_twisted_marking});
	and (ii) the irregular classes are $W$-\emph{orbits} of irregular types---%
	rather than $\Aut(\Phi)$-orbits.
\end{rema}

\begin{rema}
	\label{rmk:about_twisted_normalizers}

	When $\dot\varphi \neq 1$,
	the restriction $\bm g_\phi \in \GL_{\mb C}(\mf t_\phi)$ of~\eqref{eq:doubly_twisted_general_pure_def_space_1_coeff} need \emph{not} define a class inside the relative Weyl group $W(\phi) = N_W(\mf t_\phi) \bs W_{\mf t_\phi} \hra \GL_{\mb C}(\mf t_\phi)$ (cf.~\eqref{eq:effective_weyl_group} +~Lem.-Def.~\ref{lem:relative_weyl_group}).
	Rather,
	the element $\bm g \in \dot\varphi W$ normalizes the setwise/pointwise stabilizers of the flat $\mf t_\phi$,
	so that $\bm g_\phi$ normalizes $W(\phi) \sse \GL_{\mb C}(\mf t_\phi)$,
	and there is a well-defined coset $\bm g_\phi W(\phi)$:
	the latter coincides with the right coset,
	but it need \emph{not} be a reflection coset when $\phi \neq \vn$ (cf.~\S~\ref{sec:howlett} +~Lem.~\ref{lem:double_twist_normalizes}).

	Somewhat equivalently,
	note that the hypotheses imply that
	\begin{equation}
		\label{eq:twisted_setwise_stabilizer}
		N_{\dot\varphi W}(\mf t_\phi)
		\ceqq \dot\varphi W \cap N_{\braket{\dot\varphi,W}}(\mf t_\phi)
		= \dot\varphi W \cap N_{\GL_{\mb C}(\mf t)}(\mf t_\phi) \neq \vn,
	\end{equation}
	as this set contains $\bm g$.
	Then the `twisted' setwise-stabilizer of~\eqref{eq:twisted_setwise_stabilizer} coincides with the coset
	\begin{equation}
		\bm g \cdot N_W(\mf t_\phi) \sse N_{\braket{\dot\varphi,W}}(\mf t_\phi),
	\end{equation}
	on which the parabolic subgroup $W_{\mf t_\phi} \sse N_W(\mf t_\phi)$ acts---%
	by right multiplication.
	Now the restriction map $\bm g'' \mt \bm g''_\phi \cl N_{\dot\varphi W}(\mf t_\phi) \to \GL_{\mb C}(\mf t_\phi)$ induces a set-theoretic \emph{injection}
	\begin{equation}
		\label{eq:relative_weyl_coset}
		\dot\varphi W(\phi)
		\ceqq N_{\dot\varphi W}(\mf t_\phi) \bs W_{\mf t_\phi} \lhra \GL_{\mb C}(\mf t_\phi),\fn{
			Injectivity need \emph{not} hold without restricting to $\dot\varphi W \sse \braket{\dot\varphi,W}$,
			and beware that $\dot\varphi$ need \emph{not} normalize/stabilize any of the subgroups/subspaces involved
			(so the notation is unambiguous).}
	\end{equation}
	whose source is naturally acted on by the quotient group $W(\phi)$.
	Finally,
	the image of this $W(\phi)$-equivariant map coincides with the `quotient' coset $\bm g_\phi W(\phi) \sse \GL_{\mb C}(\mf t_\phi)$.
\end{rema}

\begin{defi}
	\label{def:relative_weyl_coset}

	The $W(\phi)$-torsor~\eqref{eq:relative_weyl_coset} is the $\dot\varphi$-\emph{twisted relative Weyl coset of} $(\mf g,\mf t,\phi)$.
	(This is \emph{nonempty} if~\eqref{eq:twisted_setwise_stabilizer} holds,
	as we assume here;
	in general,
	this yields a \emph{pseudotorsor}~\cite{nlab_pseudo_torsor}.)
\end{defi}

\subsubsection{}
\label{sec:relative_weyl_coset}

Analogously to Lem.-Def.~\ref{lem:relative_weyl_group},
one may `lift' this construction at the group-theoretic level by introducing the usual connected reductive Adjoint stabilizer $L \ceqq G^A \sse G$ of any vector $A \in \mf t$ which lies in general position in $\mf t_\phi$,
so that $W_{\mf t_\phi} \sse W$ is the Weyl group of $(L,T)$.
This coincides with the centralizer of the subtorus $T_\phi \sse T$ integrating the flat $\mf t_\phi \sse \mf t$,
and then one must twist $G \to \Inn(G) \sse \Aut(G)$ by the chosen group automorphism $\bm\varphi \in \Aut(G)$ preserving $(L,T)$ (cf.~Lem.~\ref{lem:full_twist_preserves_centralizer});
etc.

In any event,
hereafter we tacitly identify the restriction of elements $\bm g'' \in N_{\dot\varphi W}(\mf t_\phi)$,
onto $\mf t_\phi \sse \mf t$,
with their classes $\bm g''_\phi \in \dot\varphi W(\phi) \hra \GL_{\mb C}(\mf t_\phi)$---%
modulo $W_{\mf t_\phi}$.
(And analogously when working recursively,
in the multilevel case.)

\subsection{Pure general case:
	higher pole order}

Let finally $\wh Q = \sum_{i = 1}^s A_i w^{-i}$ be arbitrary---%
possibly multilevel.
Iterating the previous arguments along the Levi-subsystem filtration $\bm \phi = (\phi_1 \sse \dm \sse \phi_s \sse \phi_{s+1} = \Phi)$ determined by $\wh Q$ yields a proof of the following:

\begin{enonce}{Theorem-Definition (cf.~Thm.-Def.~\ref{thm:general_pure_twisted_deformation_space})}
	\label{thm:general_pure_doubly_twisted_deformation_space}

	Choose $g' \in \mb L_{\dot\varphi,r} \bigl( \wh Q \bigr)$.
	Then:
	\begin{enumerate}
		\item
		      the subspace $\mf t_{\phi_i} \sse \mf t$ is $\bm g$-invariant,
		      for $i \in \set{1,\dc,s}$;

		\item
		      and moreover,
		      setting
		      \begin{equation}
			      \label{eq:doubly_twisted_pure_def_space_factor}
			      \bm B_{\bm g,r} \bigl( \wh Q,i \bigr)
			      \ceqq \mf t_{\phi_i} (\bm g_{\phi_i},\zeta^i_r) \, \bigsm \, \bigcup_{\phi_{i+1} \sm \phi_i} H_\alpha(\bm g_{\phi_i},\zeta_r^i),
		      \end{equation}
		      in the notation of~\eqref{eq:doubly_twisted_general_pure_def_space_1_coeff}--\eqref{eq:doubly_twisted_hyperplane_restriction},
		      it follows that:
		      \begin{enumerate}
			      \item
			            there is a (topological) factorization
			            \begin{equation}
				            \label{eq:factorization_multilevel_pure_space_twisted}
				            \bm B_{\bm g,r} \bigl( \wh Q \bigr)
				            = \prod_{i = 1}^s \bm B_{\bm g,r} \bigl( \wh Q,i \bigr) \sse \mf t^s;
			            \end{equation}

			      \item
			            and the spaces~\eqref{eq:doubly_twisted_pure_def_space_factor} are \emph{nonempty} complex hyperplane complements.
		      \end{enumerate}
	\end{enumerate}
\end{enonce}

\begin{proof}
	The first statement is Lem.~\ref{lem:twisted_flag_stabilizer}.
	Then combine the proofs of Prop.~\ref{prop:pure_generic_doubly_twisted_def_space} +~Prop.-Def.~\ref{prop:pure_doubly_twisted_def_space_general_1_coeff}.
\end{proof}

\begin{lemm}
	\label{lem:twisted_flag_stabilizer}

	One has
	\begin{equation}
		\label{eq:twisted_flag_stabilizer}
		\dot\varphi \cdot \mb L_{\dot\varphi,r} \bigl( \wh Q \bigr) \sse N_{\dot\varphi W} (\mf t_{\bm\phi})
		\ceqq \dot\varphi W \cap P_{\bm\phi} \sse \GL_{\mb C}(\mf t),
	\end{equation}
	in the notation of~\eqref{eq:flag_parabolic},
	and extending the notation of~\eqref{eq:flag_parabolic_in_weyl}.
\end{lemm}

\begin{proof}
	The proof of Lem.-Def.~\ref{lem:flag_stabilizer} extends,
	by iterating Lem.~\ref{lem:doubly_twisted_setwise_stabilizer},
	using the hypothesis that $\bm g (A_i) = \zeta_r^i A_i \in \bm B \bigl( \wh Q,i \bigr)$,
	for $i \in \set{1,\dc,s}$.

	(Note that $N_{\dot\varphi W}(\mf t_{\bm\phi}) \sse N_{\braket{\dot\varphi,W}}(\mf t_{\bm \phi})$ is a coset through the flag-preserving subgroup $N_W(\mf t_{\bm\phi}) \sse W$,
	cf.~Rmk.~\ref{rmk:about_twisted_normalizers} in the monolevel case.)
\end{proof}

\subsection{Reduction to the (\texorpdfstring{$\dot\varphi$}{phi}-)simple pure case}
\label{sec:reduction_simple_case_pure_doubly_twisted}

In this extended setting it is still possible to reduce the study of the pure deformation spaces to the case where $\mf g$ is \emph{simple},
despite the fact that the outermorphisms make the---%
isomorphic---%
simple factors interact with each other.

This is due to Prop.~\ref{prop:killing_multiplicities},
after passing through a `twisted' notion of simplicity (cf.~Lem.~\ref{lem:reduction_to_simple_case_twisted}).

\subsubsection{}

Namely,
in the notation of Lem.~\ref{lem:reduction_to_simple_case},
up to refining the factorization $\mf g = \prod_i \mf I_i$ one can suppose that each (nonzero) Lie-ideal $\mf I_i \sse \mf g$ is \emph{minimal},
i.e.,
that it contains no proper nonzero Lie-ideal.
If $I$ is the (finite) index set of the new factorization,
this means that for all $i \in I$ the ideal $\mf I_i$ is either:
(i) a \emph{simple} Lie-ideal of the semisimple part $\mf g'$;
or (ii) a \emph{line} of the Lie-centre $\mf Z$.
(The former are uniquely determined,
while the latter are \emph{not} if $\dim_{\mb C} ( \mf Z ) \geq 2$.)

Now split as usual $\dot\varphi = (f,\dot\varphi') \in \GL_{\mb C} ( \mf Z ) \ts \Out(\Phi')$.
Decompose also $I = I^{\mf Z} \cup I'$,
so that $\mf I_i \sse \mf Z$ for $i \in I^{\mf Z}$ (resp.,
it lies in $\mf g'$ for $i \in I'$)---%
and one has $I^{\mf Z} \cap I' = \vn$.
The outermorphism $\dot\varphi'$ permutes isomorphic simple factors,
and so for all $i \in I'$ there exists a suitable index $j \ceqq \dot\varphi'(i) \in I'$ such that
\begin{equation}
	\mf I_i \simeq \dot\varphi'(\mf I_i)
	= \mf I_j \sse \mf g'.
\end{equation}
Then we further partition $I'$ into $\braket{ \dot\varphi' }$-orbits,
and denote by $i \mt \ol i \cl I' \thra \ol I'$ the quotient map onto the quotient set of indices.
(In the notation of \S~\ref{sec:lehrer_springer_theory},
in a different context.)

It follows that an ideal of $\mf g$ is $\dot\varphi$-\emph{stable} if and only if it is a product of:
(a) simple/irreducible $f$-submodules of $\mf Z$,
i.e.,
$f$-\emph{eigenlines} (since $f$ has finite order);
and/or (b) \emph{isotypic} semisimple Lie-subalgebras of $\mf g'$ of the form
\begin{equation}
	\wt{\mf I}_i
	\ceqq \prod_{j \in \ol i} \mf I_j
	\simeq \mf I_i^{k_i},
	\qquad k_i \ceqq \lvert \ol i \rvert \geq 1.
\end{equation}

Overall,
if we rather let $\mf I_i \sse \mf Z$ be an $f$-eigenline for all $i \in I^{\mf Z}$,
then there is a factorization of $\mf g$ into \emph{minimal} $\dot\varphi$-\emph{stable ideals},
i.e.,
nonzero $\dot\varphi$-stable ideals which contain no proper nonzero $\dot\varphi$-stable ideal:
\begin{equation}
	\label{eq:minimal_phi_stable_factorization}
	\mf g
	= \prod_{I^{\mf Z}} \bigl( \mf I_i \bigr) \ts \prod_{\ol I'} \wt{\mf I}_{\ol i},
	\qquad \wt{\mf I}_{\ol i}
	\ceqq \wt{\mf I}_i \sse \mf g'.
\end{equation}

Now one can state the first:

\begin{lemm}[Cf.~Lem.~\ref{lem:reduction_to_simple_case}]
	\label{lem:reduction_to_simple_case_twisted}

	Let $\mf g = \prod_j \mf J_j$ denote the factorization of~\eqref{eq:minimal_phi_stable_factorization},
	for $j \in J \ceqq I^{\mf Z} \cup \ol I'$,
	where
	\begin{equation}
		\mf J_j
		\ceqq
		\begin{cases}
			\mf I_i,            & \quad j = i \in I^{\mf Z},  \\
			\wt{\mf I}_{\ol i}, & \quad j = \ol i \in \ol I'.
		\end{cases}
	\end{equation}
	Moreover,
	split accordingly
	\begin{equation}
		\mf t
		= \prod_j \mf t_j,
		\quad \Phi
		= \bigcup_j \Phi_j,
		\quad \wh Q
		= \sum_j \wh Q_j,
		\qquad \mf t_j
		\ceqq \mf t \cap \mf J_j,
		\quad \Phi_j
		\ceqq \Phi \cap \mf t_j^{\dual}.
	\end{equation}
	Then there is a (topological) factorization
	\begin{equation}
		\label{eq:reduction_to_minimal_stable_pure}
		\bm B_{\bm g,r} \bigl( \wh Q \bigr)
		= \prod_j \bm B_{\bm g_j,r} \bigl( \wh Q_j \bigr),
		\qquad \bm g_j
		\ceqq \eval[1]{\bm g}_{\mf t_j} \in \GL_{\mb C}(\mf t_j),
	\end{equation}
	in the canonical identification $\mf t^s \simeq \prod_j \mf t_j^s$.
\end{lemm}

\begin{proof}
	Choose $j \in J$.
	One has $\dot\varphi(\mf t_j) = \mf t_j$ by $\dot\varphi$-stability,
	and $g' \in W$ also preserves $\mf t_j$---%
	acting trivially when $j \in I^{\mf Z}$---%
	in the usual factorization
	\begin{equation}
		\label{eq:weyl_group_varphi_invariant_split}
		W
		= \prod_{\ol I'} W_j,
		\qquad W_j
		\ceqq W(\mf J_j,\mf t_j).
	\end{equation}
	It follows that $\bm g(\mf t_j) = \mf t_j$,
	and so
	\begin{equation}
		\mf t(\bm g,\zeta)
		= \prod_j \mf t_j(\bm g_j,\zeta),
		\qquad \zeta \in \mb C^{\ts}.
	\end{equation}
	Now the proof of Lem.~\ref{lem:reduction_to_simple_case} extends,
	using the factorization~\eqref{eq:factorization_multilevel_pure_space_twisted}.
\end{proof}

\subsubsection{}
\label{sec:from_varphi_simple_to_simple}

So the pieces $\bm B_{\bm g_j,r} \bigl( \wh Q_j \bigr)$ of~\eqref{eq:reduction_to_minimal_stable_pure} are the elementary building blocks for the pure $(\dot\varphi,r)$-admissible deformation spaces.

The `central' ones ($j \in I^{\mf Z}$) remain contractible,
as there are---%
as usual---%
no roots to evaluate,
and do not contribute to the local WMCGs.
The `semisimple' ones ($j \in \ol I'$) instead are the $\dot\varphi$-twisted generalizations of the simple ideals of $\mf g'$:
define a nonzero $\dot\varphi$-stable Lie-ideal $\mf J \sse \mf g$ to be $\dot\varphi$-\emph{simple} if it is nonabelian,
and if it contains no proper nonzero $\dot\varphi$-stable Lie-ideal (whence $\dot\varphi = 1$ yields an ordinary simple ideal).
These are then precisely the isotypic pieces $\wt{\mf I}_{\ol i} \sse \mf g'$.

\subsubsection{}

At this stage,
we have reduced the general topological problem to the $\dot\varphi$-simple case.

Precisely,
suppose that $\mf g = \mf g'$ is \emph{simple} (whence $\Phi = \Phi'$,
etc.).
Moreover,
choose an integer multiplicity $k \geq 1$,
and consider the isotypic semisimple Lie algebra $\wt{\mf g} \ceqq \mf g^k$.
Equip the latter with the `power' Cartan subalgebra $\wt{\mf t} \ceqq \mf t^k$,
so that there is a group isomorphism
\begin{equation}
	\Out\bigl(\wt \Phi\bigr)
	\simeq \Out(\Phi) \wr \mf S_k \lhra \Aut\bigl(\wt \Phi\bigr),
	\qquad \wt \Phi
	\ceqq \Phi \bigl( \wt{\mf g},\wt{\mf t} \, \bigr)
\end{equation}
where the symmetric group acts by permuting the simple factors---%
preserving $\wt{\mf t} \sse \wt{\mf g}$.

Now label each factor by a number $l \in \set{1,\dc,k}$,
so that $\wt{\mf g} = \prod_{l = 1}^k \mf g_l$ and $\wt{\mf t} = \prod_{l = 1}^k \mf t_l$.
Then consider an outermorphism
\begin{equation}
	\wt{\dot\varphi}
	= \rho \cdot \dot{\bm \varphi},
	\qquad \rho \in \mf S_k,
	\quad \dot{\bm \varphi}
	= (\dot\varphi_1,\dc,\dot\varphi_k) \in \Out(\Phi)^k \simeq \prod_{l = 1}^k \Out(\Phi_l),
\end{equation}
which is necessarily of finite order here---%
with $\Phi_l \ceqq \Phi(\mf g_l,\mf t_l) = \Phi$.
The fact that $\wt{\mf g}$ is $\wt{\dot\varphi}$-\emph{simple} is equivalent to asking that $\rho$ has a single orbit in the set $\set{1,\dc,k}$,
i.e.,
up to conjugating by $\mf S_k \lts 1 \sse \Out \bigl( \wt\Phi \bigr)$,
that it cyclically permutes $\mf g_l \mt \mf g_{l+1}$---%
modulo $k$.
(This is w.l.o.g.,
as far as the homeomorphism class of admissible-deformation spaces is concerned,
cf.~\S~\ref{sec:aut_phi_conjugacy}.)

Finally,
let $\wt g = (g_1,\dc,g_k) \in \wt W$ be the inner part of a root-system automorphism $\wt{\bm g} = \wt{\dot\varphi} \wt g \in \Aut \bigl( \wt \Phi \bigr)$,
lying in the Weyl group
\begin{equation}
	\wt W \ceqq W \bigl( \wt{\mf g},\wt{\mf t} \, \bigr) \simeq W^k \simeq \prod_{l = 1}^k W(\mf g_l,\mf t_l).
\end{equation}
If $A  = \bigl(A^{(1)},\dc,A^{(k)}\bigr) \in \wt{\mf t}$ ($A^{(l)} \in \mf t_l$),
and writing $\ul A = \bigl( \ul A^{(1)},\dc,\ul A^{(k)} \bigr) \ceqq \wt{\bm g}(A)$,
one now has
\begin{equation}
	\label{eq:orbit_shift}
	\ul A^{(l+1)}
	= \bm g_l \bigl( A^{(l)} \bigr),
	\qquad \bm g_l
	\ceqq \dot\varphi_l g_l \in \Aut(\Phi),
	\quad l \in \mb Z \bs k \mb Z.
\end{equation}
Then the multiplicity can also be `removed' by pushing everything down to the first labelled factor,
thereby completing the reduction to the simple/irreducible case;
but only up to modifying the twist/ramification:

\begin{prop}
	\label{prop:killing_multiplicities}

	Let $\wh Q = \sum_{i = 1}^s A_i w^{-i}$ be an irregular type for $\wt{\mf g}$,
	such that $\wt g \in \mb L_{\wt{\dot\varphi},r}\bigl( \wh Q \bigr)$.
	Moreover,
	consider the integer $r' \ceqq r \slash (r \wdg k) \geq 1$ and the \emph{primitive} $r'$-th root $\wt \zeta \ceqq \zeta_r^k$---%
	of $1$.
	If $A \mt A^{(1)} \cl \wt{\mf t} \thra \mf t_1$ denotes the canonical projection,
	then:
	\begin{enumerate}
		\item
		      the irregular type $\wh Q^{(1)} \ceqq \sum_i A^{(1)}_i w^{-i}$---%
		      for $\mf g = \mf g_1$---%
		      has $r'$-Galois-orbit generated by $\bm h \ceqq \bm g_k \dm \bm g_1 \in \Aut ( \Phi )$,
		      and it has the same pole order $s \geq 1$ as $\wh Q$;

		\item
		      and the canonical projection restricts to a \emph{homeomorphism}
		      \begin{equation}
			      \label{eq:orbit_descent}
			      \pi^{(1)} \cl \bm B_{\wt{\bm g},r} \bigl( \wh Q \bigr)
			      \lxra{\simeq} \bm B_{\bm h,r'} \bigl( \wh Q^{(1)} \bigr),
		      \end{equation}
		      with the right-hand side defined relative to the generator $\wt \zeta \in \mu_{r'}$.
		      (One has $\wt \zeta = \zeta_{r'} \in \mb C^{\ts}$ if and only if $k \equiv (r \wdg k) \pmod r$.)
	\end{enumerate}
\end{prop}

\begin{proof}[Proof postponed to~\ref{proof:prop_killing_multiplicities}]
\end{proof}

\begin{rema}
	\label{rmk:product_outermorphisms}

	The proof~\ref{proof:prop_killing_multiplicities} actually shows that the projection on the first cyclic piece yields a factorwise $\mb C$-\emph{linear isomorphism} of the ambient eigenspaces.

	Moreover,
	note that one has by construction
	\begin{equation}
		\label{eq:product_outermorphisms}
		\bm h
		= \dot\varphi \cdot h,
		\qquad \dot\varphi \ceqq \dot\varphi_k \dm \dot\varphi_1 \in \Out(\Phi),
	\end{equation}
	for a suitable $h \in W$ (cf.~\eqref{eq:orbit_shift},
	making the outermorphisms commute to the left).
	Therefore,
	the irregular type $\wh Q^{(1)}$ is $(\dot\varphi,r')$-Galois-closed---%
	with $(\dot\varphi,r')$-Galois-orbit generated by $h$.
\end{rema}

\subsection{Forgetting the marking}
\label{sec:no_doubly_twisted_marking}

As in \S~\ref{sec:no_marking},
one can get rid of the choice of the `inner' part of the twisting automorphism:

\begin{prop}
	Choose $g',g'' \in \mb L_{\dot\varphi,r} \bigl( \wh Q \bigr)$,
	for an \emph{arbitrary} irregular type $\wh Q$.
	Then one has $\bm B_{\dot\varphi g',r} \bigl( \wh Q \bigr) = \bm B_{\dot\varphi g'',r} \bigl( \wh Q \bigr)$.
\end{prop}

\begin{proof}
	The proof of Prop.~\ref{prop:no_marking_general} extends,
	after establishing Lem.~\ref{lem:no_doubly_twisted_marking}.
\end{proof}

\begin{lemm}
	\label{lem:no_doubly_twisted_marking}

	Choose $g',g'' \in W$ such that $A \in \mf t(\dot\varphi g',\zeta) \cap \mf t(\dot\varphi g'',\zeta)$,
	for some $\zeta \in \mb C^{\ts}$.
	Then $(\dot\varphi g')_\phi = (\dot\varphi g'')_\phi \in \GL_{\mb C}(\mf t_\phi)$ (cf.~Prop.-Def.~\ref{prop:pure_doubly_twisted_def_space_general_1_coeff}),
	where $\phi = \phi_A \sse \Phi$ is the Levi annihilator of $A$.
\end{lemm}

\begin{proof}[Proof postponed to~\ref{proof:lem_no_doubly_twisted_marking}]
\end{proof}

\begin{rema}
	\label{rmk:doubly_twisted_independence_restricted_eigenspaces}

	If in particular $A$ is regular,
	then Lem.~\ref{lem:no_doubly_twisted_marking} just states that $\dot\varphi g'= \dot\varphi g'' \in \GL_{\mb C}(\mf t)$,
	and so there is precisely \emph{one} $\dot\varphi$-regular element generating the $(\dot\varphi,r)$-Galois-orbit.
	More generally,
	this holds whenever the centralizer of $\wh Q$ is the maximal torus $T \sse G$,
	as the subset $\mb L_{\dot\varphi,r} \bigl( \wh Q \bigr) \sse W$ is a right $W_{\mf t_{\phi_1}}$-torsor (generalizing~Rmk.~\ref{rmk:toral_fission},
	cf.~the proof of Cor.-Def.~\ref{cor:levi_subcoset_twisted}~(3.)).
\end{rema}

\subsubsection{}

This establishes the equality $\bm B_{\dot\varphi,r} \bigl( \wh Q \bigr) = \bm B_{\bm g,r} \bigl( \wh Q \bigr)$,
for any choice of generator $g' \in \mb L_{\dot\varphi,r} \bigl( \wh Q \bigr)$ of the $(\dot\varphi,r)$-Galois-orbit.

\subsection{Full/nonpure quasigeneric case:
	minimal pole order}

We now describe the deformations of a $(\dot\varphi,r)$-Galois-closed irregular class $\wh\Theta = \wh\Theta \bigl( \wh Q \bigr)$,
when $\wh Q = Aw^{-1}$ is homogeneous with $A \in \mf t_{\reg}$.
Again in general the point is to determine the setwise-/pointwise-stabilizer of the pure deformation space in $W$,
and conclude that the action of the corresponding subquotient of $W$ is \emph{free}---%
so that it yields a Galois covering.

\begin{enonce}{Proposition-Definition}
	\label{prop:full_generic_doubly_twisted_deformation_space_1_coeff}

	Denote by $g' = g'_A \in W$ the \emph{unique} element of $\mb L_{\dot\varphi,r} \bigl( \wh Q \bigr)$ (cf.~Rmk.~\ref{rmk:doubly_twisted_independence_restricted_eigenspaces}).
	Then:
	\begin{enumerate}
		\item
		      the `centralizer' subgroup of $\bm g$ in $W$,
		      i.e.,
		      \begin{equation}
			      \label{eq:doubly_twisted_centralizer}
			      Z_W(\bm g)
			      \ceqq W \cap Z_{\GL_{\mb C}(\mf t)}(\bm g)
			      = \Set{ g'' \in W | g'' \bm g = \bm g g'' \in \GL_{\mb C}(\mf t) },
		      \end{equation}
		      is isomorphic to the \emph{complex reflection group} of the arrangement~\eqref{eq:generic_doubly_twisted_deformation_space_1_coeff} (cf.~\eqref{eq:outer_centralizer_2});

		\item
		      there is a \emph{Galois covering}
		      \begin{equation}
			      \label{eq:full_generic_doubly_twisted_deformation_space_1_coeff}
			      \bm B_{\dot\varphi,r} \bigl( \wh Q \bigr) \lthra \bm B_{\dot\varphi,r} \bigl( \wh Q \bigr) \bs Z_W(\bm g);
		      \end{equation}

		\item
		      and there is a \emph{homeomorphism} $\bm B_{\dot\varphi,r} \bigl( \wh\Theta \bigr) \simeq \bm B_{\dot\varphi,r} \bigl( \wh Q \bigr) \bs Z_W(\bm g)$.
	\end{enumerate}
\end{enonce}

\begin{proof}
	The proof of Prop.-Def.~\ref{prop:full_generic_twisted_deformation_space_1_coeff} extends,
	using nonsplit Lehrer--Springer theory (cf.~\cite[Thm.~6.4~(iii)]{springer_1974_regular_elements_of_finite_reflection_groups}),
	after proving Lem.~\ref{lem:restricted_generic_doubly_twisted_orbit}.
\end{proof}

\begin{lemm}[cf.~Lemm.~\ref{lem:restricted_generic_twisted_orbit} +~\ref{lem:doubly_twisted_restricted_commutators}]
	\label{lem:restricted_generic_doubly_twisted_orbit}

	Choose a $\dot\varphi$-regular element $g' \in W$,
	and a regular eigenvector $A \in \mf t(\bm g,\zeta) \cap \mf t_{\reg}$,
	for some $\zeta \in \mb C^{\ts}$.
	Then the following conditions are \emph{equivalent},
	for any other element $g'' \in W$:
	\begin{enumerate}
		\item
		      one has $g'' \in Z_W(\bm g)$;

		\item
		      the element $g''$ lies in the setwise-stabilizer of the eigenspace~\eqref{eq:twisted_regular_eigenspace},
		      i.e.,
		      in the subgroup
		      \begin{equation}
			      \label{eq:doubly_twisted_eigenspace_normalizer}
			      N_W(\bm g,\zeta)
			      = N_W \bigl( \mf t(\bm g,\zeta) \bigr)
			      \ceqq \Set{ g''' \in W | g''' \bigl( \mf t(\bm g,\zeta) \bigr) \sse \mf t(\bm g,\zeta) };
		      \end{equation}

		\item
		      and one has $g''(A) \in \mf t(\bm g,\zeta)$.
	\end{enumerate}
\end{lemm}

\begin{proof}
	This is the `absolute' case of Lem.~\ref{lem:doubly_twisted_restricted_commutators},
	where one takes $\phi = \vn$.
\end{proof}

\subsection{Full/nonpure quasigeneric case:
	higher pole order}

Once more,
the situation is essentially the same when the pole order of $\wh Q = \sum_{i = 1}^s A_iw^{-i}$ is higher:

\begin{prop}
	\label{Prop:full_generic_doubly_twisted_deformation_space}

	Suppose that $A_s \in \mf t_{\reg}$,
	and denote by $g' = g'_{A_s} \in W$ the---%
	$\dot\varphi$-regular---%
	element determined by $\bm g(A_s) = \zeta_r^s A_s \in \mf t$.
	Then:
	\begin{enumerate}
		\item
		      there is a homeomorphism $\bm B_{\dot\varphi,r} \bigl( \wh\Theta \bigr) \simeq \bm B_{\bm g,r} \bigl( \wh Q \bigr) \bs Z_W(\bm g)$,
		      which is the base of a Galois covering analogous to~\eqref{eq:full_generic_doubly_twisted_deformation_space_1_coeff};

		\item
		      and $\bm B_{\dot\varphi,r} \bigl( \wh\Theta \bigr)$ has the homotopy type of $U \bs Z_W(\bm g)$ (cf.~Prop.~\ref{prop:pure_generic_doubly_twisted_def_space}).
	\end{enumerate}
\end{prop}

\begin{proof}
	The proof of Prop.~\ref{Prop:full_generic_twisted_deformation_space} extends.
\end{proof}

\subsection{Full/nonpure general case:
	minimal pole order}

As for the monolevel case,
suppose that one has $\wh Q = Aw^{-1}$ with \emph{no} constraint on $A \in \mf t$.
Let also $\phi \ceqq \phi_A \sse \Phi$ be the Levi annihilator of $A$.

\begin{enonce}{Proposition-Definition}
	\label{prop:full_general_doubly_twisted_deformation_space_1_coeff}

	Consider again the relative Weyl group $W(\phi) = N_W(\mf t_\phi) \bs W_{\mf t_\phi}$ of ~\eqref{eq:effective_weyl_group}.
	Choose also---%
	w.l.o.g.---%
	an element $g' \in \mb L_{\dot\varphi,r} \bigl( \wh Q \bigr)$,
	and introduce the `centralizer' subgroup
	\begin{equation}
		\label{eq:doubly_twisted_centralizer_in_quotient}
		Z_{W,\phi}(\bm g)
		\ceqq W(\phi) \cap Z_{\GL_{\mb C}(\mf t_\phi)} (\bm g_\phi)
		= \Set{ g''_\phi \in W(\phi) | g''_\phi \bm g_\phi = \bm g_\phi g''_\phi \in \GL_{\mb C}(\mf t_\phi) },
	\end{equation}
	generalizing~\eqref{eq:centralizer_in_quotient} +~\eqref{eq:doubly_twisted_centralizer} (cf.~\eqref{eq:outer_centralizer_2}).
	Then:
	\begin{enumerate}
		\item
		      there is a \emph{Galois covering}
		      \begin{equation}
			      \bm B_{\bm g,r} \bigl( \wh Q \bigr) \lthra \bm B_{\bm g,r} \bigl( \wh Q \bigr) \bs Z_{W,\phi}(\bm g);
		      \end{equation}

		\item
		      and one has a \emph{homeomorphism} $\bm B_{\dot\varphi,r} \bigl( \wh\Theta \bigr) \simeq \bm B_{\bm g,r} \bigl( \wh Q \bigr) \bs Z_{W,\phi}(\bm g)$.
	\end{enumerate}
\end{enonce}

\begin{proof}
	The proof of Prop.-Def.~\ref{prop:full_general_twisted_deformation_space_1_coeff} extends,
	after establishing Lem.~\ref{lem:doubly_twisted_restricted_commutators}.
\end{proof}

\begin{lemm}[cf.~Lemm.~\ref{lem:restricted_commutators} +~\ref{lem:restricted_generic_doubly_twisted_orbit}]
	\label{lem:doubly_twisted_restricted_commutators}

	Choose an element $g' \in W$,
	and an eigenvector $A \in \mf t(\bm g,\zeta)$ with Levi annihilator $\phi = \phi_A \sse \Phi$,
	for some $\zeta \in \mb C^{\ts}$.
	Then,
	in the notation of~\eqref{eq:effective_weyl_group} +~\eqref{eq:doubly_twisted_eigenspace_normalizer}--\eqref{eq:doubly_twisted_centralizer_in_quotient},
	the following conditions are \emph{equivalent}	for any other element $g'' \in N_W(\mf t_\phi)$:
	\begin{enumerate}
		\item
		      one has $g''_\phi \in Z_{W,\phi}(\bm g)$;

		\item
		      one has $g''_\phi \in N_{W(\phi)}(\bm g_\phi,\zeta) = N_{W(\phi)} \bigl( \mf t_\phi (\bm g_\phi,\zeta) \bigr)$;

		\item
		      and one has $g''_\phi(A) \in \mf t_\phi(\bm g_\phi,\zeta)$.
	\end{enumerate}
\end{lemm}

\begin{proof}
	The proof~\ref{proof:lem_restricted_commutators} extends,
	by the following observation:
	for any pair of elements $g',g'' \in W$,
	the commutator $h \ceqq \dot\varphi g' g'' (\dot\varphi g')^{-1} (g'')^{-1} \in \GL_{\mb C}(\mf t)$ lies in $W$.
\end{proof}

\begin{rema}
	\label{rmk:restricted_doubly_twisted_centralizer_independence}

	The group~\eqref{eq:doubly_twisted_centralizer_in_quotient} does \emph{not} depend on the choice of $g' \in \mb L_{\dot\varphi,r} \bigl( \wh Q \bigr)$ (cf.~Rmk.~\ref{rmk:restricted_centralizer_independence}).
	Again,
	this can phrased in terms of Levi subcosets and WWGs (cf.~\S~\ref{sec:twisted_modular_groups_and_cabled_weyl_groups}).
\end{rema}

\begin{exem}
	Suppose,
	e.g.,
	that $1 \in \mb L_{\dot\varphi,r} \bigl( \wh Q \bigr)$,
	i.e.,
	that one has the identity $\dot\varphi (A) = \zeta_r A \in \mf t$.
	It follows that $g' \in W$ generates the $(\dot\varphi,r)$-Galois-orbit of $\wh Q$ if and only if $g' \bigl( \wh Q \bigr) = \wh Q$,
	viz.,
	if and only if $g' \in W_{\mf t_\phi}$ (cf.~Cor.-Def.~\ref{cor:levi_subcoset_twisted}),
	and that the flat $\mf t_\phi$ is $\dot\varphi$-stable.

	One finds
	\begin{equation}
		Z_{W,\phi}(\bm g)
		= Z_{W,\phi}(\dot\varphi)
		= \Set{ g''_\phi \in W(\phi) | \dot\varphi g''_\phi \dot\varphi^{-1} = g''_\phi },
	\end{equation}
	with tacit restriction of $\dot\varphi$.
	(And analogously,
	more generally,
	in the multilevel case.)
\end{exem}

\subsection{Full/nonpure general case:
	higher pole order}

In the multilevel case,
consider the Levi-subsystem filtration $\bm\phi \ceqq \bm\phi \bigl( \wh Q \bigr)$ determined by $\wh Q = \sum_{i = 1}^s A_i w^{-i}$,
as well as the kernel flag $\mf t_{\bm\phi}$ of~\eqref{eq:kernel_flag}.
Then one can state the most general:

\begin{enonce}{Theorem-Definition}
	\label{thm:full_general_doubly_twisted_deformation_space}

	Choose---%
	w.l.o.g.---%
	an element $g' \in \mb L_{\dot\varphi,r} \bigl( \wh Q \bigr)$,
	and let
	\begin{equation}
		\label{eq:restricted_doubly_twisted_centralizer_intersection}
		p_{\bm \phi}^{-1}(\bm g)
		\ceqq \bigcap_{i = 1}^s p_{\phi_i}^{-1} \bigl( Z_{W,\phi_i}(\bm g) \bigr) \sse N_W(\mf t_{\bm \phi}).
	\end{equation}
	(Cf.~Thm.-Def.~\ref{thm:full_general_twisted_deformation_space} +~Prop.-Def.~\ref{prop:full_general_doubly_twisted_deformation_space_1_coeff}.)
	Then:
	\begin{enumerate}
		\item
		      there is a \emph{Galois covering}
		      \begin{equation}
			      \label{eq:doubly_twisted_galois_covering}
			      \bm B_{\dot\varphi,r} \bigl( \wh Q \bigr) \lthra \bm B_{\dot\varphi,r} \bigl( \wh Q \bigr) \bs Z_{W,\bm \phi} (\bm g),
			      \qquad Z_{W,\bm \phi} (\bm g) \ceqq p_{\bm \phi}^{-1} (\bm g) \bs W_{\mf t_{\phi_1}};
		      \end{equation}

		\item
		      and one has a \emph{homeomorphism} $\bm B_{\dot\varphi,r} \bigl( \wh\Theta \bigr) \simeq \bm B_{\dot\varphi,r} \bigl( \wh Q \bigr) \bs Z_{W,\bm \phi} (\bm g)$.
	\end{enumerate}
\end{enonce}

\begin{proof}
	The proof of Thm.-Def.~\ref{thm:full_general_twisted_deformation_space} extends.
\end{proof}

\subsubsection{}

As usual,
we just write $W_{\dot\varphi,r} \bigl( \wh Q \bigr) \ceqq Z_{W,\bm\phi}(\bm g)$,
as this `modular' subquotient of $W$ does \emph{not} depend on the choice of $\bm g$.

Again one can give a more explicit description of such groups by untangling the intersection~\eqref{eq:restricted_doubly_twisted_centralizer_intersection} down to the last piece:

\begin{coro}[cf.~Cor.-Def.~\ref{cor:explicit_deck_group}]
	In the notation of~\eqref{eq:subflag_stabilizer},
	there is a \emph{canonical} group isomorphism
	\begin{equation}
		W_{\dot\varphi,r}\bigl( \wh Q \bigr) \simeq P_{\bm\phi'} \cap Z_{W,\phi_1}(\bm g) \sse W(\phi_1).
	\end{equation}
\end{coro}

\begin{proof}
	One considers classes/restrictions $g'_{\phi_i} \in W(\phi_i)$ of flag-preserving elements $g' \in N_W(\mf t_{\bm\phi})$ which commute with $\bm g_{\phi_i}$ (in $\GL_{\mb C}(\mf t_{\phi_i})$),
	for $i \in \set{1,\dc,s}$:
	it is enough to require that the latter holds at $i = 1$,
	extending the proof~\ref{proof:cor_explicit_deck_group},
	and identifying
	\begin{align}
		\label{eq:most_general_wwg}
		Z_{W,\bm\phi}(\bm g) \simeq \Bigl\{ \, g''_{\phi_1} \in W(\phi_1) \, \Bigr| \, & g''_{\phi_1} \bm g_{\phi_1} = \bm g_{\phi_1} g''_{\phi_1} \in \GL_{\mb C}(\mf t_{\phi_1});                     \\
		                                                                              & \Bigl. g''_{\phi_1}(\mf t_{\phi_i}) \sse \mf t_{\phi_i} \text{ for } i \in \set{2,\dc,s} \, \Bigr\}. \qedhere
	\end{align}
\end{proof}

\subsection{Reduction to the (\texorpdfstring{$\dot\varphi$}{phi}-)simple full case}

The study of full/nonpure deformation spaces also reduces to considering simple Lie algebras,
building on \S~\ref{sec:reduction_simple_case_pure_doubly_twisted}.

Precisely,
first one establishes the following:

\begin{lemm}
	In the notation of~Lem.~\ref{lem:reduction_to_simple_case_twisted},
	let also $\wh\Theta \ceqq \wh\Theta \bigl( \wh Q \bigr)$ (resp.,
	$\wh\Theta_j \ceqq \wh\Theta \bigl( \wh Q_j \bigr)$) be the $(\dot\varphi,r)$-Galois-closed irregular class of $\wh Q$ (resp.,
    the $(\dot\varphi_j,r)$-Galois-closed irregular class of its $\mf t_j$-valued summand,
    setting $\dot\varphi_j \ceqq \eval[1]{\dot\varphi}_{\mf t_j} \in \GL_{\mb C}(\mf t_j)$).\fn{
        If $j \in \ol I'$,
        then this lies in $\Out(\Phi_j)$.
        Otherwise,
        if $j \in I^{\mf Z}$ then it rather equals the scalar $\eval[1]f_{\mf t_j} \in \GL_{\mb C}(\mf t_j) \simeq \mb C^{\ts}$ (which need \emph{not} be $1$;
        but $\Out(\Phi_j)$ is trivial here).}~Then there is a (topological) factorization
	\begin{equation}
		\bm B_{\dot\varphi,r} \bigl( \wh\Theta \bigr)
		= \prod_j \bm B_{\dot\varphi_j,r} \bigl( \wh\Theta_j \bigr),
	\end{equation}
	in the canonical identification $\mb A_s \simeq \prod_j (\mf t_j^s \bs W_j)$ (using~\eqref{eq:weyl_group_varphi_invariant_split}).
\end{lemm}

\begin{proof}
	Using the Galois covering~\eqref{eq:doubly_twisted_galois_covering},
	the proof of Lem.~\ref{lem:reduction_to_simple_full} extends by observing that the `centralizer'~\eqref{eq:doubly_twisted_centralizer_in_quotient} splits as well.
	(Again,
	the caveat is that the passage to irregular classes only involves the action of $W$,
	and \emph{not} the whole of $\Aut(\Phi)$.)
\end{proof}

\subsubsection{}

Once more,
at this stage the general problem thus reduces to treating $\dot\varphi$-\emph{simple} ambient Lie algebras.
Then,
keeping all notation from \S~\ref{sec:from_varphi_simple_to_simple},
one reduces to the simple/irreducible case by means of the:

\begin{prop}
	\label{prop:killing_multiplicities_full}

	Keeping also the notation of Prop.~\ref{prop:killing_multiplicities},
	consider the irregular classes $\wh\Theta \ceqq \wh \Theta \bigl( \wh Q \bigr)$ and $\wh \Theta^{(1)} \ceqq \wh \Theta \bigl( \wh Q^{(1)} \bigr)$,
	so that the former is $(\wt{\dot\varphi},r)$-Galois-closed,
	and the latter is $(\dot\varphi,r')$-Galois-closed (cf.~\eqref{eq:product_outermorphisms}).
	Then the canonical projection $A \mt A^{(1)} \cl \wt{\mf t} \thra \mf t_1$ also induces a \emph{homeomorphism}
	\begin{equation}
		\ol{\pi}^{(1)} \cl \bm B_{\wt{\dot\varphi},r} \bigl( \wh \Theta \bigr) \lxra{\simeq} \bm B_{\dot\varphi,r'} \bigl( \wh \Theta^{(1)} \bigr).
	\end{equation}
\end{prop}

\begin{proof}[Proof postponed to~\ref{proof:prop_killing_multiplicities_full}]
\end{proof}

\subsection{Some more Lie/Levi/Weyl/Howlett theory}

The material of \S~\ref{sec:relative_weyl_groups} still holds,
as it only involves a Levi subsystem $\phi \sse \Phi$---%
it is not about $r$-Galois-orbits.
(Cf.~also \S~\ref{sec:relative_weyl_coset}.)

Now let again $\wh Q = \sum_{i = 1}^s A_i w^{-i}$ be a $(\dot\varphi,r)$-Galois-closed irregular type.
As in \S~\ref{sec:levi_subcosets},
the irregular type $\wh Q$ first determines a filtration of split Levi subcosets $\mb L(\phi_i) = (\mf t,W_{\mf t_{\phi_i}})$ of $\mb G = (\mf t,W)$.
In the ramified case,
this was enriched by adding on a nested twist:
but now we also twist the whole of $\mb G$ \emph{before} starting the $\wh Q$-fission of $G$,
working within the reflection coset $\wt{\mb G} = (\mf t,\dot\varphi W)$ of $\mb G$,
and so in general we rather prove the following:

\begin{enonce}{Corollary-Definition}
	\label{cor:levi_subcoset_twisted}

	Choose $g' \in \mb L_{\dot\varphi,r} \bigl( \wh Q \bigr)$.
	Then:
	\begin{enumerate}
		\item
		      the (rational) reflection coset $\wt{\mb L}_{\dot\varphi,r}(\phi_i) \ceqq \bigl( \mf t,\bm g W_{\mf t_{\phi_i}} \bigr)$ is a \emph{Levi subcoset} of $\wt{\mb G}$,
		      for $i \in \set{1,\dc,s}$;

		\item
		      if $\dot\varphi \neq 1$ then the latter is \emph{nonsplit} for all $i$;

		\item
		      one has
		      \begin{equation}
			      \label{eq:generating_r_galois_vs_generating_phi_r}
			      \bigl( \mf t,\wt{\mb L}_{\dot\varphi,r} \bigl( \wh Q \bigr) \bigr)
			      = \wt{\mb L}_{\dot\varphi,r}(\phi_1),
			      \qquad \wt{\mb L}_{\dot\varphi,r} \bigl( \wh Q \bigr)
			      \ceqq \dot\varphi \cdot \mb L_{\dot\varphi,r} \bigl( \wh Q \bigr);
		      \end{equation}

		\item
		      and $\wt{\mb L}_{\dot\varphi,r}(\phi_i)$ is---%
		      also---%
		      a \emph{Levi subcoset} of $\wt{\mb L}_{\dot\varphi,r}(\phi_j)$,
		      for $i \leq j \in \set{1,\dc,s}$.
	\end{enumerate}
\end{enonce}

\begin{proof}
	For the first and last statements,
	the proofs of Cor.-Def.~\ref{cor:levi_subcoset}~(1.)+(4.) extend,
	noting that
	\begin{equation}
		\mb L_{\dot\varphi,r} \bigl( \wh Q \bigr)
		= \Set{ g'' \in W | A_i \in \mf t(\dot\varphi g'',\zeta_r^i) \text{ for } i \in \set{1,\dc,s} }.
	\end{equation}

	For the second statement,
	if one had $\bm g = h_i \in W_{\mf t_{\phi_i}}$,
	then $\dot\varphi = h_i (g')^{-1} \in W$.

	The third statement follows because the subset $\mb L_{\dot\varphi,r} \bigl( \wh Q \bigr) \sse W$ is---%
	still---%
	a right $W_{\mf t_{\phi_1}}$-torsor (cf.~Rmk.~\ref{rmk:doubly_twisted_independence_restricted_eigenspaces}).
	Precisely,
	consider the \emph{abstract/algebraic} left coset $\mb L'_1 \ceqq g' W_{\mf t_{\phi_1}} \sse W$:
	it need \emph{not} be a Levi reflection subcoset (nor coincide with the right coset),
	as $g'$ need not normalize the parabolic $W_{\mf t_{\phi_1}}$.
	Nevertheless,
	since $W_{\mf t_{\phi_1}}$ is the stabilizer of $\wh Q$,
	one has $\mb L'_1 \sse \mb L_{\dot\varphi,r} \bigl( \wh Q \bigr)$;
	conversely,
	choosing $g'' \in \mb L_{\dot\varphi,r} \bigl( \wh Q \bigr)$,
	the equality $\dot\varphi g' \bigl( \wh Q \bigr) = \dot\varphi g'' \bigl( \wh Q \bigr)$ implies that $g''W_{\mf t_{\phi_1}} = \mb L'_1$ after deleting $\dot\varphi$,
	i.e.,
	one has $g'' \in \mb L'_1$.
	This yields $\mb L_{\dot\varphi,r} \bigl( \wh Q \bigr) = \mb L'_1$ and $\bm g W_{\mf t_{\phi_1}} = \dot\varphi \mb L'_1 = \wt{\mb L}_{\dot\varphi,r} \bigl( \wh Q \bigr)$.
\end{proof}

\subsubsection{}
\label{sec:about_inner_forms_twisted}

As expected,
here we are in the nonsplit setting even when $r = 1$,
at \emph{all} steps of the fission.

Now choose an $\Aut(\Phi')$-\emph{invariant} positive-definite Hermitian form $(\cdot \mid \cdot)$ on $\mf t'$,
rather than just a $W$-invariant one:
it follows that $(\cdot \mid \cdot)$ is also $\Aut(\Phi)$-invariant.
In the nonsemisimple case one \emph{cannot} extend the latter to an $\Aut(\Phi)$-invariant form on the whole of $\mf t$;
but only still to an $\Aut(\Phi')$-invariant one.
Analogously,
the groups $\Aut(\Phi') \sse \Aut(\Phi)$ both preserve the $\mb Q$-form of $\mf t'$,
and the latter preserves the $\mb Q$-form of $\mf t$ only when $\mf t = \mf t'$.
Nonetheless:
(i) up to averaging,
one can assume that $(\cdot \mid \cdot)$ is $\bm g$-\emph{invariant};\fn{
	This is used when defining the twisted quasi-Hamiltonian structure of twisted wild character varieties~\cite{boalch_yamakawa_twisted_wild_character_varieties}.}~and (ii) the element $\bm g$ also preserves the chosen $\mb Q$-form of $\mf t$ (cf.~\S~\ref{sec:outer_isogeny}).
Moreover,
the projected coroots $\alpha^{\dual}_\phi$ actually lie in the subspace $\mf t'_\phi \ceqq \ker(\phi') = \mf t' \cap \mf t_\phi$,
for any Levi subsystem $\phi \sse \Phi$ (cf.~\eqref{eq:flat_split}).
So the order-$2$ reflection~\eqref{eq:reduced_reflection} acts trivially on $\mf Z \sse \mf t_\phi$---%
this was tacit in \S~\ref{sec:howlett_twist}.

With these caveats,
one can now rephrase/prove the analogue of Lem.~\ref{lem:inner_product},
and then Prop.-Def.~\ref{prop:howlett} holds unchanged (for any choice of compatible base $\Delta \sse \Phi$;
cf.~Rmk.~\ref{rmk:dominant_levis}).
The groups of Fig.~\ref{fig:levi_diagram} are defined in the same way,
and Cor.~\ref{cor:howlett_twisted_reflection_cosets} extends:

\begin{lemm}
	\label{lem:double_twist_normalizes}

	Choose an integer $r \geq 1$,
    an outermorphism $\dot\varphi \in \Out(\Phi)$,
    and a vector $A \in \mf t$;
	let $\phi = \phi_A$ be the ($\Delta$-\emph{dominant}) Levi annihilator of $A$,
    and suppose that the irregular type $\wh Q \ceqq Aw^{-1}$ is $(\dot\varphi,r)$-Galois-closed.
	Then there is a \emph{rational} reflection subcoset of $\wt{\mb G} = (\mf t,\dot\varphi W)$,
	defined by
	\begin{equation}
		\wt{\mb G}''_{\dot\varphi,r}(\phi)
		\ceqq \bigl( \mf t_\phi, \bm g_\phi W''(\phi) \bigr),
		\qquad g' \in \mb L_{\dot\varphi,r} \bigl( \wh Q \bigr),
	\end{equation}
    where as usual $\bm g \ceqq \dot\varphi g' \in \Aut(\Phi)$ restricts to $\bm g_\phi$ on $\mf t_\phi$---%
    and the choice of $g'$ is immaterial.
\end{lemm}

\begin{proof}
	Lem.~\ref{lem:doubly_twisted_setwise_stabilizer} implies that $\bm g$ stabilizes $\mf t_\phi$,
    and---%
    by construction---%
    it normalizes $W$.
    It follows that $\bm g$ normalizes $N_W(\mf t_\phi) \sse W$.
	Moreover,
	by Cor.-Def.~\ref{cor:levi_subcoset_twisted} the parabolic subgroup $W_{\mf t_\phi} \sse W$ is also normalized,
	and so the conjugation action of $\bm g_\phi$ preserves $W(\phi)$.
	
    Now Howlett's relative reflection group $W''(\phi)$ is the subgroup of $W(\phi)$ generated by \emph{all} the reflections of the latter in its action on $\mf t_\phi$ (cf.~\eqref{eq:reduced_reflection}),
	and this description is intrinsic:
	conjugation by $\bm g_\phi$ is induced by a $\mb C$-linear automorphism of $\mf t_\phi$,
	so it carries an element fixing a hyperplane $H \sse \mf t_\phi$ pointwise to an element fixing $\bm g_\phi(H)$ pointwise.
	Since $\bm g_\phi$ normalizes $W(\phi)$,
	it therefore permutes the set of reflections lying in $W(\phi)$,
	and preserves the subgroup $W''(\phi)$ which they generate.
	(In particular this holds whether or not Howlett's obstruction $\wt W(\phi)$ is trivial,
    cf.~\eqref{eq:relative_reflection_group_and_twist}.)

	Finally,
	the coset is \emph{rational}:
	the outermorphism $\dot\varphi$ preserves the $\mb Q$-form $\mf t_{\mb Q} \sse \mf t$ (cf.~\S~\ref{sec:outer_isogeny} for the central part,
	the semisimple part preserving the coroot lattice),
	and so does $g' \in W$;
	hence $\bm g_\phi$ preserves $\mf t_{\phi,\mb Q} = \mf t_\phi \cap \mf t_{\mb Q}$,
    while $W''(\phi)$ does so by Prop.-Def.~\ref{prop:howlett}~(2.).
\end{proof}

\subsubsection{}

Finally,
again,
in general~\eqref{eq:doubly_twisted_centralizer_in_quotient} need \emph{not} act as a complex reflection group.
Nonetheless:

\begin{theo}
	In the notation of Lem.~\ref{lem:double_twist_normalizes},
    suppose in addition that Howlett's obstruction \emph{vanishes}---%
    i.e.,
    that $W(\phi) = W''(\phi)$.
    Then the `centralizer' subgroup $Z_{W,\phi}(\bm g) \sse W(\phi)$ of~\eqref{eq:doubly_twisted_centralizer_in_quotient} does act (faithfully) on $\mf t_\phi(\bm g_\phi,\zeta_r)$ as a \emph{complex reflection group}.
\end{theo}

\begin{proof}
	The proof of Thm.~\ref{thm:complex_refl_groups_from_gauge} extends,
	by viewing Lem.~\ref{lem:doubly_twisted_restricted_commutators} in `relative' \emph{nonsplit} Lehrer--Springer theory:
	consider the (inclusion-maximal) $W(\phi)$-regular eigenspace $E \ceqq \mf t_\phi(\bm g_\phi,\zeta_r) \in \on E \bigl( \dot\varphi W(\phi),\zeta_r \bigr)$,
	sitting within the `twisted' poset
	\begin{equation}
		E \bigl( \dot\varphi W(\phi),\zeta_r \bigr)
		\ceqq \Set{ \mf t_\phi( \bm g_\phi g''_\phi,\zeta_r ) | g''_\phi \in W(\phi) }.
	\end{equation}
	(Here $\dot\varphi W(\phi) = \bm g_\phi W(\phi)$ is the $\dot\varphi$-twisted relative Weyl coset,
	cf.~Def.~\ref{def:relative_weyl_coset}.)
\end{proof}

\subsection{Modular groups and twisted wild Weyl groups}
\label{sec:twisted_modular_groups_and_cabled_weyl_groups}

We also generalize the material of \S~\ref{sec:modular_groups_are_twisted_cabled_weyl},
finding discrete parameters which control the most general possible modular groups in meromorphic reductive $2d$ gauge theory.

Let as usual $\bm\phi = (\phi_1,\dc,\phi_s)$ be a Levi-subsystem filtration of $\phi_{s+1} = \Phi$,
of depth $s \geq 1$.

\subsubsection{}

In addition,
let $\dot\varphi \in \Out(\Phi)$ be an outermorphism such that the reflection coset $\dot\varphi W \sse \GL_{\mb C}(\mf t)$ contains an element $\bm g = \dot\varphi g'$ preserving the flag of flats $\mf t_{\bm\phi}$ determined by $\bm\phi$,
i.e.,
such that one has $\bm g \in N_{\dot\varphi W}(\mf t_{\bm\phi}) \neq \vn$ (cf.~Lem.~\ref{lem:twisted_flag_stabilizer}):
the latter condition is \emph{necessary} in order to consider Levi subcosets of $\wt{\mb G} = (\mf t,\dot\varphi W)$ through the nested parabolic subgroups $W_i \ceqq W_{\mf t_{\phi_i}} \sse W$ ($i \in \set{1,\dc,s})$.

Now,
while Lem.-Def.~\ref{lem:cabled_relative_weyl} holds unchanged,
it will be helpful to introduce the natural multilevel version of Def.~\ref{def:relative_weyl_coset} (cf.~\S~\ref{sec:twisted_modular_reparameterization}).
Precisely,
note/recall that:
(i) the `twisted' flag stabilizer $N_{\dot\varphi W}(\mf t_{\bm\phi}) \sse N_{\braket{\dot\varphi, W}}(\mf t_{\bm\phi})$ of~\eqref{eq:twisted_flag_stabilizer} is a coset through the subgroup $N_W(\mf t_{\bm\phi}) \sse W$;
and (ii) the smallest parabolic subgroup $W_1$ acts thereon,
by right multiplication (cf.~Rmk.~\ref{rmk:about_twisted_normalizers},
in the monolevel case).
Then:

\begin{enonce}{Lemma-Definition}
	\label{lem:cabled_weyl_coset}
	\leavevmode

	\begin{enumerate}
		\item
		      The $\dot\varphi$-\emph{twisted wild Weyl coset of} $(\mf g,\mf t,\bm\phi)$,
		      \emph{relative to} $\phi_1$,
		      is the $W(\bm\phi)$-torsor defined by the quotient
		      \begin{equation}
			      \label{eq:cabled_weyl_coset}
			      \dot\varphi W(\bm\phi)
			      \ceqq N_{\dot\varphi W}(\mf t_{\bm\phi}) \bs W_1 \sse \dot\varphi W(\phi_1).
		      \end{equation}

		\item
		      And the restriction map $\bm g'' \mt \bm g''_{\phi_1} \cl N_{\dot\varphi W}(\mf t_{\phi_1}) \to \GL_{\mb C}(\mf t_{\phi_1})$ induces an \emph{injection} $\dot\varphi W(\phi_1) \hra \GL_{\mb C}(\mf t_{\phi_1})$,
		      which further restricts to a $W(\bm\phi)$-\emph{equivariant injection} $\dot\varphi W(\bm \phi) \hra \GL_{\mb C}(\mf t_{\phi_1})$.
	\end{enumerate}
\end{enonce}

\begin{proof}[Proof omitted]
\end{proof}

\subsubsection{}

As usual,
we typically omit the pair $(\mf g,\mf t)$ and just attach wild Weyl cosets to the Levi-subsystem filtration $\bm\phi$ and the twist $\dot\varphi$.
Then,
as for the groups of deck transformations of the Galois coverings of Thm.-Def.~\ref{thm:full_general_doubly_twisted_deformation_space}:

\begin{coro}
	In the notation of~Thm.-Def.~\ref{thm:full_general_doubly_twisted_deformation_space},
	there is a canonical group isomorphism
	\begin{equation}
		Z_{W,\bm\phi}(\bm g)
		\simeq \bigl( N_W(\Delta_{\phi_1}) \cap N_W(\mf t_{\phi_1},\bm g) \bigr) \cap \bigcap_{i = 2}^s N_W(\mf t_{\phi_i},\bm g) \sse W,
	\end{equation}
	where $N_W(\mf t_{\phi_i},\bm g)$ is the \emph{normalizer}---%
	in $N_W(\mf t_{\phi_i})$---%
	\emph{of the Levi subcoset} $\wt{\mb L}_{\dot\varphi,r}(\phi_i) = (\mf t,\bm g W_{\mf t_{\phi_i}})$,
	for $i \in \set{1,\dc,s}$ (cf.~Cor.-Def.~\ref{cor:levi_subcoset_twisted}~(1.)).
\end{coro}

\begin{proof}
	The proof of Cor.~\ref{cor:twisted_weyl_from_coset_normalizers} extends.
\end{proof}

\subsubsection{}

Again,
one can now abstractly consider---%
depth-$s$---%
Levi-subcoset filtrations $\ul{\wt{\mb L}} = (\wt{\mb L}_1,\dc,\wt{\mb L}_s,\wt{\mb L}_{s+1})$ of the starting reflection coset $\wt{\mb L}_{s+1} \ceqq \wt{\mb G}$,
in the natural generalization of Def.~\ref{def:levi_subcoset_filtrations}.
(Again,
we typically drop the ambient reflection coset from the notation.)
So $\wt{\mb L}_i = (\mf t,\bm g_i W_i)$ is a Levi subcoset of $\wt{\mb L}_{i+1}$,
for suitable elements $\bm g_i \in \GL_{\mb C}(\mf t)$ normalizing nested parabolic subgroups $W_1 \sse \dm \sse W_s \sse W_{s+1} \ceqq W$.

The general structural properties all generalize as follows:

\begin{lemm}
    Let $\ul{\wt{\mb L}} = (\wt{\mb L}_1,\dc,\wt{\mb L}_s)$ be a Levi-subcoset filtration of $\wt{\mb G}$.
	Then:
	\begin{enumerate}
		\item
		      the reflection coset $\wt{\mb L}_i$ is a Levi subcoset of $\wt{\mb L}_{i+1},\dc,\wt{\mb L}_s,\wt{\mb L}_{s+1}$;

		\item
		      in particular,
		      all of $\wt{\mb L}_1,\dc,\wt{\mb L}_s$ are Levi subcosets of $\wt{\mb G}$;

		\item
		      there is an underlying Levi-subsystem filtration $\bm \phi = \bm \phi (\ul{\wt{\mb L}}) = (\phi_1,\dc,\phi_s,\Phi)$,
		      determined by the equalities
		      \begin{equation}
			      \wt{\mb L}_i
			      = \bm g_i W_{\mf t_{\phi_i}} \sse \dot\varphi W,
			      \qquad i \in \set{1,\dc,s+1};
		      \end{equation}

		\item
		      viewing $\wt{\mb L}_i$ as a subset of $\GL_{\mb C}(\mf t)$,
		      there are set-theoretic inclusions
		      \begin{equation}
			      \wt{\mb L}_1 \sse \dm \sse \wt{\mb L}_s \sse \dot\varphi W \sse \GL_{\mb C}(\mf t);
		      \end{equation}

		\item
		      conversely,
		      if such set-theoretic inclusions hold for a sequence $(\wt{\mb L}_1,\dc,\wt{\mb L}_s)$ of Levi subcosets of $\wt{\mb G}$,
		      then one has a \emph{Levi-subcoset filtration};

		\item
		      and one has the equalities
		      \begin{equation}
			      \label{eq:twisted_coset_generator}
			      \wt{\mb L}_i
			      = \bm g W_i \sse \dot\varphi W,
			      \qquad i \in \set{1,\dc,s},
		      \end{equation}
		      if and only if $\bm g \in \wt{\mb L}_1$;
		      we say that $\bm g$ \emph{generates} $\ul{\wt{\mb L}}$ if~\eqref{eq:twisted_coset_generator} holds,
		      and it follows that $\bm g \in N_{\dot\varphi W}(\mf t_{\bm\phi}) \sse \dot\varphi W$ (cf.~Lem.~\ref{lem:twisted_flag_stabilizer}).
	\end{enumerate}
\end{lemm}

\begin{proof}
	The arguments of Lem.~\ref{lem:levi_subcoset_filtration_set_inclusions} +~Lem.-Def.~\ref{lem:principal_levi_subcoset_filtrations} (and surrounding remarks) generalize.
	E.g.,
    an element of $\braket{\dot\varphi,W} \sse N_{\GL_{\mb C}(\mf t)}(W)$ normalizes $W_i$ if and only if it preserves the corresponding flat $\mf t_{\phi_i}$;
	etc.
\end{proof}

\subsubsection{}

More explicitly,
therefore,
we look at nested sequences of reflection cosets of the form $\bm g W_i = \dot\varphi g' W_i \sse \dot\varphi W$,
for a suitable generator $\bm g = \dot\varphi g' \in \dot\varphi W$.
The latter is canonically defined modulo $W_1$,
and so,
e.g.,
this yields a canonical element of $\dot\varphi W$ in the regime of complete/toral fission---%
where $W_1$ is trivial,
and the smallest subcoset is a maximal torus $\wt{\mb T} = (\mf t,\bm g)$ of $\wt{\mb G}$,
etc.

In general,
only the class/restriction $\bm g_{\phi_1}$ of $\bm g$ in the $\dot\varphi$-twisted wild Weyl coset $\dot\varphi W(\bm \phi)$ is well-determined (cf.~Lem.-Def.~\ref{lem:cabled_weyl_coset}).

\subsubsection{}

Now,
given a Levi subsystem $\phi \sse \Phi$,
the normalizer subgroup $W^{\wt{\mb L}} \sse W$ of a Levi subcoset $\bm g W_{\mf t_\phi} \sse \dot\varphi W$ through $W_{\mf t_\phi}$ is defined as in~\eqref{eq:levi_subcoset_stabilizer}:
it coincides with the elements $g'' \in N_W(\mf t_\phi)$ which commute with $\bm g$,
modulo $W_{\mf t_\phi}$;
i.e.,
with the classes/restrictions $g''_\phi$ which commute with $\bm g_\phi$---%
inside $\GL_{\mb C}(\mf t_\phi)$ (cf.~Lem.~\ref{lem:coset_normalizer_is_stabilizer},
which extends).

Iterating along a Levi-subcoset filtration $\ul{\wt{\mb L}}$ of $\wt{\mb G}$,
we define the quotient group $W\bigl(\ul{\wt{\mb L}}\bigr) = W^{\ul{\wt{\mb L}}} \bs W_1$ as in~\eqref{eq:twisted_cabled_weyl},
and we still regard it as a \emph{wild Weyl group} (= WWG)---%
relative to $\phi_1$.\fn{
	The latter is a \emph{subgroup} of $W(\bm\phi)$,
	rather than a \emph{coset} thereof;
	it should \emph{not} be confused with~\eqref{eq:cabled_weyl_coset}.}

Moreover,
the nested flat-density condition of~\eqref{eq:nested_flat_density} is imposed in analogous fashion,
noting that the eigenspaces
\begin{equation}
	\mf t_{\phi_i}(\wt{\mb L}_i,\zeta)
	\ceqq \mf t_{\phi_i} (\bm g_{\phi_i},\zeta) \sse \mf t_{\phi_i},
	\qquad i \in \set{1,\dc,s},
	\quad \zeta \in \mb C^{\ts},
\end{equation}
remain well-defined.
This leads to the notions of $(\dot\varphi,r)$-\emph{Galois-closed} Levi-subcoset filtrations,
and of $\dot\varphi$-\emph{twisted} $r$-\emph{ramified WWGs},
denoted by $W_{\dot\varphi,r}\bigl(\ul{\wt{\mb L}}\bigr)$ (cf.~Lem.-Def.~\ref{lem:twisted_cabled_weyl},
which extends).

\subsubsection{}

Now one can prove that:
(i) the $\dot\varphi$-twisted $r$-ramified WWGs are precisely the groups of deck transformations arising from Thm.-Def.~\ref{thm:full_general_doubly_twisted_deformation_space};
(ii) the Levi-subcoset filtrations provide \emph{complete invariants} for $(\dot\varphi,r)$-admissible deformations of irregular types;
and (iii) they also determine the $(\dot\varphi,r)$-admissible deformation spaces---%
now having discrete/countable parameters.
Precisely:

\begin{enonce}{Theorem-Definition}
	\leavevmode
	\begin{enumerate}
		\item
		      Let $\wh Q$ be a $(\dot\varphi,r)$-Galois-closed irregular type $\wh Q$ with underlying Levi-subcoset filtration $\ul{\wt{\mb L}} \ceqq \ul{\wt{\mb L}} \bigl( \wh Q \bigr)$.
		      (The latter is determined by $\wt{\mb L}_i \ceqq \wt{\mb L}_{\dot\varphi,r}(\phi_i)$, cf.~Cor.-Def.~\ref{cor:levi_subcoset_twisted}~(1.).)
		      Then:
		      \begin{enumerate}
			      \item
			            the lowest term $\wt{\mb L}_1$ of $\ul{\wt{\mb L}}$ is the left $\dot\varphi$-translation~\eqref{eq:generating_r_galois_vs_generating_phi_r} of the set of elements $g' \in W$ generating the $(\dot\varphi,r)$-Galois-orbit of $\wh Q$,
			            i.e.,
			            the set of elements $\bm g \in \dot\varphi W$ generating its $r$-Galois-orbit;

			      \item
			            given any $g' \in \mb L_{\dot\varphi,r} \bigl( \wh Q \bigr) = \dot\varphi^{-1}\wt{\mb L}_1$,
			            the $i$-th direct factor~\eqref{eq:doubly_twisted_pure_def_space_factor} coincides with the (nonempty) complex hyperplane complement
			            \begin{equation}
				            \label{eq:intrinsic_pure_deformation_factors_twisted}
				            \bm B_{\dot\varphi,r}\bigl(\ul{\wt{\mb L}},i\bigr) \sse \mf t_{\phi_i}\bigl(\wt{\mb L}_i,\zeta_r^i\bigr),
			            \end{equation}
			            defined as in~\eqref{eq:intrinsic_pure_deformation_factors};

			      \item
			            there are canonical identifications of:
			            \begin{enumerate}
				            \item
				                  the admissible deformation space $\bm B_{\bm g,r} \bigl( \wh Q \bigr)$ with the $\dot\varphi$-\emph{twisted} $r$-\emph{ramified stratum of} $\ul{\wt{\mb L}}$,
				                  i.e.,
				                  with the topological subspace
				                  \begin{equation}
					                  \label{eq:twisted_ramified_subcoset_filtration_stratum}
					                  \bm B_{\dot\varphi,r}\bigl(\ul{\wt{\mb L}}\bigr)
					                  \ceqq \prod_{i = 1}^s \bm B_{\dot\varphi,r}\bigl(\ul{\wt{\mb L}},i\bigr) \sse \mf t^s;
				                  \end{equation}
				                  (As usual,
				                  the latter now manifestly does \emph{not} depend on the choice of the generator $\bm g \in \wt{\mb L}_1$.)

				            \item
				                  and of $W_{\dot\varphi,r}\bigl(\ul{\wt{\mb L}}\bigr)$ with the subquotient $Z_{W,\bm\phi}(\bm g)$ of~\eqref{eq:doubly_twisted_galois_covering};
			            \end{enumerate}

			      \item
			            and another $(\dot\varphi,r)$-Galois-closed irregular type $\wh Q'$ satisfies $\wh Q' \sim_{\dot\varphi,r} \wh Q$ if and only if $\ul{\wt{\mb L}} \bigl( \wh Q' \bigr) = \ul{\wt{\mb L}}$.
		      \end{enumerate}

		\item
		      Conversely,
		      let $\ul{\wt{\mb L}}$ be a Levi-subcoset filtration of $\wt{\mb G}$.
		      Then the stratum~\eqref{eq:twisted_ramified_subcoset_filtration_stratum} is \emph{nonempty} if and only if $\ul{\wt{\mb L}}$ is $(\dot\varphi,r)$-Galois-closed,
		      and this happens if and only if there exists a $(\dot\varphi,r)$-Galois-closed irregular type $\wh Q$ such that $\ul{\wt{\mb L}} = \ul{\wt{\mb L}} \bigl( \wh Q \bigr)$.
	\end{enumerate}
\end{enonce}

\begin{proof}
	The proofs of Thm.-Def.~\ref{thm:levi_subcosets_and_modular_groups} +~Cor.-Def.~\ref{cor:levi_subcoset_filtrations_are_complete_invariants} extend.
\end{proof}

\subsubsection{}

Finally,
consider again the $W$-action $\ul{\wt{\mb L}} \mt h.\ul{\wt{\mb L}}$ of~\eqref{eq:levi_subcoset_filtration_orbit},
but this time on Levi-subcoset filtrations $\ul{\wt{\mb L}}$ of $\wt{\mb G}$---%
rather than $\mb G$;
this preserves the ambient reflection coset $\dot\varphi W \sse \GL_{\mb C}(\mf t)$,
and so indeed the result is still a nested sequence $h.\wt{\mb L}_1 \sse \dm \sse h.\wt{\mb L}_s \sse \dot\varphi W$.
Once more,
a \emph{Levi-subcoset-filtration orbit} $W.\ul{\wt{\mb L}}$ is $(\dot\varphi,r)$-\emph{Galois-closed} if this holds for any representative therein,
which remains a well-defined notion.

Now:

\begin{enonce}{Corollary-Definition}
	\leavevmode
	\begin{enumerate}
		\item
		      Two $(\dot\varphi,r)$-Galois-closed irregular classes $\wh \Theta$ and $\wh \Theta'$ are mutual $(\dot\varphi,r)$-admissible deformations if and only if their underlying Levi-subcoset-filtration orbits $W.\ul{\wt{\mb L}}$ and $W.\ul{\wt{\mb L}}'$ coincide;
		      the latter are (well-)defined as the $W$-orbits of $\ul{\wt{\mb L}} \ceqq \ul{\wt{\mb L}} \bigl( \wh Q \bigr)$ and $\ul{\wt{\mb L}}' \ceqq \ul{\wt{\mb L}} \bigl( \wh Q' \bigr)$,
		      where $\wh Q$ and $\wh Q'$ are irregular types such that $\wh\Theta = \wh\Theta \bigl( \wh Q \bigr)$ and $\wh\Theta' = \wh\Theta \bigl( \wh Q' \bigr)$---%
		      respectively.

		\item
		      There is a canonical isomorphism of $\bm B_{\dot\varphi,r} \bigl( \wh Q \bigr) \thra \bm B_{\dot\varphi,r} \bigl( \wh \Theta \bigr)$ with the Galois covering
		      \begin{equation}
			      \bm B_{\dot\varphi,r} \bigl(\ul{\wt{\mb L}}\bigr) \lthra \bm B_{\dot\varphi,r} \bigl(W.\ul{\wt{\mb L}}\bigr),
		      \end{equation}
		      with Galois group $W_{\dot\varphi,r}\bigl(\ul{\wt{\mb L}}\bigr)$,
		      by also introducing the $\dot\varphi$-\emph{twisted} $r$-\emph{ramified} (\emph{quotient}) \emph{stratum of} $W.\ul{\wt{\mb L}}$,
		      i.e.,
		      the topological subspace
		      \begin{equation}
			      \label{eq:twisted_ramified_quotient_stratum_from_subcoset_filtration_orbit}
			      \bm B_{\dot\varphi,r} \bigl( W.\ul{\wt{\mb L}} \bigr)
			      \ceqq \bigl( W.\bm B_{\dot\varphi,r}\bigl(\ul{\wt{\mb L}}\bigr) \bigr) \bs W \simeq
			      \bm B_{\dot\varphi,r} \bigl( \ul{\wt{\mb L}} \bigr) \bs W_{\dot\varphi,r}\bigl(\ul{\wt{\mb L}}\bigr) \sse \mb A_s.
		      \end{equation}
		      (Cf.~\eqref{eq:stratum_of_subcoset_orbit} +~\eqref{eq:twisted_ramified_subcoset_filtration_stratum}.)

		\item
		      And conversely,
		      the stratum~\eqref{eq:twisted_ramified_quotient_stratum_from_subcoset_filtration_orbit} is \emph{nonempty} if and only if $W.\ul{\wt{\mb L}}$ is $(\dot\varphi,r)$-Galois-closed,
		      i.e.,
		      if and only if it comes from a $(\dot\varphi,r)$-Galois-closed irregular class.
	\end{enumerate}
\end{enonce}

\begin{proof}
	The proofs of~Corr.-Deff.~\ref{cor:levi_subcoset_orbit_from_irr_class} +~\ref{cor:levi_subcoset_filtration_orbits_are_complete_invariants} generalize,
	upon establishing the $\dot\varphi$-twisted version of Lem.~\ref{lem:irr_type_levi_subcosets_equivariance};
	in turn,
	the latter holds because of the identity
	\begin{equation}
		\label{eq:twisted_conjugated_orbit_generators}
		h \cdot \wt{\mb L}_{\dot\varphi,r} \bigl( \wh Q \bigr) \cdot h^{-1}
		= \wt{\mb L}_{\dot\varphi,r} \bigl( \wh Q' \bigr) \sse \dot\varphi W,
		\qquad h \in W,
		\quad \wh Q'
		\ceqq h \bigl( \wh Q \bigr),
	\end{equation}
	generalizing~\eqref{eq:conjugated_orbit_generators}---%
	which is shown as in~\ref{proof:lem_more_on_irregular_types}.
\end{proof}

\subsection{Twisted modular reparameterizations}
\label{sec:twisted_modular_reparameterization}

As in \S~\ref{sec:modular_reparameterization},
one can provide an equivalent parameterization of Levi-subcoset filtrations---%
and their $W$-orbits---%
which separates the `untwisted' Levi-subsystem filtration of $\Phi$ from the chosen twist $\bm g = \dot\varphi g' \in \dot\varphi W$.

\subsubsection{}

Namely,
consider pairs $(\bm\phi,\bm g_{\bm \phi})$,
where:
(i) $\bm \phi$ is a Levi-subsystem filtration;
and (ii) $\bm g_{\bm\phi} = \bm g_{\phi_1} \in \dot\varphi W(\bm\phi)$ is the class/restriction of a flag-preserving element $\bm g \in N_{\dot\varphi W} (\mf t_{\bm\phi})$,
i.e.,
an element of the $\dot\varphi$-twisted wild Weyl coset of $\bm\phi$ (cf.~Lem.-Def.~\ref{lem:cabled_weyl_coset};
the implicit assumption is still that $N_{\dot\varphi W}(\mf t_{\bm\phi}) \neq \vn$).
Let then $\ul{\wt{\mb L}} = \ul{\wt{\mb L}} (\bm\phi,\bm g_{\bm\phi})$ be the Levi-subcoset filtration of $\wt{\mb G}$ defined by the following sequence of nested reflection subcosets:
\begin{equation}
	\wt{\mb L}_i ( \bm\phi,\bm g_{\bm\phi})
	\ceqq (\mf t,\bm g W_{\mf t_{\phi_i}}),
	\qquad i \in \set{1,\dc,s}.
\end{equation}
(Cf.~\eqref{eq:levi_subcoset_filtration_from_pair}.)
Moreover,
given an element $h \in W$,
consider the new pair consisting of:
(a) the translated Levi-subsystem filtration $\bm \phi' \ceqq h(\bm \phi)$;
and (b) the class/restriction $\bm g'_{\bm\phi'} \in \dot\varphi W(\bm\phi')$ of the conjugate element $\bm g' \ceqq h \bm g h^{-1} \in N_{\dot\varphi W} (\mf t_{\bm\phi'})$.
Then:

\begin{prop}
	\label{prop:twisted_pair_parameters_levi_subcoset_filtrations}

	The map $(\bm\phi,\bm g_{\bm\phi}) \mt \ul{\wt{\mb L}}(\bm\phi,\bm g_{\bm\phi})$,
	targeting the set of Levi-subcoset filtrations of $\wt{\mb G}$,
	is a $W$-\emph{equivariant bijection}.
\end{prop}

\begin{proof}
	The proof of Prop.~\ref{prop:pair_parameters_levi_subcoset_filtrations} extends.
\end{proof}

\subsubsection{}

There is again an induced bijection from $W$-orbits of pairs $(\bm\phi,\bm g_{\bm \phi})$ to Levi-subcoset-filtration orbits $W.\ul{\wt{\mb L}}$.

Moreover,
the set of $(\dot\varphi,r)$-\emph{modular pairs} consists---%
by definition---%
of elements $(\bm\phi,\bm g_{\bm\phi} )$ whose associated Levi-subcoset filtration is $(\dot\varphi,r)$-Galois-closed.
In turn,
fixing the Levi-subsystem filtration $\bm \phi$,
the set of $(\bm\phi,\dot\varphi,r)$-\emph{modular classes} is
\begin{equation}
	\label{eq:twisted_modular_classes}
	M_{\dot\varphi,r} (\bm\phi)
	\ceqq \Set{ \bm g_{\bm\phi} \in \dot\varphi W(\bm\phi) | \ul{\wt{\mb L}} (\bm\phi,\bm g_{\bm\phi}) \text{ is } (\dot\varphi,r)\text{-Galois-closed}}.
\end{equation}
Finally,
the latter can be identified with the set of cosets $\wt{\mb L}_{\dot\varphi,r} \bigl( \wh Q \bigr)$,
through the parabolic subgroup $W_1 = W_{\mf t_{\phi_1}}$,
where $\wh Q$ is a $(\dot\varphi,r)$-Galois-closed irregular type such that $\bm\phi \bigl( \wh Q \bigr) = \bm\phi$ (cf.~\eqref{eq:modular_classes} +~\eqref{eq:generating_r_galois_vs_generating_phi_r};
the usual caveat is that one has $W_1 = W^{\wh Q} \sse W$ for all such irregular types).

The group $W(\bm\phi)$ acts on the wild Weyl coset $\dot\varphi W(\bm\phi)$ \emph{by conjugation},
and the subset~\eqref{eq:twisted_modular_classes} is stable under this action.
We now view this as a fairly vast generalization of Lehrer--Springer's \emph{twisted/nonsplit} $r$-regularity condition,
in view of the following:

\begin{exem}[cf.~Exmp.~\ref{exmp:springer_again}]
	Consider $s = 1$ and $\phi = \phi_1 \ceqq \vn$.
	Then $N_{\dot\varphi W}(\mf t_\phi) = \dot\varphi W$,
	and an element $\bm g \in \dot\varphi W \simeq \dot\varphi W(\phi)$ is $(\phi,\dot\varphi,r)$-modular if and only if it admits a $W$-regular $\zeta_r$-eigenvector.
	So~\eqref{eq:twisted_modular_classes} is the set of $r$-regular elements of the reflection coset $\wt{\mb G} = \dot\varphi W$,
	in the sense of~\cite[\S~6]{springer_1974_regular_elements_of_finite_reflection_groups},
	which consists of a single conjugacy class under $W = W(\phi)$ (cf.~Thmm.~6.2. +~6.4 of op.~cit.).

	The WWG is the centralizer of the maximal torus
	\begin{equation}
		\wt{\mb T} \ceqq \wt{\mb L}(\vn,\bm g) = (\mf t,\bm g) \sse \wt{\mb G},
	\end{equation}
	i.e.,
	the `centralizer' of $\bm g$.
	This generalizes to arbitrary $s \geq 1$,
	in the \emph{quasigeneric} case,
	where $\phi_i = \vn$ for $i \in \set{1,\dc,s}$.
\end{exem}

\begin{exem}[cf.~Exmp.~\ref{exmp:modular_toral_fission}]
	More generally,
	if fission is complete/toral,
	then one has $\dot\varphi W(\bm\phi) = N_{\dot\varphi W}(\mf t_{\bm\phi}) \sse \dot\varphi W$ and $W_{\dot\varphi,r}(\ul{\wt{\mb L}}) = W^{\ul{\wt{\mb L}}}$,
    where $\ul{\wt{\mb L}} \ceqq \ul{\wt{\mb L}}(\bm\phi,\bm g)$ for any $\bm g \in M_{\dot\varphi,r}(\bm\phi)$,
	and the WWG coincides with the `centralizer' of $\bm g$ in $N_W(\mf t_{\bm\phi})$.
\end{exem}

\begin{rema}
	The question of Rmk.~\ref{rmk:transitivity_question} generalizes to the present setting,
	asking whether the group $W(\bm\phi)$ acts \emph{transitively} on $M_{\dot\varphi,r} (\bm\phi)$,
	for suitable parameters $(\dot\varphi,\bm\phi,r)$.
\end{rema}

\subsubsection{}

There are then nonreflection obstructions analogous to \S~\ref{sec:nonreflection_obstructions_again},
and the proof of Prop.~\ref{prop:modular_examples} (resp.,
of Cor.~\ref{cor:complete_fission_modular_examples}) extends to treat Levi subcosets $\wt{\mb L} = (\mf t,\bm g W_{\mf t_\phi}) \sse \wt{\mb G}$ such that $\bm g^r \in W_{\mf t_\phi}$,
for any Levi subsystem $\phi \sse \Phi$ (resp.,
maximal tori $\wt{\mb T} = (\mf t,\bm g) \sse \wt{\mb G}$ such that $\bm g^r = 1$).

\subsection{Some more nonsplit Lehrer--Springer theory}

Consider again the relative reflection groups $G_\bullet(\phi)$ of \S\S~\ref{sec:type_A_relative_refl_group}--\ref{sec:type_D_relative_refl_group},
for Weyl groups of classical type $\bullet \in \set{A,B\!C,D}$,
whose definitions only depend on a Levi subsystem $\phi$ of the corresponding root system---%
not on the ramification $r$,
or the twist $\dot\varphi$.

\subsubsection{}

The point is to extend the statement of Lem.~\ref{lem:restricted_reflection_group} in \emph{nonsplit} Lehrer--Springer theory,
so that Corr.~\ref{cor:restricted_reflection_group} +~\ref{cor:pure_factors_as_reflection_complements} go through as well.
This works,
in view of the following:

\begin{lemm}
	\label{lem:twisted_normalization_relative_refl_group}

	Choose an element $\bm g \in N_{\dot\varphi W}(\mf t_\phi)$.
	Then:
	\begin{enumerate}
		\item
		      If $\bullet \in \set{A,B\!C}$,
		      then the restriction $\bm g_\phi$ \emph{normalizes} the relative reflection group $G_\bullet(\phi) \sse \GL_{\mb C}(\mf t_\phi)$;

		\item
		      and if the ambient rank $m$ is \emph{greater than} $4$,
		      then this also holds in type $\bullet = D$.
	\end{enumerate}
\end{lemm}

\begin{proof}[Proof postponed to~\ref{proof:lem_twisted_normalization_relative_refl_group}]
\end{proof}

\subsubsection{}

So one can now work within the \emph{possibly nonsplit} (rational) reflection coset $\bigl(\mf t_\phi,\bm g_\phi G_\bullet (\phi)\bigr)$,
based on the general version of~\cite[Thm.~1.1 + Thm.~A]{lehrer_springer_1999_reflection_subquotients_of_unitary_reflection_groups},
where the maximality condition~\eqref{eq:maximality_in_coset} rather reads
\begin{equation}
	\dim_{\mb C} \bigl( \mf t_\phi( \bm g_\phi,\zeta_r ) \bigr) \geq \dim_{\mb C} \bigl( \mf t_\phi( \bm g_\phi g'',\zeta_r ) \bigr),
	\qquad g'' \in G_\bullet(\phi).
\end{equation}

Moreover,
the complex reflection group~\eqref{eq:outer_centralizer_2} now gets replaced by the `centralizer'
\begin{equation}
			      Z_{G_\bullet(\phi)}(\bm g_\phi)
                  \ceqq G_\bullet(\phi) \cap Z_{\GL_{\mb C}(\mf t_\phi)}(\bm g_\phi)
                  = \Set{ g'' \in G_\bullet(\phi) | \bm g_\phi g'' = g'' \bm g_\phi \in \GL_{\mb C}(\mf t_\phi) }.
		      \end{equation}
              As in the split case,
              the latter contains the modular WWG $Z_{W,\phi}(\bm g) \sse W(\phi)$,
              and its regular part can---%
              and will---%
              be identified with the pure $(\dot\varphi,r)$-admissible deformation space $\bm B_{\dot\varphi,r} \bigl( \wh Q \bigr) \sse \mf t_\phi (\bm g_\phi,\zeta_r)$,
              where $A \in \mf t_\phi (\bm g_\phi,\zeta_r)$ is an eigenvector with trivial $G_\bullet(\phi)$-stabilizer,
              and setting $\wh Q \ceqq A w^{-1}$.

This makes it possible to extend the crystallographic classification of \S\S~\ref{sec:type_A}--\ref{sec:type_D}.

\subsection{Doubly-twisted pure type \texorpdfstring{$A$}{A}}
\label{sec:doubly_twisted_type_A}

By the proof of Lem.~\ref{lem:twisted_normalization_relative_refl_group},
the new quasigeneric examples involve the $2$-twisted reflection coset $( \mb C^m,- \mf S_m )$ ($m \geq 3$).
They are treated by the following:

\begin{prop}[cf.~Prop.~\ref{prop:classification_type_A}]
    \label{prop:twisted_classification_type_A}

	Choose integers $r \geq 1$ and $m \geq 3$,
	and consider a $\dot\varphi$-regular element $g'$ of $\bigl( V^+_m,W_A(m) \bigr)$.
	Set also
	\begin{equation}
		\bm B_{A,\dot\varphi,r}(m)
		= \bm B_{A,\bm g,r}(m)
		\ceqq V^+_m(\bm g,\zeta_r) \cap V^+_{m,\reg}.
	\end{equation}
	Then:
	\begin{enumerate}
		\item
		      one of the four following (mutually-exclusive) situations occurs:
		      \begin{enumerate}
			      \item
			            the integer $r$ is odd,
			            and either $2r \mid m$ or $2r \mid (m-1)$;
			            if $k \geq 1$ is the quotient of the division then there is a $\mb C$-linear isomorphism $\bm B_{A,\dot\varphi,r}(m) \simeq \mc M^\sharp(2r,k)$,
			            of complex hyperplane complements;

			      \item
			            or $r \equiv 0 \pmod 4$,
			            and either $r \mid m$ or $r \mid (m-1)$;
			            if again $k \geq 1$ is the quotient then $\bm B_{A,\dot\varphi,r}(m) \simeq \mc M^\sharp(r,k)$;

			      \item
			            or $r \geq 4$ and $r \equiv 2 \pmod 4$,
			            and either $r \slash 2 \mid m$ or $r \slash 2 \mid (m-1)$;
			            if $k \geq 1$ is the quotient then $\bm B_{A,\dot\varphi,r}(m) \simeq \mc M^\sharp(r \slash 2,k)$;

			      \item
			            or $r = 2$,
			            in which case $\bm B_{A,\dot\varphi,r}(m) = \bm B_{A,\dot\varphi,2}(m) \simeq \mc M(1,m)$.
		      \end{enumerate}

		\item
		      and the corresponding complex reflection group is a generalized symmetric group $G(\rho,1,k)$,
		      where---%
		      accordingly---%
		      one has $\rho \in \set{2r,r,r \slash 2}$ in cases (a)--(c),
		      and $\rho = 1$ in case (d).
	\end{enumerate}
\end{prop}

\begin{proof}
	One looks at $W_A(m)$-regular vectors $A \in V^+_{m,\reg}$ such that $-g'(A) = \zeta_r A$.
	This means that
	\begin{equation}
		A \in V^+_m(g',-\zeta_r) \cap V^+_{m,\reg}
		= \bm B_{A,\dot\varphi,r}(m),
	\end{equation}
	and so in particular $g'$ is a regular element of $W_A(m)$.
	Its order $d \in \set{ 2r,r,2\slash 2 }$ is controlled by~Lem.~\ref{lem:even_odd_roots_of_1}~(3.)---%
	after establishing the order of $-\zeta_r$ in $\mb C^{\ts}$.
	In particular,
	if $d \geq 2$ then we conclude by Prop.~\ref{prop:classification_type_A},
	and this happens if and only if $r \neq 2$.
	If instead $r = 2$,
	whence $d = 1$,
	then we just find $g' = 1$ and
	\begin{equation}
		\bm B_{A,\dot\varphi,r}(m) = \bm B_{A,\dot\varphi,2}(m) = V^+_{m,\reg}.
	\end{equation}

	As for the complex reflection groups,
	note that $\dot\varphi$ here is a \emph{homothety},
	and so the centralizers are computed as before.
\end{proof}

\subsection{Doubly twisted pure type \texorpdfstring{$B\!C$}{BC}}
\label{sec:no_doubly_twist_BC}

As already mentioned,
here there are \emph{no} nontrivial outermorphisms and the classification stays as in \S~\ref{sec:type_BC}.

\subsection{Doubly twisted pure type \texorpdfstring{$D$}{D}}
\label{sec:doubly_twisted_type_D}

By the proof of Lem.~\ref{lem:twisted_normalization_relative_refl_group},
the new examples in rank $m \geq 5$ involve the $2$-twisted reflection coset $\bigl( \mb C^m, \dot\varphi W_D(m) \bigr)$,
in the quasigeneric case,
where $\dot\varphi$ is the sign-swap of the last coordinate.
(The remaining crystallographic examples are equivalent to lower-rank quasigeneric cases in type $B\!C$,
whence any automorphism becomes \emph{inner}.)

These are treated by the following:

\begin{prop}[cf.~Prop.~\ref{prop:generic_classification_type_D}]
	\label{prop:twisted_generic_classification_type_D}

	Choose integers $r \geq 1$ and $m \geq 4$,
	and consider a $\dot\varphi$-regular element $g \in W_D(m)$.
	Set also
	\begin{equation}
        \label{eq:regular_complement_D_double_twisted}
		\bm B_{D,\dot\varphi,r}(m)
		= \bm B_{D,\bm g,r}(m)
		\ceqq \wt V_m(\bm g,\zeta_r) \cap \wt V_{m,\reg}.
	\end{equation}
	Then one of the five following (mutually-exclusive) situations occurs:
	\begin{enumerate}
		\item
		      the integer $r$ is odd and $r \mid (m-1)$,
		      and if $k \geq 1$ is the quotient of the division then:
		      \begin{enumerate}
			      \item
			            there is a $\mb C$-linear isomorphism $\bm B_{D,\dot\varphi,r}(m) \simeq \mc M^\sharp(2r,k)$,
			            of complex hyperplane complements;

			      \item
			            and the corresponding complex reflection group is the monomial entry $G(2r,1,k)$ in Shephard--Todd's list;
		      \end{enumerate}

		\item
		      or $r \geq 4$ is even and $r \mid (2m)$,
		      and the quotient $k \geq 1$ is \emph{odd},
		      and then:
		      \begin{enumerate}
			      \item
			            there is a $\mb C$-linear isomorphism $\bm B_{D,\dot\varphi,r}(m) \simeq \mc M^\sharp(r,k)$;

			      \item
			            and the corresponding complex reflection group is $G(r,2,k)$;
		      \end{enumerate}

		\item
		      or $r \geq 4$ is even and $r \mid (2m-2)$,
		      and if $k$ is the quotient then:
		      \begin{enumerate}
			      \item
			            one has $\bm B_{D,\dot\varphi,r}(m) \simeq \mc M^\sharp(r,k)$;

			      \item
			            and the corresponding complex reflection group is $G(r,1,k)$;
		      \end{enumerate}

		\item
		      or $r = 2$ and $m$ is \emph{even},
		      and then:
		      \begin{enumerate}
			      \item
			            one has $\bm B_{D,\dot\varphi,r}(m) = \bm B_{D,\dot\varphi,2}(m) \simeq \mc M^\sharp(2,m-1)$;

			      \item
			            and the corresponding reflection group is $G(2,1,m-1) \simeq W_{B\!C}(m-1)$;
		      \end{enumerate}

		\item
		      or finally $r = 2$ and $m$ is \emph{odd},
		      and then:
		      \begin{enumerate}
			      \item
			            one has $\bm B_{D,\dot\varphi,r}(m) = \bm B_{D,\dot\varphi,2}(m) \simeq \mc M(2,m)$;

			      \item
			            and the corresponding reflection group is $G(2,2,m) \simeq W_D(m)$;
		      \end{enumerate}

	\end{enumerate}
\end{prop}

\begin{proof}
	As for the hyperplane complements,
	the Weyl group $W_{B\!C}(m)$ acts by automorphisms of the---%
	nonlevi---%
	root subsystem $\Phi_D(m) = \Phi_B(m) \cap \Phi_C(m) \sse \Phi_{B/C}(m)$
	(of long/short roots,
	respectively);
	so we are in the context of~\cite[Lem.~6.8]{springer_1974_regular_elements_of_finite_reflection_groups}.
	Then we appeal to the classification of \S~6.11 of op.~cit.~(cf.~\cite{steinberg_1968_endomorphisms_of_linear_algebraic_groups}),
	i.e.,
	one of the following holds:
	\begin{enumerate}
		\item
		      (i) the integer $r$ is odd and $r \mid (m-1)$%
		      ---say $m-1 = rk$ for an integer $k \geq 1$;
		      and (ii) $\bm g$ is a product of $k$ positive disjoint $r$-cycles and one negative transposition;

		\item
		      or (i) $r \geq 4$ is even and $r \mid (2m)$;
		      (ii) the quotient of the division is odd%
		      ---say $r = 2r'$ and $m = kr'$ for an integer $r' \geq 2$ and an odd integer $k \geq 1$;
		      and (iii) $\bm g$ is a product of $k$ negative disjoint $r'$-cycles;

		\item
		      or (i) $r$ is even and $r \mid (2m-2)$%
		      ---say $r = 2r'$ and $m-1 = kr'$ for integers $k,r' \geq 1$;
		      (ii) $\bm g$ is a product of $k$ negative disjoint $r'$-cycles,
		      plus possibly a negative transposition;\fn{
			      This seems to correct a misprint in~\cite[\S~6.11~(b)]{springer_1974_regular_elements_of_finite_reflection_groups}.
		      }~and (iii) the latter appears if and only if $k$ is \emph{even}.
		      (I.e.,
		      if $k$ is odd then $g$ fixes each element of an opposite pair.)
	\end{enumerate}

	The first case reduces to Prop.~\ref{prop:classification_type_BC}~(1.),
	in codimension $1$,
	because $-\zeta_r \neq 1$.
	(One has $r \neq 2$ here.)
	Moreover,
	the second case matches up with a subcase of (2.)~in loc.~cit.
	Analogously,
	if $r \geq 4$ then the third case still reduces to the one mentioned just above,
	because (again) $-\zeta_r \neq 1$.
	Finally,
	if $r = 2$ in the third case,
	whence $r' = 1$ and $m = k + 1$,
	then there are two mutually-exclusive situations:
	these are exactly as in the proof of Prop.~\ref{prop:generic_classification_type_D},
	but with parity \emph{swapped}.

	As for the complex reflection groups,
	which must be \emph{monomial} (by~\cite[Prop.~5.2]{lehrer_springer_1999_reflection_subquotients_of_unitary_reflection_groups},
	again),
	recall from~\cite[Lem.~6.1]{springer_1974_regular_elements_of_finite_reflection_groups} that the outermorphism $\dot\varphi$ has an associated tuple $\bm \epsilon = (\epsilon_1,\dc,\epsilon_m)$ of \emph{factors}.
	These are suitable roots $\epsilon_i \in \mb C^{\ts}$ of $1$,
	determining the action of $\dot\varphi$ on any tuple $(f_1,\dc,f_m)$ of algebraically-independent generators of $\mb C[\wt V_m]^W \simeq \mb C[f_1,\dc,f_m]$---%
	with $W \ceqq W_D(m)$.
	In the present situation one has $\epsilon_i = 1$ for $i \in \set{1,\dc,m-1}$ and $\epsilon_m = -1$,
	and in turn the degrees $d'_i$ of the `centralizer' $W(\dot\varphi,r) \ceqq Z_W(\bm g)$ are the degrees $d_i = \deg(f_i) \in \set{ 2,4,\dc,2(m - 1); m }$---%
	of $W$---%
	such that $\epsilon_i \zeta_r^{d_i} = 1 \in \mb C^{\ts}$ (cf.~Thm.~6.4.(iii) of op.~cit.,
	and again \S~6.11 therein).
	The main difference with the split case is that the equality $\zeta_r^{d_m} = \zeta_r^m = -1$ forces:
	(i) $r = 2r'$ even;
	and (ii) $m \equiv r' \pmod r$.
	Analogously to the proof of Prop.~\ref{prop:generic_classification_type_D},
	one then finds:
	\begin{enumerate}
		\item
		      if $m = kr + 1$ with $r$ odd,
		      the integers $d'_i \in \set{ 2r,4r,\dc,2kr }$;

		\item
		      if $m = kr'$ with $k$ odd,
		      and $r = 2r' \geq 4$ even,
		      the multiset of integers $d'_i \in \set{ 2r',4r',\dc,2(k-1)r';kr' }$ ($\zeta_r^m = (-1)^k = -1$ here);

		\item
		      and if $m = kr' + 1$ with $r = 2r'$ even,
		      then:
		      \begin{enumerate}
			      \item
			            if $r \geq 4$,
			            the integers $\set{ 2r',4r',\dc,2kr' }$ ($\zeta_r^m = (-1)^k \cdot e^{\pi \sqrt{-1} \slash r'} \neq -1$ here);

			      \item
			            and if $r = 2$,
			            whence again $m = k+1$,
			            then:
			            \begin{enumerate}
				            \item
				                  the integers $d'_i \in \set{ 2,4,\dc,2k;k+1 }$,
				                  if $k$ is even;

				            \item
				                  and the integers $d'_i \in \set{ 2,4,\dc,2k }$,
				                  if $k$ is odd.
			            \end{enumerate}
		      \end{enumerate}
	\end{enumerate}
	Reasoning as in the split setting,
	the first case (resp.,
	the second case;
	resp.,
	case (3.a)) yields $W(\dot\varphi,r) \simeq G(2r,1,k)$ (resp., $W(\dot\varphi,r) \simeq G(r,2,k)$;
	resp.,
	$W(\dot\varphi,r) \simeq G(r,1,k)$).
	Finally,
	case (3.b) coincides with case (4.b) at the end of the proof of Prop.~\ref{prop:generic_classification_type_D}
\end{proof}

\begin{rema}
	Contrary to Prop.~\ref{prop:generic_classification_type_D},
	when $k$ is odd there are \emph{no} constraints on the---%
	even---%
	ramification of the index-$2$ subgroups $G(r,2,k) \sse G(r,1,k)$.
	In conclusion,
	\emph{all} the monomial groups $G(r,2,k)$ have a modular interpretation in twisted meromorphic $2d$ gauge theory (cf.~Cor.~\ref{cor:modular_reflection_groups_twisted}).
\end{rema}

\subsubsection{}
\label{sec:triality}

It remains to consider the ambient rank $m = 4$.
Recall that $\Aut_D(4) \ceqq \Aut \bigl( \Phi_D(4) \bigr)$ extends the Weyl group by adding on an action of $\mf S_3$ which permutes the noncentral nodes of the Dynkin diagram (cf.~Rmk.~\ref{rmk:about_D_components};
they correspond to the simple roots $\theta_1,\theta_3,\theta_D$),
so that
\begin{equation}
	\abs{ \Aut_D(4) }
	= \abs{ \mf S_3 } \cdot \abs{ W_D(4) }
	= 6 \cdot 192.
\end{equation}
On the other hand,
the rank-$4$ $B\!C$-Weyl group has cardinality $\abs{ W_{B\!C}(4) } = 2 \cdot 192$,
and distinguishing one leg of the Dynkin diagram yields an index-$3$ embedding as a subgroup of $\Aut_D(4)$:
this breaks the $\mf S_3$-action down to a transposition/involution,
and the corresponding flip has been already considered in Prop.~\ref{prop:twisted_generic_classification_type_D}.

What is genuinely new is the order-$3$ twist which cyclically permutes the legs,
a.k.a.~the \emph{triality} element of $\Aut_D(4)$.
Therefore,
fix an outermorphism $\dot\varphi \in \Out \bigl( \Phi_D(4) \bigr) \simeq \mf S_3$ of order $\dot r = 3$,
and consider the nonsplit coset $\bigl( \mb C^4, \dot\varphi W_D(4) \bigr)$.
(This is denoted by $\prescript{3}{}G(2,2,4) = \prescript{3}{}D_4$ in~\cite[\S~3E]{broue_malle_michel_1999_towards_spetses_i},
noting that $W_D(4) = G(2,2,4)$ in Shephard--Todd's list.)
Replacing $\dot\varphi$ by $\dot\varphi^2 = \dot\varphi^{-1}$ yields the same cyclic subgroup of outermorphisms, 
and does not affect the classification of the following:

\begin{prop}
	\label{prop:triality_classification}

	Choose an integer $r \geq 1$,
    and consider a $\dot\varphi$-regular element $g \in W_D(4)$.
    If $\bm B_{D,\dot\varphi,r}(4) \sse \wt V_4$ is as in~\eqref{eq:regular_complement_D_double_twisted},
    then one of the three following (mutually-exclusive) situations occurs:
		      \begin{enumerate}
			      \item
			            one has $r \in \set{1,2}$,
                        and there is a $\mb C$-linear isomorphism of the complex hyperplane complement $\bm B_{D,\dot\varphi,r}(4)$ with the regular part of the \emph{imprimitive} group $G(6,6,2)$,
                        i.e.,
                        the exceptional Weyl group of type $G_2$,
                                                        which is the corresponding complex reflection group;
			      \item
			            or $r \in \set{3,6}$,
			            and there is a $\mb C$-linear isomorphism of $\bm B_{D,\dot\varphi,r} (4)$ with the regular part of the \emph{primitive} group $G_4$,
                        which is the corresponding complex reflection group;

                \item 
                    of $r = 12$,
                    and there is a $\mb C$-linear isomorphism $\bm B_{D,\dot\varphi,r}(4) \simeq \mb C^\ast$,
                    and the corresponding complex reflection group is the cyclic group $G(4,1,1) \simeq \mb Z \bs 4 \mb Z$.
                        \end{enumerate}
\end{prop}

\begin{proof}
    Let just $W \ceqq W_D(4)$.

    First,
    the degrees of $W$ are the integers $\bm d = (d_1,\dc,d_4) = (2,4,4,6)$.
    Second,
    the `reduced' \emph{polynomial order} $P \in \mb Z[T]$,
    of the reflection coset $\wt{\mb G} = \bigl(\mb C^4,\dot\varphi W_D(4) \bigr)$,
    is the monic polynomial $P = (T^2 - 1)(T^6 - 1)(T^8 + T^4 + 1)$.
    (We drop the factor $T^{12}$,
    for the number of reflecting hyperplanes,
    from the full polynomial order $\lvert \wt{\mb G} \rvert \in \mb Z[T]$;
    cf.~\cite[Prop.~3.5 +~\S~3E]{broue_malle_michel_1999_towards_spetses_i}.)
    It follows that the factors $\bm \epsilon = (\epsilon_1,\dc,\epsilon_4)$ of $\dot\varphi$ are 
    \begin{equation}
        \bm \epsilon
        = (1,\omega,\omega^2,1),
        \qquad \omega 
        \ceqq \zeta_3 \in \mb C^{\ts},
    \end{equation}
    because,
    indeed,
    this yields $P = \prod_{i = 1}^4 (T^{d_i} - \epsilon_i)$.
    (No other $4$-tuple would,
    up to swapping the degree-four factors.)

    Now (again) Springer theory~\cite[\S~6]{springer_1974_regular_elements_of_finite_reflection_groups} determines the degrees $d'_i \geq 1$ of the `centralizer' $W(\dot\varphi,r) \ceqq Z_W(\bm g)$ as the subset of degrees $d_i$ of $W$ such that $\epsilon_i \zeta_r^{d_i} = 1$ ($i \in \set{1,\dc,4}$).
    This condition implies that 
    \begin{equation}
        \bigl( \epsilon_i \zeta_r^{d_i} \bigr)^3
        = \zeta_r^{3d_i}
        = 1 \in \mb C^{\ts},
    \end{equation}
    whence one has $r \mid (3d_i)$ for some $d_i$ in the multiset $\set{ 2,4,6; 4}$,
    which forces $r \in \set{1,2,3,4,6,9,12,18}$.
    Now there are the following cases:
	\begin{enumerate}
		\item
		      one has $r \in \set{1,2}$,
              and so $\zeta_r^{d_i} = 1$ for $i \in \set{1,\dc,4}$,
		      and the two solutions are $d'_1 = 2$ and $d'_2 = 6$ (from $\epsilon_i = 1$);

		\item
            or one has $r \in \set{3,6}$,
            and the two solutions are $d'_1 = 4$ and $d'_2 = 6$ (from $\epsilon_i \in \set{\omega,\omega^2}$);

		\item
		    or one has $r = 12$,
            and the only solution is $d'_1 = 4$ (from $\epsilon_i = \omega^2$);

		\item
		    otherwise,
            there are \emph{no} solutions.
	\end{enumerate}
    In particular,
    the rank of $W(\dot\varphi,r)$ in these four cases is $2$,
    $2$,
    $1$,
    and $0$---%
    respectively.
    Moreover,
    in the nontrivial cases where $r \in \set{1,2,3,6,12}$,
    one can show that there exists indeed an element $\bm g \in \dot\varphi W$ with regular eigenvector of eigenvalue $1$,
    $-1$,
    $\omega$,
    $\zeta_6$,
    and $\zeta_{12}$,
    respectively.
    (E.g.,
    for $r = 1$ one can use $\dot\varphi$ itself,
    while $-\dot\varphi$ works for $r = 2$,
    and the twist of a Coxeter element when $r = 12$.)

    Now the combination of rank/degrees determines a \emph{unique} irreducible complex reflection group when $r \in \set{1,2,12}$,
    as per the statement.
    Finally,
    if $r \in \set{3,6}$ then both $G(6,3,2)$ and $G_4$ have degrees $(4,6)$ and rank $2$,
    and one can disambiguate by the reflection arrangement:
    in each case one can show that the $24$ roots of $\Phi_D(4)$ cut \emph{four} distinct lines on the eigenplane $\wt V_4(\bm g,\zeta_r) \sse \wt V_4$,
    while the monomial group $G(6,3,2)$ has \emph{eight} reflection lines/hyperplanes.
\end{proof}

\subsubsection{}

	In particular,
    Prop.~\ref{prop:triality_classification}~(2.)~produces a \emph{primitive} complex reflection group out of a \emph{classical} ambient Weyl group,
    and this cannot happen when the interior twist is trivial (cf.~Corr.~\ref{cor:modular_reflection_groups} +~\ref{cor:modular_reflection_groups_twisted}.)

\begin{rema}
	Incidentally,
	recall that $\Phi_D(4)$ embeds as the (nonlevi) root subsystem of the $24$ long/short roots in the exceptional type $F_4$,
	and so the Weyl group $W_F(4)$ of the latter acts (faithfully) by root-system automorphisms of the former.
	Since $\abs{ W_F(4) } = 1152$,
	one actually has $W_F(4) \simeq \Aut_D(4)$ (cf.~\cite[Prop.~3.13]{broue_malle_michel_1999_towards_spetses_i}).
\end{rema}

\subsection{Twisted ramified local \texorpdfstring{$G$}{G}-wild mapping class groups}
\label{sec:doubly_twisted_lwmgc}

We also extend the material of \S~\ref{sec:general_tglwmcgs}.
Consider again the \emph{standard} formal disc $\mc D \ceqq \Spec \mb C \llb w \rrb$,
and its closed point $0 \ceqq w \mb C \llb w \rrb$,
so that $(\mc D,0)$ is the standard model for---%
a ramified covering of---%
the formal germ of a smooth $1$-pointed projective curve over $\mb C$.

Then let $r,s \geq 1$ be integers,
and make the:

\begin{defi}[cf.~Def.~\ref{def:tglwmcgs}]
	\label{def:doubly_twisted_wmcg}

	Choose a pole-order-bounded $(\dot\varphi,r)$-Galois-closed irregular type $\wh Q \in \smash{\wh{\mc{IT}}}^{\mathsmaller\leq s}_{\dot\varphi,r}$,
	and let $\wh\Theta \ceqq \wh\Theta \bigl( \wh Q \bigr) \in \smash{\wh{\mc{IT}}}^{\mathsmaller\leq s}_{\dot\varphi,r} \bs W$ be the associated irregular class.
	Then:
	\begin{enumerate}
		\item
		      the \emph{pure} $\dot\varphi$-\emph{twisted} $r$-\emph{ramified} \emph{local WMCG of} $\bm{\mc D} = \bigl(\mc D,0,\wh Q\bigr)$ is the fundamental group
		      \begin{equation}
			      \Gamma_{\dot\varphi,r} \bigl( \wh Q \bigr)
			      = \Gamma_{\dot\varphi,r}^{\mathsmaller\leq s} \bigl( \wh Q \bigr)
			      \ceqq \pi_1 \bigl( \bm B_{\dot\varphi,r}( \wh Q),\wh Q \bigr);
		      \end{equation}

		\item
		      and the (\emph{full/nonpure}) $\dot\varphi$-\emph{twisted} $r$-\emph{ramified} \emph{local WMCG of} $\ol{\bm{\mc D}} = \bigl(\mc D,0,\wh\Theta\bigr)$ is the fundamental group
		      \begin{equation}
			      \ol \Gamma_{\dot\varphi,r} \bigl( \wh\Theta \bigr)
			      = \ol \Gamma_{\dot\varphi,r}^{\mathsmaller\leq s} \bigl( \wh \Theta \bigr)
			      \ceqq \pi_1 \bigl( \bm B_{\dot\varphi,r}( \wh\Theta),\wh\Theta \bigr).
		      \end{equation}
	\end{enumerate}
\end{defi}

\begin{theo}[cf.~Thm.~\ref{thm:wild_mapping_class_groups}]
	\label{thm:doubly_twisted_wmcg}

	Write $\wh Q = \sum_{i = 1}^s A_iw^{-i}$,
	and let $\ul{\wt{\mb L}} \ceqq \ul{\wt{\mb L}} \bigl( \wh Q \bigr)$ be the underlying Levi-subcoset filtration.
	Define also the group $Z_{W,\bm \phi}(\bm g) \sse N_W(\mf t_{\bm \phi}) \bs W_{\mf t_{\phi_1}}$ as in~Thm.-Def.~\ref{thm:full_general_doubly_twisted_deformation_space} (for $\bm g \in \wt{\mb L}_{\dot\varphi,r} \bigl( \wh Q \bigr)$),
	and identify it with the $\dot\varphi$-twisted $r$-ramified WWG $W_{\dot\varphi,r} \bigl( \wh Q \bigr) = W_{\dot\varphi,r} \bigl( \ul{\wt{\mb L}} \bigr)$ (cf.~\S~\ref{sec:twisted_modular_groups_and_cabled_weyl_groups}).
	Then:
	\begin{enumerate}
		\item
		      there is a \emph{group isomorphism}
		      \begin{equation}
			      \label{eq:most_general_pure_group_factorization}
			      \Gamma_{\dot\varphi,r} \bigl( \wh Q \bigr)
			      \simeq \prod_{i = 1}^s \Gamma_{\dot\varphi,r} \bigl( \wh Q,i \bigr),
			      \qquad \Gamma_{\dot\varphi,r} \bigl( \wh Q,i \bigr)
			      \ceqq \pi_1 \bigl( \bm B_{\dot\varphi,r}\bigl(\ul{\wt{\mb L}},i\bigr),A_i \bigr),
		      \end{equation}
		      in the notation of~\eqref{eq:intrinsic_pure_deformation_factors_twisted};

		\item
		      and there is a (typically nonsplit) short \emph{exact group sequence}
		      \begin{equation}
			      \label{eq:most_general_exact_group_sequence}
			      1 \lra \Gamma_{\dot\varphi,r} \bigl( \wh Q \bigr) \lra \ol \Gamma_{\dot\varphi,r} \bigl( \wh\Theta \bigr) \lra W_{\dot\varphi,r} \bigl( \wh Q \bigr) \lra 1.
		      \end{equation}
	\end{enumerate}
\end{theo}

\begin{proof}
	This follows from all that was established in this last section (\S~\ref{sec:interior_twist}).
\end{proof}

\section{Outlook}
\label{sec:outlook}

\subsection{Future directions (for everyone interested!)}

\subsubsection{}

Rmkk.~\ref{rmk:forest_bc} +~\ref{rmk:forest_d} are about skeleta and admissible deformations of wild curves/Riemann surfaces with classical structure group.
One could now study the corresponding moduli stacks $\mc{WM}_{g,\bm F}$,
e.g.,
determining the exact Deligne--Mumford condition (cf.~\cite[Prop.~5.1.5]{doucot_rembado_tamiozzo_moduli_spaces_of_untwisted_wild_riemann_surfaces}).
More generally,
it is possible to define skeleta for any connected reductive $G$ building on \S~\ref{sec:modular_groups_are_twisted_cabled_weyl} (or even the interior-twisted case of~\S~\ref{sec:twisted_modular_groups_and_cabled_weyl_groups}),
cf.~\S~\ref{sec:topological_skeleta};
then one can ask the same---%
more general---%
questions.

\begin{rema}
	\label{rmk:stacks}

	In this viewpoint,
	this text deals with the fibres of a natural forgetful stack fibration onto $\mc M_{g,n}$ (or the unordered-point version $\mc M_{g,\set{n}} = \mc M_{g,n} \bs \mf S_n$).
	Again,
	the fundamental group of the total space is the \emph{global} WMCG,
	extending the usual $2d$ mapping class group by the `local' ones defined in this text:
	cf.~\cite{doucot_rembado_tamiozzo_moduli_spaces_of_untwisted_wild_riemann_surfaces} for a precise statement in the pure untwisted case,
	and note that the complex dimension of the new stacks leads to a generalization of Riemann's count~\cite{riemann_1857_theorie_der_abelschen_functionen} for the number of moduli of complex/projective-curve structures on a genus-$g$ closed oriented $C^\infty$/topological surface.
\end{rema}

\subsubsection{}

As already mentioned,
we wish to study in more detail the Poisson/symplectic actions of WMCGs on the wild character varieties (cf.~\cite[Exmp.~9.1]{doucot_rembado_tamiozzo_2026_local_wild_mapping_class_groups_and_cabled_braids}),
possibly twisted,
relating in particular with the approach of~\cite{paul_ramis_2015_dynamics_on_wild_character_varieties,
	klimes_2024_wild_monodromy_of_the_fifth_painleve_equation_and_its_action_on_wild_character_variety_an_approach_of_confluence}.

\subsubsection{}

In the setup of \S~\ref{sec:twisted_loop_algebra},
one might relate to~\cite{vinberg_1976_the_weyl_group_of_a_graded_lie_algebra},
considering the cyclic grading of $\mf g$ determined by the eigenspace decomposition of a finite-order automorphism.
Moreover,
considering coadjoint orbits in \emph{twisted} truncated-current Lie algebras might yield a generalization of~\cite{calaque_felder_rembado_wentworth_wild_orbits_and_generalised_singularity_modules_stratifications_and_quantisation,chaffe_rembado_yamakawa_wild_genus_zero_quantum_de_rham_spaces}.

\subsubsection{}

The type-$A$ fission tree of an irregular class may be viewed as a convenient modification/truncation of the Eggers--Wall tree~\cite{wall_2004_singular_points_of_plane_curves} of the corresponding curve singularity,
involving the irregular-singular Higgs fields obtained from the connection via \emph{wild nonabelian Hodge theory} (on curves,
cf.~\cite{sabbah_1999_harmonic_metrics_and_connections_with_irregular_singularities,biquard_boalch_2004_wild_non_abelian_hodge_theory_on_curves,huang_sun_2023_meromorphic_parahoric_higgs_torsors_and_filtered_stokes_g_local_systems_on_curves,huang_sun_filtered_stokes_g_local_systems_in_nonabelian_hodge_theory_on_curves}).
The new $G$-fission trees might thus lead to variations of $G$-versions of Eggers--Wall trees.

Incidentally,
we also plan to consider the transitivity question of Rmk.~\ref{rmk:transitivity_question} a bit further,
relating to the newly defined $G$-fission trees of \S\S~\ref{sec:BC_trees}--\ref{sec:D_trees}.
More broadly,
one might explicitly define the \emph{truncated} twisted root-valuation stratifications which arise in meromorphic reductive $2d$ gauge theory,
akin to~\cite{goresky_kottwitz_macpherson_2009_codimensions_of_root_valuation_strata,sabitova_2012_a_simplification_of_root_valuation_data_for_classical_groups},
generalizing the viewpoint of~\cite{calaque_felder_rembado_wentworth_wild_orbits_and_generalised_singularity_modules_stratifications_and_quantisation};
particularly in the full/nonpure case (which sheds some light on how the quotient strata fit together),
and including the genuinely new interior twists.

\subsubsection{}

In the monolevel case,
a uniform parameterization of modular subquotients of $W$ can be obtained by putting forward the eigenspace poset of Rmk.~\ref{rmk:about_eigenspace_poset},
relying on~\cite{miller_2015_eigenspace_arrangements_of_reflection_groups,koonin_topology_of_eigenspace_posets_for_unitary_reflection_groups},
in a particular case of \S\S~\ref{sec:modular_groups_are_twisted_cabled_weyl}--\ref{sec:modular_reparameterization}.

One should also relate this to Broué--Malle--Michel's `cyclotomic' theory for reflection cosets~\cite{broue_malle_1992_theoremes_de_sylow_generiques_pour_les_groupes_reductifs_sur_les_corps_finis,
	broue_malle_1993_zyklotomische_heckealgebren,
	broue_malle_michel_1993_generic_blocks_of_finite_reductive_groups,
	malle_1998_spetses,
	broue_malle_michel_1999_towards_spetses_i,
	malle_2000_on_the_generic_degrees_of_cyclotomic_algebras,
	broue_malle_michel_2014_split_spetses_for_primitive_reflection_groups},
beyond the Sylow case,
after establishing the independence of the notion of $r$-modularity from the choice of a primitive $r$-th root of $1$ (cf.~Rmk.~\ref{rmk:about_rational_forms}).

We plan to spell this out elsewhere,
and perhaps extend it into the multilevel case.

\subsubsection{}

The recent~\cite{sagan_swanson_2026_stirling_numbers_for_complex_reflection_groups} introduces Stirling numbers for complex reflection groups,
generalizing the numbers of classical types $\bullet \in \set{A,B,C,D}$ (cf.~\S~\ref{sec:lehrer_springer_theory}).
This seems to be very helpful to count the $r$-admissible-deformation families of irregular types---%
starting from the quasigeneric/monolevel cases.
(And perhaps also to obtain \emph{multicoloured} operadic interpretations in the twisted/ramified setting.)

Again,
the full/nonpure count would generalize~\cite{boalch_2025_counting_the_fission_trees_and_nonabelian_hodge_graphs_untwisted_case} in the twisted case,
also involving different structure groups.

\subsubsection{}

Finally,
the finitely many examples of pure/full $r$-admissible-deformation spaces arising from a simple Lie algebra $\mf g$ of \emph{exceptional} type $\bullet \in \set{E,F,G}$ should be separately described elsewhere,
including the type-$E_6$ Dynkin-diagram flip in the interior-twisted case (cf.~also Question~\ref{question:k_pi_1} and \S~\ref{sec:exceptional_types}).

\section*{Acknowledgements}

Besides P.~Boalch and M.~Tamiozzo,
who have collaborated on this project since its very inception,
we also thank---%
in alphabetical order---%
the following mathematicians for helpful discussions and for listening to/answering to some of our questions:
C.~Bonnafé,
C.~Dupont,
R.~Fioresi,
J.~Michel,
D.~Schwein,
and V.~Toledano Laredo.

\appendix

\section{Some background notion/notation}
\label{sec:background}

\subsection{Complex reflection groups}
\label{sec:background_complex_reflections}

Please refer to~\cite{orlik_terao_1992_arrangements_of_hyperplanes},
or to~\cite{broue_2001_reflection_groups_braid_groups_hecke_algebras_finite_reductive_groups,broue_2008_introduction_to_complex_reflection_groups_and_their_braid_groups,broue_2010_introduction_to_complex_reflection_groups_and_their_braid_groups},
or to~\cite{lehrer_taylor_2009_unitary_reflection_groups},
for more details about this standard material.

\subsubsection{}

Let $V$ be a finite-dimensional complex vector space.
A \emph{reflection in} $V$ is a nontrivial finite-order element $g \in \GL_{\mb C}(V)$ which acts as the identity on a%
---\emph{reflecting}---%
linear hyperplane $H \sse V$.
(This is a.k.a.~a `complex reflection',
or a `pseudoreflection';
or a `unitary reflection',
when using positive-definite Hermitian forms.)
A subgroup $W' \sse \GL_{\mb C}(V)$ generated by reflections is a \emph{complex reflection group}:
in this text we always assume that $W'$ is \emph{finite}.

A complex reflection group $(V,W')$ is \emph{essential} if $0 \in V$ is the only $W'$-invariant vector,
and it is \emph{irreducible} if there are no nonzero proper $W'$-invariant subspaces.
The quotient of $V$ modulo the \emph{fixed-point space of} $W'$,
viz.,
the subspace
\begin{equation}
	V^{W'} \ceqq
	\Set{ A \in V | g(A) = A \text{ for } g \in W' },
\end{equation}
carries an essential reflection representation of $W'$,
which in turn splits into finitely many irreducible components.

The \emph{rank of} $(V,W')$ is $m \ceqq \dim_{\mb C}(V)$.

\subsubsection{}
\label{sec:background_flats}

The set of the hyperplanes of $V$ which are fixed by some reflection contained in $W'$ is the \emph{reflection arrangement of} $(V,W')$,
typically denoted by $\mc A(W') = \mc A(V,W') \sse \mb P(V^{\dual})$---%
in the usual identification $H = \ker(\alpha) \mt \mb C \cdot \alpha$,
for $\alpha \in V^{\dual} \sm (0)$.
The complement of the union of the reflecting hyperplanes is the $W'$-\emph{regular part of} $V$,
denoted by
\begin{equation}
	V_{\reg} = V_{\reg,W'}
	\ceqq V \, \bigsm \, \bigcup_{\mc A(W')} \, H \sse V.
\end{equation}
A vector $A \in V$ is \emph{regular} if it lies in $V_{\reg}$,
and this happens if and only if it has trivial \emph{stabilizer} in $W'$.

\subsubsection{}
\label{sec:r_regularity}

Following Springer~\cite{springer_1974_regular_elements_of_finite_reflection_groups},
an element $g \in W'$ is said to be:
(i) $\zeta$-\emph{regular},
if it admits a regular eigenvector of eigenvalue $\zeta \in \mb C^{\ts}$---%
which must be a root of $1$;
and (ii) just \emph{regular},
if it is $\zeta$-regular for some number $\zeta$.
Accordingly,
a $\zeta$-eigenspace $E = V(g,\zeta) \ceqq \ker(g - \zeta \cdot \Id_V)$ is \emph{regular} if $E \cap V_{\reg} \neq \vn$.
Incidentally,
the set of $\zeta$-eigenspaces of elements of $W'$ is denoted by
\begin{equation}
    \on E(W',\zeta)
    \ceqq \Set{ V(g,\zeta) |g \in W' },
\end{equation}
and in this text it is partially ordered by \emph{inclusion}.
(So the maximal eigenspaces are the inclusion-maximal ones;
this poset is well-studied,
cf.,
e.g.,~\cite{koonin_2014_topology_of_eigenspace_posets_for_imprimitive_reflection_groups,koonin_topology_of_eigenspace_posets_for_unitary_reflection_groups,miller_2015_eigenspace_arrangements_of_reflection_groups}.)
Then the regular eigenspaces are maximal,
but not conversely.
Nonetheless,
the maximal eigenspaces lie in a single $W'$-orbit,
and are precisely those of \emph{maximal dimension};
this dimension equals the number $a(d) = a(W',d) \geq 0$ of degrees of $W'$ which are multiples of $d$ (cf.~\S~\ref{sec:more_complex_refl_group_background}),
where $d \geq 1$ is the order of $\zeta$ as a root of $1$.

The \emph{regular numbers of} $W'$ are the orders $d \geq 2$ of \emph{nontrivial} regular elements.

Moreover,
if $\zeta$ is a primitive $d$-th root of $1$,
then there is a \emph{unique} $W'$-conjugacy class of $\zeta$-regular elements.
They are all of order $d$,
and so we also say that they are---%
in particular---%
$d$-\emph{regular}.

\begin{rema}
    \label{rmk:r_regular_for_rationals}

	Beware that $d$-regular elements need \emph{not} be conjugate if they correspond to different primitive $d$-th roots of $1$.
	However,
	if $W'$ is \emph{rational} then this does hold (cf.~Rmk.~\ref{rmk:about_rational_forms} +~\S~\ref{sec:complex_real_rational_reflection_groups}).
\end{rema}

\subsubsection{}

The \emph{intersection lattice of} $(V,W')$ is the (finite) set $\mc F(W') = \mc F(V,W')$ of vector subspaces of $V$ obtained by intersecting any number of reflecting hyperplanes of $W'$;
its elements are the \emph{flats of} $(V,W')$.
We order $\mc F(W')$ by \emph{inclusion},
and so:
(i) the \emph{inf/meet of} $F,F' \in \mc F(W')$ is $F \cap F'$;
and (ii) the \emph{sup/join of} $F,F'$ is the smallest flat containing $F + F'$,
i.e.,
the intersection of all flats $F'' \in \mc F(W')$ such that $F + F' \sse F''$.
The latter is a.k.a.~the \emph{flat-closure of} $F + F'$ (relative to $W'$),
denoted by $\ol{F + F'}^{\mathsmaller{W'}} \sse V$.

More generally,
if $V' \sse V$ is a vector subspace and $F \in \mc F(W')$ a flat,
we say that $V'$ is \emph{flat-dense in} $F$ if one has $F = \ol{V'}^{\mathsmaller{W'}}$.
The flat-closure operator is $W'$-equivariant,
i.e.,
one has
\begin{equation}
	g \bigl(\ol{V'}^{\mathsmaller{W'}}\bigr)
	= \ol{g(V')}^{\mathsmaller{W'}} \sse V,
	\qquad g \in W'.
\end{equation}

\subsubsection{}

Let again $V' \sse V$ be a vector subspace.
The \emph{parabolic subgroup of} $(V,W')$ \emph{corresponding to} $V'$ is the \emph{pointwise-stabilizer of} $V'$,
i.e.,
the subgroup
\begin{equation}
	\label{eq:abstract_parabolic_subgroup}
	W'_{V'}
	\ceqq \Set{ g \in W' | \eval[1]g_{V'} = \Id_{V'} }.
\end{equation}
It is still a reflection group in $V$,
generated by the reflections of $W'$ about the reflecting hyperplanes $H \in \mc A(W')$ such that $V' \sse H$,
as per Steinberg's theorem~\cite{steinberg_1964_differential_equations_invariant_under_finite_reflection_groups} (cf.~\cite{lehrer_2004_a_new_proof_of_steinberg_s_fixed_point_theorem}).
The \emph{parabolic subgroups of} $(V,W')$ are the subgroups of the form~\eqref{eq:abstract_parabolic_subgroup},
as $V'$ ranges amongst the vector subspaces of $V$.
The group $W'_{V'}$ is normalized by the \emph{setwise-stabilizer}
\begin{equation}
	\label{eq:abstract_setwise_stabilizer}
	N_{W'}(V')
	\ceqq \Set{ g \in W' | g(V') \sse V' }.
\end{equation}
The quotient $W'(V') \ceqq N_{W'}(V') \bs W'_{V'}$ can---%
and will---%
be identified with the group of restrictions $\eval[1]g_{V'} \in \GL_{\mb C}(V')$ of elements $g \in N_{W'}(V')$.

The main relations with flat-closures are that:
(i) the pointwise-stabilizers of $V'$ and $\ol{V'}^{\mathsmaller W'}$ coincide;
(ii) conversely,
the subspace $\ol{V'}^{\mathsmaller W'} \sse V$ is the fixed-point space of $W'_{V'} \sse W'$;
and (iii) the \emph{normalizer of} $W'_{V'}$---%
in $W'$---%
is the setwise-stabilizer of $\ol{V'}^{\mathsmaller W'}$.
In our convention,
this establishes an \emph{antitone} bijection between the set of parabolic subgroups of $W'$ and the intersection lattice $\mc F(W')$ (i.e.,
an antitone \emph{Galois connection/correspondence},
cf.~\cite{ore_1944_galois_connexions,erne_koslowski_melton_strecker_1993_a_primer_on_galois_connections}).
Moreover,
the flats of $W'$ also coincide with the fixed-point spaces of elements of $W'$---%
but in nonbijective fashion.

\subsubsection{}
\label{sec:complex_real_rational_reflection_groups}

A complex reflection group $(V,W')$ is a \emph{real reflection group} if $V$ admits a $W'$-invariant $\mb R$-form,
i.e.,
if there exists a vector subspace $V_{\mb R} \sse V$ over $\mb R$ such that one has
\begin{equation}
	\label{eq:real_form}
	V_{\mb R} \ots_{\mb R} \mb C
	\simeq V, \quad g(V_{\mb R}) \sse V_{\mb R},
	\qquad g \in W'.
\end{equation}
Then $W' \hra \GL_{\mb R}(V_{\mb R})$ is generated by reflections of order $2$,
and it is a---%
finite---%
Coxeter group.
(There are \emph{nonreal} complex reflection groups generated by elements of order $2$.)

A complex reflection group $(V,W')$ is a \emph{rational reflection group} if $V$ admits a $W'$-invariant $\mb Q$-form $V_{\mb Q} \sse V$,
analogously to~\eqref{eq:real_form}.
In particular it is real,
and moreover $W'$ also preserves a lattice (i.e.,
it is \emph{crystallographic}).
We typically denote the underlying group by $W$,
and there now exists a split reductive Lie algebra $(\mf g,\mf t)$ over $\mb C$ such that:
(i) one can take $V \ceqq \mf t$;
and (ii) the group $W = W(\mf g,\mf t)$ is the \emph{Weyl group} generated by the order-2 reflections $g_\alpha \in \GL_{\mb C}(\mf t)$ about the kernels of the roots $\alpha \in \Phi = \Phi(\mf g,\mf t) \sse V^{\dual}$,
which negate the corresponding coroot $\alpha^{\dual} \in \Phi^{\dual} \sse V$.
It follows that $\mf g$ is semisimple (resp.,
simple) if and only if $(V,W)$ is essential (resp.,
irreducible).

There is a corresponding contragredient/inverse-transpose (left) action on $\mf t^{\dual}$,
i.e.,
\begin{equation}
	\Braket{ g(\lambda),A }
	\ceqq \Braket{ \prescript{t}{}g^{-1}(\lambda),A }
	= \Braket{ \lambda,g^{-1}(A) },
	\qquad \lambda \in \mf t^{\dual},
	\quad A \in \mf t,
	\quad g \in W,
\end{equation}
where $\braket{ \cdot,\cdot } \cl \mf t^{\dual} \ots_{\mb C} \mf t \to \mb C$ is the canonical pairing.
Therefore,
we identify the Weyl groups of $\Phi$ and $\Phi^{\dual}$,
and we regard a single group as operating on two mutually-dual vector spaces (cf.~\cite[Chp.~VI, \S~1]{bourbaki_1968_groupes_et_algebres_de_lie_chapitres_iv_v_vi});
e.g.,
we will identify the irreducible Weyl groups of types $B$ and $C$ (cf.~\S~\ref{sec:background_type_BCD}).

Moreover,
the \emph{annihilator/kernel maps} are now $W$-equivariant,
viz.,
one has
\begin{equation}
	h (S^\perp) = \bigl( h(S) \bigr)^\perp \sse \mf t^{\dual},
	\quad h\bigl( \ker(T) \bigr)
	= \ker \bigl( h(T) \bigr) \sse \mf t,
	\qquad S \sse \mf t,
	\quad T \sse \mf t^{\dual},
	\quad h \in W.
\end{equation}

\subsubsection{}
\label{sec:levis_and_flats}

In the case of Weyl groups one has
\begin{equation}
	\mc A(W)
	= \Set{ H_\alpha = \ker(\alpha) | \alpha \in \Phi },
\end{equation}
and the Galois connection between flats/parabolic subgroups can be rephrased in terms of \emph{Levi subsystems} of roots.

Precisely,
recall that a subset $\phi \sse \Phi$ is a Levi subsystem if it is closed in $\Phi$ under $\mb Q$-linear combinations of its elements (cf.~\S~\ref{sec:pure_nongeneric_case}):
let $\mc L(\Phi) \sse 2^\Phi$ denote their poset,
ordered by \emph{anti-inclusion}.
This is actually a \emph{ranked lattice} (cf.,
e.g.,~\cite{calaque_felder_rembado_wentworth_wild_orbits_and_generalised_singularity_modules_stratifications_and_quantisation}).
Namely,
the \emph{rank of} $\phi \in \mc L(\Phi)$ is given by $\rho(\phi) = \dim_{\mb C}(\mf t_\phi)$,
where $\mf t_\phi \ceqq \ker(\phi)$ (this is the dimension of the Lie-centre of the corresponding reductive Lie subalgebra $\mf l_\phi \sse \mf g$,
cf.~\eqref{eq:levi_subalgebra}).
As for the lattice structure,
let
\begin{equation}
	\sup \set{\phi,\phi'} \ceqq \phi \cap \phi',
	\quad \inf \set{ \phi,\phi' } \ceqq \spann_{\mb Q}(\phi \cup \phi') \cap \Phi,
	\qquad \phi,\phi' \in \mc L(\Phi).
\end{equation}
Then one can prove the following:

\begin{lemm}
    \label{lem:levis_and_flats}
	\leavevmode
	\begin{enumerate}
		\item
		      Mapping $\phi \mt \mf t_\phi$ yields a \emph{lattice isomorphism} $\mc L(\Phi) \lxra{\simeq} \mc F(W)$.

		\item
		      And the \emph{maximal proper subflats of} $\mf t_\phi$ are the traces---%
		      on $\mf t_\phi$---%
		      of the `complementary' root hyperplanes,
		      i.e.
		      \begin{equation}
			      \Set{ H_\alpha(\phi) = \mf t_\phi \cap H_\alpha | \alpha \in \Phi \sm \phi } \sse \mb P \bigl( \mf t_\phi^{\dual} \bigr).
		      \end{equation}
		      (The fact that $\phi$ is Levi is \emph{equivalent} to having $\mf t_\phi \nsse H_\alpha$ for $\alpha \not\in \phi$.)
	\end{enumerate}
\end{lemm}

\begin{proof}[Proof omitted]
\end{proof}

\subsubsection{}
\label{sec:root_automorphisms}

In the setting of \S~\ref{sec:complex_real_rational_reflection_groups},
denote by $\mf g' \ceqq [\mf g,\mf g] \sse \mf g$ the (semisimple) derived Lie subalgebra.
We identify the Weyl groups of $(\mf g,\mf t)$ and $(\mf g',\mf t')$ as \emph{abstract} groups,
i.e.,
forgetting the ambient reflection representation.

Precisely,
consider the $\mb C$-linear splitting $\mf t = \mf Z \ops \mf t'$,
where $\mf t' \ceqq \mf g' \cap \mf t$ (which remains a Cartan subalgebra) and $\mf Z \sse \mf g$ is the Lie-centre (so that $\mf g = \mf Z \ts \mf g'$ as Lie algebras).
Then:
(i) the Weyl group of $(\mf g,\mf t)$ acts trivially on $\mf Z$;
and (ii) it preserves $\mf t'$;
and (iii) restricting its action onto $\mf t'$ maps it isomorphically onto the Weyl group of $(\mf g',\mf t')$.
So we denote both groups abusively by $W$.

Now note that the root systems $\Phi$ and $\Phi' = \Phi(\mf g',\mf t')$ consist of \emph{one and the same} finite set in the vector subspace $(\mf t')^{\dual} \simeq \mf Z^\perp \sse \mf t^{\dual}$.
The roots span $(\mf t')^{\dual}$,
but they only span $\mf t^{\dual}$ in the semisimple case---%
where $(\mf g,\mf t) = (\mf g',\mf t')$.
Then we will consider the groups of \emph{root-system automorphisms},
viz.,
\begin{equation}
	\Aut(\Phi')
	\ceqq \Set{ \bm g' \in \GL_{\mb C} \bigl( (\mf t')^{\dual} \bigr) | \bm g'(\Phi') \sse \Phi' },
\end{equation}
and
\begin{equation}
	\label{eq:root_autos_big}
	\Aut(\Phi)
	\ceqq \Set{ \bm g \in \GL_{\mb C} \bigl( \mf t^{\dual} \bigr) | \bm g(\Phi) \sse \Phi \text{ and } \prescript{t}{}{\bm g}^{-1} (\Phi^{\dual}) \sse \Phi^{\dual} }.
\end{equation}
By the given conditions,
any element of~\eqref{eq:root_autos_big} preserves the subspace
\begin{equation}
	(\mf t')^{\dual} = \spann_{\mb C}(\Phi) = \spann_{\mb C}(\Phi'),
\end{equation}
as well as the complement $\mf Z^{\dual} \simeq (\mf t')^\perp \sse \mf t^{\dual}$;
therefore,
its elements are `block-diagonal' with respect to the splitting $\mf t^{\dual} = \mf Z^{\dual} \ops (\mf t')^{\dual}$,
and there is a \emph{group isomorphism}
\begin{equation}
	\Aut(\Phi) \simeq \GL_{\mb C} \bigl( \mf Z^{\dual} \bigr) \ts \Aut(\Phi').
\end{equation}
The caveat in the nonsemisimple case is that preserving the subspace $\mf Z^{\dual} \sse \mf t^{\dual}$ does \emph{not} follow by just asking that $\Phi$ be stabilized (there could be `off-diagonal' terms).

Analogous definitions can be made for the dual root systems $(\Phi')^{\dual} = \Phi^{\dual} \sse \mf t'$,
and one finds \emph{canonical} group isomorphisms $\Aut(\Phi') \simeq \Aut \bigl( (\Phi')^{\dual} \bigr)$ and $\Aut(\Phi) \simeq \Aut(\Phi^{\dual})$---%
extending~\S~\ref{sec:complex_real_rational_reflection_groups}.
Moreover,
these groups normalize the Weyl groups,
and the corresponding quotients are the groups of \emph{outer automorphisms} of the root systems,
i.e.
\begin{equation}
	\label{eq:outer_root_autos}
	\Out(\Phi')
	\ceqq \Aut(\Phi') \bs W,
	\qquad \Out(\Phi)
	\ceqq \Aut(\Phi) \bs W
	\simeq \GL_{\mb C} \bigl( \mf Z^{\dual} \bigr) \ts \Out(\Phi').
\end{equation}
(We shall also tacitly identity $\Out(\Phi) \simeq \Out (\Phi^{\dual})$ and $\Out(\Phi') \simeq \Out \bigl( (\Phi')^{\dual} \bigr)$.)

\begin{rema}
	\label{rmk:root_datum}

	Alternatively,
	the group~\eqref{eq:root_autos_big} is the $\mb C$-linear `relaxation' of the automorphism group of the \emph{root datum} $\mc R = (X,\Phi,X^{\dual},\Phi^{\dual})$ of a (split) connected reductive algebraic group $G$ over $\mb C$,
	with maximal torus $T \sse G$,
	such that $\mf t = \Lie(T)$ and $\mf g = \Lie(G)$ (so $X$ and $X^{\dual}$ are the character/cocharacter lattices of $T$,
	respectively;
	cf.~\cite[Exp.~XXI]{demazure_grothendieck_1970_sga_3_iii}).
	Precisely,
	one has $\Aut(\mc R) \simeq \Aut(\Phi) \cap \GL_{\mb Z}(X)$,
	and relaxing the integrality condition forgets the lattice---%
	and the chosen integration.

	Note that $\Aut(\mc R)$ is \emph{discrete},
	while $\Aut(\Phi)$ has a positive-dimensional `central' factor in the nonsemisimple case;
	so this really loses some information and the group~\eqref{eq:root_autos_big} does \emph{not} depend on the choice of $(G,T)$.
\end{rema}

\begin{rema}
	\label{rmk:cartan_integers}

	If $\alpha,\beta \in \Phi'$,
	then any element of $\Aut(\Phi')$ maps the $\alpha$-string
	\begin{equation}
		S_\beta(\alpha)
		\ceqq \bigl( \beta + \mb Z \cdot \alpha \bigr) \cap \Phi',
	\end{equation}
	through $\beta$,
	to another such string.
	Moreover,
	recall that the \emph{Cartan number} $n(\beta,\alpha) \in \set{0,\pm 1,\pm 2,\pm 3}$ equals $q - p$,
	where $p,q \geq 0$ are the---%
	well-defined---%
	integers such that $S_\beta(\alpha) = \set{ \beta - q \alpha,\dc,\beta + p \alpha }$ (cf.~\cite[Chp.~VI, Prop.~9]{bourbaki_1968_groupes_et_algebres_de_lie_chapitres_iv_v_vi}).
	It follows that $\Aut(\Phi')$ preserves the Cartan integers.
	Conversely,
	any permutation of $\Phi'$ which preserves the Cartan integers extends uniquely to a root-system automorphism of $\Phi'$.
	(This provides an equivalent definition of $\Aut(\Phi')$.)

	In the nonsemisimple case,
	the difference is that the Cartan-admissible permutations of $\Phi$ extend \emph{nonuniquely} to $\mb C$-linear maps on $\mf t^{\dual}$.
\end{rema}

\subsubsection{}
\label{sec:more_complex_refl_group_background}

Let again $(V,W')$ be a complex reflection group,
and denote by $\mb C[V] \ceqq \Sym (V^{\dual})$ the $\mb C$-algebra of regular/polynomial functions on $V$%
---viewed as an affine space over $\mb C$.
It is well-known~\cite{shephard_todd_1954_finite_unitary_reflection_groups,
	chevalley_1955_invariants_of_finite_groups_generated_by_reflections,
	bourbaki_1968_groupes_et_algebres_de_lie_chapitres_iv_v_vi} that the invariant subring $\mb C[V]^{W'} \sse \mb C[V]$ is a polynomial ring generated by algebraically-independent homogeneous functions $f_1,\dc,f_m \in \mb C[V]^{W'}$,
for some integer $m \geq 1$,
whose degrees $d_1,\dc,d_m \geq 1$%
---ordered in increasing fashion---%
are intrinsically determined by $(V,W')$:
these are the \emph{degrees of} $(V,W')$,
and one has $\abs{W'} = \prod_{i = 1}^m d_i$.
The number of generators equals the rank of $(V,W')$,
and the degrees can repeat (so they yield a \emph{multiset}).

Note that the degree $1$ appears as many times as the dimension of the fix space $V^{W'}$,
and so $1$ is \emph{not} a degree if $(V,W')$ is essential.

Moreover,
if $(V,W')$ is irreducible:
(i) the number of reflections contained in $W'$ equals the sum of the degrees minus the rank;
(ii) there exists a $W'$-invariant $\mb R$-form of $V$ if and only if $2$ is a degree;
(iii) the centre of $(V,W')$ is cyclic,
of order equal to the GCD of the degrees;
and (iv) the number $\abs{\mc A(W')}$ of reflecting hyperplanes of $(V,W')$ equals the sum of the \emph{codegrees} $d_1^*,\dc,d_k^* \in \mb Z_{\geq 0}$ plus the rank.
In turn,
the latter are a shift of the \emph{coexponents}~\cite[\S~1.A]{broue_malle_rouquier_1998_complex_reflection_groups_braid_groups_hecke_algebras} (cf.~\cite{orlik_solomon_1980_unitary_reflection_groups_and_cohomology}),
and in the setting of regular Springer theory these are determined---%
from those of $W'$---%
in~\cite[Thm.~2.8]{denef_loeser_1995_regular_elements_and_monodromy_of_discriminants_of_finite_reflection_groups}.

\subsection{Reflection cosets}
\label{sec:reflection_cosets}

Let again $V$ be a finite-dimensional complex vector space,
and refer to~\cite[Def.~3.1 + Lem.~3.2 +~\S~3.A]{broue_malle_michel_1999_towards_spetses_i} and~\cite[Deff.~1.6 + 1.7]{broue_malle_michel_2014_split_spetses_for_primitive_reflection_groups} for the following quite versatile terminology.
(We write cosets on the \emph{left};
here they coincide with right ones.)

\subsubsection{}

A \emph{complex reflection coset} is a pair $\wt{\mb G} = (V,\dot\varphi W')$,
where:
(i) $(V,W')$ is a complex reflection group;
and (ii) $\dot\varphi W' \sse \GL_{\mb C}(V)$ is a coset of $W'$ through a finite-order element $\dot\varphi \in \GL_{\mb C}(V)$ normalizing $W'$.
(The element $\dot\varphi$ is \emph{not} part of the data.)
If $W' = \dot\varphi W'$,
i.e.,
if $\dot\varphi \in W'$,
one says that the coset is \emph{split},
else it is \emph{nonsplit}.
The class of $\dot\varphi$ in the quotient group $N_{\GL_{\mb C}(V)}(W') \bs W'$ is the---%
well-defined---%
\emph{twist of} $\wt{\mb G}$,
whose order is denoted by $\delta_{\wt{\mb G}} \geq 1$.
(The latter \emph{twist-order} is the smallest integer $k \geq 1$ such that $\dot\varphi^k \in W'$.)

If the subgroup $\braket{\dot\varphi,W'} \sse \GL_{\mb C}(V)$ generated by $W'$ and $\dot\varphi$\fn{
	\label{fn:finite_generated_group}

	Incidentally,
	the latter is \emph{finite},
	as it fits in the middle of a short exact group sequence
	\begin{equation}
		1 \lra W' \lra \braket{\dot\varphi,W'} \lra \braket{\dot\varphi}_{W'} \lra 1,
		\qquad \braket{\dot\varphi}_{W'}
		\ceqq \braket{\dot\varphi} \bs (\braket {\dot\varphi} \cap W') \simeq \braket{\dot\varphi,W'} \bs W',
	\end{equation}
	whose last term is cyclic of order $\delta_{\wt{\mb G}}$. }~preserves an $\mb R$-form $V_{\mb R} \sse V$ (resp.,
a $\mb Q$-form $V_{\mb Q} \sse V$),
then $\wt{\mb G}$ is a \emph{real reflection coset} (resp.,
a \emph{rational reflection coset},
cf.~\eqref{eq:real_form}).
In any event,
the \emph{rank of} $\wt{\mb G}$ is---%
still---%
defined as
\begin{equation}
	m \ceqq \dim_{\mb C}(V) \bigl( = \dim_{\mb R}(V_{\mb R}) = \dim_{\mb Q}(V_{\mb Q}) \bigr).
\end{equation}

\subsubsection{}

Let $\wt{\mb G} = (V,\dot\varphi W')$ be a complex reflection coset.
A \emph{reflection subcoset of} $\wt{\mb G}$ is a reflection coset of the form
\begin{equation}
	\wt{\mb G}'
	= \bigl( V',(\dot\varphi g')_{V'} W'' \bigr),
	\qquad (\dot\varphi g')_{V'}
	\ceqq \eval[1]{(\dot\varphi g')}_{V'} \in \GL_{\mb C}(V'),
\end{equation}
where:
(i) $V' \sse V$ is a vector subspace;
(ii) $W''$ is a subgroup of the setwise-modulo-pointwise-stabilizer $W'(V') = N_{W'}(V') \bs W'_{V'}$ (cf.~\eqref{eq:abstract_parabolic_subgroup}--\eqref{eq:abstract_setwise_stabilizer}),
acting on $V'$ as a \emph{reflection group};
and (iii) $g' \in W'$ is an element such that $\dot\varphi g' \in \dot\varphi W'$ stabilizes $V'$,
and the restriction $(\dot \varphi g')_{V'}$ normalizes $W'' \sse W'(V') \hra \GL_{\mb C}(V')$.
(The order of $\dot\varphi g' \in \braket{\dot\varphi,W'}$ is automatically \emph{finite},
cf.~Fn.~\ref{fn:finite_generated_group}.)

A \emph{Levi (reflection) subcoset of} $\wt{\mb G}$ is a---%
full-rank---%
reflection subcoset of the form
\begin{equation}
	\wt{\mb L}
	= (V,\dot\varphi g'W''),
\end{equation}
where:
(i) $W'' \sse W'$ is a parabolic subgroup;
and (ii) $g' \in W'$ is an element such that $\dot\varphi g' \in \dot\varphi W'$ normalizes $W''$---%
in the sense above.
The latter can be equivalently defined as the centralizer
\begin{equation}
	Z_{\wt{\mb G}}\bigl(\wt{\mb T}'\bigr) \ceqq (V,\dot\varphi g' W'_{V'}),
\end{equation}
of a \emph{torus of} $\wt{\mb G}$,
i.e.,
of a reflection subcoset of the form $\wt{\mb T}' = \bigl(V',(\dot\varphi g')_{V'}\bigr)$---%
through the trivial reflection group.
Any torus of $\wt{\mb G}$ is a reflection subcoset of a \emph{maximal torus of} $\wt{\mb G}$,
viz.,
a torus of the form $\wt{\mb T} = (V,\dot\varphi g')$.

One insight of this work is that nested sequences of Levi subcosets play a pivotal role in the topology of admissible-deformation spaces of irregular types/classes in \emph{twisted} meromorphic reductive $2d$ gauge theory.

\subsection{Classical Weyl groups}

It will be useful to base all the classical Weyl groups on type $A$,
as follows.

\subsubsection{}
\label{sec:background_type_A}

Given an integer $m \geq 1$,
consider the complex vector space $V^+_m \ceqq \mb C^m$,
with canonical basis $(e_1,\dc,e_m)$ indexed by the set $\ul m^+ \ceqq \set{1,\dc,m}$.
Identify $V^+_m$ with the standard Cartan subalgebra of diagonal matrices inside the general linear Lie algebra $\mf g = \mf{gl}_m(\mb C)$,
mapping
\begin{equation}
	e_i \lmt E_i \ceqq \diag (\underbrace{0,\dc,0,1}_{i\text{-th slot}},0,\dc,0).
\end{equation}

The Weyl group is then identified with the symmetric group $\mf S_m = \mf S^+_m$ of permutations of the set $\ul m^+$,
acting on $V^+_m$ by permuting the basis vectors---%
and the associated coordinates.
This action is \emph{not} essential,
and it induces irreducible actions on:
(i) the quotient modulo $\mb C \cdot (e_1 + \dm + e_m) \sse V^+_m$,
which yields the standard Cartan subalgebra of the projectivization $\mf{pgl}_m(\mb C) \simeq \mf{psl}_m(\mb C)$;
and (ii) the `trace-free' hyperplane $\Set{ \sum_{i = 1}^m \lambda_i e_i | \sum_i \lambda_i = 0 } \sse V^+_m$,
which yields the standard Cartan subalgebra of the (semi)simple part $\mf{sl}_m(\mb C) \sse \mf{gl}_m(\mb C)$.
These are isomorphic irreducible $\mf S_m$-modules,
i.e.,
isomorphic (rational) \emph{reflection} groups;
but recall that we also identify the Weyl group of $\mf{gl}_m(\mb C)$ and $\mf{sl}_m(\mb C) \simeq \mf{psl}_m(\mb C) \simeq \mf{pgl}_m(\mb C)$ as \emph{abstract} groups---%
acting as reflection groups on different spaces.
So we denote all these groups by $W_A(m) \simeq \mf S_m$.

For an integer $d \geq 2$,
a $d$-\emph{cycle} is an element $c^+ \in W_A(m)$ with a single nontrivial orbit in $\ul m^+$,
of cardinality $d$,
which is its \emph{support}.
We write
\begin{equation}
	\label{eq:type_A_cycle}
	c^+
	= ( a_1 \mid \dm \mid a_d ),
	\qquad a_1,\dc,a_d \in \ul m^+ \text{ distinct},
\end{equation}
to denote the $d$-cycle mapping $a_i \mt a_{i+1}$ (taking $i$ modulo $d$)---%
with support $\set{a_1,\dc,a_d}$.
Any element $g \in W_A(m)$ can be decomposed into a product of cycles with pairwise disjoint supports,
and this decomposition is unique up to reordering the factors.
Two elements of $W_A(m)$ are conjugate if and only if they have the same number of $d$-cycles in any such factorization,
for all $d \geq 2$ (i.e.,
if they have the same \emph{cycle-type},
cf.~\cite[Prop.~23]{carter_1972_conjugacy_classes_in_the_weyl_group}).

\subsubsection{}
\label{sec:background_type_BCD}

Consider now a second copy $V^-_m$ of $\mb C^m$,
with basis $(e_{-1},\dc,e_{-m})$.
The symmetric group $\mf S_m^\pm \simeq W_A(2m)$ of permutations of the set $\ul m^\pm \ceqq \set{\pm 1,\dc,\pm m}$ acts on $V^\pm_m \ceqq V^+_m \ops V^-_m \simeq \mb C^{2m}$ as above.
The other classical Weyl groups are subgroups thereof,
consisting of `signed' permutations.

Precisely:
\begin{enumerate}
	\item
	      the \emph{rank}-$m$ \emph{Weyl group of type} $B\!C$ can be defined as
	      \begin{equation}
		      \label{eq:weyl_type_BC}
		      W_{B\!C}(m)
		      \ceqq \Set{ g \in \mf S^\pm_m | g(i) + g(-i) = 0 \text{ for } i \in \ul m^\pm }
		      \simeq \mb Z^{\ts} \wr \mf S_m;
	      \end{equation}

	\item
	      and the \emph{rank}-$m$ \emph{Weyl group of type} $D$ as the index-$2$ subgroup
	      \begin{equation}
		      \label{eq:weyl_type_D}
		      W_D(m)
		      \ceqq \Biggl\{ \, g \in W_{B\!C}(m) \,\, \Biggl| \,\, \prod_{i = 1}^m g(i) > 0 \, \Biggr\}.
	      \end{equation}
\end{enumerate}

Then the group~\eqref{eq:weyl_type_BC} acts irreducibly on the subspace
\begin{equation}
	\label{eq:classical_cartan}
	\wt V_m
	\ceqq \Biggl\{ \, \sum_{i \in \ul m^\pm} \lambda_i e_i \in V^\pm_m \,\, \Biggl| \,\, \lambda_i + \lambda_{-i} = 0 \text{ for } i \in \ul m^\pm \, \Biggr\} \simeq \mb C^m,
\end{equation}
with basis $\wt e_{i,-i} \ceqq e_i - e_{-i}$ ($i \in \ul m^+$),
which will be identified with the Cartan subalgebras of diagonal matrices inside the Lie algebras of type $B_m$ and $C_m$,
viz.,
$\mf{so}_{2m+1}(\mb C)$ and $\mf{sp}_{2m}(\mb C)$ (respectively,
cf.~\cite{schoenenberg_2021_cartan_subalgebras_of_classical_lie_algebras_are_diagonal_matrices}).
Apart from irreducibility,
which fails in rank $m \in \set{1,2}$,
this works the same in type $D_m$:
so~\eqref{eq:classical_cartan} will also be identified with the `diagonal' Cartan subalgebra of $\mf{so}_{2m}(\mb C)$.

\subsubsection{}

For an integer $d \geq 2$,
a \emph{positive} $d$-\emph{cycle} is an element $\wt c \in \mf S^\pm_m$%
---of order $d$---%
of the form
\begin{equation}
	\label{eq:positive_cycle}
	\wt c
	= c^+c^-,
	\qquad c^+
	= ( a_1 \mid \dm \mid a_d), \quad c^-
	= ( -a_1 \mid \dm \mid -a_d ),
\end{equation}
for distinct elements $a_1,\dc,a_d \in \ul m^+$,
extending the notation of~\eqref{eq:type_A_cycle}.
For an integer $d \geq 1$,
a \emph{negative} $d$-\emph{cycle} instead is an element%
---of order $2d$---%
of the form
\begin{equation}
	\label{eq:negative_cycle}
	\wt c
	= ( a_1 \mid \dm \mid a_d \mid -a_1 \mid \dm \mid -a_d ).
\end{equation}
In both cases,
the \emph{support of} $\wt c$ is the subset $\set{\pm a_1,\dc,\pm a_d} \sse \ul m^\pm$.
The positive $2$-cycles (resp.~negative $1$-cycles) are also called \emph{positive transpositions} (resp.~\emph{negative transpositions});
the positive transpositions generate a copy of $W_A(m) \simeq \mf S_m \sse \mf S^\pm_m$.

Any element $g \in W_{B\!C}(m)$ decomposes into a product of disjoint positive/negative cycles,
unique up to reordering.
Two elements are conjugate if and only if they have the same number of positive/negative $d$-cycles in any such factorization,
for all $d \geq 1$ (i.e.,
the same \emph{signed cycle-type}).
Analogously,
any element $g \in W_D(m)$ decomposes into disjoint positive/negative cycles,
unique up to reordering,
but having an \emph{even} number of negative ones.
The signed cycle-type classifies conjugacy classes,
except when all cycles are positive of even length;
conversely,
there are \emph{two} conjugacy classes of elements of any such signed cycle-type (cf.~\cite[Prop.~25]{carter_1972_conjugacy_classes_in_the_weyl_group};
this subtlety propagates into \S~\ref{sec:D_trees}).

\subsection{Classical root systems}
\label{sec:root_systems_background}

We also use the roots of the above reflection groups.

\subsubsection{}

In the notation of \S\S~\ref{sec:background_type_A}--\ref{sec:background_type_BCD},
for $i \in \ul m^+$ denote by $\alpha_i \ceqq e_i^{\dual} \in (V^+_m)^{\dual}$ the (co)vectors of the dual basis,
i.e.,
the canonical complex coordinates of $\mb C^m$.
The root system of type $A_{m-1}$ is the finite set
\begin{equation}
	\label{eq:type_A_roots}
	\Phi_A(m-1)
	\ceqq \Set{ \alpha_{ij} | i \neq j \in \ul m^+ } \sse (V^+_m)^{\dual},
	\qquad \alpha_{ij} \ceqq \alpha_i - \alpha_j.\fn{
		Beware that in our notation the Weyl group (isomorphic to $\mf S_m$) is thus indexed with a shift $m-1 \mt m$,
		which rather recalls the \emph{nonsemisimple} rank of $\mf{gl}_m(\mb C)$.}
\end{equation}
(Using $\alpha_{ij} + \alpha_{ji} = 0$.)

Now extend this notation by setting $\alpha_i \ceqq e_i^{\dual} \in (V^\pm_m)^{\dual}$,
for $i \in \ul m^\pm$.
Moreover,
given a linear functional $\lambda \cl V^\pm_m \to \mb C$ let $\wt\lambda \ceqq \eval[1]\lambda_{\wt V_m}$ be its restriction onto~\eqref{eq:classical_cartan},
and write $\wt\alpha_i \ceqq \wt{\alpha_i}$ and $\wt{\alpha}_{ij} \ceqq \wt{\alpha_{ij}}$,
for $i \neq j \in \ul m^\pm$.
Then $\wt{\alpha}_{ij} = \wt{\alpha}_i - \wt{\alpha}_j$,
and one has the identities $\wt{\alpha}_i + \wt{\alpha}_{-i} = 0$.

Now the irreducible---%
nonreduced---%
root system of type $B\!C_m$ can be written
\begin{equation}
	\label{eq:type_BC_roots}
	\Phi_{B\!C}(m)
	\ceqq \Set{ \wt \alpha_{ij}, \wt \alpha_i | i \neq j \in \ul m^\pm } \sse \bigl( \wt V_m \bigr)^{\!\dual}.
\end{equation}
(Using $\wt \alpha_{i,-j} = \wt \alpha_i + \wt \alpha_j$,
whence in particular $\wt \alpha_{i,-i} = 2 \wt \alpha_i$.)
Its Weyl group is isomorphic to that of its root subsystems of type $B_m/C_m$,
viz.,
\begin{equation}
	\label{eq:type_b_c_roots}
	\Phi_B(m)
	\ceqq \Set{ \wt \alpha_{ij}, \wt \alpha_i }_{i \neq \pm j},
	\Phi_C(m) \ceqq \Set{ \wt \alpha_{ij} }_{i \neq j} \sse \Phi_{B\!C}(m).
\end{equation}
Then the root system of type $D_m$ is
\begin{equation}
	\label{eq:type_d_roots}
	\Phi_D(m)
	\ceqq \Set{ \wt \alpha_{ij} | i \neq \pm j \in \ul m^\pm } = \Phi_B(m) \cap \Phi_C(m),
\end{equation}
and we view $\Phi_A(m-1)$ as a subset of all the other classical root systems via $\alpha_{ij} \mt \wt \alpha_{ij}$.

Whenever needed,
we shall use the standard bases of such classical root systems.
They correspond to the following sets $\Delta_\bullet$ of \emph{simple roots}:
\begin{enumerate}
	\item[(A)]
	      $\Delta_A = \Delta_A(m-1) = \Set{ \theta_1,\dc,\theta_{m-1} } \sse \Phi_A(m-1)$,
	      with $\theta_i \ceqq \alpha_{i,i+1}$ for $i \in \ul{m-1}^+$;

	\item[(B)]
	      $\Delta_B = \Delta_B(m) = \Delta_A \cup \set{\theta_B} \sse \Phi_B(m)$,
	      where in addition $\theta_B \ceqq \wt\alpha_m$;

	\item[(C)]
	      $\Delta_C = \Delta_C(m) = \Delta_A \cup \set{\theta_C} \sse \Phi_C(m)$,
	      where instead $\theta_C \ceqq \wt\alpha_{m,-m}$;

	\item[(D)]
	      and $\Delta_D = \Delta_D(m) = \Delta_A \cup \set{\theta_D} \sse \Phi_D(m)$,
	      where $\theta_D \ceqq \wt \alpha_{m-1,-m}$.
\end{enumerate}

Finally,
we tacitly identify $W_{B\!C}(m) \simeq W_B(m) \simeq W_C(m)$ as the reflection groups generated by the $B\!C$-hyperplane arrangement in $\wt V_m$,
consisting of the (redundant) set of linear hyperplanes
\begin{equation}
	\label{eq:BC_root_hyperplanes}
	\Set{ \wt H_{ij} \ceqq \ker( \wt \alpha_{ij}), \, \wt H_i \ceqq \ker(\wt \alpha_i) | i \neq j \in \ul m^\pm } \sse \mb P \bigl( \bigl( \wt V_m \bigr)^{\!\dual} \bigr).
\end{equation}
So $\wt H_i = \wt H_{i,-i} = \ker\bigl( \wt \alpha_{i,-i} \bigr) \sse \wt V_m$ is a coordinate hyperplane,
while $\wt H_{ij}$ are the `diagonals' or `antidiagonals' when $i \neq \pm j$---%
where coordinates are either \emph{equal} or \emph{opposite},
respectively.

Viewing them as groups of signed permutations as above,
the Weyl-group action on the roots now just amounts to permuting the indices,
i.e.,
e.g.,
\begin{equation}
	\label{eq:permutation_action}
	g(\wt \alpha_{ij}) = \wt \alpha_{kl},
	\qquad k \ceqq g(i),
	\quad l \ceqq g(j),
	\quad i \neq j \in \ul m^\pm.
\end{equation}

\begin{rema}
	\label{rmk:roots_on_positive_indices}

	Overall,
	in type $B\!C$ the \emph{long} roots are $\wt \alpha_{i,-i}$,
	the \emph{medium} roots are $\wt \alpha_{ij}$,
	and the \emph{short} roots are $\wt \alpha_i$,
	with $i \neq \pm j \in \ul m^\pm$.

	Moreover,
	if one insists on using positive indices:
	compared to type $A$,
	in type $D$ we add roots of the form $\pm (\wt \alpha_i + \wt \alpha_j)$,
	for $i \neq j \in \ul m^+$;
	and in type $B$ (resp.,
	type $C$) also those of the form $\pm \wt \alpha_i$ (resp.,
	$\pm 2 \wt\alpha_i$).
\end{rema}

\begin{rema}
	\label{rmk:reductive_base}

	By a `base' of an abstract (reduced) root system $(\Phi,V)$ we mean a linearly-independent subset $\Delta \sse \Phi$ such that any root is a $\mb Z$-linear combination of elements of $\Delta$ with coefficients of constant sign;
	so the $\Delta$-\emph{positive roots} lie in the subset $\Phi^+ = \Phi^+_\Delta \ceqq \spann_{\mb Z_{\geq 0}} (\Delta) \cap \Phi$,
	and the $\Delta$-\emph{negative} ones are $\Phi^- \ceqq -\Phi^+$.
	The caveat is that $\Delta$ spans the `essential' part $V' \ceqq \spann_{\mb C}(\Phi) \sse V$,
	but not necessarily the whole of $V$.
	(Again this is needed to allow for nonsemisimple Lie algebras.)

	Somewhat analogously,
	the `Dynkin diagram' of the (pinned) root system of a split reductive Lie algebra is---%
	by definition---%
	that of its semisimple part.
\end{rema}

\section{About exceptional types}
\label{sec:exceptional_types}

\subsection{Springer's regular reflection groups}
\label{sec:springer_regular_reflection_groups}

In~\cite[\S~5.4]{springer_1974_regular_elements_of_finite_reflection_groups},
Springer considers complex reflection groups arising as centralizers $Z_W(g) \sse W$ of (nontrivial) regular elements $g \in W$,
where $W$ is an irreducible \emph{exceptional} Weyl group;
equivalently,
these groups are the setwise-stabilizers of \emph{regular eigenspaces}---%
containing a regular eigenvector $A \in \mf t_{\reg}$.

They now have a modular interpretation in $2d$ meromorphic gauge theory.

\subsubsection{}

The $W$-conjugacy classes of the subgroups $Z_W(g)$ are determined by the order $d \geq 2$ of $g$,
which is a regular number of $W$.
Neglecting Coxeter elements,
whose centralizer is cyclic of order equal to the Coxeter number $h \geq 6$ of $W$ (cf.~\cite{coxeter_1951_the_product_of_the_generators_of_a_finite_group_generated_by_reflections,kostant_1959_the_principal_three_dimensional_subgroup_and_the_betti_numbers_of_a_complex_simple_lie_group}),
this yields a priori $30$ additional cases amongst all exceptional types $\bullet \in \set{E_6,E_7,E_8,F_4,G_2}$,
corresponding to the following lists of regular numbers:
\begin{enumerate}
	\item[$(E_6)$]
	      if $\bullet = E_6$,
	      one has $d \in \set{2,3,4,6,8,9}$ (and $h = 12$);

	\item[$(E_7)$]
	      if $\bullet = E_7$,
	      one has $d \in \set{2,3,6,7,9,14}$ (and $h = 18$);

	\item[$(E_8)$]
	      if $\bullet = E_8$,
	      one has $d \in \set{2,3,4,5,6,8,10,12,15,20,24}$ (and $h = 30$);

	\item[$(F_4)$]
	      if $\bullet = F_4$,
	      one has $d \in \set{2,3,4,6,8}$ (and $h = 12$);

	\item[$(G_2)$]
	      and if $\bullet = G_2$,
	      one has $d \in \set{2,3}$ (and $h = 6$).
\end{enumerate}
(As already mentioned,
in type $E_8$ the case $d = 4$ yields $Z_W(g) \simeq G_{31}$;
cf.~also~\S~\ref{sec:bessis_literature}.)

\subsubsection{}

But some of the non-Coxeter cases are still classified by a rank count.
Namely,
if $A(\bullet,d)$ is the multiset of degrees of the exceptional Weyl group $W(\bullet)$ which are divisible by $d$,
then $W'(\bullet,d) \ceqq Z_{W(\bullet)}(g)$ acts irreducibly on a space of dimension
\begin{equation}
	\dim_{\mb C} \bigl( \mf t(g,\zeta_d) \bigr) = a(\bullet,d),
	\qquad a(\bullet,d) \ceqq \abs{A(\bullet,d)} \in \mb Z_{\geq 1},
\end{equation}
where $\zeta_d = e^{2 \pi \sqrt{-1} \slash d} \in \mb C^{\ts}$ (cf.~Rmk.~\ref{rmk:about_regular_numbers}).
Hence,
the cases where $d$ divides only one degree still yield cyclic groups acting on a complex line.
At the opposite end,
if $a(\bullet,d) \in \set{2,4,6,7,8}$ coincides with the rank $m(\bullet)$ of $W(\bullet)$,
then $\mf t(g,\zeta_d) = \mf t$:
this yields $g = \zeta_d \cdot \Id_{\mf t} \in W$ and $W'(\bullet,d) = W(\bullet)$.

Therefore,
for the classification one can restrict attention to regular numbers $d$ such that $a(\bullet,d) \in \set{2,\dc,m(\bullet)-1}$,
which immediately excludes $\bullet = G_2$.

\subsubsection{}

Now one can read again the tables at~\cite[\S~5.4]{springer_1974_regular_elements_of_finite_reflection_groups},
which list the degrees of the irreducible complex reflection group $W'(\bullet,d)$ in each case,
i.e.,
the elements of $A(\bullet,d)$.
Focussing on the columns where the number of such degrees yields a rank of $W'(\bullet,d)$ in the relevant range,
this leaves in principle $16$ examples:
\begin{enumerate}
	\item[($E_6$)]
	      four cases $d \in \set{2,3,4,6}$ if $\bullet = E_6$,
	      in dimensions $a(E_6,d) \in \set{2,3,4}$;

	\item[($E_7$)]
	      two cases $d \in \set{3,6}$ if $\bullet = E_7$,
	      all in dimensions $a(E_7,d) = 3$;

	\item[($E_8$)]
	      seven cases $d \in \set{3,4,5,6,8,10,12}$ if $\bullet = E_8$,
	      in dimensions $a(E_8,d) \in \set{2,4}$;\fn{
		      There seems to be a degree $30$ missing from the case $d = 6$ in loc.~cit.}

	\item[($F_4$)]
	      and three cases $d \in \set{3,4,6}$ if $\bullet = F_4$,
	      all in dimension $2$.
\end{enumerate}

But a priori reading the lists $A(\bullet,d)$ is \emph{not} enough,
as there are nonisomorphic irreducible complex reflection groups with the same rank/degrees (e.g.,
$G_8$ and $G_{13}$ from Shephard--Todd's list,
in rank two).

\subsubsection{}

Now one can turn to~\cite{miller_2015_eigenspace_arrangements_of_reflection_groups}.
Precisely,
by Cor.~3.6 of op.~cit.,
the identity $A(\bullet,d) = A(\bullet,d')$ implies that the \emph{maximal} $\zeta_d$-eigenspaces of elements of $W(\bullet)$---%
of dimension $a(\bullet,d)$---%
coincide with the $\zeta_{d'}$-ones.
Therefore,
up to conjugation by $W(\bullet)$,
an equality in Springer's tables actually yields $W'(\bullet,d) = W'(\bullet,d') \sse W(\bullet)$.
This observation removes the following redundancy:
(i) one has $A(E_7,3) = A(E_7,6)$;
(ii) one has $A(E_8,3) = A(E_8,6)$ and $A(E_8,5) = A(E_8,10)$;
and (iii) one has $A(F_4,3) = A(F_4,6)$.

So there are at most $12$ examples:
but one might have $W'(\bullet,d) \simeq W'(\bullet',d')$ for different exceptional types $\bullet,\bullet'$ and regular numbers $d,d'$.
Notably,
one reads that
\begin{equation}
	A(E_6,4)
	= A(F_4,4)
	= \set{8,12},
\end{equation}
and
\begin{equation}
	A(E_6,6)
	= A(F_4,3)
	= A(F_4,6)
	= \set{6,12}.
\end{equation}

\subsubsection{}

This is where a direct inspection becomes necessary.
Finally,
we appeal to~\cite[Tabb.~2+10]{miller_2015_eigenspace_arrangements_of_reflection_groups} (cf.~\cite[\S~5]{lehrer_springer_1999_reflection_subquotients_of_unitary_reflection_groups} and~\cite{lehrer_springer_1999_intersection_multiplicities_and_reflection_subquotients_of_unitary_reflection_groups_i}),
which in particular give
\begin{equation}
	W'(E_6,4) \simeq W'(F_4,4)
	= G_8,
\end{equation}
disambiguating from $G_{13}$,
and
\begin{equation}
	W'(E_6,6)
	\simeq W'(F_4,3)
	\simeq W'(F_4,6)
	= G_5.
\end{equation}
In the end one thus finds $10$ genuinely distinct groups,
and recalling that $G_{28} = W(F_4)$ leads to the statement of Cor.~\ref{cor:modular_reflection_groups}~(2.b).

\subsection{About exceptional diagram automorphisms}

Finally,
recall that there are exceptional Dynkin-diagram automorphisms which are relevant in the context of \S~\ref{sec:interior_twist}:
the aforementioned triality of $D_4$ (cf.~\S~\ref{sec:triality}),
and the diagram flipping of $E_6$.
They now also play a role in meromorphic reductive $2d$ gauge theory.

\section{Lifting Springer theory}
\label{sec:bessis_literature}

\subsection{Finite complex-reflection arrangements are \texorpdfstring{$K(\pi,1)$}{K(pi,1)}}

The article~\cite{bessis_2015_finite_complex_reflection_arrangements_are_k_pi_1} proves that the universal cover of $V_{\reg,W'} \sse V$ is \emph{contractible},
for any---%
finite---%
complex reflection group $(V,W')$.
This had been a long-standing conjecture~\cite{brieskorn_1973_sur_les_groupes_de_tresses_d_apres_v_i_arnold,
	orlik_terao_1992_arrangements_of_hyperplanes},
previously proven in all but six exceptional cases~\cite{fadell_neuwirth_1962_configuration_spaces,
	deligne_1972_les_immeubles_des_groupes_de_tresses_generalises,
	nakamura_1983_a_note_on_the_k_pi_1_property_of_the_orbit_space_of_the_unitary_reflection_group_g_m_l_n,
	orlik_solomon_1988_discriminants_in_the_invariant_theory_of_reflection_groups}:
cf.~\cite[\S~10]{miller_2015_eigenspace_arrangements_of_reflection_groups} for more details about the progress achieved in each article (and~\cite{paolini_salvetti_2021_proof_of_the_k_pi_1_conjecture_for_affine_artin_groups} in the affine case).

In any event,
the point of~\cite{bessis_2015_finite_complex_reflection_arrangements_are_k_pi_1} was to treat the remaining examples of
\begin{equation}
	G_i,
	\qquad i \in \set{24,27,29,31,33,34},
\end{equation}
in Shephard--Todd's primitive classification.
It turns out that all of them,
but $G_{31}$,
are well-generated,\fn{
	They can be generated by as many reflections as their rank;
	the rank-$4$ group $G_{31}$ requires five.}~and op.~cit.~first establishes the result under this hypothesis.

As far as this text is concerned (notably \S~\ref{sec:braid_springer_theory}),
the relevant facts are that:
(i) $G_{31}$ can be realized as the centralizer of an order-$4$ regular element in $G_{37}$,
i.e.,
the (well-generated) Weyl group of type $E_8$ (cf.~Exmp.~1.10 of op.~cit.,
and Tab.~\ref{tab:springer_from_weyl});
(ii) Thm.~0.3 of op.~cit.~states that the $K(\pi,1)$ property is inherited under this operation;
and (iii) Thm.~12.4 of op.~cit.~establishes properties of the braid group of the corresponding reflection arrangement,
which can be seen as a lift of Springer theory through the augmentation group morphism $\pi_1(V_{\reg} \bs W') \thra W'$ (cf.~also Broué's lectures~\cite[Chp.~II]{broue_2001_reflection_groups_braid_groups_hecke_algebras_finite_reductive_groups},
\cite[Chp.~III \S~18]{broue_2008_introduction_to_complex_reflection_groups_and_their_braid_groups},
and~\cite[\S~5.3.3]{broue_2010_introduction_to_complex_reflection_groups_and_their_braid_groups}).

\section{Deferred proofs}
\label{sec:missing_proofs}

\subsection{Proof of Lem.-Def.~\ref{lem:generating_galois_orbit}}
\label{proof:lem_generating_galois_orbit}

For the nontrivial implication,
suppose that $g \in W$ satisfies~\eqref{eq:monodromy_realization}.
Then one has
\begin{equation}
	g^j \bigl( \wh Q \bigr)
	= \wh Q( \zeta_r^{-j} w),
	\qquad j \in \mb Z_{> 0},
\end{equation}
because $\sigma$ commutes with the $W$-action.
In particular $g^r \bigl( \wh Q \bigr) = \wh Q$,
and the irregular types $g^j \bigl( \wh Q \bigr) \in \wh{\mc{IT}}$ are well-defined for $j \in \mb Z \bs r \mb Z$---%
and they equal $\sigma^j \bigl( \wh Q \bigr)$.
Then the finite sets of~\eqref{eq:orbit_comparison} all coincide with
\begin{equation}
	\label{eq:galois_orbit_from_generator}
	\Set{ \wh Q, g \bigl( \wh Q \bigr),\dc, g^{k-1} \bigl( \wh Q \bigr) } \sse \wh{\mc{IT}},
\end{equation}
where $k \geq 1$ is the minimal divisor of $r$ such that $g^k \in W$ stabilizes $\wh Q$.
(This is the \emph{twist-order} of the Levi subcoset $(\mf t,g W^{\wh Q})$ of $\mb G = (\mf t,W)$,
where $W^{\wh Q} \sse W$ is the stabilizer of $\wh Q$;
cf.~\S\S~\ref{sec:modular_groups_are_twisted_cabled_weyl} +~\ref{sec:reflection_cosets}.)
In particular,
one has $\braket{ \wh Q }_r \sse \wh\Theta \ceqq \wh\Theta \bigl( \wh Q \bigr)$.

\subsection{Proof of Lem.~\ref{lem:more_on_irregular_types}}
\label{proof:lem_more_on_irregular_types}

For the first statement,
one inclusion follows from the equalities
\begin{equation}
	h g' h^{-1} \bigl( \wh Q' \bigr)
	= h g' \bigl( \wh Q \bigr)
	= h \bigl( \sigma \bigl( \wh Q \bigr) \bigr)
	= \sigma \bigl( \wh Q' \bigr) \in \wh{\mc{IT}},
\end{equation}
for any $g' \in \mb L_rr \bigl( \wh Q \bigr)$---%
as the $W$-action commutes with~\eqref{eq:monodromy}.
Now applying this to $h^{-1}$ and $\wh Q'$ yields
\begin{equation}
	h^{-1} \mb L_rr \bigl( \wh Q' \bigr) h \sse \mb L_rr \bigl( \wh Q \bigr) \sse W.
\end{equation}

The second statement follows directly from Def.~\ref{def:twisted_pure_admissible_deformations}~(1.).
(But the naive analogue of~\eqref{eq:deformations_are_intersections} is \emph{false} for $r$-admissible deformations.)

The last statement now follows from~\eqref{eq:conjugated_orbit_generators}--\eqref{eq:deformations_are_intersections},
as the former yields the equality $h \bigl( \wh{\mc{IT}}_{g,r} \bigr) = \wh{\mc{IT}}_{g',r}$,
noting in addition that $h \bigl( \bm B^{\mathsmaller\leq s} \bigl( \wh Q \bigr) \bigr) = \bm B^{\mathsmaller\leq s} \bigl( \wh Q' \bigr)$---%
i.e.,
that the Weyl group permutes root-valuation tuples.

\subsection{Proof of Lem.~\ref{lem:primitive_irregular_types}}
\label{proof:lem_primitive_irregular_types}

Consider the subset
\begin{equation}
	E \bigl( \wh Q \bigr)
	\ceqq \Set{ i \in \mb Z_{\geq 1} | A_i \neq 0 } \sse \set{1,\dc,s},
\end{equation}
corresponding to the nonvanishing coefficients of~\eqref{eq:explicit_irregular_type}.
The number $r'$ of elements in the Galois-orbit~\eqref{eq:global_r_galois_orbit} equals the least common multiple
\begin{equation}
	\bigvee_{\!E(\wh Q)} \bigl( r \slash (r \wdg i) \bigr) \geq 1.
\end{equation}
Moreover,
one has $r' \mid r$,
and the integer $\wt r \ceqq r \slash r'$ divides the greatest common divisor $\bigwedge E \bigl( \wh Q \bigr)$.
Therefore:
(i) the irregular type $\wh Q' \ceqq \sum_{E (\wh Q)} A_i w^{-i/\wt r}$ is (untwisted and) $r'$-Galois-closed;
and (ii) one has $\bigl| \braket{ \wh Q' }_{r'} \bigr| = r'$.

\subsection{Proof of Lem.~\ref{lem:intrinsicality}}
\label{proof:lem_intrinsicality}

The first statement follows by observing that the whole $H$-action preserves the valuations of $\mc L_{\geq 0}$ and $\mc L$,
and that---%
by construction---%
it commutes with the diagonal Adjoint action of $N_G(T) \sse G$ on $\mf t$-valued coefficients of formal power series;
then one concludes by the defining identity of~\eqref{eq:equivariant_field_automorphisms}:
\begin{equation}
	g \bigl( \wh Q (w') \bigr)
	= \wh Q \bigl( f(\zeta_r^{-1} w) \bigr)
	= \wh Q \bigl( \zeta_r^{-1} w' \bigr),
	\qquad w' \ceqq f(w),
	\quad g \in W,
	\quad f \in H_r.
\end{equation}

The second statement then follows from~\eqref{eq:deformations_are_intersections} and~\cite[Lem.~5.1.2~(2.)]{chaffe_rembado_yamakawa_wild_genus_zero_quantum_de_rham_spaces},
as the latter implies that the \emph{untwisted} admissible-deformation spaces are $H$-invariant.
(They are controlled by sequences of root systems of reductive Lie subalgebras $\mf l \sse \mf g$,
and such sequences are stable under coordinate change;
cf.~\S~\ref{sec:general_pure_twisted_adm_def_space}).

For the third statement,
recall that any element $f \in H$ is determined by the power series $f(w) = w \cdot u \in \mc L_{\geq 0}$,
where $u$ is a unit in the ring;
i.e.,
one has $(u - a_0) \in w\mb C\llb w \rrb$ for some $a_0 \in \mb C^{\ts}$.\fn{
	This bijection $H \lxra{\simeq} \mb C \llb w \rrb^{\ts}$ is \emph{not} compatible with the multiplication of the group of units,
	and it rather endows the target with a nonabelian `composition' structure---%
	with trivial centre.}~Explicitly,
if $u = \sum_{i \geq 0} a_i w^i$ with coefficients $a_i \in \mb C$,
then the condition~\eqref{eq:equivariant_field_automorphisms} reads
\begin{equation}
	a_i(1 - \zeta_r^i)
	= 0,
	\qquad i \geq 0,
\end{equation}
which implies that $a_i = 0$ unless $r \mid i$.
Thus,
the element $u \in \mb C \llb w^r \rrb = \mb C \llb z \rrb$ is actually a unit in the Galois-fixed subring $\mb C\llb w \rrb^{\mb Z \slash r \mb Z} \sse \mb C\llb w \rrb$.
Now any such element descends to a change of uniformizer $z$ by mapping
\begin{equation}
	w^r = z \lmt \bigl( f(w) \bigr)^r = ( w \cdot u )^r = z \cdot u^r \in \mb C \llb z \rrb,
\end{equation}
and this defines a group morphism $H_r \to H(z)$ with kernel
\begin{equation}
	\mu_r \bigl( \mb C\llb z \rrb \bigr)
	\ceqq \Set{ u \in \mb C \llb z \rrb^{\ts} | u^r = 1 } \simeq \Set{ a_0 \in \mb C^{\ts} | a_0^r = 1}
	= \mu_r.
\end{equation}
Finally,
surjectivity holds in the setting of complete DVRs of equicharacteristic zero,
with algebraically closed coefficient/residue field.
Indeed,
given a unit $v = \sum_{i \geq 0} b_i z^i \in \mb C\llb z \rrb^{\ts}$,
one can first---%
noncanonically---%
extract an $r$-th root $a_0 \in \mb C^{\ts}$ of the leading coefficient $b_0 \neq 0$.
Then,
writing
\begin{equation}
	v = b_0(1 + \wt v),
	\qquad \wt v
	\ceqq \sum_{i \geq 1} b_0^{-1}b_i z^i,
\end{equation}
observe that one has $u(\wt v,1\slash r)^r = 1 + \wt v$ by evaluating the exponent of the corresponding formal binomial series at the number $c = 1 \slash r$:
\begin{equation}
	u(\wt v,c)
	= (1 + \wt v)^{c}
	= \sum_{i \geq 0} \binom{c}{i} \wt v^i,
	\qquad \binom{c}{i} \ceqq \frac{c \cdot (c - 1) \dotsm (c -i + 1)}{i!},
	\quad c \in \mb C.
\end{equation}

\subsection{Proof of Lem.-Def.~\ref{lem:flag_stabilizer}}
\label{proof:lem_flag_stabilizer}

Analogously to Prop.-Def.~\ref{prop:pure_general_twisted_deformation_space_1_coeff}~(1.),
the given hypothesis yields $g(A_s) = \zeta_r^s A_s \in \bm B \bigl( \wh Q,s \bigr)$,
whence $g(\phi_s) = \phi_s$ and $g \in N_W(\mf t_{\phi_s})$ (cf.~Lem.~\ref{lem:about_setwise_stabilizers}).
To iterate,
let $\wt \phi_i \ceqq \phi_{A_i} \sse \Phi$ be the Levi annihilator of the $i$-th coefficient,
so that by construction
\begin{equation}
	\label{eq:levi_filtration_intersection}
	\phi_i = \wt \phi_i \cap \dm \cap \wt \phi_s \sse \Phi,
	\qquad i \in \set{1,\dc,s}.
\end{equation}
(And in particular $\phi_s = \wt \phi_s$.)
Moreover,
\begin{equation}
	g(A_{s-1})
	= \zeta_r^{s-1} A_{s-1} \in \bm B \bigl( \wh Q,s-1 \bigr),
\end{equation}
and it follows again that $g(\wt \phi_{s-1}) = \wt \phi_{s-1}$ and $g \in N_W(\mf t_{\wt \phi_{s-1}}) \sse W$;
therefore,
one has
\begin{equation}
	g(\phi_{s-1})
	= g\bigl(\wt \phi_{s-1} \cap \wt \phi_s\bigr)
	= g\bigl(\wt \phi_{s-1}\bigr) \cap g\bigl(\wt \phi_s\bigr)
	= \wt \phi_{s-1} \cap \wt \phi_s
	= \phi_{s-1},
\end{equation}
as $g$ acts bijectively on $\Phi$.
Etc.

\subsection{Proof of Lem.~\ref{lem:about_setwise_stabilizers}}
\label{proof:lem_about_setwise_stabilizers}

\begin{description}
	\item[(1.) $\Rightarrow$ (2.)]
	      It is equivalent to prove that $g^{-1}(\phi) = \phi \sse \Phi$,
	      so suppose by contradiction that $\beta \ceqq g^{-1}(\alpha) \in \Phi \sm \phi$ for some $\alpha \in \phi$.
	      Choosing \emph{any} $A' \in \bm B \bigl( \wh Q \bigr)$ (which is nonempty),
	      one has $\Braket{ \alpha,g(A') } = \braket{ \beta,A' } \neq 0$,
	      whence $g(A') \not\in \mf t_\phi$.

	\item[(2.) $\Leftrightarrow$ (3.)]
	      Clear,
	      as $g$ induces a bijection of $\Phi$.

	\item[(3.) $\Rightarrow$ (4.)]
	      Observe that $g$ permutes the root hyperplanes via
	      \begin{equation}
		      \label{eq:wall_permutation}
		      g (H_\alpha)
		      = H_{g(\alpha)},
		      \qquad \alpha \in \Phi,
	      \end{equation}
	      in the notation of~\eqref{eq:regular_cartan}.
	      Thus,
	      one has
	      \begin{equation}
		      g(\mf t_\phi)
		      = \mf t_{\phi'},
		      \quad
		      g \bigl( H_\alpha(\phi) \bigr)
		      = H_{g(\alpha)}(\phi'),
		      \qquad \phi'
		      \ceqq g(\phi) \sse \Phi,
	      \end{equation}
	      in the notation of~\eqref{eq:pure_untwisted_nongeneric_deformation_space_1_coeff}.
	      Now take hyperplane complements.

	\item[(4.) $\Rightarrow$ (5.)]
	      Obvious.

	\item[(5.) $\Rightarrow$ (6.)]
	      Obvious.

	\item[(2.) $\Rightarrow$ (1.)]
	      This also follows from~\eqref{eq:wall_permutation}.
	      (So (1.) $\Leftrightarrow$ (2.) now.)

	\item[(6.) $\Rightarrow$ (2.)]
	      As in the proof of the first implication,
	      by contradiction,
	      if $g^{-1}(\phi) \neq \phi$ then $g(A') \not\in \mf t_\phi$ for \emph{any} vector $A' \in \bm B \bigl( \wh Q \bigr)$.
	      (Now (1.)--(6.) are equivalent.)

	\item[(4.) $\Rightarrow$ (7.)]
	      Obvious.

	\item[(7.) $\Rightarrow$ (6.)]
	      Obvious.
\end{description}

\subsection{Proof of Lem.~\ref{lem:about_pointwise_stabilizers}}
\label{proof:lem_about_pointwise_stabilizers}

\begin{description}
	\item[(1.) $\Rightarrow$ (2.)]
	      Obvious.

	\item[(2.) $\Rightarrow$ (3.)]
	      Obvious.

	\item[(3.) $\Rightarrow$ (1.)]
	      Choose a lift $\wt g \in N_G(T)$ of $g$.
	      By hypothesis,
	      the element $\wt g$ lies in the connected reductive algebraic subgroup $L \sse G$ centralizing $A'$ (cf.~\eqref{eq:levi_subgroup}),
	      i.e.,
	      one has $\wt g \in N_L(T) = L \cap N_G(T)$.
	      Therefore,
	      $g$ actually lies in the Weyl group $N_L(T) \bs T \hra W$ of $(L,T)$,
	      and it is a product of reflections about the root hyperplanes of the Levi subsystem $\Phi(\mf l,\mf t) \sse \Phi$,
	      with $\mf l \ceqq \Lie(L)$ as in~\eqref{eq:levi_subalgebra};
	      but by construction one has $\Phi(\mf l,\mf t) = \phi$.
\end{description}

\subsection{Proof of Lem.~\ref{lem:no_marking}}
\label{proof:lem_no_marking}

The kernel of $\phi$ is a stable subspace for both $g$ and $g'$,
i.e.,
$g,g' \in N_W(\mf t_\phi)$ (cf.~Lem.~\ref{lem:about_setwise_stabilizers}).
Moreover,
by hypothesis,
one has $g^{-1} g' (A) = A$,
whence $(g^{-1}g')_{\phi} = g_\phi^{-1} g'_\phi$ acts as the identity on $\mf t_\phi$ (cf.~Lem.~\ref{lem:about_pointwise_stabilizers}).

\subsection{Proof of Lem.~\ref{lem:restricted_commutators}}
\label{proof:lem_restricted_commutators}

\begin{description}
	\item[(1.) $\Rightarrow$ (2.)]
	      Obvious.

	\item[(2.) $\Rightarrow$ (3.)]
	      Obvious.

	\item[(3.) $\Rightarrow$ (1.)]
	      Let $h \ceqq g^{-1} (g')^{-1} g g' \in W$ be the commutator of $g$ and $g'$.
	      It preserves $\mf t_\phi \sse \mf t$,
	      and its class/restriction $h_\phi$ coincides with the commutator of $g_\phi, g'_\phi \in W(\phi)$.
	      Then one computes
	      \begin{equation}
		      \label{eq:commutator_in_parabolic}
		      h(A')
		      = h_\phi(A')
		      = g_\phi^{-1} (g_\phi')^{-1} \cdot \zeta \bigl( g_\phi'(A') \bigr)
		      = \zeta g_\phi^{-1} (A')
		      = A'.
	      \end{equation}
	      One concludes by Lem.~\ref{lem:about_pointwise_stabilizers}.
	      (As $A'$ is in general position in $\mf t_\phi$.)
\end{description}

\subsection{Proof of Cor.-Def.~\ref{cor:explicit_deck_group}}
\label{proof:cor_explicit_deck_group}

The group~\eqref{eq:flag_parabolic} stabilizes the flag $\mf t_{\bm\phi}$,
and so the classes/restrictions $g'_{\phi_i} \in \GL_{\mb C}(\mf t_{\phi_i})$ of any element $g' \in N_W(\mf t_{\bm\phi})$ are compatible.
I.e.,
if $i < j \in \set{1,\dc, s}$,
then $g'_{\phi_j}$ is the restriction of $g'_{\phi_i}$ onto $\mf t_{\phi_j}$;
i.e.,
$g'_{\phi_j}$ is the image of $g'$ under the group morphism
\begin{equation}
	W^{\phi_i,\phi_j} \lthra W^{\phi_i,\phi_j} \bs W_{\mf t_{\phi_i}} \sse W(\phi_i),
	\qquad W^{\phi_i,\phi_j} \ceqq N_W(\mf t_{\phi_i}) \cap N_W(\mf t_{\phi_j}),
\end{equation}
composed with
\begin{equation}
	W^{\phi_i,\phi_j} \bs W_{\mf t_{\phi_i}} \lthra W^{\phi_i,\phi_j} \bs \bigl( N_W(\mf t_{\phi_i}) \cap W_{\mf t_{\phi_j}} \bigr) \sse W(\phi_j).
\end{equation}
(Note that $W_{\mf t_{\phi_i}} \sse N_W(\mf t_{\phi_i}) \cap W_{\mf t_{\phi_j}} \sse W^{\phi_i,\phi_j}$,
and the middle term is the kernel of the composition $W^{\phi_i,\phi_j} \hra N_W(\mf t_{\phi_j}) \thra W(\phi_j)$.)
In particular,
if $g_{\phi_i},g'_{\phi_i} \in W(\phi_i)$ commute then this holds a fortiori for $g_{\phi_j},g'_{\phi_j} \in W(\phi_j)$.

Therefore,
the definition~\eqref{eq:restricted_centralizer_intersection} yields the subgroup of flag-preserving elements whose image on the \emph{largest} Levi kernel $\mf t_{\phi_1} \sse \mf t$ commutes with that of $g$:
\begin{equation}
	p^{-1}_{\bm\phi}(g)
	= \wt p^{-1}_{\phi_1} \bigl( Z_{W,\phi_1}(g) \bigr),
	\qquad \wt p_{\phi_1}
	\ceqq \eval[1]{p_{\phi_1}}_{N_W(\mf t_{\bm\phi})} \cl N_W(\mf t_{\bm\phi}) \lra W(\phi_1).
\end{equation}
The difference with the monolevel case is that $\wt p_{\phi_1}$ need \emph{not} be surjective,
and taking the image of the preimage yields the stated identification:
\begin{equation}
	Z_{W,\bm\phi}(g)
	\simeq \Set{ g'_{\phi_1} \in Z_{W,\phi_1}(g) | g'_{\phi_1}(\mf t_{\phi_i}) \sse \mf t_{\phi_i} \text{ for } i \in \set{2,\dc,s} }.
\end{equation}

\subsection{Proof of Lem.~\ref{lem:relative_weyl_group}}
\label{proof:lem_relative_weyl_group}

We first construct the isomorphism $W(\phi) \lxra{\simeq} W_L$.
By Lem.~\ref{lem:about_setwise_stabilizers},
any element $g \in N_W(\mf t_\phi)$ preserves $\phi \sse \Phi$.
Let $\wt g \in N_G(T)$ be a lift of $g$,
so that $\Ad_{\wt g}(\mf l) \sse \mf l$ (cf.~\eqref{eq:levi_subalgebra}).
Thus,
if $C_{\wt g} \in \Aut(G)$ is the conjugation action of $\wt g$,
one has $C_{\wt g}(e^{\mf l}) \sse L$;
whence $C_{\wt g}(L) \sse L$---%
since $L$ is connected.
Moreover,
the coset $g_L \ceqq \tilde g L \in W_L$ does \emph{not} depend on the choice of the lift,
because $T \sse L$.
This yields a well-defined function $F = F_{\phi} \cl N_W(\mf t_\phi) \to W_L$,
which is actually a group morphism.
Now suppose that $g \in W_{\mf t_\phi} \sse N_W(\mf t_\phi)$.
Then $\wt g \in N_L(T) \sse L$ by Lem.~\ref{lem:about_pointwise_stabilizers},
and so $g_L \in W_L$ is trivial.
Conversely,
if $g \in N_W(\mf t_\phi)$ is such that the coset $g_L$ is trivial,
then $\wt g$ actually lies in $L = G^A$,
and so $g(A) = A$;
the same lemma also implies that $g \in W_{\mf t_\phi}$.
Overall,
there is an exact group sequence
\begin{equation}
	1 \lra W_{\mf t_\phi} \lra N_W(\mf t_\phi) \lxra{F}{} W_L.
\end{equation}
To prove surjectivity,
choose any element $g \in W_L$,
and lift it to an element $\wt g \in N_G(L)$.
It follows that $\Ad_{\wt g} \in \Aut(\mf g)$ restricts to an automorphism of $\mf l$,
and so it maps the Cartan subalgebra $\mf t \sse \mf l$ to another Cartan subalgebra $\mf t' \ceqq \Ad_{\wt g}(\mf t)$.
But all the Cartan subalgebras of $\mf l$ are conjugate,
since $\mf l$ is reductive (cf.~\cite[\S~16]{humphreys_1972_introduction_to_lie_algebras_and_representation_theory},
and~\cite[\S~4.1]{varadarajan_1984_lie_groups_lie_algebras_and_their_representations},
and~\cite[Thm.~2.1.11]{collingwood_mcgovern_1993_nilpotent_orbits_in_semisimple_lie_algebras});
thus,
there exists an element $\wt g' \in L$ such that $\Ad_{\wt g'}(\mf t') = \mf t$,
and in turn
\begin{equation}
	\Ad_{\wt g''}(\mf t)
	= \mf t,
	\qquad \wt g''
	\ceqq \wt g' \, \wt g \in N_G(L).
\end{equation}
Again one has $C_{\wt g''}(T) \sse T$,
and so the element $g'' \ceqq \wt g'' T \in W$ is well-defined.
The corresponding permutation action on the roots now preserves $\phi \sse \Phi$,
and in the end $F(g'') = g \in W_L$.

For the other group isomorphism,
choose again $g \in N_W(\mf t_\phi) \sse W$.
Reasoning as above proves that any lift $\wt g$ of $g$ lies in $N_G(T) \cap N_G(T_\phi) \sse G$.
Therefore,
up to the identification $W(\phi) \simeq W_L$ (which was just established) there is an embedding
\begin{equation}
	W_L \lhra N_G(T_\phi) \bs L
	= N_G(T_\phi) \bs Z_G(T_\phi).
\end{equation}
The converse follows from the inclusion $N_G(T_\phi) \sse N_G(L)$,
which is true for abstract reasons:
if $\wt g \in G$ normalizes a subgroup of $G$,
then it also normalizes the centralizer of that subgroup.

\subsection{Proof of Lem.-Def.~\ref{lem:inner_product}}
\label{proof:lem_inner_product}

The first statement follows from the induced musical isomorphism $(\cdot \mid \cdot)^\sharp \cl \mf t^{\dual} \lxra{\simeq} \ol{\mf t}$.
(Note that $\alpha_\phi^{\dual} = 0$ would be equivalent to $\ker(\phi) \sse \ker(\alpha)$,
which in turn forces $\alpha \in \phi$,
since $\phi$ is a Levi subsystem.)

For the second statement,
suppose w.l.o.g.~that $\mf g$ is simple,
noting that $W$ acts trivially on the Lie-centre $\mf Z \sse \mf g$,\fn{
	One can put \emph{any} positive-definite Hermitian form on the Lie-centre,
	defined over $\mb Q$.}~and in view of the splitting argument of Lem.~\ref{lem:reduction_to_simple_case}.
Then $W$ acts irreducibly on $\mf t$,
and $(\cdot \mid \cdot)$ is unique up to dilation by $\mb R_{> 0}$:
the hypothesis is that a constant has been chosen such that the crystallographic condition holds,
viz.,
one has
\begin{equation}
	(\alpha^{\dual} \mid \beta^{\dual}) \in \mb Q,
	\qquad \alpha,\beta \in \Phi.
\end{equation}
Now there is a $(\cdot \mid \cdot)$-orthogonal splitting
\begin{equation}
	\mf t = \mf t_\phi \ops \mf t'_\phi,
	\qquad \mf t'_\phi
	\ceqq \spann_{\mb C} (\phi^{\dual}),
\end{equation}
where $\phi^{\dual} \sse \Phi^{\dual}$ is the dual/inverse Levi subsystem.
(This restricts the decomposition $\mf l_\phi = \mf Z(\mf l_\phi) \ops [\mf l_\phi,\mf l_\phi]$ onto $\mf t \sse \mf l_\phi$;
cf.~\eqref{eq:levi_subalgebra}.)
If $\pi_\phi \cl \mf t \thra \mf t'_\phi$ is the `complementary' orthogonal projection,
one has
\begin{equation}
	\alpha^{\dual}_\phi
	= \alpha^{\dual} - \pi_\phi ( \alpha^{\dual} ) \in \mf t_\phi,
	\qquad \alpha \in \Phi \sm \phi,
\end{equation}
and so it is enough to prove that
\begin{equation}
	\label{eq:rationality}
	\bigl( \pi_\phi(\alpha^{\dual}) \mid \pi_\phi(\beta^{\dual}) \bigr) \in \mb Q,
	\qquad \alpha,\beta \in \Phi \sm \phi.
\end{equation}
Let then $\Delta \sse \Phi$ be a base of simple roots such that $\Delta_\phi \ceqq \Delta \cap \phi$ is a base of $\phi$ (cf.~Prop.-Def.~\ref{prop:howlett}).
Write $\Delta_\phi = \set{\theta_1,\dc,\theta_{m'}}$,
where $m' \ceqq \dim_{\mb C}(\mf t'_\phi)$ is the rank of $\phi$;
then the square matrix with coefficients $m_{ij} \ceqq (\theta_i^{\dual} \mid \theta_j^{\dual})$ lies in $\GL_{m'}(\mb Q)$.
If now $\pi_\phi ( \alpha^{\dual} ) = \sum_{i = 1}^{m'} c_i \theta_i^{\dual}$ for some $\alpha \in \Phi \sm \phi$,
and for suitable coefficients $c_i \in \mb C$,
then
\begin{equation}
	\sum_i c_i \cdot m_{ij}
	= \bigl( \pi_\phi(\alpha^{\dual}) \mid \theta_j^{\dual} \bigr)
	=     (\alpha^{\dual} \mid \theta_j^{\dual}) \in \mb Q,
	\qquad j \in \set{1,\dc,m'}.
\end{equation}
This implies that $c_1,\dc,c_{m'} \in \mb Q$,
whence~\eqref{eq:rationality} holds.

\subsection{Proof of Lem.~\ref{lem:levi_subcoset_filtration_set_inclusions}}
\label{proof:lem_levi_subcoset_filtration_set_inclusions}

For the two top levels,
the hypothesis provides elements $g'_i \in N_W(W_i)$ such that $\mb L_i = g'_i W_i$,
for $i \in \set{s-1,s}$.
Now the inclusion $g'_{s-1}W_{s-1} \sse g'_s W_s$ ensures that $g \ceqq (g'_s)^{-1} g'_{s-1} \in W_s$.
Moreover,
if $h \in W_{s-1}$ then
\begin{equation}
	g h
	= (g'_s)^{-1} (g'_{s-1} h) \in (g'_s)^{-1} (g'_s W_s) = W_s,
\end{equation}
and so $h = g^{-1} (gh) \in W_s$.
So the parabolic subgroups of $W$ are nested.
Moreover,
the group $W_{s-1}$ is also a parabolic subgroup of $W_s$.
Indeed,
one may choose Levi subsystems $\phi_{s-1},\phi_s \sse \Phi$ such that
\begin{equation}
	W_i
	= W_{\mf t_i},
	\qquad \mf t_i
	\ceqq \mf t_{\phi_i} \sse \mf t,
	\qquad i \in \set{s-1,s}.
\end{equation}
The inclusion $W_{s-1} \sse W_s$ now yields
\begin{equation}
	\mf t_s
	= \ol{\mf t_s}^{\mathsmaller W} \sse \ol{\mf t_{s-1}}^{\mathsmaller W}
	= \mf t_{s-1} \sse \mf t.
\end{equation}
It follows that $\phi_{s-1} \sse \phi_s$,
and so $\mf t_{s-1}$ is also a flat of $W_s$,
whence as usual $W_{s-1} = (W_s)_{\mf t_{s-1}} \sse W_s$.
Finally,
by construction,
one has $g'_{s-1} W_{s-1} = g'_s g W_{s-1}$,
which is indeed a Levi subcoset of $\mb L_s$.
Now iterate.

\subsection{Proof of Lem.-Def.~\ref{lem:twisted_cabled_weyl}}
\label{proof:lem_twisted_cabled_weyl}

For the first statement,
Lem.~\ref{lem:coset_normalizer_is_stabilizer} implies that $W^{\ul{\mb L}} \sse W^{\mb L_1} \sse N_W(W_1)$,
so it is enough to establish the inclusions $W_1 \sse W^{\mb L_i}$ for all $i \in \set{1,\dc,s}$.
Now,
choosing $g_i \in N_W(W_i)$ such that $\mb L_i = g_i W_i \sse W$,
one has
\begin{equation}
	h.\mb L_i
	= h g_i \cdot W_i
	= g_i (h' W_i)
	= \mb L_i,
	\qquad
	h \in W_1,
	\quad h'
	\ceqq g_i^{-1} h g_i \in W_i,
\end{equation}
using that $W_1 \sse W_i$ for all $i$---%
twice.

For the third statement,
the inclusion $W^{\ul{\mb L}} \sse N_W(\mf t_{\bm\phi})$ follows from Lem.~\ref{lem:coset_normalizer_is_stabilizer},
and $W^{\ul{\mb L}} \sse W^{\mb L_1}$ by construction.
For the converse inclusion,
choose $h \in N_W(\mf t_{\bm\phi})$ such that $h.\mb L_1 = \mb L_1$,
i.e.,
such that $hg_1h^{-1} = g_1 g'_1 \in W$ for some $g'_1 \in W_1$.
For any $i \in \set{1,\dc,s}$,
the fact that $g_1 \in \mb L_1 \sse \mb L_i$ implies that there exists also $g'_i \in W_i$ such that $g_1 = g_ig'_i \in W$.
Therefore,
one has
\begin{equation}
	g_i g'_i g'_1
	= g_1 g'_1
	= hg_1h^{-1}
	= h(g_ig'_i) h^{-1}
	= h g_i h^{-1} \cdot \wt g'_i,
	\qquad \wt g'_i
	\ceqq h g'_i h^{-1} \in W,
\end{equation}
whence $h g_ih^{-1} = g_i \cdot \bigl(g'_i g'_1 (\wt g'_i)^{-1}\bigr) \in \mb L_i$,
using both $W_1 \sse W_i$ and $h W_i h^{-1} = W_i$.

It remains to prove the equivalence between~\eqref{eq:nested_flat_density}--\eqref{eq:untangled_nested_flat_density},
which we do recursively,
for $i \in \set{s,\dc,1}$ (starting from the top).
The base is clear,
as (by construction) one has
\begin{equation}
	\bm E_{\geq s}
	= \mf t_{\phi_s}(\mb L_s,\zeta_r^s) \sse \mf t_{\phi_s},
	\qquad W_{s+1}
	= W,
\end{equation}
and so the two conditions literally match up.
Now set also $E_j \ceqq \mf t_{\phi_j}(\mb L_j,\zeta_r^j)$ for $j \in \set{1,\dc,s}$,
whence:
(i) one has $\bm E_{\geq i} = \sum_{j = i}^s E_j$;
and (ii) the flat $\ol{\bm E_{\geq i-1}}^{\mathsmaller W} \sse \mf t$ is the fixed-point space of the parabolic subgroup $W'_{i-1} \ceqq W_{E_{i-1}} \cap W_{\bm E_{\geq i}} \sse W$.
The inclusions $E_{i-1} \sse \mf t_{\phi_{i-1}}$ and $\bm E_{\geq i} \sse \mf t_{\phi_i}$ yield $W_{i-1} \sse W_{E_{i-1}}$ and $W_i \sse W_{\bm E_{\geq i}}$,
respectively,
and so one always has $W_{i-1} \sse W'_{i-1}$.
Now suppose that~\eqref{eq:nested_flat_density} holds at the step $i-1$,
and that~\eqref{eq:untangled_nested_flat_density} holds at the step $i$,
for some $i \in \set{1,\dc,s}$.
I.e.,
one has $\ol{E_{i-1}}^{\mathsmaller{W_i}} = \mf t_{\phi_{i-1}}$ and $W_{\bm E_{\geq i}} = W_i$.
Then one also has the opposite inclusion $W'_{i-1} \sse W_{i-1}$:
there exists a vector $A_{i-1} \in E_{i-1}$ such that the equality $g(A_{i-1}) = A_{i-1} \in \mf t_{\phi_{i-1}}$ implies that $g \in W_{i-1}$,
\emph{provided that} $g \in W_i$.
This yields $\ol{\bm E_{\geq i-1}}^{\mathsmaller W} = \mf t_{\phi_{i-1}}$.
Conversely,
suppose that~\eqref{eq:untangled_nested_flat_density} holds at the step $i - 1$.
I.e.,
the space $\bm E_{\geq i-1}$ is flat-dense in $\mf t_{\phi_{i-1}}$---%
with respect to $W$.
The inclusion $E_{i-1} \sse \mf t_{\phi_{i-1}}$ yields $\ol{E_{i-1}}^{\mathsmaller{W_i}} \sse \mf t_{\phi_{i-1}}$,
as $\mf t_{\phi_{i-1}} \sse \mf t$ is a flat for the reflection group $(\mf t,W_i)$,
and we must establish equality.
So suppose by contradiction that there exists a vector $A_{i-1} \in \mf t_{\phi_{i-1}}$,
as well as a reflecting hyperplane $H \in \mc A(W_i)$ (for $W_i$),
such that $E_{i-1} \sse H$ and $A_{i-1} \not\in H$.
By Steinberg's theorem~\cite{steinberg_1964_differential_equations_invariant_under_finite_reflection_groups},
the reflection arrangement $\mc A(W_i) \sse \mc A(W)$ consists of the set of reflecting hyperplanes of $W$ containing $\mf t_{\phi_i} \sse \mf t$,
and so one has $E_{i-1} + \mf t_{\phi_i} \sse H$.
A fortiori $\bm E_{\geq i-1} = E_{i-1} + \bm E_{\geq i} \sse H$,
which (by hypothesis) implies that $\mf t_{\phi_{i-1}} = \ol{\bm E_{\geq i-1}}^{\mathsmaller W} \sse H$.
One concludes that $A_{i-1} \not\in \mf t_{\phi_{i-1}}$:
absurd.

\subsection{Proof of Lem.-Def.~\ref{lem:principal_levi_subcoset_filtrations}}
\label{proof:lem_principal_levi_subcoset_filtrations}

For the first statement:
\begin{enumerate}
	\item[(a)]
	      The first substatement is due to the equality $g W_1 = \mb L_1 \sse W$.

	\item[(b)]
	      The second substatement is due to the following observation:
	      if $g W_i = g_i W_i \sse W$,
	      and if $g_i \in N_W(W_i)$,
	      then $g$ normalizes $W_i$ too ($i \in \set{1,\dc,s}$).
	      Now observe that an element $g \in W$ normalizes $W_i$ if and only if it preserves its fixed-locus,
	      i.e.,
	      the flat $\mf t_{\phi_i} \sse \mf t$.

	\item[(c)]
	      For the last substatement,
	      we are now in the situation of \S~\ref{sec:twisted_cabled_weyl}:
	      indeed,
	      the condition that the restricted eigenspace $\mf t_{\phi_i}(\mb L_i,\zeta_r^i) = \mf t_{\phi_i}(g_{\phi_i},\zeta_r^i)$ is flat-dense in $\mf t_{\phi_i}$ for all $i$,
	      \emph{with respect to} $W_{i+1}$,
	      means precisely that $g$ possesses eigenvectors $(A_1,\dc,A_s)$ in such `nested' general positions.
	      The caveat is that $A_1 \in \mf t_{\phi_{A_1}} \sse \mf t_{\phi_1}$ might be annihilated by roots in $\Phi \sm \phi_2 \sse \Phi$---%
	      when $s \geq 2$.
	      Nonetheless,
	      let $A'_s \ceqq A_s$,
	      and observe that there exist scalars $\lambda_s,\lambda_{s-1} \in \mb C$ such that the linear combination $A'_{s-1} \ceqq \lambda_s A'_s + \lambda_{s-1} A_{s-1} \in \mf t$ is annihilated precisely by the Levi subsystem $\phi_{s-1} = \phi_{A_s} \cap \phi_{A_{s-1}} \sse \Phi$;
	      it suffices,
	      e.g.,
	      to choose $\lambda_s,\lambda_{s-1} \neq 0$ such that
	      \begin{equation}
		      \frac{\lambda_s}{\lambda_{s-1}} \neq -\frac{\alpha_{s-1}}{\alpha_s} \in \mb C^{\ts},
		      \qquad \alpha_i
		      \ceqq \braket {\alpha,A_i } \neq 0,
		      \quad \alpha \in \Phi \sm (\phi_s \cup \phi_{s-1}).
	      \end{equation}
	      (There are \emph{finitely many} ratios to avoid;
	      incidentally,
	      one may well take $\lambda_s \ceqq 1$.)
	      Iterating,
	      one shows---%
	      inductively---%
	      that there exists a new tuple $(A'_s,A'_{s-1},\dc,A'_1) \in \mf t^s$ such that:
	      (i) one has $A'_i \in \mf t_{\phi_i} \sm \bigl( \bigcup_{\Phi \sm \phi_i} H_\alpha \bigr)$;
	      and (ii) one also has $A'_i \in \spann_{\mb C} \set{A_i,\dc,A_s}$.
	      In particular,
	      the lowest vector $A = A'_1$ has Levi annihilator $\phi_A = \phi_1 \sse \Phi$,
	      and by construction it lies in
	      \begin{equation}
		      \label{eq:r_th_root_eigenspace}
		      \mf t(g,r)
		      \ceqq
		      \bops_{\mb Z \slash r \mb Z} \mf t(g,\zeta^i_r)
		      = \bops_{\zeta^r = 1} \mf t(g,\zeta) \sse \mf t.\fn{
			      The notation is unambiguous:
			      the number $\zeta_r \in \mb C^{\ts}$ is a positive integer if and only if $r = 1$,
			      whence $\zeta_r = 1$ and $\mf t(g,\zeta_r) = \mf t(g,r) \sse \mf t$ is the fixed-point space of $g$---%
			      a flat of $W$.}
	      \end{equation}
	      But the latter vector space is precisely $\mf t(g^r,1) = \ker(g^r - \Id_{\mf t})$,
	      i.e.,
	      the fixed-point space of $g^r \in W$,
	      and the conclusion follows from Lem.~\ref{lem:about_pointwise_stabilizers}.
\end{enumerate}

For the second statement,
choose $g \in \mb L_1$:
then $g \in \mb L_i$ for all $i \in \set{1,\dc,s}$,
as the cosets are nested,
and so $\mb L_i = g W_i$ for all $i$.

\begin{rema}
	There is another proof of Lem.-Def.~\ref{lem:principal_levi_subcoset_filtrations}~(1.c),
	based on~\eqref{eq:untangled_nested_flat_density}.
	Namely,
	that condition implies that the restriction $\eval[1]{g^r}_{\bm E}$ is \emph{trivial},
	setting $\bm E \ceqq \bm E_{\geq 1}$,
	as $g$ acts as the scalar $\zeta_r^j$ on the eigenspaces $E_j = \mf t_{\phi_j}(\mb L_j,\zeta_r^j) \sse \mf t_{\phi_j}$,
	for $j \in \set{1,\dc,s}$;
	and the subspace $\bm E \sse \mf t_{\phi_1}$ if flat-dense in $\mf t_{\phi_1}$.
\end{rema}

\subsection{Proof of Lem.~\ref{lem:irr_type_levi_subcosets_equivariance}}
\label{proof:lem_irr_type_levi_subcoset_equivariance}

Choose an $r$-Galois-closed irregular type $\wh Q$,
and let $\wh Q' \ceqq h \bigl( \wh Q \bigr)$,
for some $h \in W$.
Then the underlying Levi-subsystem filtrations $\bm\phi \ceqq \bm\phi \bigl( \wh Q \bigr) = (\phi_1,\dc,\phi_s)$ and $\bm\phi' \ceqq \bm\phi \bigl( \wh Q' \bigr) = (\phi'_1,\dc,\phi'_s)$ satisfy
\begin{equation}
	\bm\phi'
	= h.\bm\phi
	= \bigl( h(\phi_1),\dc,h(\phi_s) \bigr),
\end{equation}
by the $W$-equivariance of annihilators/kernels in $\mf t$ and $\mf t^{\dual}$ (cf.~\S~\ref{sec:complex_real_rational_reflection_groups}).
The conclusion now follows from~\eqref{eq:conjugated_orbit_generators} +~Lem.-Def.~\ref{lem:principal_levi_subcoset_filtrations},
since one has $\mb L_r \bigl( \wh Q \bigr) = \mb L_1$ and $\mb L_r \bigl( \wh Q' \bigr) = \mb L'_1$,
writing as usual
\begin{equation}
	\ul{\mb L} \ceqq \ul{\mb L} \bigl( \wh Q \bigr) = (\mb L_1,\dc,\mb L_s),
	\qquad \ul{\mb L}' \ceqq \ul{\mb L} \bigl( \wh Q' \bigr) = (\mb L'_1,\dc,\mb L'_s).
\end{equation}
Indeed,
if $g' \in \mb L'_1$ is a generator of $\ul{\mb L}'$,
then one has $g' = h g h^{-1}$ for some generator $g \in \mb L_1$ of $\ul{\mb L}$,
and this yields
\begin{equation}
	\mb L'_i
	= g'W'_i
	= h g h^{-1} \cdot h W_i h^{-1}
	= h (g W_i) h^{-1}
	= h.\mb L_i \sse W,
	\qquad i \in \set{1,\dc,s}.
\end{equation}
(Setting $W_i \ceqq W_{\mf t_{\phi_i}}$,
etc.)

\subsection{Proof of Prop.~\ref{prop:modular_examples}}
\label{proof:prop_modular_examples}

The hypothesis implies that $(g^r)_\phi = g_\phi^r = 1 \in W(\phi)$,
and so one has
\begin{equation}
	\mf t_\phi(g_\phi,r)
	\ceqq \bops_{\mb Z \slash r \mb Z} \mf t_\phi(g_\phi,\zeta_r^i)
	= \mf t_\phi,
\end{equation}
cf.~\eqref{eq:r_th_root_eigenspace}.
Therefore,
for any vector $A$ in general position in $\mf t_\phi$ there exist \emph{unique} eigenvectors $A_i \in \mf t_\phi(g_\phi,\zeta_r^i)$ ($i \in \set{1,\dc,r}$) such that $A = \sum_{i = 1}^r A_i \in \mf t_\phi$.
In turn:
(i) the element $g$ now generates the $r$-Galois-orbit of the irregular type $\wh Q  = \wh Q(g,A) \ceqq \sum_{i = 1}^r A_i w^{-i}$;
and (ii) taking $\bm\phi \ceqq \bm\phi \bigl( \wh Q \bigr)$,
one has $\phi_1 = \phi$,
because
\begin{equation}
	\phi \sse \bigcap_{\mb Z \slash r \mb Z} \phi_{A_i}
	\sse \phi_A = \phi.
\end{equation}

\subsection{Proof of Lem.~\ref{prop:relative_refl_group_A}}
\label{proof:lem_relative_refl_group_A}

For the first statement,
choose $g \in N_W(\mf t_\phi)$:
it must preserve $\phi$ (cf.~Lem.~\ref{lem:about_setwise_stabilizers}).
Then~\eqref{eq:permutation_action} implies that $i \sim_A j$ if and only if one has $g(i) \sim_A g(j)$,
and so $g$ induces a permutation of $J_\phi$.
Moreover,
in the given identification $\mf t_\phi \simeq V_{m_\phi}^+ = \mb C^{m_\phi}$,
the induced action $\ol i \mt g(\ol i) \ceqq \ol{g(i)}$---%
on the coordinates---%
matches up with the restriction/class $g_\phi \in W(\phi)$.

For the second statement,
the reflection arrangement $\mc A_A(\phi)$ consists of the diagonals $H_{\ol i \, \ol j} \sse \mf t_\phi$ where the `$\ol i$-th' and `$\ol j$-th' coordinate coincide,
for $\ol i \neq \ol j \in J_\phi$.
If now $\alpha = \alpha_{ij} \in \Phi \sm \phi$ for suitable $i \neq j \in \ul m^+$,
i.e.,
if one has $i \not\sim_A j$,
then $\alpha_\phi \neq 0$ and
\begin{equation}
	H_\alpha(\phi)
	= \ker(\alpha) \cap \mf t_\phi
	= H_{\ol i \, \ol j} \in \mc A_A(\phi).
\end{equation}
Conversely,
if $\ol i \neq \ol j$ then one can choose \emph{any} $k \in \ol i$ and $l \in \ol j$,
so that $H_{\ol i \, \ol j} = H_{\alpha_{kl}}(\phi)$---%
and \emph{all} diagonals arise.
(This uses the fact that the blocks are \emph{nonempty}.)

\subsection{Proof of Lem.~\ref{lem:transitivity_BC_relation}}
\label{proof:lem_transitivity_BC_relation}

Choose $i,j,k \in \ul m^\pm$ such that $i \sim_C j$ and $j \sim_C k$.
If two indices coincide the conclusion is clear,
so assume that $i \neq j \neq k \neq i \in \ul m^\pm$.

Then one has $\wt \alpha_{ij},\wt \alpha_{jk} \in \phi$,
whence
\begin{equation}
	\wt \alpha_{ik}
	= \wt \alpha_{ij} + \wt \alpha_{jk} \in (\phi + \phi) \cap \Phi \sse \phi,
\end{equation}
since $\phi$ is---%
in particular---%
closed under the sums of its elements which remain roots.
So indeed $i \sim_C k$.

\subsection{Proof of Prop.~\ref{prop:relative_refl_group_BC}}
\label{proof:prop_relative_refl_group_BC}

The first statement is proven as~Prop.~\ref{prop:relative_refl_group_A}~(1.),
noting that $g \in N_W(\mf t_\phi) \sse \mf S^\pm_m$ also satisfies $g(i) + g(-i) = 0$ for all $i \in \ul m^\pm$,
so that $g$ induces a permutation of $J_\phi$ which commutes with the involution and fixes $0$---%
if present.

For the second statement,
choose first $i \neq j \in \ul m^\pm$ such that $\wt \alpha \ceqq \wt \alpha_{ij} \not\in \phi$,
i.e.,
such that $i \not\sim_C j$.
Then $\wt \alpha_\phi \neq 0$,
and the identifications $\mf t_\phi \simeq \wt V_{m_\phi} \simeq \mb C^{m_\phi}$ yield three cases:
\begin{enumerate}
	\item
	      if $\ol j = 0 \in J_\phi$,
	      then $\wt\alpha_\phi$ matches up with the linear functional $\wt \alpha_{\ol i}$,
	      which is a short root of type $B\!C_{m_\phi}$,
	      whose kernel is the coordinate hyperplane $\wt H_{\ol i} \sse \mb C^{m_\phi}$;

	\item
	      if instead $\ol i = 0$,
	      just note that $\wt \alpha_{ij} = - \wt \alpha_{ji}$ (and the kernels are the same),
	      so this reduces to the previous case;

	\item
	      otherwise $\ol i \neq 0 \neq \ol j$,
	      and there are two subcases:
	      \begin{enumerate}
		      \item
		            if $i \neq \pm j$,
		            then $\wt \alpha_\phi$ matches up with $\wt \alpha_{\ol i\,\ol j}$,
		            which is a medium root of type $B\!C_{m_\phi}$,
		            whose kernel is a `diagonal' or `antidiagonal' $\wt H_{\ol i \, \ol j} \sse \mb C^{m_\phi}$ (where coordinates are equal/opposite,
		            respectively);

		      \item
		            and if $i + j = 0$,
		            then $\wt \alpha_\phi = (\wt \alpha_{i,-i})_\phi$ matches up with $\wt \alpha_{\ol i,-\ol i} = 2\wt \alpha_{\ol i}$,
		            which is a long root of type $B\!C_{m_\phi}$,
		            whose kernel is still a coordinate hyperplane.
		            (Recall that $\ol{-i} = - \ol i \in J_\phi$.)
	      \end{enumerate}
\end{enumerate}
Conversely,
the reflection arrangement $\mc A_{B\!C}(\phi)$ consists precisely of diagonals,
antidiagonals,
and coordinate hyperplanes (cf.~\eqref{eq:BC_root_hyperplanes});
and they all arise from restrictions as above.
Indeed:
(i) the diagonals/antidiagonals arise by restricting $\wt \alpha_{kl}$ for indices $k \in \ol i$ and $l \in \ol j$ for any pair of blocks such that $\ol i \neq \pm \ol j \in J_\phi$;
and (ii) the coordinate hyperplanes arise by choosing,
e.g.,
indices $k \in \ol i$ and $-k \in -\ol i$,
for any pair of---%
nonzero---%
opposite blocks $\set{ \ol i,-\ol i} \sse J_\phi$.
(The latter is \emph{not} available in type $D$,
which makes it the most delicate classical case to treat.)

\subsection{Proof of Lem.~\ref{lem:inverse_bijection_BC}}
\label{proof:lem_inverse_bijection_BC}

Let $J = J_\sim$ be the set of $\sim$-equivalence classes,
and write as usual the quotient map as $i \mt \ol i \cl \ul m^\pm \thra J$.
Now break symmetry---%
as in \S~\ref{sec:signed_partitions_to_relative_BC_groups}---%
by choosing a `positive' class $\ol i$ within each pair $\set{\ol i,-\ol i} \sse J$ of \emph{nonzero} classes.
Then choose nonvanishing numbers $\wt \lambda_{\ol i} \in \mb C^{\ts}$ for all $\ol i \neq \ol 0$,
such that:
(i) $\wt \lambda_{-\ol i} = - \wt \lambda_{\ol i}$ for $\ol i \in J \sm \set{\ol 0}$;
and (ii) $\wt \lambda_{\ol i} \neq \pm \wt \lambda_{\ol j}$ if $\ol i \neq \pm \ol j$.

Consider then the vector $\bm\lambda = (\lambda_1,\dc,\lambda_m,\lambda_{-1},\dc,\lambda_{-m}) \in V^+_m \ops V^-_m$,
where
\begin{equation}
	\lambda_i \ceqq
	\begin{cases}
		\wt \lambda_{\ol i}, & \quad \ol i \neq \ol 0, \\
		0,                   & \quad \ol i = \ol 0.
	\end{cases}
\end{equation}
By construction,
this vector lies in the subspace $\mf t \simeq \wt V_m \sse V^+_m \ops V^-_m$.
Finally,
the Levi annihilator of $\bm\lambda$ now equals $\phi_\sim$---%
and so $\bm\lambda$ is in general position in $\mf t_{\phi_\sim} \sse \mf t$.
Indeed,
if $i \neq j \in \ul m^\pm$ then one has $\wt \alpha_{ij}(\bm \lambda) = \lambda_i - \lambda_j$,
which vanishes if and only if $i \sim j$.
(As usual,
this covers the case $j = -i$,
where the vanishing condition reads $2\lambda_i = 0$,
i.e.,
$\ol i = \ol 0$.)

\subsection{Proof of Lem.-Def.~\ref{lem:type_D_part}}
\label{proof:lem_type_D_part}

The first statement can be established by the same argument as in Rmk.~\ref{rmk:type_C_component}:
choose a vector of $\wt V_m$ whose components vanish over the zero block,
and which is `generic' otherwise.,
so that $\phi_D$ is its Levi-annihilator.

For the nontrivial implication of the second statement,
suppose that $\ol 0 \neq \vn$,
i.e.,
that there exists $i \in \ul m^\pm$ such that $i \sim_D (-i)$.
In view of Lem.-Def.~\ref{lem:good_relation_type_D}~(2.),
this means that there is a length-$k$ chain
\begin{equation}
	i \eqqc i_0,i_1,\dc,i_{k-1},i_k \ceqq -i,
	\qquad k \geq 1,
\end{equation}
of indices---%
in $\ul m^\pm$---%
such that for all $l \in \set{0,\dc,k-1}$ one has $i_l \neq \pm i_{l+1}$ and $\wt \alpha_{i_l i_{l+1}} \in \phi$;
and since $2 \wt \alpha_i = \wt \alpha_{i,-i} \in \Phi_{B\!C}(m) \sm \Phi_D(m)$ one must have $k \geq 2$.
Now:
\begin{enumerate}
	\item
	      If $k = 2$,
	      then the intermediate index $j \ceqq i_1$ is such that $\Set{\wt \alpha_{i,j},\wt\alpha_{j,-i}} \sse \phi$:
	      so one has $\wt \alpha_{-i,-j} = - \wt \alpha_{ij} \in \phi = -\phi$,
	      and then
	      \begin{equation}
		      \label{eq:bc_root_in_span}
		      \wt \alpha_{j,-j}
		      = \wt \alpha_{j,-i} + \wt \alpha_{-i,-j} \in \spann_{\mb C}(\phi).
	      \end{equation}
	      Taking kernels,
	      this now yields
	      \begin{equation}
		      \mf t_\phi \sse \ker \bigl( \wt \alpha_{j,-j} \bigr)
		      = \wt H_j,
	      \end{equation}
	      which implies that $j \in \ol 0$ as well,
	      whence $\wt \alpha_{ij} \in \phi_D \neq \vn$.

	\item
	      If instead $k \geq 3$,
	      just observe that
	      \begin{equation}
		      \wt \alpha_{i,i_{k-1}}
		      = \wt \alpha_{i_0i_1} + \dm + \wt\alpha_{i_{k-2}i_{k-1}} \in \spann_{\mb Q}(\phi) \cap \Phi
		      = \phi,
	      \end{equation}
	      and now one can repeat the same argument with $j \ceqq i_{k-1} \neq \pm i$.
\end{enumerate}
Incidentally,
we have also proven the following:
if $\ol 0 \neq \vn$,
then one has $\abs{\ol 0} \geq 4$.

For the third statement,
again note that if $\wt \alpha_{ij} \in \phi_D \neq \vn$ then one has $\set{ \wt \alpha_{ij},\wt \alpha_{i,-j} } \sse \phi_D$,
since $\ol 0 = -\ol 0$;
and the two roots are \emph{linearly independent}---%
using that $i \neq \pm j$.

As per the last statement,
up to dualizing,
it is enough to show that
\begin{equation}
	\label{eq:span_comparison}
	\spann_{\mb C}(\phi_D)
	= \bops_{i \in \ol 0^+} \mb C \cdot \wt \alpha_{i,-i} \sse \mf t^{\dual}.
\end{equation}
If $\ol 0 = \ol 0^+ = \vn$ the conclusion is clear,
so suppose that $\ol 0^+ \neq \vn$.
Now the mutual inclusions of the vector spaces of~\eqref{eq:span_comparison} follow from the identities
\begin{equation}
	\wt \alpha_{i,\pm j}
	= \frac 1 2 \bigl( \wt \alpha_{i,-i} \mp \wt \alpha_{j,-j} \bigr),
\end{equation}
and from the usual
\begin{equation}
	\wt \alpha_{i,-i}
	= \wt \alpha_{ij} + \wt \alpha_{j,-i}
	= \wt \alpha_{ij} + \wt \alpha_{i,-j},
\end{equation}
for $i \neq j \in \ol 0^+$.
(For the latter,
strictly speaking,
for any $i \in \ol 0^+$ one chooses \emph{any} index $j \in \ol 0^+ \sm \set{i}$,
which exists since $\lvert \ol 0^+ \rvert \geq 2$,
by the proof of the second statement.)

\subsection{Proof of Prop.-Def.~\ref{prop:relative_refl_group_D}}
\label{proof:prop_relative_refl_group_D}

For the second statement,
choose as usual $g \in N_W(\mf t_\phi)$.
This is a signed-permutation preserving $\phi$,
and one has $g(\phi_D) = \phi_D$ (cf.~\eqref{eq:type_D_preservation}).
Therefore,
as in type $B\!C$,
the class/restriction $g_\phi \in W(\phi)$ induces a signed permutation of the set of nonzero pairs for the relation $\sim_D$,
under the usual identifications $\mf t_\phi \simeq \wt V_{m_\phi} \simeq \mb C^{m_\phi}$.

For the last statement,
the proof that every restricted hyperplane $H_\alpha(\phi)$ ($\alpha \in \Phi \sm \phi$) is either a coordinate hyperplane,
a diagonal,
or an antidiagonal,
is as in type $B\!C$;
and these hyperplanes exhaust $\mc A_D(\phi)$.
The crux of the matter is proving that \emph{all} such hyperplanes arise from restrictions of complementary roots.
\begin{description}
	\item[Diagonals/antidiagonals]
	      Given $\ol i \neq \pm \ol j \in \ul m^\pm_\phi$,
	      choose $k \in \ol i$ and $l \in \ol j$---%
	      so that $k \neq \pm l \in \ul m^\pm$.
	      Then $\wt \alpha \ceqq \wt \alpha_{kl} \in \Phi$ does \emph{not} lie in $\phi$,
	      for otherwise one would have $k \approx_D l$,
	      whence a fortiori $k \approx_D^* l$,
	      which yields $k \sim_D l$ (cf.~Lem.-Def.~\ref{lem:good_relation_type_D},
	      and its proof).
	      Now $\wt \alpha_\phi = \wt \alpha_{\ol i \, \ol j}$,
	      and so all the type-$D$ root hyperplanes $\wt H_{\ol i \, \ol j} \sse \mf t_\phi$ do arise.

	\item[Coordinates]
	      Given $\ol i \in \ul m_\phi^\pm$,
	      choose $k \in \ol i$.
	      The crucial point now is that $\ol 0 \neq \vn$ (by hypothesis,
	      in view of Lem.-Def.~\ref{lem:type_D_part}~(2.)),
	      and so we can also choose $l \in \ol 0$.
	      It follows that $k \neq \pm l \in \ul m^\pm$,
	      and now $\wt \alpha \ceqq \wt \alpha_{kl} \in \Phi$ does \emph{not} lie in $\phi$.
	      Indeed,
	      if $k \approx_D l$ then one would find---%
	      as above---%
	      that $k \sim_D l$,
	      whence $k \in \ol 0$.
	      Now $\wt \alpha_\phi = \wt \alpha_{\ol i}$ is a short (type-$B$) root in $\wt V_{m_\phi}$,
	      and so the hyperplanes $\wt H_{\ol i} = \wt H_{\ol i,-\ol i} \sse \mf t_\phi$ do arise. \qedhere
\end{description}

\subsection{Proof of Prop.~\ref{prop:nonreflection_arrangement}}
\label{proof:prop_nonreflection_arrangement}

For the first statement,
if $k = k' = 1$ one then removes three distinct lines from $\mb C^2$.
This situation can be brought back into a copy of the reflection arrangement of type $A_2$,
by the transitive action of $(\on P)\!\GL_2(\mb C)$.
Explicitly,
e.g.,
one can change coordinates as
\begin{equation}
	(\lambda,\mu) \lmt (\lambda_1,\lambda_2,\lambda_3),
	\qquad \lambda_1 \ceqq 3\lambda - \mu,
	\quad \lambda_2 \ceqq -3\lambda - \mu,
	\quad \lambda_3 \ceqq 2\mu,
\end{equation}
noting that $\lambda_1 + \lambda_2 + \lambda_3 = 0$,
and that $\lambda = 0$ if and only if $\lambda_1 = \lambda_2$,
etc.

For the second statement,
one can check that there are \emph{no} finite irreducible complex reflection groups of the Shephard--Todd list with the correct rank/hyperplane count.
Namely,
consider first the infinite family $G(\rho,e,m_\phi)$,
for integers $\rho,e \geq 1$ such that $e \mid \rho$.
If $\mc A(\rho,e,m_\phi) \sse \mb P \bigl( \smash{\bigl( \mb C^{m_\phi} \bigr)}^{\!\dual} \bigr)$ is their arrangement of reflection hyperplanes,
then the number of hyperplanes is
\begin{equation}
	\label{eq:hyperplane_number}
	N'(\rho,e,m_\phi)
	\ceqq \abs{ \mc A(\rho,e,m_\phi) }
	= \begin{cases}
		\binom{m_\phi}2 \rho,          & \quad e = \rho, \\
		\binom{m_\phi}2 \rho + m_\phi, & \quad e < \rho.
	\end{cases}
\end{equation}
(The former corresponds to~\eqref{eq:standard_complement_type_D},
and the latter to~\eqref{eq:standard_complement},
cf.~again~\cite[Lem.~3.3]{broue_malle_rouquier_1998_complex_reflection_groups_braid_groups_hecke_algebras}.)
Similarly,
if $\mc A(k,k')$ is the set of hyperplanes removed from $\mb C^{m_\phi}$ in~\eqref{eq:exotic_arrangement},
then
\begin{equation}
	\label{eq:hyperplane_number_2}
	N(k,k')
	\ceqq \abs{ \mc A(k,k') }
	= m_\phi(m_\phi - 1) + k
	= m_\phi^2 - k'.
\end{equation}
Now equating~\eqref{eq:hyperplane_number}--\eqref{eq:hyperplane_number_2} yields in particular
\begin{equation}
	\label{eq:ramification_rational}
	\rho
	= \begin{cases}
		2 + \frac{2k}{l_\phi}, \quad  & e = \rho, \\
		2 - \frac{2k'}{l_\phi}, \quad & e < \rho,
	\end{cases}
\end{equation}
where $l_\phi \ceqq m_\phi(m_\phi - 1) \geq 6$.
But these identities are impossible in $\mb Z_{> 0}$:
(i) the former forces $2 < \rho < 3$,
since $2k < m_\phi k < l_\phi$,
recalling that $2 < m_\phi$ and $k < m_\phi$;
and (ii) the latter forces $1 < \rho < 2$,
since analogously $2k' < l_\phi$.
As for the $34$ exceptional groups $G_4,\dc,G_{37}$,
let $\mc A_i = \mc A(G_i)$ be their reflection arrangement.
Then consider the minimal number of such hyperplanes in each relevant rank,
viz.,
\begin{equation}
	\label{eq:hyperplane_lower_bound}
	N''(m_\phi)
	\ceqq \min \Set{ \abs{\mc A_i} | i \in \set{4,\dc,37}, \quad \rk(G_i) = m_\phi },
	\qquad m_\phi \in \set{3,\dc,8}.
\end{equation}
One can check that in this range one has
\begin{equation}
	\max \Set{ N(k,k') | 0 < k,k' < m_\phi}
	= m_\phi^2 - 1 < N''(m_\phi),
\end{equation}
as per Tab.~\ref{tab:exceptional_hyperplane_bound} (cf.~\eqref{eq:hyperplane_number} and~\cite{orlik_terao_1992_arrangements_of_hyperplanes};
note the near miss in rank $6$).

\begin{table}
	\centering
	\begin{tabular}{cccc}
		\toprule
		Rank ($m_\phi$) & $m_\phi^2 - 1$ & $N''(m_\phi)$ & Minimizing group  \\
		\midrule
		3               & 8              & 12            & $G_{25}$          \\
		4               & 15             & 24            & $G_{28} = W(F_4)$ \\
		5               & 24             & 45            & $G_{33}$          \\
		6               & 35             & 36            & $G_{35} = W(E_6)$ \\
		7               & 48             & 63            & $G_{36} = W(E_7)$ \\
		8               & 63             & 120           & $G_{37} = W(E_8)$ \\
		\bottomrule
	\end{tabular}
	\vspace{5pt}
	\caption{\label{tab:exceptional_hyperplane_bound}
		Lower bounds for exceptional complex-reflection arrangement cardinalities vs.~upper bounds for exotic arrangement cardinalities.}
\end{table}

The last statement follows from:
\begin{enumerate}
	\item
	      the existence of a locally-trivial fibre bundle $\mc M^{\musDoubleSharp}(k,1) \thra \mc M^\sharp(2,k)$,
	      whose fibres are $(2k)$-punctured complex lines,
	      if $k' = 1$ (cf.~\cite[Prop.~8.1]{doucot_rembado_tamiozzo_2026_local_wild_mapping_class_groups_and_cabled_braids});

	\item
	      and the $2$-ramified case of~\cite[Thm.~1]{amend_deligne_roehrle_2020_on_the_k_pi_1_problem_for_restrictions_of_complex_reflection_arrangements},
	      in general.
\end{enumerate}

\subsection{Proof of Lem.~\ref{lem:full_gives_pointed_BC}}
\label{proof:lem_full_gives_pointed_BC}

We must prove that any Weyl-orbit of full irregular types contains a pointed one.

To this end,
fix integers $n_i, r_i \geq 1$ and consider a full irregular type $\wt Q_i$,
of rank $n_i \cdot r_i$.
If the Galois-orbit of $\wt Q_i$ is generated by $g^+_i \ceqq \prod_{j = 1}^{n_i} \wt c_{ij} \in W_{B\!C}(n_i \cdot r_i)$,
in the notation of~\eqref{eq:nonspecial_positive_cycles},
then the former is a concatenation of lists
\begin{equation}
	\wt Q_i
	= \bigl( \, \wt l^{(+,0)},\dc,\wt l^{(+,r_i-1)},-\wt l^{(+,0)},\dc,-\wt l^{(+,r_i-1)} \bigr),
\end{equation}
where $\wt l^{(+,j)} = \bigl( \wt q_1^{(j)},\dc,\wt q_{n_i}^{(j)} \bigr)$ for $j \in \mb Z \bs r_i \mb Z$,
and for suitable $r_i$-ramified exponential factors $\wt q_1,\dc,\wt q_{n_i}$---%
with tame/nonspecial Stokes circles.
Then there exists a group element
\begin{equation}
	h_i \in \mf S_{n_i} \simeq W_A(n_i) \sse W_{B\!C}(n_i),
\end{equation}
embedded diagonally into $W_{B\!C}(n_i \cdot r_i)$,
such that $h_i \bigl( \wt Q_i \bigr)$ will be pointed:
just permute the `base' list $\wt l^{(+,0)}$ so that the result can be split into maximal sublists of \emph{identical} exponential factors.\fn{
	\label{fn:compatible_pit}
	The lists $\wt l^{(+,j)}$ can also be cyclically permuted by an element of $W_A(r_i)$ which preserves the partition $\set{1,\dc,n_i \cdot r_i} = \set{1,\dc,n_i} \cup \dm \cup \set{(r_i-1)n_i + 1,\dc, n_i \cdot r_i}$,
	using the block-permutation embedding $\mf S_{r_i} \hra \mf S_{n_i \cdot r_i}$.
	This leads to `compatible' pointed irregular types:
	cf.~Def.~\ref{def:compatible_pit}.}

Analogously,
given integers $m_i,r_i' \geq 1$,
consider a full irregular type $\wt Q_i$ of rank $m_i \cdot r_i' \geq 1$.
If its Galois-orbit is generated by $g^-_i \ceqq \prod_{j = 1}^{m_i} \wt c_{ij} \in W_{B\!C}(m_i \cdot r_i')$,
in the notation of~\eqref{eq:special_negative_cycles},
then
\begin{equation}
	\wt Q_i
	= \bigl( \, \wt l^{(-,0)},\dc,\wt l^{(-,r'_i-1)},- \wt l^{(-,0)},\dc,-\wt l^{(-,r'_i-1)} \bigr),
\end{equation}
where $\wt l^{(-,j)} = \bigl( \wt q_1^{(j)},\dc,\wt q_{m_i}^{(j)} \bigr)$ for $j \in \mb Z \bs r'_i \mb Z$,
and for suitable exponential factors $\wt q_1,\dc,\wt q_{m_i}$---%
with special Stokes circles.
Again,
up to acting by the Weyl subgroup of type $A_{m_i}$,
one finds a pointed irregular type.

Now observe that if $g \in W_{B\!C}(m)$ generates the Galois-orbit of a full irregular type $\wt Q$,
then $g' \ceqq hgh^{-1}$ will generate that of $\wt Q' \ceqq h \bigl( \wt Q \bigr)$,
for any other element $h \in W_{B\!C}(m)$.
Then the conclusion follows from~\cite[Prop.~24]{carter_1972_conjugacy_classes_in_the_weyl_group},
since:
(i) the element $g$ decomposes into disjoint positive/negative cycles;
and (ii) it is conjugate to an element whose factors are of the form~\eqref{eq:nonspecial_positive_cycles}--\eqref{eq:special_negative_cycles}.
(The latter \emph{fails} in type $D$,
leading to the notion of `pointed quasi-irregular types' of~\S~\ref{sec:pointed_types_D}.)

\subsection{Proof of Lem.~\ref{lem:type_BC_levels}}
\label{proof:lem_type_BC_levels}

Write $I = \braket q$,
so that $q - a z^{-k}$ only has less singular terms than $q$ for a suitable complex number $a \neq 0$.
If $k$ is an integer,
then either:
(i) $I$ is unramified and $L_A(I) = \vn$;
or (ii) $I$ is ramified and the Galois conjugates of $q$ have the same leading term $az^{-k}$.
In the latter case,
the terms of exponent $k$ cancel in the differences $q^{(i)} - q^{(j)}$,
whence $\max \bigl( L_A(I) \bigr) < k$.

Conversely,
if $k$ is not an integer then $I$ is ramified.
Writing $r_0 \ceqq \on{den}(k)$,
the leading terms of the Galois conjugates of $q$ are all the monomials of the form $a \zeta \cdot z^{-k}$,
with $\zeta^{r_0} = 1$.
Thus,
if $\zeta$ and $\zeta'$ are two distinct $r_0$-th roots of $1$,
the monomial $a(\zeta - \zeta') \cdot z^{-k}$ appears as the leading term of a difference $q^{(i)} - q^{(j)}$,
whence $k\in L_A(I)$.

\subsection{Proof of Prop.~\ref{prop:admissible_inconsequential_bc}}
\label{proof:prop_admissible_inconsequential_bc}

Let $k \in \on{Adm}_{B\!C}(\bm L)$.
If $\bm S \neq \vn$, i.e.,
if $\bm S = \set{k_0}$ is a singleton,
then $k \leq k_0$.
Moreover,
$k$ must be $A$-admissible for $\bm L_A$.
Finally,
if $\bm L^+ \neq \vn$,
i.e.,
if $\bm L^+ = \set{l}$ is a singleton,
then $k$ \emph{cannot} be a good breaking of $\bm L_A$ with $k > l$ (otherwise $\bm L^+ = \set{k}$),
so that $k$ satisfies the desired conditions.

Conversely,
suppose that $k$ satisfies the conditions of the statement,
and consider a Stokes-circle-up-to-sign $\pm I$ with $E(\pm I) = \bm L \cup \set{k}$:
the previous computations of level data now imply $L_{B\!C}(\pm I) = \bm L$.

\subsection{Proof of Lem.~\ref{lem:exterior_slopes_two_circles_bc}}
\label{proof:lem_exterior_slopes_two_circles_bc}

A nonzero slope
\begin{equation}
	S_{ij}^\pm
	\ceqq \on{slope} \bigl( q^{(i)} \pm \wt q^{(j)} \bigr),
	\qquad (i,j) \in \mb Z \bs r \mb Z \ts \mb Z \bs \tilde r \mb Z,
\end{equation}
satisfies one of the following:
(i) either $S_{ij}^\pm > f_{q, \wt q}$,
whence $S_{ij}^\pm$ is equal to an interior slope,
i.e.,
a $B\!C$-level,
for $q_c$;
or (ii) $S_{ij}^\pm \leq f_{q, \wt q}$,
whence $S_{ij}^\pm = f_{q, \wt q}$.
(Else $\braket{\pm q}$ and $\braket{\pm\wt q}$ would have the same truncation at height $f_{q, \wt q}$).

\subsection{Proof of Prop.~\ref{prop:types_of_fission_BC}}
\label{proof:prop_types_of_fission_BC}

By definition,
one has $f_{q,\wt q} = k$ if and only if
\begin{equation}
	q^{(i)} \neq \pm \wt q^{(j)},
	\qquad (i,j) \in \mb Z \bs r \mb Z \ts \mb Z \bs \tilde r \mb Z.
\end{equation}
For any such pair $(i,j)$,
one has $q^{(i)} = \pm \wt q^{(j)}$ if and only if
\begin{equation}
	q_c^{(i)}
	= \pm q_c^{(j)}
	\quad\text{ and }\quad
	a \cdot \exp \bigl(2\pi \sqrt{-1} \cdot ik \bigr)
	= \pm \wt a \cdot \exp \bigl(2\pi \sqrt{-1} \cdot jk \bigr),
\end{equation}
writing $q_c^{(i)} \ceqq \sigma^i(q_c)$.
The former condition is controlled by $B\!C$-level data,
considered above,
so we focus on the latter.

Consider case (I),
where $q_c$ is nonspecial.
Then $q_c^{(i)} + q_c^{(j)} \neq 0$ for all $i,j$,
and $q_c^{(i)} - q_c^{(j)} = 0$ if and only if $i \equiv j \pmod{r_c}$.
In the latter case,
choose an integer $l$ such that $j = i+l r_c$:
the coefficient of exponent $k$ in $q^{(i)} - \wt q^{(j)}$ is
\begin{equation}
	\exp \bigl(2\pi \sqrt{-1} \cdot ik \bigr) \cdot \bigl( a - \wt a \cdot \exp \bigl(2\pi \sqrt{-1} \cdot lr_ck \bigr) \bigr) \in \mb C,
\end{equation}
which vanishes when
\begin{equation}
	a
	= \wt a \cdot \exp \bigl(2\pi \sqrt{-1} \cdot lr_ck \bigr).
\end{equation}
Now note that $r_c k = \frac{\rho n}N$,
setting $\rho \ceqq \frac{r_c}{d \wdg r_c}$, and that $(\rho n) \wdg N = 1$.
If we denote---%
again---%
by $\mb Z \bs N \mb Z \simeq \mu_N$ the group of $N$-th roots of $1$,
it follows that
\begin{equation}
	\label{eq:conjugates_when_fission}
	\Set{ \exp \bigl(2\pi \sqrt{-1} \cdot lr_ck \bigr) | l \in \mb Z \bs r_c \mb Z }
	= \mu_N.
\end{equation}
In conclusion,
$f_{q, \wt q} = k$ if and only if $a \neq \wt a \cdot \zeta$ for all $\zeta \in \mu_N$;
in particular,
if $f_{q, \wt q} = k$ then we cannot have $a = \wt a =0$.

Now consider case (II),
where $q_c$ is special.
As above one has $q_c^{(i)} - q_c^{(j)} = 0$ if and only if $i \equiv j \pmod{r_c}$,
and $q^{(i)} - \wt q^{(j)} \neq 0$ if and only if  $a \notin \wt a \cdot \mu_N \sse \mb C^{\ast}$.
However,
because of specialness,
one might have $q_c^{(i)} + q_c^{(j)} = 0$,
and this happens precisely when $j - i \equiv \frac{r_c}2 \pmod{r_c}$.
In this case,
choose an integer $l$ such that $j = i + (l + \frac 1 2)r_c$:
the coefficient of exponent $k$ in $q^{(i)} + \wt q^{(j)}$ is
\begin{equation}
	\exp \bigl(2\pi \sqrt{-1} \cdot ik \bigr) \cdot \bigl( a + \wt a \cdot \exp\bigl(2\pi \sqrt{-1} \cdot (l + 1 \slash 2)r_ck \bigr) \bigr) \in \mb C,
\end{equation}
which vanishes when
\begin{equation}
	a
	= -\wt a \cdot \exp \bigl(2\pi \sqrt{-1} \cdot (l + 1 \slash 2)r_ck \bigr)
	= -\wt a \cdot \exp \bigl(2\pi \sqrt{-1} \cdot lr_ck \bigr) \cdot e^{\pi \sqrt{-1} \cdot r_ck}.
\end{equation}
By~\eqref{eq:conjugates_when_fission},
this implies that $q^{(i)} + \wt q^{(j)} \neq 0$,
for all $i,j$ such that $q_c^{(i)} + q_c^{(j)} = 0$,
if and only if
\begin{equation}
	\label{eq:condition_fission_+}
	a \notin \bigl( -\wt a \cdot e^{\pi \sqrt{-1} \cdot r_ck} \bigr) \cdot \mu_N \sse \mb C^{\ast}.
\end{equation}
There are now two (mutually-exclusive) subcases.
\begin{enumerate}
	\item
	      If $N > 1$,
	      then $-e^{\pi\sqrt{-1} \cdot r_ck} = -e^{\pi\sqrt{-1} \cdot \frac {\rho n} {N}}$ is a root of unity of order $N$ or $2N$, since $(\rho n) \wdg N = 1$.
	      To see which, we look at its $N$-th power, given by $(-1)^{N+n\rho}$, so using that $\rho n=k r$ the condition \eqref{eq:condition_fission_+} yields
	      \begin{equation}
		      a \notin \bigl( - \wt a \cdot e^{\pi\sqrt{-1} \cdot \frac 1 N} \bigr) \cdot \mu_N
		      \quad\Longleftrightarrow
		      \quad a \notin
		      \begin{cases}
			      \wt a \cdot \mu_N,                                         & \quad N+kr \text{ even},  \\
			      \wt a \cdot e^{\pi \sqrt{-1} \cdot \frac 1 N} \cdot \mu_N, & \quad N+kr \text{ odd}.
		      \end{cases}
	      \end{equation}
	      If both $a$ and $\wt a$ are nonzero,
	      this corresponds to subcase (II.2b.B).

	\item
	      If $N = 1$,
	      then $d$ divides $r_c$,
	      i.e.,
	      $r_c = dd'$ for some integer $d' \geq 1$,
	      and one has $e^{\pi\sqrt{-1} \cdot r_ck} = e^{\pi \sqrt{-1} \cdot d' n}$.
	      There are now two (mutually-exclusive) subsubcases:
	      \begin{enumerate}
		      \item If $d'$ is even,
		            then $e^{\pi\sqrt{-1} \cdot d' n} = 1$,
		            so the condition \eqref{eq:condition_fission_+} yields $a \neq -\wt a$.
		            In this case,
		            $k$ is a breaking of specialness for $E(q_c)$,
		            so it is \emph{not} an inconsequential exponent.
		            If both $a$ and $\wt a$ are nonzero,
		            this corresponds to subcase (II.2b.A).

		      \item If instead $d'$ is odd,
		            then,
		            since $r_c$ is even,
		            $d$ must be even;
		            again $n$ is odd and $e^{\pi\sqrt{-1} \cdot d' n} = -1$.
		            The condition \eqref{eq:condition_fission_+} yields $a \neq \wt a$.
		            In this case,
		            $k$ is \emph{not} a breaking of specialness,
		            so it is an inconsequential exponent.
		            This corresponds to case (II.1).
	      \end{enumerate}
\end{enumerate}
Moreover,
putting together the conditions coming from sum/differences $q^{(i)} \pm \wt q^{(j)}$,
and using the fact that if $N$ is even then $\mu_N \cup \bigl( e^{\pi \sqrt{-1} \cdot \frac 1 N} \cdot \mu_N \bigr) = \mu_{2N}$,
we obtain the desired statements for all remaining subcases of (II).

Consider now case (III),
where $q_c=0$.
Now $q \neq 0$ if and only if $a \neq 0$,
in which case $r = d = N$.
Similarly $\wt q \neq 0$ if and only if $\wt a \neq 0$,
whence $\wt r = d = N$.
It is clear that if $f_{q, \wt q} = k$ then $a =\wt a =0$ is impossible.
Moreover,
if $a = 0$ then $f_{q,\wt q} = k$ if and only if $\wt a \neq 0$.
If instead $a \neq 0 \neq \wt a$,
then for any $i,j \in \set{0, \dc, d-1}$ the coefficient of exponent $k$ in $q^{(i)} \pm \wt q^{(j)}$ is
\begin{equation}
	a \cdot \exp \Bigl(2\pi \sqrt{-1} \cdot \frac{in}d \Bigr) \pm \wt a \cdot \exp \Bigl(2 \pi\sqrt{-1} \cdot \frac{jn}d \Bigr),
\end{equation}
and in turn $f_{q, \wt q} = k$ if and only if
\begin{equation}
	a \neq \wt a \cdot \zeta,
	\qquad \zeta\in \pm \mu_N,
\end{equation}
and subcase (III.b) follows.

The statements about the exponent $k$ follow from the following observations:
\begin{enumerate}
	\item
	      if $k$ is an inconsequential exponent for $L_{B\!C}(q_c)$,
	      then in particular $N = 1$,
	      and $k$ is inconsequential for $q$ and $\wt q$,
	      regardless of whether either of $a, \wt a$ vanishes;
	\item
	      otherwise:
	      \begin{enumerate}
		      \item
		            if either of $a, \wt a$ vanishes,
		            e.g.,
		            $a = 0$,
		            then $\wt a \neq 0$,
		            so that $k$ is a $B\!C$-level of $\wt q$ (and conversely if $\wt a = 0$);

		      \item
		            and if both of $a,\wt a$ are nonzero,
		            then $k$ is a $B\!C$-level of both $q$ and $\wt q$. \qedhere
	      \end{enumerate}
\end{enumerate}

\subsection{Proof of Cor.~\ref{cor:weyl_groups_are_the_same}}
\label{proof:cor_weyl_groups_are_the_same}

Rmk.~\ref{rmk:config_spaces_are_the_same} provides homeomorphisms
\begin{equation}
	F \cl \bm B_r \bigl( \wh Q \bigr) \lxra{\simeq} \bm B_{B\!C,r}(\dot Q),
	\qquad \ol F \cl \bm B_r \bigl( \wh\Theta \bigr) \lxra{\simeq} \bm B_{B\!C,r}(\Theta),
\end{equation}
taking the irregular classes $\wh\Theta \ceqq \wh\Theta \bigl( \wh Q)$ and $\Theta \ceqq \Theta(\dot Q)$.
Moreover,
the orbits of the setwise-stabilizer of $\bm B_r \bigl( \wh Q \bigr)$ in $W_{B\!C}$ match up (under $F$) with the orbits of $W_{B\!C}(\dot{\mc T})$:
this follows from~\eqref{eq:transformed_pointed_irr_type},
together with the remaining statements of Thm.~\ref{thm:full_wmcg_from_tree_type_BC}.
Therefore,
there is a commutative diagram of continuous maps
\begin{equation}
	\begin{tikzcd}
		\bm B_r \bigl( \wh Q \bigr) \ar{d} \ar{r}{F} & \bm B_{B\!C,r}(\dot Q) \ar{d} \\
		\bm B_r \bigl( \wh\Theta \bigr) \ar{r}[swap]{\ol F} & \bm B_{B\!C,r}(\Theta)
	\end{tikzcd},
\end{equation}
involving the canonical projections.
The conclusion follows from the fact that isomorphic Galois coverings have isomorphic groups of deck transformations---%
upon conjugating by $F$,
in our setting.

\subsection{Proof of Prop.~\ref{prop:form_type_D_irreg_classes}}
\label{proof:prop_form_type_D_irreg_classes}

Let $Q$ be a full $D_m$-irregular type.
If $I$ is a Stokes circle we can argue as in the proof of Prop.~\ref{prop:form_type_BC_irregular_class} to define multiplicities $n_Q(\pm I) \in \frac 1 2 \mb Z_{> 0}$,
and we must show that the following linear combination is $D$-compatible,
in the sense of Def.~\ref{def:type_D_pseudo_irreg_class}~(4.):
\begin{equation}
	\label{eq:pseudo_irr_class}
	\wt\Theta
	\ceqq \sum_{\pm I\in \mc S \slash \mb Z^{\ts}} n_Q(\pm I) \cdot (\pm I).
\end{equation}

Assume that $n_Q \bigl( \braket 0 \bigr) = 0$.
Denote by $\pm I_1, \dc, \pm I_a$ the active nonspecial Stokes-circles-up-to-sign of $Q$,
and by $\pm J_1, \dc, \pm J_b$ the active special Stokes-circles-up-to-sign,
with representatives $q_1,\dc,q_a$ and $\wt q_1,\dc,\wt q_b$---%
respectively.
For $i \in \set{1,\dc,a}$,
let $r_i \ceqq \on{ram}(q_i)$ and $m_i \ceqq m_Q(\pm I_i)$;
and for $\iota \in \set{1, \dc, b}$,
let $\wt r_\iota \ceqq \on{ram}(\wt q_\iota)$ and $\wt m_\iota \ceqq m_Q(\pm J_\iota)$.
Up to acting by the type-$A$ Weyl subgroup $\mf S_m \simeq W_A(m) \sse W_D(m)$ (cf.~the proof~\ref{proof:lem_full_gives_pointed_BC}),
we may assume that $Q$ is a concatenation of lists
\begin{equation}
	Q
	= (l_1, \dc, l_a, \ell_1, \dc, \ell_b, -l_1, \dc, -l_a, -\ell_1, \dc, -\ell_b),
\end{equation}
where:
\begin{enumerate}
	\item
	      one has
	      \begin{equation}
		      l_i
		      = \bigl( \varepsilon^{(1)}_{i,0} q_i^{(0)}, \dc, \varepsilon^{(1)}_{i,r_i-1}q_i^{(r_i - 1)}, \dc, \dc, \varepsilon^{(m_i)}_{i,0}q_i^{(0)}, \dc, \varepsilon^{(m_i)}_{i,r_i-1}q_i^{(r_i - 1)} \bigr),
	      \end{equation}
	      for $i \in \set{1,\dc,a}$,
	      where $\varepsilon_{i,j}^{(k)} \in \mb Z^{\ts}$ for $(j,k) \in \set{ 0, \dc, r_i-1} \ts \set{ 1, \dc, m_i}$;

	\item
	      and
	      \begin{equation}
		      \ell_\iota
		      =
		      \bigl( \wt \varepsilon^{(1)}_{\iota,0} \wt q_{\iota}^{(0)}, \dc, \wt \varepsilon^{(1)}_{\iota,\wt r_\iota \slash 2 - 1} \wt q_{\iota}^{(\wt r_\iota \slash 2 - 1)}, \dc,\dc, \wt \varepsilon^{(\wt m_\iota)}_{\iota,0} \wt q_{\iota}^{(0)}, \dc, \wt \varepsilon^{(\wt m_\iota)}_{\iota,\wt r_\iota \slash 2 - 1} \wt q_{\iota}^{( \wt r_\iota \slash 2 - 1)} \bigr),
	      \end{equation}
	      for $\iota \in \set{ 1, \dc, b}$,
	      where $\wt \varepsilon_{\iota,j}^{(k)} \in \mb Z^{\ts}$ for $(j,k) \in \set{ 0, \dc, \wt r_\iota \slash 2 - 1} \ts \set{1, \dc, \wt m_\iota}$.
\end{enumerate}
The signs are uniquely determined,
and the fact that $Q$ is $W_D(m)$-Galois-closed restricts them.
Indeed, the monodromy of $Q$ is obtained by concatenating those of $l_i$ and $\ell_{\iota}$,
which read
\begin{equation}
	\sigma(l_i)
	= \bigl( \varepsilon^{(1)}_{i,0} q_i^{(1)}, \dc, \varepsilon^{(1)}_{i,r_i-1} q_i^{(0)}, \dc,\dc, \varepsilon^{(m_i)}_{i,0} q_i^{(1)}, \dc, \varepsilon^{(m_i)}_{i,r_i-1} q_i^{(0)} \bigr),
\end{equation}
for $i \in \set{1,\dc,a}$,
and
\begin{equation}
	\sigma(\ell_\iota)
	= \bigl( \wt\varepsilon^{(1)}_{\iota,0} \wt q_{\iota}^{(1)}, \dc, \wt\varepsilon^{(1)}_{\iota, \wt r_\iota \slash 2 - 1} \wt q_{\iota}^{(\wt r_\iota \slash 2)}, \dc,\dc, \wt\varepsilon^{(\wt m_\iota)}_{\iota,0} \wt q_{\iota}^{(1)}, \dc, \wt\varepsilon^{(\wt m_\iota)}_{\iota,\wt r_\iota \slash 2 - 1} \wt q_{\iota}^{(\wt r_\iota \slash 2)} \bigr),
\end{equation}
for $\iota \in \set{1,\dc,b}$.
For nonspecial Stokes circles,
the sets of signs are the same for $l_i$ and $\sigma(l_i)$.
For special ones,
since $\wt q_\iota^{( \wt r_\iota \slash 2) } + \wt q_\iota^{(0)} = 0$,
there are $\wt m_\iota$ sign-changes between the set of signs in $\ell_\iota$ and $\sigma(\ell_\iota)$.
Galois-closedness implies that there is an \emph{even} number of sign-changes between $Q$ and $\sigma(Q)$,
i.e.,
that $\sum_{\iota = 1}^b \wt m_\iota$ is even.
This is equivalent to $\sum_{\pm I\in \mc S \slash \mb Z^{\ts}} n_Q(\pm I) \in \mb Z$,
as per the first statement.

For the second statement,
let $\wt\Theta$ be a pseudo-$D_m$-irregular class.
If $\wt \Theta$ contains the tame circle,
then for any irregular types $Q$ and $Q'$ with pseudo-irregular class $\wt\Theta$ we have $\Theta(Q) = \Theta(Q')$.
Indeed,
both $Q$ and $Q'$ contain the vanishing exponential factor $q = 0$,
hence are not modified if we perform a sign-change on this exponential factor.
Thus,
there exists an element $g \in W_D(m)$ such that $Q' = g\cdot Q$.
Finally,
assume that $\wt\Theta$ does \emph{not} contain the tame circle.
Write as above
\begin{equation}
	\wt\Theta
	= \sum_{i = 1}^a n_i \cdot (\pm I_i) + \sum_{\iota = 1}^b \wt n_\iota \cdot (\pm J_\iota).
\end{equation}
Choose again representatives $q_i$ and $\wt q_{\iota}$ for all Stokes circles,
and consider a full $D_m$-irregular type $Q = (q'_1, \dc, q'_m, -q'_1, \dc,-q'_m)$ with pseudo-irregular class $\wt\Theta$.
Then,
for $k \in \set{1, \dc, m}$:
\begin{enumerate}
	\item
	      if $\braket{q'_k}$ is nonspecial,
	      there exists a unique pair $(i,j) \in \set{ 1, \dc, a } \ts \set{ 0,\dc, r_i-1 }$,
	      and a unique sign $\varepsilon_k(Q) \in \mb Z^{\ts}$,
	      such that $q'_k = \varepsilon_k(Q) q_i^{(j)}$;

	\item
	      and if $\braket{q'_k}$ is special,
	      there exists a unique pair $(\iota,j) \in \set{ 1, \dc, b } \ts \set{ 0,\dc, \wt r_\iota \slash 2-1 }$,
	      and a unique sign $\wt \varepsilon_k(Q) \in \mb Z^{\ts}$,
	      such that $q'_k = \wt \varepsilon_k(Q) \wt q_{\iota}^{(j)}$.
\end{enumerate}
Now set $\varepsilon(Q) \ceqq \prod_{k = 1}^m \varepsilon_k(Q)$,
which yields a function
\begin{equation}
	\label{eq:sign_function}
	\varepsilon \cl \mc Q_{\wt\Theta} \lra \mb Z^{\ts},
\end{equation}
denoting by $\mc Q_{\wt\Theta}$ the set of full $D_m$-irregular types with pseudo-irregular class $\wt\Theta$.
If $Q,Q' \in \mc Q_{\wt\Theta}$,
the equality $\Theta(Q) = \Theta(Q')$ holds if and only if $\varepsilon(Q) = \varepsilon(Q')$,
as the latter is satisfied if and only if there is an even number of sign-changes between the multisets of signs $\set {\varepsilon_1(Q), \dc, \varepsilon_m(Q)}$ and $\set{\varepsilon_1(Q'), \dc, \varepsilon_m(Q')}$.
Therefore,
denoting by $\ol{\mc Q}_{\wt\Theta}$ the set of $D_m$-irregular classes with pseudo-$D_m$-irregular class $\wt\Theta$,
the function~\eqref{eq:sign_function} induces a \emph{bijection} $\ol\varepsilon \cl \ol{\mc Q}_{\wt\Theta} \lxra{\simeq} \mb Z^{\ts}$.

Finally,
swapping the sign of an \emph{odd} number of exponential factors induces well-defined involutions
\begin{equation}
	\label{eq:sign_involution}
	Q \lmt -Q \cl \mc Q_{\wt\Theta} \lra \mc Q_{\wt\Theta},
	\qquad \Theta \lmt -\Theta \cl \ol{\mc Q}_{\wt\Theta} \lra \ol{\mc Q}_{\wt\Theta},
\end{equation}
such that
\begin{equation}
	\varepsilon(Q) + \varepsilon(-Q)
	= 0
	= \ol\varepsilon(\Theta) + \ol\varepsilon(-\Theta) \in \mb Z^{\ts}.
\end{equation}
The latter makes $\ol{\mc Q}_{\wt\Theta}$ into a principal homogeneous set for the 2-element group $\mb Z^{\ts}$.

\subsection{Proof of Cor.~\ref{cor:enhancements}}
\label{proof:cor_enhancements}

If $\dot Q$ is a pointed quasi-irregular type containing the tame circle,
then it admits a unique enhancement which is an enhanced pointed irregular type---%
choosing $\varepsilon = 1$ in Lem.~\ref{lem:enhancements}~(1.).

Otherwise,
there are two (mutually-exclusive) situations:
\begin{enumerate}
	\item
	      either $\dot Q$ is a pointed irregular type,
	      so that the unique enhancement $\ul{\dot Q}$ of $\dot Q$---%
	      as per Lem.~\ref{lem:enhancements}~(.2)---%
	      satisfies $\varepsilon(\ul{\dot Q}) = 1$,
	      and $\ul{\dot Q}$ is an enhanced pointed irregular type;

	\item
	      or $\dot Q$ is \emph{not} a pointed irregular type,
	      whence $\varepsilon(\ul{\dot Q}) = -1$ in the notation of Lem.~\ref{lem:enhancements},
	      and there is no enhancement of $\dot Q$ into an enhanced pointed irregular type. \qedhere
\end{enumerate}

\subsection{Proof of Lem.~\ref{lem:equivalence_enhanced_non_enhanced}}
\label{proof:lem_equivalence_enhanced_non_enhanced}

Consider:
(i) an integer $p \geq 1$;
(ii) two pointed irregular types
\begin{equation}
	\dot Q
	= \bigl( (m_1, q_1),\dc, (m_p, q_p) \bigr),
	\qquad \dot Q'
	= \bigl( (m_1, q'_1), \dc, (m_p, q'_p) \bigr);
\end{equation}
and (iii) their associated enhanced pointed irregular types
\begin{equation}
	\ul{\dot Q}
	= \bigl( (m_1, q_1,\varepsilon_1), \dc, (m_p, q_p, \varepsilon_p) \bigr),
	\qquad \smash{\ul{\dot Q}}'
	= \bigl( (m_1, q'_1,\varepsilon'_1), \dc, (m_p, q'_p, \varepsilon'_p) \bigr).
\end{equation}
Assume that $\dot Q \sim_D \dot Q'$:
we must show that $\varepsilon_i = \varepsilon'_i$ for $i \in \set{1, \dc, p}$.

To this end,
let $\dot{\mc T} \ceqq \dot{\mc T}(\dot Q)$ and $\dot{\mc T}' \ceqq \dot{\mc T}(\dot Q')$ be the labelled fission trees of $\dot Q$ and $\dot Q'$,
respectively,
as in \S~\ref{sec:trees_from_d_irr_types_classes}.
Moreover,
for $i \in \set{1, \dc, p}$ denote by $\ul{\bm L}_i = (\bm L_i, \varepsilon_i)$ and $\ul{\bm L}'_i=(\bm L'_i, \varepsilon'_i)$ the enhanced $D$-level data of the $i$-th full branch of $\dot{\mc T}$ and $\dot{\mc T}'$.
Since $\dot Q$ and $\dot Q'$ are mutual $D$-admissible deformations,
one has $\bm L_i = \bm L'_i$;
and moreover,
if $\bm L_i \notin \set{\vn_{B\!C},\vn_D}$,
then the very definition of enhancements implies that $\varepsilon_i = \varepsilon'_i$.
Assume therefore that there exists $i_0 \in \set{1,\dc,p}$ such that $\bm L_{i_0} \in \set{\vn_{B\!C}, \vn_D}$,
and consider the following (mutually-exclusive) possibilities.

\begin{enumerate}
	\item
	      Either $\bm L_{i_0} = \vn_{B\!C}$ and $\bm L_j \neq \vn_D$ for all $j \in \set{1, \dc, p}$,
	      whence $\varepsilon_j = \varepsilon'_j$ for $j \neq i_0$.
	      We conclude by the global sign condition,
	      i.e.,
	      by the identities $\varepsilon_1 \dm \varepsilon_p = 1 = \varepsilon'_1 \dm \varepsilon'_p$.

	\item
	      Or $\bm L_{i_0} = \vn_{B\!C}$ and $\bm L_j = \vn_D$ for some $j \neq i_0$.
	      Then $q_{i_0} = 0$,
	      whence $\on{slope}(q_j - q_{i_0}) = \on{slope}(q_j)$ is integer (resp.,
	      half-integer) if $\varepsilon_j = 1$ (resp.,
	      if $\varepsilon_j = -1$).
	      The same holds for $\on{slope}(q'_j) = \on{slope}(q'_j - q'_{i_0})$ and $\varepsilon'_j$,
	      so the equality $\varepsilon_j = \varepsilon'_j$ follows from the hypothesis that $\on{slope}(q_j - q_{i_0}) = \on{slope}(q'_j - q'_{i_0})$,
	      and again we conclude by the global sign condition.

	\item
	      Or $\bm L_{i_0} = \vn_D$ and $\bm L_j \neq \vn_{B\!C}$ for all $j \in \set{1,\dc,p}$,
	      whence the full $D$-branches $\mc B_{i_0}$ and $\mc B'_{i_0}$ (of $\dot{\mc T}$ and $\dot{\mc T}'$,
	      respectively) are \emph{hybrid},
	      and we conclude by establishing the identity $\varepsilon_{i_0} = \varepsilon'_{i_0}$.
	      To this end,
	      one can show that the corresponding hybridization vertices $v \in \mc B_{i_0}$ and $v' \in \mc B'_{i_0}$ have the same height/fission type,
	      and consider the following subcases.
	      \begin{enumerate}
		      \item
		            If the fission at $v$ is of type (IV.a),
		            as per Def.~\ref{def:fission_tree_type_D},
		            then the first nonempty descendant-vertex $w$ of $v$---%
		            in $\mc B_{i_0}$---%
		            is indirectly mandatory.
		            If $k \neq i_0$ is such that the $k$-th leaf of $\dot{\mc T}$ is a descendant-vertex of the empty sibling-vertex of $w$,
		            then $\on{slope}(q_{i_0} - q_k) = \on{slope}(q_{i_0}) = h(w)$.
		            Analogously,
		            the fission at $v'$ is of type (IV.a),
		            the first nonempty vertex $w'$ of $\mc B'_{i_0}$ below $v'$ is indirectly mandatory,
		            and by hypothesis
		            \begin{equation}
			            h(w)
			            = \on{slope}(q_{i_0} - q_k)
			            = \on{slope}(q'_{i_0} - q'_k)
			            = h(w').
		            \end{equation}
		            The conclusion follows,
		            since $h(w)$ is integer (resp.,
		            half-integer) if $\varepsilon_{i_0} = 1$ (resp.,
		            if $\varepsilon_{i_0} = -1$),
		            and analogously for the pair $\bigl( h(w'),\varepsilon'_{i_0} \bigr)$.

		      \item
		            Otherwise,
		            the fission at $v$ is of type (IV.b).
		            Let then $w_{i_0}$ be the first nonempty vertex of $\mc B_{i_0}$ below $v$,
		            and denote by $w_k$ a sibling-vertex of $w_{i_0}$,
		            with $k \neq i_0$,
		            such that the $k$-th leaf of $\dot{\mc T}$ is a descendant-vertex of $w_k$.
		            Then:
		            \begin{equation}
			            \on{slope}(q_{i_0} - q_k)
			            = h(w_{i_0})
			            = h(w_k).
		            \end{equation}
		            But there is the same fission pattern at $v'$:
		            if $w'_{i_0} \in \mc B'_{i_0}$ is the first nonempty vertex below $v'$,
		            and its sibling-vertex $w'_k$ is an ancestor-vertex of the $k$-th leaf of $\dot{\mc T}'$,
		            then
		            \begin{equation}
			            \on{slope}(q'_{i_0} - q'_k)
			            = h(w'_{i_0})
			            = h(w'_k).
		            \end{equation}
		            We conclude as above.
	      \end{enumerate}
\end{enumerate}

\subsection{Proof of Lem.~\ref{lem:type_d_level_data}}
\label{proof:lem_type_d_level_data}

By definition of level data,
one has
\begin{equation}
	L_{B\!C}(\dot Q)
	= \wt L_D(\dot Q) \cup \Set{ \on{slope}(q) }.
\end{equation}
Thus,
it is enough to understand whether $\on{slope}(q) \in \mb Q_{> 0}$ is a $D$-level of $\dot Q$,
or not.

To this end,
suppose first that $r \ceqq \on{ram}(q) = 1$ and $\ul m = 1$.
Then $Q = (q,-q)$,
and there are \emph{no} slopes of the form $\on{slope} \bigl( q^{(i)} + q^{(j)} \bigr)$ with $i \neq j$,
whence $\wt L_D(Q) = \vn$.
If instead $\ul m > 1$,
then $\on{slope}(q + q) = \on{slope}(q)$ \emph{does} appear amongst the numbers $\on{slope} \bigl( q^{(i)} + q^{(j)} \bigr)$,
whence $\on{slope}(q) \in \wt L_D(q)$ and $\wt L_D(\dot Q) = L_{B\!C}(\dot Q)$.
Similarly,
if:
(i) $r = 2$;
(ii) $q$ is special;
and (iii) $\ul m = 1$;
then $q$ is of the form
\begin{equation}
	q
	= \sum_{i=1}^s a_i z^{-i-\frac 1 2},
	\qquad a_i \in \mb C,
	\quad a_s \neq 0.
\end{equation}
Therefore $Q = (q,-q)$,
and again there are \emph{no} slopes of the form $\on{slope} \bigl( q^{(i)} + q^{(j)} \bigr)$ with $i \neq j$.

Otherwise $\dot Q$ is of rank $m > 1$, and we can write $Q = (\wt q_1, \dc, \wt q_m, -\wt q_1, \dc, -\wt q_m)$ for the full quasi-irregular type corresponding to $\dot Q$.
Since $\on{slope}(\wt q_1) = \on{slope}(\wt q_2) = \on{slope}(q) \eqqc k$,
either $\on{slope} \bigl( \wt q_1 + \wt q_2 \bigr) = k$ or $\on{slope} \bigl( \wt q_1 - \wt q_2 \bigr) = k$,
whence $\on{slope}(q) \in \wt L_D(q)$ and $\wt L_D(\dot Q) = L_{B\!C}(\dot Q)$.

\subsection{Proof of Lem.-Def.~\ref{lem:outer_automorphisms}}
\label{proof:lem_outer_automorphisms}

The first statement follows from:
(i) the Lie-algebra splitting $\mf g = \mf Z \ops \mf g'$;
(ii) the fact that all Lie-algebra automorphisms preserve the Lie-centre;
and (iii) the fact that they restrict to---%
any---%
$\mb C$-linear automorphism thereof.

For the remaining statements,
choose a \emph{pinning of} $(\mf g',\mf t')$.\fn{
	An `épinglage'~\cite[Exp.~XXIII, Déf.~1.1]{demazure_grothendieck_1970_sga_3_iii}.
	We work over $S = \Spec \mb C$,
	so everything splits,
	etc.;
	and we only need the infinitesimal/Lie-algebra version of this.}~This consists of:
(i) a base $\Delta' \sse \Phi'$ of simple roots for the root system $\Phi' = \Phi(\mf g',\mf t') \sse (\mf t')^{\dual}$;
and (ii) generators $e_\theta \in \mf g'_\theta \sm (0)$ for the root-lines $\mf g'_\theta \sse \mf g'$---%
corresponding to $\theta \in \Delta'$.
Then the `simple negative' root-vectors $f_\theta = e_{-\theta} \in \mf g'_{-\theta} \sm (0)$ are canonically determined by imposing that
\begin{equation}
	[e_\theta,f_\theta]
	= h_\theta \ceqq \theta^{\dual} \in [\mf g'_\theta,\mf g'_{-\theta}] \sse \mf t'.
\end{equation}
Now define the group
\begin{equation}
	\label{eq:diagram_automorphisms}
	\Aut(\Delta')
	\ceqq \Set{ \tau \in \Aut(\Phi') | \tau (\Delta') = \Delta' },
\end{equation}
and identify it (tacitly) with its contragredient version---%
sitting inside $\GL_{\mb C}(\mf t')$.
Up to unique $\mb C$-linear extensions,
this matches up with the group of permutations of $\Delta'$ which preserve Cartan integers (cf.~Rmk.~\ref{rmk:cartan_integers}),
i.e.,
the group of automorphisms of the Dynkin diagram of $(\mf g',\mf t',\Delta')$ (cf.~\cite{armstrong_2010_the_automorphism_group_of_a_root_system}).
There is now a semidirect splitting $\Aut(\Phi') \simeq \Aut(\Delta') \lts W$ (cf.~\cite[Chp.~VI, \S~5, Prop.~16]{bourbaki_1968_groupes_et_algebres_de_lie_chapitres_iv_v_vi}),
which can be further lifted to $\Aut(\mf g')$ by using the whole pinning.
Precisely,
given $\tau \in \Aut(\Delta')$,
the existence of Chevalley--Serre presentations of $\mf g'$~\cite{serre_1987_complex_semisimple_lie_algebras} implies that the prescription
\begin{equation}
	\label{eq:diagram_auto_lift}
	A \lmt \tau (A),
	\quad e_\theta \lmt e_{\tau (\theta)},
	\quad f_\theta \lmt f_{\tau (\theta)},
	\qquad A \in \mf t',
	\quad \theta \in \Delta',
\end{equation}
extends uniquely to a Lie-algebra automorphism $\wt\tau $ of $\mf g'$;
whence a group embedding $s \cl \tau \mt \wt\tau \cl \Aut(\Delta') \hra \Aut(\mf g')$---%
whose image preserves $\mf t'$.

Now one can show that:
(i) the intersection $\Inn(\mf g') \cap s \bigl( \Aut(\Delta') \bigr)$ is trivial;
and (ii) the group $\Aut(\mf g')$ is generated by $\Inn(\mf g')$ and $s \bigl( \Aut(\Delta') \bigr)$.
(This yields a group isomorphism $\Out(\mf g') \simeq \Aut(\Delta')$,
and completes the proof.)
For the first claim,
recall that the Weyl group $W$ can be identified with the setwise-modulo-pointwise-stabilizer of $\mf t'$ in $\Inn(\mf g')$ (cf.~\cite[\S~4]{varadarajan_1984_lie_groups_lie_algebras_and_their_representations} and~\cite[Lem.~2.2.3]{collingwood_mcgovern_1993_nilpotent_orbits_in_semisimple_lie_algebras}),
and that $W$ permutes bases of $\Phi'$ in simply-transitive fashion.
Since the extension of~\eqref{eq:diagram_auto_lift} preserves the subspace
\begin{equation}
	(\mf t')^{\dual} \simeq \Bigl( \bops_{\Phi'} \mf g'_\alpha \Bigr)^\perp \sse (\mf g')^{\dual},
\end{equation}
as well as $\Delta'$,
it would be trivial if inner.
For the second claim,
start from any automorphism $\dot{\bm\varphi}' \cl \mf g' \lxra{\simeq} \mf g'$.
Conjugating by a suitable element of $\Inn(\mf g')$,
assume that it preserves $\mf t'$ (all Cartan subalgebras of $\mf g'$ are conjugate,
cf.~again~\cite[\S~16]{humphreys_1972_introduction_to_lie_algebras_and_representation_theory}).
Now the restriction of $\dot{\bm\varphi}'$ to $\mf t'$ is a root-system automorphism.\fn{
	Choosing $A \in \mf t'$,
	a root $\alpha \in \Phi'$,
	and a root vector $X \in \mf g'_\alpha$,
	one has
	\begin{equation}
		\bigl[ A,\dot{\bm\varphi}'(X) \bigr]
		= \dot{\bm\varphi}' \bigl( [A',X] \bigr)
		= \braket{ \alpha' \mid A } \dot{\bm\varphi}'(X),
		\qquad A' \ceqq (\dot{\bm\varphi}')^{-1}(A) \in \mf t,
		\quad \alpha' \ceqq \prescript{t}{}{(\dot{\bm\varphi}')}^{-1}(\alpha) \in \mf t^{\dual}.
	\end{equation}
	This means that $(0) \neq \dot{\bm\varphi}'(\mf g'_\alpha) \sse \mf g'_{\alpha'}$,
	whence $\alpha' \in \Phi'$.}~Moreover,
conjugating by an inner automorphism stabilizing $\mf t'$ (lifting an element of $W$),
one can assume that $\dot{\bm\varphi}'$ preserves $\Delta'$.
At this stage,
there exists $\tau \in \Aut(\Delta')$ such that $\dot{\bm\varphi}'$ maps
\begin{equation}
	\quad e_\theta \lmt c_\theta e_{\tau(\theta)},
	\quad f_\theta \lmt c'_\theta f_{\tau(\theta)},
\end{equation}
for suitable $\Delta'$-tuples $\bm c = (c_\theta)_\theta, \bm c' = (c'_\theta)_\theta \in (\mb C^{\ts})^{\Delta'}$.
Furthermore,
the condition that $\dot{\bm\varphi}'(h_\theta) = h_{\tau(\theta)}$ forces $c_\theta c'_\theta = 1$ for $\theta \in \Delta'$.
Finally,
one can conjugate by the Adjoint action of an element of $T'$ so that $\bm c$ is trivial,
and the result lies in $s \bigl( \Aut(\Delta') \bigr)$.

\subsection{Proof of Lem.~\ref{lem:full_twist_preserves_centralizer}}
\label{proof:lem_full_twist_preserves_centralizer}

Suppose first that $\wh Q = Aw^{-1}$ for some $A \in \mf t$:
then we impose that $\zeta_r A = \dot{\bm \varphi}(A) \in \mf t$.
Choose now an element $Y \in \mf g^A \sse \mf g$,
and compute
\begin{equation}
	\bigl[ \dot{\bm\varphi}(Y),A \bigr]
	= \dot{\bm\varphi} \bigl( [Y,\dot{\bm\varphi}^{-1}(A)] \bigr)
	= \zeta_r^{-1}\dot{\bm\varphi} \bigl( [Y,A] \bigr) = 0.
\end{equation}
The statement follows,
since $L = G^A \sse G$ is connected with Lie algebra $\mf l = \mf g^A$.
In the general case where $\wh Q = \sum_{i = 1}^s A_iw^{-i}$,
just iterate the same argument starting from the leading coefficient $A_s \in \mf t$,
proving that $\dot{\bm\varphi} (\mf g^{A_i}) \sse \mf g^{A_i}$ for $i \in \set{1,\dc,s}$,
etc.
(As usual,
this uses the fact that the nested centralizers remain connected.)

\subsection{Proof of Lem.~\ref{lem:doubly_twisted_setwise_stabilizer}}
\label{proof:lem_doubly_twisted_setwise_stablizer}

Identify the elements of $\Phi' \sse (\mf t')^{\dual} \simeq \mf Z^\perp$ with the restriction of the covectors of $\Phi \sse \mf t^{\dual}$---%
onto $\mf t' \sse \mf t$.
For any $\alpha \in \Phi$,
one then has
\begin{equation}
	H_\alpha
	= \mf Z \ts H_{\alpha'} \sse \mf Z \ts \mf t',
	\qquad \alpha'
	\ceqq \eval[1]\alpha_{\mf t'} \in \Phi',
\end{equation}
and so $H_\alpha(\phi) = \mf Z \ts H_{\alpha'}(\phi') \sse \mf t$.
This yields
\begin{equation}
	\label{eq:flat_split}
	\mf t_\phi
    = \mf Z \ts \mf t'_{\phi'},
	\qquad \phi'
	\ceqq \Set{ \alpha' | \alpha \in \phi } \sse \Phi',
\end{equation}
where as usual $\mf t'_{\phi'} = \ker(\phi') \sse \mf t'$;
whence
\begin{equation}
	\label{eq:stratum_split}
	\bm B \bigl( \wh Q \bigr)
	= \mf Z \ts \Bigl( \mf t'_{\phi'} \sm \bigcup_{\Phi' \sm \phi'} H_{\alpha'}(\phi') \Bigr) \sse \mf t_\phi,
\end{equation}
in the notation of~\eqref{eq:pure_untwisted_nongeneric_deformation_space_1_coeff}.
Now the central part $f$---%
of $\bm g$---%
acts trivially on $\mf t' \sse \mf t$,
and so:
(i) it preserves~\eqref{eq:flat_split}--\eqref{eq:stratum_split};
and (ii) it acts trivially on $\Phi'$ and $\Phi$.

So it is enough to prove the equivalences in the first part of the statement,
for the `semisimple part' $\bm g'$---%
of $\bm g$.
Now the proof~\ref{proof:lem_about_setwise_stabilizers} extends,
because:
(i)  the action of $\Aut(\Phi')$ induces bijections of $\Phi'$ and $(\Phi')^{\dual}$;
and (ii) it permutes the root-hyperplanes analogously to~\eqref{eq:wall_permutation}.

\subsection{Proof of Prop.~\ref{prop:killing_multiplicities}}
\label{proof:prop_killing_multiplicities}

For the first statement,
by~\eqref{eq:orbit_shift},
the spectral constraints~\eqref{eq:phi_monodromy_realization} amount to the identities
\begin{equation}
	A^{(l+1)}_i
	= \zeta^{-i}_r \bm g_l \bigl( A^{(l)}_i \bigr) \in \mf t_{l+1},
	\qquad l \in \set{1,\dc,k},
	\quad i \in \set{1,\dc,s},
\end{equation}
setting $\mf t_{k+1} \ceqq \mf t_1$ and $A^{(k+1)}_i \ceqq A^{(1)}_i$---%
in cyclic fashion.
Iterating from $l = 1$ then yields
\begin{equation}
	\label{eq:orbit_coordinates}
	A^{(l)}_i
	= \zeta^{-i(l-1)}_r \bm h_l \bigl( A^{(1)}_i \bigr),
	\qquad \bm h_l
	\ceqq \bm g_{l-1} \dm \bm g_1,
	\quad l \in \set{1,\dc,k+1},
\end{equation}
with the convention that $\bm h_1 \ceqq \Id_{\mf t}$.
In particular,
the case $l = k+1$ yields the `closing' condition
\begin{equation}
	\bm h \bigl( A^{(1)}_i \bigr)
	= \zeta_r^{ik} A^{(1)}_i
	= \wt\zeta^{\, i} A^{(1)}_i,
\end{equation}
which is precisely the twisted Galois-closedness of $\wh Q^{(1)}$,
writing $\bm h = \bm h_{k+1}$ (as in the statement).
The pole order is \emph{unchanged},
since~\eqref{eq:orbit_coordinates} implies that the vanishing $A^{(1)}_s = 0 \in \mf t = \mf t_1$ forces $A_s = 0 \in \wt{\mf t}$.

For the second statement,
it now follows that the projection $A \mt A^{(1)}$ restricts to---%
continuous---%
$\mb C$-linear \emph{bijections} $\wt{\mf t} \bigl( \wt{\bm g},\zeta^i_r \bigr) \lxra{\simeq} \mf t \bigl( \bm h,\wt \zeta^i \bigr)$ ($i \in \set{1,\dc,s}$),
with---%
continuous---%
inverses prescribed by~\eqref{eq:orbit_coordinates}.
Moreover,
the $\bm h_l$ are root-system automorphisms of $\Phi$,
and so for $\alpha \in \Phi$ one has $\braket{ \alpha, A^{(l)}_i } = 0$ if and only if $\braket{ \bm h_l^{-1} (\alpha), A^{(1)}_i } = 0$ (using $\zeta^{-i(l-1)}_r \neq 0$).
Therefore,
the nested Levi annihilators $\phi_i = \set{ A_i,\dc,A_s }^\perp \cap \wt\Phi$ satisfy
\begin{equation}
    \phi_i \cap \Phi_l 
    = \bm h_l^{-1} \bigl( \phi^{(1)}_i \bigr),
    \qquad \phi^{(1)}_i
    \ceqq \set{ A_i^{(1)},\dc,A_s^{(1)} }^\perp \cap \Phi,
    \quad l \in \set{1,\dc,k},
\end{equation}
and so all matches up:
the kernel-flags,
the hyperplane traces,
and the nested flat-density conditions.

\subsection{Proof of Lem.~\ref{lem:no_doubly_twisted_marking}}
\label{proof:lem_no_doubly_twisted_marking}

The proof~\ref{proof:lem_no_marking} extends:
the point is that one twists both elements of $W$ by one and the same outermorphism.

In more details,
as usual,
the fact that $\dot\varphi g'(A) = \dot\varphi g''(A) = \zeta_r A \in \mf t_\phi$ implies that $\dot \varphi g', \dot\varphi g'' \in N_{\GL_{\mb C}(\mf t)}(\mf t_\phi)$.
Moreover,
deleting $\dot\varphi$ yields $g'(A) = g''(A)$,
and so $h \ceqq (g')^{-1}g'' \in W$ acts as the identity on $\mf t_\phi$ (cf.~Lem.~\ref{lem:about_pointwise_stabilizers}).
It follows that
\begin{equation}
	\eval[1]{g'}_{\mf t_\phi}
	= \eval[1]{g''}_{\mf t_\phi} \cl \mf t_\phi \lra \mf t,
\end{equation}
and so $\dot\varphi g'$ and $\dot\varphi g''$ also coincide upon restriction.
(Note that this does \emph{not} prove that $g',g'' \in N_W(\mf t_\phi)$,
nor that $\dot\varphi \in N_{\GL_{\mb C}(\mf t)}(\mf t_\phi)$.)

\subsection{Proof of Prop.~\ref{prop:killing_multiplicities_full}}
\label{proof:prop_killing_multiplicities_full}

The canonical projection $\wt{\mf t} \thra \mf t_1 = \mf t$ is \emph{equivariant} with respect to the actions of $\wt W \simeq W^k$ on the source and of $W = W(\mf g_1,\mf t_1)$ on the target.
So the same holds for their respective diagonal actions on $\wt{\mf t}^s$ and $\mf t^s$,
identified---%
as usual---%
with the vector spaces of pole-order bounded irregular types.
In particular~\eqref{eq:orbit_descent} induces a \emph{continuous surjection} of the quotients,
using the description
\begin{equation}
	\bm B_{\wt{\dot\varphi},r} \bigl( \wh \Theta \bigr)
	\simeq \bm B_{\wt{\bm g},r} \bigl( \wh Q \bigr) \bs Z_{\wt W,\bm\phi}(\wt{\bm g})
	\simeq \bigl( \wt W.\bm B_{\wt{\bm g},r} \bigl( \wh Q \bigr) \bigr) \bs \wt W,
\end{equation}
involving the $\wt W$-translate of the pure admissible-deformation space (cf.~\eqref{eq:stratum_of_subcoset_orbit});
and analogously for the target.

Let us now prove that the descended map is \emph{injective}.
To this end,
choose irregular types $\wh Q' = \sum_{i = 1}^s A'_iw^{-i}$ and $\wh Q'' = \sum_{i = 1}^s A''_i w^{-i}$,
lying in $\bm B_{\wt{\bm g},r} \bigl( \wh Q \bigr)$,
such that one has $(\wh Q'')^{(1)} = u \bigl( (\wh Q')^{(1)} \bigr)$ for some $u \in W$;
with coefficients $A'_i = \bigl( (A'_i)^{(1)},\dc,(A'_i)^{(k)} \bigr)$ and $ A''_i = \bigl( (A''_i)^{(1)},\dc,(A''_i)^{(k)} \bigr)$ lying in $\wt{\mf t}$,
having `cyclic' components $(A'_i)^{(l)}$ and $(A''_i)^{(l)}$ lying in $\mf t_l = \mf t$ ($i \in \set{1,\dc,s}$ and $l \in \set{1,\dc,k}$)---%
respectively.
Then~\eqref{eq:orbit_coordinates} yields
\begin{align}
	(A''_i)^{(l)}
	= \zeta^{-i(l-1)}_r \bm h_l \bigl( (A''_i)^{(1)} \bigr)
	 & = \bm h_l \cdot u \cdot \bm h^{-1}_l \Bigl( \zeta^{-i(l-1)}_r \bm h_l \bigl( (A_i')^{(1)} \bigr) \Bigr) \\
	 & = \bm h_l \cdot u \cdot \bm h^{-1}_l \bigl( (A_i')^{(l)} \bigr),
\end{align}
for all $i$ and $l$.
As usual,
the `partial' products $\bm h_l = \bm g_{l-1} \dm \bm g_1 \in \Aut(\Phi)$ (cf.~\eqref{eq:orbit_coordinates}) \emph{normalize} the Weyl group,
whence overall
\begin{equation}
	\label{eq:rotating_weyl_element}
	\wh Q''
	= \wt u \bigl( \wh Q' \bigr),
	\qquad \wt u
	\ceqq \bigl( \bm h_1 u \bm h^{-1}_1,\dc,\bm h_k u \bm h^{-1}_k \bigr) \in W^k \simeq \wt W.
\end{equation}
So $\wh Q'$ and $\wh Q''$ lie in the same $\wt W$-orbit,
as needed.
(Recall that $\bm h_1 = \Id_{\mf t}$,
and so the first component of $\wt u$ is $u$ itself;
i.e.,
the map $u \mt \wt u$ is a section of the canonical projection $\wt W \thra W$.)

Finally,
the descended map $\ol\pi^{(1)}$ is \emph{open},
for abstract reasons:
(i) the quotient map for the $W$-action,
denoted by
\begin{equation}
	q \cl W.\bigl(\bm B_{\bm h,r'} \bigl( \wh Q^{(1)} \bigr) \bigr) \lthra \bm B_{\dot\varphi,r'} \bigl( \wh\Theta^{(1)} \bigr),
\end{equation}
is open;
(ii) the homeomorphism $\pi^{(1)}$ is also open;
(iii) one now has
\begin{equation}
	\ol\pi^{(1)} \bigl( \wt q(U) \bigr)
	= q \bigl( \pi^{(1)}(U) \bigr) \sse \bm B_{\dot\varphi,r'} \bigl( \wh\Theta^{(1)} \bigr),
\end{equation}
for any open subspace $U \sse \bm B_{\wt{\bm g},r} \bigl( \wh Q \bigr)$,
where $\wt q$ is the quotient map for the $\wt W$-action;
and (iv) every open subspace $\ol U = \wt q \bigl( \wt q^{-1}(\ol U) \bigr) \sse \bm B_{\wt{\dot\varphi},r} \bigl( \wh \Theta \bigr)$ can be saturated.

\subsection{Proof of Lem.~\ref{lem:twisted_normalization_relative_refl_group}}
\label{proof:lem_twisted_normalization_relative_refl_group}

Let us work type by type.
\begin{description}
	\item[$A$]
	      The central part $f$ of $\bm g = (f,\bm g')$ plays no role here,
	      and can be arbitrary.
	      As per the `semisimple' part,
	      write $\Phi_\bullet = \Phi_A(m-1)$ for some integer $m \geq 1$.
	      Then:
	      \begin{enumerate}
		      \item
		            If $m \in \set{1,2}$ one has $\dot\varphi' = 1$,
		            and the result in the split case was already established.

		      \item
		            If instead $m \geq 3$,
		            then the nontrivial outermorphism corresponds to flipping---%
		            the standard presentation of---%
		            the Dynkin diagram left-to-right.
		            For the standard choice of base $\Delta_A = \set{\theta_1,\dc,\theta_{m-1}}$,
		            this corresponds to mapping $\theta_i \mt \theta_{m-i}$ for $i \in \set{1,\dc,m-1}$ under the involution $\dot\varphi' = - w_0$,
		            where $w_0 \in W_A(m) \simeq \mf S_m$ is the \emph{longest element}.

		            But it is more convenient to work with another representative of the outer part,
		            thereby relabelling the elements of the coset.
		            Namely,
		            let $\dot\varphi'(A) \ceqq -A$,
		            for all $A \in V^+_m$,
		            which is induced by the negative-transposition of square matrices.
		            Now an element of the reflection coset $\bigl( V^+_m,-W_A(m) \bigr)$ which preserves a Levi subsystem $\phi$ restricts to the opposite of a block permutation in $\mf t_\phi \simeq V^+_{m_\phi}$ (preserving block sizes,
		            cf.~again \S~\ref{sec:type_A_relative_refl_group}):
		            so $\bm g_\phi \in - \mf S_{m_\phi}$ normalizes $G_A(\phi) = W_A(m_\phi) \simeq \mf S_{m_\phi}$.
	      \end{enumerate}

	\item[$B\!C$]
	      Here there are just \emph{no} outermorphisms,
	      and so the statement follows from the previously established split version.

	\item[$D$]
	      Take $\Phi = \Phi_D(m)$,
	      with $m \geq 5$.
	      The only nontrivial outermorphism corresponds to flipping---%
	      the standard presentation of---%
	      the Dynkin diagram upside-down.
	      We represent it by the negative transposition $\dot\varphi = (m \mid -m)$ of $\ul m^\pm = \set{\pm 1,\dc,\pm m}$,
	      whose corresponding $2$-twisted reflection coset is $\bigl( \wt V_m, \dot\varphi W_D(m) \bigr)$:
	      so $\dot\varphi W_D(m) \sse W_{B\!C}(m)$ is the nontrivial coset of the index-$2$ type-$D$ Weyl group,
	      sitting inside the type-$B\!C$ one.

	      Now let $\phi_D \sse \phi$ be the type-$D$ part of $\phi$,
	      as per Lem.-Def.~\ref{lem:type_D_part}.
	      Observe that \emph{any} signed permutation of $\ul m^\pm$ which preserves $\phi$ (such as $\bm g$),
	      will also preserve the zero block of the equivalence relation $\sim_D$ associated with $\phi$;
	      this follows from the description~\eqref{eq:type_D_zero_block},
	      and from the fact that $\bm g(\mf t_\phi) = \mf t_\phi$.
	      So one has $\bm g(\phi_D) = \phi_D$,
	      extending the split case,
	      and again---%
	      in the usual identifications $\wt V_{m_\phi} \simeq \mf t_\phi \simeq \mb C^{m_\phi}$---%
	      the restriction $\bm g_\phi$ will induce a signed permutation of $\ul m^\pm_\phi$,
	      i.e.,
	      an element of $W_{B\!C}(m_\phi)$.
	      Now the conclusion follows from the fact that one has $G_D(\phi) \in \Set{W_D(m_\phi),W_{B\!C}(m_\phi)}$ (with the former only occurring in the quasigeneric case).
\end{description}

Incidentally,
observe that in all cases the semisimple part $\bm g'$ induces a signed block-permutation of the coordinates of $\mb C^{m_\phi}$ (in the usual embedding $\pm W_A(m_\phi) \hra W_{B\!C}(m_\phi)$),
and so it is \emph{unitary} with respect to the standard Hermitian form used to define $G_\bullet(\phi)$ (cf.~Rmk.~\ref{rmk:different_coroots})---%
and one can indeed appeal to Lehrer--Springer theory.

\bibliographystyle{amsplain}
\bibliography{bibliography}

\end{document}